\documentclass[numbers=enddot,12pt,final,onecolumn,notitlepage]{scrartcl}%
\usepackage[headsepline,footsepline,manualmark]{scrlayer-scrpage}
\usepackage{amsfonts}
\usepackage{amssymb}
\usepackage{amsmath}
\usepackage{amsthm}
\usepackage{framed}
\usepackage{comment}
\usepackage{color}
\usepackage[breaklinks=true]{hyperref}
\usepackage[sc]{mathpazo}
\usepackage[T1]{fontenc}
\usepackage{needspace}
\usepackage{pythonhighlight}
\usepackage{tabls}
\providecommand{\U}[1]{\protect\rule{.1in}{.1in}}
\theoremstyle{definition}
\newtheorem{theo}{Theorem}[section]
\newenvironment{theorem}[1][]
{\begin{theo}[#1]\begin{leftbar}}
{\end{leftbar}\end{theo}}
\newtheorem{lem}[theo]{Lemma}
\newenvironment{lemma}[1][]
{\begin{lem}[#1]\begin{leftbar}}
{\end{leftbar}\end{lem}}
\newtheorem{prop}[theo]{Proposition}
\newenvironment{proposition}[1][]
{\begin{prop}[#1]\begin{leftbar}}
{\end{leftbar}\end{prop}}
\newtheorem{defi}[theo]{Definition}
\newenvironment{definition}[1][]
{\begin{defi}[#1]\begin{leftbar}}
{\end{leftbar}\end{defi}}
\newtheorem{remk}[theo]{Remark}
\newenvironment{remark}[1][]
{\begin{remk}[#1]\begin{leftbar}}
{\end{leftbar}\end{remk}}
\newtheorem{coro}[theo]{Corollary}
\newenvironment{corollary}[1][]
{\begin{coro}[#1]\begin{leftbar}}
{\end{leftbar}\end{coro}}
\newtheorem{conv}[theo]{Convention}
\newenvironment{convention}[1][]
{\begin{conv}[#1]\begin{leftbar}}
{\end{leftbar}\end{conv}}
\newtheorem{conj}[theo]{Conjecture}
\newenvironment{conjecture}[1][]
{\begin{conj}[#1]\begin{leftbar}}
{\end{leftbar}\end{conj}}
\newtheorem{exam}[theo]{Example}
\newenvironment{example}[1][]
{\begin{exam}[#1]\begin{leftbar}}
{\end{leftbar}\end{exam}}
\newenvironment{statement}{\begin{quote}}{\end{quote}}
\newenvironment{fineprint}{\begin{small}}{\end{small}}
\newcommand{\silentsubsection}{\subsection}

\let\sumnonlimits\sum
\let\prodnonlimits\prod
\let\cupnonlimits\bigcup
\let\capnonlimits\bigcap
\renewcommand{\sum}{\sumnonlimits\limits}
\renewcommand{\prod}{\prodnonlimits\limits}
\renewcommand{\bigcup}{\cupnonlimits\limits}
\renewcommand{\bigcap}{\capnonlimits\limits}
\newenvironment{verlong}{}{}
\newenvironment{vershort}{}{}
\newenvironment{noncompile}{}{}
\excludecomment{verlong}
\includecomment{vershort}
\excludecomment{noncompile}
\ihead{Petrie symmetric functions}
\ohead{page \thepage}
\begin{document}

\title{Petrie symmetric functions}
\author{Darij Grinberg}
\date{April 22, 2021}
\maketitle

\begin{abstract}
\textbf{Abstract.} For any positive integer $k$ and nonnegative integer $m$,
we consider the symmetric function $G\left(  k,m\right)  $ defined as the sum
of all monomials of degree $m$ that involve only exponents smaller than $k$.
We call $G\left(  k,m\right)  $ a \textit{Petrie symmetric function} in honor
of Flinders Petrie, as the coefficients in its expansion in the Schur basis
are determinants of Petrie matrices (and thus belong to $\left\{
0,1,-1\right\}  $ by a classical result of Gordon and Wilkinson). More
generally, we prove a Pieri-like rule for expanding a product of the form
$G\left(  k,m\right)  \cdot s_{\mu}$ in the Schur basis whenever $\mu$ is a
partition; all coefficients in this expansion belong to $\left\{
0,1,-1\right\}  $. We also show that $G\left(  k,1\right)  ,G\left(
k,2\right)  ,G\left(  k,3\right)  ,\ldots$ form an algebraically independent
generating set for the symmetric functions when $1-k$ is invertible in the
base ring, and we prove a conjecture of Liu and Polo about the expansion of
$G\left(  k,2k-1\right)  $ in the Schur basis.

\textbf{Keywords:} symmetric functions, Schur functions, Schur polynomials,
combinatorial Hopf algebras, Petrie matrices, Pieri rules, Murnaghan--Nakayama rule.

\textbf{MSC2010 Mathematics Subject Classifications:} 05E05.

\end{abstract}
\tableofcontents

\section*{***}

Considered as a ring, the \textit{symmetric functions} (which is short for
\textquotedblleft formal power series in countably many indeterminates
$x_{1},x_{2},x_{3},\ldots$ that are of bounded degree and fixed under
permutations of the indeterminates\textquotedblright) are hardly a remarkable
object: By a classical result essentially known to Gauss, they form a
polynomial ring in countably many indeterminates. The true theory of symmetric
functions is rather the study of specific \textit{families} of symmetric
functions, often defined by combinatorial formulas (e.g., as multivariate
generating functions) but interacting deeply with many other fields of
mathematics. Classical families are, for example, the \textit{monomial
symmetric functions} $m_{\lambda}$, the \textit{complete homogeneous symmetric
functions }$h_{n}$, the \textit{power-sum symmetric functions} $p_{n}$, and
the \textit{Schur functions }$s_{\lambda}$. Some of these families -- such as
the monomial symmetric functions $m_{\lambda}$ and the Schur functions
$s_{\lambda}$ -- form bases of the ring of symmetric functions (as a module
over the base ring).

In this paper, we introduce a new family $\left(  G\left(  k,m\right)
\right)  _{k\geq1;\ m\geq0}$ of symmetric functions, which we call the
\textit{Petrie symmetric functions} in honor of Flinders Petrie. For any
integers $k\geq1$ and $m\geq0$, we define $G\left(  k,m\right)  $ as the sum
of all monomials of degree $m$ (in $x_{1},x_{2},x_{3},\ldots$) that involve
only exponents smaller than $k$. When $G\left(  k,m\right)  $ is expanded in
the Schur basis (i.e., as a linear combination of Schur functions $s_{\lambda
}$), all coefficients belong to $\left\{  0,1,-1\right\}  $ by a classical
result of Gordon and Wilkinson, as they are determinants of so-called
\textit{Petrie matrices} (whence our name for $G\left(  k,m\right)  $). We
give an explicit combinatorial description for the coefficients as well. More
generally, we prove a Pieri-like rule for expanding a product of the form
$G\left(  k,m\right)  \cdot s_{\mu}$ in the Schur basis whenever $\mu$ is a
partition; all coefficients in this expansion again belong to $\left\{
0,1,-1\right\}  $ (although we have no explicit combinatorial rule for them).
We show some further properties of $G\left(  k,m\right)  $ and prove that if
$k$ is a fixed positive integer such that $1-k$ is invertible in the base
ring, then $G\left(  k,1\right)  ,G\left(  k,2\right)  ,G\left(  k,3\right)
,\ldots$ form an algebraically independent generating set for the symmetric
functions. We prove a conjecture of Liu and Polo in \cite[Remark
1.4.5]{LiuPol19} about the expansion of $G\left(  k,2k-1\right)  $ in the
Schur basis.

This paper begins with Section \ref{sect.not}, in which we introduce the
notions and notations that the paper will rely on. (Further notations will
occasionally be introduced as the need arises.) The rest of the paper consists
of two essentially independent parts. The first part comprises Section
\ref{sect.thms}, in which we define the Petrie symmetric functions $G\left(
k,m\right)  $ (and the related power series $G\left(  k\right)  $) and state
several of their properties, and Section \ref{sect.proofs}, in which we prove
said properties. The second part is Section \ref{sect.liu}, which is devoted
to proving the conjecture of Liu and Polo.\footnote{This proof is independent
of the first part of the paper, except that it uses the very simple
Proposition \ref{prop.G.basics} \textbf{(c)}.} A final Section \ref{sect.fin}
adds comments, formulates two conjectures, and (in its last subsection)
explores a more general family of symmetric functions that still shares some
of the properties of the Petrie functions $G\left(  k,m\right)  $. (As a
byproduct of the latter generalization, a formula for the antipode of
$G\left(  k,m\right)  $ -- Corollary \ref{cor.Gkm.antipode} -- emerges.)

\subsubsection*{Acknowledgments}

I thank Moussa Ahmia, Per Alexandersson, Fran\c{c}ois Bergeron, Steve Doty,
Ira Gessel, Jim Haglund, Linyuan Liu, Patrick Polo, Sasha Postnikov,
Christopher Ryba, Richard Stanley, Ole Warnaar and Mark Wildon for interesting
and helpful conversations, and two referees for helpful suggestions. Special
thanks are due to Sasha Postnikov for his permission to include his
generalization of the Petrie symmetric functions.

This paper was started at the Mathematisches Forschungsinstitut Oberwolfach,
where I was staying as a Leibniz fellow in Summer 2019, and finished during a
semester program at the Institut Mittag--Leffler in 2020. I thank both
institutes for their hospitality. The SageMath computer algebra system
\cite{SageMath} has been used in discovering some of the results.

\begin{fineprint}
This material is based upon work supported by the Swedish Research Council
under grant no. 2016-06596 while the author was in residence at Institut
Mittag--Leffler in Djursholm, Sweden during Spring 2020.
\end{fineprint}

\subsubsection*{Remarks}

\textbf{1.} A short exposition of the main results of this paper (without
proofs), along with an additional question motivated by it, can be found in
\cite{Grinbe20}.

\textbf{2.} While finishing this work, I have become aware of three
independent discoveries of the Petrie symmetric functions $G\left(
k,m\right)  $:

\begin{enumerate}
\item[\textbf{(a)}] In \cite[\S 3.3]{DotWal92}, Stephen Doty and Grant Walker
define a \emph{modular complete symmetric function }$h_{d}^{\prime}$, which is
precisely our Petrie symmetric function $G\left(  k,m\right)  $ up to a
renaming of variables (namely, their $m$ and $d$ correspond to our $k$ and
$m$). Some of our results appear in their work: Our Theorem
\ref{thm.Gkm-genset} is (a slight generalization of) \cite[Corollary
3.9]{DotWal92}; our Theorem \ref{thm.Uk.main} is (part of) \cite[Proposition
3.15]{DotWal92} restated in the language of Hopf algebras. The $h_{d}^{\prime
}$ are studied further in Walker's follow-up paper \cite{Walker94}, some of
whose results mirror ours again (in particular, the maps $\psi^{p}$ and
$\psi_{p}$ from \cite{Walker94} are our $\mathbf{f}_{p}$ and $\mathbf{v}_{p}$).

\item[\textbf{(b)}] The preprint \cite{FuMei20} by Houshan Fu and Zhousheng
Mei introduces the Petrie symmetric functions $G\left(  k,m\right)  $ and
refers to them as \emph{truncated homogeneous symmetric functions}
$h_{m}^{\left[  k-1\right]  }$. Some results below are also independently
obtained in \cite{FuMei20}. In particular, Theorem \ref{thm.G.main} is a
formula in \cite[\S 2]{FuMei20}, and Theorem \ref{thm.petk.explicit} is
equivalent to \cite[Proposition 2.9]{FuMei20}. The particular case of Theorem
\ref{thm.Gkm-genset} when $\mathbf{k}=\mathbb{Q}$ is part of \cite[Theorem
2.7]{FuMei20}.

\item[\textbf{(c)}] The paper \cite{BaAhBe18} by Bazeniar, Ahmia and Belbachir
introduces the symmetric functions $G\left(  k,m\right)  $ as well, or rather
their evaluations $\left(  G\left(  k,m\right)  \right)  \left(  x_{1}%
,x_{2},\ldots,x_{n}\right)  $ at finitely many variables; it denotes them by
$E_{m}^{\left(  k-1\right)  }\left(  n\right)  =E_{m}^{\left(  k-1\right)
}\left(  x_{1},x_{2},\ldots,x_{n}\right)  $. Ahmia and Merca continue the
study of these $E_{m}^{\left(  k-1\right)  }\left(  x_{1},x_{2},\ldots
,x_{n}\right)  $ in \cite{AhmMer20}. Our Theorem \ref{thm.G.frob} is
equivalent to the second formula in \cite[Theorem 3.3]{AhmMer20} (although we
are using infinitely many variables).

\item[\textbf{(d)}] The formal power series $G\left(  k\right)  $ also appears
in \cite[Chapter I, \S 6]{FulLan85}, under the guise of \emph{Bott's
cannibalistic class} $\theta^{j}\left(  e\right)  $ (for $j=k$ and rewritten
in the language of $\lambda$-ring operations\footnote{See \cite[\S 16.74]%
{Hazewi08} for the connection between symmetric functions (over $\mathbb{Z}$)
and universal operations on $\lambda$-rings. To be specific: If $a$ is an
element of a $\lambda$-ring $A$, then the canonical $\lambda$-ring morphism
$\Lambda_{\mathbb{Z}}\rightarrow A$ (where $\Lambda_{\mathbb{Z}}$ is the ring
of symmetric functions over $\mathbb{Z}$) that sends $e_{1}=x_{1}+x_{2}%
+x_{3}+\cdots\in\Lambda_{\mathbb{Z}}$ to $a\in A$ will send the Petrie
symmetric function $G\left(  k,m\right)  $ to the \textquotedblleft$m$-th
graded component\textquotedblright\ of Bott's cannibalistic class $\theta
^{k}\left(  a\right)  $. (Bott's cannibalistic class $\theta^{k}\left(
a\right)  $ itself is defined only if $a$ is a \textquotedblleft positive
element\textquotedblright\ in the sense of \cite{FulLan85} (or can only be
defined in an appropriate closure of $A$). When it is defined, it is the image
of the series $G\left(  k\right)  $. Otherwise, its \textquotedblleft graded
components\textquotedblright\ are the right object to consider.)}); it is used
there to prove an abstract Riemann--Roch theorem. An application to group
representations appears in \cite{AtiTal69}.
\end{enumerate}

\textbf{3.} The Petrie symmetric functions have been added to Per
Alexandersson's collection of symmetric functions at
\url{https://www.math.upenn.edu/~peal/polynomials/petrie.htm} .

\subsection*{Remark on alternative versions}

\begin{vershort}
This paper also has a detailed version \cite{verlong}, which includes some
proofs that have been omitted from the present version (mostly basic
properties of symmetric functions).
\end{vershort}

\begin{verlong}
You are reading the detailed version of this paper. For the standard version
(which is shorter by virtue of omitting some straightforward or well-known
proofs), see \cite{vershort}.
\end{verlong}

\section{\label{sect.not}Notations}

We will use the following notations (most of which are also used in
\cite[\S 2.1]{GriRei}):

\begin{itemize}
\item We let $\mathbb{N}=\left\{  0,1,2,\ldots\right\}  $.

\item The words \textquotedblleft positive\textquotedblright,
\textquotedblleft larger\textquotedblright, etc. will be used in their
Anglophone meaning (so that $0$ is neither positive nor larger than itself).

\item We fix a commutative ring $\mathbf{k}$; we will use this $\mathbf{k}$ as
the base ring in what follows.

\item A \emph{weak composition} means an infinite sequence of nonnegative
integers that contains only finitely many nonzero entries (i.e., a sequence
$\left(  \alpha_{1},\alpha_{2},\alpha_{3},\ldots\right)  \in\mathbb{N}%
^{\infty}$ such that all but finitely many $i\in\left\{  1,2,3,\ldots\right\}
$ satisfy $\alpha_{i}=0$).

\item We let $\operatorname*{WC}$ denote the set of all weak compositions.

\item For any weak composition $\alpha$ and any positive integer $i$, we let
$\alpha_{i}$ denote the $i$-th entry of $\alpha$ (so that $\alpha=\left(
\alpha_{1},\alpha_{2},\alpha_{3},\ldots\right)  $). More generally, we use
this notation whenever $\alpha$ is an infinite sequence of any kind of objects.

\item The \emph{size} $\left\vert \alpha\right\vert $ of a weak composition
$\alpha$ is defined to be $\alpha_{1}+\alpha_{2}+\alpha_{3}+\cdots
\in\mathbb{N}$.

\item A \emph{partition} means a weak composition whose entries weakly
decrease (i.e., a weak composition $\alpha$ satisfying $\alpha_{1}\geq
\alpha_{2}\geq\alpha_{3}\geq\cdots$).

\item If $n\in\mathbb{Z}$, then a \emph{partition of} $n$ means a partition
$\alpha$ having size $n$ (that is, satisfying $\left\vert \alpha\right\vert
=n$).

\item We let $\operatorname*{Par}$ denote the set of all partitions. For each
$n\in\mathbb{Z}$, we let $\operatorname*{Par}\nolimits_{n}$ denote the set of
all partitions of $n$.

\item We will sometimes omit trailing zeroes from partitions: i.e., a
partition $\lambda=\left(  \lambda_{1},\lambda_{2},\lambda_{3},\ldots\right)
$ will be identified with the $k$-tuple $\left(  \lambda_{1},\lambda
_{2},\ldots,\lambda_{k}\right)  $ whenever $k\in\mathbb{N}$ satisfies
$\lambda_{k+1}=\lambda_{k+2}=\lambda_{k+3}=\cdots=0$. For example, $\left(
3,2,1,0,0,0,\ldots\right)  =\left(  3,2,1\right)  =\left(  3,2,1,0\right)  $.

\item The partition $\left(  0,0,0,\ldots\right)  =\left(  {}\right)  $ is
called the \emph{empty partition} and denoted by $\varnothing$.

\item A \emph{part} of a partition $\lambda$ means a nonzero entry of
$\lambda$. For example, the parts of the partition $\left(  3,1,1\right)
=\left(  3,1,1,0,0,0,\ldots\right)  $ are $3,1,1$.

\item We will use the notation $1^{k}$ for \textquotedblleft%
$\underbrace{1,1,\ldots,1}_{k\text{ times}}$\textquotedblright\ in partitions.
(For example, $\left(  2,1^{4}\right)  =\left(  2,1,1,1,1\right)  $. This
notation is a particular case of the more general notation $m^{k}$ for
\textquotedblleft$\underbrace{m,m,\ldots,m}_{k\text{ times}}$%
\textquotedblright\ in partitions, used, e.g., in \cite[Definition
2.2.1]{GriRei}.)

\item We let $\Lambda$ denote the ring of symmetric functions in infinitely
many variables $x_{1},x_{2},x_{3},\ldots$ over $\mathbf{k}$. This is a subring
of the ring $\mathbf{k}\left[  \left[  x_{1},x_{2},x_{3},\ldots\right]
\right]  $ of formal power series. To be more specific, $\Lambda$ consists of
all power series in $\mathbf{k}\left[  \left[  x_{1},x_{2},x_{3}%
,\ldots\right]  \right]  $ that are symmetric (i.e., invariant under
permutations of the variables) and of bounded degree (see \cite[\S 2.1]%
{GriRei} for the precise meaning of this).

\item A \emph{monomial} shall mean a formal expression of the form
$x_{1}^{\alpha_{1}}x_{2}^{\alpha_{2}}x_{3}^{\alpha_{3}}\cdots$ with $\alpha
\in\operatorname*{WC}$. Formal power series are formal infinite $\mathbf{k}%
$-linear combinations of such monomials.

\item For any weak composition $\alpha$, we let $\mathbf{x}^{\alpha}$ denote
the monomial $x_{1}^{\alpha_{1}}x_{2}^{\alpha_{2}}x_{3}^{\alpha_{3}}\cdots$.

\item The \emph{degree} of a monomial $\mathbf{x}^{\alpha}$ is defined to be
$\left\vert \alpha\right\vert $.

\item A formal power series is said to be \emph{homogeneous of degree }$n$
(for some $n\in\mathbb{N}$) if all monomials appearing in it (with nonzero
coefficient) have degree $n$. In particular, the power series $0$ is
homogeneous of any degree.

\item If $f\in\mathbf{k}\left[  \left[  x_{1},x_{2},x_{3},\ldots\right]
\right]  $ is a power series, then there is a unique family $\left(
f_{i}\right)  _{i\in\mathbb{N}}=\left(  f_{0},f_{1},f_{2},\ldots\right)  $ of
formal power series $f_{i}\in\mathbf{k}\left[  \left[  x_{1},x_{2}%
,x_{3},\ldots\right]  \right]  $ such that each $f_{i}$ is homogeneous of
degree $i$ and such that $f=\sum_{i\in\mathbb{N}}f_{i}$. This family $\left(
f_{i}\right)  _{i\in\mathbb{N}}$ is called the \emph{homogeneous
decomposition} of $f$, and its entry $f_{i}$ (for any given $i\in\mathbb{N}$)
is called the $i$\emph{-th degree homogeneous component} of $f$.

\item The $\mathbf{k}$-algebra $\Lambda$ is graded: i.e., any symmetric
function $f$ can be uniquely written as a sum $\sum_{i\in\mathbb{N}}f_{i}$,
where each $f_{i}$ is a homogeneous symmetric function of degree $i$, and
where all but finitely many $i\in\mathbb{N}$ satisfy $f_{i}=0$.
\end{itemize}

We shall use the symmetric functions $m_{\lambda},h_{n},e_{n},p_{n}%
,s_{\lambda}$ in $\Lambda$ as defined in \cite[Sections 2.1 and 2.2]{GriRei}.
Let us briefly recall how they are defined:

\begin{vershort}
\begin{itemize}
\item For any partition $\lambda$, we define the \emph{monomial symmetric
function }$m_{\lambda}\in\Lambda$ by\footnote{This definition of $m_{\lambda}$
is not the same as the one given in \cite[Definition 2.1.3]{GriRei}; but it is
easily seen to be equivalent to the latter (i.e., it defines the same
$m_{\lambda}$). See \cite{verlong} for the details of the proof.}
\[
m_{\lambda}=\sum\mathbf{x}^{\alpha},
\]
where the sum ranges over all weak compositions $\alpha\in\operatorname*{WC}$
that can be obtained from $\lambda$ by permuting entries\footnote{Here, we
understand $\lambda$ to be an infinite sequence, not a finite tuple, so the
entries being permuted include infinitely many $0$'s.}. For example,
\[
m_{\left(  2,2,1\right)  }=\sum_{i<j<k}x_{i}^{2}x_{j}^{2}x_{k}+\sum
_{i<j<k}x_{i}^{2}x_{j}x_{k}^{2}+\sum_{i<j<k}x_{i}x_{j}^{2}x_{k}^{2}.
\]
The family $\left(  m_{\lambda}\right)  _{\lambda\in\operatorname*{Par}}$
(that is, the family of the symmetric functions $m_{\lambda}$ as $\lambda$
ranges over all partitions) is a basis of the $\mathbf{k}$-module $\Lambda$.
\end{itemize}
\end{vershort}

\begin{verlong}
\begin{itemize}
\item For any partition $\lambda$, we define the \emph{monomial symmetric
function }$m_{\lambda}\in\Lambda$ by\footnote{This definition of $m_{\lambda}$
is not the same as the one given in \cite[Definition 2.1.3]{GriRei}; but it is
easily seen to be equivalent to the latter (i.e., it defines the same
$m_{\lambda}$). See Subsection \ref{subsect.proofs.Sinf} below (and the proof
of Proposition \ref{prop.proofs.mlam-eq} in particular) for the details.}
\[
m_{\lambda}=\sum\mathbf{x}^{\alpha},
\]
where the sum ranges over all weak compositions $\alpha\in\operatorname*{WC}$
that can be obtained from $\lambda$ by permuting entries\footnote{Here, we
understand $\lambda$ to be an infinite sequence, not a finite tuple, so the
entries being permuted include infinitely many $0$'s.}. For example,
\[
m_{\left(  2,2,1\right)  }=\sum_{i<j<k}x_{i}^{2}x_{j}^{2}x_{k}+\sum
_{i<j<k}x_{i}^{2}x_{j}x_{k}^{2}+\sum_{i<j<k}x_{i}x_{j}^{2}x_{k}^{2}.
\]
The family $\left(  m_{\lambda}\right)  _{\lambda\in\operatorname*{Par}}$
(that is, the family of the symmetric functions $m_{\lambda}$ as $\lambda$
ranges over all partitions) is a basis of the $\mathbf{k}$-module $\Lambda$.
\end{itemize}
\end{verlong}

\begin{itemize}
\item For each $n\in\mathbb{Z}$, we define the \emph{complete homogeneous
symmetric function }$h_{n}\in\Lambda$ by%
\[
h_{n}=\sum_{i_{1}\leq i_{2}\leq\cdots\leq i_{n}}x_{i_{1}}x_{i_{2}}\cdots
x_{i_{n}}=\sum_{\substack{\alpha\in\operatorname*{WC};\\\left\vert
\alpha\right\vert =n}}\mathbf{x}^{\alpha}=\sum_{\lambda\in\operatorname*{Par}%
\nolimits_{n}}m_{\lambda}.
\]
Thus, $h_{0}=1$ and $h_{n}=0$ for all $n<0$.

We know (e.g., from \cite[Proposition 2.4.1]{GriRei}) that the family $\left(
h_{n}\right)  _{n\geq1}=\left(  h_{1},h_{2},h_{3},\ldots\right)  $ is
algebraically independent and generates $\Lambda$ as a $\mathbf{k}$-algebra.
In other words, $\Lambda$ is freely generated by $h_{1},h_{2},h_{3},\ldots$ as
a commutative $\mathbf{k}$-algebra.

\item For each $n\in\mathbb{Z}$, we define the \emph{elementary symmetric
function }$e_{n}\in\Lambda$ by%
\[
e_{n}=\sum_{i_{1}<i_{2}<\cdots<i_{n}}x_{i_{1}}x_{i_{2}}\cdots x_{i_{n}}%
=\sum_{\substack{\alpha\in\operatorname*{WC}\cap\left\{  0,1\right\}
^{\infty};\\\left\vert \alpha\right\vert =n}}\mathbf{x}^{\alpha}.
\]
Thus, $e_{0}=1$ and $e_{n}=0$ for all $n<0$. If $n\geq0$, then $e_{n}%
=m_{\left(  1^{n}\right)  }$, where, as we have agreed above, the notation
$\left(  1^{n}\right)  $ stands for the $n$-tuple $\left(  1,1,\ldots
,1\right)  $.

\item For each positive integer $n$, we define the \emph{power-sum symmetric
function} $p_{n}\in\Lambda$ by%
\[
p_{n}=x_{1}^{n}+x_{2}^{n}+x_{3}^{n}+\cdots=m_{\left(  n\right)  }.
\]

\item For each partition $\lambda$, we define the \emph{Schur function
}$s_{\lambda}\in\Lambda$ by%
\[
s_{\lambda}=\sum\mathbf{x}_{T},
\]
where the sum ranges over all semistandard tableaux $T$ of shape $\lambda$,
and where $\mathbf{x}_{T}$ denotes the monomial obtained by multiplying the
$x_{i}$ for all entries $i$ of $T$. We refer the reader to \cite[Definition
2.2.1]{GriRei} or to \cite[\S 7.10]{Stanley-EC2} for the details of this
definition and further descriptions of the Schur functions. One of the most
important properties of Schur functions (see, e.g., \cite[(2.4.16) for
$\mu=\varnothing$]{GriRei} or \cite[Theorem 2.32]{MenRem15} or \cite[Theorem
7.16.1 for $\mu=\varnothing$]{Stanley-EC2} or \cite[Theorem 7.2.3
(a)]{Sagan20}) is the fact that%
\begin{equation}
s_{\lambda}=\det\left(  \left(  h_{\lambda_{i}-i+j}\right)  _{1\leq i\leq
\ell,\ 1\leq j\leq\ell}\right)  \label{eq.schur.JTsh}%
\end{equation}
for any partition $\lambda=\left(  \lambda_{1},\lambda_{2},\ldots
,\lambda_{\ell}\right)  $. This is known as the \emph{(first, straight-shape)
Jacobi--Trudi formula}.

The family $\left(  s_{\lambda}\right)  _{\lambda\in\operatorname*{Par}}$ is a
basis of the $\mathbf{k}$-module $\Lambda$, and is known as the \emph{Schur
basis}. It is easy to see that each $n\in\mathbb{N}$ satisfies $s_{\left(
n\right)  }=h_{n}$ and $s_{\left(  1^{n}\right)  }=e_{n}$. Moreover, for each
partition $\lambda$, the Schur function $s_{\lambda}\in\Lambda$ is homogeneous
of degree $\left\vert \lambda\right\vert $.
\end{itemize}

Among the many relations between these symmetric functions is an expression
for the power-sum symmetric function $p_{n}$ in terms of the Schur basis:

\begin{proposition}
\label{prop.p-as-sum}Let $n$ be a positive integer. Then,%
\[
p_{n}=\sum_{i=0}^{n-1}\left(  -1\right)  ^{i}s_{\left(  n-i,1^{i}\right)  }.
\]

\end{proposition}

\begin{proof}
This is a classical formula, and appears (e.g.) in \cite[Problem 4.21]%
{Egge19}, \cite[Exercise 5.4.12(g)]{GriRei} and \cite[Exercise 2.2]{MenRem15}.
Alternatively, this is an easy consequence of the Murnaghan--Nakayama rule
(see \cite[Theorem 6.3]{MenRem15} or \cite[Theorem 4.4.2]{Sam17} or
\cite[Theorem 7.17.3]{Stanley-EC2} or \cite[(1)]{Wildon15}), applied to the
product $p_{n}s_{\varnothing}$ (since $s_{\varnothing}=1$).
\end{proof}

Finally, we will sometimes use the \emph{Hall inner product} $\left\langle
\cdot,\cdot\right\rangle :\Lambda\times\Lambda\rightarrow\mathbf{k}$ as
defined in \cite[Definition 2.5.12]{GriRei}.\footnote{However, it is denoted
by $\left(  \cdot,\cdot\right)  $ rather than by $\left\langle \cdot
,\cdot\right\rangle $ in \cite{GriRei}. (That is, what we call $\left\langle
a,b\right\rangle $ is denoted by $\left(  a,b\right)  $ in \cite{GriRei}.)
\par
The Hall inner product also appears (for $\mathbf{k}=\mathbb{Z}$ and
$\mathbf{k}=\mathbb{Q}$) in \cite[Definition 7.5]{Egge19}, in \cite[\S 7.9]%
{Stanley-EC2} and in \cite[Section I.4]{Macdon95} (note that it is called the
\textquotedblleft scalar product\textquotedblright\ in the latter two
references). The definitions of the Hall inner product in \cite[\S 7.9]%
{Stanley-EC2} and in \cite[Section I.4]{Macdon95} are different from ours, but
they are equivalent to ours (because of \cite[Corollary 7.12.2]{Stanley-EC2}
and \cite[Chapter I, (4.8)]{Macdon95}).} This is the $\mathbf{k}$-bilinear
form on $\Lambda$ that is defined by the requirement that
\[
\left\langle s_{\lambda},s_{\mu}\right\rangle =\delta_{\lambda,\mu
}\ \ \ \ \ \ \ \ \ \ \text{for any }\lambda,\mu\in\operatorname*{Par}%
\]
(where $\delta_{\lambda,\mu}$ denotes the Kronecker delta). Thus, the Schur
basis $\left(  s_{\lambda}\right)  _{\lambda\in\operatorname*{Par}}$ of
$\Lambda$ is an orthonormal basis with respect to the Hall inner product. It
is easy to see\footnote{See, e.g., \cite[Exercise 2.5.13(a)]{GriRei} for a
proof.} that the Hall inner product $\left(  \cdot,\cdot\right)  $ is graded:
i.e., we have%
\begin{equation}
\left\langle f,g\right\rangle =0 \label{eq.Hall.graded}%
\end{equation}
if $f$ and $g$ are two homogeneous symmetric functions of different degrees.
We shall also use the following two known evaluations of the Hall inner product:

\begin{proposition}
\label{prop.hjpj}Let $n$ be a positive integer. Then, $\left\langle
h_{n},p_{n}\right\rangle =1$.
\end{proposition}

\begin{proposition}
\label{prop.ejpj}Let $n$ be a positive integer. Then, $\left\langle
e_{n},p_{n}\right\rangle =\left(  -1\right)  ^{n-1}$.
\end{proposition}

See Subsection \ref{subsect.proofs.ejpj} for the proofs of these two propositions.

\section{\label{sect.thms}Theorems}

\subsection{\label{subsect.petk.def}Definitions}

The main role in this paper is played by two power series that we will now define:

\begin{definition}
\label{def.G}\textbf{(a)} For any positive integer $k$, we let\footnotemark%
\begin{equation}
G\left(  k\right)  =\sum_{\substack{\alpha\in\operatorname*{WC};\\\alpha
_{i}<k\text{ for all }i}}\mathbf{x}^{\alpha}. \label{eq.Gk=}%
\end{equation}
This is a symmetric formal power series in $\mathbf{k}\left[  \left[
x_{1},x_{2},x_{3},\ldots\right]  \right]  $ (but does not belong to $\Lambda$
in general).

\textbf{(b)} For any positive integer $k$ and any $m\in\mathbb{N}$, we let%
\begin{equation}
G\left(  k,m\right)  =\sum_{\substack{\alpha\in\operatorname*{WC};\\\left\vert
\alpha\right\vert =m;\\\alpha_{i}<k\text{ for all }i}}\mathbf{x}^{\alpha}%
\in\Lambda. \label{eq.Gkm=}%
\end{equation}

\end{definition}

\footnotetext{Here and in all similar situations, \textquotedblleft for all
$i$\textquotedblright\ means \textquotedblleft for all positive integers
$i$\textquotedblright.}

\begin{example}
\textbf{(a)} We have%
\begin{align*}
G\left(  2\right)   &  =\sum_{\substack{\alpha\in\operatorname*{WC}%
;\\\alpha_{i}<2\text{ for all }i}}\mathbf{x}^{\alpha}\\
&  =1+x_{1}+x_{2}+x_{3}+\cdots+x_{1}x_{2}+x_{1}x_{3}+x_{2}x_{3}+\cdots\\
&  \ \ \ \ \ \ \ \ \ \ +x_{1}x_{2}x_{3}+x_{1}x_{2}x_{4}+x_{2}x_{3}x_{4}%
+\cdots\\
&  \ \ \ \ \ \ \ \ \ \ +\cdots\\
&  =\sum_{m\in\mathbb{N}}\ \ \underbrace{\sum_{1\leq i_{1}<i_{2}<\cdots<i_{m}%
}x_{i_{1}}x_{i_{2}}\cdots x_{i_{m}}}_{=e_{m}}=\sum_{m\in\mathbb{N}}e_{m}.
\end{align*}

\textbf{(b)} For each $m\in\mathbb{N}$, we have%
\[
G\left(  2,m\right)  =\sum_{\substack{\alpha\in\operatorname*{WC};\\\left\vert
\alpha\right\vert =m;\\\alpha_{i}<2\text{ for all }i}}\mathbf{x}^{\alpha}%
=\sum_{1\leq i_{1}<i_{2}<\cdots<i_{m}}x_{i_{1}}x_{i_{2}}\cdots x_{i_{m}}%
=e_{m}.
\]

\end{example}

We suggest the name $k$\emph{-Petrie symmetric series} for $G\left(  k\right)
$ and the name $\left(  k,m\right)  $\emph{-Petrie symmetric function} for
$G\left(  k,m\right)  $. The reason for this naming is that the coefficients
of these functions in the Schur basis of $\Lambda$ are determinants of Petrie
matrices, as we will see in Subsection \ref{subsect.proofs.petk.-101}.

\subsection{\label{subsect.petk.basics}Basic identities}

We begin our study of the $G\left(  k\right)  $ and $G\left(  k,m\right)  $
with some simple properties:

\begin{proposition}
\label{prop.G.basics}Let $k$ be a positive integer.

\textbf{(a)} The symmetric function $G\left(  k,m\right)  $ is the $m$-th
degree homogeneous component of $G\left(  k\right)  $ for each $m\in
\mathbb{N}$.

\textbf{(b)} We have%
\[
G\left(  k\right)  =\sum_{\substack{\alpha\in\operatorname*{WC};\\\alpha
_{i}<k\text{ for all }i}}\mathbf{x}^{\alpha}=\sum_{\substack{\lambda
\in\operatorname*{Par};\\\lambda_{i}<k\text{ for all }i}}m_{\lambda}%
=\prod_{i=1}^{\infty}\left(  x_{i}^{0}+x_{i}^{1}+\cdots+x_{i}^{k-1}\right)  .
\]

\textbf{(c)} We have
\[
G\left(  k,m\right)  =\sum_{\substack{\alpha\in\operatorname*{WC};\\\left\vert
\alpha\right\vert =m;\\\alpha_{i}<k\text{ for all }i}}\mathbf{x}^{\alpha}%
=\sum_{\substack{\lambda\in\operatorname*{Par};\\\left\vert \lambda\right\vert
=m;\\\lambda_{i}<k\text{ for all }i}}m_{\lambda}%
\]
for each $m\in\mathbb{N}$.

\textbf{(d)} If $m\in\mathbb{N}$ satisfies $k>m$, then $G\left(  k,m\right)
=h_{m}$.

\textbf{(e)} If $m\in\mathbb{N}$ and $k=2$, then $G\left(  k,m\right)  =e_{m}$.

\textbf{(f)} If $m=k$, then $G\left(  k,m\right)  =h_{m}-p_{m}$.
\end{proposition}

\begin{vershort}
Proving Proposition \ref{prop.G.basics} makes good practice in understanding
the definitions of $m_{\lambda}$, $h_{n}$, $e_{n}$, $p_{n}$, $G\left(
k\right)  $ and $G\left(  k,n\right)  $. We omit the proof here; it can be
found in full (hardly necessary) detail in \cite{verlong}.
\end{vershort}

\begin{verlong}
We shall prove Proposition \ref{prop.G.basics} in Subsection
\ref{subsect.proofs.basics} below. (The easy proof is good practice in
understanding the definitions of $m_{\lambda}$, $h_{n}$, $e_{n}$, $p_{n}$,
$G\left(  k\right)  $ and $G\left(  k,n\right)  $.)
\end{verlong}

Parts \textbf{(d)} and \textbf{(e)} of Proposition \ref{prop.G.basics} show
that the Petrie symmetric functions $G\left(  k,m\right)  $ can be seen as
interpolating between the $h_{m}$ and the $e_{m}$.

\subsection{\label{subsect.petk.schur}The Schur expansion}

The solution to \cite[Exercise 7.3]{Stanley-EC2} gives an expansion of
$G\left(  3\right)  $ in terms of the elementary symmetric functions (due to
I. M. Gessel); this expansion can be rewritten as%
\[
G\left(  3\right)  =\sum_{n\in\mathbb{N}}e_{n}^{2}+\sum_{m<n}c_{m,n}e_{m}%
e_{n},\ \ \ \ \ \ \ \ \ \ \text{where }c_{m,n}=\left(  -1\right)  ^{m-n}%
\begin{cases}
2, & \text{if }3\mid m-n;\\
-1, & \text{if }3\nmid m-n
\end{cases}
\ \ \ .
\]
We shall instead expand $G\left(  k\right)  $ in terms of Schur functions. For
this, we need to define some notations.

\begin{convention}
\label{conv.iverson}We shall use the \emph{Iverson bracket notation}: i.e., if
$\mathcal{A}$ is a logical statement, then $\left[  \mathcal{A}\right]  $
shall denote the truth value of $\mathcal{A}$ (that is, the integer $%
\begin{cases}
1, & \text{if }\mathcal{A}\text{ is true;}\\
0, & \text{if }\mathcal{A}\text{ is false}%
\end{cases}
$\ \ ).
\end{convention}

We shall furthermore use the notation $\left(  a_{i,j}\right)  _{1\leq
i\leq\ell,\ 1\leq j\leq\ell}$ for the $\ell\times\ell$-matrix whose $\left(
i,j\right)  $-th entry is $a_{i,j}$ for each $i,j\in\left\{  1,2,\ldots
,\ell\right\}  $.

\begin{definition}
\label{def.petklam} Let $\lambda=\left(  \lambda_{1},\lambda_{2}%
,\ldots,\lambda_{\ell}\right)  \in\operatorname*{Par}$ and $\mu=\left(
\mu_{1},\mu_{2},\ldots,\mu_{\ell}\right)  \in\operatorname*{Par}$, and let $k$
be a positive integer. Then, the $k$\emph{-Petrie number} $\operatorname*{pet}%
\nolimits_{k}\left(  \lambda,\mu\right)  $ of $\lambda$ and $\mu$ is the
integer defined by%
\[
\operatorname*{pet}\nolimits_{k}\left(  \lambda,\mu\right)  =\det\left(
\left(  \left[  0\leq\lambda_{i}-\mu_{j}-i+j<k\right]  \right)  _{1\leq
i\leq\ell,\ 1\leq j\leq\ell}\right)  .
\]
Note that this integer does not depend on the choice of $\ell$ (in the sense
that it does not change if we enlarge $\ell$ by adding trailing zeroes to the
representations of $\lambda$ and $\mu$); this follows from Lemma
\ref{lem.petkrel.indep} below.
\end{definition}

\begin{example}
Let $\lambda$ be the partition $\left(  3,2,1\right)  \in\operatorname*{Par}$,
let $\mu$ be the partition $\left(  1,1\right)  \in\operatorname*{Par}$, let
$\ell=3$, and let $k$ be a positive integer. Then, the definition of
$\operatorname*{pet}\nolimits_{k}\left(  \lambda,\mu\right)  $ yields%
\begin{align*}
&  \operatorname*{pet}\nolimits_{k}\left(  \lambda,\mu\right) \\
&  =\det\left(  \left(  \left[  0\leq\lambda_{i}-\mu_{j}-i+j<k\right]
\right)  _{1\leq i\leq\ell,\ 1\leq j\leq\ell}\right) \\
&  =\det\left(
\begin{array}
[c]{ccc}%
\left[  0\leq\lambda_{1}-\mu_{1}<k\right]  & \left[  0\leq\lambda_{1}-\mu
_{2}+1<k\right]  & \left[  0\leq\lambda_{1}-\mu_{3}+2<k\right] \\
\left[  0\leq\lambda_{2}-\mu_{1}-1<k\right]  & \left[  0\leq\lambda_{2}%
-\mu_{2}<k\right]  & \left[  0\leq\lambda_{2}-\mu_{3}+1<k\right] \\
\left[  0\leq\lambda_{3}-\mu_{1}-2<k\right]  & \left[  0\leq\lambda_{3}%
-\mu_{2}-1<k\right]  & \left[  0\leq\lambda_{3}-\mu_{3}<k\right]
\end{array}
\right) \\
&  =\det\left(
\begin{array}
[c]{ccc}%
\left[  0\leq3-1<k\right]  & \left[  0\leq3-1+1<k\right]  & \left[
0\leq3-0+2<k\right] \\
\left[  0\leq2-1-1<k\right]  & \left[  0\leq2-1<k\right]  & \left[
0\leq2-0+1<k\right] \\
\left[  0\leq1-1-2<k\right]  & \left[  0\leq1-1-1<k\right]  & \left[
0\leq1-0<k\right]
\end{array}
\right) \\
&  \ \ \ \ \ \ \ \ \ \ \ \ \ \ \ \ \ \ \ \ \left(
\begin{array}
[c]{c}%
\text{since }\lambda_{1}=3\text{ and }\lambda_{2}=2\text{ and }\lambda_{3}=1\\
\text{and }\mu_{1}=1\text{ and }\mu_{2}=1\text{ and }\mu_{3}=0
\end{array}
\right) \\
&  =\det\left(
\begin{array}
[c]{ccc}%
\left[  0\leq2<k\right]  & \left[  0\leq3<k\right]  & \left[  0\leq5<k\right]
\\
\left[  0\leq0<k\right]  & \left[  0\leq1<k\right]  & \left[  0\leq3<k\right]
\\
\left[  0\leq-2<k\right]  & \left[  0\leq-1<k\right]  & \left[  0\leq
1<k\right]
\end{array}
\right)  .
\end{align*}
Thus, taking $k=4$, we obtain%
\begin{align*}
\operatorname*{pet}\nolimits_{4}\left(  \lambda,\mu\right)   &  =\det\left(
\begin{array}
[c]{ccc}%
\left[  0\leq2<4\right]  & \left[  0\leq3<4\right]  & \left[  0\leq5<4\right]
\\
\left[  0\leq0<4\right]  & \left[  0\leq1<4\right]  & \left[  0\leq3<4\right]
\\
\left[  0\leq-2<4\right]  & \left[  0\leq-1<4\right]  & \left[  0\leq
1<4\right]
\end{array}
\right) \\
&  =\det\left(
\begin{array}
[c]{ccc}%
1 & 1 & 0\\
1 & 1 & 1\\
0 & 0 & 1
\end{array}
\right)  =0.
\end{align*}
On the other hand, taking $k=3$, we obtain%
\begin{align*}
\operatorname*{pet}\nolimits_{3}\left(  \lambda,\mu\right)   &  =\det\left(
\begin{array}
[c]{ccc}%
\left[  0\leq2<3\right]  & \left[  0\leq3<3\right]  & \left[  0\leq5<3\right]
\\
\left[  0\leq0<3\right]  & \left[  0\leq1<3\right]  & \left[  0\leq3<3\right]
\\
\left[  0\leq-2<3\right]  & \left[  0\leq-1<3\right]  & \left[  0\leq
1<3\right]
\end{array}
\right) \\
&  =\det\left(
\begin{array}
[c]{ccc}%
1 & 0 & 0\\
1 & 1 & 0\\
0 & 0 & 1
\end{array}
\right)  =1.
\end{align*}

\end{example}

\begin{lemma}
\label{lem.petkrel.indep}Let $\lambda\in\operatorname*{Par}$ and $\mu
\in\operatorname*{Par}$, and let $k$ be a positive integer. Let $\ell
\in\mathbb{N}$ be such that $\lambda=\left(  \lambda_{1},\lambda_{2}%
,\ldots,\lambda_{\ell}\right)  $ and $\mu=\left(  \mu_{1},\mu_{2},\ldots
,\mu_{\ell}\right)  $. Then, the determinant $\det\left(  \left(  \left[
0\leq\lambda_{i}-\mu_{j}-i+j<k\right]  \right)  _{1\leq i\leq\ell,\ 1\leq
j\leq\ell}\right)  $ does not depend on the choice of $\ell$.
\end{lemma}

See Subsection \ref{subsect.proofs.petkrel.indep} for the simple proof of
Lemma \ref{lem.petkrel.indep}.

Surprisingly, the $k$-Petrie numbers $\operatorname*{pet}\nolimits_{k}\left(
\lambda,\mu\right)  $ can take only three possible values:

\begin{proposition}
\label{prop.petkrel.-101}Let $\lambda\in\operatorname*{Par}$ and $\mu
\in\operatorname*{Par}$, and let $k$ be a positive integer. Then,
$\operatorname*{pet}\nolimits_{k}\left(  \lambda,\mu\right)  \in\left\{
-1,0,1\right\}  $.
\end{proposition}

Proposition \ref{prop.petkrel.-101} will be proved in Subsection
\ref{subsect.proofs.petk.-101}.

We can now expand the Petrie symmetric functions $G\left(  k,m\right)  $ and
the power series $G\left(  k\right)  $ in the basis $\left(  s_{\lambda
}\right)  _{\lambda\in\operatorname*{Par}}$ of $\Lambda$:

\begin{theorem}
\label{thm.G.main}Let $k$ be a positive integer. Then,%
\[
G\left(  k\right)  =\sum_{\lambda\in\operatorname*{Par}}\operatorname*{pet}%
\nolimits_{k}\left(  \lambda,\varnothing\right)  s_{\lambda}.
\]
(Recall that $\varnothing$ denotes the empty partition $\left(  {}\right)
=\left(  0,0,0,\ldots\right)  $.)
\end{theorem}

We will not prove Theorem \ref{thm.G.main} directly; instead, we will first
show a stronger result (Theorem \ref{thm.G.pieri}), and then derive Theorem
\ref{thm.G.main} from it in Subsection \ref{subsect.proofs.G.cors}.

\begin{corollary}
\label{cor.Gkm.main}Let $k$ be a positive integer. Let $m\in\mathbb{N}$. Then,%
\[
G\left(  k,m\right)  =\sum_{\lambda\in\operatorname*{Par}\nolimits_{m}%
}\operatorname*{pet}\nolimits_{k}\left(  \lambda,\varnothing\right)
s_{\lambda}.
\]

\end{corollary}

Corollary \ref{cor.Gkm.main} easily follows from Theorem \ref{thm.G.main}
using Proposition \ref{prop.G.basics} \textbf{(a)}; but again, we shall
instead derive it from a stronger result (Corollary \ref{cor.Gkm.pieri}) in
Subsection \ref{subsect.proofs.G.cors}.

We will see a more explicit description of the $k$-Petrie numbers
$\operatorname*{pet}\nolimits_{k}\left(  \lambda,\varnothing\right)  $ in
Subsection \ref{subsect.petk.mu}.

\begin{remark}
Corollary \ref{cor.Gkm.main}, in combination with Proposition
\ref{prop.petkrel.-101}, shows that each $k$-Petrie function $G\left(
k,m\right)  $ (for any $k>0$ and $m\in\mathbb{N}$) is a linear combination of
Schur functions, with all coefficients belonging to $\left\{  -1,0,1\right\}
$. It is natural to expect the more general symmetric functions
\[
\widetilde{G}\left(  k,k^{\prime},m\right)  =\sum_{\substack{\alpha
\in\operatorname*{WC};\\\left\vert \alpha\right\vert =m;\\k^{\prime}\leq
\alpha_{i}<k\text{ for all }i}}\mathbf{x}^{\alpha}%
,\ \ \ \ \ \ \ \ \ \ \text{where }0<k^{\prime}\leq k,
\]
to have the same property. However, this is not the case. For example,%
\[
\widetilde{G}\left(  4,2,5\right)  =m_{\left(  3,2\right)  }=-2s_{\left(
1,1,1,1,1\right)  }+2s_{\left(  2,1,1,1\right)  }-s_{\left(  2,2,1\right)
}-s_{\left(  3,1,1\right)  }+s_{\left(  3,2\right)  }.
\]

\end{remark}

\subsection{\label{subsect.petk.mu}An explicit description of the $k$-Petrie
numbers $\operatorname*{pet}\nolimits_{k}\left(  \lambda,\varnothing\right)
$}

Can the $k$-Petrie numbers $\operatorname*{pet}\nolimits_{k}\left(
\lambda,\varnothing\right)  $ from Definition \ref{def.petklam} be described
more explicitly than as determinants? To be somewhat pedantic, the answer to
this question depends on one's notion of \textquotedblleft
explicit\textquotedblright, as determinants are not hard to compute, and
another algorithm for calculating $\operatorname*{pet}\nolimits_{k}\left(
\lambda,\varnothing\right)  $ can be extracted from our proof of Proposition
\ref{prop.petkrel.-101} (when combined with \cite[proof of Theorem
1]{GorWil74}). Nevertheless, there is a more explicit description. This
description will be stated in Theorem \ref{thm.petk.explicit} further below.

First, let us get a simple case out of the way:

\begin{proposition}
\label{prop.petk.0}Let $\lambda\in\operatorname*{Par}$, and let $k$ be a
positive integer such that $\lambda_{1}\geq k$. Then, $\operatorname*{pet}%
\nolimits_{k}\left(  \lambda,\varnothing\right)  =0$.
\end{proposition}

\begin{proof}
[Proof of Proposition \ref{prop.petk.0}.]Write $\lambda$ as $\lambda=\left(
\lambda_{1},\lambda_{2},\ldots,\lambda_{\ell}\right)  $. Thus, $\ell\geq1$
(since $\lambda_{1}\geq k>0$). Moreover, the empty partition $\varnothing$ can
be written as $\varnothing=\left(  \varnothing_{1},\varnothing_{2}%
,\ldots,\varnothing_{\ell}\right)  $ (since $\varnothing_{i}=0$ for each
integer $i>\ell$).

Thus, we have $\lambda=\left(  \lambda_{1},\lambda_{2},\ldots,\lambda_{\ell
}\right)  $ and $\varnothing=\left(  \varnothing_{1},\varnothing_{2}%
,\ldots,\varnothing_{\ell}\right)  $. Hence, the definition of
$\operatorname*{pet}\nolimits_{k}\left(  \lambda,\varnothing\right)  $ yields
\begin{align}
\operatorname*{pet}\nolimits_{k}\left(  \lambda,\varnothing\right)   &
=\det\left(  \left(  \left[  0\leq\lambda_{i}-\underbrace{\varnothing_{j}%
}_{=0}-i+j<k\right]  \right)  _{1\leq i\leq\ell,\ 1\leq j\leq\ell}\right)
\nonumber\\
&  =\det\left(  \left(  \left[  0\leq\lambda_{i}-i+j<k\right]  \right)
_{1\leq i\leq\ell,\ 1\leq j\leq\ell}\right)  .
\label{pf.prop.petk.0.petklam0=}%
\end{align}

But each $j\in\left\{  1,2,\ldots,\ell\right\}  $ satisfies $\left[
0\leq\lambda_{1}-1+j<k\right]  =0$ (since $\lambda_{1}-1+\underbrace{j}%
_{\geq1}\geq\lambda_{1}-1+1=\lambda_{1}\geq k$). In other words, the
$\ell\times\ell$-matrix $\left(  \left[  0\leq\lambda_{i}-i+j<k\right]
\right)  _{1\leq i\leq\ell,\ 1\leq j\leq\ell}$ has first row $\left(
0,0,\ldots,0\right)  $. Therefore, its determinant is $0$. In other words,
$\operatorname*{pet}\nolimits_{k}\left(  \lambda,\varnothing\right)  =0$
(since $\operatorname*{pet}\nolimits_{k}\left(  \lambda,\varnothing\right)  $
is its determinant\footnote{by (\ref{pf.prop.petk.0.petklam0=})}). This proves
Proposition \ref{prop.petk.0}.
\end{proof}

Stating Theorem \ref{thm.petk.explicit} will require some notation:

\begin{definition}
\label{def.transpose}For any $\lambda\in\operatorname*{Par}$, we define the
\emph{transpose} of $\lambda$ to be the partition $\lambda^{t}\in
\operatorname*{Par}$ determined by%
\[
\left(  \lambda^{t}\right)  _{i}=\left\vert \left\{  j\in\left\{
1,2,3,\ldots\right\}  \ \mid\ \lambda_{j}\geq i\right\}  \right\vert
\ \ \ \ \ \ \ \ \ \ \text{for each }i\geq1.
\]
This partition $\lambda^{t}$ is also known as the \emph{conjugate} of
$\lambda$, and is perhaps easiest to understand in terms of Young diagrams: To
wit, the Young diagram of $\lambda^{t}$ is obtained from that of $\lambda$ by
a flip across the main diagonal.
\end{definition}

One important use of transpose partitions is the following fact (see, e.g.,
\cite[(2.4.17) for $\mu=\varnothing$]{GriRei} or \cite[Theorem 2.32]{MenRem15}
or \cite[Theorem 7.16.2 applied to $\lambda^{t}$ and $\varnothing$ instead of
$\lambda$ and $\mu$]{Stanley-EC2} for proofs): We have%
\begin{equation}
s_{\lambda^{t}}=\det\left(  \left(  e_{\lambda_{i}-i+j}\right)  _{1\leq
i\leq\ell,\ 1\leq j\leq\ell}\right)  \label{eq.schur.JTsh-e}%
\end{equation}
for any partition $\lambda=\left(  \lambda_{1},\lambda_{2},\ldots
,\lambda_{\ell}\right)  $. This is known as the \emph{(second, straight-shape)
Jacobi--Trudi formula}.

We will use the following notation for quotients and remainders:

\begin{convention}
Let $k$ be a positive integer. Let $n\in\mathbb{Z}$. Then, $n\%k$ shall denote
the remainder of $n$ divided by $k$, whereas $n//k$ shall denote the quotient
of this division (an integer). Thus, $n//k$ and $n\%k$ are uniquely determined
by the three requirements that $n//k\in\mathbb{Z}$ and $n\%k\in\left\{
0,1,\ldots,k-1\right\}  $ and $n=\left(  n//k\right)  \cdot k+\left(
n\%k\right)  $.

The \textquotedblleft$//$\textquotedblright\ and \textquotedblleft%
$\%$\textquotedblright\ signs bind more strongly than the \textquotedblleft%
$+$\textquotedblright\ and \textquotedblleft$-$\textquotedblright\ signs. That
is, for example, the expression \textquotedblleft$a+b\%k$\textquotedblright%
\ shall be understood to mean \textquotedblleft$a+\left(  b\%k\right)
$\textquotedblright\ rather than \textquotedblleft$\left(  a+b\right)
\%k$\textquotedblright.
\end{convention}

Now, we can state our \textquotedblleft formula\textquotedblright\ for
$k$-Petrie numbers of the form $\operatorname*{pet}\nolimits_{k}\left(
\lambda,\varnothing\right)  $.

\begin{theorem}
\label{thm.petk.explicit}Let $\lambda\in\operatorname*{Par}$, and let $k$ be a
positive integer. Let $\mu=\lambda^{t}$.

\textbf{(a)} If $\mu_{k}\neq0$, then $\operatorname*{pet}\nolimits_{k}\left(
\lambda,\varnothing\right)  =0$.

From now on, let us assume that $\mu_{k}=0$.

Define a $\left(  k-1\right)  $-tuple $\left(  \beta_{1},\beta_{2}%
,\ldots,\beta_{k-1}\right)  \in\mathbb{Z}^{k-1}$ by setting%
\begin{equation}
\beta_{i}=\mu_{i}-i\ \ \ \ \ \ \ \ \ \ \text{for each }i\in\left\{
1,2,\ldots,k-1\right\}  . \label{eq.thm.petk.explicit.beti=}%
\end{equation}

Define a $\left(  k-1\right)  $-tuple $\left(  \gamma_{1},\gamma_{2}%
,\ldots,\gamma_{k-1}\right)  \in\left\{  1,2,\ldots,k\right\}  ^{k-1}$ by
setting%
\begin{equation}
\gamma_{i}=1+\left(  \beta_{i}-1\right)  \%k\ \ \ \ \ \ \ \ \ \ \text{for each
}i\in\left\{  1,2,\ldots,k-1\right\}  . \label{eq.thm.petk.explicit.gammi=}%
\end{equation}

\textbf{(b)} If the $k-1$ numbers $\gamma_{1},\gamma_{2},\ldots,\gamma_{k-1}$
are not distinct, then $\operatorname*{pet}\nolimits_{k}\left(  \lambda
,\varnothing\right)  =0$.

\textbf{(c)} Assume that the $k-1$ numbers $\gamma_{1},\gamma_{2}%
,\ldots,\gamma_{k-1}$ are distinct. Let
\[
g=\left\vert \left\{  \left(  i,j\right)  \in\left\{  1,2,\ldots,k-1\right\}
^{2}\ \mid\ i<j\text{ and }\gamma_{i}<\gamma_{j}\right\}  \right\vert .
\]
Then, $\operatorname*{pet}\nolimits_{k}\left(  \lambda,\varnothing\right)
=\left(  -1\right)  ^{\left(  \beta_{1}+\beta_{2}+\cdots+\beta_{k-1}\right)
+g+\left(  \gamma_{1}+\gamma_{2}+\cdots+\gamma_{k-1}\right)  }$.
\end{theorem}

The proof of this theorem is technical and will be given in Subsection
\ref{subsect.proofs.petk.explicit}.

It is possible to restate part of Theorem \ref{thm.petk.explicit} without
using $\lambda^{t}$:

\begin{proposition}
\label{prop.petk.explicit-old}Let $\lambda\in\operatorname*{Par}$, and let $k$
be a positive integer. Assume that $\lambda_{1}<k$. Define a subset $B$ of
$\mathbb{Z}$ by
\[
B=\left\{  \lambda_{i}-i\ \mid\ i\in\left\{  1,2,3,\ldots\right\}  \right\}
.
\]

Let $\overline{0},\overline{1},\ldots,\overline{k-1}$ be the residue classes
of the integers $0,1,\ldots,k-1$ modulo $k$ (considered as subsets of
$\mathbb{Z}$). Let $W$ be the set of all integers smaller than $k-1$.

Then, $\operatorname*{pet}\nolimits_{k}\left(  \lambda,\varnothing\right)
\neq0$ if and only if each $i\in\left\{  0,1,\ldots,k-1\right\}  $ satisfies
$\left\vert \left(  \overline{i}\cap W\right)  \setminus B\right\vert \leq1$.
\end{proposition}

In Subsection \ref{subsect.proofs.petk.explicit}, we will outline how this
proposition can be derived from Theorem \ref{thm.petk.explicit}.

The sets $B$ and $\left(  \overline{i}\cap W\right)  \setminus B$ in
Proposition \ref{prop.petk.explicit-old} are related to the $k$%
\textit{-modular structure} of the partition $\lambda$, such as the $\beta
$\textit{-set}, the $k$\textit{-abacus}, the $k$\textit{-core} and the
$k$\textit{-quotient} (see \cite[\S \S 1--3]{Olsson93} for some of these
concepts). Essentially equivalent concepts include the \textit{Maya diagram}
of $\lambda$ (see, e.g., \cite[\S 3.3]{Crane18})\footnote{The \emph{Maya
diagram} of $\lambda$ is a coloring of the set $\left\{  z+\dfrac{1}{2}%
\ \mid\ z\in\mathbb{Z}\right\}  $ with the colors black and white, in which
the elements $\lambda_{i}-i+\dfrac{1}{2}$ (for all $i\in\left\{
1,2,3,\ldots\right\}  $) are colored black while all remaining elements are
colored white. Borcherds's proof of the Jacobi triple product identity
(\cite[\S 13.3]{Camero94}) also essentially constructs this Maya diagram
(wording it in terms of the \textquotedblleft\textit{Dirac sea}%
\textquotedblright\ model for electrons).} and the \textit{first column hook
lengths} of $\lambda$ (see \cite[Proposition (1.3)]{Olsson93}).

\subsection{\label{subsect.thms.pieri}A \textquotedblleft
Pieri\textquotedblright\ rule}

Now, the following generalization of Theorem \ref{thm.G.main} holds:

\begin{theorem}
\label{thm.G.pieri}Let $k$ be a positive integer. Let $\mu\in
\operatorname*{Par}$. Then,%
\[
G\left(  k\right)  \cdot s_{\mu}=\sum_{\lambda\in\operatorname*{Par}%
}\operatorname*{pet}\nolimits_{k}\left(  \lambda,\mu\right)  s_{\lambda}.
\]

\end{theorem}

Theorem \ref{thm.G.main} is the particular case of Theorem \ref{thm.G.pieri}
for $\mu=\varnothing$.

We shall give two proofs of Theorem \ref{thm.G.pieri} in Subsections
\ref{subsect.proofs.G.pieri.1st} and \ref{subsect.proofs.G.pieri.2nd}.

We can also generalize Corollary \ref{cor.Gkm.main} to obtain a Pieri-like
rule for multiplication by $G\left(  k,m\right)  $:

\begin{corollary}
\label{cor.Gkm.pieri}Let $k$ be a positive integer. Let $m\in\mathbb{N}$. Let
$\mu\in\operatorname*{Par}$. Then,%
\[
G\left(  k,m\right)  \cdot s_{\mu}=\sum_{\lambda\in\operatorname*{Par}%
\nolimits_{m+\left\vert \mu\right\vert }}\operatorname*{pet}\nolimits_{k}%
\left(  \lambda,\mu\right)  s_{\lambda}.
\]

\end{corollary}

Corollary \ref{cor.Gkm.pieri} follows from Theorem \ref{thm.G.pieri} by
projecting onto the $\left(  m+\left\vert \mu\right\vert \right)  $-th graded
component of $\Lambda$. (We shall explain this argument in more detail in
Subsection \ref{subsect.proofs.G.cors}.)

\subsection{\label{subsect.thms.coprod}Coproducts of Petrie functions}

In the following, the \textquotedblleft$\otimes$\textquotedblright\ sign will
always stand for $\otimes_{\mathbf{k}}$ (that is, tensor product of
$\mathbf{k}$-modules or of $\mathbf{k}$-algebras).

The $\mathbf{k}$-algebra $\Lambda$ is a Hopf algebra due to the presence of a
comultiplication $\Delta:\Lambda\rightarrow\Lambda\otimes\Lambda$. We recall
(from \cite[\S 2.1]{GriRei}) one way to define this comultiplication:

Consider the rings
\[
\mathbf{k}\left[  \left[  \mathbf{x}\right]  \right]  :=\mathbf{k}\left[
\left[  x_{1},x_{2},x_{3},\ldots\right]  \right]
\ \ \ \ \ \ \ \ \ \ \text{and}\ \ \ \ \ \ \ \ \ \ \mathbf{k}\left[  \left[
\mathbf{x},\mathbf{y}\right]  \right]  :=\mathbf{k}\left[  \left[  x_{1}%
,x_{2},x_{3},\ldots,y_{1},y_{2},y_{3},\ldots\right]  \right]
\]
of formal power series. We shall use the notations $\mathbf{x}$ and
$\mathbf{y}$ for the sequences $\left(  x_{1},x_{2},x_{3},\ldots\right)  $ and
$\left(  y_{1},y_{2},y_{3},\ldots\right)  $ of indeterminates. If
$f\in\mathbf{k}\left[  \left[  \mathbf{x}\right]  \right]  $ is any formal
power series, then $f\left(  \mathbf{y}\right)  $ shall mean the result of
substituting $y_{1},y_{2},y_{3},\ldots$ for the variables $x_{1},x_{2}%
,x_{3},\ldots$ in $f$. (This will be a formal power series in $\mathbf{k}%
\left[  \left[  y_{1},y_{2},y_{3},\ldots\right]  \right]  $.) For the sake of
symmetry, we also use the analogous notation $f\left(  \mathbf{x}\right)  $
for the result of substituting $x_{1},x_{2},x_{3},\ldots$ for $x_{1}%
,x_{2},x_{3},\ldots$ in $f$; of course, this $f\left(  \mathbf{x}\right)  $ is
just $f$. Finally, if the power series $f\in\mathbf{k}\left[  \left[
\mathbf{x}\right]  \right]  $ is symmetric, then we use the notation $f\left(
\mathbf{x},\mathbf{y}\right)  $ for the result of substituting the variables
$x_{1},x_{2},x_{3},\ldots,y_{1},y_{2},y_{3},\ldots$ for the variables
$x_{1},x_{2},x_{3},\ldots$ in $f$ (that is, choosing some bijection
$\phi:\left\{  x_{1},x_{2},x_{3},\ldots\right\}  \rightarrow\left\{
x_{1},x_{2},x_{3},\ldots,y_{1},y_{2},y_{3},\ldots\right\}  $%
\ \ \ \ \footnote{Such bijections clearly exist, since the sets $\left\{
x_{1},x_{2},x_{3},\ldots\right\}  $ and $\left\{  x_{1},x_{2},x_{3}%
,\ldots,y_{1},y_{2},y_{3},\ldots\right\}  $ have the same cardinality (namely,
$\aleph_{0}$). This is one of several observations commonly illustrated by the
metaphor of \textquotedblleft Hilbert's hotel\textquotedblright.} and
substituting $\phi\left(  x_{i}\right)  $ for each $x_{i}$ in $f$). This
result does not depend on the order in which the former variables are
substituted for the latter (i.e., on the choice of the bijection $\phi$)
because $f$ is symmetric.

Now, the comultiplication of $\Lambda$ is the map $\Delta:\Lambda
\rightarrow\Lambda\otimes\Lambda$ determined as follows: For a symmetric
function $f\in\Lambda$, we have%
\begin{equation}
\Delta\left(  f\right)  =\sum_{i\in I}f_{1,i}\otimes f_{2,i},
\label{eq.Delta-on-Lam.then}%
\end{equation}
where $f_{1,i},f_{2,i}\in\Lambda$ are such that%
\begin{equation}
f\left(  \mathbf{x},\mathbf{y}\right)  =\sum_{i\in I}f_{1,i}\left(
\mathbf{x}\right)  f_{2,i}\left(  \mathbf{y}\right)  .
\label{eq.Delta-on-Lam.if}%
\end{equation}
More precisely, if $f\in\Lambda$, if $I$ is a finite set, and if $\left(
f_{1,i}\right)  _{i\in I}\in\Lambda^{I}$ and $\left(  f_{2,i}\right)  _{i\in
I}\in\Lambda^{I}$ are two families satisfying (\ref{eq.Delta-on-Lam.if}), then
$\Delta\left(  f\right)  $ is given by (\ref{eq.Delta-on-Lam.then}).
\footnote{In the language of \cite[\S 2.1]{GriRei}, this can be restated as
$\Delta\left(  f\right)  =f\left(  \mathbf{x},\mathbf{y}\right)  $, because
$\Lambda\otimes\Lambda$ is identified with a certain subring of $\mathbf{k}%
\left[  \left[  \mathbf{x},\mathbf{y}\right]  \right]  $ in \cite[\S 2.1]%
{GriRei} (via the injection $\Lambda\otimes\Lambda\rightarrow\mathbf{k}\left[
\left[  \mathbf{x},\mathbf{y}\right]  \right]  $ that sends any $u\otimes
v\in\Lambda\otimes\Lambda$ to $u\left(  \mathbf{x}\right)  v\left(
\mathbf{y}\right)  \in\mathbf{k}\left[  \left[  \mathbf{x},\mathbf{y}\right]
\right]  $).}

For example, for any $n\in\mathbb{N}$, it is easy to see that%
\[
e_{n}\left(  \mathbf{x},\mathbf{y}\right)  =\sum_{i=0}^{n}e_{i}\left(
\mathbf{x}\right)  e_{n-i}\left(  \mathbf{y}\right)  ,
\]
and thus the above definition of $\Delta$ yields%
\[
\Delta\left(  e_{n}\right)  =\sum_{i=0}^{n}e_{i}\otimes e_{n-i}.
\]

A similar formula exists for the image of a Petrie symmetric function under
$\Delta$:

\begin{theorem}
\label{thm.DeltaGkm}Let $k$ be a positive integer. Let $m\in\mathbb{N}$. Then,%
\[
\Delta\left(  G\left(  k,m\right)  \right)  =\sum_{i=0}^{m}G\left(
k,i\right)  \otimes G\left(  k,m-i\right)  .
\]

\end{theorem}

The proof of Theorem \ref{thm.DeltaGkm} is given in Subsection
\ref{subsect.proofs.DeltaGkm}; it is a simple consequence (albeit somewhat
painful to explain) of (\ref{eq.Delta-on-Lam.then}).

It is well-known that $\Delta:\Lambda\rightarrow\Lambda\otimes\Lambda$ is a
$\mathbf{k}$-algebra homomorphism. Equipping the $\mathbf{k}$-algebra
$\Lambda$ with the comultiplication $\Delta$ (as well as a counit
$\varepsilon:\Lambda\rightarrow\mathbf{k}$, which we won't need here) yields a
connected graded Hopf algebra. (See, e.g., \cite[\S 2.1]{GriRei} for proofs.)

\subsection{The Frobenius endomorphisms and Petrie functions}

We shall next derive another formula for the Petrie symmetric functions
$G\left(  k,m\right)  $. For this formula, we need the following definition
(\cite[Exercise 2.9.9]{GriRei}):

\begin{definition}
\label{def.fn}Let $n\in\left\{  1,2,3,\ldots\right\}  $. We define a map
$\mathbf{f}_{n}:\Lambda\rightarrow\Lambda$ by%
\[
\left(  \mathbf{f}_{n}\left(  a\right)  =a\left(  x_{1}^{n},x_{2}^{n}%
,x_{3}^{n},\ldots\right)  \ \ \ \ \ \ \ \ \ \ \text{for each }a\in
\Lambda\right)  .
\]
This map $\mathbf{f}_{n}$ is called the $n$\emph{-th Frobenius endomorphism}
of $\Lambda$.
\end{definition}

Clearly, this map $\mathbf{f}_{n}$ is a $\mathbf{k}$-algebra endomorphism of
$\Lambda$ (since it amounts to a substitution of indeterminates). It is known
(from \cite[Exercise 2.9.9(d)]{GriRei}) that this map $\mathbf{f}_{n}%
:\Lambda\rightarrow\Lambda$ is a Hopf algebra endomorphism of $\Lambda$.

Using the notion of \emph{plethysm} (see, e.g., \cite[Chapter 7, Definition
A2.6]{Stanley-EC2} or \cite[\S I.8]{Macdon95}\footnote{Note that
\cite{Stanley-EC2} uses the notation $f\left[  g\right]  $ for the plethysm of
$f$ with $g$, whereas \cite{Macdon95} uses the notation $f\circ g$ for this.
We shall use $f\left[  g\right]  $.}), we can view the map $\mathbf{f}_{n}$ as
a plethysm with the $n$-th power-sum symmetric function $p_{n}$, in the sense
that any $a\in\Lambda$ satisfies $\mathbf{f}_{n}\left(  a\right)  =a\left[
p_{n}\right]  =p_{n}\left[  a\right]  $ as long as $\mathbf{k}=\mathbb{Z}$.
(Plethysm becomes somewhat subtle when the base ring $\mathbf{k}$ is
complicated; $\mathbf{f}_{n}\left(  a\right)  =a\left[  p_{n}\right]  $ holds
for any $\mathbf{k}$, while $\mathbf{f}_{n}\left(  a\right)  =p_{n}\left[
a\right]  $ relies on good properties of $\mathbf{k}$.) The plethystic
viewpoint makes some properties of $\mathbf{f}_{n}$ clear, but we shall avoid
it for reasons of elementarity.

Now, we can express the Petrie symmetric functions $G\left(  k,m\right)  $
using Frobenius endomorphisms as follows:

\begin{theorem}
\label{thm.G.frob}Let $k$ be a positive integer. Let $m\in\mathbb{N}$. Then,%
\[
G\left(  k,m\right)  =\sum_{i\in\mathbb{N}}\left(  -1\right)  ^{i}%
h_{m-ki}\cdot\mathbf{f}_{k}\left(  e_{i}\right)  .
\]
(The sum on the right hand side of this equality is well-defined, since all
sufficiently high $i\in\mathbb{N}$ satisfy $m-ki<0$ and thus $h_{m-ki}=0$.)
\end{theorem}

Theorem \ref{thm.G.frob} will be proved in Subsection
\ref{subsect.proofs.G.frob} below.

\subsection{\label{subsect.thms.genset}The Petrie functions as polynomial
generators of $\Lambda$}

We now claim the following:

\begin{theorem}
\label{thm.Gkm-genset}Fix a positive integer $k$. Assume that $1-k$ is
invertible in $\mathbf{k}$.

Then, the family $\left(  G\left(  k,m\right)  \right)  _{m\geq1}=\left(
G\left(  k,1\right)  ,G\left(  k,2\right)  ,G\left(  k,3\right)
,\ldots\right)  $ is an algebraically independent generating set of the
commutative $\mathbf{k}$-algebra $\Lambda$. (In other words, the canonical
$\mathbf{k}$-algebra homomorphism
\begin{align*}
\mathbf{k}\left[  u_{1},u_{2},u_{3},\ldots\right]   &  \rightarrow\Lambda,\\
u_{m}  &  \mapsto G\left(  k,m\right)
\end{align*}
is an isomorphism.)
\end{theorem}

We shall prove Theorem \ref{thm.Gkm-genset} in Subsection
\ref{subsect.proofs.Gkm-genset}. The proof uses the following two formulas for
Hall inner products:\footnote{Here, we are again using the Iverson bracket
notation.}

\begin{lemma}
\label{lem.hall-hfep}Let $k$ and $m$ be positive integers. Let $j\in
\mathbb{N}$. Then, $\left\langle p_{m},\mathbf{f}_{k}\left(  e_{j}\right)
\right\rangle =\left(  -1\right)  ^{j-1}\left[  m=kj\right]  k$.
\end{lemma}

\begin{proposition}
\label{prop.hall-pmGkm}Let $k$ and $m$ be positive integers. Then,
$\left\langle p_{m},G\left(  k,m\right)  \right\rangle =1-\left[  k\mid
m\right]  k$.
\end{proposition}

Both of these formulas will be proved in Subsection
\ref{subsect.proofs.Gkm-genset} as well.

\subsection{The Verschiebung endomorphisms}

Now we recall another definition (\cite[Exercise 2.9.10]{GriRei}):

\begin{definition}
\label{def.vn} Let $n\in\left\{  1,2,3,\ldots\right\}  $. We define a
$\mathbf{k}$-algebra homomorphism $\mathbf{v}_{n}:\Lambda\rightarrow\Lambda$
by%
\[
\left(  \mathbf{v}_{n}\left(  h_{m}\right)  =%
\begin{cases}
h_{m/n}, & \text{if }n\mid m;\\
0, & \text{if }n\nmid m
\end{cases}
\ \ \ \ \ \ \ \ \ \ \text{for each }m>0\right)  .
\]
(This is well-defined, since the sequence $\left(  h_{1},h_{2},h_{3}%
,\ldots\right)  $ is an algebraically independent generating set of the
commutative $\mathbf{k}$-algebra $\Lambda$.)

This map $\mathbf{v}_{n}$ is called the $n$\emph{-th Verschiebung
endomorphism} of $\Lambda$.
\end{definition}

Again, it is known (\cite[Exercise 2.9.10(e)]{GriRei}) that this map
$\mathbf{v}_{n}:\Lambda\rightarrow\Lambda$ is a Hopf algebra endomorphism of
$\Lambda$. Moreover, the following holds (\cite[Exercise 2.9.10(f)]{GriRei}):

\begin{proposition}
\label{prop.f-v-dual}Let $n\in\left\{  1,2,3,\ldots\right\}  $. Then, the maps
$\mathbf{f}_{n}:\Lambda\rightarrow\Lambda$ and $\mathbf{v}_{n}:\Lambda
\rightarrow\Lambda$ are adjoint with respect to the Hall inner product on
$\Lambda$. That is, any $a\in\Lambda$ and $b\in\Lambda$ satisfy%
\[
\left\langle a,\mathbf{f}_{n}\left(  b\right)  \right\rangle =\left\langle
\mathbf{v}_{n}\left(  a\right)  ,b\right\rangle .
\]

\end{proposition}

Furthermore, any positive integers $n$ and $m$ satisfy%
\begin{equation}
\mathbf{v}_{n}\left(  p_{m}\right)  =%
\begin{cases}
np_{m/n}, & \text{if }n\mid m;\\
0, & \text{if }n\nmid m
\end{cases}
\ \ \ . \label{eq.vnpm}%
\end{equation}
(This is \cite[Exercise 2.9.10(a)]{GriRei}.)

\subsection{The Hopf endomorphisms $U_{k}$ and $V_{k}$}

In this final subsection, we shall show another way to obtain the Petrie
symmetric functions $G\left(  k,m\right)  $ using the machinery of Hopf
algebras. We refer, e.g., to \cite[Chapters 1 and 2]{GriRei} for everything we
will use about Hopf algebras.

\begin{convention}
As already mentioned, $\Lambda$ is a connected graded Hopf algebra. We let $S$
denote its antipode.
\end{convention}

\begin{definition}
\label{def.convolution}If $C$ is a $\mathbf{k}$-coalgebra and $A$ is a
$\mathbf{k}$-algebra, and if $f,g:C\rightarrow A$ are two $\mathbf{k}$-linear
maps, then the \emph{convolution} $f\star g$ of $f$ and $g$ is defined to be
the $\mathbf{k}$-linear map $m_{A}\circ\left(  f\otimes g\right)  \circ
\Delta_{C}:C\rightarrow A$, where $\Delta_{C}:C\rightarrow C\otimes C$ is the
comultiplication of the $\mathbf{k}$-coalgebra $C$, and where $m_{A}:A\otimes
A\rightarrow A$ is the $\mathbf{k}$-linear map sending each pure tensor
$a\otimes b\in A\otimes A$ to $ab\in A$.
\end{definition}

We also recall Definition \ref{def.vn} and Definition \ref{def.fn}. We now
claim the following.

\begin{theorem}
\label{thm.Uk.main}Fix a positive integer $k$. Let $U_{k}$ be the map
$\mathbf{f}_{k}\circ S\circ\mathbf{v}_{k}:\Lambda\rightarrow\Lambda$. Let
$V_{k}$ be the map $\operatorname*{id}\nolimits_{\Lambda}\star U_{k}%
:\Lambda\rightarrow\Lambda$. (This is well-defined by Definition
\ref{def.convolution}, since $\Lambda$ is both a $\mathbf{k}$-coalgebra and a
$\mathbf{k}$-algebra.) Then:

\textbf{(a)} The map $U_{k}$ is a $\mathbf{k}$-Hopf algebra homomorphism.

\textbf{(b)} The map $V_{k}$ is a $\mathbf{k}$-Hopf algebra homomorphism.

\textbf{(c)} We have $V_{k}\left(  h_{m}\right)  =G\left(  k,m\right)  $ for
each $m\in\mathbb{N}$.

\textbf{(d)} We have $V_{k}\left(  p_{n}\right)  =\left(  1-\left[  k\mid
n\right]  k\right)  p_{n}$ for each positive integer $n$.
\end{theorem}

See Subsection \ref{subsect.proofs.Uk.main} for a proof of this theorem.

\begin{vershort}
Note that Theorem \ref{thm.Uk.main} can be used to give a second proof of
Theorem \ref{thm.DeltaGkm}; see \cite{verlong} for this.
\end{vershort}

\begin{verlong}
Using Theorem \ref{thm.Uk.main}, we can give a new proof for Theorem
\ref{thm.DeltaGkm}; see Subsection \ref{subsect.proofs.DeltaGkm.2nd} for this.
\end{verlong}

We also obtain the following corollary from Theorem \ref{thm.DeltaGkm}:

\begin{corollary}
\label{cor.p-via-G}Let $k$ and $n$ be two positive integers. Then, there
exists a polynomial $f\in\mathbf{k}\left[  x_{1},x_{2},x_{3},\ldots\right]  $
such that%
\begin{equation}
\left(  1-\left[  k\mid n\right]  k\right)  p_{n}=f\left(  G\left(
k,1\right)  ,G\left(  k,2\right)  ,G\left(  k,3\right)  ,\ldots\right)  .
\label{eq.cor.p-via-G.eq}%
\end{equation}

\end{corollary}

This corollary will be proved in Subsection \ref{subsect.proofs.p-via-G}.

\section{\label{sect.proofs}Proofs}

\begin{verlong}
\silentsubsection{\label{subsect.proofs.Sinf}The infinite and finitary
symmetric groups}

We now approach the proofs of the many claims made above. First, let us
briefly discuss some technicalities in the definition of monomial symmetric
functions $m_{\lambda}$.

Let $\mathfrak{S}_{\infty}$ be the group of all permutations of the set
$\left\{  1,2,3,\ldots\right\}  $. (The group operation is given by composition.)

A permutation $\sigma\in\mathfrak{S}_{\infty}$ is said to be \emph{finitary}
if it leaves all but finitely many elements of $\left\{  1,2,3,\ldots\right\}
$ invariant (i.e., if all but finitely many $i\in\left\{  1,2,3,\ldots
\right\}  $ satisfy $\sigma\left(  i\right)  =i$).

The set of all finitary permutations $\sigma\in\mathfrak{S}_{\infty}$ is a
subgroup of $\mathfrak{S}_{\infty}$; it is called the \emph{finitary symmetric
group}, and will be denoted by $\mathfrak{S}_{\left(  \infty\right)  }$.

The group $\mathfrak{S}_{\infty}$ acts on the set $\operatorname*{WC}$ of all
weak compositions by permuting their entries:%
\begin{align*}
\sigma\cdot\left(  \alpha_{1},\alpha_{2},\alpha_{3},\ldots\right)   &
=\left(  \alpha_{\sigma^{-1}\left(  1\right)  },\alpha_{\sigma^{-1}\left(
2\right)  },\alpha_{\sigma^{-1}\left(  3\right)  },\ldots\right) \\
&  \ \ \ \ \ \ \ \ \ \ \text{for any }\left(  \alpha_{1},\alpha_{2},\alpha
_{3},\ldots\right)  \in\operatorname*{WC}\text{ and }\sigma\in\mathfrak{S}%
_{\infty}.
\end{align*}
Thus, the subgroup $\mathfrak{S}_{\left(  \infty\right)  }$ of $\mathfrak{S}%
_{\infty}$ acts on $\operatorname*{WC}$ as well. We now claim the following:

\begin{lemma}
\label{lem.proofs.Sinf-orbits}Let $\beta\in\operatorname*{WC}$. Then, the
orbit $\mathfrak{S}_{\left(  \infty\right)  }\beta$ of $\beta$ under the
action of $\mathfrak{S}_{\left(  \infty\right)  }$ is identical with the orbit
$\mathfrak{S}_{\infty}\beta$ of $\beta$ under the action of $\mathfrak{S}%
_{\infty}$.
\end{lemma}

Our proof of Lemma \ref{lem.proofs.Sinf-orbits} will rely on a lemma about how
two bijections $\varphi:X\rightarrow X^{\prime}$ and $\psi:Y\rightarrow
Y^{\prime}$ can be combined (\textquotedblleft glued
together\textquotedblright) to a bijection $X\cup Y\rightarrow X^{\prime}\cup
Y^{\prime}$ as long as $X$ and $Y$ are disjoint and $X^{\prime}$ and
$Y^{\prime}$ are disjoint:

\begin{lemma}
\label{lem.proofs.Sinf-glue1}Let $X$, $Y$, $X^{\prime}$ and $Y^{\prime}$ be
four sets. Assume that $X$ and $Y$ are disjoint. Assume that $X^{\prime}$ and
$Y^{\prime}$ are disjoint. Let $\varphi:X\rightarrow X^{\prime}$ be a
bijection. Let $\psi:Y\rightarrow Y^{\prime}$ be a bijection. Then, the map%
\begin{align*}
X\cup Y  &  \rightarrow X^{\prime}\cup Y^{\prime},\\
z  &  \mapsto%
\begin{cases}
\varphi\left(  z\right)  , & \text{if }z\in X;\\
\psi\left(  z\right)  , & \text{if }z\in Y
\end{cases}
\end{align*}
is well-defined and is a bijection from $X\cup Y$ to $X^{\prime}\cup
Y^{\prime}$.
\end{lemma}

Lemma \ref{lem.proofs.Sinf-glue1} is a basic fact in set theory, and its proof
is straightforward (whence we omit it).

\begin{proof}
[Proof of Lemma \ref{lem.proofs.Sinf-orbits}.]From $\mathfrak{S}_{\left(
\infty\right)  }\subseteq\mathfrak{S}_{\infty}$, we obtain $\mathfrak{S}%
_{\left(  \infty\right)  }\beta\subseteq\mathfrak{S}_{\infty}\beta$. We shall
now show that $\mathfrak{S}_{\infty}\beta\subseteq\mathfrak{S}_{\left(
\infty\right)  }\beta$.

Indeed, let $\gamma\in\mathfrak{S}_{\infty}\beta$. Thus, there exists some
permutation $\tau\in\mathfrak{S}_{\infty}$ such that $\gamma=\tau\beta$.
Consider this $\tau$. The map $\tau$ is a permutation of the set $\left\{
1,2,3,\ldots\right\}  $ (since $\tau\in\mathfrak{S}_{\infty}$), and thus is a
bijection from $\left\{  1,2,3,\ldots\right\}  $ to $\left\{  1,2,3,\ldots
\right\}  $. We have $\gamma=\left(  \gamma_{1},\gamma_{2},\gamma_{3}%
,\ldots\right)  $ and thus%
\begin{align*}
\left(  \gamma_{1},\gamma_{2},\gamma_{3},\ldots\right)   &  =\gamma
=\tau\underbrace{\beta}_{=\left(  \beta_{1},\beta_{2},\beta_{3},\ldots\right)
}=\tau\cdot\left(  \beta_{1},\beta_{2},\beta_{3},\ldots\right) \\
&  =\left(  \beta_{\tau^{-1}\left(  1\right)  },\beta_{\tau^{-1}\left(
2\right)  },\beta_{\tau^{-1}\left(  3\right)  },\ldots\right)
\end{align*}
(by the definition of the action of $\mathfrak{S}_{\infty}$ on
$\operatorname*{WC}$). In other words,%
\begin{equation}
\gamma_{i}=\beta_{\tau^{-1}\left(  i\right)  }\ \ \ \ \ \ \ \ \ \ \text{for
every positive integer }i. \label{pf.lem.proofs.Sinf-orbits.gbt}%
\end{equation}

Define two subsets $B$ and $C$ of $\left\{  1,2,3,\ldots\right\}  $ by%
\begin{align*}
B  &  =\left\{  i\in\left\{  1,2,3,\ldots\right\}  \ \mid\ \beta_{i}%
\neq0\right\}  \ \ \ \ \ \ \ \ \ \ \text{and}\\
C  &  =\left\{  i\in\left\{  1,2,3,\ldots\right\}  \ \mid\ \gamma_{i}%
\neq0\right\}  .
\end{align*}
Then, the set $B$ is finite\footnote{\textit{Proof.} Recall that $\beta$ is a
weak composition. Thus, $\beta$ contains only finitely many nonzero entries
(by the definition of a weak composition). In other words, there are only
finitely many $i\in\left\{  1,2,3,\ldots\right\}  $ satisfying $\beta_{i}%
\neq0$. In other words, the set $\left\{  i\in\left\{  1,2,3,\ldots\right\}
\ \mid\ \beta_{i}\neq0\right\}  $ is finite. In other words, the set $B$ is
finite (since $B=\left\{  i\in\left\{  1,2,3,\ldots\right\}  \ \mid\ \beta
_{i}\neq0\right\}  $).}. Hence, the set $B\setminus C$ is finite (since
$B\setminus C$ is a subset of $B$).

We have $\tau\left(  j\right)  \in C$ for each $j\in B$%
\ \ \ \ \footnote{\textit{Proof.} Let $j\in B$. Thus, $j\in B=\left\{
i\in\left\{  1,2,3,\ldots\right\}  \ \mid\ \beta_{i}\neq0\right\}  $. In other
words, $j$ is an $i\in\left\{  1,2,3,\ldots\right\}  $ satisfying $\beta
_{i}\neq0$. In other words, $j$ is an element of $\left\{  1,2,3,\ldots
\right\}  $ and satisfies $\beta_{j}\neq0$. Hence, $j\in\left\{
1,2,3,\ldots\right\}  $ (since $j$ is an element of $\left\{  1,2,3,\ldots
\right\}  $), so that $\tau\left(  j\right)  \in\left\{  1,2,3,\ldots\right\}
$ (since $\tau$ is a bijection from $\left\{  1,2,3,\ldots\right\}  $ to
$\left\{  1,2,3,\ldots\right\}  $). Thus, $\tau\left(  j\right)  $ is a
positive integer. Hence, (\ref{pf.lem.proofs.Sinf-orbits.gbt}) (applied to
$i=\tau\left(  j\right)  $) yields $\gamma_{\tau\left(  j\right)  }%
=\beta_{\tau^{-1}\left(  \tau\left(  j\right)  \right)  }=\beta_{j}$ (since
$\tau^{-1}\left(  \tau\left(  j\right)  \right)  =j$). Therefore,
$\gamma_{\tau\left(  j\right)  }=\beta_{j}\neq0$. Thus, $\tau\left(  j\right)
$ is an $i\in\left\{  1,2,3,\ldots\right\}  $ satisfying $\gamma_{i}\neq0$
(since $\tau\left(  j\right)  \in\left\{  1,2,3,\ldots\right\}  $ and
$\gamma_{\tau\left(  j\right)  }\neq0$). In other words, $\tau\left(
j\right)  \in\left\{  i\in\left\{  1,2,3,\ldots\right\}  \ \mid\ \gamma
_{i}\neq0\right\}  $. In other words, $\tau\left(  j\right)  \in C$ (since
$C=\left\{  i\in\left\{  1,2,3,\ldots\right\}  \ \mid\ \gamma_{i}%
\neq0\right\}  $). Qed.}. Hence, the map
\begin{align*}
\overline{\tau}:B  &  \rightarrow C,\\
j  &  \mapsto\tau\left(  j\right)
\end{align*}
is well-defined. Consider this map $\overline{\tau}$.

We have $\tau^{-1}\left(  j\right)  \in B$ for each $j\in C$%
\ \ \ \ \footnote{\textit{Proof.} Let $j\in C$. Thus, $j\in C=\left\{
i\in\left\{  1,2,3,\ldots\right\}  \ \mid\ \gamma_{i}\neq0\right\}  $. In
other words, $j$ is an $i\in\left\{  1,2,3,\ldots\right\}  $ satisfying
$\gamma_{i}\neq0$. In other words, $j$ is an element of $\left\{
1,2,3,\ldots\right\}  $ and satisfies $\gamma_{j}\neq0$. Hence, $j\in\left\{
1,2,3,\ldots\right\}  $ (since $j$ is an element of $\left\{  1,2,3,\ldots
\right\}  $), so that $\tau^{-1}\left(  j\right)  \in\left\{  1,2,3,\ldots
\right\}  $ (since $\tau$ is a bijection from $\left\{  1,2,3,\ldots\right\}
$ to $\left\{  1,2,3,\ldots\right\}  $). Thus, $\tau^{-1}\left(  j\right)  $
is a positive integer. Hence, (\ref{pf.lem.proofs.Sinf-orbits.gbt}) (applied
to $i=j$) yields $\gamma_{j}=\beta_{\tau^{-1}\left(  j\right)  }$. Therefore,
$\beta_{\tau^{-1}\left(  j\right)  }=\gamma_{j}\neq0$. Thus, $\tau^{-1}\left(
j\right)  $ is an $i\in\left\{  1,2,3,\ldots\right\}  $ satisfying $\beta
_{i}\neq0$ (since $\tau^{-1}\left(  j\right)  \in\left\{  1,2,3,\ldots
\right\}  $ and $\beta_{\tau^{-1}\left(  j\right)  }\neq0$). In other words,
$\tau^{-1}\left(  j\right)  \in\left\{  i\in\left\{  1,2,3,\ldots\right\}
\ \mid\ \beta_{i}\neq0\right\}  $. In other words, $\tau^{-1}\left(  j\right)
\in B$ (since $B=\left\{  i\in\left\{  1,2,3,\ldots\right\}  \ \mid\ \beta
_{i}\neq0\right\}  $). Qed.}. Hence, the map
\begin{align*}
\widehat{\tau}:C  &  \rightarrow B,\\
j  &  \mapsto\tau^{-1}\left(  j\right)
\end{align*}
is well-defined. Consider this map $\widehat{\tau}$.

We have $\overline{\tau}\circ\widehat{\tau}=\operatorname*{id}$%
\ \ \ \ \footnote{\textit{Proof.} Every $j\in C$ satisfies%
\begin{align*}
\left(  \overline{\tau}\circ\widehat{\tau}\right)  \left(  j\right)   &
=\overline{\tau}\left(  \underbrace{\widehat{\tau}\left(  j\right)
}_{\substack{=\tau^{-1}\left(  j\right)  \\\text{(by the definition of
}\widehat{\tau}\text{)}}}\right)  =\overline{\tau}\left(  \tau^{-1}\left(
j\right)  \right)  =\tau\left(  \tau^{-1}\left(  j\right)  \right)
\ \ \ \ \ \ \ \ \ \ \left(  \text{by the definition of }\overline{\tau}\right)
\\
&  =j=\operatorname*{id}\left(  j\right)  .
\end{align*}
Thus, $\overline{\tau}\circ\widehat{\tau}=\operatorname*{id}$.} and
$\widehat{\tau}\circ\overline{\tau}=\operatorname*{id}$%
\ \ \ \ \footnote{\textit{Proof.} Every $j\in B$ satisfies%
\begin{align*}
\left(  \widehat{\tau}\circ\overline{\tau}\right)  \left(  j\right)   &
=\widehat{\tau}\left(  \underbrace{\overline{\tau}\left(  j\right)
}_{\substack{=\tau\left(  j\right)  \\\text{(by the definition of }%
\overline{\tau}\text{)}}}\right)  =\widehat{\tau}\left(  \tau\left(  j\right)
\right)  =\tau^{-1}\left(  \tau\left(  j\right)  \right)
\ \ \ \ \ \ \ \ \ \ \left(  \text{by the definition of }\widehat{\tau}\right)
\\
&  =j=\operatorname*{id}\left(  j\right)  .
\end{align*}
Thus, $\widehat{\tau}\circ\overline{\tau}=\operatorname*{id}$.}. Combining
these two equalities, we conclude that the maps $\overline{\tau}$ and
$\widehat{\tau}$ are mutually inverse. Hence, the map $\overline{\tau}$ is
invertible, i.e., is a bijection. Thus, we have found a bijection
$\overline{\tau}:B\rightarrow C$. Therefore, the set $C$ has the same
cardinality as $B$. Thus, the set $C$ is finite (since the set $B$ is finite).
Hence, the set $C\setminus B$ is finite (since $C\setminus B$ is a subset of
$C$).

We have showed that the set $C$ has the same cardinality as $B$. In other
words, $\left\vert C\right\vert =\left\vert B\right\vert $. But the set $C$ is
the union of the two disjoint sets $C\cap B$ and $C\setminus B$; hence, we
have $\left\vert C\right\vert =\left\vert C\cap B\right\vert +\left\vert
C\setminus B\right\vert $. Hence, $\left\vert C\setminus B\right\vert
=\left\vert C\right\vert -\left\vert C\cap B\right\vert $. The same argument
(with the roles of $B$ and $C$ interchanged) yields $\left\vert B\setminus
C\right\vert =\left\vert B\right\vert -\left\vert B\cap C\right\vert $.
Comparing this with $\left\vert C\setminus B\right\vert
=\underbrace{\left\vert C\right\vert }_{=\left\vert B\right\vert }-\left\vert
\underbrace{C\cap B}_{=B\cap C}\right\vert =\left\vert B\right\vert
-\left\vert B\cap C\right\vert $, we obtain $\left\vert C\setminus
B\right\vert =\left\vert B\setminus C\right\vert $. In other words, the two
sets $C\setminus B$ and $B\setminus C$ have the same cardinality. Hence, there
exists a bijection $\rho:C\setminus B\rightarrow B\setminus C$. Consider this
$\rho$.

Elementary set theory shows that $C\cup\left(  C\setminus B\right)  =C\cup
B=B\cup C$ and $B\cup\left(  B\setminus C\right)  =B\cup C$.

The sets $B$ and $C\setminus B$ are disjoint. The sets $C$ and $B\setminus C$
are disjoint. The maps $\overline{\tau}:B\rightarrow C$ and $\rho:C\setminus
B\rightarrow B\setminus C$ are bijections. Hence, Lemma
\ref{lem.proofs.Sinf-glue1} (applied to $X=B$, $Y=C\setminus B$, $X^{\prime
}=C$, $Y^{\prime}=B\setminus C$, $\varphi=\overline{\tau}$ and $\psi=\rho$)
yields that the map%
\begin{align*}
C\cup\left(  C\setminus B\right)   &  \rightarrow B\cup\left(  B\setminus
C\right)  ,\\
z  &  \mapsto%
\begin{cases}
\overline{\tau}\left(  z\right)  , & \text{if }z\in B;\\
\rho\left(  z\right)  , & \text{if }z\in C\setminus B
\end{cases}
\end{align*}
is well-defined and is a bijection from $C\cup\left(  C\setminus B\right)  $
to $B\cup\left(  B\setminus C\right)  $. In view of $C\cup\left(  C\setminus
B\right)  =B\cup C$ and $B\cup\left(  B\setminus C\right)  =B\cup C$, we can
restate this result as follows: The map%
\begin{align*}
B\cup C  &  \rightarrow B\cup C,\\
z  &  \mapsto%
\begin{cases}
\overline{\tau}\left(  z\right)  , & \text{if }z\in B;\\
\rho\left(  z\right)  , & \text{if }z\in C\setminus B
\end{cases}
\end{align*}
is well-defined and is a bijection from $B\cup C$ to $B\cup C$. Let us denote
this map by $\eta$. Thus, $\eta$ is a bijection from $B\cup C$ to $B\cup C$.
In other words, the map $\eta:B\cup C\rightarrow B\cup C$ is a bijection.

Clearly, $B$ and $C$ are subsets of $\left\{  1,2,3,\ldots\right\}  $; thus,
$B\cup C$ is a subset of $\left\{  1,2,3,\ldots\right\}  $. Let $D$ denote the
complement of this subset $B\cup C$ in $\left\{  1,2,3,\ldots\right\}  $. That
is, we have $D=\left\{  1,2,3,\ldots\right\}  \setminus\left(  B\cup C\right)
$. Thus, the set $D$ is disjoint from $B\cup C$ and satisfies $\left(  B\cup
C\right)  \cup D=\left\{  1,2,3,\ldots\right\}  $.

The sets $B\cup C$ and $D$ are disjoint (since $D$ is disjoint from $B\cup
C$). The maps $\eta:B\cup C\rightarrow B\cup C$ and $\operatorname*{id}%
\nolimits_{D}:D\rightarrow D$ are bijections. Hence, Lemma
\ref{lem.proofs.Sinf-glue1} (applied to $X=B\cup C$, $Y=D$, $X^{\prime}=B\cup
C$, $Y^{\prime}=D$, $\varphi=\eta$ and $\psi=\operatorname*{id}\nolimits_{D}$)
yields that the map%
\begin{align*}
\left(  B\cup C\right)  \cup D  &  \rightarrow\left(  B\cup C\right)  \cup
D,\\
z  &  \mapsto%
\begin{cases}
\eta\left(  z\right)  , & \text{if }z\in B\cup C;\\
\operatorname*{id}\nolimits_{D}\left(  z\right)  , & \text{if }z\in D
\end{cases}
\end{align*}
is well-defined and is a bijection from $\left(  B\cup C\right)  \cup D$ to
$\left(  B\cup C\right)  \cup D$. In view of $\left(  B\cup C\right)  \cup
D=\left\{  1,2,3,\ldots\right\}  $, we can restate this result as follows: The
map%
\begin{align*}
\left\{  1,2,3,\ldots\right\}   &  \rightarrow\left\{  1,2,3,\ldots\right\}
,\\
z  &  \mapsto%
\begin{cases}
\eta\left(  z\right)  , & \text{if }z\in B\cup C;\\
\operatorname*{id}\nolimits_{D}\left(  z\right)  , & \text{if }z\in D
\end{cases}
\end{align*}
is well-defined and is a bijection from $\left\{  1,2,3,\ldots\right\}  $ to
$\left\{  1,2,3,\ldots\right\}  $. Let us denote this map by $\sigma$. Thus,
$\sigma$ is a bijection from $\left\{  1,2,3,\ldots\right\}  $ to $\left\{
1,2,3,\ldots\right\}  $. In other words, $\sigma$ is a permutation of
$\left\{  1,2,3,\ldots\right\}  $. In other words, $\sigma\in\mathfrak{S}%
_{\infty}$ (since $\mathfrak{S}_{\infty}$ is the set of all permutations of
$\left\{  1,2,3,\ldots\right\}  $).

The set $B\cup C$ is finite (since it is the union of the two finite sets $B$
and $C$). Hence, all but finitely many $i\in\left\{  1,2,3,\ldots\right\}  $
satisfy $i\notin B\cup C$. But each $i\in\left\{  1,2,3,\ldots\right\}  $
satisfying $i\notin B\cup C$ must satisfy $\sigma\left(  i\right)
=i$\ \ \ \ \footnote{\textit{Proof.} Let $i\in\left\{  1,2,3,\ldots\right\}  $
be such that $i\notin B\cup C$. Thus, $i\in\left\{  1,2,3,\ldots\right\}
\setminus\left(  B\cup C\right)  =D$ (since $D=\left\{  1,2,3,\ldots\right\}
\setminus\left(  B\cup C\right)  $). The definition of $\sigma$ yields
\begin{align*}
\sigma\left(  i\right)   &  =%
\begin{cases}
\eta\left(  i\right)  , & \text{if }i\in B\cup C;\\
\operatorname*{id}\nolimits_{D}\left(  i\right)  , & \text{if }i\in D
\end{cases}
\ \ \ =\operatorname*{id}\nolimits_{D}\left(  i\right)
\ \ \ \ \ \ \ \ \ \ \left(  \text{since }i\in D\right) \\
&  =i.
\end{align*}
Qed.}. Hence, all but finitely many $i\in\left\{  1,2,3,\ldots\right\}  $
satisfy $\sigma\left(  i\right)  =i$ (since all but finitely many
$i\in\left\{  1,2,3,\ldots\right\}  $ satisfy $i\notin B\cup C$). In other
words, the permutation $\sigma$ leaves all but finitely many elements of
$\left\{  1,2,3,\ldots\right\}  $ invariant. In other words, the permutation
$\sigma$ is finitary (by the definition of \textquotedblleft
finitary\textquotedblright). Hence, $\sigma$ is a finitary permutation in
$\mathfrak{S}_{\infty}$. In other words, $\sigma\in\mathfrak{S}_{\left(
\infty\right)  }$ (since $\mathfrak{S}_{\left(  \infty\right)  }$ is the set
of all finitary permutations in $\mathfrak{S}_{\infty}$).

Now, we claim that%
\begin{equation}
\gamma_{\sigma\left(  j\right)  }=\beta_{j}\ \ \ \ \ \ \ \ \ \ \text{for every
positive integer }j. \label{pf.lem.proofs.Sinf-orbits.gambet}%
\end{equation}

[\textit{Proof of (\ref{pf.lem.proofs.Sinf-orbits.gambet}):} Let $j$ be a
positive integer. We must prove (\ref{pf.lem.proofs.Sinf-orbits.gambet}).

We have $j\in\left\{  1,2,3,\ldots\right\}  $ (since $j$ is a positive
integer). Thus, $\sigma\left(  j\right)  \in\left\{  1,2,3,\ldots\right\}  $
(since $\sigma$ is a bijection from $\left\{  1,2,3,\ldots\right\}  $ to
$\left\{  1,2,3,\ldots\right\}  $). In other words, $\sigma\left(  j\right)  $
is a positive integer. Hence, (\ref{pf.lem.proofs.Sinf-orbits.gbt}) (applied
to $i=\sigma\left(  j\right)  $) yields $\gamma_{\sigma\left(  j\right)
}=\beta_{\tau^{-1}\left(  \sigma\left(  j\right)  \right)  }$.

We are in one of the following three cases:

\textit{Case 1:} We have $j\in B$.

\textit{Case 2:} We have $j\in C\setminus B$.

\textit{Case 3:} We have neither $j\in B$ nor $j\in C\setminus B$.

Let us first consider Case 1. In this case, we have $j\in B$. Hence, $j\in
B\subseteq B\cup C$. Now, the definition of $\sigma$ yields
\begin{align*}
\sigma\left(  j\right)   &  =%
\begin{cases}
\eta\left(  j\right)  , & \text{if }j\in B\cup C;\\
\operatorname*{id}\nolimits_{D}\left(  j\right)  , & \text{if }j\in D
\end{cases}
\ \ \ =\eta\left(  j\right)  \ \ \ \ \ \ \ \ \ \ \left(  \text{since }j\in
B\cup C\right) \\
&  =%
\begin{cases}
\overline{\tau}\left(  j\right)  , & \text{if }j\in B;\\
\rho\left(  j\right)  , & \text{if }j\in C\setminus B
\end{cases}
\ \ \ \ \ \ \ \ \ \ \left(  \text{by the definition of }\eta\right) \\
&  =\overline{\tau}\left(  j\right)  \ \ \ \ \ \ \ \ \ \ \left(  \text{since
}j\in B\right) \\
&  =\tau\left(  j\right)  \ \ \ \ \ \ \ \ \ \ \left(  \text{by the definition
of }\overline{\tau}\right)  .
\end{align*}
Hence, $\tau^{-1}\left(  \sigma\left(  j\right)  \right)  =j$. But recall that
$\gamma_{\sigma\left(  j\right)  }=\beta_{\tau^{-1}\left(  \sigma\left(
j\right)  \right)  }$. In view of $\tau^{-1}\left(  \sigma\left(  j\right)
\right)  =j$, this rewrites as $\gamma_{\sigma\left(  j\right)  }=\beta_{j}$.
Thus, (\ref{pf.lem.proofs.Sinf-orbits.gambet}) is proved in Case 1.

Let us next consider Case 2. In this case, we have $j\in C\setminus B$. Thus,
$j\in C\setminus B\subseteq C\subseteq B\cup C$. Now, the definition of
$\sigma$ yields
\begin{align*}
\sigma\left(  j\right)   &  =%
\begin{cases}
\eta\left(  j\right)  , & \text{if }j\in B\cup C;\\
\operatorname*{id}\nolimits_{D}\left(  j\right)  , & \text{if }j\in D
\end{cases}
\ \ \ =\eta\left(  j\right)  \ \ \ \ \ \ \ \ \ \ \left(  \text{since }j\in
B\cup C\right) \\
&  =%
\begin{cases}
\overline{\tau}\left(  j\right)  , & \text{if }j\in B;\\
\rho\left(  j\right)  , & \text{if }j\in C\setminus B
\end{cases}
\ \ \ \ \ \ \ \ \ \ \left(  \text{by the definition of }\eta\right) \\
&  =\rho\left(  \underbrace{j}_{\in C\setminus B}\right)
\ \ \ \ \ \ \ \ \ \ \left(  \text{since }j\in C\setminus B\right) \\
&  \in\rho\left(  C\setminus B\right)  =B\setminus
C\ \ \ \ \ \ \ \ \ \ \left(  \text{since }\rho\text{ is a bijection from
}C\setminus B\text{ to }B\setminus C\right)  .
\end{align*}
In other words, $\sigma\left(  j\right)  \in B$ and $\sigma\left(  j\right)
\notin C$. Thus, we can easily conclude that $\gamma_{\sigma\left(  j\right)
}=0$\ \ \ \ \footnote{\textit{Proof.} Assume the contrary. Thus,
$\gamma_{\sigma\left(  j\right)  }\neq0$. Hence, $\sigma\left(  j\right)  $ is
an $i\in\left\{  1,2,3,\ldots\right\}  $ satisfying $\gamma_{i}\neq0$ (since
$\sigma\left(  j\right)  \in\left\{  1,2,3,\ldots\right\}  $ and
$\gamma_{\sigma\left(  j\right)  }\neq0$). In other words, $\sigma\left(
j\right)  \in\left\{  i\in\left\{  1,2,3,\ldots\right\}  \ \mid\ \gamma
_{i}\neq0\right\}  $. But this contradicts $\sigma\left(  j\right)  \notin
C=\left\{  i\in\left\{  1,2,3,\ldots\right\}  \ \mid\ \gamma_{i}\neq0\right\}
$. This contradiction shows that our assumption was false, qed.}. Also, from
$j\in C\setminus B$, we obtain $j\in C$ and $j\notin B$. Hence, $\beta_{j}%
=0$\ \ \ \ \footnote{\textit{Proof.} Assume the contrary. Thus, $\beta_{j}%
\neq0$. Hence, $j$ is an $i\in\left\{  1,2,3,\ldots\right\}  $ satisfying
$\beta_{i}\neq0$ (since $j\in\left\{  1,2,3,\ldots\right\}  $ and $\beta
_{j}\neq0$). In other words, $j\in\left\{  i\in\left\{  1,2,3,\ldots\right\}
\ \mid\ \beta_{i}\neq0\right\}  $. But this contradicts $j\notin B=\left\{
i\in\left\{  1,2,3,\ldots\right\}  \ \mid\ \beta_{i}\neq0\right\}  $. This
contradiction shows that our assumption was false, qed.}. Comparing this with
$\gamma_{\sigma\left(  j\right)  }=0$, we obtain $\gamma_{\sigma\left(
j\right)  }=\beta_{j}$. Thus, (\ref{pf.lem.proofs.Sinf-orbits.gambet}) is
proved in Case 2.

Let us finally consider Case 3. In this case, we have neither $j\in B$ nor
$j\in C\setminus B$. In other words, we have $j\notin B$ and $j\notin
C\setminus B$. Hence, we have $j\notin B\cup\left(  C\setminus B\right)  $. In
view of $B\cup\left(  C\setminus B\right)  =B\cup C$ (which follows from
elementary set theory), we can restate this as $j\notin B\cup C$. In other
words, $j\notin B$ and $j\notin C$. Thus, we can easily conclude that
$\gamma_{j}=0$\ \ \ \ \footnote{\textit{Proof.} Assume the contrary. Thus,
$\gamma_{j}\neq0$. Hence, $j$ is an $i\in\left\{  1,2,3,\ldots\right\}  $
satisfying $\gamma_{i}\neq0$ (since $j\in\left\{  1,2,3,\ldots\right\}  $ and
$\gamma_{j}\neq0$). In other words, $j\in\left\{  i\in\left\{  1,2,3,\ldots
\right\}  \ \mid\ \gamma_{i}\neq0\right\}  $. But this contradicts $j\notin
C=\left\{  i\in\left\{  1,2,3,\ldots\right\}  \ \mid\ \gamma_{i}\neq0\right\}
$. This contradiction shows that our assumption was false, qed.} and
$\beta_{j}=0$\ \ \ \ \footnote{\textit{Proof.} Assume the contrary. Thus,
$\beta_{j}\neq0$. Hence, $j$ is an $i\in\left\{  1,2,3,\ldots\right\}  $
satisfying $\beta_{i}\neq0$ (since $j\in\left\{  1,2,3,\ldots\right\}  $ and
$\beta_{j}\neq0$). In other words, $j\in\left\{  i\in\left\{  1,2,3,\ldots
\right\}  \ \mid\ \beta_{i}\neq0\right\}  $. But this contradicts $j\notin
B=\left\{  i\in\left\{  1,2,3,\ldots\right\}  \ \mid\ \beta_{i}\neq0\right\}
$. This contradiction shows that our assumption was false, qed.}.

From $j\in\left\{  1,2,3,\ldots\right\}  $ and $j\notin B\cup C$, we obtain
$j\in\left\{  1,2,3,\ldots\right\}  \setminus\left(  B\cup C\right)  $. In
other words, $j\in D$ (since $D=\left\{  1,2,3,\ldots\right\}  \setminus
\left(  B\cup C\right)  $). Now, the definition of $\sigma$ yields
\begin{align*}
\sigma\left(  j\right)   &  =%
\begin{cases}
\eta\left(  j\right)  , & \text{if }j\in B\cup C;\\
\operatorname*{id}\nolimits_{D}\left(  j\right)  , & \text{if }j\in D
\end{cases}
\ \ \ =\operatorname*{id}\nolimits_{D}\left(  j\right)
\ \ \ \ \ \ \ \ \ \ \left(  \text{since }j\in D\right) \\
&  =j.
\end{align*}
Hence, $\gamma_{\sigma\left(  j\right)  }=\gamma_{j}=0=\beta_{j}$ (since
$\beta_{j}=0$). Thus, (\ref{pf.lem.proofs.Sinf-orbits.gambet}) is proved in
Case 3.

We have now proved (\ref{pf.lem.proofs.Sinf-orbits.gambet}) in each of the
three Cases 1, 2 and 3. Since these three Cases cover all possibilities, we
thus conclude that (\ref{pf.lem.proofs.Sinf-orbits.gambet}) always holds.]

We have now proved (\ref{pf.lem.proofs.Sinf-orbits.gambet}). Hence, we have%
\begin{equation}
\gamma_{i}=\beta_{\sigma^{-1}\left(  i\right)  }\ \ \ \ \ \ \ \ \ \ \text{for
every positive integer }i. \label{pf.lem.proofs.Sinf-orbits.gbs}%
\end{equation}

[\textit{Proof of (\ref{pf.lem.proofs.Sinf-orbits.gbs}):} Let $i$ be a
positive integer. Thus, $i\in\left\{  1,2,3,\ldots\right\}  $, so that
$\sigma^{-1}\left(  i\right)  \in\left\{  1,2,3,\ldots\right\}  $ (since
$\sigma$ is a bijection from $\left\{  1,2,3,\ldots\right\}  $ to $\left\{
1,2,3,\ldots\right\}  $). In other words, $\sigma^{-1}\left(  i\right)  $ is a
positive integer. Hence, (\ref{pf.lem.proofs.Sinf-orbits.gambet}) (applied to
$j=\sigma^{-1}\left(  i\right)  $) yields $\gamma_{\sigma\left(  \sigma
^{-1}\left(  i\right)  \right)  }=\beta_{\sigma^{-1}\left(  i\right)  }$. In
other words, $\gamma_{i}=\beta_{\sigma^{-1}\left(  i\right)  }$ (since
$\sigma\left(  \sigma^{-1}\left(  i\right)  \right)  =i$). This proves
(\ref{pf.lem.proofs.Sinf-orbits.gbs}).]

Now, we have proved (\ref{pf.lem.proofs.Sinf-orbits.gbs}). In other words, we
have proved that $\gamma_{i}=\beta_{\sigma^{-1}\left(  i\right)  }$ for every
positive integer $i$. In other words,
\begin{equation}
\left(  \gamma_{1},\gamma_{2},\gamma_{3},\ldots\right)  =\left(  \beta
_{\sigma^{-1}\left(  1\right)  },\beta_{\sigma^{-1}\left(  2\right)  }%
,\beta_{\sigma^{-1}\left(  3\right)  },\ldots\right)  .
\label{pf.lem.proofs.Sinf-orbits.at1}%
\end{equation}
Now, the definition of the action of $\mathfrak{S}_{\infty}$ on
$\operatorname*{WC}$ yields
\[
\sigma\cdot\left(  \beta_{1},\beta_{2},\beta_{3},\ldots\right)  =\left(
\beta_{\sigma^{-1}\left(  1\right)  },\beta_{\sigma^{-1}\left(  2\right)
},\beta_{\sigma^{-1}\left(  3\right)  },\ldots\right)  .
\]
Hence,%
\begin{align*}
\sigma\cdot\underbrace{\beta}_{=\left(  \beta_{1},\beta_{2},\beta_{3}%
,\ldots\right)  }  &  =\sigma\cdot\left(  \beta_{1},\beta_{2},\beta_{3}%
,\ldots\right)  =\left(  \beta_{\sigma^{-1}\left(  1\right)  },\beta
_{\sigma^{-1}\left(  2\right)  },\beta_{\sigma^{-1}\left(  3\right)  }%
,\ldots\right) \\
&  =\left(  \gamma_{1},\gamma_{2},\gamma_{3},\ldots\right)
\ \ \ \ \ \ \ \ \ \ \left(  \text{by (\ref{pf.lem.proofs.Sinf-orbits.at1}%
)}\right) \\
&  =\gamma.
\end{align*}
Thus, $\gamma=\underbrace{\sigma}_{\in\mathfrak{S}_{\left(  \infty\right)  }%
}\cdot\beta\in\mathfrak{S}_{\left(  \infty\right)  }\beta$.

Forget that we fixed $\gamma$. We thus have shown that $\gamma\in
\mathfrak{S}_{\left(  \infty\right)  }\beta$ for each $\gamma\in
\mathfrak{S}_{\infty}\beta$. In other words, $\mathfrak{S}_{\infty}%
\beta\subseteq\mathfrak{S}_{\left(  \infty\right)  }\beta$. Combining this
with $\mathfrak{S}_{\left(  \infty\right)  }\beta\subseteq\mathfrak{S}%
_{\infty}\beta$, we obtain $\mathfrak{S}_{\left(  \infty\right)  }%
\beta=\mathfrak{S}_{\infty}\beta$. In other words, the orbit $\mathfrak{S}%
_{\left(  \infty\right)  }\beta$ of $\beta$ under the action of $\mathfrak{S}%
_{\left(  \infty\right)  }$ is identical with the orbit $\mathfrak{S}_{\infty
}\beta$ of $\beta$ under the action of $\mathfrak{S}_{\infty}$. This proves
Lemma \ref{lem.proofs.Sinf-orbits}.
\end{proof}

Recall that each partition is a weak composition. In other words, we have
$\operatorname*{Par}\subseteq\operatorname*{WC}$.

Now, we easily obtain the following:

\begin{proposition}
\label{prop.proofs.mlam-eq}Let $\lambda\in\operatorname*{Par}$. Our above
definition of $m_{\lambda}$ is equivalent to the definition of $m_{\lambda}$
given in \cite[Definition 2.1.3]{GriRei}.
\end{proposition}

\begin{proof}
[Proof of Proposition \ref{prop.proofs.mlam-eq}.]We have $\lambda
\in\operatorname*{Par}\subseteq\operatorname*{WC}$. Hence, Lemma
\ref{lem.proofs.Sinf-orbits} (applied to $\beta=\lambda$) shows that the orbit
$\mathfrak{S}_{\left(  \infty\right)  }\lambda$ of $\lambda$ under the action
of $\mathfrak{S}_{\left(  \infty\right)  }$ is identical with the orbit
$\mathfrak{S}_{\infty}\lambda$ of $\lambda$ under the action of $\mathfrak{S}%
_{\infty}$. In other words, $\mathfrak{S}_{\left(  \infty\right)  }%
\lambda=\mathfrak{S}_{\infty}\lambda$. In other words, $\mathfrak{S}_{\infty
}\lambda=\mathfrak{S}_{\left(  \infty\right)  }\lambda$.

Recall that the group $\mathfrak{S}_{\infty}$ acts on the set
$\operatorname*{WC}$ of all weak compositions by permuting their entries.
Thus, for any weak composition $\alpha$, we have the following chain of
logical equivalences:%
\begin{align*}
&  \ \left(  \alpha\text{ can be obtained from }\lambda\text{ by permuting
entries}\right) \\
&  \Longleftrightarrow\ \left(  \alpha\text{ can be obtained from }%
\lambda\text{ by the action of some }\sigma\in\mathfrak{S}_{\infty}\right) \\
&  \Longleftrightarrow\ \left(  \text{there exists some }\sigma\in
\mathfrak{S}_{\infty}\text{ such that }\alpha=\sigma\lambda\right) \\
&  \Longleftrightarrow\ \left(  \alpha\in\mathfrak{S}_{\infty}\lambda\right)
.
\end{align*}
Thus, we have the following equality of summation signs:%
\begin{equation}
\sum_{\substack{\alpha\in\operatorname*{WC};\\\alpha\text{ can be obtained
from }\lambda\text{ by permuting entries}}}=\sum_{\substack{\alpha
\in\operatorname*{WC};\\\alpha\in\mathfrak{S}_{\infty}\lambda}}=\sum
_{\alpha\in\mathfrak{S}_{\infty}\lambda}
\label{pf.prop.proofs.mlam-eq.sum=sum}%
\end{equation}
(since $\mathfrak{S}_{\infty}\lambda\subseteq\operatorname*{WC}$). Hence,%
\begin{align}
&  \sum_{\substack{\alpha\in\operatorname*{WC};\\\alpha\text{ can be obtained
from }\lambda\text{ by permuting entries}}}\mathbf{x}^{\alpha}\nonumber\\
&  =\sum_{\alpha\in\mathfrak{S}_{\infty}\lambda}\mathbf{x}^{\alpha}%
=\sum_{\alpha\in\mathfrak{S}_{\left(  \infty\right)  }\lambda}\mathbf{x}%
^{\alpha} \label{pf.prop.proofs.mlam-eq.1}%
\end{align}
(since $\mathfrak{S}_{\infty}\lambda=\mathfrak{S}_{\left(  \infty\right)
}\lambda$).

Our above definition of $m_{\lambda}$ says that%
\[
m_{\lambda}=\sum\mathbf{x}^{\alpha},
\]
where the sum ranges over all weak compositions $\alpha\in\operatorname*{WC}$
that can be obtained from $\lambda$ by permuting entries. In other words, it
says that%
\[
m_{\lambda}=\sum_{\substack{\alpha\in\operatorname*{WC};\\\alpha\text{ can be
obtained from }\lambda\text{ by permuting entries}}}\mathbf{x}^{\alpha}.
\]
On the other hand, the definition of $m_{\lambda}$ given in \cite[Definition
2.1.3]{GriRei} says that%
\begin{equation}
m_{\lambda}=\sum_{\alpha\in\mathfrak{S}_{\left(  \infty\right)  }\lambda
}\mathbf{x}^{\alpha}. \label{pf.prop.G.basics.b.1}%
\end{equation}
But the right hand sides of these two equalities are equal (because of
(\ref{pf.prop.proofs.mlam-eq.1})). Hence, the left hand sides must be equal as
well. In other words, $m_{\lambda}$ defined according to our above definition
of $m_{\lambda}$ is equal to $m_{\lambda}$ defined according to
\cite[Definition 2.1.3]{GriRei}. In other words, these two definitions are
equivalent. This proves Proposition \ref{prop.proofs.mlam-eq}.
\end{proof}
\end{verlong}

\subsection{The symmetric functions $h_{\lambda}$}

\begin{vershort}
We shall now approach the proofs of the claims made above. First, let us
introduce a family of symmetric functions, obtained by multiplying several
$h_{n}$'s:
\end{vershort}

\begin{verlong}
Next, let us introduce a family of symmetric functions, obtained by
multiplying several $h_{n}$'s:
\end{verlong}

\begin{definition}
\label{def.hlam}Let $\lambda$ be a partition. Write $\lambda$ in the form
$\lambda=\left(  \lambda_{1},\lambda_{2},\ldots,\lambda_{\ell}\right)  $,
where $\lambda_{1},\lambda_{2},\ldots,\lambda_{\ell}$ are positive integers.
Then, we define a symmetric function $h_{\lambda}\in\Lambda$ by%
\[
h_{\lambda}=h_{\lambda_{1}}h_{\lambda_{2}}\cdots h_{\lambda_{\ell}}.
\]

\end{definition}

\begin{verlong}
Note that this definition also appears in \cite[Definition 2.2.1]{GriRei}.
\end{verlong}

The symmetric function $h_{\lambda}$ is called the \emph{complete homogeneous
symmetric function} corresponding to the partition $\lambda$.

From \cite[Corollary 2.5.17(a)]{GriRei}, we know that the families $\left(
h_{\lambda}\right)  _{\lambda\in\operatorname*{Par}}$ and $\left(  m_{\lambda
}\right)  _{\lambda\in\operatorname*{Par}}$ are dual bases with respect to the
Hall inner product. Thus,%
\begin{equation}
\left\langle h_{\lambda},m_{\mu}\right\rangle =\delta_{\lambda,\mu
}\ \ \ \ \ \ \ \ \ \ \text{for any }\lambda\in\operatorname*{Par}\text{ and
}\mu\in\operatorname*{Par}. \label{eq.hlam-mlam-dual}%
\end{equation}

Let us record a slightly different way to express $h_{\lambda}$:

\begin{proposition}
\label{prop.hlam.inf}Let $\lambda$ be a partition. Then,%
\[
h_{\lambda}=h_{\lambda_{1}}h_{\lambda_{2}}h_{\lambda_{3}}\cdots.
\]
(Here, the infinite product $h_{\lambda_{1}}h_{\lambda_{2}}h_{\lambda_{3}%
}\cdots$ is well-defined, since every sufficiently high positive integer $i$
satisfies $\lambda_{i}=0$ and thus $h_{\lambda_{i}}=h_{0}=1$.)
\end{proposition}

This is how $h_{\lambda}$ is defined in \cite[Section I.2]{Macdon95}.

\begin{vershort}
\begin{proof}
[Proof of Proposition \ref{prop.hlam.inf}.]This is an easy consequence of
$h_{0}=1$.
\end{proof}
\end{vershort}

\begin{verlong}
\begin{proof}
[Proof of Proposition \ref{prop.hlam.inf}.]Write the partition $\lambda$ in
the form $\lambda=\left(  \lambda_{1},\lambda_{2},\ldots,\lambda_{\ell
}\right)  $, where $\lambda_{1},\lambda_{2},\ldots,\lambda_{\ell}$ are
positive integers. Then, the definition of $h_{\lambda}$ yields $h_{\lambda
}=h_{\lambda_{1}}h_{\lambda_{2}}\cdots h_{\lambda_{\ell}}$. But from
$\lambda=\left(  \lambda_{1},\lambda_{2},\ldots,\lambda_{\ell}\right)  $, we
obtain $\lambda_{\ell+1}=\lambda_{\ell+2}=\lambda_{\ell+3}=\cdots=0$. In other
words, each $i\in\left\{  \ell+1,\ell+2,\ell+3,\ldots\right\}  $ satisfies
$\lambda_{i}=0$. Hence, each $i\in\left\{  \ell+1,\ell+2,\ell+3,\ldots
\right\}  $ satisfies
\begin{align}
h_{\lambda_{i}}  &  =h_{0}\ \ \ \ \ \ \ \ \ \ \left(  \text{since }\lambda
_{i}=0\right) \nonumber\\
&  =1. \label{pf.prop.hlam.inf.1}%
\end{align}

Now,
\begin{align*}
h_{\lambda_{1}}h_{\lambda_{2}}h_{\lambda_{3}}\cdots &  =\left(  h_{\lambda
_{1}}h_{\lambda_{2}}\cdots h_{\lambda_{\ell}}\right)  \underbrace{\left(
h_{\lambda_{\ell+1}}h_{\lambda_{\ell+2}}h_{\lambda_{\ell+3}}\cdots\right)
}_{=\prod_{i=\ell+1}^{\infty}h_{\lambda_{i}}}=\left(  h_{\lambda_{1}%
}h_{\lambda_{2}}\cdots h_{\lambda_{\ell}}\right)  \prod_{i=\ell+1}^{\infty
}\underbrace{h_{\lambda_{i}}}_{\substack{=1\\\text{(by
(\ref{pf.prop.hlam.inf.1}))}}}\\
&  =\left(  h_{\lambda_{1}}h_{\lambda_{2}}\cdots h_{\lambda_{\ell}}\right)
\underbrace{\prod_{i=\ell+1}^{\infty}1}_{=1}=h_{\lambda_{1}}h_{\lambda_{2}%
}\cdots h_{\lambda_{\ell}}.
\end{align*}
Comparing this with $h_{\lambda}=h_{\lambda_{1}}h_{\lambda_{2}}\cdots
h_{\lambda_{\ell}}$, we obtain $h_{\lambda}=h_{\lambda_{1}}h_{\lambda_{2}%
}h_{\lambda_{3}}\cdots$. This proves Proposition \ref{prop.hlam.inf}.
\end{proof}
\end{verlong}

\subsection{\label{subsect.proofs.ejpj}Proofs of Proposition \ref{prop.hjpj}
and Proposition \ref{prop.ejpj}}

\begin{proof}
[Proof of Proposition \ref{prop.hjpj}.]There are myriad ways to prove this.
Here is perhaps the simplest one: Let us use the notation $h_{\lambda}$ as
defined in Definition \ref{def.hlam}. Thus, $h_{\left(  n\right)  }=h_{n}$
(since $n$ is a positive integer). Applying (\ref{eq.hlam-mlam-dual}) to
$\lambda=\left(  n\right)  $ and $\mu=\left(  n\right)  $, we obtain
$\left\langle h_{\left(  n\right)  },m_{\left(  n\right)  }\right\rangle
=\delta_{\left(  n\right)  ,\left(  n\right)  }=1$. In view of $h_{\left(
n\right)  }=h_{n}$ and $m_{\left(  n\right)  }=p_{n}$, this rewrites as
$\left\langle h_{n},p_{n}\right\rangle =1$. This proves Proposition
\ref{prop.hjpj}.
\end{proof}

\begin{proof}
[Proof of Proposition \ref{prop.ejpj}.]This is \cite[Exercise 2.8.8(a)]%
{GriRei}. But here is a self-contained proof: Proposition \ref{prop.p-as-sum}
yields%
\begin{align*}
\left\langle e_{n},p_{n}\right\rangle  &  =\left\langle e_{n},\sum_{i=0}%
^{n-1}\left(  -1\right)  ^{i}s_{\left(  n-i,1^{i}\right)  }\right\rangle
=\sum_{i=0}^{n-1}\left(  -1\right)  ^{i}\underbrace{\left\langle
e_{n},s_{\left(  n-i,1^{i}\right)  }\right\rangle }_{\substack{=\left\langle
s_{\left(  1^{n}\right)  },s_{\left(  n-i,1^{i}\right)  }\right\rangle
\\\text{(since }e_{n}=s_{\left(  1^{n}\right)  }\text{)}}}\\
&  =\sum_{i=0}^{n-1}\left(  -1\right)  ^{i}\underbrace{\left\langle s_{\left(
1^{n}\right)  },s_{\left(  n-i,1^{i}\right)  }\right\rangle }%
_{\substack{=\delta_{\left(  1^{n}\right)  ,\left(  n-i,1^{i}\right)
}\\\text{(since the basis }\left(  s_{\lambda}\right)  _{\lambda
\in\operatorname*{Par}}\text{ of }\Lambda\text{ is}\\\text{orthonormal with
respect to}\\\text{the Hall inner product)}}}=\sum_{i=0}^{n-1}\left(
-1\right)  ^{i}\underbrace{\delta_{\left(  1^{n}\right)  ,\left(
n-i,1^{i}\right)  }}_{\substack{=%
\begin{cases}
1, & \text{if }\left(  1^{n}\right)  =\left(  n-i,1^{i}\right)  ;\\
0, & \text{if }\left(  1^{n}\right)  \neq\left(  n-i,1^{i}\right)
\end{cases}
\\=%
\begin{cases}
1, & \text{if }i=n-1;\\
0, & \text{if }i\neq n-1
\end{cases}
\\\text{(since we have }\left(  1^{n}\right)  =\left(  n-i,1^{i}\right)
\\\text{if and only if }i=n-1\text{)}}}\\
&  =\sum_{i=0}^{n-1}\left(  -1\right)  ^{i}%
\begin{cases}
1, & \text{if }i=n-1;\\
0, & \text{if }i\neq n-1
\end{cases}
\ \ \ =\left(  -1\right)  ^{n-1}.
\end{align*}

\end{proof}

\begin{verlong}
\silentsubsection{\label{subsect.proofs.basics}Proof of Proposition
\ref{prop.G.basics}}

Our next goal is to prove {Proposition \ref{prop.G.basics}. Speaking frankly,
the proof is obvious, but making it fully rigorous will require us to prove
some lemmas first. We shall use the notations from Subsection
\ref{subsect.proofs.Sinf}; in particular, we recall how the group
}$\mathfrak{S}_{\infty}$ and its subgroup $\mathfrak{S}_{\left(
\infty\right)  }$ act on $\operatorname*{WC}$. It is intuitively clear that
any weak composition has exactly one partition among its rearrangements
(namely, the partition obtained by sorting its entries into weakly decreasing
order). In other words, each orbit of the action of $\mathfrak{S}_{\left(
\infty\right)  }$ on $\operatorname*{WC}$ contains exactly one partition. Let
us state this as a lemma and outline a rigorous proof:

\begin{lemma}
\label{lem.G.basics.unique-perm}Let $\alpha\in\operatorname*{WC}$. Then, there
exists a unique partition $\lambda\in\operatorname*{Par}$ such that $\alpha
\in\mathfrak{S}_{\left(  \infty\right)  }\lambda$.
\end{lemma}

\begin{proof}
[Proof of Lemma \ref{lem.G.basics.unique-perm} (sketched).]We have $\alpha
\in\operatorname*{WC}$. In other words, $\alpha$ is a weak composition. In
other words, $\alpha$ is a sequence $\left(  \alpha_{1},\alpha_{2},\alpha
_{3},\ldots\right)  $ of nonnegative integers that contains only finitely many
nonzero entries. Thus, there exists some $k\in\mathbb{N}$ such that
$\alpha_{k+1}=\alpha_{k+2}=\alpha_{k+3}=\cdots=0$. Consider this $k$. Thus,
$\alpha=\left(  \alpha_{1},\alpha_{2},\ldots,\alpha_{k},0,0,0,\ldots\right)
$. Now, by sorting the first $k$ entries $\alpha_{1},\alpha_{2},\ldots
,\alpha_{k}$ of $\alpha$ into weakly decreasing order, we obtain a new weak
composition $\beta\in\operatorname*{WC}$ that has the form $\beta=\left(
\alpha_{\sigma\left(  1\right)  },\alpha_{\sigma\left(  2\right)  }%
,\ldots,\alpha_{\sigma\left(  k\right)  },0,0,0,\ldots\right)  $ for some
permutation $\sigma$ of $\left\{  1,2,\ldots,k\right\}  $ and satisfies
$\alpha_{\sigma\left(  1\right)  }\geq\alpha_{\sigma\left(  2\right)  }%
\geq\cdots\geq\alpha_{\sigma\left(  k\right)  }$. Consider this $\beta$ and
this $\sigma$. From $\beta=\left(  \alpha_{\sigma\left(  1\right)  }%
,\alpha_{\sigma\left(  2\right)  },\ldots,\alpha_{\sigma\left(  k\right)
},0,0,0,\ldots\right)  $, we conclude that the entries of $\beta$ are weakly
decreasing (since $\alpha_{\sigma\left(  1\right)  }\geq\alpha_{\sigma\left(
2\right)  }\geq\cdots\geq\alpha_{\sigma\left(  k\right)  }\geq0\geq0\geq
0\geq\cdots$); in other words, $\beta$ is a partition. In other words,
$\beta\in\operatorname*{Par}$. Also, $\beta$ was obtained from $\alpha$ by
sorting the first $k$ entries into weakly decreasing order; thus, $\beta$ was
obtained from $\alpha$ by permuting the first $k$ entries. Hence, $\alpha$ can
be obtained from $\beta$ by permuting the first $k$ entries. This shows that
$\alpha\in\mathfrak{S}_{\left(  \infty\right)  }\beta$ (since permuting the
first $k$ entries of a weak composition can be achieved by the action of some
permutation $\sigma\in\mathfrak{S}_{\left(  \infty\right)  }$). Therefore,
there exists a partition $\lambda\in\operatorname*{Par}$ such that $\alpha
\in\mathfrak{S}_{\left(  \infty\right)  }\lambda$ (namely, $\lambda=\beta$).
It thus remains to prove that this $\lambda$ is unique. In other words, it
remains to prove that there exists at most one partition $\lambda
\in\operatorname*{Par}$ such that $\alpha\in\mathfrak{S}_{\left(
\infty\right)  }\lambda$.

But this is easy: Let $\lambda\in\operatorname*{Par}$ be a partition such that
$\alpha\in\mathfrak{S}_{\left(  \infty\right)  }\lambda$. From $\alpha
\in\mathfrak{S}_{\left(  \infty\right)  }\lambda$, we conclude that the
sequence $\alpha$ is a rearrangement of the sequence $\lambda$. Hence, the
sequences $\alpha$ and $\lambda$ differ only in the order of their entries.
Hence, for each $i\in\mathbb{N}$, we have%
\begin{align}
&  \left(  \text{the number of times }i\text{ appears in }\alpha\right)
\nonumber\\
&  =\left(  \text{the number of times }i\text{ appears in }\lambda\right)  .
\label{pf.lem.G.basics.unique-perm.uni.1}%
\end{align}
The same argument (applied to $\beta$ instead of $\lambda$) yields that for
each $i\in\mathbb{N}$, we have%
\begin{align*}
&  \left(  \text{the number of times }i\text{ appears in }\alpha\right) \\
&  =\left(  \text{the number of times }i\text{ appears in }\beta\right)
\end{align*}
(since $\alpha\in\mathfrak{S}_{\left(  \infty\right)  }\beta$). Comparing this
with (\ref{pf.lem.G.basics.unique-perm.uni.1}), we conclude that%
\begin{align*}
&  \left(  \text{the number of times }i\text{ appears in }\beta\right) \\
&  =\left(  \text{the number of times }i\text{ appears in }\lambda\right)
\end{align*}
for each $i\in\mathbb{N}$. In other words, the two partitions $\beta$ and
$\lambda$ contain each $i\in\mathbb{N}$ the same number of times. But this
entails that the partitions $\beta$ and $\lambda$ are equal (because a
partition is uniquely determined by how often it contains each $i\in
\mathbb{N}$). In other words, $\beta=\lambda$. Thus, $\lambda=\beta$. Now,
forget that we fixed $\lambda$. We thus have proved that every partition
$\lambda\in\operatorname*{Par}$ such that $\alpha\in\mathfrak{S}_{\left(
\infty\right)  }\lambda$ will satisfy $\lambda=\beta$. Hence, there exists at
most one partition $\lambda\in\operatorname*{Par}$ such that $\alpha
\in\mathfrak{S}_{\left(  \infty\right)  }\lambda$. This completes the proof of
Lemma \ref{lem.G.basics.unique-perm}.
\end{proof}

The next lemma tells us that the action of $\mathfrak{S}_{\left(
\infty\right)  }$ on $\operatorname*{WC}$ leaves some properties of a
partition unchanged:

\begin{lemma}
\label{lem.G.basics.lam-al}Let $\lambda\in\operatorname*{Par}$. Let $\alpha
\in\mathfrak{S}_{\left(  \infty\right)  }\lambda$. Then:

\textbf{(a)} We have $\left\vert \lambda\right\vert =\left\vert \alpha
\right\vert $.

\textbf{(b)} For every positive integer $k$, we have the logical
equivalence\footnotemark%
\[
\left(  \lambda_{i}<k\text{ for all }i\right)  \ \Longleftrightarrow\ \left(
\alpha_{i}<k\text{ for all }i\right)  .
\]

\end{lemma}

\footnotetext{Here and in all similar situations, \textquotedblleft for all
$i$\textquotedblright\ means \textquotedblleft for all positive integers
$i$\textquotedblright.} (Note that we could have just as well required
$\lambda\in\operatorname*{WC}$ instead of $\lambda\in\operatorname*{Par}$ in
Lemma \ref{lem.G.basics.lam-al}, and we could have required $\alpha
\in\mathfrak{S}_{\infty}\lambda$ instead of $\alpha\in\mathfrak{S}_{\left(
\infty\right)  }\lambda$. But we have chosen to state the lemma in the setting
in which we will be applying it later on.)

\begin{proof}
[Proof of Lemma \ref{lem.G.basics.lam-al}.]We have $\lambda\in
\operatorname*{Par}\subseteq\operatorname*{WC}$, so that $\lambda=\left(
\lambda_{1},\lambda_{2},\lambda_{3},\ldots\right)  $.

We have $\alpha\in\mathfrak{S}_{\left(  \infty\right)  }\lambda$. In other
words, there exists some $\sigma\in\mathfrak{S}_{\left(  \infty\right)  }$
such that $\alpha=\sigma\cdot\lambda$. Consider this $\sigma$.

We have $\sigma\in\mathfrak{S}_{\left(  \infty\right)  }\subseteq
\mathfrak{S}_{\infty}$. Thus, $\sigma$ is a permutation of the set $\left\{
1,2,3,\ldots\right\}  $; in other words, $\sigma$ is a bijection from
$\left\{  1,2,3,\ldots\right\}  $ to $\left\{  1,2,3,\ldots\right\}  $. Hence,
its inverse $\sigma^{-1}$ is a bijection from $\left\{  1,2,3,\ldots\right\}
$ to $\left\{  1,2,3,\ldots\right\}  $ as well. In other words, $\sigma^{-1}$
is a permutation of the set $\left\{  1,2,3,\ldots\right\}  $.

We have $\lambda=\left(  \lambda_{1},\lambda_{2},\lambda_{3},\ldots\right)  $.
Now, $\alpha=\left(  \alpha_{1},\alpha_{2},\alpha_{3},\ldots\right)  $, so
that%
\begin{align*}
\left(  \alpha_{1},\alpha_{2},\alpha_{3},\ldots\right)   &  =\alpha
=\sigma\cdot\underbrace{\lambda}_{=\left(  \lambda_{1},\lambda_{2},\lambda
_{3},\ldots\right)  }=\sigma\cdot\left(  \lambda_{1},\lambda_{2},\lambda
_{3},\ldots\right) \\
&  =\left(  \lambda_{\sigma^{-1}\left(  1\right)  },\lambda_{\sigma
^{-1}\left(  2\right)  },\lambda_{\sigma^{-1}\left(  3\right)  }%
,\ldots\right)
\end{align*}
(by the definition of the action of $\mathfrak{S}_{\infty}$ on
$\operatorname*{WC}$).

\textbf{(a)} The definition of $\left\vert \lambda\right\vert $ yields
\begin{align*}
\left\vert \lambda\right\vert  &  =\lambda_{1}+\lambda_{2}+\lambda_{3}%
+\cdots=\sum_{i\in\left\{  1,2,3,\ldots\right\}  }\lambda_{i}=\sum
_{i\in\left\{  1,2,3,\ldots\right\}  }\lambda_{\sigma^{-1}\left(  i\right)
}\\
&  \ \ \ \ \ \ \ \ \ \ \left(
\begin{array}
[c]{c}%
\text{here, we have substituted }\sigma^{-1}\left(  i\right)  \text{ for
}i\text{ in the sum,}\\
\text{since }\sigma^{-1}\text{ is a bijection from }\left\{  1,2,3,\ldots
\right\}  \text{ to }\left\{  1,2,3,\ldots\right\}
\end{array}
\right) \\
&  =\lambda_{\sigma^{-1}\left(  1\right)  }+\lambda_{\sigma^{-1}\left(
2\right)  }+\lambda_{\sigma^{-1}\left(  3\right)  }+\cdots.
\end{align*}
The definition of $\left\vert \alpha\right\vert $ yields%
\begin{align*}
\left\vert \alpha\right\vert  &  =\alpha_{1}+\alpha_{2}+\alpha_{3}%
+\cdots=\lambda_{\sigma^{-1}\left(  1\right)  }+\lambda_{\sigma^{-1}\left(
2\right)  }+\lambda_{\sigma^{-1}\left(  3\right)  }+\cdots\\
&  \ \ \ \ \ \ \ \ \ \ \left(  \text{since }\left(  \alpha_{1},\alpha
_{2},\alpha_{3},\ldots\right)  =\left(  \lambda_{\sigma^{-1}\left(  1\right)
},\lambda_{\sigma^{-1}\left(  2\right)  },\lambda_{\sigma^{-1}\left(
3\right)  },\ldots\right)  \right)  .
\end{align*}
Comparing these two equalities, we find $\left\vert \lambda\right\vert
=\left\vert \alpha\right\vert $. This proves Lemma \ref{lem.G.basics.lam-al}
\textbf{(a)}.

\textbf{(b)} Recall that $\sigma^{-1}$ is a permutation of the set $\left\{
1,2,3,\ldots\right\}  $. Hence, the numbers $\lambda_{\sigma^{-1}\left(
1\right)  },\lambda_{\sigma^{-1}\left(  2\right)  },\lambda_{\sigma
^{-1}\left(  3\right)  },\ldots$ are precisely the numbers $\lambda
_{1},\lambda_{2},\lambda_{3},\ldots$, except possibly in a different order.
Thus, the numbers $\lambda_{\sigma^{-1}\left(  1\right)  },\lambda
_{\sigma^{-1}\left(  2\right)  },\lambda_{\sigma^{-1}\left(  3\right)
},\ldots$ are all $<k$ if and only if the numbers $\lambda_{1},\lambda
_{2},\lambda_{3},\ldots$ are all $<k$. In other words, we have the following
logical equivalence:%
\begin{align}
&  \ \left(  \text{the numbers }\lambda_{\sigma^{-1}\left(  1\right)
},\lambda_{\sigma^{-1}\left(  2\right)  },\lambda_{\sigma^{-1}\left(
3\right)  },\ldots\text{ are all }<k\right) \nonumber\\
&  \Longleftrightarrow\ \left(  \text{the numbers }\lambda_{1},\lambda
_{2},\lambda_{3},\ldots\text{ are all }<k\right)  .
\label{pf.lem.G.basics.lam-al.b.1}%
\end{align}

Now, we have the following chain of logical equivalences:%
\begin{align*}
&  \ \left(  \alpha_{i}<k\text{ for all }i\right) \\
&  \Longleftrightarrow\ \left(  \text{the numbers }\alpha_{1},\alpha
_{2},\alpha_{3},\ldots\text{ are all }<k\right) \\
&  \Longleftrightarrow\ \left(  \text{the numbers }\lambda_{\sigma^{-1}\left(
1\right)  },\lambda_{\sigma^{-1}\left(  2\right)  },\lambda_{\sigma
^{-1}\left(  3\right)  },\ldots\text{ are all }<k\right) \\
&  \ \ \ \ \ \ \ \ \ \ \ \ \ \ \ \ \ \ \ \ \left(  \text{since }\left(
\alpha_{1},\alpha_{2},\alpha_{3},\ldots\right)  =\left(  \lambda_{\sigma
^{-1}\left(  1\right)  },\lambda_{\sigma^{-1}\left(  2\right)  }%
,\lambda_{\sigma^{-1}\left(  3\right)  },\ldots\right)  \right) \\
&  \Longleftrightarrow\ \left(  \text{the numbers }\lambda_{1},\lambda
_{2},\lambda_{3},\ldots\text{ are all }<k\right)  \ \ \ \ \ \ \ \ \ \ \left(
\text{by (\ref{pf.lem.G.basics.lam-al.b.1})}\right) \\
&  \Longleftrightarrow\ \left(  \lambda_{i}<k\text{ for all }i\right)  .
\end{align*}
In other words, we have the equivalence $\left(  \lambda_{i}<k\text{ for all
}i\right)  \ \Longleftrightarrow\ \left(  \alpha_{i}<k\text{ for all
}i\right)  $. This proves Lemma \ref{lem.G.basics.lam-al} \textbf{(b)}.
\end{proof}

\begin{proof}
[Proof of Proposition \ref{prop.G.basics}.]\textbf{(a)} It is easy to see that
for any $m\in\mathbb{N}$, the formal power series $G\left(  k,m\right)  $ is
homogeneous of degree $m$\ \ \ \ \footnote{\textit{Proof.} Let $m\in
\mathbb{N}$. For any $\alpha\in\operatorname*{WC}$, the monomial
$\mathbf{x}^{\alpha}$ is a monomial of degree $\left\vert \alpha\right\vert $.
Thus, if $\alpha\in\operatorname*{WC}$ satisfies $\left\vert \alpha\right\vert
=m$, then $\mathbf{x}^{\alpha}$ is a monomial of degree $m$ (since $\left\vert
\alpha\right\vert =m$). Hence, $\sum_{\substack{\alpha\in\operatorname*{WC}%
;\\\left\vert \alpha\right\vert =m;\\\alpha_{i}<k\text{ for all }i}%
}\mathbf{x}^{\alpha}$ is a sum of monomials of degree $m$. In view of
\[
G\left(  k,m\right)  =\sum_{\substack{\alpha\in\operatorname*{WC};\\\left\vert
\alpha\right\vert =m;\\\alpha_{i}<k\text{ for all }i}}\mathbf{x}^{\alpha
}\ \ \ \ \ \ \ \ \ \ \left(  \text{by (\ref{eq.Gkm=})}\right)  ,
\]
we can restate this as follows: $G\left(  k,m\right)  $ is a sum of monomials
of degree $m$. Thus, the formal power series $G\left(  k,m\right)  $ is
homogeneous of degree $m$. Qed.}. Moreover, (\ref{eq.Gk=}) yields%
\[
G\left(  k\right)  =\underbrace{\sum_{\substack{\alpha\in\operatorname*{WC}%
;\\\alpha_{i}<k\text{ for all }i}}}_{\substack{=\sum_{m\in\mathbb{N}}%
\ \ \sum_{\substack{\alpha\in\operatorname*{WC};\\\left\vert \alpha\right\vert
=m;\\\alpha_{i}<k\text{ for all }i}}\\\text{(since }\left\vert \alpha
\right\vert \in\mathbb{N}\text{ for each }\alpha\in\operatorname*{WC}\text{)}%
}}\mathbf{x}^{\alpha}=\sum_{m\in\mathbb{N}}\ \ \underbrace{\sum
_{\substack{\alpha\in\operatorname*{WC};\\\left\vert \alpha\right\vert
=m;\\\alpha_{i}<k\text{ for all }i}}\mathbf{x}^{\alpha}}_{\substack{=G\left(
k,m\right)  \\\text{(by (\ref{eq.Gkm=}))}}}=\sum_{m\in\mathbb{N}}G\left(
k,m\right)  .
\]
Thus, the family $\left(  G\left(  k,m\right)  \right)  _{m\in\mathbb{N}}$ is
the homogeneous decomposition of $G\left(  k\right)  $ (since each $G\left(
k,m\right)  $ is homogeneous of degree $m$). Hence, for each $m\in\mathbb{N}$,
the power series $G\left(  k,m\right)  $ is the $m$-th degree homogeneous
component of $G\left(  k\right)  $. This proves Proposition
\ref{prop.G.basics} \textbf{(a)}.

\textbf{(b)} Let us define the group $\mathfrak{S}_{\left(  \infty\right)  }$
and its action on the set $\operatorname*{WC}$ as in Subsection
\ref{subsect.proofs.Sinf}. Then,%
\begin{equation}
\sum_{\substack{\lambda\in\operatorname*{Par};\\\lambda_{i}<k\text{ for all
}i}}\ \ \underbrace{m_{\lambda}}_{\substack{=\sum_{\alpha\in\mathfrak{S}%
_{\left(  \infty\right)  }\lambda}\mathbf{x}^{\alpha}\\\text{(by
(\ref{pf.prop.G.basics.b.1}))}}}=\sum_{\substack{\lambda\in\operatorname*{Par}%
;\\\lambda_{i}<k\text{ for all }i}}\ \ \sum_{\alpha\in\mathfrak{S}_{\left(
\infty\right)  }\lambda}\mathbf{x}^{\alpha}. \label{pf.prop.G.basics.b.new1}%
\end{equation}

Now, we have the following equality of summation signs:%
\begin{align*}
\sum_{\substack{\lambda\in\operatorname*{Par};\\\lambda_{i}<k\text{ for all
}i}}\ \ \sum_{\alpha\in\mathfrak{S}_{\left(  \infty\right)  }\lambda}  &
=\sum_{\lambda\in\operatorname*{Par}}\ \ \underbrace{\sum_{\substack{\alpha
\in\mathfrak{S}_{\left(  \infty\right)  }\lambda;\\\lambda_{i}<k\text{ for all
}i}}}_{\substack{=\sum_{\substack{\alpha\in\mathfrak{S}_{\left(
\infty\right)  }\lambda;\\\alpha_{i}<k\text{ for all }i}}\\\text{(because for
each }\alpha\in\mathfrak{S}_{\left(  \infty\right)  }\lambda\text{, we have
the}\\\text{equivalence }\left(  \lambda_{i}<k\text{ for all }i\right)
\ \Longleftrightarrow\ \left(  \alpha_{i}<k\text{ for all }i\right)
\\\text{(by Lemma \ref{lem.G.basics.lam-al} \textbf{(b)}))}}}=\sum_{\lambda
\in\operatorname*{Par}}\underbrace{\sum_{\substack{\alpha\in\mathfrak{S}%
_{\left(  \infty\right)  }\lambda;\\\alpha_{i}<k\text{ for all }i}%
}}_{\substack{=\sum_{\substack{\alpha\in\operatorname*{WC};\\\alpha
\in\mathfrak{S}_{\left(  \infty\right)  }\lambda;\\\alpha_{i}<k\text{ for all
}i}}\\\text{(since }\mathfrak{S}_{\left(  \infty\right)  }\lambda
\subseteq\operatorname*{WC}\text{)}}}\\
&  =\sum_{\lambda\in\operatorname*{Par}}\ \ \sum_{\substack{\alpha
\in\operatorname*{WC};\\\alpha\in\mathfrak{S}_{\left(  \infty\right)  }%
\lambda;\\\alpha_{i}<k\text{ for all }i}}=\sum_{\substack{\alpha
\in\operatorname*{WC};\\\alpha_{i}<k\text{ for all }i}}\ \ \sum
_{\substack{\lambda\in\operatorname*{Par};\\\alpha\in\mathfrak{S}_{\left(
\infty\right)  }\lambda}}.
\end{align*}
Hence, (\ref{pf.prop.G.basics.b.new1}) becomes%
\begin{align}
\sum_{\substack{\lambda\in\operatorname*{Par};\\\lambda_{i}<k\text{ for all
}i}}m_{\lambda}  &  =\underbrace{\sum_{\substack{\lambda\in\operatorname*{Par}%
;\\\lambda_{i}<k\text{ for all }i}}\ \ \sum_{\alpha\in\mathfrak{S}_{\left(
\infty\right)  }\lambda}}_{=\sum_{\substack{\alpha\in\operatorname*{WC}%
;\\\alpha_{i}<k\text{ for all }i}}\ \ \sum_{\substack{\lambda\in
\operatorname*{Par};\\\alpha\in\mathfrak{S}_{\left(  \infty\right)  }\lambda
}}}\mathbf{x}^{\alpha}\nonumber\\
&  =\sum_{\substack{\alpha\in\operatorname*{WC};\\\alpha_{i}<k\text{ for all
}i}}\ \ \sum_{\substack{\lambda\in\operatorname*{Par};\\\alpha\in
\mathfrak{S}_{\left(  \infty\right)  }\lambda}}\mathbf{x}^{\alpha}.
\label{pf.prop.G.basics.b.new2}%
\end{align}

Now, fix some $\alpha\in\operatorname*{WC}$. Then, Lemma
\ref{lem.G.basics.unique-perm} yields that there exists a unique partition
$\lambda\in\operatorname*{Par}$ such that $\alpha\in\mathfrak{S}_{\left(
\infty\right)  }\lambda$. Thus, the sum $\sum_{\substack{\lambda
\in\operatorname*{Par};\\\alpha\in\mathfrak{S}_{\left(  \infty\right)
}\lambda}}\mathbf{x}^{\alpha}$ has exactly one addend. Hence, this sum
simplifies as follows:%
\begin{equation}
\sum_{\substack{\lambda\in\operatorname*{Par};\\\alpha\in\mathfrak{S}_{\left(
\infty\right)  }\lambda}}\mathbf{x}^{\alpha}=\mathbf{x}^{\alpha}.
\label{pf.prop.G.basics.b.new3}%
\end{equation}

Forget that we fixed $\alpha$. We thus have proved
(\ref{pf.prop.G.basics.b.new3}) for each $\alpha\in\operatorname*{WC}$. Thus,
(\ref{pf.prop.G.basics.b.new2}) becomes%
\[
\sum_{\substack{\lambda\in\operatorname*{Par};\\\lambda_{i}<k\text{ for all
}i}}m_{\lambda}=\sum_{\substack{\alpha\in\operatorname*{WC};\\\alpha
_{i}<k\text{ for all }i}}\ \ \underbrace{\sum_{\substack{\lambda
\in\operatorname*{Par};\\\alpha\in\mathfrak{S}_{\left(  \infty\right)
}\lambda}}\mathbf{x}^{\alpha}}_{\substack{=\mathbf{x}^{\alpha}\\\text{(by
(\ref{pf.prop.G.basics.b.new3}))}}}=\sum_{\substack{\alpha\in
\operatorname*{WC};\\\alpha_{i}<k\text{ for all }i}}\mathbf{x}^{\alpha}.
\]
Comparing this with (\ref{eq.Gk=}), we obtain%
\begin{equation}
G\left(  k\right)  =\sum_{\substack{\lambda\in\operatorname*{Par}%
;\\\lambda_{i}<k\text{ for all }i}}m_{\lambda}.\label{pf.prop.G.basics.b.2}%
\end{equation}
Comparing (\ref{eq.Gk=}) with%
\begin{align*}
&  \prod_{i=1}^{\infty}\underbrace{\left(  x_{i}^{0}+x_{i}^{1}+\cdots
+x_{i}^{k-1}\right)  }_{=\sum_{u\in\left\{  0,1,\ldots,k-1\right\}  }x_{i}%
^{u}}\\
&  =\prod_{i=1}^{\infty}\ \ \sum_{u\in\left\{  0,1,\ldots,k-1\right\}  }%
x_{i}^{u}=\underbrace{\sum_{\substack{\left(  u_{1},u_{2},u_{3},\ldots\right)
\in\left\{  0,1,\ldots,k-1\right\}  ^{\infty};\\\text{all but finitely many
}i\text{ satisfy }u_{i}=0}}}_{\substack{=\sum_{\substack{\left(  u_{1}%
,u_{2},u_{3},\ldots\right)  \in\left\{  0,1,\ldots,k-1\right\}  ^{\infty
};\\\left(  u_{1},u_{2},u_{3},\ldots\right)  \in\operatorname*{WC}%
}}\\\text{(since a sequence }\left(  u_{1},u_{2},u_{3},\ldots\right)  \text{
of nonnegative integers}\\\text{satisfies the statement \textquotedblleft all
but finitely many }i\text{ satisfy }u_{i}=0\text{\textquotedblright}\\\text{if
and only if it satisfies }\left(  u_{1},u_{2},u_{3},\ldots\right)
\in\operatorname*{WC}\text{)}}}x_{1}^{u_{1}}x_{2}^{u_{2}}x_{3}^{u_{3}}\cdots\\
&  \ \ \ \ \ \ \ \ \ \ \ \ \ \ \ \ \ \ \ \ \left(  \text{by the product
rule}\right)  \\
&  =\underbrace{\sum_{\substack{\left(  u_{1},u_{2},u_{3},\ldots\right)
\in\left\{  0,1,\ldots,k-1\right\}  ^{\infty};\\\left(  u_{1},u_{2}%
,u_{3},\ldots\right)  \in\operatorname*{WC}}}}_{\substack{=\sum
_{\substack{\left(  u_{1},u_{2},u_{3},\ldots\right)  \in\operatorname*{WC}%
;\\\left(  u_{1},u_{2},u_{3},\ldots\right)  \in\left\{  0,1,\ldots
,k-1\right\}  ^{\infty}}}\\=\sum_{\substack{\left(  u_{1},u_{2},u_{3}%
,\ldots\right)  \in\operatorname*{WC};\\u_{i}<k\text{ for all }i}%
}\\\text{(because a weak composition }\left(  u_{1},u_{2},u_{3},\ldots\right)
\in\operatorname*{WC}\\\text{satisfies the statement }\left(  u_{1}%
,u_{2},u_{3},\ldots\right)  \in\left\{  0,1,\ldots,k-1\right\}  ^{\infty
}\\\text{if and only if it satisfies \textquotedblleft}u_{i}<k\text{ for all
}i\text{\textquotedblright)}}}x_{1}^{u_{1}}x_{2}^{u_{2}}x_{3}^{u_{3}}\cdots\\
&  =\sum_{\substack{\left(  u_{1},u_{2},u_{3},\ldots\right)  \in
\operatorname*{WC};\\u_{i}<k\text{ for all }i}}x_{1}^{u_{1}}x_{2}^{u_{2}}%
x_{3}^{u_{3}}\cdots=\sum_{\substack{\left(  \alpha_{1},\alpha_{2},\alpha
_{3},\ldots\right)  \in\operatorname*{WC};\\\alpha_{i}<k\text{ for all }%
i}}\underbrace{x_{1}^{\alpha_{1}}x_{2}^{\alpha_{2}}x_{3}^{\alpha_{3}}\cdots
}_{\substack{=\mathbf{x}^{\left(  \alpha_{1},\alpha_{2},\alpha_{3}%
,\ldots\right)  }\\\text{(by the definition of }\mathbf{x}^{\left(  \alpha
_{1},\alpha_{2},\alpha_{3},\ldots\right)  }\text{)}}}\\
&  \ \ \ \ \ \ \ \ \ \ \ \ \ \ \ \ \ \ \ \ \left(
\begin{array}
[c]{c}%
\text{here, we have renamed the}\\
\text{summation index }\left(  u_{1},u_{2},u_{3},\ldots\right)  \text{ as
}\left(  \alpha_{1},\alpha_{2},\alpha_{3},\ldots\right)
\end{array}
\right)  \\
&  =\sum_{\substack{\left(  \alpha_{1},\alpha_{2},\alpha_{3},\ldots\right)
\in\operatorname*{WC};\\\alpha_{i}<k\text{ for all }i}}\mathbf{x}^{\left(
\alpha_{1},\alpha_{2},\alpha_{3},\ldots\right)  }=\sum_{\substack{\alpha
\in\operatorname*{WC};\\\alpha_{i}<k\text{ for all }i}}\mathbf{x}^{\alpha}\\
&  \ \ \ \ \ \ \ \ \ \ \ \ \ \ \ \ \ \ \ \ \left(
\begin{array}
[c]{c}%
\text{here, we have renamed the}\\
\text{summation index }\left(  \alpha_{1},\alpha_{2},\alpha_{3},\ldots\right)
\text{ as }\alpha
\end{array}
\right)  ,
\end{align*}
we obtain%
\[
G\left(  k\right)  =\prod_{i=1}^{\infty}\left(  x_{i}^{0}+x_{i}^{1}%
+\cdots+x_{i}^{k-1}\right)  .
\]
Combining this equality with (\ref{pf.prop.G.basics.b.2}) and (\ref{eq.Gk=}),
we obtain%
\[
G\left(  k\right)  =\sum_{\substack{\alpha\in\operatorname*{WC};\\\alpha
_{i}<k\text{ for all }i}}\mathbf{x}^{\alpha}=\sum_{\substack{\lambda
\in\operatorname*{Par};\\\lambda_{i}<k\text{ for all }i}}m_{\lambda}%
=\prod_{i=1}^{\infty}\left(  x_{i}^{0}+x_{i}^{1}+\cdots+x_{i}^{k-1}\right)  .
\]
This proves Proposition \ref{prop.G.basics} \textbf{(b)}.

\textbf{(c)} Let $m\in\mathbb{N}$. Let us define the group $\mathfrak{S}%
_{\left(  \infty\right)  }$ and its action on the set $\operatorname*{WC}$ as
in Subsection \ref{subsect.proofs.Sinf}. Then,%
\begin{equation}
\sum_{\substack{\lambda\in\operatorname*{Par};\\\left\vert \lambda\right\vert
=m;\\\lambda_{i}<k\text{ for all }i}}\ \ \underbrace{m_{\lambda}%
}_{\substack{=\sum_{\alpha\in\mathfrak{S}_{\left(  \infty\right)  }\lambda
}\mathbf{x}^{\alpha}\\\text{(by (\ref{pf.prop.G.basics.b.1}))}}}=\sum
_{\substack{\lambda\in\operatorname*{Par};\\\left\vert \lambda\right\vert
=m;\\\lambda_{i}<k\text{ for all }i}}\ \ \sum_{\alpha\in\mathfrak{S}_{\left(
\infty\right)  }\lambda}\mathbf{x}^{\alpha}. \label{pf.prop.G.basics.c.new1}%
\end{equation}

Now, we have the following equality of summation signs:%
\begin{align*}
\sum_{\substack{\lambda\in\operatorname*{Par};\\\left\vert \lambda\right\vert
=m;\\\lambda_{i}<k\text{ for all }i}}\ \ \sum_{\alpha\in\mathfrak{S}_{\left(
\infty\right)  }\lambda}  &  =\sum_{\lambda\in\operatorname*{Par}%
}\ \ \underbrace{\sum_{\substack{\alpha\in\mathfrak{S}_{\left(  \infty\right)
}\lambda;\\\left\vert \lambda\right\vert =m;\\\lambda_{i}<k\text{ for all }%
i}}}_{\substack{=\sum_{\substack{\alpha\in\mathfrak{S}_{\left(  \infty\right)
}\lambda;\\\left\vert \lambda\right\vert =m;\\\alpha_{i}<k\text{ for all }%
i}}\\\text{(because for each }\alpha\in\mathfrak{S}_{\left(  \infty\right)
}\lambda\text{, we have the}\\\text{equivalence }\left(  \lambda_{i}<k\text{
for all }i\right)  \ \Longleftrightarrow\ \left(  \alpha_{i}<k\text{ for all
}i\right)  \\\text{(by Lemma \ref{lem.G.basics.lam-al} \textbf{(b)}))}}}\\
&  =\sum_{\lambda\in\operatorname*{Par}}\ \ \underbrace{\sum_{\substack{\alpha
\in\mathfrak{S}_{\left(  \infty\right)  }\lambda;\\\left\vert \lambda
\right\vert =m;\\\alpha_{i}<k\text{ for all }i}}}_{\substack{=\sum
_{\substack{\alpha\in\mathfrak{S}_{\left(  \infty\right)  }\lambda
;\\\left\vert \alpha\right\vert =m;\\\alpha_{i}<k\text{ for all }%
i}}\\\text{(because for each }\alpha\in\mathfrak{S}_{\left(  \infty\right)
}\lambda\text{,}\\\text{we have }\left\vert \lambda\right\vert =\left\vert
\alpha\right\vert \\\text{(by Lemma \ref{lem.G.basics.lam-al} \textbf{(a)}))}%
}}=\sum_{\lambda\in\operatorname*{Par}}\ \ \underbrace{\sum_{\substack{\alpha
\in\mathfrak{S}_{\left(  \infty\right)  }\lambda;\\\left\vert \alpha
\right\vert =m;\\\alpha_{i}<k\text{ for all }i}}}_{\substack{=\sum
_{\substack{\alpha\in\operatorname*{WC};\\\alpha\in\mathfrak{S}_{\left(
\infty\right)  }\lambda;\\\left\vert \alpha\right\vert =m;\\\alpha_{i}<k\text{
for all }i}}\\\text{(since }\mathfrak{S}_{\left(  \infty\right)  }%
\lambda\subseteq\operatorname*{WC}\text{)}}}\\
&  =\sum_{\lambda\in\operatorname*{Par}}\ \ \sum_{\substack{\alpha
\in\operatorname*{WC};\\\alpha\in\mathfrak{S}_{\left(  \infty\right)  }%
\lambda;\\\left\vert \alpha\right\vert =m;\\\alpha_{i}<k\text{ for all }%
i}}=\sum_{\substack{\alpha\in\operatorname*{WC};\\\left\vert \alpha\right\vert
=m;\\\alpha_{i}<k\text{ for all }i}}\ \ \sum_{\substack{\lambda\in
\operatorname*{Par};\\\alpha\in\mathfrak{S}_{\left(  \infty\right)  }\lambda
}}.
\end{align*}
Hence, (\ref{pf.prop.G.basics.c.new1}) becomes%
\begin{align}
\sum_{\substack{\lambda\in\operatorname*{Par};\\\left\vert \lambda\right\vert
=m;\\\lambda_{i}<k\text{ for all }i}}m_{\lambda}  &  =\underbrace{\sum
_{\substack{\lambda\in\operatorname*{Par};\\\left\vert \lambda\right\vert
=m;\\\lambda_{i}<k\text{ for all }i}}\ \ \sum_{\alpha\in\mathfrak{S}_{\left(
\infty\right)  }\lambda}}_{=\sum_{\substack{\alpha\in\operatorname*{WC}%
;\\\left\vert \alpha\right\vert =m;\\\alpha_{i}<k\text{ for all }i}%
}\ \ \sum_{\substack{\lambda\in\operatorname*{Par};\\\alpha\in\mathfrak{S}%
_{\left(  \infty\right)  }\lambda}}}\mathbf{x}^{\alpha}=\sum_{\substack{\alpha
\in\operatorname*{WC};\\\left\vert \alpha\right\vert =m;\\\alpha_{i}<k\text{
for all }i}}\ \ \underbrace{\sum_{\substack{\lambda\in\operatorname*{Par}%
;\\\alpha\in\mathfrak{S}_{\left(  \infty\right)  }\lambda}}\mathbf{x}^{\alpha
}}_{\substack{=\mathbf{x}^{\alpha}\\\text{(by (\ref{pf.prop.G.basics.b.new3}%
))}}}\nonumber\\
&  =\sum_{\substack{\alpha\in\operatorname*{WC};\\\left\vert \alpha\right\vert
=m;\\\alpha_{i}<k\text{ for all }i}}\mathbf{x}^{\alpha}.
\label{pf.prop.G.basics.c.new5}%
\end{align}
Now, (\ref{eq.Gkm=}) becomes%
\[
G\left(  k,m\right)  =\sum_{\substack{\alpha\in\operatorname*{WC};\\\left\vert
\alpha\right\vert =m;\\\alpha_{i}<k\text{ for all }i}}\mathbf{x}^{\alpha}%
=\sum_{\substack{\lambda\in\operatorname*{Par};\\\left\vert \lambda\right\vert
=m;\\\lambda_{i}<k\text{ for all }i}}m_{\lambda}\ \ \ \ \ \ \ \ \ \ \left(
\text{by (\ref{pf.prop.G.basics.c.new5})}\right)  .
\]
This proves Proposition \ref{prop.G.basics} \textbf{(c)}.

\textbf{(d)} Let $m\in\mathbb{N}$ satisfy $k>m$. Then, each $\alpha
\in\operatorname*{WC}$ satisfying $\left\vert \alpha\right\vert =m$ must
automatically satisfy $\left(  \alpha_{i}<k\text{ for all }i\right)
$\ \ \ \ \footnote{\textit{Proof.} Let $\alpha\in\operatorname*{WC}$ satisfy
$\left\vert \alpha\right\vert =m$. We must prove that $\alpha_{i}<k$ for all
$i$.
\par
Indeed, let $i$ be a positive integer. We must prove that $\alpha_{i}<k$.
\par
We have $i\in\left\{  1,2,3,\ldots\right\}  $ (since $i$ is a positive
integer). The definition of $\left\vert \alpha\right\vert $ yields
\begin{align*}
\left\vert \alpha\right\vert  &  =\alpha_{1}+\alpha_{2}+\alpha_{3}+\cdots
=\sum_{j\geq1}\alpha_{j}=\alpha_{i}+\sum_{\substack{j\geq1;\\j\neq
i}}\underbrace{\alpha_{j}}_{\substack{\geq0\\\text{(since }\alpha
\in\operatorname*{WC}\subseteq\mathbb{N}^{\infty}\text{)}}}\\
&  \ \ \ \ \ \ \ \ \ \ \left(  \text{here, we have split off the addend for
}j=i\text{ from the sum}\right) \\
&  \geq\alpha_{i}+\underbrace{\sum_{\substack{j\geq1;\\j\neq i}}0}_{=0}%
=\alpha_{i},
\end{align*}
so that $\alpha_{i}\leq\left\vert \alpha\right\vert =m<k$ (since $k>m$). Thus,
we have proved that $\alpha_{i}<k$. Qed.}. Hence, the condition
\textquotedblleft$\alpha_{i}<k$ for all $i$\textquotedblright\ under the
summation sign \textquotedblleft$\sum_{\substack{\alpha\in\operatorname*{WC}%
;\\\left\vert \alpha\right\vert =m;\\\alpha_{i}<k\text{ for all }i}%
}$\textquotedblright\ is redundant and can be removed. In other words, we have
the following equality between summation signs:%
\[
\sum_{\substack{\alpha\in\operatorname*{WC};\\\left\vert \alpha\right\vert
=m;\\\alpha_{i}<k\text{ for all }i}}=\sum_{\substack{\alpha\in
\operatorname*{WC};\\\left\vert \alpha\right\vert =m}}\ .
\]
Now, (\ref{eq.Gkm=}) yields%
\[
G\left(  k,m\right)  =\underbrace{\sum_{\substack{\alpha\in\operatorname*{WC}%
;\\\left\vert \alpha\right\vert =m;\\\alpha_{i}<k\text{ for all }i}}}%
_{=\sum_{\substack{\alpha\in\operatorname*{WC};\\\left\vert \alpha\right\vert
=m}}}\mathbf{x}^{\alpha}=\sum_{\substack{\alpha\in\operatorname*{WC}%
;\\\left\vert \alpha\right\vert =m}}\mathbf{x}^{\alpha}=h_{m}%
\]
(since the definition of $h_{m}$ yields $h_{m}=\sum_{\substack{\alpha
\in\operatorname*{WC};\\\left\vert \alpha\right\vert =m}}\mathbf{x}^{\alpha}%
$). This proves Proposition \ref{prop.G.basics} \textbf{(d)}.

\textbf{(e)} Let $m\in\mathbb{N}$, and assume that $k=2$. Then, an $\alpha
\in\operatorname*{WC}$ satisfies $\left(  \alpha_{i}<k\text{ for all
}i\right)  $ if and only if it satisfies $\alpha\in\left\{  0,1\right\}
^{\infty}$ (because we have the chain of logical equivalences%
\begin{align*}
\left(  \alpha_{i}<k\text{ for all }i\right)  \  &  \Longleftrightarrow
\ \left(  \alpha_{i}<2\text{ for all }i\right)  \ \ \ \ \ \ \ \ \ \ \left(
\text{since }k=2\right) \\
&  \Longleftrightarrow\ \left(  \alpha_{i}\in\left\{  0,1\right\}  \text{ for
all }i\right)  \ \ \ \ \ \ \ \ \ \ \left(  \text{since }\alpha_{i}%
\in\mathbb{N}\text{ for all }i\right) \\
&  \Longleftrightarrow\ \left(  \alpha\in\left\{  0,1\right\}  ^{\infty
}\right)
\end{align*}
for each $\alpha\in\operatorname*{WC}$). Therefore, the condition
\textquotedblleft$\alpha_{i}<k$ for all $i$\textquotedblright\ under the
summation sign \textquotedblleft$\sum_{\substack{\alpha\in\operatorname*{WC}%
;\\\left\vert \alpha\right\vert =m;\\\alpha_{i}<k\text{ for all }i}%
}$\textquotedblright\ can be replaced by \textquotedblleft$\alpha\in\left\{
0,1\right\}  ^{\infty}$\textquotedblright. Thus, we obtain the following
equality between summation signs:%
\[
\sum_{\substack{\alpha\in\operatorname*{WC};\\\left\vert \alpha\right\vert
=m;\\\alpha_{i}<k\text{ for all }i}}=\sum_{\substack{\alpha\in
\operatorname*{WC};\\\left\vert \alpha\right\vert =m;\\\alpha\in\left\{
0,1\right\}  ^{\infty}}}=\sum_{\substack{\alpha\in\operatorname*{WC}%
;\\\alpha\in\left\{  0,1\right\}  ^{\infty};\\\left\vert \alpha\right\vert
=m}}=\sum_{\substack{\alpha\in\operatorname*{WC}\cap\left\{  0,1\right\}
^{\infty};\\\left\vert \alpha\right\vert =m}}\ \ .
\]
Now, (\ref{eq.Gkm=}) yields%
\[
G\left(  k,m\right)  =\underbrace{\sum_{\substack{\alpha\in\operatorname*{WC}%
;\\\left\vert \alpha\right\vert =m;\\\alpha_{i}<k\text{ for all }i}}}%
_{=\sum_{\substack{\alpha\in\operatorname*{WC}\cap\left\{  0,1\right\}
^{\infty};\\\left\vert \alpha\right\vert =m}}}\mathbf{x}^{\alpha}%
=\sum_{\substack{\alpha\in\operatorname*{WC}\cap\left\{  0,1\right\}
^{\infty};\\\left\vert \alpha\right\vert =m}}\mathbf{x}^{\alpha}=e_{m}%
\]
(since the definition of $e_{m}$ yields $e_{m}=\sum_{\substack{\alpha
\in\operatorname*{WC}\cap\left\{  0,1\right\}  ^{\infty};\\\left\vert
\alpha\right\vert =m}}\mathbf{x}^{\alpha}$). This proves Proposition
\ref{prop.G.basics} \textbf{(e)}.

\textbf{(f)} Let $m=k$. Thus, $p_{m}=p_{k}=m_{\left(  k\right)  }$ (since $k$
is a positive integer).

Now, if $\lambda\in\operatorname*{Par}$ satisfies $\left\vert \lambda
\right\vert =k$, then we have the logical equivalence%
\begin{equation}
\left(  \lambda_{i}<k\text{ for all }i\right)  \ \Longleftrightarrow\ \left(
\lambda\neq\left(  k\right)  \right)  . \label{pf.prop.G.basics.f.1}%
\end{equation}

[\textit{Proof of (\ref{pf.prop.G.basics.f.1}):} Let $\lambda\in
\operatorname*{Par}$ satisfy $\left\vert \lambda\right\vert =k$. We must prove
the equivalence (\ref{pf.prop.G.basics.f.1}). We shall prove the
\textquotedblleft$\Longrightarrow$\textquotedblright\ and \textquotedblleft%
$\Longleftarrow$\textquotedblright\ directions of this equivalence separately:

$\Longrightarrow:$ Assume that $\left(  \lambda_{i}<k\text{ for all }i\right)
$. We must show that $\lambda\neq\left(  k\right)  $.

We have assumed that $\left(  \lambda_{i}<k\text{ for all }i\right)  $.
Applying this to $i=1$, we obtain $\lambda_{1}<k$. But if we had
$\lambda=\left(  k\right)  $, then we would have $\lambda_{1}=\left(
k\right)  _{1}=k$, which would contradict $\lambda_{1}<k$. Hence, we cannot
have $\lambda=\left(  k\right)  $. Thus, $\lambda\neq\left(  k\right)  $. This
proves the \textquotedblleft$\Longrightarrow$\textquotedblright\ implication
of the equivalence (\ref{pf.prop.G.basics.f.1}).

$\Longleftarrow:$ Assume that $\lambda\neq\left(  k\right)  $. We must show
that $\left(  \lambda_{i}<k\text{ for all }i\right)  $.

Indeed, let $i$ be a positive integer. Assume (for the sake of contradiction)
that $\lambda_{i}\geq k$. But $\lambda$ is a partition (since $\lambda
\in\operatorname*{Par}$); thus, $\lambda_{1}\geq\lambda_{2}\geq\lambda_{3}%
\geq\cdots$ and therefore $\lambda_{1}\geq\lambda_{i}\geq k$. But $\left\vert
\lambda\right\vert =k$, so that
\begin{align*}
k  &  =\left\vert \lambda\right\vert =\lambda_{1}+\lambda_{2}+\lambda
_{3}+\cdots\ \ \ \ \ \ \ \ \ \ \left(  \text{by the definition of }\left\vert
\lambda\right\vert \right) \\
&  =\underbrace{\lambda_{1}}_{\geq k}+\left(  \lambda_{2}+\lambda_{3}%
+\lambda_{4}+\cdots\right)  \geq k+\left(  \lambda_{2}+\lambda_{3}+\lambda
_{4}+\cdots\right)  .
\end{align*}
Subtracting $k$ from both sides of this inequality, we obtain $0\geq
\lambda_{2}+\lambda_{3}+\lambda_{4}+\cdots$, so that $\lambda_{2}+\lambda
_{3}+\lambda_{4}+\cdots\leq0$. In other words, the sum of the numbers
$\lambda_{2},\lambda_{3},\lambda_{4},\ldots$ is $\leq0$.

But the numbers $\lambda_{2},\lambda_{3},\lambda_{4},\ldots$ are nonnegative
integers (since $\lambda$ is a partition). Hence, the only way their sum can
be $\leq0$ is if they are all $=0$. Since their sum is $\leq0$, we thus
conclude that they are all $=0$. In other words, $\lambda_{2}=\lambda
_{3}=\lambda_{4}=\cdots=0$. Hence, $\lambda=\left(  \lambda_{1}\right)  $.
Thus, $\left\vert \lambda\right\vert =\lambda_{1}$, so that $\lambda
_{1}=\left\vert \lambda\right\vert =k$. Hence, $\lambda=\left(
\underbrace{\lambda_{1}}_{=k}\right)  =\left(  k\right)  $. This contradicts
$\lambda\neq\left(  k\right)  $.

This contradiction shows that our assumption (that $\lambda_{i}\geq k$) was
false. Hence, we must have $\lambda_{i}<k$.

Forget that we fixed $i$. We have now showed that $\left(  \lambda_{i}<k\text{
for all }i\right)  $. This proves the \textquotedblleft$\Longleftarrow
$\textquotedblright\ implication of the equivalence
(\ref{pf.prop.G.basics.f.1}).

We have now proven both implications of the equivalence
(\ref{pf.prop.G.basics.f.1}). This concludes the proof of
(\ref{pf.prop.G.basics.f.1}).]

The logical equivalence (\ref{pf.prop.G.basics.f.1}) yields the following
equality of summation signs:%
\begin{equation}
\sum_{\substack{\lambda\in\operatorname*{Par};\\\left\vert \lambda\right\vert
=k;\\\lambda_{i}<k\text{ for all }i}}=\sum_{\substack{\lambda\in
\operatorname*{Par};\\\left\vert \lambda\right\vert =k;\\\lambda\neq\left(
k\right)  }}\ \ . \label{pf.prop.G.basics.f.1sum=sum}%
\end{equation}

Now, one of the definitions of $h_{k}$ yields%
\begin{align*}
h_{k}  &  =\sum_{\lambda\in\operatorname*{Par}\nolimits_{k}}m_{\lambda}%
=\sum_{\substack{\lambda\in\operatorname*{Par};\\\left\vert \lambda\right\vert
=k}}m_{\lambda}\\
&  \ \ \ \ \ \ \ \ \ \ \left(  \text{since }\operatorname*{Par}\nolimits_{k}%
\text{ is the set of all }\lambda\in\operatorname*{Par}\text{ satisfying
}\left\vert \lambda\right\vert =k\right) \\
&  =\underbrace{m_{\left(  k\right)  }}_{=p_{k}}+\sum_{\substack{\lambda
\in\operatorname*{Par};\\\left\vert \lambda\right\vert =k;\\\lambda\neq\left(
k\right)  }}m_{\lambda}\\
&  \ \ \ \ \ \ \ \ \ \ \left(  \text{here, we have split off the addend for
}\lambda=\left(  k\right)  \text{ from the sum}\right) \\
&  =p_{k}+\sum_{\substack{\lambda\in\operatorname*{Par};\\\left\vert
\lambda\right\vert =k;\\\lambda\neq\left(  k\right)  }}m_{\lambda}.
\end{align*}
Hence,%
\[
h_{k}-p_{k}=\sum_{\substack{\lambda\in\operatorname*{Par};\\\left\vert
\lambda\right\vert =k;\\\lambda\neq\left(  k\right)  }}m_{\lambda}.
\]

On the other hand, Proposition \ref{prop.G.basics} \textbf{(c)} (applied to
$k$ instead of $m$) yields%
\[
G\left(  k,k\right)  =\sum_{\substack{\alpha\in\operatorname*{WC};\\\left\vert
\alpha\right\vert =k;\\\alpha_{i}<k\text{ for all }i}}\mathbf{x}^{\alpha
}=\underbrace{\sum_{\substack{\lambda\in\operatorname*{Par};\\\left\vert
\lambda\right\vert =k;\\\lambda_{i}<k\text{ for all }i}}}_{\substack{=\sum
_{\substack{\lambda\in\operatorname*{Par};\\\left\vert \lambda\right\vert
=k;\\\lambda\neq\left(  k\right)  }}\\\text{(by
(\ref{pf.prop.G.basics.f.1sum=sum}))}}}m_{\lambda}=\sum_{\substack{\lambda
\in\operatorname*{Par};\\\left\vert \lambda\right\vert =k;\\\lambda\neq\left(
k\right)  }}m_{\lambda}.
\]
Comparing these two equalities, we obtain $G\left(  k,k\right)  =h_{k}-p_{k}$.
In other words, $G\left(  k,m\right)  =h_{m}-p_{m}$ (since $m=k$). This proves
Proposition \ref{prop.G.basics} \textbf{(f)}.
\end{proof}
\end{verlong}

\subsection{\label{subsect.proofs.skew}Skew Schur functions}

Let us define a classical partial order on $\operatorname*{Par}$ (see, e.g.,
\cite[Definition 2.3.1]{GriRei}):

\begin{definition}
Let $\lambda$ and $\mu$ be two partitions.

We say that $\mu\subseteq\lambda$ if each $i\in\left\{  1,2,3,\ldots\right\}
$ satisfies $\mu_{i}\leq\lambda_{i}$.

We say that $\mu\not \subseteq \lambda$ if we don't have $\mu\subseteq\lambda$.
\end{definition}

For example, $\left(  3,2\right)  \subseteq\left(  4,2,1\right)  $, but
$\left(  3,2,1\right)  \not \subseteq \left(  4,2\right)  $ (since $\left(
3,2,1\right)  _{3}=1$ is not $\leq$ to $\left(  4,2\right)  _{3}=0$).

For any two partitions $\lambda$ and $\mu$, a symmetric function
$s_{\lambda/\mu}$ called a \textit{skew Schur function} is defined in
\cite[Definition 2.3.1]{GriRei} and in \cite[\S I.5]{Macdon95} (see also
\cite[Definition 7.10.1]{Stanley-EC2} for the case when $\mu\subseteq\lambda
$). We shall not recall its standard definition here, but rather state a few properties.

The first property (which can in fact be used as an alternative definition of
$s_{\lambda/\mu}$) is the \textit{first Jacobi--Trudi formula} for skew
shapes; it states the following:

\begin{theorem}
\label{thm.JTs.h}Let $\lambda=\left(  \lambda_{1},\lambda_{2},\ldots
,\lambda_{\ell}\right)  $ and $\mu=\left(  \mu_{1},\mu_{2},\ldots,\mu_{\ell
}\right)  $ be two partitions. Then,%
\begin{equation}
s_{\lambda/\mu}=\det\left(  \left(  h_{\lambda_{i}-\mu_{j}-i+j}\right)
_{1\leq i\leq\ell,\ 1\leq j\leq\ell}\right)  . \label{eq.schur.JT.sh-lammu}%
\end{equation}

\end{theorem}

Theorem \ref{thm.JTs.h} appears (with proof) in \cite[(2.4.16)]{GriRei} and in
\cite[Chapter I, (5.4)]{Macdon95}.

The following properties of skew Schur functions are easy to see:

\begin{itemize}
\item If $\lambda$ is any partition, then $s_{\lambda/\varnothing}=s_{\lambda
}$. (Recall that $\varnothing$ denotes the empty partition.)

\item If $\lambda$ and $\mu$ are two partitions satisfying $\mu\not \subseteq
\lambda$, then $s_{\lambda/\mu}=0$.
\end{itemize}

\subsection{\label{subsect.proofs.cauchy}A Cauchy-like identity}

We shall use the following identity, which connects the skew Schur functions
$s_{\lambda/\mu}$, the symmetric functions $h_{\lambda}$ from Definition
\ref{def.hlam} and the monomial symmetric functions $m_{\lambda}$:

\begin{theorem}
\label{thm.our-cauchy}Recall the symmetric functions $h_{\lambda}$ defined in
Definition \ref{def.hlam}. Let $\mu$ be any partition. Then, in the ring
$\mathbf{k}\left[  \left[  \mathbf{x},\mathbf{y}\right]  \right]  $, we have%
\[
\sum_{\lambda\in\operatorname*{Par}}s_{\lambda/\mu}\left(  \mathbf{x}\right)
s_{\lambda}\left(  \mathbf{y}\right)  =s_{\mu}\left(  \mathbf{y}\right)
\cdot\sum_{\lambda\in\operatorname*{Par}}h_{\lambda}\left(  \mathbf{x}\right)
m_{\lambda}\left(  \mathbf{y}\right)  .
\]
Here, we are using the notations introduced in Subsection
\ref{subsect.thms.coprod}.
\end{theorem}

Theorem \ref{thm.our-cauchy} appears in \cite[fourth display on page
70]{Macdon95}, but let us give a proof for the sake of completeness:

\begin{proof}
[Proof of Theorem \ref{thm.our-cauchy}.]A well-known identity (proved, e.g.,
in \cite[Chapter I, (4.2)]{Macdon95} and in \cite[proof of Proposition
2.5.15]{GriRei}) says that%
\begin{equation}
\prod_{i,j=1}^{\infty}\left(  1-x_{i}y_{j}\right)  ^{-1}=\sum_{\lambda
\in\operatorname*{Par}}h_{\lambda}\left(  \mathbf{x}\right)  m_{\lambda
}\left(  \mathbf{y}\right)  . \label{pf.thm.our-cauchy.hmxy}%
\end{equation}
(Here, the product sign \textquotedblleft$\prod_{i,j=1}^{\infty}%
$\textquotedblright\ means \textquotedblleft$\prod_{\left(  i,j\right)
\in\left\{  1,2,3,\ldots\right\}  ^{2}}$\textquotedblright.)

Another well-known identity (proved, e.g., in \cite[\S I.5, example
26]{Macdon95} and in \cite[Exercise 2.5.11(a)]{GriRei}) says that%
\[
\sum_{\lambda\in\operatorname*{Par}}s_{\lambda}\left(  \mathbf{x}\right)
s_{\lambda/\mu}\left(  \mathbf{y}\right)  =s_{\mu}\left(  \mathbf{x}\right)
\cdot\prod_{i,j=1}^{\infty}\left(  1-x_{i}y_{j}\right)  ^{-1}.
\]
If we swap the roles of $\mathbf{x}=\left(  x_{1},x_{2},x_{3},\ldots\right)  $
and $\mathbf{y}=\left(  y_{1},y_{2},y_{3},\ldots\right)  $ in this identity,
then we obtain%
\[
\sum_{\lambda\in\operatorname*{Par}}s_{\lambda}\left(  \mathbf{y}\right)
s_{\lambda/\mu}\left(  \mathbf{x}\right)  =s_{\mu}\left(  \mathbf{y}\right)
\cdot\prod_{i,j=1}^{\infty}\left(  1-y_{i}x_{j}\right)  ^{-1}.
\]
In view of%
\[
\sum_{\lambda\in\operatorname*{Par}}\underbrace{s_{\lambda}\left(
\mathbf{y}\right)  s_{\lambda/\mu}\left(  \mathbf{x}\right)  }_{=s_{\lambda
/\mu}\left(  \mathbf{x}\right)  s_{\lambda}\left(  \mathbf{y}\right)  }%
=\sum_{\lambda\in\operatorname*{Par}}s_{\lambda/\mu}\left(  \mathbf{x}\right)
s_{\lambda}\left(  \mathbf{y}\right)
\]
and
\begin{align*}
\prod_{i,j=1}^{\infty}\left(  1-y_{i}x_{j}\right)  ^{-1}  &
=\underbrace{\prod_{j,i=1}^{\infty}}_{=\prod_{i,j=1}^{\infty}}\left(
1-\underbrace{y_{j}x_{i}}_{=x_{i}y_{j}}\right)  ^{-1}\\
&  \ \ \ \ \ \ \ \ \ \ \left(
\begin{array}
[c]{c}%
\text{here, we have renamed the index }\left(  i,j\right)  \text{ as }\left(
j,i\right) \\
\text{in the product}%
\end{array}
\right) \\
&  =\prod_{i,j=1}^{\infty}\left(  1-x_{i}y_{j}\right)  ^{-1}=\sum_{\lambda
\in\operatorname*{Par}}h_{\lambda}\left(  \mathbf{x}\right)  m_{\lambda
}\left(  \mathbf{y}\right)  \ \ \ \ \ \ \ \ \ \ \left(  \text{by
(\ref{pf.thm.our-cauchy.hmxy})}\right)  ,
\end{align*}
this rewrites as%
\[
\sum_{\lambda\in\operatorname*{Par}}s_{\lambda/\mu}\left(  \mathbf{x}\right)
s_{\lambda}\left(  \mathbf{y}\right)  =s_{\mu}\left(  \mathbf{y}\right)
\cdot\sum_{\lambda\in\operatorname*{Par}}h_{\lambda}\left(  \mathbf{x}\right)
m_{\lambda}\left(  \mathbf{y}\right)  .
\]
This proves Theorem \ref{thm.our-cauchy}.
\end{proof}

\subsection{\label{subsect.proofs.petk.alphak}The $\mathbf{k}$-algebra
homomorphism $\alpha_{k}:\Lambda\rightarrow\mathbf{k}$}

Recall that the family $\left(  h_{n}\right)  _{n\geq1}=\left(  h_{1}%
,h_{2},h_{3},\ldots\right)  $ is algebraically independent and generates
$\Lambda$ as a $\mathbf{k}$-algebra. Thus, $\Lambda$ can be viewed as a
polynomial ring in the (infinitely many) indeterminates $h_{1},h_{2}%
,h_{3},\ldots$. The universal property of a polynomial ring thus shows that if
$A$ is any commutative $\mathbf{k}$-algebra, and if $\left(  a_{1},a_{2}%
,a_{3},\ldots\right)  $ is any sequence of elements of $A$, then there is a
unique $\mathbf{k}$-algebra homomorphism from $\Lambda$ to $A$ that sends
$h_{i}$ to $a_{i}$ for all positive integers $i$. We shall refer to this as
the \textit{h-universal property of }$\Lambda$. It lets us make the following
definition:\footnote{We are using the Iverson bracket notation (see Convention
\ref{conv.iverson}) here.}

\begin{definition}
\label{def.alphak}Let $k$ be a positive integer. The h-universal property of
$\Lambda$ shows that there is a unique $\mathbf{k}$-algebra homomorphism
$\alpha_{k}:\Lambda\rightarrow\mathbf{k}$ that sends $h_{i}$ to $\left[
i<k\right]  $ for all positive integers $i$. Consider this $\alpha_{k}$.
\end{definition}

We will use this homomorphism $\alpha_{k}$ several times in what follows; let
us thus begin by stating some elementary properties of $\alpha_{k}$.

\begin{lemma}
\label{lem.alphak.h}Let $k$ be a positive integer.

\textbf{(a)} We have%
\begin{equation}
\alpha_{k}\left(  h_{i}\right)  =\left[  i<k\right]
\ \ \ \ \ \ \ \ \ \ \text{for all }i\in\mathbb{N}. \label{pf.thm.G.main.alk}%
\end{equation}

\textbf{(b)} We have%
\begin{equation}
\alpha_{k}\left(  h_{i}\right)  =\left[  0\leq i<k\right]
\ \ \ \ \ \ \ \ \ \ \text{for all }i\in\mathbb{Z}. \label{pf.thm.G.main.alkZ}%
\end{equation}

\textbf{(c)} Let $\lambda$ be a partition. Define $h_{\lambda}$ as in
Definition \ref{def.hlam}. Then,%
\begin{equation}
\alpha_{k}\left(  h_{\lambda}\right)  =\left[  \lambda_{i}<k\text{ for all
}i\right]  . \label{pf.thm.G.main.alkh}%
\end{equation}
(Here, \textquotedblleft for all $i$\textquotedblright\ means
\textquotedblleft for all positive integers $i$\textquotedblright.)
\end{lemma}

\begin{proof}
[Proof of Lemma \ref{lem.alphak.h}.]Note that $0<k$ (since $k$ is positive).

\textbf{(a)} Let $i\in\mathbb{N}$. We must prove that $\alpha_{k}\left(
h_{i}\right)  =\left[  i<k\right]  $. If $i>0$, then this follows from the
definition of $\alpha_{k}$. Thus, we WLOG assume that we don't have $i>0$.
Hence, $i=0$ (since $i\in\mathbb{N}$). Therefore, $h_{i}=h_{0}=1$, so that
$\alpha_{k}\left(  h_{i}\right)  =\alpha_{k}\left(  1\right)  =1$ (since
$\alpha_{k}$ is a $\mathbf{k}$-algebra homomorphism). On the other hand,
$i=0<k$, so that $\left[  i<k\right]  =1$. Comparing this with $\alpha
_{k}\left(  h_{i}\right)  =1$, we obtain $\alpha_{k}\left(  h_{i}\right)
=\left[  i<k\right]  $. This proves Lemma \ref{lem.alphak.h} \textbf{(a)}.

\begin{vershort}
\textbf{(b)} Let $i\in\mathbb{Z}$. We must prove that $\alpha_{k}\left(
h_{i}\right)  =\left[  0\leq i<k\right]  $. If $i<0$, then this boils down to
$0=0$ (since $i<0$ entails $h_{i}=0$ and thus $\alpha_{k}\left(  h_{i}\right)
=\alpha_{k}\left(  0\right)  =0$, but also $\left[  0\leq i<k\right]  =0$).
Hence, we WLOG assume that $i\geq0$. Thus, $i\in\mathbb{N}$. Hence, Lemma
\ref{lem.alphak.h} \textbf{(a)} yields $\alpha_{k}\left(  h_{i}\right)
=\left[  i<k\right]  $. On the other hand, the statement \textquotedblleft%
$0\leq i<k$\textquotedblright\ is equivalent to the statement
\textquotedblleft$i<k$\textquotedblright\ (since $0\leq i$ holds
automatically\footnote{because $i\geq0$}); thus, $\left[  0\leq i<k\right]
=\left[  i<k\right]  $. Comparing this with $\alpha_{k}\left(  h_{i}\right)
=\left[  i<k\right]  $, we obtain $\alpha_{k}\left(  h_{i}\right)  =\left[
0\leq i<k\right]  $. This proves Lemma \ref{lem.alphak.h} \textbf{(b)}.
\end{vershort}

\begin{verlong}
\textbf{(b)} Let $i\in\mathbb{Z}$. We must prove that $\alpha_{k}\left(
h_{i}\right)  =\left[  0\leq i<k\right]  $. If $i<0$, then this easily follows
from $0=0$\ \ \ \ \footnote{\textit{Proof.} Assume that $i<0$. Thus, $h_{i}%
=0$, so that $\alpha_{k}\left(  h_{i}\right)  =\alpha_{k}\left(  0\right)  =0$
(since $\alpha_{k}$ is a $\mathbf{k}$-algebra homomorphism). But the statement
\textquotedblleft$0\leq i<k$\textquotedblright\ is false (since $i<0$); hence,
$\left[  0\leq i<k\right]  =0$. Comparing this with $\alpha_{k}\left(
h_{i}\right)  =0$, we obtain $\alpha_{k}\left(  h_{i}\right)  =\left[  0\leq
i<k\right]  $, qed.}. Hence, we WLOG assume that we don't have $i<0$.
Therefore, $i\geq0$, so that $i\in\mathbb{N}$. Hence, Lemma \ref{lem.alphak.h}
\textbf{(a)} yields $\alpha_{k}\left(  h_{i}\right)  =\left[  i<k\right]  $.
On the other hand, the statement \textquotedblleft$0\leq i<k$%
\textquotedblright\ is equivalent to the statement \textquotedblleft%
$i<k$\textquotedblright\ (since $0\leq i$ holds automatically\footnote{because
$i\geq0$}); thus, $\left[  0\leq i<k\right]  =\left[  i<k\right]  $. Comparing
this with $\alpha_{k}\left(  h_{i}\right)  =\left[  i<k\right]  $, we obtain
$\alpha_{k}\left(  h_{i}\right)  =\left[  0\leq i<k\right]  $. This proves
Lemma \ref{lem.alphak.h} \textbf{(b)}.
\end{verlong}

\textbf{(c)} We recall the following simple property of truth values: If
$\ell\in\mathbb{N}$, and if $\mathcal{A}_{1},\mathcal{A}_{2},\ldots
,\mathcal{A}_{\ell}$ are $\ell$ logical statements, then
\begin{align}
\prod_{i=1}^{\ell}\left[  \mathcal{A}_{i}\right]   &  =\left[  \mathcal{A}%
_{1}\right]  \left[  \mathcal{A}_{2}\right]  \cdots\left[  \mathcal{A}_{\ell
}\right]  =\left[  \mathcal{A}_{1}\wedge\mathcal{A}_{2}\wedge\cdots
\wedge\mathcal{A}_{\ell}\right] \nonumber\\
&  =\left[  \mathcal{A}_{i}\text{ for all }i\in\left\{  1,2,\ldots
,\ell\right\}  \right]  . \label{pf.lem.alphak.h.c.iversons}%
\end{align}

Now, write the partition $\lambda$ in the form $\lambda=\left(  \lambda
_{1},\lambda_{2},\ldots,\lambda_{\ell}\right)  $, where $\lambda_{1}%
,\lambda_{2},\ldots,\lambda_{\ell}$ are positive integers. Then, the
definition of $h_{\lambda}$ yields
\[
h_{\lambda}=h_{\lambda_{1}}h_{\lambda_{2}}\cdots h_{\lambda_{\ell}}%
=\prod_{i=1}^{\ell}h_{\lambda_{i}}.
\]
Applying the map $\alpha_{k}$ to both sides of this equality, we find%
\begin{align}
\alpha_{k}\left(  h_{\lambda}\right)   &  =\alpha_{k}\left(  \prod_{i=1}%
^{\ell}h_{\lambda_{i}}\right)  =\prod_{i=1}^{\ell}\underbrace{\alpha
_{k}\left(  h_{\lambda_{i}}\right)  }_{\substack{=\left[  \lambda
_{i}<k\right]  \\\text{(by (\ref{pf.thm.G.main.alk}), applied to }\lambda
_{i}\text{ instead of }i\text{)}}}\nonumber\\
&  \ \ \ \ \ \ \ \ \ \ \ \ \ \ \ \ \ \ \ \ \left(  \text{since }\alpha
_{k}\text{ is a }\mathbf{k}\text{-algebra homomorphism}\right) \nonumber\\
&  =\prod_{i=1}^{\ell}\left[  \lambda_{i}<k\right]  =\left[  \lambda
_{i}<k\text{ for all }i\in\left\{  1,2,\ldots,\ell\right\}  \right]
\label{pf.lem.alphak.h.c.2}\\
&  \ \ \ \ \ \ \ \ \ \ \ \ \ \ \ \ \ \ \ \ \left(  \text{by
(\ref{pf.lem.alphak.h.c.iversons}), applied to }\mathcal{A}_{i}=\left(
\text{\textquotedblleft}\lambda_{i}<k\text{\textquotedblright}\right)
\right)  .\nonumber
\end{align}

\begin{vershort}
But we have $\lambda=\left(  \lambda_{1},\lambda_{2},\ldots,\lambda_{\ell
}\right)  $ and thus $\lambda_{\ell+1}=\lambda_{\ell+2}=\lambda_{\ell
+3}=\cdots=0$. In other words, we have $\lambda_{i}=0$ for all $i\in\left\{
\ell+1,\ell+2,\ell+3,\ldots\right\}  $. Hence, we have $\lambda_{i}<k$ for all
$i\in\left\{  \ell+1,\ell+2,\ell+3,\ldots\right\}  $ (since all these $i$
satisfy $\lambda_{i}=0<k$). Thus, the statement \textquotedblleft$\lambda
_{i}<k$ for all $i$\textquotedblright\ is equivalent to the statement
\textquotedblleft$\lambda_{i}<k$ for all $i\in\left\{  1,2,\ldots
,\ell\right\}  $\textquotedblright.
\end{vershort}

\begin{verlong}
But we have $\lambda=\left(  \lambda_{1},\lambda_{2},\ldots,\lambda_{\ell
}\right)  $ and thus $\lambda_{\ell+1}=\lambda_{\ell+2}=\lambda_{\ell
+3}=\cdots=0$. In other words, we have $\lambda_{i}=0$ for all $i\in\left\{
\ell+1,\ell+2,\ell+3,\ldots\right\}  $. Hence, we have $\lambda_{i}<k$ for all
$i\in\left\{  \ell+1,\ell+2,\ell+3,\ldots\right\}  $ (since $\lambda_{i}=0$
implies $\lambda_{i}<k$ (because $0<k$)).

Now, we have the following equivalence of logical statements:%
\begin{align*}
&  \ \left(  \lambda_{i}<k\text{ for all }i\right) \\
&  \Longleftrightarrow\ \left(  \lambda_{i}<k\text{ for all positive integers
}i\right) \\
&  \Longleftrightarrow\ \left(  \lambda_{i}<k\text{ for all }i\in
\underbrace{\left\{  1,2,3,\ldots\right\}  }_{=\left\{  1,2,\ldots
,\ell\right\}  \cup\left\{  \ell+1,\ell+2,\ell+3,\ldots\right\}  }\right) \\
&  \Longleftrightarrow\ \left(  \lambda_{i}<k\text{ for all }i\in\left\{
1,2,\ldots,\ell\right\}  \cup\left\{  \ell+1,\ell+2,\ell+3,\ldots\right\}
\right) \\
&  \Longleftrightarrow\ \left(  \lambda_{i}<k\text{ for all }i\in\left\{
1,2,\ldots,\ell\right\}  \text{ and all }i\in\left\{  \ell+1,\ell
+2,\ell+3,\ldots\right\}  \right) \\
&  \Longleftrightarrow\ \left(  \lambda_{i}<k\text{ for all }i\in\left\{
1,2,\ldots,\ell\right\}  \right)  \wedge\underbrace{\left(  \lambda
_{i}<k\text{ for all }i\in\left\{  \ell+1,\ell+2,\ell+3,\ldots\right\}
\right)  }_{\substack{\Longleftrightarrow\ \left(  \text{true}\right)
\\\text{(since we have }\lambda_{i}<k\text{ for all }i\in\left\{  \ell
+1,\ell+2,\ell+3,\ldots\right\}  \text{)}}}\\
&  \Longleftrightarrow\ \left(  \lambda_{i}<k\text{ for all }i\in\left\{
1,2,\ldots,\ell\right\}  \right)  \wedge\left(  \text{true}\right) \\
&  \Longleftrightarrow\ \left(  \lambda_{i}<k\text{ for all }i\in\left\{
1,2,\ldots,\ell\right\}  \right)  .
\end{align*}

\end{verlong}

Hence,%
\[
\left[  \lambda_{i}<k\text{ for all }i\right]  =\left[  \lambda_{i}<k\text{
for all }i\in\left\{  1,2,\ldots,\ell\right\}  \right]  .
\]
Comparing this with (\ref{pf.lem.alphak.h.c.2}), we obtain $\alpha_{k}\left(
h_{\lambda}\right)  =\left[  \lambda_{i}<k\text{ for all }i\right]  $. Thus,
Lemma \ref{lem.alphak.h} \textbf{(c)} is proved.
\end{proof}

\subsection{\label{subsect.proofs.petkrel.indep}Proof of Lemma
\ref{lem.petkrel.indep}}

\begin{vershort}
Lemma \ref{lem.petkrel.indep} can be proved directly using Laplace expansion
of determinants. But the homomorphism $\alpha_{k}$ from Definition
\ref{def.alphak} allows for a slicker proof:

\begin{proof}
[Proof of Lemma \ref{lem.petkrel.indep}.]Recall that $\alpha_{k}$ is a
$\mathbf{k}$-algebra homomorphism. Thus, $\alpha_{k}$ respects determinants;
i.e., if $\left(  a_{i,j}\right)  _{1\leq i\leq m,\ 1\leq j\leq m}\in
\Lambda^{m\times m}$ is an $m\times m$-matrix, then%
\begin{align}
&  \alpha_{k}\left(  \det\left(  \left(  a_{i,j}\right)  _{1\leq i\leq
m,\ 1\leq j\leq m}\right)  \right) \nonumber\\
&  =\det\left(  \left(  \alpha_{k}\left(  a_{i,j}\right)  \right)  _{1\leq
i\leq m,\ 1\leq j\leq m}\right)  . \label{pf.lem.petkrel.indep.short.adet}%
\end{align}

Applying $\alpha_{k}$ to both sides of (\ref{eq.schur.JT.sh-lammu}), we obtain%
\begin{align}
\alpha_{k}\left(  s_{\lambda/\mu}\right)   &  =\alpha_{k}\left(  \det\left(
\left(  h_{\lambda_{i}-\mu_{j}-i+j}\right)  _{1\leq i\leq\ell,\ 1\leq
j\leq\ell}\right)  \right) \nonumber\\
&  =\det\left(  \left(  \underbrace{\alpha_{k}\left(  h_{\lambda_{i}-\mu
_{j}-i+j}\right)  }_{\substack{=\left[  0\leq\lambda_{i}-\mu_{j}-i+j<k\right]
\\\text{(by (\ref{pf.thm.G.main.alkZ}))}}}\right)  _{1\leq i\leq\ell,\ 1\leq
j\leq\ell}\right) \nonumber\\
&  \ \ \ \ \ \ \ \ \ \ \ \ \ \ \ \ \ \ \ \ \left(  \text{by
(\ref{pf.lem.petkrel.indep.short.adet}), applied to }m=\ell\text{ and }%
a_{i,j}=h_{\lambda_{i}-\mu_{j}-i+j}\right) \nonumber\\
&  =\det\left(  \left(  \left[  0\leq\lambda_{i}-\mu_{j}-i+j<k\right]
\right)  _{1\leq i\leq\ell,\ 1\leq j\leq\ell}\right)  .
\label{pf.lem.petkrel.indep.short.1}%
\end{align}
Thus, $\det\left(  \left(  \left[  0\leq\lambda_{i}-\mu_{j}-i+j<k\right]
\right)  _{1\leq i\leq\ell,\ 1\leq j\leq\ell}\right)  $ does not depend on the
choice of $\ell$ (since $\alpha_{k}\left(  s_{\lambda/\mu}\right)  $ does not
depend on the choice of $\ell$). This proves Lemma \ref{lem.petkrel.indep}.
\end{proof}
\end{vershort}

\begin{verlong}
We shall give two proofs of Lemma \ref{lem.petkrel.indep}: one using the
homomorphism $\alpha_{k}$ from Definition \ref{def.alphak}, and one by direct
manipulation of determinants.

\begin{proof}
[First proof of Lemma \ref{lem.petkrel.indep}.]Recall that $\alpha_{k}%
:\Lambda\rightarrow\mathbf{k}$ is a $\mathbf{k}$-algebra homomorphism. Thus,
$\alpha_{k}$ respects determinants; i.e., if $\left(  a_{i,j}\right)  _{1\leq
i\leq m,\ 1\leq j\leq m}\in\Lambda^{m\times m}$ is an $m\times m$-matrix over
$\Lambda$ (for some $m\in\mathbb{N}$), then%
\begin{align}
&  \alpha_{k}\left(  \det\left(  \left(  a_{i,j}\right)  _{1\leq i\leq
m,\ 1\leq j\leq m}\right)  \right) \nonumber\\
&  =\det\left(  \left(  \alpha_{k}\left(  a_{i,j}\right)  \right)  _{1\leq
i\leq m,\ 1\leq j\leq m}\right)  . \label{pf.lem.petkrel.indep.adet}%
\end{align}

Applying $\alpha_{k}$ to both sides of the equality
(\ref{eq.schur.JT.sh-lammu}), we obtain%
\begin{align}
\alpha_{k}\left(  s_{\lambda/\mu}\right)   &  =\alpha_{k}\left(  \det\left(
\left(  h_{\lambda_{i}-\mu_{j}-i+j}\right)  _{1\leq i\leq\ell,\ 1\leq
j\leq\ell}\right)  \right) \nonumber\\
&  =\det\left(  \left(  \underbrace{\alpha_{k}\left(  h_{\lambda_{i}-\mu
_{j}-i+j}\right)  }_{\substack{=\left[  0\leq\lambda_{i}-\mu_{j}-i+j<k\right]
\\\text{(by (\ref{pf.thm.G.main.alkZ}), applied to }\lambda_{i}-\mu
_{j}-i+j\\\text{instead of }i\text{)}}}\right)  _{1\leq i\leq\ell,\ 1\leq
j\leq\ell}\right) \nonumber\\
&  \ \ \ \ \ \ \ \ \ \ \left(  \text{by (\ref{pf.lem.petkrel.indep.adet}),
applied to }m=\ell\text{ and }a_{i,j}=h_{\lambda_{i}-\mu_{j}-i+j}\right)
\nonumber\\
&  =\det\left(  \left(  \left[  0\leq\lambda_{i}-\mu_{j}-i+j<k\right]
\right)  _{1\leq i\leq\ell,\ 1\leq j\leq\ell}\right)  .
\label{pf.lem.petkrel.indep.1}%
\end{align}
Clearly, the integer $\alpha_{k}\left(  s_{\lambda/\mu}\right)  $ does not
depend on the choice of $\ell$. In view of the equality
(\ref{pf.lem.petkrel.indep.1}), we can rewrite this as follows: The integer
\newline$\det\left(  \left(  \left[  0\leq\lambda_{i}-\mu_{j}-i+j<k\right]
\right)  _{1\leq i\leq\ell,\ 1\leq j\leq\ell}\right)  $ does not depend on the
choice of $\ell$. This proves Lemma \ref{lem.petkrel.indep}.
\end{proof}

\begin{proof}
[Second proof of Lemma \ref{lem.petkrel.indep}.]It suffices to show that%
\begin{align*}
&  \det\left(  \left(  \left[  0\leq\lambda_{i}-\mu_{j}-i+j<k\right]  \right)
_{1\leq i\leq\ell,\ 1\leq j\leq\ell}\right) \\
&  =\det\left(  \left(  \left[  0\leq\lambda_{i}-\mu_{j}-i+j<k\right]
\right)  _{1\leq i\leq\ell+1,\ 1\leq j\leq\ell+1}\right)  .
\end{align*}
So let us prove this.

From $\mu=\left(  \mu_{1},\mu_{2},\ldots,\mu_{\ell}\right)  $, we obtain
$\mu_{\ell+1}=0$. From $\lambda=\left(  \lambda_{1},\lambda_{2},\ldots
,\lambda_{\ell}\right)  $, we obtain $\lambda_{\ell+1}=0$. Hence,%
\[
\left[  0\leq\underbrace{\lambda_{\ell+1}-\mu_{\ell+1}-\left(  \ell+1\right)
+\left(  \ell+1\right)  }_{\substack{=\lambda_{\ell+1}-\mu_{\ell
+1}=0\\\text{(since }\lambda_{\ell+1}=0\text{ and }\mu_{\ell+1}=0\text{)}%
}}<k\right]  =\left[  0\leq0<k\right]  =1
\]
(since $0\leq0<k$).

Recall that $\lambda_{\ell+1}=0$. Hence, each $j\in\left\{  1,2,\ldots
,\ell\right\}  $ satisfies%
\begin{equation}
\left[  0\leq\lambda_{\ell+1}-\mu_{j}-\left(  \ell+1\right)  +j<k\right]  =0
\label{pf.lem.petkrel.indep.3}%
\end{equation}
\footnote{\textit{Proof of (\ref{pf.lem.petkrel.indep.3}):} Let $j\in\left\{
1,2,\ldots,\ell\right\}  $. Then, $\underbrace{\lambda_{\ell+1}}%
_{=0}-\underbrace{\mu_{j}}_{\geq0}-\underbrace{\left(  \ell+1\right)  }%
_{>\ell}+\underbrace{j}_{\leq\ell}<0-0-\ell+\ell=0$. Hence, $0\leq
\lambda_{\ell+1}-\mu_{j}-\left(  \ell+1\right)  +j<k$ cannot hold. Thus,
$\left[  0\leq\lambda_{\ell+1}-\mu_{j}-\left(  \ell+1\right)  +j<k\right]
=0$. Qed.}. In other words, the first $\ell$ entries of the last row of the
matrix \newline$\left(  \left[  0\leq\lambda_{i}-\mu_{j}-i+j<k\right]
\right)  _{1\leq i\leq\ell+1,\ 1\leq j\leq\ell+1}$ are $0$. Hence, if we
expand the determinant of this matrix along this last row, then we obtain a
sum having only one (potentially) nonzero addend, namely%
\begin{align*}
&  \underbrace{\left[  0\leq\lambda_{\ell+1}-\mu_{\ell+1}-\left(
\ell+1\right)  +\left(  \ell+1\right)  <k\right]  }_{=1}\cdot\det\left(
\left(  \left[  0\leq\lambda_{i}-\mu_{j}-i+j<k\right]  \right)  _{1\leq
i\leq\ell,\ 1\leq j\leq\ell}\right) \\
&  =\det\left(  \left(  \left[  0\leq\lambda_{i}-\mu_{j}-i+j<k\right]
\right)  _{1\leq i\leq\ell,\ 1\leq j\leq\ell}\right)  .
\end{align*}
Thus,%
\begin{align*}
&  \det\left(  \left(  \left[  0\leq\lambda_{i}-\mu_{j}-i+j<k\right]  \right)
_{1\leq i\leq\ell,\ 1\leq j\leq\ell}\right) \\
&  =\det\left(  \left(  \left[  0\leq\lambda_{i}-\mu_{j}-i+j<k\right]
\right)  _{1\leq i\leq\ell+1,\ 1\leq j\leq\ell+1}\right)  .
\end{align*}
This completes our second proof of Lemma \ref{lem.petkrel.indep}.
\end{proof}
\end{verlong}

\begin{vershort}
We record a useful consequence of the above proof:
\end{vershort}

\begin{verlong}
Our first proof of Lemma \ref{lem.petkrel.indep} has an additional consequence
that will be useful to us:
\end{verlong}

\begin{lemma}
\label{lem.alphak.slammu}Let $k$ be a positive integer. Let $\lambda$ and
$\mu$ be two partitions. Then, the homomorphism $\alpha_{k}:\Lambda
\rightarrow\mathbf{k}$ from Definition \ref{def.alphak} satisfies%
\begin{equation}
\alpha_{k}\left(  s_{\lambda/\mu}\right)  =\operatorname*{pet}\nolimits_{k}%
\left(  \lambda,\mu\right)  . \label{pf.thm.G.pieri.alks}%
\end{equation}

\end{lemma}

\begin{vershort}
\begin{proof}
[Proof of Lemma \ref{lem.alphak.slammu}.]Write the partitions $\lambda$ and
$\mu$ in the forms $\lambda=\left(  \lambda_{1},\lambda_{2},\ldots
,\lambda_{\ell}\right)  $ and $\mu=\left(  \mu_{1},\mu_{2},\ldots,\mu_{\ell
}\right)  $ for some $\ell\in\mathbb{N}$. Then, the equality
(\ref{pf.lem.petkrel.indep.short.1}) (which we showed in our proof of Lemma
\ref{lem.petkrel.indep}) yields%
\[
\alpha_{k}\left(  s_{\lambda/\mu}\right)  =\det\left(  \left(  \left[
0\leq\lambda_{i}-\mu_{j}-i+j<k\right]  \right)  _{1\leq i\leq\ell,\ 1\leq
j\leq\ell}\right)  =\operatorname*{pet}\nolimits_{k}\left(  \lambda
,\mu\right)
\]
(by the definition of $\operatorname*{pet}\nolimits_{k}\left(  \lambda
,\mu\right)  $). This proves Lemma \ref{lem.alphak.slammu}.
\end{proof}
\end{vershort}

\begin{verlong}
\begin{proof}
[Proof of Lemma \ref{lem.alphak.slammu}.]Write the partitions $\lambda$ and
$\mu$ in the forms $\lambda=\left(  \lambda_{1},\lambda_{2},\ldots
,\lambda_{\ell}\right)  $ and $\mu=\left(  \mu_{1},\mu_{2},\ldots,\mu_{\ell
}\right)  $ for some $\ell\in\mathbb{N}$\ \ \ \ \footnote{Such an $\ell$ can
always be found, since each of $\lambda$ and $\mu$ has only finitely many
nonzero entries.}. Then, the equality (\ref{pf.lem.petkrel.indep.1}) (which we
showed in our first proof of Lemma \ref{lem.petkrel.indep}) yields%
\[
\alpha_{k}\left(  s_{\lambda/\mu}\right)  =\det\left(  \left(  \left[
0\leq\lambda_{i}-\mu_{j}-i+j<k\right]  \right)  _{1\leq i\leq\ell,\ 1\leq
j\leq\ell}\right)  =\operatorname*{pet}\nolimits_{k}\left(  \lambda
,\mu\right)
\]
(by the definition of $\operatorname*{pet}\nolimits_{k}\left(  \lambda
,\mu\right)  $). This proves Lemma \ref{lem.alphak.slammu}.
\end{proof}
\end{verlong}

\subsection{\label{subsect.proofs.petk.-101}Proof of Proposition
\ref{prop.petkrel.-101}}

\begin{proof}
[Proof of Proposition \ref{prop.petkrel.-101}.]Write the partitions $\lambda$
and $\mu$ in the forms $\lambda=\left(  \lambda_{1},\lambda_{2},\ldots
,\lambda_{\ell}\right)  $ and $\mu=\left(  \mu_{1},\mu_{2},\ldots,\mu_{\ell
}\right)  $ for some $\ell\in\mathbb{N}$\ \ \ \ \footnote{Such an $\ell$ can
always be found, since each of $\lambda$ and $\mu$ has only finitely many
nonzero entries.}. The definition of $\operatorname*{pet}\nolimits_{k}\left(
\lambda,\mu\right)  $ yields
\begin{align}
\operatorname*{pet}\nolimits_{k}\left(  \lambda,\mu\right)   &  =\det\left(
\left(  \left[  \underbrace{0\leq\lambda_{i}-\mu_{j}-i+j<k}_{\text{This is
equivalent to }\mu_{j}-j\leq\lambda_{i}-i<\mu_{j}-j+k}\right]  \right)
_{1\leq i\leq\ell,\ 1\leq j\leq\ell}\right) \nonumber\\
&  =\det\left(  \left(  \left[  \mu_{j}-j\leq\lambda_{i}-i<\mu_{j}-j+k\right]
\right)  _{1\leq i\leq\ell,\ 1\leq j\leq\ell}\right)  .
\label{pf.prop.petkrel.-101.0}%
\end{align}

Let $B$ be the $\ell\times\ell$-matrix $\left(  \left[  \mu_{j}-j\leq
\lambda_{i}-i<\mu_{j}-j+k\right]  \right)  _{1\leq i\leq\ell,\ 1\leq j\leq
\ell}\in\mathbf{k}^{\ell\times\ell}$. Then, (\ref{pf.prop.petkrel.-101.0})
rewrites as follows:%
\begin{equation}
\operatorname*{pet}\nolimits_{k}\left(  \lambda,\mu\right)  =\det B.
\label{pf.prop.petkrel.-101.=detB}%
\end{equation}

We will use the concept of Petrie matrices (see \cite[Theorem 1]{GorWil74}).
Namely, a \emph{Petrie matrix} is a matrix whose entries all belong to
$\left\{  0,1\right\}  $ and such that the $1$'s in each column occur
consecutively (i.e., as a contiguous block). In other words, a Petrie matrix
is a matrix whose each column has the form
\begin{equation}
\left(  \underbrace{0,0,\ldots,0}_{a\text{ zeroes}},\underbrace{1,1,\ldots
,1}_{b\text{ ones}},\underbrace{0,0,\ldots,0}_{c\text{ zeroes}}\right)  ^{T}
\label{pf.prop.petk.-101.petcol}%
\end{equation}
for some nonnegative integers $a,b,c$ (where any of $a,b,c$ can be $0$). For
example, $\left(
\begin{array}
[c]{ccc}%
0 & 0 & 1\\
1 & 0 & 1\\
0 & 0 & 0
\end{array}
\right)  $ is a Petrie matrix, but $\left(
\begin{array}
[c]{ccc}%
0 & 1 & 0\\
1 & 0 & 0\\
0 & 1 & 1
\end{array}
\right)  $ is not.

A well-known result due to Fulkerson and Gross (first stated in \cite[\S 8]%
{FulGro65}\footnote{See \cite[Theorem 1]{GorWil74} for an explicit proof.})
says that if a square matrix $A$ is a Petrie matrix, then%
\begin{equation}
\det A\in\left\{  -1,0,1\right\}  . \label{pf.prop.petk.-101.1}%
\end{equation}

Now, we shall show that $B$ is a Petrie matrix.

Indeed, fix some $j\in\left\{  1,2,\ldots,\ell\right\}  $. We have
$\lambda_{1}\geq\lambda_{2}\geq\cdots\geq\lambda_{\ell}$ (since $\lambda$ is a
partition) and thus $\lambda_{1}-1>\lambda_{2}-2>\cdots>\lambda_{\ell}-\ell$.
In other words, the numbers $\lambda_{i}-i$ for $i\in\left\{  1,2,\ldots
,\ell\right\}  $ decrease as $i$ increases. Hence, the set of all
$i\in\left\{  1,2,\ldots,\ell\right\}  $ satisfying $\mu_{j}-j\leq\lambda
_{i}-i<\mu_{j}-j+k$ is a (possibly empty) integer interval\footnote{An
\emph{integer interval} means a subset of $\mathbb{Z}$ that has the form
$\left\{  a,a+1,\ldots,b\right\}  $ for some $a\in\mathbb{Z}$ and
$b\in\mathbb{Z}$. (If $a>b$, then this is the empty set.)}. Let us denote this
integer interval by $I_{j}$. Therefore, if we let $i$ range over $\left\{
1,2,\ldots,\ell\right\}  $, then the truth value $\left[  \mu_{j}-j\leq
\lambda_{i}-i<\mu_{j}-j+k\right]  $ will be $1$ for all $i\in I_{j}$, and $0$
for all other $i$. Since $I_{j}$ is an integer interval, this means that this
truth value will be $1$ when $i$ lies in a certain integer interval (namely,
$I_{j}$) and $0$ when $i$ lies outside it. In other words, the $j$-th column
of the matrix $B$ has a contiguous (but possibly empty) block of $1$'s (in the
rows corresponding to all $i\in I_{j}$), while all other entries of this
column are $0$ (because the entries of the $j$-th column of $B$ are precisely
these truth values $\left[  \mu_{j}-j\leq\lambda_{i}-i<\mu_{j}-j+k\right]  $
for all $i\in\left\{  1,2,\ldots,\ell\right\}  $). Therefore, this column has
the form (\ref{pf.prop.petk.-101.petcol}) for some nonnegative integers
$a,b,c$.

Now, forget that we fixed $j$. We thus have proved that for each $j\in\left\{
1,2,\ldots,\ell\right\}  $, the $j$-th column of $B$ has the form
(\ref{pf.prop.petk.-101.petcol}) for some nonnegative integers $a,b,c$. In
other words, each column of $B$ has this form. Hence, $B$ is a Petrie matrix
(by the definition of a Petrie matrix). Therefore, (\ref{pf.prop.petk.-101.1})
(applied to $A=B$) yields $\det B\in\left\{  -1,0,1\right\}  $. Thus,
(\ref{pf.prop.petkrel.-101.=detB}) becomes $\operatorname*{pet}\nolimits_{k}%
\left(  \lambda,\mu\right)  =\det B\in\left\{  -1,0,1\right\}  $. This proves
Proposition \ref{prop.petkrel.-101}.
\end{proof}

\subsection{\label{subsect.proofs.G.pieri.1st}First proof of Theorem
\ref{thm.G.pieri}}

We are now ready for our first proof of Theorem \ref{thm.G.pieri}:

\begin{proof}
[First proof of Theorem \ref{thm.G.pieri}.]We shall use the notations
$\mathbf{k}\left[  \left[  \mathbf{x}\right]  \right]  $, $\mathbf{k}\left[
\left[  \mathbf{x},\mathbf{y}\right]  \right]  $, $\mathbf{x}$, $\mathbf{y}$,
$f\left(  \mathbf{x}\right)  $ and $f\left(  \mathbf{y}\right)  $ introduced
in Subsection \ref{subsect.thms.coprod}. If $R$ is any commutative ring, then
$R\left[  \left[  \mathbf{y}\right]  \right]  $ shall denote the ring
$R\left[  \left[  y_{1},y_{2},y_{3},\ldots\right]  \right]  $ of formal power
series in the indeterminates $y_{1},y_{2},y_{3},\ldots$ over the ring $R$. We
will identify the ring $\mathbf{k}\left[  \left[  \mathbf{x},\mathbf{y}%
\right]  \right]  $ with the ring $\left(  \mathbf{k}\left[  \left[
\mathbf{x}\right]  \right]  \right)  \left[  \left[  \mathbf{y}\right]
\right]  =\left(  \mathbf{k}\left[  \left[  x_{1},x_{2},x_{3},\ldots\right]
\right]  \right)  \left[  \left[  y_{1},y_{2},y_{3},\ldots\right]  \right]  $.
Note that $\Lambda\subseteq\mathbf{k}\left[  \left[  \mathbf{x}\right]
\right]  $ and thus $\Lambda\left[  \left[  \mathbf{y}\right]  \right]
\subseteq\left(  \mathbf{k}\left[  \left[  \mathbf{x}\right]  \right]
\right)  \left[  \left[  \mathbf{y}\right]  \right]  =\mathbf{k}\left[
\left[  \mathbf{x},\mathbf{y}\right]  \right]  $. We equip the rings
$\mathbf{k}\left[  \left[  \mathbf{y}\right]  \right]  $, $\Lambda\left[
\left[  \mathbf{y}\right]  \right]  $ and $\mathbf{k}\left[  \left[
\mathbf{x},\mathbf{y}\right]  \right]  $ with the usual topologies that are
defined on rings of power series, where $\Lambda$ itself is equipped with the
discrete topology. This has the somewhat confusing consequence that
$\Lambda\left[  \left[  \mathbf{y}\right]  \right]  \subseteq\mathbf{k}\left[
\left[  \mathbf{x},\mathbf{y}\right]  \right]  $ is an inclusion of rings but
not of topological spaces; however, this will not cause us any trouble, since
all infinite sums in $\Lambda\left[  \left[  \mathbf{y}\right]  \right]  $ we
will consider (such as $\sum_{\lambda\in\operatorname*{Par}}s_{\lambda/\mu
}\left(  \mathbf{x}\right)  s_{\lambda}\left(  \mathbf{y}\right)  $ and
$\sum_{\lambda\in\operatorname*{Par}}h_{\lambda}\left(  \mathbf{x}\right)
m_{\lambda}\left(  \mathbf{y}\right)  $) will converge to the same value in
either topology.

We consider both $\mathbf{k}\left[  \left[  \mathbf{y}\right]  \right]  $ and
$\Lambda$ as subrings of $\Lambda\left[  \left[  \mathbf{y}\right]  \right]  $
(indeed, $\mathbf{k}\left[  \left[  \mathbf{y}\right]  \right]  $ embeds into
$\Lambda\left[  \left[  \mathbf{y}\right]  \right]  $ because $\mathbf{k}$ is
a subring of $\Lambda$, whereas $\Lambda$ embeds into $\Lambda\left[  \left[
\mathbf{y}\right]  \right]  $ because $\Lambda\left[  \left[  \mathbf{y}%
\right]  \right]  $ is a ring of power series over $\Lambda$).

In this proof, the word \textquotedblleft monomial\textquotedblright\ may
refer to a monomial in any set of variables (not necessarily in $x_{1}%
,x_{2},x_{3},\ldots$).

Recall the $\mathbf{k}$-algebra homomorphism $\alpha_{k}:\Lambda
\rightarrow\mathbf{k}$ from Definition \ref{def.alphak}. This $\mathbf{k}%
$-algebra homomorphism $\alpha_{k}:\Lambda\rightarrow\mathbf{k}$ induces a
$\mathbf{k}\left[  \left[  \mathbf{y}\right]  \right]  $-algebra homomorphism
$\alpha_{k}\left[  \left[  \mathbf{y}\right]  \right]  :\Lambda\left[  \left[
\mathbf{y}\right]  \right]  \rightarrow\mathbf{k}\left[  \left[
\mathbf{y}\right]  \right]  $, which is given by the formula%
\[
\left(  \alpha_{k}\left[  \left[  \mathbf{y}\right]  \right]  \right)  \left(
\sum_{\substack{\mathfrak{n}\text{ is a monomial}\\\text{in }y_{1},y_{2}%
,y_{3},\ldots}}f_{\mathfrak{n}}\mathfrak{n}\right)  =\sum
_{\substack{\mathfrak{n}\text{ is a monomial}\\\text{in }y_{1},y_{2}%
,y_{3},\ldots}}\alpha_{k}\left(  f_{\mathfrak{n}}\right)  \mathfrak{n}%
\]
for any family $\left(  f_{\mathfrak{n}}\right)  _{\mathfrak{n}\text{ is a
monomial in }y_{1},y_{2},y_{3},\ldots}$ of elements of $\Lambda$. This induced
$\mathbf{k}\left[  \left[  \mathbf{y}\right]  \right]  $-algebra homomorphism
$\alpha_{k}\left[  \left[  \mathbf{y}\right]  \right]  $ is $\mathbf{k}\left[
\left[  \mathbf{y}\right]  \right]  $-linear and continuous (with respect to
the usual topologies on the power series rings $\Lambda\left[  \left[
\mathbf{y}\right]  \right]  $ and $\mathbf{k}\left[  \left[  \mathbf{y}%
\right]  \right]  $), and thus preserves infinite $\mathbf{k}\left[  \left[
\mathbf{y}\right]  \right]  $-linear combinations. Moreover, it extends
$\alpha_{k}$ (that is, for any $f\in\Lambda$, we have $\left(  \alpha
_{k}\left[  \left[  \mathbf{y}\right]  \right]  \right)  \left(  f\right)
=\alpha_{k}\left(  f\right)  $).

\begin{noncompile}
Also, the map $\alpha_{k}\left[  \left[  \mathbf{y}\right]  \right]  $ leaves
every formal power series $g\in\mathbf{k}\left[  \left[  \mathbf{y}\right]
\right]  $ unchanged (that is, for any $g\in\mathbf{k}\left[  \left[
\mathbf{y}\right]  \right]  $, we have $\left(  \alpha_{k}\left[  \left[
\mathbf{y}\right]  \right]  \right)  \left(  g\right)  =g$).
\end{noncompile}

Recall the skew Schur functions $s_{\lambda/\mu}$ defined in Subsection
\ref{subsect.proofs.skew}. Also, recall the symmetric functions $h_{\lambda}$
defined in Definition \ref{def.hlam}. Theorem \ref{thm.our-cauchy} yields%
\begin{align*}
\sum_{\lambda\in\operatorname*{Par}}s_{\lambda/\mu}\left(  \mathbf{x}\right)
s_{\lambda}\left(  \mathbf{y}\right)   &  =s_{\mu}\left(  \mathbf{y}\right)
\cdot\sum_{\lambda\in\operatorname*{Par}}h_{\lambda}\left(  \mathbf{x}\right)
m_{\lambda}\left(  \mathbf{y}\right)  =\sum_{\lambda\in\operatorname*{Par}%
}s_{\mu}\left(  \mathbf{y}\right)  \underbrace{h_{\lambda}\left(
\mathbf{x}\right)  m_{\lambda}\left(  \mathbf{y}\right)  }_{=m_{\lambda
}\left(  \mathbf{y}\right)  h_{\lambda}\left(  \mathbf{x}\right)  }\\
&  =\sum_{\lambda\in\operatorname*{Par}}s_{\mu}\left(  \mathbf{y}\right)
m_{\lambda}\left(  \mathbf{y}\right)  \underbrace{h_{\lambda}\left(
\mathbf{x}\right)  }_{=h_{\lambda}}=\sum_{\lambda\in\operatorname*{Par}}%
s_{\mu}\left(  \mathbf{y}\right)  m_{\lambda}\left(  \mathbf{y}\right)
h_{\lambda}.
\end{align*}
Comparing this with%
\[
\sum_{\lambda\in\operatorname*{Par}}\underbrace{s_{\lambda/\mu}\left(
\mathbf{x}\right)  s_{\lambda}\left(  \mathbf{y}\right)  }_{=s_{\lambda
}\left(  \mathbf{y}\right)  s_{\lambda/\mu}\left(  \mathbf{x}\right)  }%
=\sum_{\lambda\in\operatorname*{Par}}s_{\lambda}\left(  \mathbf{y}\right)
\underbrace{s_{\lambda/\mu}\left(  \mathbf{x}\right)  }_{=s_{\lambda/\mu}%
}=\sum_{\lambda\in\operatorname*{Par}}s_{\lambda}\left(  \mathbf{y}\right)
s_{\lambda/\mu},
\]
we obtain%
\[
\sum_{\lambda\in\operatorname*{Par}}s_{\lambda}\left(  \mathbf{y}\right)
s_{\lambda/\mu}=\sum_{\lambda\in\operatorname*{Par}}s_{\mu}\left(
\mathbf{y}\right)  m_{\lambda}\left(  \mathbf{y}\right)  h_{\lambda}.
\]
Consider this as an equality in the ring $\Lambda\left[  \left[
\mathbf{y}\right]  \right]  =\Lambda\left[  \left[  y_{1},y_{2},y_{3}%
,\ldots\right]  \right]  $. Apply the map $\alpha_{k}\left[  \left[
\mathbf{y}\right]  \right]  :\Lambda\left[  \left[  \mathbf{y}\right]
\right]  \rightarrow\mathbf{k}\left[  \left[  \mathbf{y}\right]  \right]  $ to
both sides of this equality. We obtain%
\[
\left(  \alpha_{k}\left[  \left[  \mathbf{y}\right]  \right]  \right)  \left(
\sum_{\lambda\in\operatorname*{Par}}s_{\lambda}\left(  \mathbf{y}\right)
s_{\lambda/\mu}\right)  =\left(  \alpha_{k}\left[  \left[  \mathbf{y}\right]
\right]  \right)  \left(  \sum_{\lambda\in\operatorname*{Par}}s_{\mu}\left(
\mathbf{y}\right)  m_{\lambda}\left(  \mathbf{y}\right)  h_{\lambda}\right)
.
\]
Comparing this with%
\begin{align*}
&  \left(  \alpha_{k}\left[  \left[  \mathbf{y}\right]  \right]  \right)
\left(  \sum_{\lambda\in\operatorname*{Par}}s_{\lambda}\left(  \mathbf{y}%
\right)  s_{\lambda/\mu}\right) \\
&  =\sum_{\lambda\in\operatorname*{Par}}s_{\lambda}\left(  \mathbf{y}\right)
\cdot\underbrace{\left(  \alpha_{k}\left[  \left[  \mathbf{y}\right]  \right]
\right)  \left(  s_{\lambda/\mu}\right)  }_{\substack{=\alpha_{k}\left(
s_{\lambda/\mu}\right)  \\\text{(since }\alpha_{k}\left[  \left[
\mathbf{y}\right]  \right]  \text{ extends }\alpha_{k}\text{)}}}\\
&  \ \ \ \ \ \ \ \ \ \ \left(  \text{since the map }\alpha_{k}\left[  \left[
\mathbf{y}\right]  \right]  \text{ preserves infinite }\mathbf{k}\left[
\left[  \mathbf{y}\right]  \right]  \text{-linear combinations}\right) \\
&  =\sum_{\lambda\in\operatorname*{Par}}s_{\lambda}\left(  \mathbf{y}\right)
\cdot\underbrace{\alpha_{k}\left(  s_{\lambda/\mu}\right)  }%
_{\substack{=\operatorname*{pet}\nolimits_{k}\left(  \lambda,\mu\right)
\\\text{(by (\ref{pf.thm.G.pieri.alks}))}}}=\sum_{\lambda\in
\operatorname*{Par}}s_{\lambda}\left(  \mathbf{y}\right)  \cdot
\operatorname*{pet}\nolimits_{k}\left(  \lambda,\mu\right)  =\sum_{\lambda
\in\operatorname*{Par}}\operatorname*{pet}\nolimits_{k}\left(  \lambda
,\mu\right)  \cdot s_{\lambda}\left(  \mathbf{y}\right)  ,
\end{align*}
we obtain%
\begin{align*}
&  \sum_{\lambda\in\operatorname*{Par}}\operatorname*{pet}\nolimits_{k}\left(
\lambda,\mu\right)  \cdot s_{\lambda}\left(  \mathbf{y}\right) \\
&  =\left(  \alpha_{k}\left[  \left[  \mathbf{y}\right]  \right]  \right)
\left(  \sum_{\lambda\in\operatorname*{Par}}s_{\mu}\left(  \mathbf{y}\right)
m_{\lambda}\left(  \mathbf{y}\right)  h_{\lambda}\right)  =\sum_{\lambda
\in\operatorname*{Par}}s_{\mu}\left(  \mathbf{y}\right)  m_{\lambda}\left(
\mathbf{y}\right)  \underbrace{\left(  \alpha_{k}\left[  \left[
\mathbf{y}\right]  \right]  \right)  \left(  h_{\lambda}\right)
}_{\substack{=\alpha_{k}\left(  h_{\lambda}\right)  \\\text{(since }\alpha
_{k}\left[  \left[  \mathbf{y}\right]  \right]  \text{ extends }\alpha
_{k}\text{)}}}\\
&  \ \ \ \ \ \ \ \ \ \ \left(  \text{since the map }\alpha_{k}\left[  \left[
\mathbf{y}\right]  \right]  \text{ preserves infinite }\mathbf{k}\left[
\left[  \mathbf{y}\right]  \right]  \text{-linear combinations}\right) \\
&  =\sum_{\lambda\in\operatorname*{Par}}s_{\mu}\left(  \mathbf{y}\right)
m_{\lambda}\left(  \mathbf{y}\right)  \underbrace{\alpha_{k}\left(
h_{\lambda}\right)  }_{\substack{=\left[  \lambda_{i}<k\text{ for all
}i\right]  \\\text{(by (\ref{pf.thm.G.main.alkh}))}}}=\sum_{\lambda
\in\operatorname*{Par}}s_{\mu}\left(  \mathbf{y}\right)  m_{\lambda}\left(
\mathbf{y}\right)  \cdot\left[  \lambda_{i}<k\text{ for all }i\right] \\
&  =\sum_{\lambda\in\operatorname*{Par}}\left[  \lambda_{i}<k\text{ for all
}i\right]  \cdot s_{\mu}\left(  \mathbf{y}\right)  m_{\lambda}\left(
\mathbf{y}\right)  .
\end{align*}
Renaming the indeterminates $\mathbf{y}=\left(  y_{1},y_{2},y_{3}%
,\ldots\right)  $ as $\mathbf{x}=\left(  x_{1},x_{2},x_{3},\ldots\right)  $ on
both sides of this equality, we obtain%
\begin{align*}
&  \sum_{\lambda\in\operatorname*{Par}}\operatorname*{pet}\nolimits_{k}\left(
\lambda,\mu\right)  \cdot s_{\lambda}\left(  \mathbf{x}\right) \\
&  =\sum_{\lambda\in\operatorname*{Par}}\left[  \lambda_{i}<k\text{ for all
}i\right]  \cdot\underbrace{s_{\mu}\left(  \mathbf{x}\right)  }_{=s_{\mu}%
}\underbrace{m_{\lambda}\left(  \mathbf{x}\right)  }_{=m_{\lambda}}%
=\sum_{\lambda\in\operatorname*{Par}}\left[  \lambda_{i}<k\text{ for all
}i\right]  \cdot s_{\mu}m_{\lambda}\\
&  =\sum_{\substack{\lambda\in\operatorname*{Par};\\\lambda_{i}<k\text{ for
all }i}}\underbrace{\left[  \lambda_{i}<k\text{ for all }i\right]
}_{\substack{=1\\\text{(since }\lambda_{i}<k\text{ for all }i\text{)}}}\cdot
s_{\mu}m_{\lambda}+\sum_{\substack{\lambda\in\operatorname*{Par};\\\text{not
}\left(  \lambda_{i}<k\text{ for all }i\right)  }}\underbrace{\left[
\lambda_{i}<k\text{ for all }i\right]  }_{\substack{=0\\\text{(since we don't
have \textquotedblleft}\lambda_{i}<k\text{ for all }i\text{\textquotedblright%
)}}}\cdot s_{\mu}m_{\lambda}\\
&  \ \ \ \ \ \ \ \ \ \ \left(  \text{since any }\lambda\in\operatorname*{Par}%
\text{ satisfies either }\left(  \lambda_{i}<k\text{ for all }i\right)  \text{
or not }\left(  \lambda_{i}<k\text{ for all }i\right)  \right) \\
&  =\sum_{\substack{\lambda\in\operatorname*{Par};\\\lambda_{i}<k\text{ for
all }i}}s_{\mu}m_{\lambda}+\underbrace{\sum_{\substack{\lambda\in
\operatorname*{Par};\\\text{not }\left(  \lambda_{i}<k\text{ for all
}i\right)  }}0s_{\mu}m_{\lambda}}_{=0}=\sum_{\substack{\lambda\in
\operatorname*{Par};\\\lambda_{i}<k\text{ for all }i}}s_{\mu}m_{\lambda}.
\end{align*}
Comparing this with%
\[
G\left(  k\right)  \cdot s_{\mu}=s_{\mu}\cdot\underbrace{G\left(  k\right)
}_{\substack{=\sum_{\substack{\lambda\in\operatorname*{Par};\\\lambda
_{i}<k\text{ for all }i}}m_{\lambda}\\\text{(by Proposition
\ref{prop.G.basics} \textbf{(b)})}}}=s_{\mu}\cdot\sum_{\substack{\lambda
\in\operatorname*{Par};\\\lambda_{i}<k\text{ for all }i}}m_{\lambda}%
=\sum_{\substack{\lambda\in\operatorname*{Par};\\\lambda_{i}<k\text{ for all
}i}}s_{\mu}m_{\lambda},
\]
we obtain
\[
G\left(  k\right)  \cdot s_{\mu}=\sum_{\lambda\in\operatorname*{Par}%
}\operatorname*{pet}\nolimits_{k}\left(  \lambda,\mu\right)  \cdot
\underbrace{s_{\lambda}\left(  \mathbf{x}\right)  }_{=s_{\lambda}}%
=\sum_{\lambda\in\operatorname*{Par}}\operatorname*{pet}\nolimits_{k}\left(
\lambda,\mu\right)  s_{\lambda}.
\]
This proves Theorem \ref{thm.G.pieri}.
\end{proof}

\subsection{\label{subsect.proofs.G.pieri.2nd}Second proof of Theorem
\ref{thm.G.pieri}}

\begin{vershort}
Our second proof of Theorem \ref{thm.G.pieri} will rely on \cite[\S 2.6]%
{GriRei} and specifically on the notion of alternants. We shall give only a
sketch of it; the reader can consult \cite{verlong} for the full version.

\begin{proof}
[Second proof of Theorem \ref{thm.G.pieri} (sketched).]If $f\in\mathbf{k}%
\left[  \left[  x_{1},x_{2},x_{3},\ldots\right]  \right]  $ is any formal
power series, and if $\ell\in\mathbb{N}$, then $f\left(  x_{1},x_{2}%
,\ldots,x_{\ell}\right)  $ shall denote the formal power series%
\[
f\left(  x_{1},x_{2},\ldots,x_{\ell},0,0,0,\ldots\right)  \in\mathbf{k}\left[
\left[  x_{1},x_{2},\ldots,x_{\ell}\right]  \right]
\]
that is obtained by substituting $0,0,0,\ldots$ for the variables $x_{\ell
+1},x_{\ell+2},x_{\ell+3},\ldots$ in $f$. Equivalently, $f\left(  x_{1}%
,x_{2},\ldots,x_{\ell}\right)  $ can be obtained from $f$ by removing all
monomials that contain any of the variables $x_{\ell+1},x_{\ell+2},x_{\ell
+3},\ldots$. This makes it clear that any $f\in\mathbf{k}\left[  \left[
x_{1},x_{2},x_{3},\ldots\right]  \right]  $ satisfies%
\begin{equation}
f=\lim\limits_{\ell\rightarrow\infty}f\left(  x_{1},x_{2},\ldots,x_{\ell
}\right)  \label{pf.thm.G.pieri.short.lim}%
\end{equation}
(where the limit is taken with respect to the usual topology on $\mathbf{k}%
\left[  \left[  x_{1},x_{2},x_{3},\ldots\right]  \right]  $).

Let $\ell\left(  \mu\right)  $ denote the number of parts of $\mu$. Fix an
$\ell\in\mathbb{N}$ such that $\ell\geq\ell\left(  \mu\right)  $. We shall
show that%
\begin{align}
&  \left(  G\left(  k\right)  \right)  \left(  x_{1},x_{2},\ldots,x_{\ell
}\right)  \cdot s_{\mu}\left(  x_{1},x_{2},\ldots,x_{\ell}\right) \nonumber\\
&  =\sum_{\lambda\in\operatorname*{Par}}\operatorname*{pet}\nolimits_{k}%
\left(  \lambda,\mu\right)  s_{\lambda}\left(  x_{1},x_{2},\ldots,x_{\ell
}\right)  . \label{pf.thm.G.pieri.short.lvars}%
\end{align}
Once this is done, the usual \textquotedblleft let $\ell$ tend to $\infty
$\textquotedblright\ argument (analogous to \cite[proof of Corollary
2.6.11]{GriRei}) will yield the validity of Theorem \ref{thm.G.pieri}.

Let $P_{\ell}$ denote the set of all partitions with at most $\ell$ parts. Any
partition $\lambda\in P_{\ell}$ satisfies $\lambda=\left(  \lambda_{1}%
,\lambda_{2},\ldots,\lambda_{\ell}\right)  $, and thus can be regarded as an
$\ell$-tuple of nonnegative integers. In other words, $P_{\ell}\subseteq
\mathbb{N}^{\ell}$. More precisely, the partitions $\lambda\in P_{\ell}$ are
precisely the $\ell$-tuples $\left(  \lambda_{1},\lambda_{2},\ldots
,\lambda_{\ell}\right)  \in\mathbb{N}^{\ell}$ satisfying $\lambda_{1}%
\geq\lambda_{2}\geq\cdots\geq\lambda_{\ell}$.

Define the $\ell$-tuple $\rho=\left(  \ell-1,\ell-2,\ldots,0\right)
\in\mathbb{N}^{\ell}$.

For any $\ell$-tuple $\alpha\in\mathbb{N}^{\ell}$ and each $i\in\left\{
1,2,\ldots,\ell\right\}  $, we shall write $\alpha_{i}$ for the $i$-th entry
of $\alpha$.

For any $\ell$-tuple $\alpha\in\mathbb{N}^{\ell}$, we let $\mathbf{x}^{\alpha
}$ denote the monomial $x_{1}^{\alpha_{1}}x_{2}^{\alpha_{2}}\cdots x_{\ell
}^{\alpha_{\ell}}$.

Let $\mathfrak{S}_{\ell}$ denote the symmetric group of the set $\left\{
1,2,\ldots,\ell\right\}  $. This group $\mathfrak{S}_{\ell}$ acts by
$\mathbf{k}$-algebra homomorphisms on the polynomial ring $\mathbf{k}\left[
x_{1},x_{2},\ldots,x_{\ell}\right]  $. It also acts on the set $\mathbb{N}%
^{\ell}$ by permuting the entries of an $\ell$-tuple; namely,%
\[
\sigma\cdot\beta=\left(  \beta_{\sigma^{-1}\left(  1\right)  },\beta
_{\sigma^{-1}\left(  2\right)  },\ldots,\beta_{\sigma^{-1}\left(  \ell\right)
}\right)  \ \ \ \ \ \ \ \ \ \ \text{for any }\sigma\in\mathfrak{S}_{\ell
}\text{ and }\beta\in\mathbb{N}^{\ell}.
\]

For any $\sigma\in\mathfrak{S}_{\ell}$, we let $\left(  -1\right)  ^{\sigma}$
denote the sign of the permutation $\sigma$.

If $\alpha\in\mathbb{N}^{\ell}$ is any $\ell$-tuple, then we define the
polynomial $a_{\alpha}\in\mathbf{k}\left[  x_{1},x_{2},\ldots,x_{\ell}\right]
$ (called the $\alpha$\emph{-alternant}) by%
\[
a_{\alpha}=\sum_{\sigma\in\mathfrak{S}_{\ell}}\left(  -1\right)  ^{\sigma
}\sigma\left(  \mathbf{x}^{\alpha}\right)  =\det\left(  \left(  x_{j}%
^{\alpha_{i}}\right)  _{1\leq i\leq\ell,\ 1\leq j\leq\ell}\right)  .
\]

We define addition of $\ell$-tuples $\alpha\in\mathbb{N}^{\ell}$ entrywise (so
that $\left(  \alpha+\beta\right)  _{i}=\alpha_{i}+\beta_{i}$ for every
$\alpha,\beta\in\mathbb{N}^{\ell}$ and $i\in\left\{  1,2,\ldots,\ell\right\}
$). Thus, $\lambda+\rho\in\mathbb{N}^{\ell}$ for each $\lambda\in P_{\ell}$
(since $\lambda\in P_{\ell}\subseteq\mathbb{N}^{\ell}$).

It is known (\cite[Corollary 2.6.7]{GriRei}) that%
\begin{equation}
s_{\lambda}\left(  x_{1},x_{2},\ldots,x_{\ell}\right)  =\dfrac{a_{\lambda
+\rho}}{a_{\rho}} \label{pf.thm.G.pieri.short.slam}%
\end{equation}
for every $\lambda\in P_{\ell}$.

The partition $\mu$ has at most $\ell$ parts (since $\ell\geq\ell\left(
\mu\right)  $). In other words, $\mu\in P_{\ell}$.

Now, define $\alpha\in\mathbb{N}^{\ell}$ by $\alpha=\mu+\rho$. Proposition
\ref{prop.G.basics} \textbf{(b)} yields%
\[
G\left(  k\right)  =\prod_{i=1}^{\infty}\left(  x_{i}^{0}+x_{i}^{1}%
+\cdots+x_{i}^{k-1}\right)  .
\]
Substituting $0,0,0,\ldots$ for $x_{\ell+1},x_{\ell+2},x_{\ell+3},\ldots$ in
this equality, we obtain%
\begin{align}
&  \left(  G\left(  k\right)  \right)  \left(  x_{1},x_{2},\ldots,x_{\ell
}\right) \nonumber\\
&  =\left(  \prod_{i=1}^{\ell}\left(  x_{i}^{0}+x_{i}^{1}+\cdots+x_{i}%
^{k-1}\right)  \right)  \cdot\left(  \prod_{i=\ell+1}^{\infty}%
\underbrace{\left(  0^{0}+0^{1}+\cdots+0^{k-1}\right)  }%
_{\substack{=1\\\text{(since }0^{0}=1\text{ and }0^{j}=0\text{ for all
}j>0\text{)}}}\right) \nonumber\\
&  =\prod_{i=1}^{\ell}\left(  x_{i}^{0}+x_{i}^{1}+\cdots+x_{i}^{k-1}\right)
\label{pf.thm.G.pieri.short.Gkxl1}\\
&  =\sum_{\left(  j_{1},j_{2},\ldots,j_{\ell}\right)  \in\left\{
0,1,\ldots,k-1\right\}  ^{\ell}}x_{1}^{j_{1}}x_{2}^{j_{2}}\cdots x_{\ell
}^{j_{\ell}}\ \ \ \ \ \ \ \ \ \ \left(  \text{by the product rule}\right)
\nonumber\\
&  =\sum_{\substack{\beta\in\mathbb{N}^{\ell};\\\beta_{i}<k\text{ for all }%
i}}\mathbf{x}^{\beta}\label{pf.thm.G.pieri.short.Gkxl2}\\
&  \ \ \ \ \ \ \ \ \ \ \left(  \text{here, we have substituted }\beta\text{
for }\left(  j_{1},j_{2},\ldots,j_{\ell}\right)  \text{ in the sum}\right)
.\nonumber
\end{align}
(Here and in the rest of this proof, \textquotedblleft for all $i$%
\textquotedblright\ means \textquotedblleft for all $i\in\left\{
1,2,\ldots,\ell\right\}  $\textquotedblright.)

From (\ref{pf.thm.G.pieri.short.Gkxl1}), we see that $\left(  G\left(
k\right)  \right)  \left(  x_{1},x_{2},\ldots,x_{\ell}\right)  $ is a
polynomial in $\mathbf{k}\left[  x_{1},x_{2},\ldots,x_{\ell}\right]  $ (not
merely a power series). From (\ref{pf.thm.G.pieri.short.Gkxl1}), we moreover
see that this polynomial $\left(  G\left(  k\right)  \right)  \left(
x_{1},x_{2},\ldots,x_{\ell}\right)  \in\mathbf{k}\left[  x_{1},x_{2}%
,\ldots,x_{\ell}\right]  $ is invariant under the action of $\mathfrak{S}%
_{\ell}$.

But from $a_{\alpha}=\sum_{\sigma\in\mathfrak{S}_{\ell}}\left(  -1\right)
^{\sigma}\sigma\left(  \mathbf{x}^{\alpha}\right)  $, we obtain%
\begin{align}
\left(  G\left(  k\right)  \right)  \left(  x_{1},x_{2},\ldots,x_{\ell
}\right)  \cdot a_{\alpha}  &  =\left(  G\left(  k\right)  \right)  \left(
x_{1},x_{2},\ldots,x_{\ell}\right)  \cdot\sum_{\sigma\in\mathfrak{S}_{\ell}%
}\left(  -1\right)  ^{\sigma}\sigma\left(  \mathbf{x}^{\alpha}\right)
\nonumber\\
&  =\sum_{\sigma\in\mathfrak{S}_{\ell}}\left(  -1\right)  ^{\sigma
}\underbrace{\left(  G\left(  k\right)  \right)  \left(  x_{1},x_{2}%
,\ldots,x_{\ell}\right)  }_{\substack{=\sigma\left(  \left(  G\left(
k\right)  \right)  \left(  x_{1},x_{2},\ldots,x_{\ell}\right)  \right)
\\\text{(since the polynomial }\left(  G\left(  k\right)  \right)  \left(
x_{1},x_{2},\ldots,x_{\ell}\right)  \\\text{is invariant under the action of
}\mathfrak{S}_{\ell}\text{)}}}\cdot\sigma\left(  \mathbf{x}^{\alpha}\right)
\nonumber\\
&  =\sum_{\sigma\in\mathfrak{S}_{\ell}}\left(  -1\right)  ^{\sigma
}\underbrace{\sigma\left(  \left(  G\left(  k\right)  \right)  \left(
x_{1},x_{2},\ldots,x_{\ell}\right)  \right)  \cdot\sigma\left(  \mathbf{x}%
^{\alpha}\right)  }_{\substack{=\sigma\left(  \left(  G\left(  k\right)
\right)  \left(  x_{1},x_{2},\ldots,x_{\ell}\right)  \cdot\mathbf{x}^{\alpha
}\right)  \\\text{(since }\mathfrak{S}_{\ell}\text{ acts on }\mathbf{k}\left[
x_{1},x_{2},\ldots,x_{\ell}\right]  \\\text{by }\mathbf{k}\text{-algebra
homomorphisms)}}}\nonumber\\
&  =\sum_{\sigma\in\mathfrak{S}_{\ell}}\left(  -1\right)  ^{\sigma}%
\sigma\left(  \left(  G\left(  k\right)  \right)  \left(  x_{1},x_{2}%
,\ldots,x_{\ell}\right)  \cdot\mathbf{x}^{\alpha}\right) \nonumber\\
&  =\sum_{\sigma\in\mathfrak{S}_{\ell}}\left(  -1\right)  ^{\sigma}%
\sigma\left(  \sum_{\substack{\beta\in\mathbb{N}^{\ell};\\\beta_{i}<k\text{
for all }i}}\mathbf{x}^{\beta}\cdot\mathbf{x}^{\alpha}\right)
\ \ \ \ \ \ \ \ \ \ \left(  \text{by (\ref{pf.thm.G.pieri.short.Gkxl2}%
)}\right) \nonumber\\
&  =\sum_{\substack{\beta\in\mathbb{N}^{\ell};\\\beta_{i}<k\text{ for all }%
i}}\sum_{\sigma\in\mathfrak{S}_{\ell}}\left(  -1\right)  ^{\sigma}%
\sigma\left(  \underbrace{\mathbf{x}^{\beta}\cdot\mathbf{x}^{\alpha}%
}_{\substack{=\mathbf{x}^{\alpha}\cdot\mathbf{x}^{\beta}\\=\mathbf{x}%
^{\alpha+\beta}}}\right) \nonumber\\
&  =\sum_{\substack{\beta\in\mathbb{N}^{\ell};\\\beta_{i}<k\text{ for all }%
i}}\underbrace{\sum_{\sigma\in\mathfrak{S}_{\ell}}\left(  -1\right)  ^{\sigma
}\sigma\left(  \mathbf{x}^{\alpha+\beta}\right)  }_{\substack{=a_{\alpha
+\beta}\\\text{(by the definition of }a_{\alpha+\beta}\text{)}}}=\sum
_{\substack{\beta\in\mathbb{N}^{\ell};\\\beta_{i}<k\text{ for all }%
i}}a_{\alpha+\beta}\nonumber\\
&  =\sum_{\substack{\gamma\in\mathbb{N}^{\ell};\\0\leq\gamma_{i}-\alpha
_{i}<k\text{ for all }i}}a_{\gamma} \label{pf.thm.G.pieri.short.prod1}%
\end{align}
(here, we have substituted $\gamma$ for $\alpha+\beta$ in the sum).

It is well-known (and easy to check using the properties of
determinants\footnote{specifically: using the fact that a square matrix with
two equal rows always has determinant $0$}) that if an $\ell$-tuple $\gamma
\in\mathbb{N}^{\ell}$ has two equal entries, then%
\begin{equation}
a_{\gamma}=0. \label{pf.thm.G.pieri.short.agam=0}%
\end{equation}
Moreover, any $\ell$-tuple $\gamma\in\mathbb{N}^{\ell}$ and any $\sigma
\in\mathfrak{S}_{\ell}$ satisfy%
\begin{equation}
a_{\sigma\cdot\gamma}=\left(  -1\right)  ^{\sigma}\cdot a_{\gamma}.
\label{pf.thm.G.pieri.short.asiggam=}%
\end{equation}
(This, too, follows from the properties of determinants\footnote{specifically:
using the fact that permuting the rows of a square matrix results in its
determinant getting multiplied by the sign of the permutation}.)

Let $SP_{\ell}$ denote the set of all $\ell$-tuples $\delta\in\mathbb{N}%
^{\ell}$ such that $\delta_{1}>\delta_{2}>\cdots>\delta_{\ell}$. Then, the map%
\begin{align}
P_{\ell}  &  \rightarrow SP_{\ell},\nonumber\\
\lambda &  \mapsto\lambda+\rho\label{pf.thm.G.pieri.short.PSPbij}%
\end{align}
is a bijection.

If an $\ell$-tuple $\gamma\in\mathbb{N}^{\ell}$ has no two equal entries, then
$\gamma$ can be uniquely written in the form $\sigma\cdot\delta$ for some
$\sigma\in\mathfrak{S}_{\ell}$ and some $\delta\in SP_{\ell}$ (indeed,
$\delta$ is the result of sorting $\gamma$ into decreasing order, while
$\sigma$ is the permutation that achieves this sorting). In other words, the
map%
\begin{align}
\mathfrak{S}_{\ell}\times SP_{\ell}  &  \rightarrow\left\{  \gamma
\in\mathbb{N}^{\ell}\ \mid\ \text{the }\ell\text{-tuple }\gamma\text{ has no
two equal entries}\right\}  ,\nonumber\\
\left(  \sigma,\delta\right)   &  \mapsto\sigma\cdot\delta
\label{pf.thm.G.pieri.short.SSPbbij}%
\end{align}
is a bijection.

Now, (\ref{pf.thm.G.pieri.short.prod1}) becomes%
\begin{align}
&  \left(  G\left(  k\right)  \right)  \left(  x_{1},x_{2},\ldots,x_{\ell
}\right)  \cdot a_{\alpha}\nonumber\\
&  =\sum_{\substack{\gamma\in\mathbb{N}^{\ell};\\0\leq\gamma_{i}-\alpha
_{i}<k\text{ for all }i}}a_{\gamma}\nonumber\\
&  =\sum_{\substack{\gamma\in\mathbb{N}^{\ell};\\0\leq\gamma_{i}-\alpha
_{i}<k\text{ for all }i;\\\text{the }\ell\text{-tuple }\gamma\text{ has no two
equal entries}}}a_{\gamma}+\sum_{\substack{\gamma\in\mathbb{N}^{\ell}%
;\\0\leq\gamma_{i}-\alpha_{i}<k\text{ for all }i;\\\text{the }\ell\text{-tuple
}\gamma\text{ has two equal entries}}}\underbrace{a_{\gamma}}%
_{\substack{=0\\\text{(by (\ref{pf.thm.G.pieri.short.agam=0}))}}}\nonumber\\
&  =\sum_{\substack{\gamma\in\mathbb{N}^{\ell};\\0\leq\gamma_{i}-\alpha
_{i}<k\text{ for all }i;\\\text{the }\ell\text{-tuple }\gamma\text{ has no two
equal entries}}}a_{\gamma}\nonumber\\
&  =\sum_{\substack{\gamma\in\mathbb{N}^{\ell};\\\text{the }\ell\text{-tuple
}\gamma\text{ has no two equal entries}}}\left[  0\leq\gamma_{i}-\alpha
_{i}<k\text{ for all }i\right]  \cdot a_{\gamma}\nonumber\\
&  =\underbrace{\sum_{\left(  \sigma,\delta\right)  \in\mathfrak{S}_{\ell
}\times SP_{\ell}}}_{=\sum_{\delta\in SP_{\ell}}\ \ \sum_{\sigma
\in\mathfrak{S}_{\ell}}}\left[  0\leq\underbrace{\left(  \sigma\cdot
\delta\right)  _{i}}_{=\delta_{\sigma^{-1}\left(  i\right)  }}-\alpha
_{i}<k\text{ for all }i\right]  \cdot\underbrace{a_{\sigma\cdot\delta}%
}_{\substack{=\left(  -1\right)  ^{\sigma}a_{\delta}\\\text{(by
(\ref{pf.thm.G.pieri.short.asiggam=}))}}}\nonumber\\
&  \ \ \ \ \ \ \ \ \ \ \left(
\begin{array}
[c]{c}%
\text{here, we have substituted }\sigma\cdot\delta\text{ for }\gamma\text{ in
the sum,}\\
\text{since the map (\ref{pf.thm.G.pieri.short.SSPbbij}) is a bijection}%
\end{array}
\right) \nonumber\\
&  =\sum_{\delta\in SP_{\ell}}\ \ \sum_{\sigma\in\mathfrak{S}_{\ell}%
}\underbrace{\left[  0\leq\delta_{\sigma^{-1}\left(  i\right)  }-\alpha
_{i}<k\text{ for all }i\right]  }_{\substack{=\prod_{i=1}^{\ell}\left[
0\leq\delta_{\sigma^{-1}\left(  i\right)  }-\alpha_{i}<k\right]
\\=\prod_{i=1}^{\ell}\left[  0\leq\delta_{i}-\alpha_{\sigma\left(  i\right)
}<k\right]  \\\text{(here, we have substituted }\sigma\left(  i\right)
\\\text{for }i\text{ in the product, since }\sigma\text{ is a bijection)}%
}}\cdot\left(  -1\right)  ^{\sigma}a_{\delta}\nonumber\\
&  =\sum_{\delta\in SP_{\ell}}\ \ \underbrace{\sum_{\sigma\in\mathfrak{S}%
_{\ell}}\left(  \prod_{i=1}^{\ell}\left[  0\leq\delta_{i}-\alpha
_{\sigma\left(  i\right)  }<k\right]  \right)  \cdot\left(  -1\right)
^{\sigma}}_{\substack{=\sum_{\sigma\in\mathfrak{S}_{\ell}}\left(  -1\right)
^{\sigma}\prod_{i=1}^{\ell}\left[  0\leq\delta_{i}-\alpha_{\sigma\left(
i\right)  }<k\right]  \\=\det\left(  \left(  \left[  0\leq\delta_{i}%
-\alpha_{j}<k\right]  \right)  _{1\leq i\leq\ell,\ 1\leq j\leq\ell}\right)
\\\text{(by the definition of a determinant)}}}a_{\delta}\nonumber\\
&  =\sum_{\delta\in SP_{\ell}}\det\left(  \left(  \left[  0\leq\delta
_{i}-\alpha_{j}<k\right]  \right)  _{1\leq i\leq\ell,\ 1\leq j\leq\ell
}\right)  a_{\delta}\nonumber\\
&  =\sum_{\lambda\in P_{\ell}}\det\left(  \left(  \left[  0\leq\left(
\lambda+\rho\right)  _{i}-\alpha_{j}<k\right]  \right)  _{1\leq i\leq
\ell,\ 1\leq j\leq\ell}\right)  a_{\lambda+\rho}
\label{pf.thm.G.pieri.short.prod2}%
\end{align}
(here, we have substituted $\lambda+\rho$ for $\delta$ in the sum, since the
map (\ref{pf.thm.G.pieri.short.PSPbij}) is a bijection).

Using the definitions of $\rho$ and $\alpha$, it is easy to see that
\begin{equation}
\left(  \lambda+\rho\right)  _{i}-\alpha_{j}=\lambda_{i}-\mu_{j}-i+j
\label{pf.thm.G.pieri.short.dif}%
\end{equation}
for every $\lambda\in P_{\ell}$ and every $i,j\in\left\{  1,2,\ldots
,\ell\right\}  $.

Now, (\ref{pf.thm.G.pieri.short.slam}) (applied to $\lambda=\mu$) yields%
\[
s_{\mu}\left(  x_{1},x_{2},\ldots,x_{\ell}\right)  =\dfrac{a_{\mu+\rho}%
}{a_{\rho}}=\dfrac{a_{\alpha}}{a_{\rho}}%
\]
(since $\mu+\rho=\alpha$). Multiplying this equality by $\left(  G\left(
k\right)  \right)  \left(  x_{1},x_{2},\ldots,x_{\ell}\right)  $, we find%
\begin{align*}
&  \left(  G\left(  k\right)  \right)  \left(  x_{1},x_{2},\ldots,x_{\ell
}\right)  \cdot s_{\mu}\left(  x_{1},x_{2},\ldots,x_{\ell}\right) \\
&  =\left(  G\left(  k\right)  \right)  \left(  x_{1},x_{2},\ldots,x_{\ell
}\right)  \cdot\dfrac{a_{\alpha}}{a_{\rho}}\\
&  =\dfrac{1}{a_{\rho}}\cdot\underbrace{\left(  G\left(  k\right)  \right)
\left(  x_{1},x_{2},\ldots,x_{\ell}\right)  \cdot a_{\alpha}}_{\substack{=\sum
_{\lambda\in P_{\ell}}\det\left(  \left(  \left[  0\leq\left(  \lambda
+\rho\right)  _{i}-\alpha_{j}<k\right]  \right)  _{1\leq i\leq\ell,\ 1\leq
j\leq\ell}\right)  a_{\lambda+\rho}\\\text{(by
(\ref{pf.thm.G.pieri.short.prod2}))}}}\\
&  =\sum_{\lambda\in P_{\ell}}\det\left(  \left(  \left[  0\leq
\underbrace{\left(  \lambda+\rho\right)  _{i}-\alpha_{j}}_{\substack{=\lambda
_{i}-\mu_{j}-i+j\\\text{(by (\ref{pf.thm.G.pieri.short.dif}))}}}<k\right]
\right)  _{1\leq i\leq\ell,\ 1\leq j\leq\ell}\right)  \underbrace{\dfrac
{a_{\lambda+\rho}}{a_{\rho}}}_{\substack{=s_{\lambda}\left(  x_{1}%
,x_{2},\ldots,x_{\ell}\right)  \\\text{(by (\ref{pf.thm.G.pieri.short.slam}%
))}}}\\
&  =\sum_{\lambda\in P_{\ell}}\underbrace{\det\left(  \left(  \left[
0\leq\lambda_{i}-\mu_{j}-i+j<k\right]  \right)  _{1\leq i\leq\ell,\ 1\leq
j\leq\ell}\right)  }_{\substack{=\operatorname*{pet}\nolimits_{k}\left(
\lambda,\mu\right)  \\\text{(by the definition of }\operatorname*{pet}%
\nolimits_{k}\left(  \lambda,\mu\right)  \text{)}}}\cdot s_{\lambda}\left(
x_{1},x_{2},\ldots,x_{\ell}\right) \\
&  =\sum_{\lambda\in P_{\ell}}\operatorname*{pet}\nolimits_{k}\left(
\lambda,\mu\right)  s_{\lambda}\left(  x_{1},x_{2},\ldots,x_{\ell}\right)
=\sum_{\lambda\in\operatorname*{Par}}\operatorname*{pet}\nolimits_{k}\left(
\lambda,\mu\right)  s_{\lambda}\left(  x_{1},x_{2},\ldots,x_{\ell}\right)  ,
\end{align*}
where the last equality sign is owed to the well-known fact (see, e.g.,
\cite[Exercise 2.3.8(b)]{GriRei}) that every $\lambda\in\operatorname*{Par}%
\setminus P_{\ell}$ (that is, every partition $\lambda$ having more than
$\ell$ parts) satisfies $s_{\lambda}\left(  x_{1},x_{2},\ldots,x_{\ell
}\right)  =0$. Thus, (\ref{pf.thm.G.pieri.short.lvars}) has been proved.

Forget that we fixed $\ell$. Thus, we have proved
(\ref{pf.thm.G.pieri.short.lvars}) for each $\ell\in\mathbb{N}$ that satisfies
$\ell\geq\ell\left(  \mu\right)  $. As we mentioned above, we can now use a
standard limiting argument (using (\ref{pf.thm.G.pieri.short.lim})) to obtain
the claim of Theorem \ref{thm.G.pieri}.
\end{proof}
\end{vershort}

\begin{verlong}
Our second proof of Theorem \ref{thm.G.pieri} will rely on \cite[\S 2.6]%
{GriRei} and specifically on the notion of alternants:

\begin{proof}
[Second proof of Theorem \ref{thm.G.pieri}.]If $f\in\mathbf{k}\left[  \left[
x_{1},x_{2},x_{3},\ldots\right]  \right]  $ is any formal power series, and if
$\ell\in\mathbb{N}$, then $f\left(  x_{1},x_{2},\ldots,x_{\ell}\right)  $
shall denote the formal power series%
\[
f\left(  x_{1},x_{2},\ldots,x_{\ell},0,0,0,\ldots\right)  \in\mathbf{k}\left[
\left[  x_{1},x_{2},\ldots,x_{\ell}\right]  \right]
\]
that is obtained by substituting $0,0,0,\ldots$ for the variables $x_{\ell
+1},x_{\ell+2},x_{\ell+3},\ldots$ in $f$. Equivalently, $f\left(  x_{1}%
,x_{2},\ldots,x_{\ell}\right)  $ can be obtained from $f$ by removing all
monomials that contain any of the variables $x_{\ell+1},x_{\ell+2},x_{\ell
+3},\ldots$. This makes it clear that any $f\in\mathbf{k}\left[  \left[
x_{1},x_{2},x_{3},\ldots\right]  \right]  $ satisfies%
\begin{equation}
f=\lim\limits_{\ell\rightarrow\infty}f\left(  x_{1},x_{2},\ldots,x_{\ell
}\right)  \label{pf.thm.G.pieri.lim}%
\end{equation}
(where the limit is taken with respect to the usual topology on $\mathbf{k}%
\left[  \left[  x_{1},x_{2},x_{3},\ldots\right]  \right]  $).

Let $\ell\left(  \mu\right)  $ denote the \emph{length} of the partition $\mu
$; this is defined as the unique $i\in\mathbb{N}$ such that $\mu_{1},\mu
_{2},\ldots,\mu_{i}$ are positive but $\mu_{i+1},\mu_{i+2},\mu_{i+3},\ldots$
are zero. Thus, $\ell\left(  \mu\right)  $ is the number of parts of $\mu$.

Fix an $\ell\in\mathbb{N}$ such that $\ell\geq\ell\left(  \mu\right)  $.

Let $P_{\ell}$ denote the set of all partitions with at most $\ell$ parts. We
shall show that%
\begin{align}
&  \left(  G\left(  k\right)  \right)  \left(  x_{1},x_{2},\ldots,x_{\ell
}\right)  \cdot s_{\mu}\left(  x_{1},x_{2},\ldots,x_{\ell}\right) \nonumber\\
&  =\sum_{\lambda\in P_{\ell}}\operatorname*{pet}\nolimits_{k}\left(
\lambda,\mu\right)  s_{\lambda}\left(  x_{1},x_{2},\ldots,x_{\ell}\right)  .
\label{pf.thm.G.pieri.lvars}%
\end{align}
Once this is done, the usual \textquotedblleft let $\ell$ tend to $\infty
$\textquotedblright\ argument (analogous to \cite[proof of Corollary
2.6.11]{GriRei}) will yield the validity of Theorem \ref{thm.G.pieri}.

Any partition $\lambda\in P_{\ell}$ satisfies $\lambda_{\ell+1}=\lambda
_{\ell+2}=\lambda_{\ell+3}=\cdots=0$ and thus $\lambda=\left(  \lambda
_{1},\lambda_{2},\ldots,\lambda_{\ell}\right)  $. Thus, any partition
$\lambda\in P_{\ell}$ can be regarded as an $\ell$-tuple of nonnegative
integers. In other words, $P_{\ell}\subseteq\mathbb{N}^{\ell}$. More
precisely, the partitions $\lambda\in P_{\ell}$ are precisely the $\ell
$-tuples $\left(  \lambda_{1},\lambda_{2},\ldots,\lambda_{\ell}\right)
\in\mathbb{N}^{\ell}$ satisfying $\lambda_{1}\geq\lambda_{2}\geq\cdots
\geq\lambda_{\ell}$.

Define the $\ell$-tuple $\rho=\left(  \ell-1,\ell-2,\ldots,0\right)
\in\mathbb{N}^{\ell}$.

For any $\ell$-tuple $\alpha\in\mathbb{N}^{\ell}$ and each $i\in\left\{
1,2,\ldots,\ell\right\}  $, we shall write $\alpha_{i}$ for the $i$-th entry
of $\alpha$.

For any $\ell$-tuple $\alpha\in\mathbb{N}^{\ell}$, we let $\mathbf{x}^{\alpha
}$ denote the monomial $x_{1}^{\alpha_{1}}x_{2}^{\alpha_{2}}\cdots x_{\ell
}^{\alpha_{\ell}}$.

Let $\mathfrak{S}_{\ell}$ denote the symmetric group of the set $\left\{
1,2,\ldots,\ell\right\}  $. This group $\mathfrak{S}_{\ell}$ acts by
$\mathbf{k}$-algebra homomorphisms on the polynomial ring $\mathbf{k}\left[
x_{1},x_{2},\ldots,x_{\ell}\right]  $.

The symmetric group $\mathfrak{S}_{\ell}$ also acts on the set $\mathbb{N}%
^{\ell}$ by permuting the entries of an $\ell$-tuple; namely,%
\[
\sigma\cdot\beta=\left(  \beta_{\sigma^{-1}\left(  1\right)  },\beta
_{\sigma^{-1}\left(  2\right)  },\ldots,\beta_{\sigma^{-1}\left(  \ell\right)
}\right)  \ \ \ \ \ \ \ \ \ \ \text{for any }\sigma\in\mathfrak{S}_{\ell
}\text{ and }\beta\in\mathbb{N}^{\ell}.
\]
This action has the property that%
\[
\mathbf{x}^{\sigma\cdot\beta}=\sigma\left(  \mathbf{x}^{\beta}\right)
\ \ \ \ \ \ \ \ \ \ \text{for any }\sigma\in\mathfrak{S}_{\ell}\text{ and
}\beta\in\mathbb{N}^{\ell}.
\]

For any $\sigma\in\mathfrak{S}_{\ell}$, we let $\left(  -1\right)  ^{\sigma}$
denote the sign of the permutation $\sigma$.

If $\alpha\in\mathbb{N}^{\ell}$ is any $\ell$-tuple, then we define the
polynomial $a_{\alpha}\in\mathbf{k}\left[  x_{1},x_{2},\ldots,x_{\ell}\right]
$ by%
\[
a_{\alpha}=\sum_{\sigma\in\mathfrak{S}_{\ell}}\left(  -1\right)  ^{\sigma
}\sigma\left(  \mathbf{x}^{\alpha}\right)  =\det\left(  \left(  x_{j}%
^{\alpha_{i}}\right)  _{1\leq i\leq\ell,\ 1\leq j\leq\ell}\right)  .
\]
\footnote{Here is why the second equality sign in this equality holds:
\par
Let $\alpha\in\mathbb{N}^{\ell}$ be any $\ell$-tuple. Then, each $\sigma
\in\mathfrak{S}_{\ell}$ satisfies%
\[
\sigma\left(  \underbrace{\mathbf{x}^{\alpha}}_{=x_{1}^{\alpha_{1}}%
x_{2}^{\alpha_{2}}\cdots x_{\ell}^{\alpha_{\ell}}}\right)  =\sigma\left(
x_{1}^{\alpha_{1}}x_{2}^{\alpha_{2}}\cdots x_{\ell}^{\alpha_{\ell}}\right)
=x_{\sigma\left(  1\right)  }^{\alpha_{1}}x_{\sigma\left(  2\right)  }%
^{\alpha_{2}}\cdots x_{\sigma\left(  \ell\right)  }^{\alpha_{\ell}}%
\]
(by the definition of the action of $\mathfrak{S}_{\ell}$ on $\mathbf{k}%
\left[  x_{1},x_{2},\ldots,x_{\ell}\right]  $). Thus,%
\[
\sum_{\sigma\in\mathfrak{S}_{\ell}}\left(  -1\right)  ^{\sigma}%
\underbrace{\sigma\left(  \mathbf{x}^{\alpha}\right)  }_{=x_{\sigma\left(
1\right)  }^{\alpha_{1}}x_{\sigma\left(  2\right)  }^{\alpha_{2}}\cdots
x_{\sigma\left(  \ell\right)  }^{\alpha_{\ell}}}=\sum_{\sigma\in
\mathfrak{S}_{\ell}}\left(  -1\right)  ^{\sigma}x_{\sigma\left(  1\right)
}^{\alpha_{1}}x_{\sigma\left(  2\right)  }^{\alpha_{2}}\cdots x_{\sigma\left(
\ell\right)  }^{\alpha_{\ell}}.
\]
Comparing this with%
\[
\det\left(  \left(  x_{j}^{\alpha_{i}}\right)  _{1\leq i\leq\ell,\ 1\leq
j\leq\ell}\right)  =\sum_{\sigma\in\mathfrak{S}_{\ell}}\left(  -1\right)
^{\sigma}x_{\sigma\left(  1\right)  }^{\alpha_{1}}x_{\sigma\left(  2\right)
}^{\alpha_{2}}\cdots x_{\sigma\left(  \ell\right)  }^{\alpha_{\ell}%
}\ \ \ \ \ \ \ \ \ \ \left(
\begin{array}
[c]{c}%
\text{by the definition}\\
\text{of a determinant}%
\end{array}
\right)  ,
\]
we obtain $\sum_{\sigma\in\mathfrak{S}_{\ell}}\left(  -1\right)  ^{\sigma
}\sigma\left(  \mathbf{x}^{\alpha}\right)  =\det\left(  \left(  x_{j}%
^{\alpha_{i}}\right)  _{1\leq i\leq\ell,\ 1\leq j\leq\ell}\right)  $. Qed.}
This polynomial $a_{\alpha}$ is called the $\alpha$\emph{-alternant}.

We define addition of $\ell$-tuples $\alpha\in\mathbb{N}^{\ell}$ entrywise (so
that $\left(  \alpha+\beta\right)  _{i}=\alpha_{i}+\beta_{i}$ for every
$\alpha,\beta\in\mathbb{N}^{\ell}$ and $i\in\left\{  1,2,\ldots,\ell\right\}
$). Thus, $\lambda+\rho\in\mathbb{N}^{\ell}$ for each $\lambda\in P_{\ell}$
(since $\lambda\in P_{\ell}\subseteq\mathbb{N}^{\ell}$). Note that
\begin{equation}
\mathbf{x}^{\alpha}\cdot\mathbf{x}^{\beta}=\mathbf{x}^{\alpha+\beta}
\label{pf.thm.G.pieri.xaxb}%
\end{equation}
for any $\alpha,\beta\in\mathbb{N}^{\ell}$.

It is known (\cite[Corollary 2.6.7]{GriRei}) that%
\begin{equation}
s_{\lambda}\left(  x_{1},x_{2},\ldots,x_{\ell}\right)  =\dfrac{a_{\lambda
+\rho}}{a_{\rho}} \label{pf.thm.G.pieri.slam}%
\end{equation}
for every $\lambda\in P_{\ell}$. (The denominator $a_{\rho}$ is a
non-zero-divisor in the ring $\mathbf{k}\left[  x_{1},x_{2},\ldots,x_{\ell
}\right]  $, and the quotient $\dfrac{a_{\lambda+\rho}}{a_{\rho}}$ exists.)

Note that $\ell\geq\ell\left(  \mu\right)  $, so that $\ell\left(  \mu\right)
\leq\ell$; in other words, the partition $\mu$ has at most $\ell$ parts (since
$\ell\left(  \mu\right)  $ is the number of parts of $\mu$). In other words,
$\mu\in P_{\ell}$.

Now, define $\alpha\in\mathbb{N}^{\ell}$ by $\alpha=\mu+\rho$. Proposition
\ref{prop.G.basics} \textbf{(b)} yields%
\[
G\left(  k\right)  =\prod_{i=1}^{\infty}\left(  x_{i}^{0}+x_{i}^{1}%
+\cdots+x_{i}^{k-1}\right)  .
\]
Substituting $0,0,0,\ldots$ for $x_{\ell+1},x_{\ell+2},x_{\ell+3},\ldots$ in
this equality, we obtain%
\begin{align}
&  \left(  G\left(  k\right)  \right)  \left(  x_{1},x_{2},\ldots,x_{\ell
}\right)  \nonumber\\
&  =\left(  \prod_{i=1}^{\ell}\left(  x_{i}^{0}+x_{i}^{1}+\cdots+x_{i}%
^{k-1}\right)  \right)  \cdot\left(  \prod_{i=\ell+1}^{\infty}%
\underbrace{\left(  0^{0}+0^{1}+\cdots+0^{k-1}\right)  }%
_{\substack{=1\\\text{(since }0^{0}=1\text{ and }0^{j}=0\text{ for all
}j>0\text{)}}}\right)  \nonumber\\
&  =\prod_{i=1}^{\ell}\underbrace{\left(  x_{i}^{0}+x_{i}^{1}+\cdots
+x_{i}^{k-1}\right)  }_{=\sum_{j=0}^{k-1}x_{i}^{j}}=\prod_{i=1}^{\ell}%
\ \ \sum_{j=0}^{k-1}x_{i}^{j}\label{pf.thm.G.pieri.Gkxl1}\\
&  =\sum_{\left(  j_{1},j_{2},\ldots,j_{\ell}\right)  \in\left\{
0,1,\ldots,k-1\right\}  ^{\ell}}\underbrace{x_{1}^{j_{1}}x_{2}^{j_{2}}\cdots
x_{\ell}^{j_{\ell}}}_{\substack{=\mathbf{x}^{\left(  j_{1},j_{2}%
,\ldots,j_{\ell}\right)  }\\\text{(by the definition of }\mathbf{x}^{\left(
j_{1},j_{2},\ldots,j_{\ell}\right)  }\text{)}}}\ \ \ \ \ \ \ \ \ \ \left(
\text{by the product rule}\right)  \nonumber\\
&  =\sum_{\left(  j_{1},j_{2},\ldots,j_{\ell}\right)  \in\left\{
0,1,\ldots,k-1\right\}  ^{\ell}}\mathbf{x}^{\left(  j_{1},j_{2},\ldots
,j_{\ell}\right)  }=\underbrace{\sum_{\beta\in\left\{  0,1,\ldots,k-1\right\}
^{\ell}}}_{=\sum_{\substack{\beta\in\mathbb{N}^{\ell};\\\beta_{i}<k\text{ for
all }i}}}\mathbf{x}^{\beta}\nonumber\\
&  \ \ \ \ \ \ \ \ \ \ \left(  \text{here, we have substituted }\beta\text{
for }\left(  j_{1},j_{2},\ldots,j_{\ell}\right)  \text{ in the sum}\right)
\nonumber\\
&  =\sum_{\substack{\beta\in\mathbb{N}^{\ell};\\\beta_{i}<k\text{ for all }%
i}}\mathbf{x}^{\beta}.\nonumber
\end{align}
(Here and in the rest of this proof, \textquotedblleft for all $i$%
\textquotedblright\ means \textquotedblleft for all $i\in\left\{
1,2,\ldots,\ell\right\}  $\textquotedblright.)

From (\ref{pf.thm.G.pieri.Gkxl1}), we see that $\left(  G\left(  k\right)
\right)  \left(  x_{1},x_{2},\ldots,x_{\ell}\right)  $ is a polynomial in
$\mathbf{k}\left[  x_{1},x_{2},\ldots,x_{\ell}\right]  $ (not merely a power
series in $\mathbf{k}\left[  \left[  x_{1},x_{2},\ldots,x_{\ell}\right]
\right]  $). This polynomial $\left(  G\left(  k\right)  \right)  \left(
x_{1},x_{2},\ldots,x_{\ell}\right)  \in\mathbf{k}\left[  x_{1},x_{2}%
,\ldots,x_{\ell}\right]  $ is invariant under the action of $\mathfrak{S}%
_{\ell}$ (because of (\ref{pf.thm.G.pieri.Gkxl1}), or alternatively, because
$G\left(  k\right)  $ is a symmetric power series). In other words,%
\begin{equation}
\sigma\left(  \left(  G\left(  k\right)  \right)  \left(  x_{1},x_{2}%
,\ldots,x_{\ell}\right)  \right)  =\left(  G\left(  k\right)  \right)  \left(
x_{1},x_{2},\ldots,x_{\ell}\right)  \label{pf.thm.G.pieri.Gkxl-symm}%
\end{equation}
for any $\sigma\in\mathfrak{S}_{\ell}$.

But from $a_{\alpha}=\sum_{\sigma\in\mathfrak{S}_{\ell}}\left(  -1\right)
^{\sigma}\sigma\left(  \mathbf{x}^{\alpha}\right)  $, we obtain%
\begin{align}
&  \left(  G\left(  k\right)  \right)  \left(  x_{1},x_{2},\ldots,x_{\ell
}\right)  \cdot a_{\alpha}\nonumber\\
&  =\left(  G\left(  k\right)  \right)  \left(  x_{1},x_{2},\ldots,x_{\ell
}\right)  \cdot\sum_{\sigma\in\mathfrak{S}_{\ell}}\left(  -1\right)  ^{\sigma
}\sigma\left(  \mathbf{x}^{\alpha}\right) \nonumber\\
&  =\sum_{\sigma\in\mathfrak{S}_{\ell}}\left(  -1\right)  ^{\sigma
}\underbrace{\left(  G\left(  k\right)  \right)  \left(  x_{1},x_{2}%
,\ldots,x_{\ell}\right)  }_{\substack{=\sigma\left(  \left(  G\left(
k\right)  \right)  \left(  x_{1},x_{2},\ldots,x_{\ell}\right)  \right)
\\\text{(by (\ref{pf.thm.G.pieri.Gkxl-symm}))}}}\cdot\sigma\left(
\mathbf{x}^{\alpha}\right) \nonumber\\
&  =\sum_{\sigma\in\mathfrak{S}_{\ell}}\left(  -1\right)  ^{\sigma
}\underbrace{\sigma\left(  \left(  G\left(  k\right)  \right)  \left(
x_{1},x_{2},\ldots,x_{\ell}\right)  \right)  \cdot\sigma\left(  \mathbf{x}%
^{\alpha}\right)  }_{\substack{=\sigma\left(  \left(  G\left(  k\right)
\right)  \left(  x_{1},x_{2},\ldots,x_{\ell}\right)  \cdot\mathbf{x}^{\alpha
}\right)  \\\text{(since }\mathfrak{S}_{\ell}\text{ acts on }\mathbf{k}\left[
x_{1},x_{2},\ldots,x_{\ell}\right]  \\\text{by }\mathbf{k}\text{-algebra
homomorphisms)}}}\nonumber\\
&  =\sum_{\sigma\in\mathfrak{S}_{\ell}}\left(  -1\right)  ^{\sigma}%
\sigma\left(  \underbrace{\left(  G\left(  k\right)  \right)  \left(
x_{1},x_{2},\ldots,x_{\ell}\right)  }_{=\sum_{\substack{\beta\in
\mathbb{N}^{\ell};\\\beta_{i}<k\text{ for all }i}}\mathbf{x}^{\beta}}%
\cdot\mathbf{x}^{\alpha}\right) \nonumber\\
&  =\sum_{\sigma\in\mathfrak{S}_{\ell}}\left(  -1\right)  ^{\sigma}%
\sigma\left(  \sum_{\substack{\beta\in\mathbb{N}^{\ell};\\\beta_{i}<k\text{
for all }i}}\mathbf{x}^{\beta}\cdot\mathbf{x}^{\alpha}\right)  =\sum
_{\substack{\beta\in\mathbb{N}^{\ell};\\\beta_{i}<k\text{ for all }i}%
}\ \ \sum_{\sigma\in\mathfrak{S}_{\ell}}\left(  -1\right)  ^{\sigma}%
\sigma\left(  \underbrace{\mathbf{x}^{\beta}\cdot\mathbf{x}^{\alpha}%
}_{\substack{=\mathbf{x}^{\alpha}\cdot\mathbf{x}^{\beta}\\=\mathbf{x}%
^{\alpha+\beta}\\\text{(by (\ref{pf.thm.G.pieri.xaxb}))}}}\right) \nonumber\\
&  =\sum_{\substack{\beta\in\mathbb{N}^{\ell};\\\beta_{i}<k\text{ for all }%
i}}\ \ \underbrace{\sum_{\sigma\in\mathfrak{S}_{\ell}}\left(  -1\right)
^{\sigma}\sigma\left(  \mathbf{x}^{\alpha+\beta}\right)  }%
_{\substack{=a_{\alpha+\beta}\\\text{(since }a_{\alpha+\beta}\text{ is
defined}\\\text{by }a_{\alpha+\beta}=\sum_{\sigma\in\mathfrak{S}_{\ell}%
}\left(  -1\right)  ^{\sigma}\sigma\left(  \mathbf{x}^{\alpha+\beta}\right)
\text{)}}}=\sum_{\substack{\beta\in\mathbb{N}^{\ell};\\\beta_{i}<k\text{ for
all }i}}a_{\alpha+\beta}\nonumber\\
&  =\sum_{\substack{\gamma\in\mathbb{N}^{\ell};\\0\leq\gamma_{i}-\alpha
_{i}<k\text{ for all }i}}a_{\gamma} \label{pf.thm.G.pieri.prod1}%
\end{align}
(here, we have substituted $\gamma$ for $\alpha+\beta$ in the sum).

It is well-known (and easy to check using the properties of
determinants\footnote{specifically: using the fact that a square matrix with
two equal rows always has determinant $0$}) that if an $\ell$-tuple $\gamma
\in\mathbb{N}^{\ell}$ has two equal entries, then%
\begin{equation}
a_{\gamma}=0. \label{pf.thm.G.pieri.agam=0}%
\end{equation}
Moreover, any $\ell$-tuple $\gamma\in\mathbb{N}^{\ell}$ and any $\sigma
\in\mathfrak{S}_{\ell}$ satisfy%
\begin{equation}
a_{\sigma\cdot\gamma}=\left(  -1\right)  ^{\sigma}\cdot a_{\gamma}.
\label{pf.thm.G.pieri.asiggam=}%
\end{equation}
(This, too, follows from the properties of determinants\footnote{specifically:
using the fact that permuting the rows of a square matrix results in its
determinant getting multiplied by the sign of the permutation}.)

Let $SP_{\ell}$ denote the set of all $\ell$-tuples $\delta\in\mathbb{N}%
^{\ell}$ such that $\delta_{1}>\delta_{2}>\cdots>\delta_{\ell}$. Then, the map%
\begin{align}
P_{\ell}  &  \rightarrow SP_{\ell},\nonumber\\
\lambda &  \mapsto\lambda+\rho\label{pf.thm.G.pieri.PSPbij}%
\end{align}
is a bijection.

If an $\ell$-tuple $\gamma\in\mathbb{N}^{\ell}$ has no two equal entries, then
$\gamma$ can be uniquely written in the form $\sigma\cdot\delta$ for some
$\sigma\in\mathfrak{S}_{\ell}$ and some $\delta\in SP_{\ell}$ (indeed,
$\delta$ is the result of sorting $\gamma$ into decreasing order, while
$\sigma$ is the permutation that achieves this sorting). In other words, the
map%
\begin{align}
\mathfrak{S}_{\ell}\times SP_{\ell}  &  \rightarrow\left\{  \gamma
\in\mathbb{N}^{\ell}\ \mid\ \text{the }\ell\text{-tuple }\gamma\text{ has no
two equal entries}\right\}  ,\nonumber\\
\left(  \sigma,\delta\right)   &  \mapsto\sigma\cdot\delta
\label{pf.thm.G.pieri.SSPbbij}%
\end{align}
is a bijection.

Now, (\ref{pf.thm.G.pieri.prod1}) becomes%
\begin{align}
&  \left(  G\left(  k\right)  \right)  \left(  x_{1},x_{2},\ldots,x_{\ell
}\right)  \cdot a_{\alpha}\nonumber\\
&  =\sum_{\substack{\gamma\in\mathbb{N}^{\ell};\\0\leq\gamma_{i}-\alpha
_{i}<k\text{ for all }i}}a_{\gamma}\nonumber\\
&  =\sum_{\substack{\gamma\in\mathbb{N}^{\ell};\\0\leq\gamma_{i}-\alpha
_{i}<k\text{ for all }i;\\\text{the }\ell\text{-tuple }\gamma\text{ has no two
equal entries}}}a_{\gamma}+\sum_{\substack{\gamma\in\mathbb{N}^{\ell}%
;\\0\leq\gamma_{i}-\alpha_{i}<k\text{ for all }i;\\\text{the }\ell\text{-tuple
}\gamma\text{ has two equal entries}}}\underbrace{a_{\gamma}}%
_{\substack{=0\\\text{(by (\ref{pf.thm.G.pieri.agam=0}))}}}\nonumber\\
&  =\sum_{\substack{\gamma\in\mathbb{N}^{\ell};\\0\leq\gamma_{i}-\alpha
_{i}<k\text{ for all }i;\\\text{the }\ell\text{-tuple }\gamma\text{ has no two
equal entries}}}a_{\gamma}\nonumber\\
&  =\sum_{\substack{\gamma\in\mathbb{N}^{\ell};\\\text{the }\ell\text{-tuple
}\gamma\text{ has no two equal entries}}}\left[  0\leq\gamma_{i}-\alpha
_{i}<k\text{ for all }i\right]  \cdot a_{\gamma}\nonumber\\
&  =\underbrace{\sum_{\left(  \sigma,\delta\right)  \in\mathfrak{S}_{\ell
}\times SP_{\ell}}}_{=\sum_{\delta\in SP_{\ell}}\ \ \sum_{\sigma
\in\mathfrak{S}_{\ell}}}\left[  0\leq\underbrace{\left(  \sigma\cdot
\delta\right)  _{i}}_{=\delta_{\sigma^{-1}\left(  i\right)  }}-\alpha
_{i}<k\text{ for all }i\right]  \cdot\underbrace{a_{\sigma\cdot\delta}%
}_{\substack{=\left(  -1\right)  ^{\sigma}a_{\delta}\\\text{(by
(\ref{pf.thm.G.pieri.asiggam=}))}}}\nonumber\\
&  \ \ \ \ \ \ \ \ \ \ \left(
\begin{array}
[c]{c}%
\text{here, we have substituted }\sigma\cdot\delta\text{ for }\gamma\text{ in
the sum,}\\
\text{since the map (\ref{pf.thm.G.pieri.SSPbbij}) is a bijection}%
\end{array}
\right) \nonumber\\
&  =\sum_{\delta\in SP_{\ell}}\ \ \sum_{\sigma\in\mathfrak{S}_{\ell}%
}\underbrace{\left[  0\leq\delta_{\sigma^{-1}\left(  i\right)  }-\alpha
_{i}<k\text{ for all }i\right]  }_{\substack{=\prod_{i=1}^{\ell}\left[
0\leq\delta_{\sigma^{-1}\left(  i\right)  }-\alpha_{i}<k\right]
\\=\prod_{i=1}^{\ell}\left[  0\leq\delta_{i}-\alpha_{\sigma\left(  i\right)
}<k\right]  \\\text{(here, we have substituted }\sigma\left(  i\right)
\\\text{for }i\text{ in the product, since }\sigma\text{ is a bijection)}%
}}\cdot\left(  -1\right)  ^{\sigma}a_{\delta}\nonumber\\
&  =\sum_{\delta\in SP_{\ell}}\ \ \underbrace{\sum_{\sigma\in\mathfrak{S}%
_{\ell}}\left(  \prod_{i=1}^{\ell}\left[  0\leq\delta_{i}-\alpha
_{\sigma\left(  i\right)  }<k\right]  \right)  \cdot\left(  -1\right)
^{\sigma}}_{\substack{=\sum_{\sigma\in\mathfrak{S}_{\ell}}\left(  -1\right)
^{\sigma}\prod_{i=1}^{\ell}\left[  0\leq\delta_{i}-\alpha_{\sigma\left(
i\right)  }<k\right]  \\=\det\left(  \left(  \left[  0\leq\delta_{i}%
-\alpha_{j}<k\right]  \right)  _{1\leq i\leq\ell,\ 1\leq j\leq\ell}\right)
\\\text{(by the definition of a determinant)}}}a_{\delta}\nonumber\\
&  =\sum_{\delta\in SP_{\ell}}\det\left(  \left(  \left[  0\leq\delta
_{i}-\alpha_{j}<k\right]  \right)  _{1\leq i\leq\ell,\ 1\leq j\leq\ell
}\right)  a_{\delta}\nonumber\\
&  =\sum_{\lambda\in P_{\ell}}\det\left(  \left(  \left[  0\leq\left(
\lambda+\rho\right)  _{i}-\alpha_{j}<k\right]  \right)  _{1\leq i\leq
\ell,\ 1\leq j\leq\ell}\right)  a_{\lambda+\rho} \label{pf.thm.G.pieri.prod2}%
\end{align}
(here, we have substituted $\lambda+\rho$ for $\delta$ in the sum, since the
map (\ref{pf.thm.G.pieri.PSPbij}) is a bijection).

Every $\lambda\in P_{\ell}$ and every $i,j\in\left\{  1,2,\ldots,\ell\right\}
$ satisfy%
\begin{align}
&  \underbrace{\left(  \lambda+\rho\right)  _{i}}_{\substack{=\lambda_{i}%
+\rho_{i}=\lambda_{i}+\ell-i\\\text{(since the definition of }\rho
\\\text{yields }\rho_{i}=\ell-i\text{)}}}-\underbrace{\alpha_{j}%
}_{\substack{=\left(  \mu+\rho\right)  _{j}\\\text{(since }\alpha=\mu
+\rho\text{)}}}\nonumber\\
&  =\lambda_{i}+\ell-i-\underbrace{\left(  \mu+\rho\right)  _{j}%
}_{\substack{=\mu_{j}+\rho_{j}=\mu_{j}+\ell-j\\\text{(since the definition of
}\rho\\\text{yields }\rho_{j}=\ell-j\text{)}}}=\lambda_{i}+\ell-i-\left(
\mu_{j}+\ell-j\right) \nonumber\\
&  =\lambda_{i}-\mu_{j}-i+j. \label{pf.thm.G.pieri.dif}%
\end{align}

Now, (\ref{pf.thm.G.pieri.slam}) (applied to $\lambda=\mu$) yields%
\[
s_{\mu}\left(  x_{1},x_{2},\ldots,x_{\ell}\right)  =\dfrac{a_{\mu+\rho}%
}{a_{\rho}}=\dfrac{a_{\alpha}}{a_{\rho}}%
\]
(since $\mu+\rho=\alpha$). Multiplying this equality by $\left(  G\left(
k\right)  \right)  \left(  x_{1},x_{2},\ldots,x_{\ell}\right)  $, we find%
\begin{align*}
&  \left(  G\left(  k\right)  \right)  \left(  x_{1},x_{2},\ldots,x_{\ell
}\right)  \cdot s_{\mu}\left(  x_{1},x_{2},\ldots,x_{\ell}\right) \\
&  =\left(  G\left(  k\right)  \right)  \left(  x_{1},x_{2},\ldots,x_{\ell
}\right)  \cdot\dfrac{a_{\alpha}}{a_{\rho}}\\
&  =\dfrac{1}{a_{\rho}}\cdot\underbrace{\left(  G\left(  k\right)  \right)
\left(  x_{1},x_{2},\ldots,x_{\ell}\right)  \cdot a_{\alpha}}_{\substack{=\sum
_{\lambda\in P_{\ell}}\det\left(  \left(  \left[  0\leq\left(  \lambda
+\rho\right)  _{i}-\alpha_{j}<k\right]  \right)  _{1\leq i\leq\ell,\ 1\leq
j\leq\ell}\right)  a_{\lambda+\rho}\\\text{(by (\ref{pf.thm.G.pieri.prod2}))}%
}}\\
&  =\dfrac{1}{a_{\rho}}\cdot\sum_{\lambda\in P_{\ell}}\det\left(  \left(
\left[  0\leq\left(  \lambda+\rho\right)  _{i}-\alpha_{j}<k\right]  \right)
_{1\leq i\leq\ell,\ 1\leq j\leq\ell}\right)  a_{\lambda+\rho}\\
&  =\sum_{\lambda\in P_{\ell}}\det\left(  \left(  \left[  0\leq
\underbrace{\left(  \lambda+\rho\right)  _{i}-\alpha_{j}}_{\substack{=\lambda
_{i}-\mu_{j}-i+j\\\text{(by (\ref{pf.thm.G.pieri.dif}))}}}<k\right]  \right)
_{1\leq i\leq\ell,\ 1\leq j\leq\ell}\right)  \underbrace{\dfrac{a_{\lambda
+\rho}}{a_{\rho}}}_{\substack{=s_{\lambda}\left(  x_{1},x_{2},\ldots,x_{\ell
}\right)  \\\text{(by (\ref{pf.thm.G.pieri.slam}))}}}\\
&  =\sum_{\lambda\in P_{\ell}}\underbrace{\det\left(  \left(  \left[
0\leq\lambda_{i}-\mu_{j}-i+j<k\right]  \right)  _{1\leq i\leq\ell,\ 1\leq
j\leq\ell}\right)  }_{\substack{=\operatorname*{pet}\nolimits_{k}\left(
\lambda,\mu\right)  \\\text{(by the definition of }\operatorname*{pet}%
\nolimits_{k}\left(  \lambda,\mu\right)  \text{)}}}\cdot s_{\lambda}\left(
x_{1},x_{2},\ldots,x_{\ell}\right) \\
&  =\sum_{\lambda\in P_{\ell}}\operatorname*{pet}\nolimits_{k}\left(
\lambda,\mu\right)  s_{\lambda}\left(  x_{1},x_{2},\ldots,x_{\ell}\right)  .
\end{align*}
This proves (\ref{pf.thm.G.pieri.lvars}).

On the other hand, it is known (see, e.g., \cite[Exercise 2.3.8(b)]{GriRei})
that if $\lambda$ is a partition having more than $\ell$ parts, then%
\begin{equation}
s_{\lambda}\left(  x_{1},x_{2},\ldots,x_{\ell}\right)  =0.
\label{pf.thm.G.pieri.too-long}%
\end{equation}

Now, each partition $\lambda\in\operatorname*{Par}$ either has at most $\ell$
parts or has more than $\ell$ parts (but not both at the same time). Hence,%
\begin{align*}
&  \sum_{\lambda\in\operatorname*{Par}}\operatorname*{pet}\nolimits_{k}\left(
\lambda,\mu\right)  s_{\lambda}\left(  x_{1},x_{2},\ldots,x_{\ell}\right) \\
&  =\underbrace{\sum_{\substack{\lambda\in\operatorname*{Par};\\\lambda\text{
has at most }\ell\text{ parts}}}}_{\substack{=\sum_{\lambda\in P_{\ell}%
}\\\text{(by the definition of }P_{\ell}\text{)}}}\operatorname*{pet}%
\nolimits_{k}\left(  \lambda,\mu\right)  s_{\lambda}\left(  x_{1},x_{2}%
,\ldots,x_{\ell}\right) \\
&  \ \ \ \ \ \ \ \ \ \ +\sum_{\substack{\lambda\in\operatorname*{Par}%
;\\\lambda\text{ has more than }\ell\text{ parts}}}\operatorname*{pet}%
\nolimits_{k}\left(  \lambda,\mu\right)  \underbrace{s_{\lambda}\left(
x_{1},x_{2},\ldots,x_{\ell}\right)  }_{\substack{=0\\\text{(by
(\ref{pf.thm.G.pieri.too-long}))}}}\\
&  =\sum_{\lambda\in P_{\ell}}\operatorname*{pet}\nolimits_{k}\left(
\lambda,\mu\right)  s_{\lambda}\left(  x_{1},x_{2},\ldots,x_{\ell}\right)
+\underbrace{\sum_{\substack{\lambda\in\operatorname*{Par};\\\lambda\text{ has
more than }\ell\text{ parts}}}\operatorname*{pet}\nolimits_{k}\left(
\lambda,\mu\right)  0}_{=0}\\
&  =\sum_{\lambda\in P_{\ell}}\operatorname*{pet}\nolimits_{k}\left(
\lambda,\mu\right)  s_{\lambda}\left(  x_{1},x_{2},\ldots,x_{\ell}\right)  .
\end{align*}
Comparing this with (\ref{pf.thm.G.pieri.lvars}), we obtain%
\begin{align}
&  \left(  G\left(  k\right)  \right)  \left(  x_{1},x_{2},\ldots,x_{\ell
}\right)  \cdot s_{\mu}\left(  x_{1},x_{2},\ldots,x_{\ell}\right) \nonumber\\
&  =\sum_{\lambda\in\operatorname*{Par}}\operatorname*{pet}\nolimits_{k}%
\left(  \lambda,\mu\right)  s_{\lambda}\left(  x_{1},x_{2},\ldots,x_{\ell
}\right)  . \label{pf.thm.G.pieri.lvars-Par}%
\end{align}

Forget that we fixed $\ell$. Thus, we have proved
(\ref{pf.thm.G.pieri.lvars-Par}) for each $\ell\in\mathbb{N}$ that satisfies
$\ell\geq\ell\left(  \mu\right)  $. Thus, for each $\ell\in\mathbb{N}$ that
satisfies $\ell\geq\ell\left(  \mu\right)  $, we have%
\begin{align}
&  \left(  G\left(  k\right)  \cdot s_{\mu}\right)  \left(  x_{1},x_{2}%
,\ldots,x_{\ell}\right) \nonumber\\
&  =\left(  G\left(  k\right)  \right)  \left(  x_{1},x_{2},\ldots,x_{\ell
}\right)  \cdot s_{\mu}\left(  x_{1},x_{2},\ldots,x_{\ell}\right) \nonumber\\
&  =\sum_{\lambda\in\operatorname*{Par}}\operatorname*{pet}\nolimits_{k}%
\left(  \lambda,\mu\right)  s_{\lambda}\left(  x_{1},x_{2},\ldots,x_{\ell
}\right)  \ \ \ \ \ \ \ \ \ \ \left(  \text{by (\ref{pf.thm.G.pieri.lvars-Par}%
)}\right)  . \label{pf.thm.G.pieri.lvars2}%
\end{align}
Now, (\ref{pf.thm.G.pieri.lim}) (applied to $f=G\left(  k\right)  \cdot
s_{\mu}$) yields%
\begin{align*}
G\left(  k\right)  \cdot s_{\mu}  &  =\lim\limits_{\ell\rightarrow\infty
}\underbrace{\left(  G\left(  k\right)  \cdot s_{\mu}\right)  \left(
x_{1},x_{2},\ldots,x_{\ell}\right)  }_{\substack{=\sum_{\lambda\in
\operatorname*{Par}}\operatorname*{pet}\nolimits_{k}\left(  \lambda
,\mu\right)  s_{\lambda}\left(  x_{1},x_{2},\ldots,x_{\ell}\right)  \text{
when }\ell\geq\ell\left(  \mu\right)  \\\text{(by (\ref{pf.thm.G.pieri.lvars2}%
))}}}\\
&  =\lim\limits_{\ell\rightarrow\infty}\sum_{\lambda\in\operatorname*{Par}%
}\operatorname*{pet}\nolimits_{k}\left(  \lambda,\mu\right)  s_{\lambda
}\left(  x_{1},x_{2},\ldots,x_{\ell}\right)  .
\end{align*}
Comparing this with%
\begin{align*}
\sum_{\lambda\in\operatorname*{Par}}\operatorname*{pet}\nolimits_{k}\left(
\lambda,\mu\right)  \underbrace{s_{\lambda}}_{\substack{=\lim\limits_{\ell
\rightarrow\infty}s_{\lambda}\left(  x_{1},x_{2},\ldots,x_{\ell}\right)
\\\text{(by (\ref{pf.thm.G.pieri.lim}))}}}  &  =\sum_{\lambda\in
\operatorname*{Par}}\operatorname*{pet}\nolimits_{k}\left(  \lambda
,\mu\right)  \lim\limits_{\ell\rightarrow\infty}s_{\lambda}\left(  x_{1}%
,x_{2},\ldots,x_{\ell}\right) \\
&  =\lim\limits_{\ell\rightarrow\infty}\sum_{\lambda\in\operatorname*{Par}%
}\operatorname*{pet}\nolimits_{k}\left(  \lambda,\mu\right)  s_{\lambda
}\left(  x_{1},x_{2},\ldots,x_{\ell}\right)
\end{align*}
\footnote{Why were we allowed to interchange the limit with the summation sign
here? One way to justify this is by realizing that each Schur function
$s_{\lambda}$ and therefore each polynomial $s_{\lambda}\left(  x_{1}%
,x_{2},\ldots,x_{\ell}\right)  $ are homogeneous of degree $\left\vert
\lambda\right\vert $, and for each $n\in\mathbb{N}$ there are only finitely
many partitions $\lambda\in\operatorname*{Par}$ satisfying $\left\vert
\lambda\right\vert =n$. This entails that each individual monomial
$\mathfrak{m}$ is affected only by finitely many addends in the sums appearing
on both sides of our equation.}, we obtain%
\[
G\left(  k\right)  \cdot s_{\mu}=\sum_{\lambda\in\operatorname*{Par}%
}\operatorname*{pet}\nolimits_{k}\left(  \lambda,\mu\right)  s_{\lambda}.
\]
This completes the second proof of Theorem \ref{thm.G.pieri}.
\end{proof}
\end{verlong}

\subsection{\label{subsect.proofs.G.cors}Proofs of Corollary
\ref{cor.Gkm.pieri}, Theorem \ref{thm.G.main} and Corollary \ref{cor.Gkm.main}%
}

Having proved Theorem \ref{thm.G.pieri}, we can now obtain Corollary
\ref{cor.Gkm.pieri}, Theorem \ref{thm.G.main} and Corollary \ref{cor.Gkm.main}
as easy consequences:

\begin{vershort}
\begin{proof}
[Proof of Corollary \ref{cor.Gkm.pieri}.]Proposition \ref{prop.G.basics}
\textbf{(a)} yields that the $m$-th degree homogeneous component of $G\left(
k\right)  $ is $G\left(  k,m\right)  $. Hence, the $\left(  m+\left\vert
\mu\right\vert \right)  $-th degree homogeneous component of $G\left(
k\right)  \cdot s_{\mu}$ is $G\left(  k,m\right)  \cdot s_{\mu}$ (because
$s_{\mu}$ is homogeneous of degree $\left\vert \mu\right\vert $).

Theorem \ref{thm.G.pieri} yields%
\[
G\left(  k\right)  \cdot s_{\mu}=\sum_{\lambda\in\operatorname*{Par}%
}\operatorname*{pet}\nolimits_{k}\left(  \lambda,\mu\right)  s_{\lambda}.
\]
Taking the $\left(  m+\left\vert \mu\right\vert \right)  $-th degree
homogeneous components on both sides of this equality, we obtain%
\[
G\left(  k,m\right)  \cdot s_{\mu}=\sum_{\lambda\in\operatorname*{Par}%
\nolimits_{m+\left\vert \mu\right\vert }}\operatorname*{pet}\nolimits_{k}%
\left(  \lambda,\mu\right)  s_{\lambda}%
\]
(because each Schur function $s_{\lambda}$ is homogeneous of degree
$\left\vert \lambda\right\vert $, whereas the $\left(  m+\left\vert
\mu\right\vert \right)  $-th degree homogeneous component of $G\left(
k\right)  \cdot s_{\mu}$ is $G\left(  k,m\right)  \cdot s_{\mu}$). This proves
Corollary \ref{cor.Gkm.pieri}.
\end{proof}
\end{vershort}

\begin{verlong}
\begin{proof}
[Proof of Corollary \ref{cor.Gkm.pieri}.]Forget that we fixed $m$. If
$n\in\mathbb{N}$, then the power series $%
\begin{cases}
G\left(  k,n-\left\vert \mu\right\vert \right)  \cdot s_{\mu}, & \text{if
}n\geq\left\vert \mu\right\vert ;\\
0, & \text{if }n<\left\vert \mu\right\vert
\end{cases}
\in\mathbf{k}\left[  \left[  x_{1},x_{2},x_{3},\ldots\right]  \right]  $ is
homogeneous of degree $n$\ \ \ \ \footnote{\textit{Proof.} Let $n\in
\mathbb{N}$. We must prove that the power series $%
\begin{cases}
G\left(  k,n-\left\vert \mu\right\vert \right)  \cdot s_{\mu}, & \text{if
}n\geq\left\vert \mu\right\vert ;\\
0, & \text{if }n<\left\vert \mu\right\vert
\end{cases}
$ is homogeneous of degree $n$.
\par
We are in one of the following two cases:
\par
\textit{Case 1:} We have $n\geq\left\vert \mu\right\vert $.
\par
\textit{Case 2:} We have $n<\left\vert \mu\right\vert $.
\par
Let us first consider Case 1. In this case, we have $n\geq\left\vert
\mu\right\vert $. Hence, $n-\left\vert \mu\right\vert \in\mathbb{N}$. But
Proposition \ref{prop.G.basics} \textbf{(a)} (applied to $m=n-\left\vert
\mu\right\vert $) yields that the symmetric function $G\left(  k,n-\left\vert
\mu\right\vert \right)  $ is the $\left(  n-\left\vert \mu\right\vert \right)
$-th degree homogeneous component of $G\left(  k\right)  $. Hence, $G\left(
k,n-\left\vert \mu\right\vert \right)  $ is homogeneous of degree
$n-\left\vert \mu\right\vert $.
\par
On the other hand, recall that for any $\lambda\in\operatorname*{Par}$, the
Schur function $s_{\lambda}$ is homogeneous of degree $\left\vert
\lambda\right\vert $. Applying this to $\lambda=\mu$, we conclude that the
Schur function $s_{\mu}$ is homogeneous of degree $\left\vert \mu\right\vert
$.
\par
So we know that $G\left(  k,n-\left\vert \mu\right\vert \right)  $ is
homogeneous of degree $n-\left\vert \mu\right\vert $, whereas $s_{\mu}$ is
homogeneous of degree $\left\vert \mu\right\vert $. Since $\Lambda$ is a
graded algebra, this entails that the power series $G\left(  k,n-\left\vert
\mu\right\vert \right)  \cdot s_{\mu}$ (being the product of $G\left(
k,n-\left\vert \mu\right\vert \right)  $ and $s_{\mu}$) is homogeneous of
degree $\left(  n-\left\vert \mu\right\vert \right)  +\left\vert
\mu\right\vert $. In other words, the power series $G\left(  k,n-\left\vert
\mu\right\vert \right)  \cdot s_{\mu}$ is homogeneous of degree $n$ (since
$\left(  n-\left\vert \mu\right\vert \right)  +\left\vert \mu\right\vert =n$).
In other words, the power series $%
\begin{cases}
G\left(  k,n-\left\vert \mu\right\vert \right)  \cdot s_{\mu}, & \text{if
}n\geq\left\vert \mu\right\vert ;\\
0, & \text{if }n<\left\vert \mu\right\vert
\end{cases}
$ is homogeneous of degree $n$ (since%
\[%
\begin{cases}
G\left(  k,n-\left\vert \mu\right\vert \right)  \cdot s_{\mu}, & \text{if
}n\geq\left\vert \mu\right\vert ;\\
0, & \text{if }n<\left\vert \mu\right\vert
\end{cases}
=G\left(  k,n-\left\vert \mu\right\vert \right)  \cdot s_{\mu}%
\ \ \ \ \ \ \ \ \ \ \left(  \text{because }n\geq\left\vert \mu\right\vert
\right)
\]
). Thus, we have proved in Case 1 that the power series $%
\begin{cases}
G\left(  k,n-\left\vert \mu\right\vert \right)  \cdot s_{\mu}, & \text{if
}n\geq\left\vert \mu\right\vert ;\\
0, & \text{if }n<\left\vert \mu\right\vert
\end{cases}
$ is homogeneous of degree $n$.
\par
Let us now consider Case 2. In this case, we have $n<\left\vert \mu\right\vert
$. The power series $0$ is homogeneous of degree $n$ (since $0$ is homogeneous
of any degree). In other words, the power series $%
\begin{cases}
G\left(  k,n-\left\vert \mu\right\vert \right)  \cdot s_{\mu}, & \text{if
}n\geq\left\vert \mu\right\vert ;\\
0, & \text{if }n<\left\vert \mu\right\vert
\end{cases}
$ is homogeneous of degree $n$ (since%
\[%
\begin{cases}
G\left(  k,n-\left\vert \mu\right\vert \right)  \cdot s_{\mu}, & \text{if
}n\geq\left\vert \mu\right\vert ;\\
0, & \text{if }n<\left\vert \mu\right\vert
\end{cases}
=0\ \ \ \ \ \ \ \ \ \ \left(  \text{because }n<\left\vert \mu\right\vert
\right)
\]
). Thus, we have proved in Case 2 that the power series $%
\begin{cases}
G\left(  k,n-\left\vert \mu\right\vert \right)  \cdot s_{\mu}, & \text{if
}n\geq\left\vert \mu\right\vert ;\\
0, & \text{if }n<\left\vert \mu\right\vert
\end{cases}
$ is homogeneous of degree $n$.
\par
Thus, our claim (namely, that the power series $%
\begin{cases}
G\left(  k,n-\left\vert \mu\right\vert \right)  \cdot s_{\mu}, & \text{if
}n\geq\left\vert \mu\right\vert ;\\
0, & \text{if }n<\left\vert \mu\right\vert
\end{cases}
$ is homogeneous of degree $n$) has been proven in both Cases 1 and 2. Since
these cases cover all possibilities, we thus conclude that our claim always
holds. Qed.}.

In the proof of Proposition \ref{prop.G.basics} \textbf{(a)}, we have shown
that $G\left(  k\right)  =\sum_{m\in\mathbb{N}}G\left(  k,m\right)  $.
Multiplying both sides of this equality by $s_{\mu}$, we find%
\[
G\left(  k\right)  \cdot s_{\mu}=\left(  \sum_{m\in\mathbb{N}}G\left(
k,m\right)  \right)  \cdot s_{\mu}=\sum_{m\in\mathbb{N}}G\left(  k,m\right)
\cdot s_{\mu}.
\]
Comparing this with%
\begin{align*}
&  \sum_{n\in\mathbb{N}}%
\begin{cases}
G\left(  k,n-\left\vert \mu\right\vert \right)  \cdot s_{\mu}, & \text{if
}n\geq\left\vert \mu\right\vert ;\\
0, & \text{if }n<\left\vert \mu\right\vert
\end{cases}
\\
&  =\sum_{\substack{n\in\mathbb{N};\\n\geq\left\vert \mu\right\vert
}}\underbrace{%
\begin{cases}
G\left(  k,n-\left\vert \mu\right\vert \right)  \cdot s_{\mu}, & \text{if
}n\geq\left\vert \mu\right\vert ;\\
0, & \text{if }n<\left\vert \mu\right\vert
\end{cases}
}_{\substack{=G\left(  k,n-\left\vert \mu\right\vert \right)  \cdot s_{\mu
}\\\text{(since }n\geq\left\vert \mu\right\vert \text{)}}}\ \ \ \ +\sum
_{\substack{n\in\mathbb{N};\\n<\left\vert \mu\right\vert }}\underbrace{%
\begin{cases}
G\left(  k,n-\left\vert \mu\right\vert \right)  \cdot s_{\mu}, & \text{if
}n\geq\left\vert \mu\right\vert ;\\
0, & \text{if }n<\left\vert \mu\right\vert
\end{cases}
}_{\substack{=0\\\text{(since }n<\left\vert \mu\right\vert \text{)}}}\\
&  \ \ \ \ \ \ \ \ \ \ \ \ \ \ \ \ \ \ \ \ \left(
\begin{array}
[c]{c}%
\text{since each }n\in\mathbb{N}\text{ satisfies either }n\geq\left\vert
\mu\right\vert \text{ or }n<\left\vert \mu\right\vert \\
\text{(but not both at the same time)}%
\end{array}
\right) \\
&  =\sum_{\substack{n\in\mathbb{N};\\n\geq\left\vert \mu\right\vert }}G\left(
k,n-\left\vert \mu\right\vert \right)  \cdot s_{\mu}+\underbrace{\sum
_{\substack{n\in\mathbb{N};\\n<\left\vert \mu\right\vert }}0}_{=0}%
=\sum_{\substack{n\in\mathbb{N};\\n\geq\left\vert \mu\right\vert }}G\left(
k,n-\left\vert \mu\right\vert \right)  \cdot s_{\mu}\\
&  =\sum_{m\in\mathbb{N}}G\left(  k,m\right)  \cdot s_{\mu}%
\ \ \ \ \ \ \ \ \ \ \left(  \text{here, we have substituted }m\text{ for
}n-\left\vert \mu\right\vert \text{ in the sum}\right)  ,
\end{align*}
we obtain%
\begin{equation}
G\left(  k\right)  \cdot s_{\mu}=\sum_{n\in\mathbb{N}}%
\begin{cases}
G\left(  k,n-\left\vert \mu\right\vert \right)  \cdot s_{\mu}, & \text{if
}n\geq\left\vert \mu\right\vert ;\\
0, & \text{if }n<\left\vert \mu\right\vert
\end{cases}
\ \ \ . \label{pf.cor.Gkm.pieri.2}%
\end{equation}

But recall that each $%
\begin{cases}
G\left(  k,n-\left\vert \mu\right\vert \right)  \cdot s_{\mu}, & \text{if
}n\geq\left\vert \mu\right\vert ;\\
0, & \text{if }n<\left\vert \mu\right\vert
\end{cases}
$ is homogeneous of degree $n$. Thus, the equality (\ref{pf.cor.Gkm.pieri.2})
reveals that the family
\[
\left(
\begin{cases}
G\left(  k,n-\left\vert \mu\right\vert \right)  \cdot s_{\mu}, & \text{if
}n\geq\left\vert \mu\right\vert ;\\
0, & \text{if }n<\left\vert \mu\right\vert
\end{cases}
\right)  _{n\in\mathbb{N}}%
\]
is the homogeneous decomposition of $G\left(  k\right)  \cdot s_{\mu}$ (by the
definition of a homogeneous decomposition). Therefore, for each $n\in
\mathbb{N}$, the power series $%
\begin{cases}
G\left(  k,n-\left\vert \mu\right\vert \right)  \cdot s_{\mu}, & \text{if
}n\geq\left\vert \mu\right\vert ;\\
0, & \text{if }n<\left\vert \mu\right\vert
\end{cases}
$ is the $n$-th degree homogeneous component of $G\left(  k\right)  \cdot
s_{\mu}$.

Now, let $m\in\mathbb{N}$. We have just shown that for each $n\in\mathbb{N}$,
the power series $%
\begin{cases}
G\left(  k,n-\left\vert \mu\right\vert \right)  \cdot s_{\mu}, & \text{if
}n\geq\left\vert \mu\right\vert ;\\
0, & \text{if }n<\left\vert \mu\right\vert
\end{cases}
$ is the $n$-th degree homogeneous component of $G\left(  k\right)  \cdot
s_{\mu}$. Applying this to $n=m+\left\vert \mu\right\vert $, we conclude that
the power series $%
\begin{cases}
G\left(  k,\left(  m+\left\vert \mu\right\vert \right)  -\left\vert
\mu\right\vert \right)  \cdot s_{\mu}, & \text{if }m+\left\vert \mu\right\vert
\geq\left\vert \mu\right\vert ;\\
0, & \text{if }m+\left\vert \mu\right\vert <\left\vert \mu\right\vert
\end{cases}
$ is the $\left(  m+\left\vert \mu\right\vert \right)  $-th degree homogeneous
component of $G\left(  k\right)  \cdot s_{\mu}$. Since%
\begin{align*}
&
\begin{cases}
G\left(  k,\left(  m+\left\vert \mu\right\vert \right)  -\left\vert
\mu\right\vert \right)  \cdot s_{\mu}, & \text{if }m+\left\vert \mu\right\vert
\geq\left\vert \mu\right\vert ;\\
0, & \text{if }m+\left\vert \mu\right\vert <\left\vert \mu\right\vert
\end{cases}
\\
&  =G\left(  k,\underbrace{\left(  m+\left\vert \mu\right\vert \right)
-\left\vert \mu\right\vert }_{=m}\right)  \cdot s_{\mu}%
\ \ \ \ \ \ \ \ \ \ \left(  \text{since }m+\left\vert \mu\right\vert
\geq\left\vert \mu\right\vert \text{ (because }m\geq0\text{)}\right) \\
&  =G\left(  k,m\right)  \cdot s_{\mu},
\end{align*}
we can rewrite this as follows: The power series $G\left(  k,m\right)  \cdot
s_{\mu}$ is the $\left(  m+\left\vert \mu\right\vert \right)  $-th degree
homogeneous component of $G\left(  k\right)  \cdot s_{\mu}$. In other words,%
\begin{align}
&  G\left(  k,m\right)  \cdot s_{\mu}\nonumber\\
&  =\left(  \text{the }\left(  m+\left\vert \mu\right\vert \right)  \text{-th
degree homogeneous component of }G\left(  k\right)  \cdot s_{\mu}\right)  .
\label{pf.cor.Gkm.pieri.left}%
\end{align}

On the other hand, Theorem \ref{thm.G.pieri} yields%
\begin{align}
G\left(  k\right)  \cdot s_{\mu}  &  =\underbrace{\sum_{\lambda\in
\operatorname*{Par}}}_{\substack{=\sum_{n\in\mathbb{N}}\ \ \sum
_{\substack{\lambda\in\operatorname*{Par};\\\left\vert \lambda\right\vert
=n}}\\\text{(since each }\lambda\in\operatorname*{Par}\\\text{satisfies
}\left\vert \lambda\right\vert \in\mathbb{N}\text{)}}}\operatorname*{pet}%
\nolimits_{k}\left(  \lambda,\mu\right)  s_{\lambda}\nonumber\\
&  =\sum_{n\in\mathbb{N}}\ \ \sum_{\substack{\lambda\in\operatorname*{Par}%
;\\\left\vert \lambda\right\vert =n}}\operatorname*{pet}\nolimits_{k}\left(
\lambda,\mu\right)  s_{\lambda}. \label{pf.cor.Gkm.pieri.4}%
\end{align}
For each $n\in\mathbb{N}$, the formal power series $\sum_{\substack{\lambda
\in\operatorname*{Par};\\\left\vert \lambda\right\vert =n}}\operatorname*{pet}%
\nolimits_{k}\left(  \lambda,\mu\right)  s_{\lambda}$ is homogeneous of degree
$n$\ \ \ \ \footnote{\textit{Proof.} Let $n\in\mathbb{N}$. Recall that for any
$\lambda\in\operatorname*{Par}$, the Schur function $s_{\lambda}$ is
homogeneous of degree $\left\vert \lambda\right\vert $. Hence, if $\lambda
\in\operatorname*{Par}$ satisfies $\left\vert \lambda\right\vert =n$, then the
Schur function $s_{\lambda}$ is homogeneous of degree $n$ (since $\left\vert
\lambda\right\vert =n$). Thus, $\sum_{\substack{\lambda\in\operatorname*{Par}%
;\\\left\vert \lambda\right\vert =n}}\operatorname*{pet}\nolimits_{k}\left(
\lambda,\mu\right)  s_{\lambda}$ is a $\mathbf{k}$-linear combination of Schur
functions that are homogeneous of degree $n$. Therefore, $\sum
_{\substack{\lambda\in\operatorname*{Par};\\\left\vert \lambda\right\vert
=n}}\operatorname*{pet}\nolimits_{k}\left(  \lambda,\mu\right)  s_{\lambda}$
is homogeneous of degree $n$. Qed.}. Thus, the equality
(\ref{pf.cor.Gkm.pieri.4}) reveals that
\[
\left(  \sum_{\substack{\lambda\in\operatorname*{Par};\\\left\vert
\lambda\right\vert =n}}\operatorname*{pet}\nolimits_{k}\left(  \lambda
,\mu\right)  s_{\lambda}\right)  _{n\in\mathbb{N}}%
\]
is the homogeneous decomposition of $G\left(  k\right)  \cdot s_{\mu}$.
Therefore, for each $n\in\mathbb{N}$, the power series $\sum
_{\substack{\lambda\in\operatorname*{Par};\\\left\vert \lambda\right\vert
=n}}\operatorname*{pet}\nolimits_{k}\left(  \lambda,\mu\right)  s_{\lambda}$
is the $n$-th degree homogeneous component of $G\left(  k\right)  \cdot
s_{\mu}$. Applying this to $n=m+\left\vert \mu\right\vert $, we conclude that
the power series $\sum_{\substack{\lambda\in\operatorname*{Par};\\\left\vert
\lambda\right\vert =m+\left\vert \mu\right\vert }}\operatorname*{pet}%
\nolimits_{k}\left(  \lambda,\mu\right)  s_{\lambda}$ is the $\left(
m+\left\vert \mu\right\vert \right)  $-th degree homogeneous component of
$G\left(  k\right)  \cdot s_{\mu}$. In other words,%
\begin{align*}
&  \sum_{\substack{\lambda\in\operatorname*{Par};\\\left\vert \lambda
\right\vert =m+\left\vert \mu\right\vert }}\operatorname*{pet}\nolimits_{k}%
\left(  \lambda,\mu\right)  s_{\lambda}\\
&  =\left(  \text{the }\left(  m+\left\vert \mu\right\vert \right)  \text{-th
degree homogeneous component of }G\left(  k\right)  \cdot s_{\mu}\right)  .
\end{align*}
Comparing this with (\ref{pf.cor.Gkm.pieri.left}), we find%
\[
G\left(  k,m\right)  \cdot s_{\mu}=\underbrace{\sum_{\substack{\lambda
\in\operatorname*{Par};\\\left\vert \lambda\right\vert =m+\left\vert
\mu\right\vert }}}_{\substack{=\sum_{\lambda\in\operatorname*{Par}%
\nolimits_{m+\left\vert \mu\right\vert }}\\\text{(since }\operatorname*{Par}%
\nolimits_{m+\left\vert \mu\right\vert }\text{ is defined as the}\\\text{set
of all }\lambda\in\operatorname*{Par}\text{ satisfying }\left\vert
\lambda\right\vert =m+\left\vert \mu\right\vert \text{)}}}\operatorname*{pet}%
\nolimits_{k}\left(  \lambda,\mu\right)  s_{\lambda}=\sum_{\lambda
\in\operatorname*{Par}\nolimits_{m+\left\vert \mu\right\vert }}%
\operatorname*{pet}\nolimits_{k}\left(  \lambda,\mu\right)  s_{\lambda}.
\]
This proves Corollary \ref{cor.Gkm.pieri}.
\end{proof}
\end{verlong}

\begin{proof}
[Proof of Theorem \ref{thm.G.main}.]Theorem \ref{thm.G.pieri} (applied to
$\mu=\varnothing$) yields%
\[
G\left(  k\right)  \cdot s_{\varnothing}=\sum_{\lambda\in\operatorname*{Par}%
}\operatorname*{pet}\nolimits_{k}\left(  \lambda,\varnothing\right)
s_{\lambda}.
\]
Comparing this with $G\left(  k\right)  \cdot\underbrace{s_{\varnothing}}%
_{=1}=G\left(  k\right)  $, we obtain%
\[
G\left(  k\right)  =\sum_{\lambda\in\operatorname*{Par}}\operatorname*{pet}%
\nolimits_{k}\left(  \lambda,\varnothing\right)  s_{\lambda}.
\]
This proves Theorem \ref{thm.G.main}.
\end{proof}

\begin{proof}
[Proof of Corollary \ref{cor.Gkm.main}.]Corollary \ref{cor.Gkm.pieri} (applied
to $\mu=\varnothing$) yields%
\[
G\left(  k,m\right)  \cdot s_{\varnothing}=\sum_{\lambda\in\operatorname*{Par}%
\nolimits_{m+\left\vert \varnothing\right\vert }}\operatorname*{pet}%
\nolimits_{k}\left(  \lambda,\varnothing\right)  s_{\lambda}.
\]
In view of $G\left(  k,m\right)  \cdot\underbrace{s_{\varnothing}}%
_{=1}=G\left(  k,m\right)  $ and $m+\underbrace{\left\vert \varnothing
\right\vert }_{=0}=m$, we can rewrite this as
\[
G\left(  k,m\right)  =\sum_{\lambda\in\operatorname*{Par}\nolimits_{m}%
}\operatorname*{pet}\nolimits_{k}\left(  \lambda,\varnothing\right)
s_{\lambda}.
\]
This proves Corollary \ref{cor.Gkm.main}.
\end{proof}

\subsection{\label{subsect.proofs.petk.explicit}Proof of Theorem
\ref{thm.petk.explicit}}

Our proof of Theorem \ref{thm.petk.explicit} will depend on two lemmas about determinants:

\begin{lemma}
\label{lem.petk.explicit.1}Let $m\in\mathbb{N}$. Let $R$ be a commutative
ring. Let $\left(  a_{i,j}\right)  _{1\leq i\leq m,\ 1\leq j\leq m}\in
R^{m\times m}$ be an $m\times m$-matrix.

\textbf{(a)} If $\tau$ is any permutation of $\left\{  1,2,\ldots,m\right\}
$, then%
\[
\det\left(  \left(  a_{\tau\left(  i\right)  ,j}\right)  _{1\leq i\leq
m,\ 1\leq j\leq m}\right)  =\left(  -1\right)  ^{\tau}\cdot\det\left(  \left(
a_{i,j}\right)  _{1\leq i\leq m,\ 1\leq j\leq m}\right)  .
\]
Here, $\left(  -1\right)  ^{\tau}$ denotes the sign of the permutation $\tau$.

\textbf{(b)} Let $u_{1},u_{2},\ldots,u_{m}$ be $m$ elements of $R$. Let
$v_{1},v_{2},\ldots,v_{m}$ be $m$ elements of $R$. Then,%
\[
\det\left(  \left(  u_{i}v_{j}a_{i,j}\right)  _{1\leq i\leq m,\ 1\leq j\leq
m}\right)  =\left(  \prod_{i=1}^{m}\left(  u_{i}v_{i}\right)  \right)
\cdot\det\left(  \left(  a_{i,j}\right)  _{1\leq i\leq m,\ 1\leq j\leq
m}\right)  .
\]

\end{lemma}

\begin{proof}
[Proof of Lemma \ref{lem.petk.explicit.1}.]\textbf{(a)} This is just the
well-known fact that if the rows of a square matrix are permuted using a
permutation $\tau$, then the determinant of this matrix gets multiplied by
$\left(  -1\right)  ^{\tau}$.

\textbf{(b)} This follows easily from the definition of the determinant.
\end{proof}

\begin{lemma}
\label{lem.petk.explicit.2}Let $k$ be a positive integer. Let $\gamma
_{1},\gamma_{2},\ldots,\gamma_{k-1}$ be $k-1$ elements of the set $\left\{
1,2,\ldots,k\right\}  $.

Let $G$ be the $\left(  k-1\right)  \times\left(  k-1\right)  $-matrix%
\[
\left(  \left(  -1\right)  ^{\left(  \gamma_{i}+j\right)  \%k}\left[  \left(
\gamma_{i}+j\right)  \%k\in\left\{  0,1\right\}  \right]  \right)  _{1\leq
i\leq k-1,\ 1\leq j\leq k-1}.
\]

\textbf{(a)} If the $k-1$ numbers $\gamma_{1},\gamma_{2},\ldots,\gamma_{k-1}$
are not distinct, then
\[
\det G=0.
\]

\textbf{(b)} If $\gamma_{1}>\gamma_{2}>\cdots>\gamma_{k-1}$, then%
\[
\det G=\left(  -1\right)  ^{\left(  \gamma_{1}+\gamma_{2}+\cdots+\gamma
_{k-1}\right)  -\left(  1+2+\cdots+\left(  k-1\right)  \right)  }.
\]

\textbf{(c)} Assume that the $k-1$ numbers $\gamma_{1},\gamma_{2}%
,\ldots,\gamma_{k-1}$ are distinct. Let%
\[
g=\left\vert \left\{  \left(  i,j\right)  \in\left\{  1,2,\ldots,k-1\right\}
^{2}\ \mid\ i<j\text{ and }\gamma_{i}<\gamma_{j}\right\}  \right\vert .
\]
Then,%
\[
\det G=\left(  -1\right)  ^{g+\left(  \gamma_{1}+\gamma_{2}+\cdots
+\gamma_{k-1}\right)  -\left(  1+2+\cdots+\left(  k-1\right)  \right)  }.
\]

\end{lemma}

\begin{proof}
[Proof of Lemma \ref{lem.petk.explicit.2}.]\textbf{(a)} Assume that the $k-1$
numbers $\gamma_{1},\gamma_{2},\ldots,\gamma_{k-1}$ are not distinct. In other
words, there exist two elements $u$ and $v$ of $\left\{  1,2,\ldots
,k-1\right\}  $ such that $u<v$ and $\gamma_{u}=\gamma_{v}$. Consider these
$u$ and $v$. Now, from $\gamma_{u}=\gamma_{v}$, we conclude that the $u$-th
and the $v$-th rows of the matrix $G$ are equal (by the construction of $G$).
Hence, the matrix $G$ has two equal rows (since $u<v$). Thus, $\det G=0$. This
proves Lemma \ref{lem.petk.explicit.2} \textbf{(a)}.

\textbf{(b)} Assume that $\gamma_{1}>\gamma_{2}>\cdots>\gamma_{k-1}$. Thus,
$\gamma_{1},\gamma_{2},\ldots,\gamma_{k-1}$ are distinct. Hence, $\left\{
\gamma_{1},\gamma_{2},\ldots,\gamma_{k-1}\right\}  $ is a $\left(  k-1\right)
$-element set. But $\left\{  \gamma_{1},\gamma_{2},\ldots,\gamma
_{k-1}\right\}  $ is a subset of $\left\{  1,2,\ldots,k\right\}  $ (since
$\gamma_{1},\gamma_{2},\ldots,\gamma_{k-1}$ are elements of $\left\{
1,2,\ldots,k\right\}  $). Therefore, $\left\{  \gamma_{1},\gamma_{2}%
,\ldots,\gamma_{k-1}\right\}  $ is a $\left(  k-1\right)  $-element subset of
$\left\{  1,2,\ldots,k\right\}  $.

\begin{vershort}
Hence, $\left\{  \gamma_{1},\gamma_{2},\ldots,\gamma_{k-1}\right\}  =\left\{
1,2,\ldots,k\right\}  \setminus\left\{  u\right\}  $ for some $u\in\left\{
1,2,\ldots,k\right\}  $ (since any $\left(  k-1\right)  $-element subset of
$\left\{  1,2,\ldots,k\right\}  $ has such a form). Consider this $u$. From
$\left\{  \gamma_{1},\gamma_{2},\ldots,\gamma_{k-1}\right\}  =\left\{
1,2,\ldots,k\right\}  \setminus\left\{  u\right\}  $, we conclude that
$\gamma_{1},\gamma_{2},\ldots,\gamma_{k-1}$ are the $k-1$ elements of the set
$\left\{  1,2,\ldots,k\right\}  \setminus\left\{  u\right\}  $, listed in
decreasing order (since $\gamma_{1}>\gamma_{2}>\cdots>\gamma_{k-1}$). In other
words,%
\begin{equation}
\left(  \gamma_{1},\gamma_{2},\ldots,\gamma_{k-1}\right)  =\left(
k,k-1,\ldots,\widehat{u},\ldots,2,1\right)  ,
\label{pf.lem.petk.explicit.2.b.short.2}%
\end{equation}
where the \textquotedblleft hat\textquotedblright\ over the $u$ signifies that
$u$ is omitted from the list (i.e., the expression \textquotedblleft$\left(
k,k-1,\ldots,\widehat{u},\ldots,2,1\right)  $\textquotedblright\ is understood
to mean the $\left(  k-1\right)  $-element list $\left(  k,k-1,\ldots
,u+1,u-1,\ldots,2,1\right)  $, which contains all $k$ integers from $1$ to $k$
in decreasing order except for $u$).
\end{vershort}

\begin{verlong}
But $\left\{  1,2,\ldots,k\right\}  $ is a $k$-element set. Hence, each
$\left(  k-1\right)  $-element subset of $\left\{  1,2,\ldots,k\right\}  $ has
the form $\left\{  1,2,\ldots,k\right\}  \setminus\left\{  u\right\}  $ for
some $u\in\left\{  1,2,\ldots,k\right\}  $. Thus, in particular, $\left\{
\gamma_{1},\gamma_{2},\ldots,\gamma_{k-1}\right\}  $ has this form (since
$\left\{  \gamma_{1},\gamma_{2},\ldots,\gamma_{k-1}\right\}  $ is a $\left(
k-1\right)  $-element subset of $\left\{  1,2,\ldots,k\right\}  $). In other
words,%
\begin{equation}
\left\{  \gamma_{1},\gamma_{2},\ldots,\gamma_{k-1}\right\}  =\left\{
1,2,\ldots,k\right\}  \setminus\left\{  u\right\}
\label{pf.lem.petk.explicit.2.b.1}%
\end{equation}
for some $u\in\left\{  1,2,\ldots,k\right\}  $. Consider this $u$. From
(\ref{pf.lem.petk.explicit.2.b.1}), we conclude that $\gamma_{1},\gamma
_{2},\ldots,\gamma_{k-1}$ are the $k-1$ elements of the set $\left\{
1,2,\ldots,k\right\}  \setminus\left\{  u\right\}  $, listed in decreasing
order (since $\gamma_{1}>\gamma_{2}>\cdots>\gamma_{k-1}$). In other words,%
\begin{equation}
\left(  \gamma_{1},\gamma_{2},\ldots,\gamma_{k-1}\right)  =\left(
k,k-1,\ldots,\widehat{u},\ldots,2,1\right)  ,
\label{pf.lem.petk.explicit.2.b.2}%
\end{equation}
where the \textquotedblleft hat\textquotedblright\ over the $u$ signifies that
$u$ is omitted from the list (i.e., the expression \textquotedblleft$\left(
k,k-1,\ldots,\widehat{u},\ldots,2,1\right)  $\textquotedblright\ is understood
to mean the $\left(  k-1\right)  $-element list $\left(  k,k-1,\ldots
,u+1,u-1,\ldots,2,1\right)  $, which contains all $k$ integers from $1$ to $k$
in decreasing order except for $u$). Thus,%
\begin{align}
\left(  \gamma_{1},\gamma_{2},\ldots,\gamma_{k-u}\right)   &  =\left(
k,k-1,\ldots,u+1\right)  \ \ \ \ \ \ \ \ \ \ \text{and}%
\label{pf.lem.petk.explicit.2.b.2a}\\
\left(  \gamma_{k-u+1},\gamma_{k-u+2},\ldots,\gamma_{k-1}\right)   &  =\left(
u-1,u-2,\ldots,1\right)  . \label{pf.lem.petk.explicit.2.b.2b}%
\end{align}

\end{verlong}

Now, we claim that%
\begin{align}
&  \left(  -1\right)  ^{\left(  \gamma_{i}+j\right)  \%k}\left[  \left(
\gamma_{i}+j\right)  \%k\in\left\{  0,1\right\}  \right] \nonumber\\
&  =\left(  -1\right)  ^{\gamma_{i}+j-k}\left[  \gamma_{i}+j\in\left\{
k,k+1\right\}  \right]  \label{pf.lem.petk.explicit.2.b.5}%
\end{align}
for any $i\in\left\{  1,2,\ldots,k-1\right\}  $ and $j\in\left\{
1,2,\ldots,k-1\right\}  $.

[\textit{Proof of (\ref{pf.lem.petk.explicit.2.b.5}):} Let $i\in\left\{
1,2,\ldots,k-1\right\}  $ and $j\in\left\{  1,2,\ldots,k-1\right\}  $. We must
prove the equality (\ref{pf.lem.petk.explicit.2.b.5}).

From $i\in\left\{  1,2,\ldots,k-1\right\}  $, we obtain $1\leq i\leq k-1$ and
thus $k-1\geq1$. Thus, $k>k-1\geq1$. Hence, $k+1<2k$, so that $\left(
k+1\right)  \%k=1$.

\begin{vershort}
If we don't have $\left(  \gamma_{i}+j\right)  \%k\in\left\{  0,1\right\}  $,
then both truth values \newline$\left[  \left(  \gamma_{i}+j\right)
\%k\in\left\{  0,1\right\}  \right]  $ and $\left[  \gamma_{i}+j\in\left\{
k,k+1\right\}  \right]  $ are $0$ (indeed, the statement \textquotedblleft%
$\gamma_{i}+j\in\left\{  k,k+1\right\}  $\textquotedblright\ is false, since
otherwise it would imply $\left(  \gamma_{i}+j\right)  \%k\in\left\{
0,1\right\}  $), and therefore the equality (\ref{pf.lem.petk.explicit.2.b.5})
simplifies to $\left(  -1\right)  ^{\left(  \gamma_{i}+j\right)  \%k}0=\left(
-1\right)  ^{\gamma_{i}+j-k}0$ in this case, which is obviously true. Hence,
for the rest of this proof, we WLOG assume that we do have $\left(  \gamma
_{i}+j\right)  \%k\in\left\{  0,1\right\}  $.
\end{vershort}

\begin{verlong}
If we don't have $\left(  \gamma_{i}+j\right)  \%k\in\left\{  0,1\right\}  $,
then we cannot have $\gamma_{i}+j\in\left\{  k,k+1\right\}  $ either (because
$\gamma_{i}+j\in\left\{  k,k+1\right\}  $ would entail $\left(  \gamma
_{i}+j\right)  \%k\in\left\{  \underbrace{k\%k}_{=0},\underbrace{\left(
k+1\right)  \%k}_{=1}\right\}  =\left\{  0,1\right\}  $). Thus, if we don't
have $\left(  \gamma_{i}+j\right)  \%k\in\left\{  0,1\right\}  $, then both
truth values \newline$\left[  \left(  \gamma_{i}+j\right)  \%k\in\left\{
0,1\right\}  \right]  $ and $\left[  \gamma_{i}+j\in\left\{  k,k+1\right\}
\right]  $ are $0$, and therefore the equality
(\ref{pf.lem.petk.explicit.2.b.5}) simplifies to $\left(  -1\right)  ^{\left(
\gamma_{i}+j\right)  \%k}0=\left(  -1\right)  ^{\gamma_{i}+j-k}0$ in this
case, which is obviously true. Hence, for the rest of this proof, we WLOG
assume that we do have $\left(  \gamma_{i}+j\right)  \%k\in\left\{
0,1\right\}  $.
\end{verlong}

\begin{vershort}
Adding $\gamma_{i}\in\left\{  1,2,\ldots,k\right\}  $ with $j\in\left\{
1,2,\ldots,k-1\right\}  $, we find $\gamma_{i}+j\in\left\{  2,3,\ldots
,2k-1\right\}  $. But if $p\in\left\{  2,3,\ldots,2k-1\right\}  $ satisfies
$p\%k\in\left\{  0,1\right\}  $, then $p\in\left\{  k,k+1\right\}  $ (since
the only numbers in $\left\{  2,3,\ldots,2k-1\right\}  $ that leave the
remainders $0$ and $1$ upon division by $k$ are $k$ and $k+1$). Applying this
to $p=\gamma_{i}+j$, we obtain $\gamma_{i}+j\in\left\{  k,k+1\right\}  $
(since $\gamma_{i}+j\in\left\{  2,3,\ldots,2k-1\right\}  $ and $\left(
\gamma_{i}+j\right)  \%k\in\left\{  0,1\right\}  $). Hence, $k\leq\gamma
_{i}+j<2k$ (since $k+1<2k$), so that $\left(  \gamma_{i}+j\right)  //k=1$. But
every integer $n$ satisfies $n=\left(  n//k\right)  k+\left(  n\%k\right)  $.
Applying this to $n=\gamma_{i}+j$, we obtain $\gamma_{i}+j=\underbrace{\left(
\left(  \gamma_{i}+j\right)  //k\right)  }_{=1}k+\left(  \left(  \gamma
_{i}+j\right)  \%k\right)  =k+\left(  \left(  \gamma_{i}+j\right)  \%k\right)
$. Hence, $\left(  \gamma_{i}+j\right)  \%k=\gamma_{i}+j-k$. Thus,%
\[
\underbrace{\left(  -1\right)  ^{\left(  \gamma_{i}+j\right)  \%k}%
}_{\substack{=\left(  -1\right)  ^{\gamma_{i}+j-k}\\\text{(since }\left(
\gamma_{i}+j\right)  \%k=\gamma_{i}+j-k\text{)}}}\underbrace{\left[  \left(
\gamma_{i}+j\right)  \%k\in\left\{  0,1\right\}  \right]  }%
_{\substack{=1\\\text{(since }\left(  \gamma_{i}+j\right)  \%k\in\left\{
0,1\right\}  \text{)}}}=\left(  -1\right)  ^{\gamma_{i}+j-k}.
\]

\end{vershort}

\begin{verlong}
But $\gamma_{i}\in\left\{  1,2,\ldots,k\right\}  $, so that $1\leq\gamma
_{i}\leq k$. Also, $j\in\left\{  1,2,\ldots,k-1\right\}  $, so that $1\leq
j\leq k-1$. Hence, $\underbrace{\gamma_{i}}_{\geq1}+\underbrace{j}_{\geq1}%
\geq1+1=2$ and $\underbrace{\gamma_{i}}_{\leq k}+\underbrace{j}_{\leq k-1}\leq
k+\left(  k-1\right)  =2k-1$. Altogether, we thus obtain $2\leq\gamma
_{i}+j\leq2k-1$, so that $\gamma_{i}+j\in\left\{  2,3,\ldots,2k-1\right\}  $.

The remainders of the numbers $2,3,\ldots,2k-1$ upon division by $k$ are
$2,3,\ldots,k-1,0,1,\ldots,k-1$ (in this order). Thus, the only numbers
$p\in\left\{  2,3,\ldots,2k-1\right\}  $ that satisfy $p\%k\in\left\{
0,1\right\}  $ are $k$ and $k+1$. In other words, for any number $p\in\left\{
2,3,\ldots,2k-1\right\}  $ satisfying $p\%k\in\left\{  0,1\right\}  $, we have
$p\in\left\{  k,k+1\right\}  $. Applying this to $p=\gamma_{i}+j$, we obtain
$\gamma_{i}+j\in\left\{  k,k+1\right\}  $ (since $\gamma_{i}+j\in\left\{
2,3,\ldots,2k-1\right\}  $ and $\left(  \gamma_{i}+j\right)  \%k\in\left\{
0,1\right\}  $). Hence, $k\leq\gamma_{i}+j\leq k+1$, so that $k\leq\gamma
_{i}+j<2k$ (since $k+1<2k$). Thus, $\left(  \gamma_{i}+j\right)  //k=1$. But
every integer $n$ satisfies $n=\left(  n//k\right)  k+\left(  n\%k\right)  $.
Applying this to $n=\gamma_{i}+j$, we obtain $\gamma_{i}+j=\underbrace{\left(
\left(  \gamma_{i}+j\right)  //k\right)  }_{=1}k+\left(  \left(  \gamma
_{i}+j\right)  \%k\right)  =k+\left(  \left(  \gamma_{i}+j\right)  \%k\right)
$. Hence, $\left(  \gamma_{i}+j\right)  \%k=\gamma_{i}+j-k$. Thus,%
\[
\underbrace{\left(  -1\right)  ^{\left(  \gamma_{i}+j\right)  \%k}%
}_{\substack{=\left(  -1\right)  ^{\gamma_{i}+j-k}\\\text{(since }\left(
\gamma_{i}+j\right)  \%k=\gamma_{i}+j-k\text{)}}}\underbrace{\left[  \left(
\gamma_{i}+j\right)  \%k\in\left\{  0,1\right\}  \right]  }%
_{\substack{=1\\\text{(since }\left(  \gamma_{i}+j\right)  \%k\in\left\{
0,1\right\}  \text{)}}}=\left(  -1\right)  ^{\gamma_{i}+j-k}.
\]

\end{verlong}

\noindent Comparing this with%
\[
\left(  -1\right)  ^{\gamma_{i}+j-k}\underbrace{\left[  \gamma_{i}%
+j\in\left\{  k,k+1\right\}  \right]  }_{\substack{=1\\\text{(since }%
\gamma_{i}+j\in\left\{  k,k+1\right\}  \text{)}}}=\left(  -1\right)
^{\gamma_{i}+j-k},
\]
we obtain%
\[
\left(  -1\right)  ^{\left(  \gamma_{i}+j\right)  \%k}\left[  \left(
\gamma_{i}+j\right)  \%k\in\left\{  0,1\right\}  \right]  =\left(  -1\right)
^{\gamma_{i}+j-k}\left[  \gamma_{i}+j\in\left\{  k,k+1\right\}  \right]  .
\]
This proves (\ref{pf.lem.petk.explicit.2.b.5}).]

Now, $G$ is a $\left(  k-1\right)  \times\left(  k-1\right)  $-matrix. For
each $i\in\left\{  1,2,\ldots,k-1\right\}  $ and $j\in\left\{  1,2,\ldots
,k-1\right\}  $, we have%
\begin{align*}
&  \left(  \text{the }\left(  i,j\right)  \text{-th entry of }G\right) \\
&  =\left(  -1\right)  ^{\left(  \gamma_{i}+j\right)  \%k}\left[  \left(
\gamma_{i}+j\right)  \%k\in\left\{  0,1\right\}  \right]
\ \ \ \ \ \ \ \ \ \ \left(  \text{by the definition of }G\right) \\
&  =\left(  -1\right)  ^{\gamma_{i}+j-k}\left[  \gamma_{i}+j\in\left\{
k,k+1\right\}  \right]  \ \ \ \ \ \ \ \ \ \ \left(  \text{by
(\ref{pf.lem.petk.explicit.2.b.5})}\right) \\
&  =%
\begin{cases}
1, & \text{if }\gamma_{i}+j=k;\\
-1, & \text{if }\gamma_{i}+j=k+1;\\
0, & \text{otherwise}%
\end{cases}
\quad=%
\begin{cases}
1, & \text{if }j=k-\gamma_{i};\\
-1, & \text{if }j=k-\gamma_{i}+1;\\
0, & \text{otherwise}%
\end{cases}
\ \ \ .
\end{align*}

\begin{vershort}
\noindent Thus, we can explicitly describe the matrix $G$ as follows: For each
$i\in\left\{  1,2,\ldots,k-1\right\}  $, the $i$-th row of $G$ has an entry
equal to $1$ in position $k-\gamma_{i}$ if $k-\gamma_{i}>0$, and an entry
equal to $-1$ in position $k-\gamma_{i}+1$ if $k-\gamma_{i}+1<k$; all
remaining entries of this row are $0$. Recalling
(\ref{pf.lem.petk.explicit.2.b.short.2}), we thus see that $G$ has the
following form:\footnote{Empty cells are understood to have entry $0$.}%
\[
G=\left(
\begin{tabular}
[c]{ccccc|ccccc}%
$-1$ &  &  &  &  &  &  &  &  & \\
$1$ & $-1$ &  &  &  &  &  &  &  & \\
& $1$ & $-1$ &  &  &  &  &  &  & \\
&  & $\ddots$ & $\ddots$ &  &  &  &  &  & \\
&  &  & $1$ & $-1$ &  &  &  &  & \\\hline
&  &  &  &  & $1$ & $-1$ &  &  & \\
&  &  &  &  &  & $1$ & $-1$ &  & \\
&  &  &  &  &  &  & $1$ & $-1$ & \\
&  &  &  &  &  &  &  & $\ddots$ & $\ddots$\\
&  &  &  &  &  &  &  &  & $1$%
\end{tabular}
\right)  ,
\]
where the horizontal bar separates the $\left(  k-u\right)  $-th row from the
$\left(  k-u+1\right)  $-st row, while the vertical bar separates the $\left(
k-u\right)  $-th column from the $\left(  k-u+1\right)  $-st column. Thus, $G$
can be written as a block-diagonal matrix
\begin{equation}
G=\left(
\begin{array}
[c]{cc}%
A & 0_{\left(  k-u\right)  \times\left(  u-1\right)  }\\
0_{\left(  u-1\right)  \times\left(  k-u\right)  } & B
\end{array}
\right)  , \label{pf.lem.petk.explicit.2.b.short.blocks}%
\end{equation}
where $A$ is a lower-triangular $\left(  k-u\right)  \times\left(  k-u\right)
$-matrix with all diagonal entries equal to $-1$, and where $B$ is an
upper-triangular $\left(  u-1\right)  \times\left(  u-1\right)  $-matrix with
all diagonal entries equal to $1$. Since the determinant of a block-diagonal
matrix equals the product of the determinants of its diagonal blocks, we thus
conclude that%
\begin{align}
\det G  &  =\underbrace{\det A}_{\substack{=\left(  -1\right)  ^{k-u}%
\\\text{(since }A\text{ is a}\\\text{lower-triangular }\left(  k-u\right)
\times\left(  k-u\right)  \text{-matrix}\\\text{with all diagonal entries
equal to }-1\text{)}}}\cdot\underbrace{\det B}_{\substack{=1\\\text{(since
}B\text{ is an}\\\text{upper-triangular }\left(  u-1\right)  \times\left(
u-1\right)  \text{-matrix}\\\text{with all diagonal entries equal to
}1\text{)}}}\nonumber\\
&  =\left(  -1\right)  ^{k-u}. \label{pf.lem.petk.explicit.2.b.short.7}%
\end{align}

But (\ref{pf.lem.petk.explicit.2.b.short.2}) yields%
\begin{align*}
\gamma_{1}+\gamma_{2}+\cdots+\gamma_{k-1}  &  =k+\left(  k-1\right)
+\cdots+\widehat{u}+\cdots+2+1\\
&  =\underbrace{\left(  k+\left(  k-1\right)  +\cdots+2+1\right)
}_{\substack{=1+2+\cdots+k\\=\left(  1+2+\cdots+\left(  k-1\right)  \right)
+k}}-u\\
&  =\left(  1+2+\cdots+\left(  k-1\right)  \right)  +k-u.
\end{align*}
Solving this for $k-u$, we find%
\[
k-u=\left(  \gamma_{1}+\gamma_{2}+\cdots+\gamma_{k-1}\right)  -\left(
1+2+\cdots+\left(  k-1\right)  \right)  .
\]
Hence, (\ref{pf.lem.petk.explicit.2.b.short.7}) rewrites as
\[
\det G=\left(  -1\right)  ^{\left(  \gamma_{1}+\gamma_{2}+\cdots+\gamma
_{k-1}\right)  -\left(  1+2+\cdots+\left(  k-1\right)  \right)  }.
\]

\end{vershort}

\begin{verlong}
\noindent Thus, we can explicitly describe the matrix $G$ as follows: For each
$i\in\left\{  1,2,\ldots,k-1\right\}  $, the $i$-th row of $G$ has an entry
equal to $1$ in position $k-\gamma_{i}$ if $k-\gamma_{i}>0$, and an entry
equal to $-1$ in position $k-\gamma_{i}+1$ if $k-\gamma_{i}+1<k$; all
remaining entries of this row are $0$. Recalling
(\ref{pf.lem.petk.explicit.2.b.2a}) and (\ref{pf.lem.petk.explicit.2.b.2b}),
we thus see that $G$ has the following form:\footnote{Empty cells are
understood to have entry $0$.}%
\[
G=\left(
\begin{tabular}
[c]{ccccc|ccccc}%
$-1$ &  &  &  &  &  &  &  &  & \\
$1$ & $-1$ &  &  &  &  &  &  &  & \\
& $1$ & $-1$ &  &  &  &  &  &  & \\
&  & $\ddots$ & $\ddots$ &  &  &  &  &  & \\
&  &  & $1$ & $-1$ &  &  &  &  & \\\hline
&  &  &  &  & $1$ & $-1$ &  &  & \\
&  &  &  &  &  & $1$ & $-1$ &  & \\
&  &  &  &  &  &  & $1$ & $-1$ & \\
&  &  &  &  &  &  &  & $\ddots$ & $\ddots$\\
&  &  &  &  &  &  &  &  & $1$%
\end{tabular}
\right)  ,
\]
where the horizontal bar separates the $\left(  k-u\right)  $-th row from the
$\left(  k-u+1\right)  $-st row, while the vertical bar separates the $\left(
k-u\right)  $-th column from the $\left(  k-u+1\right)  $-st column. In other
words, $G$ can be written as a block matrix
\begin{equation}
G=\left(
\begin{array}
[c]{cc}%
A & 0_{\left(  k-u\right)  \times\left(  u-1\right)  }\\
0_{\left(  u-1\right)  \times\left(  k-u\right)  } & B
\end{array}
\right)  , \label{pf.lem.petk.explicit.2.b.blocks}%
\end{equation}
where $A$ is the $\left(  k-u\right)  \times\left(  k-u\right)  $-matrix
$\left(
\begin{array}
[c]{ccccc}%
-1 &  &  &  & \\
1 & -1 &  &  & \\
& 1 & -1 &  & \\
&  & \ddots & \ddots & \\
&  &  & 1 & -1
\end{array}
\right)  $ (that is, the $\left(  k-u\right)  \times\left(  k-u\right)
$-matrix whose diagonal entries are $-1$ and whose entries immediately below
the diagonal are $1$, while all its other entries are $0$), and where $B$ is
the $\left(  u-1\right)  \times\left(  u-1\right)  $-matrix $\left(
\begin{array}
[c]{ccccc}%
1 & -1 &  &  & \\
& 1 & -1 &  & \\
&  & 1 & -1 & \\
&  &  & \ddots & \ddots\\
&  &  &  & 1
\end{array}
\right)  $ (that is, the $\left(  u-1\right)  \times\left(  u-1\right)
$-matrix whose diagonal entries are $1$ and whose entries immediately above
the diagonal are $-1$, while all its other entries are $0$). Thus, $G$ (as
written in (\ref{pf.lem.petk.explicit.2.b.blocks})) is a block-diagonal matrix
(since $A$ and $B$ are square matrices). Since the determinant of a
block-diagonal matrix equals the product of the determinants of its diagonal
blocks, we thus conclude that%
\begin{align}
\det G  &  =\underbrace{\det A}_{\substack{=\left(  -1\right)  ^{k-u}%
\\\text{(since }A\text{ is a}\\\text{lower-triangular }\left(  k-u\right)
\times\left(  k-u\right)  \text{-matrix}\\\text{whose all diagonal entries
equal }-1\text{)}}}\cdot\underbrace{\det B}_{\substack{=1\\\text{(since
}B\text{ is an}\\\text{upper-triangular }\left(  u-1\right)  \times\left(
u-1\right)  \text{-matrix}\\\text{whose all diagonal entries equal }1\text{)}%
}}\nonumber\\
&  =\left(  -1\right)  ^{k-u}. \label{pf.lem.petk.explicit.2.b.7}%
\end{align}

But (\ref{pf.lem.petk.explicit.2.b.2}) yields%
\begin{align*}
\gamma_{1}+\gamma_{2}+\cdots+\gamma_{k-1}  &  =k+\left(  k-1\right)
+\cdots+\widehat{u}+\cdots+2+1\\
&  =\underbrace{\left(  k+\left(  k-1\right)  +\cdots+2+1\right)
}_{\substack{=1+2+\cdots+k\\=\left(  1+2+\cdots+\left(  k-1\right)  \right)
+k}}-u\\
&  =\left(  1+2+\cdots+\left(  k-1\right)  \right)  +k-u.
\end{align*}
Solving this for $k-u$, we find%
\[
k-u=\left(  \gamma_{1}+\gamma_{2}+\cdots+\gamma_{k-1}\right)  -\left(
1+2+\cdots+\left(  k-1\right)  \right)  .
\]
Hence, (\ref{pf.lem.petk.explicit.2.b.7}) rewrites as
\[
\det G=\left(  -1\right)  ^{\left(  \gamma_{1}+\gamma_{2}+\cdots+\gamma
_{k-1}\right)  -\left(  1+2+\cdots+\left(  k-1\right)  \right)  }.
\]

\end{verlong}

\noindent This proves Lemma \ref{lem.petk.explicit.2} \textbf{(b)}.

\textbf{(c)} Assume that the $k-1$ numbers $\gamma_{1},\gamma_{2}%
,\ldots,\gamma_{k-1}$ are distinct. Then, there exists a unique permutation
$\sigma$ of $\left\{  1,2,\ldots,k-1\right\}  $ such that $\gamma
_{\sigma\left(  1\right)  }>\gamma_{\sigma\left(  2\right)  }>\cdots
>\gamma_{\sigma\left(  k-1\right)  }$ (indeed, this is simply saying that the
$\left(  k-1\right)  $-tuple $\left(  \gamma_{1},\gamma_{2},\ldots
,\gamma_{k-1}\right)  $ can be sorted into decreasing order by a unique
permutation). Consider this $\sigma$.

Let $\tau$ denote the permutation $\sigma^{-1}$. Thus, $\tau$ is a permutation
of $\left\{  1,2,\ldots,k-1\right\}  $ and satisfies $\sigma\circ
\tau=\operatorname*{id}$.

Let $\delta_{1},\delta_{2},\ldots,\delta_{k-1}$ denote the $k-1$ elements
$\gamma_{\sigma\left(  1\right)  },\gamma_{\sigma\left(  2\right)  }%
,\ldots,\gamma_{\sigma\left(  k-1\right)  }$ of $\left\{  1,2,\ldots
,k\right\}  $. Thus, for each $j\in\left\{  1,2,\ldots,k-1\right\}  $, we have%
\begin{equation}
\delta_{j}=\gamma_{\sigma\left(  j\right)  }.
\label{pf.lem.petk.explicit.2.c.1}%
\end{equation}
Hence, the chain of inequalities $\gamma_{\sigma\left(  1\right)  }%
>\gamma_{\sigma\left(  2\right)  }>\cdots>\gamma_{\sigma\left(  k-1\right)  }$
(which is true) can be rewritten as $\delta_{1}>\delta_{2}>\cdots>\delta
_{k-1}$.

Moreover, from (\ref{pf.lem.petk.explicit.2.c.1}), we obtain%
\begin{align}
\delta_{1}+\delta_{2}+\cdots+\delta_{k-1}  &  =\gamma_{\sigma\left(  1\right)
}+\gamma_{\sigma\left(  2\right)  }+\cdots+\gamma_{\sigma\left(  k-1\right)
}\nonumber\\
&  =\gamma_{1}+\gamma_{2}+\cdots+\gamma_{k-1}
\label{pf.lem.petk.explicit.2.c.sumdel=sumgam}%
\end{align}
(since $\sigma$ is a permutation of $\left\{  1,2,\ldots,k-1\right\}  $).

Moreover, for each $i\in\left\{  1,2,\ldots,k-1\right\}  $, we have%
\begin{align}
\delta_{\tau\left(  i\right)  }  &  =\gamma_{\sigma\left(  \tau\left(
i\right)  \right)  }\ \ \ \ \ \ \ \ \ \ \left(  \text{by
(\ref{pf.lem.petk.explicit.2.c.1}), applied to }j=\tau\left(  i\right)
\right) \nonumber\\
&  =\gamma_{i}\ \ \ \ \ \ \ \ \ \ \left(  \text{since }\sigma\left(
\tau\left(  i\right)  \right)  =\underbrace{\left(  \sigma\circ\tau\right)
}_{=\operatorname*{id}}\left(  i\right)  =i\right)  .
\label{pf.lem.petk.explicit.2.c.2}%
\end{align}

Recall that an \emph{inversion} of the permutation $\tau$ is defined to be a
pair $\left(  i,j\right)  $ of elements of $\left\{  1,2,\ldots,k-1\right\}  $
satisfying $i<j$ and $\tau\left(  i\right)  >\tau\left(  j\right)  $. Hence,%
\begin{align}
&  \left\{  \text{the inversions of }\tau\right\} \nonumber\\
&  =\left\{  \left(  i,j\right)  \in\left\{  1,2,\ldots,k-1\right\}
^{2}\ \mid\ i<j\text{ and }\underbrace{\tau\left(  i\right)  >\tau\left(
j\right)  }_{\substack{\text{This is equivalent to }\left(  \delta
_{\tau\left(  i\right)  }<\delta_{\tau\left(  j\right)  }\right)
\\\text{(since }\delta_{1}>\delta_{2}>\cdots>\delta_{k-1}\text{)}}}\right\}
\nonumber\\
&  =\left\{  \left(  i,j\right)  \in\left\{  1,2,\ldots,k-1\right\}
^{2}\ \mid\ i<j\text{ and }\underbrace{\delta_{\tau\left(  i\right)  }%
}_{\substack{=\gamma_{i}\\\text{(by (\ref{pf.lem.petk.explicit.2.c.2}))}%
}}<\underbrace{\delta_{\tau\left(  j\right)  }}_{\substack{=\gamma
_{j}\\\text{(by (\ref{pf.lem.petk.explicit.2.c.2}))}}}\right\} \nonumber\\
&  =\left\{  \left(  i,j\right)  \in\left\{  1,2,\ldots,k-1\right\}
^{2}\ \mid\ i<j\text{ and }\gamma_{i}<\gamma_{j}\right\}  .
\label{pf.lem.petk.explicit.2.c.Inv}%
\end{align}

Recall that the \emph{length} $\ell\left(  \tau\right)  $ of the permutation
$\tau$ is defined to be the number of inversions of $\tau$. Thus,%
\begin{align*}
\ell\left(  \tau\right)   &  =\left(  \text{the number of inversions of }%
\tau\right) \\
&  =\left\vert \left\{  \text{the inversions of }\tau\right\}  \right\vert \\
&  =\left\vert \left\{  \left(  i,j\right)  \in\left\{  1,2,\ldots
,k-1\right\}  ^{2}\ \mid\ i<j\text{ and }\gamma_{i}<\gamma_{j}\right\}
\right\vert \ \ \ \ \ \ \ \ \ \ \left(  \text{by
(\ref{pf.lem.petk.explicit.2.c.Inv})}\right) \\
&  =g\ \ \ \ \ \ \ \ \ \ \left(  \text{by the definition of }g\right)  .
\end{align*}

Recall that the \emph{sign} $\left(  -1\right)  ^{\tau}$ of the permutation
$\tau$ is defined by $\left(  -1\right)  ^{\tau}=\left(  -1\right)
^{\ell\left(  \tau\right)  }$. Hence, $\left(  -1\right)  ^{\tau}=\left(
-1\right)  ^{\ell\left(  \tau\right)  }=\left(  -1\right)  ^{g}$ (since
$\ell\left(  \tau\right)  =g$).

Let $H$ be the $\left(  k-1\right)  \times\left(  k-1\right)  $-matrix%
\[
\left(  \left(  -1\right)  ^{\left(  \delta_{i}+j\right)  \%k}\left[  \left(
\delta_{i}+j\right)  \%k\in\left\{  0,1\right\}  \right]  \right)  _{1\leq
i\leq k-1,\ 1\leq j\leq k-1}.
\]
Then, we can apply Lemma \ref{lem.petk.explicit.2} \textbf{(b)} to $\delta
_{i}$ and $H$ instead of $\gamma_{i}$ and $G$ (since $\delta_{1},\delta
_{2},\ldots,\delta_{k-1}$ are $k-1$ elements of $\left\{  1,2,\ldots
,k\right\}  $ and satisfy $\delta_{1}>\delta_{2}>\cdots>\delta_{k-1}$). We
thus obtain%
\[
\det H=\left(  -1\right)  ^{\left(  \delta_{1}+\delta_{2}+\cdots+\delta
_{k-1}\right)  -\left(  1+2+\cdots+\left(  k-1\right)  \right)  }=\left(
-1\right)  ^{\left(  \gamma_{1}+\gamma_{2}+\cdots+\gamma_{k-1}\right)
-\left(  1+2+\cdots+\left(  k-1\right)  \right)  }%
\]
(by (\ref{pf.lem.petk.explicit.2.c.sumdel=sumgam})).

But the definition of $G$ yields%
\begin{align*}
G  &  =\left(  \underbrace{\left(  -1\right)  ^{\left(  \gamma_{i}+j\right)
\%k}\left[  \left(  \gamma_{i}+j\right)  \%k\in\left\{  0,1\right\}  \right]
}_{\substack{=\left(  -1\right)  ^{\left(  \delta_{\tau\left(  i\right)
}+j\right)  \%k}\left[  \left(  \delta_{\tau\left(  i\right)  }+j\right)
\%k\in\left\{  0,1\right\}  \right]  \\\text{(since
(\ref{pf.lem.petk.explicit.2.c.2}) yields }\gamma_{i}=\delta_{\tau\left(
i\right)  }\text{)}}}\right)  _{1\leq i\leq k-1,\ 1\leq j\leq k-1}\\
&  =\left(  \left(  -1\right)  ^{\left(  \delta_{\tau\left(  i\right)
}+j\right)  \%k}\left[  \left(  \delta_{\tau\left(  i\right)  }+j\right)
\%k\in\left\{  0,1\right\}  \right]  \right)  _{1\leq i\leq k-1,\ 1\leq j\leq
k-1}.
\end{align*}
Hence,%
\begin{align*}
\det G  &  =\det\left(  \left(  \left(  -1\right)  ^{\left(  \delta
_{\tau\left(  i\right)  }+j\right)  \%k}\left[  \left(  \delta_{\tau\left(
i\right)  }+j\right)  \%k\in\left\{  0,1\right\}  \right]  \right)  _{1\leq
i\leq k-1,\ 1\leq j\leq k-1}\right) \\
&  =\underbrace{\left(  -1\right)  ^{\tau}}_{=\left(  -1\right)  ^{g}}%
\det\left(  \underbrace{\left(  \left(  -1\right)  ^{\left(  \delta
_{i}+j\right)  \%k}\left[  \left(  \delta_{i}+j\right)  \%k\in\left\{
0,1\right\}  \right]  \right)  _{1\leq i\leq k-1,\ 1\leq j\leq k-1}%
}_{\substack{=H\\\text{(by the definition of }H\text{)}}}\right) \\
&  \ \ \ \ \ \ \ \ \ \ \ \ \ \ \ \ \ \ \ \ \left(
\begin{array}
[c]{c}%
\text{by Lemma \ref{lem.petk.explicit.1} \textbf{(a)}, applied to }m=k-1\text{
and }R=\mathbf{k}\\
\text{and }a_{i,j}=\left(  -1\right)  ^{\left(  \delta_{i}+j\right)
\%k}\left[  \left(  \delta_{i}+j\right)  \%k\in\left\{  0,1\right\}  \right]
\end{array}
\right) \\
&  =\left(  -1\right)  ^{g}\underbrace{\det H}_{=\left(  -1\right)  ^{\left(
\gamma_{1}+\gamma_{2}+\cdots+\gamma_{k-1}\right)  -\left(  1+2+\cdots+\left(
k-1\right)  \right)  }}\\
&  =\left(  -1\right)  ^{g}\left(  -1\right)  ^{\left(  \gamma_{1}+\gamma
_{2}+\cdots+\gamma_{k-1}\right)  -\left(  1+2+\cdots+\left(  k-1\right)
\right)  }\\
&  =\left(  -1\right)  ^{g+\left(  \gamma_{1}+\gamma_{2}+\cdots+\gamma
_{k-1}\right)  -\left(  1+2+\cdots+\left(  k-1\right)  \right)  }.
\end{align*}
This proves Lemma \ref{lem.petk.explicit.2} \textbf{(c)}.
\end{proof}

Next, we recall a well-known property of symmetric functions:

\begin{lemma}
\label{lem.h-e-reciprocity-t}Consider the ring $\Lambda\left[  \left[
t\right]  \right]  $ of formal power series in one indeterminate $t$ over
$\Lambda$. In this ring, we have%
\begin{equation}
1=\left(  \sum_{n\geq0}\left(  -1\right)  ^{n}e_{n}t^{n}\right)  \left(
\sum_{n\geq0}h_{n}t^{n}\right)  . \label{pf.thm.petk.explicit.aker.pf.1}%
\end{equation}

\end{lemma}

Lemma \ref{lem.h-e-reciprocity-t} is a well-known identity (see, e.g.,
\cite[proof of Theorem 7.6.1]{Stanley-EC2} or \cite[(2.4.3)]{GriRei}); for the
sake of completeness, let us nevertheless give a proof:

\begin{proof}
[Proof of Lemma \ref{lem.h-e-reciprocity-t}.]Consider the ring $\left(
\mathbf{k}\left[  \left[  x_{1},x_{2},x_{3},\ldots\right]  \right]  \right)
\left[  \left[  t\right]  \right]  $ of formal power series in one
indeterminate $t$ over $\mathbf{k}\left[  \left[  x_{1},x_{2},x_{3}%
,\ldots\right]  \right]  $. In this ring, we have the equalities%
\begin{equation}
\prod_{i=1}^{\infty}\left(  1-x_{i}t\right)  ^{-1}=\sum_{n\geq0}h_{n}t^{n}
\label{pf.lem.h-e-reciprocity-t.h}%
\end{equation}
and%
\begin{equation}
\prod_{i=1}^{\infty}\left(  1+x_{i}t\right)  =\sum_{n\geq0}e_{n}t^{n}.
\label{pf.lem.h-e-reciprocity-t.e}%
\end{equation}
(Indeed, the first of these two equalities is \cite[(2.2.18)]{GriRei}, whereas
the second is \cite[(2.2.19)]{GriRei}.)

\begin{vershort}
Substituting $-t$ for $t$ in the equality (\ref{pf.lem.h-e-reciprocity-t.e}),
and multiplying the resulting equality by (\ref{pf.lem.h-e-reciprocity-t.h}),
we obtain%
\[
\left(  \prod_{i=1}^{\infty}\left(  1-x_{i}t\right)  \right)  \left(
\prod_{i=1}^{\infty}\left(  1-x_{i}t\right)  ^{-1}\right)  =\left(
\sum_{n\geq0}e_{n}\left(  -t\right)  ^{n}\right)  \left(  \sum_{n\geq0}%
h_{n}t^{n}\right)  .
\]
The left hand side of this equality simplifies to $1$, while the right hand
side is the right hand side of (\ref{pf.thm.petk.explicit.aker.pf.1}). Thus,
(\ref{pf.thm.petk.explicit.aker.pf.1}) holds, which proves Lemma
\ref{lem.h-e-reciprocity-t}.
\end{vershort}

\begin{verlong}
Substituting $-t$ for $t$ in the equality (\ref{pf.lem.h-e-reciprocity-t.e}),
we obtain%
\[
\prod_{i=1}^{\infty}\left(  1-x_{i}t\right)  =\sum_{n\geq0}e_{n}%
\underbrace{\left(  -t\right)  ^{n}}_{=\left(  -1\right)  ^{n}t^{n}}%
=\sum_{n\geq0}\left(  -1\right)  ^{n}e_{n}t^{n}.
\]
Multiplying this equality by (\ref{pf.lem.h-e-reciprocity-t.h}), we find%
\[
\left(  \prod_{i=1}^{\infty}\left(  1-x_{i}t\right)  \right)  \left(
\prod_{i=1}^{\infty}\left(  1-x_{i}t\right)  ^{-1}\right)  =\left(
\sum_{n\geq0}\left(  -1\right)  ^{n}e_{n}t^{n}\right)  \left(  \sum_{n\geq
0}h_{n}t^{n}\right)  .
\]
Comparing this with%
\[
\left(  \prod_{i=1}^{\infty}\left(  1-x_{i}t\right)  \right)  \left(
\prod_{i=1}^{\infty}\left(  1-x_{i}t\right)  ^{-1}\right)  =1,
\]
we obtain%
\[
1=\left(  \sum_{n\geq0}\left(  -1\right)  ^{n}e_{n}t^{n}\right)  \left(
\sum_{n\geq0}h_{n}t^{n}\right)  .
\]
This equality is an equality in $\Lambda\left[  \left[  t\right]  \right]  $
(since both of its sides belong to $\Lambda\left[  \left[  t\right]  \right]
$). This proves Lemma \ref{lem.h-e-reciprocity-t}.
\end{verlong}
\end{proof}

Next, we shall prove yet another evaluation of the homomorphism $\alpha_{k}$:

\begin{lemma}
\label{lem.alphak.e}Let $k$ be a positive integer such that $k>1$. Consider
the $\mathbf{k}$-algebra homomorphism $\alpha_{k}:\Lambda\rightarrow
\mathbf{k}$ from Definition \ref{def.alphak}. Also, recall Convention
\ref{conv.iverson}. Let $r$ be an integer such that $r>-k+1$. Then,%
\begin{equation}
\alpha_{k}\left(  e_{r}\right)  =\left(  -1\right)  ^{r+r\%k}\left[
r\%k\in\left\{  0,1\right\}  \right]  . \label{pf.thm.petk.explicit.aker}%
\end{equation}

\end{lemma}

\begin{proof}
[Proof of Lemma \ref{lem.alphak.e}.]Consider the ring $\Lambda\left[  \left[
t\right]  \right]  $ of formal power series in one indeterminate $t$ over
$\Lambda$. Consider also the analogous ring $\mathbf{k}\left[  \left[
t\right]  \right]  $ over $\mathbf{k}$.

\begin{vershort}
The map $\alpha_{k}:\Lambda\rightarrow\mathbf{k}$ is a $\mathbf{k}$-algebra
homomorphism. Hence, it induces a continuous\footnote{Continuity is defined
with respect to the usual topologies on $\Lambda\left[  \left[  t\right]
\right]  $ and $\mathbf{k}\left[  \left[  t\right]  \right]  $, where we equip
both $\Lambda$ and $\mathbf{k}$ with the discrete topologies.} $\mathbf{k}%
\left[  \left[  t\right]  \right]  $-algebra homomorphism%
\[
\alpha_{k}\left[  \left[  t\right]  \right]  :\Lambda\left[  \left[  t\right]
\right]  \rightarrow\mathbf{k}\left[  \left[  t\right]  \right]
\]
that sends each formal power series $\sum_{n\geq0}a_{n}t^{n}\in\Lambda\left[
\left[  t\right]  \right]  $ (with $a_{n}\in\Lambda$) to $\sum_{n\geq0}%
\alpha_{k}\left(  a_{n}\right)  t^{n}$. Consider this $\mathbf{k}\left[
\left[  t\right]  \right]  $-algebra homomorphism $\alpha_{k}\left[  \left[
t\right]  \right]  $.
\end{vershort}

\begin{verlong}
The map $\alpha_{k}:\Lambda\rightarrow\mathbf{k}$ is a $\mathbf{k}$-algebra
homomorphism. Hence, it induces a continuous\footnote{Continuity is defined
with respect to the usual topologies on $\Lambda\left[  \left[  t\right]
\right]  $ and $\mathbf{k}\left[  \left[  t\right]  \right]  $, where we equip
both $\Lambda$ and $\mathbf{k}$ with the discrete topologies.} $\mathbf{k}%
\left[  \left[  t\right]  \right]  $-algebra homomorphism%
\[
\alpha_{k}\left[  \left[  t\right]  \right]  :\Lambda\left[  \left[  t\right]
\right]  \rightarrow\mathbf{k}\left[  \left[  t\right]  \right]
\]
that sends each formal power series $\sum_{n\geq0}a_{n}t^{n}\in\Lambda\left[
\left[  t\right]  \right]  $ (with $a_{n}\in\Lambda$) to $\sum_{n\geq0}%
\alpha_{k}\left(  a_{n}\right)  t^{n}$. Consider this $\mathbf{k}\left[
\left[  t\right]  \right]  $-algebra homomorphism $\alpha_{k}\left[  \left[
t\right]  \right]  $. In particular, it satisfies
\[
\left(  \alpha_{k}\left[  \left[  t\right]  \right]  \right)  \left(
t^{i}\right)  =t^{i}\ \ \ \ \ \ \ \ \ \ \text{for each }i\in\mathbb{N}.
\]

\end{verlong}

Applying the map $\alpha_{k}\left[  \left[  t\right]  \right]  $ to both sides
of the equality (\ref{pf.thm.petk.explicit.aker.pf.1}), we obtain%
\begin{align*}
\left(  \alpha_{k}\left[  \left[  t\right]  \right]  \right)  \left(
1\right)   &  =\left(  \alpha_{k}\left[  \left[  t\right]  \right]  \right)
\left(  \left(  \sum_{n\geq0}\left(  -1\right)  ^{n}e_{n}t^{n}\right)  \left(
\sum_{n\geq0}h_{n}t^{n}\right)  \right) \\
&  =\underbrace{\left(  \alpha_{k}\left[  \left[  t\right]  \right]  \right)
\left(  \sum_{n\geq0}\left(  -1\right)  ^{n}e_{n}t^{n}\right)  }%
_{\substack{=\sum_{n\geq0}\alpha_{k}\left(  \left(  -1\right)  ^{n}%
e_{n}\right)  t^{n}\\\text{(by the definition of }\alpha_{k}\left[  \left[
t\right]  \right]  \text{)}}}\cdot\underbrace{\left(  \alpha_{k}\left[
\left[  t\right]  \right]  \right)  \left(  \sum_{n\geq0}h_{n}t^{n}\right)
}_{\substack{=\sum_{n\geq0}\alpha_{k}\left(  h_{n}\right)  t^{n}\\\text{(by
the definition of }\alpha_{k}\left[  \left[  t\right]  \right]  \text{)}}}\\
&  \ \ \ \ \ \ \ \ \ \ \ \ \ \ \ \ \ \ \ \ \left(  \text{since }\alpha
_{k}\left[  \left[  t\right]  \right]  \text{ is a }\mathbf{k}\left[  \left[
t\right]  \right]  \text{-algebra homomorphism}\right) \\
&  =\left(  \sum_{n\geq0}\underbrace{\alpha_{k}\left(  \left(  -1\right)
^{n}e_{n}\right)  }_{\substack{=\left(  -1\right)  ^{n}\alpha_{k}\left(
e_{n}\right)  \\\text{(since }\alpha_{k}\text{ is }\mathbf{k}\text{-linear)}%
}}t^{n}\right)  \cdot\left(  \sum_{n\geq0}\underbrace{\alpha_{k}\left(
h_{n}\right)  }_{\substack{=\left[  n<k\right]  \\\text{(by
(\ref{pf.thm.G.main.alk}))}}}t^{n}\right) \\
&  =\left(  \sum_{n\geq0}\left(  -1\right)  ^{n}\alpha_{k}\left(
e_{n}\right)  t^{n}\right)  \cdot\underbrace{\left(  \sum_{n\geq0}\left[
n<k\right]  t^{n}\right)  }_{=t^{0}+t^{1}+\cdots+t^{k-1}=\dfrac{1-t^{k}}{1-t}%
}\\
&  =\left(  \sum_{n\geq0}\left(  -1\right)  ^{n}\alpha_{k}\left(
e_{n}\right)  t^{n}\right)  \cdot\dfrac{1-t^{k}}{1-t}.
\end{align*}
Comparing this with%
\[
\left(  \alpha_{k}\left[  \left[  t\right]  \right]  \right)  \left(
1\right)  =1\ \ \ \ \ \ \ \ \ \ \left(  \text{since }\alpha_{k}\left[  \left[
t\right]  \right]  \text{ is a }\mathbf{k}\left[  \left[  t\right]  \right]
\text{-algebra homomorphism}\right)  ,
\]
we obtain%
\[
\left(  \sum_{n\geq0}\left(  -1\right)  ^{n}\alpha_{k}\left(  e_{n}\right)
t^{n}\right)  \cdot\dfrac{1-t^{k}}{1-t}=1.
\]
Hence,%
\begin{align*}
\sum_{n\geq0}\left(  -1\right)  ^{n}\alpha_{k}\left(  e_{n}\right)  t^{n}  &
=\dfrac{1-t}{1-t^{k}}=\left(  1-t\right)  \cdot\underbrace{\dfrac{1}{1-t^{k}}%
}_{=1+t^{k}+t^{2k}+t^{3k}+\cdots}\\
&  =\left(  1-t\right)  \cdot\left(  1+t^{k}+t^{2k}+t^{3k}+\cdots\right) \\
&  =1-t+t^{k}-t^{k+1}+t^{2k}-t^{2k+1}+t^{3k}-t^{3k+1}\pm\cdots\\
&  =\sum_{n\geq0}\left(  -1\right)  ^{n\%k}\left[  n\%k\in\left\{
0,1\right\}  \right]  t^{n}%
\end{align*}
(here, we have used that $k>1$, since for $k=1$ there would be cancellations
in the sum $1-t+t^{k}-t^{k+1}+t^{2k}-t^{2k+1}+t^{3k}-t^{3k+1}\pm\cdots$).
Comparing coefficients before $t^{m}$ on both sides of this equality, we
obtain%
\begin{equation}
\left(  -1\right)  ^{m}\alpha_{k}\left(  e_{m}\right)  =\left(  -1\right)
^{m\%k}\left[  m\%k\in\left\{  0,1\right\}  \right]
\label{pf.thm.petk.explicit.aker.pf.4}%
\end{equation}
for each $m\in\mathbb{N}$.

\begin{vershort}
Multiplying both sides of this equality by $\left(  -1\right)  ^{m}$, we
obtain%
\begin{equation}
\alpha_{k}\left(  e_{m}\right)  =\left(  -1\right)  ^{m+m\%k}\left[
m\%k\in\left\{  0,1\right\}  \right]  .
\label{pf.thm.petk.explicit.aker.pf.short.4b}%
\end{equation}

\end{vershort}

\begin{verlong}
Now, each $m\in\mathbb{N}$ satisfies $\underbrace{\left(  -1\right)
^{m}\left(  -1\right)  ^{m}}_{=\left(  -1\right)  ^{2m}=1}\alpha_{k}\left(
e_{m}\right)  =\alpha_{k}\left(  e_{m}\right)  $ and thus%
\begin{align}
\alpha_{k}\left(  e_{m}\right)   &  =\left(  -1\right)  ^{m}%
\underbrace{\left(  -1\right)  ^{m}\alpha_{k}\left(  e_{m}\right)
}_{\substack{=\left(  -1\right)  ^{m\%k}\left[  m\%k\in\left\{  0,1\right\}
\right]  \\\text{(by (\ref{pf.thm.petk.explicit.aker.pf.4}))}}%
}=\underbrace{\left(  -1\right)  ^{m}\left(  -1\right)  ^{m\%k}}_{=\left(
-1\right)  ^{m+m\%k}}\left[  m\%k\in\left\{  0,1\right\}  \right] \nonumber\\
&  =\left(  -1\right)  ^{m+m\%k}\left[  m\%k\in\left\{  0,1\right\}  \right]
. \label{pf.thm.petk.explicit.aker.pf.4b}%
\end{align}

\end{verlong}

\begin{vershort}
We must prove that
\[
\alpha_{k}\left(  e_{r}\right)  =\left(  -1\right)  ^{r+r\%k}\left[
r\%k\in\left\{  0,1\right\}  \right]  .
\]
If $r\in\mathbb{N}$, then this follows by applying
(\ref{pf.thm.petk.explicit.aker.pf.short.4b}) to $m=r$. Hence, for the rest of
this proof, we WLOG assume that $r\notin\mathbb{N}$. Thus, $r$ is negative, so
that $r\in\left\{  -k+2,-k+3,\ldots,-1\right\}  $ (since $r>-k+1$). Hence,
$r\%k\in\left\{  2,3,\ldots,k-1\right\}  $, so that $r\%k\notin\left\{
0,1\right\}  $. Consequently, $\left[  r\%k\in\left\{  0,1\right\}  \right]
=0$. Also, $e_{r}=0$ (since $r$ is negative) and thus $\alpha_{k}\left(
e_{r}\right)  =\alpha_{k}\left(  0\right)  =0$. Comparing this with $\left(
-1\right)  ^{r+r\%k}\underbrace{\left[  r\%k\in\left\{  0,1\right\}  \right]
}_{=0}=0$, we obtain $\alpha_{k}\left(  e_{r}\right)  =\left(  -1\right)
^{r+r\%k}\left[  r\%k\in\left\{  0,1\right\}  \right]  $. This concludes the
proof of Lemma \ref{lem.alphak.e}.
\end{vershort}

\begin{verlong}
We must prove that
\[
\alpha_{k}\left(  e_{r}\right)  =\left(  -1\right)  ^{r+r\%k}\left[
r\%k\in\left\{  0,1\right\}  \right]  .
\]
If $r\in\mathbb{N}$, then this follows by applying
(\ref{pf.thm.petk.explicit.aker.pf.4b}) to $m=r$. Hence, for the rest of this
proof, we WLOG assume that $r\notin\mathbb{N}$. Thus, $r$ is negative (since
$r$ is an integer). In view of $r>-k+1$, this yields $r\in\left\{
-k+2,-k+3,\ldots,-1\right\}  $. Hence, $r\%k\in\left\{  2,3,\ldots
,k-1\right\}  $. Thus, $r\%k\notin\left\{  0,1\right\}  $. Hence, $\left[
r\%k\in\left\{  0,1\right\}  \right]  =0$. Also, $e_{r}=0$ (since $r$ is
negative) and thus $\alpha_{k}\left(  e_{r}\right)  =\alpha_{k}\left(
0\right)  =0$ (since the map $\alpha_{k}$ is $\mathbf{k}$-linear). Comparing
this with $\left(  -1\right)  ^{r+r\%k}\underbrace{\left[  r\%k\in\left\{
0,1\right\}  \right]  }_{=0}=0$, we obtain $\alpha_{k}\left(  e_{r}\right)
=\left(  -1\right)  ^{r+r\%k}\left[  r\%k\in\left\{  0,1\right\}  \right]  $.
This concludes the proof of Lemma \ref{lem.alphak.e}.
\end{verlong}
\end{proof}

\begin{proof}
[Proof of Theorem \ref{thm.petk.explicit}.]\textbf{(a)} Assume that $\mu
_{k}\neq0$. But $\mu=\lambda^{t}$, whence%
\[
\mu_{k}=\left(  \lambda^{t}\right)  _{k}=\left\vert \left\{  j\in\left\{
1,2,3,\ldots\right\}  \ \mid\ \lambda_{j}\geq k\right\}  \right\vert
\ \ \ \ \ \ \ \ \ \ \left(  \text{by the definition of }\lambda^{t}\right)  .
\]
Hence,
\[
\left\vert \left\{  j\in\left\{  1,2,3,\ldots\right\}  \ \mid\ \lambda_{j}\geq
k\right\}  \right\vert =\mu_{k}\neq0.
\]
In other words, the set $\left\{  j\in\left\{  1,2,3,\ldots\right\}
\ \mid\ \lambda_{j}\geq k\right\}  $ is nonempty. Hence, there exists some
$j\in\left\{  1,2,3,\ldots\right\}  $ satisfying $\lambda_{j}\geq k$. Consider
this $j$. We have $\lambda_{1}\geq\lambda_{2}\geq\lambda_{3}\geq\cdots$ (since
$\lambda\in\operatorname*{Par}$) and thus $\lambda_{1}\geq\lambda_{j}$ (since
$1\leq j$). Hence, $\lambda_{1}\geq\lambda_{j}\geq k$. Thus, Proposition
\ref{prop.petk.0} yields $\operatorname*{pet}\nolimits_{k}\left(
\lambda,\varnothing\right)  =0$. This proves Theorem \ref{thm.petk.explicit}
\textbf{(a)}.

Now, let us prepare for the proof of parts \textbf{(b)} and \textbf{(c)}.

Consider the $\mathbf{k}$-algebra homomorphism $\alpha_{k}:\Lambda
\rightarrow\mathbf{k}$ from Definition \ref{def.alphak}.

\begin{vershort}
For each $i\in\left\{  1,2,\ldots,k-1\right\}  $, we have $\left(  \beta
_{i}-1\right)  \%k\in\left\{  0,1,\ldots,k-1\right\}  $ (by the definition of
a remainder) and thus $\gamma_{i}\in\left\{  1,2,\ldots,k\right\}  $ (by
(\ref{eq.thm.petk.explicit.gammi=})). In other words, $\gamma_{1},\gamma
_{2},\ldots,\gamma_{k-1}$ are $k-1$ elements of the set $\left\{
1,2,\ldots,k\right\}  $.
\end{vershort}

\begin{verlong}
For each $i\in\left\{  1,2,\ldots,k-1\right\}  $, we have%
\begin{align*}
\gamma_{i}  &  =1+\left(  \beta_{i}-1\right)  \%k\ \ \ \ \ \ \ \ \ \ \left(
\text{by (\ref{eq.thm.petk.explicit.gammi=})}\right) \\
&  \in\left\{  1,2,\ldots,k\right\}  \ \ \ \ \ \ \ \ \ \ \left(
\begin{array}
[c]{c}%
\text{since }\left(  \beta_{i}-1\right)  \%k\in\left\{  0,1,\ldots,k-1\right\}
\\
\text{(by the definition of a remainder)}%
\end{array}
\right)  .
\end{align*}
Hence, the $\left(  k-1\right)  $-tuple $\left(  \gamma_{1},\gamma_{2}%
,\ldots,\gamma_{k-1}\right)  $ really belongs to $\left\{  1,2,\ldots
,k\right\}  ^{k-1}$. In other words, $\gamma_{1},\gamma_{2},\ldots
,\gamma_{k-1}$ are $k-1$ elements of the set $\left\{  1,2,\ldots,k\right\}  $.
\end{verlong}

\begin{vershort}
Assume that $\mu_{k}=0$. Thus, $\mu=\left(  \mu_{1},\mu_{2},\ldots,\mu
_{k-1}\right)  $ (since $\mu\in\operatorname*{Par}$).
\end{vershort}

\begin{verlong}
Assume that $\mu_{k}=0$. But $\mu\in\operatorname*{Par}$ and thus $\mu_{1}%
\geq\mu_{2}\geq\mu_{3}\geq\cdots$. Hence, from $\mu_{k}=0$, we obtain $\mu
_{k}=\mu_{k+1}=\mu_{k+2}=\cdots=0$. Thus, $\mu=\left(  \mu_{1},\mu_{2}%
,\ldots,\mu_{k-1}\right)  $.
\end{verlong}

It is known that taking the transpose of the transpose of a partition returns
the original partition. Thus, $\left(  \lambda^{t}\right)  ^{t}=\lambda$. In
view of $\mu=\lambda^{t}$, this rewrites as $\mu^{t}=\lambda$. Hence,
$\lambda=\mu^{t}$. Therefore,%
\[
s_{\lambda}=s_{\mu^{t}}=\det\left(  \left(  e_{\mu_{i}-i+j}\right)  _{1\leq
i\leq k-1,\ 1\leq j\leq k-1}\right)
\]
(by (\ref{eq.schur.JTsh-e}), applied to $\mu$ and $k-1$ instead of $\lambda$
and $\ell$), because $\mu=\left(  \mu_{1},\mu_{2},\ldots,\mu_{k-1}\right)  $.
Applying the map $\alpha_{k}$ to both sides of this equality, we find%
\begin{align*}
\alpha_{k}\left(  s_{\lambda}\right)   &  =\alpha_{k}\left(  \det\left(
\left(  \underbrace{e_{\mu_{i}-i+j}}_{\substack{=e_{\beta_{i}+j}\\\text{(since
(\ref{eq.thm.petk.explicit.beti=}) yields }\mu_{i}-i=\beta_{i}\text{)}%
}}\right)  _{1\leq i\leq k-1,\ 1\leq j\leq k-1}\right)  \right) \\
&  =\alpha_{k}\left(  \det\left(  \left(  e_{\beta_{i}+j}\right)  _{1\leq
i\leq k-1,\ 1\leq j\leq k-1}\right)  \right) \\
&  =\det\left(  \left(  \alpha_{k}\left(  e_{\beta_{i}+j}\right)  \right)
_{1\leq i\leq k-1,\ 1\leq j\leq k-1}\right)
\end{align*}
(since $\alpha_{k}$ is a $\mathbf{k}$-algebra homomorphism, and thus commutes
with taking determinants of matrices). On the other hand,%
\[
\alpha_{k}\left(  \underbrace{s_{\lambda}}_{=s_{\lambda/\varnothing}}\right)
=\alpha_{k}\left(  s_{\lambda/\varnothing}\right)  =\operatorname*{pet}%
\nolimits_{k}\left(  \lambda,\varnothing\right)
\]
(by (\ref{pf.thm.G.pieri.alks}), applied to $\varnothing$ instead of $\mu$).
Comparing these two equalities, we obtain%
\begin{equation}
\operatorname*{pet}\nolimits_{k}\left(  \lambda,\varnothing\right)
=\det\left(  \left(  \alpha_{k}\left(  e_{\beta_{i}+j}\right)  \right)
_{1\leq i\leq k-1,\ 1\leq j\leq k-1}\right)  .
\label{pf.thm.petk.explicit.as-det-1}%
\end{equation}
But each $i\in\left\{  1,2,\ldots,k-1\right\}  $ and $j\in\left\{
1,2,\ldots,k-1\right\}  $ satisfy $k>1$\ \ \ \ \footnote{Indeed, if
$i\in\left\{  1,2,\ldots,k-1\right\}  $, then $1\leq i\leq k-1$ and thus
$k-1\geq1>0$, so that $k>1$.} and%
\[
\underbrace{\beta_{i}}_{\substack{=\mu_{i}-i\\\text{(by
(\ref{eq.thm.petk.explicit.beti=}))}}}+j=\underbrace{\mu_{i}}_{\geq
0}-\underbrace{i}_{\leq k-1}+\underbrace{j}_{>0}>0-\left(  k-1\right)
+0=-k+1
\]
and thus%
\begin{equation}
\alpha_{k}\left(  e_{\beta_{i}+j}\right)  =\left(  -1\right)  ^{\left(
\beta_{i}+j\right)  +\left(  \beta_{i}+j\right)  \%k}\left[  \left(  \beta
_{i}+j\right)  \%k\in\left\{  0,1\right\}  \right]
\label{pf.thm.petk.explicit.akem1a}%
\end{equation}
(by (\ref{pf.thm.petk.explicit.aker}), applied to $r=\beta_{i}+j$).

Furthermore, each $i\in\left\{  1,2,\ldots,k-1\right\}  $ and $j\in\left\{
1,2,\ldots,k-1\right\}  $ satisfy%
\begin{align}
&  \left(  -1\right)  ^{\left(  \beta_{i}+j\right)  +\left(  \beta
_{i}+j\right)  \%k}\left[  \left(  \beta_{i}+j\right)  \%k\in\left\{
0,1\right\}  \right] \nonumber\\
&  =\left(  -1\right)  ^{\beta_{i}}\left(  -1\right)  ^{j}\left(  -1\right)
^{\left(  \gamma_{i}+j\right)  \%k}\left[  \left(  \gamma_{i}+j\right)
\%k\in\left\{  0,1\right\}  \right]  . \label{pf.thm.petk.explicit.akem2}%
\end{align}

\begin{vershort}
[\textit{Proof of (\ref{pf.thm.petk.explicit.akem2}):} Let $i\in\left\{
1,2,\ldots,k-1\right\}  $ and $j\in\left\{  1,2,\ldots,k-1\right\}  $. The
definition of $\gamma_{i}$ yields%
\[
\gamma_{i}=1+\underbrace{\left(  \beta_{i}-1\right)  \%k}_{\substack{\equiv
\beta_{i}-1\operatorname{mod}k\\\text{(since }u\%k\equiv u\operatorname{mod}%
k\text{ for any }u\in\mathbb{Z}\text{)}}}\equiv1+\left(  \beta_{i}-1\right)
=\beta_{i}\operatorname{mod}k.
\]
Hence, $\underbrace{\gamma_{i}}_{\equiv\beta_{i}\operatorname{mod}k}%
+j\equiv\beta_{i}+j\operatorname{mod}k$. But if $u\in\mathbb{Z}$ and
$v\in\mathbb{Z}$ satisfy $u\equiv v\operatorname{mod}k$, then $u\%k=v\%k$.
Applying this to $u=\gamma_{i}+j$ and $v=\beta_{i}+j$, we obtain $\left(
\gamma_{i}+j\right)  \%k=\left(  \beta_{i}+j\right)  \%k$. Hence,%
\begin{align*}
&  \left(  -1\right)  ^{\beta_{i}}\left(  -1\right)  ^{j}\left(  -1\right)
^{\left(  \gamma_{i}+j\right)  \%k}\left[  \left(  \gamma_{i}+j\right)
\%k\in\left\{  0,1\right\}  \right] \\
&  =\underbrace{\left(  -1\right)  ^{\beta_{i}}\left(  -1\right)  ^{j}\left(
-1\right)  ^{\left(  \beta_{i}+j\right)  \%k}}_{=\left(  -1\right)  ^{\left(
\beta_{i}+j\right)  +\left(  \beta_{i}+j\right)  \%k}}\left[  \left(
\beta_{i}+j\right)  \%k\in\left\{  0,1\right\}  \right] \\
&  =\left(  -1\right)  ^{\left(  \beta_{i}+j\right)  +\left(  \beta
_{i}+j\right)  \%k}\left[  \left(  \beta_{i}+j\right)  \%k\in\left\{
0,1\right\}  \right]  .
\end{align*}
This proves (\ref{pf.thm.petk.explicit.akem2}).]
\end{vershort}

\begin{verlong}
[\textit{Proof of (\ref{pf.thm.petk.explicit.akem2}):} Let $i\in\left\{
1,2,\ldots,k-1\right\}  $ and $j\in\left\{  1,2,\ldots,k-1\right\}  $. The
definition of $\gamma_{i}$ yields%
\[
\gamma_{i}=1+\underbrace{\left(  \beta_{i}-1\right)  \%k}_{\substack{\equiv
\beta_{i}-1\operatorname{mod}k\\\text{(since }u\%k\equiv u\operatorname{mod}%
k\text{ for any }u\in\mathbb{Z}\text{)}}}\equiv1+\left(  \beta_{i}-1\right)
=\beta_{i}\operatorname{mod}k.
\]
Hence, $\underbrace{\gamma_{i}}_{\equiv\beta_{i}\operatorname{mod}k}%
+j\equiv\beta_{i}+j\operatorname{mod}k$. But if two integers are congruent
modulo $k$, then they must leave the same remainder upon division by $k$. In
other words, if $u\in\mathbb{Z}$ and $v\in\mathbb{Z}$ satisfy $u\equiv
v\operatorname{mod}k$, then $u\%k=v\%k$. Applying this to $u=\gamma_{i}+j$ and
$v=\beta_{i}+j$, we obtain $\left(  \gamma_{i}+j\right)  \%k=\left(  \beta
_{i}+j\right)  \%k$. Hence,%
\begin{align*}
&  \left(  -1\right)  ^{\beta_{i}}\left(  -1\right)  ^{j}\left(  -1\right)
^{\left(  \gamma_{i}+j\right)  \%k}\left[  \left(  \gamma_{i}+j\right)
\%k\in\left\{  0,1\right\}  \right] \\
&  =\underbrace{\left(  -1\right)  ^{\beta_{i}}\left(  -1\right)  ^{j}\left(
-1\right)  ^{\left(  \beta_{i}+j\right)  \%k}}_{=\left(  -1\right)  ^{\left(
\beta_{i}+j\right)  +\left(  \beta_{i}+j\right)  \%k}}\left[  \left(
\beta_{i}+j\right)  \%k\in\left\{  0,1\right\}  \right] \\
&  =\left(  -1\right)  ^{\left(  \beta_{i}+j\right)  +\left(  \beta
_{i}+j\right)  \%k}\left[  \left(  \beta_{i}+j\right)  \%k\in\left\{
0,1\right\}  \right]  .
\end{align*}
This proves (\ref{pf.thm.petk.explicit.akem2}).]
\end{verlong}

Now, (\ref{pf.thm.petk.explicit.as-det-1}) becomes%
\begin{align*}
&  \operatorname*{pet}\nolimits_{k}\left(  \lambda,\varnothing\right) \\
&  =\det\left(  \left(  \underbrace{\alpha_{k}\left(  e_{\beta_{i}+j}\right)
}_{\substack{=\left(  -1\right)  ^{\left(  \beta_{i}+j\right)  +\left(
\beta_{i}+j\right)  \%k}\left[  \left(  \beta_{i}+j\right)  \%k\in\left\{
0,1\right\}  \right]  \\\text{(by (\ref{pf.thm.petk.explicit.akem1a}))}%
}}\right)  _{1\leq i\leq k-1,\ 1\leq j\leq k-1}\right) \\
&  =\det\left(  \left(  \underbrace{\left(  -1\right)  ^{\left(  \beta
_{i}+j\right)  +\left(  \beta_{i}+j\right)  \%k}\left[  \left(  \beta
_{i}+j\right)  \%k\in\left\{  0,1\right\}  \right]  }_{\substack{=\left(
-1\right)  ^{\beta_{i}}\left(  -1\right)  ^{j}\left(  -1\right)  ^{\left(
\gamma_{i}+j\right)  \%k}\left[  \left(  \gamma_{i}+j\right)  \%k\in\left\{
0,1\right\}  \right]  \\\text{(by (\ref{pf.thm.petk.explicit.akem2}))}%
}}\right)  _{1\leq i\leq k-1,\ 1\leq j\leq k-1}\right) \\
&  =\det\left(  \left(  \left(  -1\right)  ^{\beta_{i}}\left(  -1\right)
^{j}\left(  -1\right)  ^{\left(  \gamma_{i}+j\right)  \%k}\left[  \left(
\gamma_{i}+j\right)  \%k\in\left\{  0,1\right\}  \right]  \right)  _{1\leq
i\leq k-1,\ 1\leq j\leq k-1}\right) \\
&  =\left(  \prod_{i=1}^{k-1}\left(  \left(  -1\right)  ^{\beta_{i}}\left(
-1\right)  ^{i}\right)  \right)  \cdot\det\left(  \left(  \left(  -1\right)
^{\left(  \gamma_{i}+j\right)  \%k}\left[  \left(  \gamma_{i}+j\right)
\%k\in\left\{  0,1\right\}  \right]  \right)  _{1\leq i\leq k-1,\ 1\leq j\leq
k-1}\right)
\end{align*}
(by Lemma \ref{lem.petk.explicit.1} \textbf{(b)}, applied to $m=k-1$ and
$R=\mathbf{k}$ and \newline$a_{i,j}=\left(  -1\right)  ^{\left(  \gamma
_{i}+j\right)  \%k}\left[  \left(  \gamma_{i}+j\right)  \%k\in\left\{
0,1\right\}  \right]  $ and $u_{i}=\left(  -1\right)  ^{\beta_{i}}$ and
$v_{j}=\left(  -1\right)  ^{j}$).

Define a $\left(  k-1\right)  \times\left(  k-1\right)  $-matrix $G$ as in
Lemma \ref{lem.petk.explicit.2}. Then, this becomes%
\begin{align}
&  \operatorname*{pet}\nolimits_{k}\left(  \lambda,\varnothing\right)
\nonumber\\
&  =\left(  \prod_{i=1}^{k-1}\left(  \left(  -1\right)  ^{\beta_{i}}\left(
-1\right)  ^{i}\right)  \right)  \cdot\det\left(  \underbrace{\left(  \left(
-1\right)  ^{\left(  \gamma_{i}+j\right)  \%k}\left[  \left(  \gamma
_{i}+j\right)  \%k\in\left\{  0,1\right\}  \right]  \right)  _{1\leq i\leq
k-1,\ 1\leq j\leq k-1}}_{\substack{=G\\\text{(by the definition of }G\text{)}%
}}\right) \nonumber\\
&  =\left(  \prod_{i=1}^{k-1}\left(  \left(  -1\right)  ^{\beta_{i}}\left(
-1\right)  ^{i}\right)  \right)  \cdot\det G.
\label{pf.thm.petk.explicit.petk-via-G}%
\end{align}

Now, we can readily prove parts \textbf{(b)} and \textbf{(c)} of Theorem
\ref{thm.petk.explicit}:

\textbf{(b)} Assume that the $k-1$ numbers $\gamma_{1},\gamma_{2}%
,\ldots,\gamma_{k-1}$ are not distinct. Then, Lemma \ref{lem.petk.explicit.2}
\textbf{(a)} yields $\det G=0$. Hence, (\ref{pf.thm.petk.explicit.petk-via-G})
yields%
\[
\operatorname*{pet}\nolimits_{k}\left(  \lambda,\varnothing\right)  =\left(
\prod_{i=1}^{k-1}\left(  \left(  -1\right)  ^{\beta_{i}}\left(  -1\right)
^{i}\right)  \right)  \cdot\underbrace{\det G}_{=0}=0.
\]
This proves Theorem \ref{thm.petk.explicit} \textbf{(b)}.

\textbf{(c)} The equality (\ref{pf.thm.petk.explicit.petk-via-G}) becomes%
\begin{align*}
\operatorname*{pet}\nolimits_{k}\left(  \lambda,\varnothing\right)   &
=\underbrace{\left(  \prod_{i=1}^{k-1}\left(  \left(  -1\right)  ^{\beta_{i}%
}\left(  -1\right)  ^{i}\right)  \right)  }_{=\left(  \prod_{i=1}^{k-1}\left(
-1\right)  ^{\beta_{i}}\right)  \left(  \prod_{i=1}^{k-1}\left(  -1\right)
^{i}\right)  }\cdot\underbrace{\det G}_{\substack{=\left(  -1\right)
^{g+\left(  \gamma_{1}+\gamma_{2}+\cdots+\gamma_{k-1}\right)  -\left(
1+2+\cdots+\left(  k-1\right)  \right)  }\\\text{(by Lemma
\ref{lem.petk.explicit.2} \textbf{(c)})}}}\\
&  =\underbrace{\left(  \prod_{i=1}^{k-1}\left(  -1\right)  ^{\beta_{i}%
}\right)  }_{=\left(  -1\right)  ^{\beta_{1}+\beta_{2}+\cdots+\beta_{k-1}}%
}\underbrace{\left(  \prod_{i=1}^{k-1}\left(  -1\right)  ^{i}\right)
}_{=\left(  -1\right)  ^{1+2+\cdots+\left(  k-1\right)  }}\cdot\left(
-1\right)  ^{g+\left(  \gamma_{1}+\gamma_{2}+\cdots+\gamma_{k-1}\right)
-\left(  1+2+\cdots+\left(  k-1\right)  \right)  }\\
&  =\left(  -1\right)  ^{\beta_{1}+\beta_{2}+\cdots+\beta_{k-1}}\left(
-1\right)  ^{1+2+\cdots+\left(  k-1\right)  }\cdot\left(  -1\right)
^{g+\left(  \gamma_{1}+\gamma_{2}+\cdots+\gamma_{k-1}\right)  -\left(
1+2+\cdots+\left(  k-1\right)  \right)  }\\
&  =\left(  -1\right)  ^{\left(  \beta_{1}+\beta_{2}+\cdots+\beta
_{k-1}\right)  +\left(  1+2+\cdots+\left(  k-1\right)  \right)  +g+\left(
\gamma_{1}+\gamma_{2}+\cdots+\gamma_{k-1}\right)  -\left(  1+2+\cdots+\left(
k-1\right)  \right)  }\\
&  =\left(  -1\right)  ^{\left(  \beta_{1}+\beta_{2}+\cdots+\beta
_{k-1}\right)  +g+\left(  \gamma_{1}+\gamma_{2}+\cdots+\gamma_{k-1}\right)  }.
\end{align*}
This proves Theorem \ref{thm.petk.explicit} \textbf{(c)}.
\end{proof}

The proof of Proposition \ref{prop.petk.explicit-old} relies on the following
known fact:

\begin{proposition}
\label{prop.partition.transpose.djun}Let $\lambda\in\operatorname*{Par}$. Let
$\mu=\lambda^{t}$. Then:

\textbf{(a)} If $i$ and $j$ are two positive integers satisfying $\lambda
_{i}\geq j$, then $\mu_{j}\geq i$.

\textbf{(b)} If $i$ and $j$ are two positive integers satisfying $\lambda
_{i}<j$, then $\mu_{j}<i$.

\textbf{(c)} Any two positive integers $i$ and $j$ satisfy $\lambda_{i}%
+\mu_{j}-i-j\neq-1$.

For each positive integer $i$, set $\alpha_{i}=\lambda_{i}-i$. For each
positive integer $j$, set $\beta_{j}=\mu_{j}-j$ and $\eta_{j}=-1-\beta_{j}$. Then:

\textbf{(d)} The two sets $\left\{  \alpha_{1},\alpha_{2},\alpha_{3}%
,\ldots\right\}  $ and $\left\{  \eta_{1},\eta_{2},\eta_{3},\ldots\right\}  $
are disjoint, and their union is $\mathbb{Z}$.

\textbf{(e)} Let $p$ be an integer such that $p\geq\lambda_{1}$. Then, the two
sets $\left\{  \alpha_{1},\alpha_{2},\alpha_{3},\ldots\right\}  $ and
$\left\{  \eta_{1},\eta_{2},\ldots,\eta_{p}\right\}  $ are disjoint, and their
union is
\[
\left\{  \ldots,p-3,p-2,p-1\right\}  =\left\{  k\in\mathbb{Z}\ \mid
\ k<p\right\}  .
\]

\textbf{(f)} Let $p$ and $q$ be two integers such that $p\geq\lambda_{1}$ and
$q\geq\mu_{1}$. Then, the two sets $\left\{  \alpha_{1},\alpha_{2}%
,\ldots,\alpha_{q}\right\}  $ and $\left\{  \eta_{1},\eta_{2},\ldots,\eta
_{p}\right\}  $ are disjoint, and their union is
\[
\left\{  -q,-q+1,\ldots,p-1\right\}  =\left\{  k\in\mathbb{Z}\ \mid\ -q\leq
k<p\right\}  .
\]

\end{proposition}

Note that Proposition \ref{prop.partition.transpose.djun} \textbf{(f)} is a
restatement of \cite[Chapter I, (1.7)]{Macdon95}.

\begin{vershort}
\begin{proof}
[Proof of Proposition \ref{prop.partition.transpose.djun} (sketched).]Left to
the reader (see \cite{verlong} for a detailed proof). The easiest way to
proceed is by proving \textbf{(a)} and \textbf{(b)} first, then deriving
\textbf{(c)} as their consequence, then deriving \textbf{(f)} from it, then
concluding \textbf{(d)} and \textbf{(e)}.
\end{proof}
\end{vershort}

\begin{verlong}
\begin{proof}
[Proof of Proposition \ref{prop.partition.transpose.djun}.]We have
$\mu=\lambda^{t}$. Thus, each positive integer $i$ satisfies%
\begin{align}
\mu_{i}  &  =\left(  \lambda^{t}\right)  _{i}=\left\vert \left\{  j\in\left\{
1,2,3,\ldots\right\}  \ \mid\ \lambda_{j}\geq i\right\}  \right\vert
\ \ \ \ \ \ \ \ \ \ \left(  \text{by Definition \ref{def.transpose}}\right)
\nonumber\\
&  =\left\vert \left\{  k\in\left\{  1,2,3,\ldots\right\}  \ \mid\ \lambda
_{k}\geq i\right\}  \right\vert \label{pf.prop.partition.transpose.djun.mui=}%
\end{align}
(here, we have renamed the index $j$ as $k$).

\textbf{(a)} Let $i$ and $j$ be two positive integers satisfying $\lambda
_{i}\geq j$. We must prove that $\mu_{j}\geq i$.

Indeed, (\ref{pf.prop.partition.transpose.djun.mui=}) (applied to $i=j$)
yields%
\begin{equation}
\mu_{j}=\left\vert \left\{  k\in\left\{  1,2,3,\ldots\right\}  \ \mid
\ \lambda_{k}\geq j\right\}  \right\vert .
\label{pf.prop.partition.transpose.djun.a.muj=}%
\end{equation}

Now, we have $\left\{  1,2,\ldots,i\right\}  \subseteq\left\{  k\in\left\{
1,2,3,\ldots\right\}  \ \mid\ \lambda_{k}\geq j\right\}  $%
\ \ \ \ \footnote{\textit{Proof.} Let $g\in\left\{  1,2,\ldots,i\right\}  $.
We shall show that $g\in\left\{  k\in\left\{  1,2,3,\ldots\right\}
\ \mid\ \lambda_{k}\geq j\right\}  $.
\par
Indeed, $g\in\left\{  1,2,\ldots,i\right\}  \subseteq\left\{  1,2,3,\ldots
\right\}  $ and $g\leq i$ (since $g\in\left\{  1,2,\ldots,i\right\}  $). But
$\lambda$ is a partition (since $\lambda\in\operatorname*{Par}$). Hence,
$\lambda_{1}\geq\lambda_{2}\geq\lambda_{3}\geq\cdots$. Thus, if $u$ and $v$
are two positive integers satisfying $u\leq v$, then $\lambda_{u}\geq
\lambda_{v}$. Applying this to $u=g$ and $v=i$, we obtain $\lambda_{g}%
\geq\lambda_{i}$ (since $g\leq i$). Hence, $\lambda_{g}\geq\lambda_{i}\geq j$.
Now, we know that $g$ is an element of $\left\{  1,2,3,\ldots\right\}  $
(since $g\in\left\{  1,2,3,\ldots\right\}  $) and satisfies $\lambda_{g}\geq
j$. In other words, $g$ is a $k\in\left\{  1,2,3,\ldots\right\}  $ satisfying
$\lambda_{k}\geq j$. In other words, $g\in\left\{  k\in\left\{  1,2,3,\ldots
\right\}  \ \mid\ \lambda_{k}\geq j\right\}  $.
\par
Now, forget that we fixed $g$. We thus have shown that $g\in\left\{
k\in\left\{  1,2,3,\ldots\right\}  \ \mid\ \lambda_{k}\geq j\right\}  $ for
each $g\in\left\{  1,2,\ldots,i\right\}  $. In other words, we have $\left\{
1,2,\ldots,i\right\}  \subseteq\left\{  k\in\left\{  1,2,3,\ldots\right\}
\ \mid\ \lambda_{k}\geq j\right\}  $.} and therefore%
\[
\left\vert \left\{  1,2,\ldots,i\right\}  \right\vert \leq\left\vert \left\{
k\in\left\{  1,2,3,\ldots\right\}  \ \mid\ \lambda_{k}\geq j\right\}
\right\vert =\mu_{j}\ \ \ \ \ \ \ \ \ \ \left(  \text{by
(\ref{pf.prop.partition.transpose.djun.a.muj=})}\right)  .
\]
Hence, $\mu_{j}\geq\left\vert \left\{  1,2,\ldots,i\right\}  \right\vert =i$.
This proves Proposition \ref{prop.partition.transpose.djun} \textbf{(a)}.

\textbf{(b)} Let $i$ and $j$ be two positive integers satisfying $\lambda
_{i}<j$. We must prove that $\mu_{j}<i$.

We have $\left\{  k\in\left\{  1,2,3,\ldots\right\}  \ \mid\ \lambda_{k}\geq
j\right\}  \subseteq\left\{  1,2,\ldots,i-1\right\}  $%
\ \ \ \ \footnote{\textit{Proof.} Let $g\in\left\{  k\in\left\{
1,2,3,\ldots\right\}  \ \mid\ \lambda_{k}\geq j\right\}  $. We shall show that
$g\in\left\{  1,2,\ldots,i-1\right\}  $.
\par
Indeed, assume the contrary. Thus, $g\notin\left\{  1,2,\ldots,i-1\right\}  $.
\par
But $g\in\left\{  k\in\left\{  1,2,3,\ldots\right\}  \ \mid\ \lambda_{k}\geq
j\right\}  $. In other words, $g$ is a $k\in\left\{  1,2,3,\ldots\right\}  $
satisfying $\lambda_{k}\geq j$. In other words, $g$ is an element of $\left\{
1,2,3,\ldots\right\}  $ and satisfies $\lambda_{g}\geq j$. Hence,
$g\in\left\{  1,2,3,\ldots\right\}  $. Combining this with $g\notin\left\{
1,2,\ldots,i-1\right\}  $, we obtain $g\in\left\{  1,2,3,\ldots\right\}
\setminus\left\{  1,2,\ldots,i-1\right\}  =\left\{  i,i+1,i+2,\ldots\right\}
$. Thus, $g\geq i$. Hence, $i\leq g$.
\par
But $\lambda$ is a partition (since $\lambda\in\operatorname*{Par}$). Hence,
$\lambda_{1}\geq\lambda_{2}\geq\lambda_{3}\geq\cdots$. Thus, if $u$ and $v$
are two positive integers satisfying $u\leq v$, then $\lambda_{u}\geq
\lambda_{v}$. Applying this to $u=i$ and $v=g$, we obtain $\lambda_{i}%
\geq\lambda_{g}$ (since $i\leq g$). Hence, $\lambda_{g}\leq\lambda_{i}<j$.
This contradicts $\lambda_{g}\geq j$. This contradiction shows that our
assumption was false. Thus, $g\in\left\{  1,2,\ldots,i-1\right\}  $ is proven.
\par
Now, forget that we fixed $g$. We thus have shown that $g\in\left\{
1,2,\ldots,i-1\right\}  $ for each $g\in\left\{  k\in\left\{  1,2,3,\ldots
\right\}  \ \mid\ \lambda_{k}\geq j\right\}  $. In other words, we have
$\left\{  k\in\left\{  1,2,3,\ldots\right\}  \ \mid\ \lambda_{k}\geq
j\right\}  \subseteq\left\{  1,2,\ldots,i-1\right\}  $.}. Hence,%
\[
\left\vert \left\{  k\in\left\{  1,2,3,\ldots\right\}  \ \mid\ \lambda_{k}\geq
j\right\}  \right\vert \leq\left\vert \left\{  1,2,\ldots,i-1\right\}
\right\vert =i-1.
\]
But (\ref{pf.prop.partition.transpose.djun.mui=}) (applied to $i=j$) yields%
\[
\mu_{j}=\left\vert \left\{  k\in\left\{  1,2,3,\ldots\right\}  \ \mid
\ \lambda_{k}\geq j\right\}  \right\vert \leq i-1<i.
\]
This proves Proposition \ref{prop.partition.transpose.djun} \textbf{(b)}.

\textbf{(c)} Let $i$ and $j$ be two positive integers. We must prove that
$\lambda_{i}+\mu_{j}-i-j\neq-1$.

We are in one of the following two cases:

\textit{Case 1:} We have $\lambda_{i}\geq j$.

\textit{Case 2:} We have $\lambda_{i}<j$.

Let us first consider Case 1. In this case, we have $\lambda_{i}\geq j$.
Hence, Proposition \ref{prop.partition.transpose.djun} \textbf{(a)} yields
$\mu_{j}\geq i$. Hence, $\underbrace{\lambda_{i}}_{\geq j}+\underbrace{\mu
_{j}}_{\geq i}-i-j\geq j+i-i-j=0>-1$. Thus, $\lambda_{i}+\mu_{j}-i-j\neq-1$.
Hence, Proposition \ref{prop.partition.transpose.djun} \textbf{(c)} is proved
in Case 1.

Let us next consider Case 2. In this case, we have $\lambda_{i}<j$. Hence,
Proposition \ref{prop.partition.transpose.djun} \textbf{(b)} yields $\mu
_{j}<i$. Hence, $\mu_{j}\leq i-1$ (since $\mu_{j}$ and $i$ are integers).
Thus, $\underbrace{\lambda_{i}}_{<j}+\underbrace{\mu_{j}}_{\leq i-1}%
-i-j<j+\left(  i-1\right)  -i-j=-1$. Thus, $\lambda_{i}+\mu_{j}-i-j\neq-1$.
Hence, Proposition \ref{prop.partition.transpose.djun} \textbf{(c)} is proved
in Case 2.

We have now proved Proposition \ref{prop.partition.transpose.djun}
\textbf{(c)} in each of the two Cases 1 and 2. Since these two Cases cover all
possibilities, we thus conclude that Proposition
\ref{prop.partition.transpose.djun} \textbf{(c)} always holds.

\textbf{(f)} We have $\alpha_{1}>\alpha_{2}>\alpha_{3}>\cdots$%
\ \ \ \ \footnote{\textit{Proof.} Let $i\in\left\{  1,2,3,\ldots\right\}  $.
We shall show that $\alpha_{i}>\alpha_{i+1}$.
\par
We know that $\lambda$ is a partition (since $\lambda\in\operatorname*{Par}$),
so that $\lambda_{1}\geq\lambda_{2}\geq\lambda_{3}\geq\cdots$. Hence,
$\lambda_{i}\geq\lambda_{i+1}$. But the definition of $\alpha_{i}$ yields
$\alpha_{i}=\lambda_{i}-i$, while the definition of $\alpha_{i+1}$ yields
$\alpha_{i+1}=\lambda_{i+1}-\left(  i+1\right)  $. Hence, $\alpha
_{i}=\underbrace{\lambda_{i}}_{\geq\lambda_{i+1}}-\underbrace{i}%
_{<i+1}>\lambda_{i+1}-\left(  i+1\right)  =\alpha_{i+1}$.
\par
Now, forget that we fixed $i$. We thus have proved that $\alpha_{i}%
>\alpha_{i+1}$ for each $i\in\left\{  1,2,3,\ldots\right\}  $. In other words,
$\alpha_{1}>\alpha_{2}>\alpha_{3}>\cdots$.} and $\beta_{1}>\beta_{2}>\beta
_{3}>\cdots$\ \ \ \ \footnote{\textit{Proof.} Let $j\in\left\{  1,2,3,\ldots
\right\}  $. We shall show that $\beta_{j}>\beta_{j+1}$.
\par
We know that $\mu$ is a partition (since $\mu=\lambda^{t}\in
\operatorname*{Par}$), so that $\mu_{1}\geq\mu_{2}\geq\mu_{3}\geq\cdots$.
Hence, $\mu_{j}\geq\mu_{j+1}$. But the definition of $\beta_{j}$ yields
$\beta_{j}=\mu_{j}-j$, while the definition of $\beta_{j+1}$ yields
$\beta_{j+1}=\mu_{j+1}-\left(  j+1\right)  $. Hence, $\beta_{j}%
=\underbrace{\mu_{j}}_{\geq\mu_{j+1}}-\underbrace{j}_{<j+1}>\mu_{j+1}-\left(
j+1\right)  =\beta_{j+1}$.
\par
Now, forget that we fixed $j$. We thus have proved that $\beta_{j}>\beta
_{j+1}$ for each $j\in\left\{  1,2,3,\ldots\right\}  $. In other words,
$\beta_{1}>\beta_{2}>\beta_{3}>\cdots$.}, hence $\eta_{1}<\eta_{2}<\eta
_{3}<\cdots$\ \ \ \ \footnote{\textit{Proof.} Let $j\in\left\{  1,2,3,\ldots
\right\}  $. We shall show that $\eta_{j}<\eta_{j+1}$.
\par
We know that $\beta_{1}>\beta_{2}>\beta_{3}>\cdots$. Hence, $\beta_{j}%
>\beta_{j+1}$. But the definition of $\eta_{j}$ yields $\eta_{j}=-1-\beta_{j}%
$, while the definition of $\eta_{j+1}$ yields $\eta_{j+1}=-1-\beta_{j+1}$.
Hence, $\eta_{j}=-1-\underbrace{\beta_{j}}_{>\beta_{j+1}}<-1-\beta_{j+1}%
=\eta_{j+1}$.
\par
Now, forget that we fixed $j$. We thus have proved that $\eta_{j}<\eta_{j+1}$
for each $j\in\left\{  1,2,3,\ldots\right\}  $. In other words, $\eta_{1}%
<\eta_{2}<\eta_{3}<\cdots$.}.

From $\alpha_{1}>\alpha_{2}>\alpha_{3}>\cdots$, we obtain $\alpha_{1}%
>\alpha_{2}>\cdots>\alpha_{q}$. Thus, the $q$ integers $\alpha_{1},\alpha
_{2},\ldots,\alpha_{q}$ are distinct. Hence, $\left\vert \left\{  \alpha
_{1},\alpha_{2},\ldots,\alpha_{q}\right\}  \right\vert =q$.

From $\eta_{1}<\eta_{2}<\eta_{3}<\cdots$, we obtain $\eta_{1}<\eta_{2}%
<\cdots<\eta_{p}$. Thus, the $p$ integers $\eta_{1},\eta_{2},\ldots,\eta_{p}$
are distinct. Hence, $\left\vert \left\{  \eta_{1},\eta_{2},\ldots,\eta
_{p}\right\}  \right\vert =p$.

Let $L$ be the finite set $\left\{  -q,-q+1,\ldots,p-1\right\}  =\left\{
k\in\mathbb{Z}\ \mid\ -q\leq k<p\right\}  $. Then,
\begin{align*}
\left\vert L\right\vert  &  =\left(  p-1\right)  -\left(  -q\right)
+1=\underbrace{q}_{=\left\vert \left\{  \alpha_{1},\alpha_{2},\ldots
,\alpha_{q}\right\}  \right\vert }+\underbrace{p}_{=\left\vert \left\{
\eta_{1},\eta_{2},\ldots,\eta_{p}\right\}  \right\vert }\\
&  =\left\vert \left\{  \alpha_{1},\alpha_{2},\ldots,\alpha_{q}\right\}
\right\vert +\left\vert \left\{  \eta_{1},\eta_{2},\ldots,\eta_{p}\right\}
\right\vert .
\end{align*}

We have $\left\{  \alpha_{1},\alpha_{2},\ldots,\alpha_{q}\right\}  \subseteq
L$\ \ \ \ \footnote{\textit{Proof.} Let $i\in\left\{  1,2,\ldots,q\right\}  $.
We shall show that $\alpha_{i}\in L$.
\par
We have $i\in\left\{  1,2,\ldots,q\right\}  $, so that $1\leq i\leq q$. Thus,
$q\geq1$. Hence, $\alpha_{q}$ is well-defined.
\par
From $\alpha_{1}>\alpha_{2}>\cdots>\alpha_{q}$, we conclude that all the $q$
numbers $\alpha_{1},\alpha_{2},\ldots,\alpha_{q}$ lie in the interval between
$\alpha_{q}$ (inclusive) and $\alpha_{1}$ (inclusive). In other words,
$\alpha_{q}\leq\alpha_{j}\leq\alpha_{1}$ for each $j\in\left\{  1,2,\ldots
,q\right\}  $. Applying this to $j=i$, we obtain $\alpha_{q}\leq\alpha_{i}%
\leq\alpha_{1}$.
\par
Now, the definition of $\alpha_{1}$ yields $\alpha_{1}=\lambda_{1}%
-1<\lambda_{1}\leq p$ (since $p\geq\lambda_{1}$). Now, $\alpha_{i}\leq
\alpha_{1}<p$. Also, the definition of $\alpha_{q}$ yields $\alpha
_{q}=\underbrace{\lambda_{q}}_{\geq0}-q\geq-q$, so that $-q\leq\alpha_{q}%
\leq\alpha_{i}$. Hence, $-q\leq\alpha_{i}<p$.
\par
Thus, we know that $\alpha_{i}$ is an element of $\mathbb{Z}$ and satisfies
$-q\leq\alpha_{i}<p$. In other words, $\alpha_{i}$ is a $k\in\mathbb{Z}$
satisfying $-q\leq k<p$. In other words, $\alpha_{i}\in\left\{  k\in
\mathbb{Z}\ \mid\ -q\leq k<p\right\}  $.
\par
But the definition of $L$ yields $L=\left\{  k\in\mathbb{Z}\ \mid\ -q\leq
k<p\right\}  $. Hence, $\alpha_{i}\in\left\{  k\in\mathbb{Z}\ \mid\ -q\leq
k<p\right\}  =L$.
\par
Forget that we fixed $i$. We thus have proven that $\alpha_{i}\in L$ for each
$i\in\left\{  1,2,\ldots,q\right\}  $. In other words, $\alpha_{1},\alpha
_{2},\ldots,\alpha_{q}$ are elements of $L$. In other words, $\left\{
\alpha_{1},\alpha_{2},\ldots,\alpha_{q}\right\}  \subseteq L$.} and $\left\{
\eta_{1},\eta_{2},\ldots,\eta_{p}\right\}  \subseteq L$%
\ \ \ \ \footnote{\textit{Proof.} Let $j\in\left\{  1,2,\ldots,p\right\}  $.
We shall show that $\eta_{j}\in L$.
\par
We have $j\in\left\{  1,2,\ldots,p\right\}  $, so that $1\leq j\leq p$. Thus,
$p\geq1$. Hence, $\eta_{p}$ is well-defined.
\par
From $\eta_{1}<\eta_{2}<\cdots<\eta_{p}$, we conclude that all the $p$ numbers
$\eta_{1},\eta_{2},\ldots,\eta_{p}$ lie in the interval between $\eta_{1}$
(inclusive) and $\eta_{p}$ (inclusive). In other words, $\eta_{1}\leq\eta
_{i}\leq\eta_{p}$ for each $i\in\left\{  1,2,\ldots,p\right\}  $. Applying
this to $i=j$, we obtain $\eta_{1}\leq\eta_{j}\leq\eta_{p}$.
\par
Now, the definition of $\beta_{1}$ yields $\beta_{1}=\underbrace{\mu_{1}%
}_{\substack{\leq q\\\text{(since }q\geq\mu_{1}\text{)}}}-1\leq q-1$. But the
definition of $\eta_{1}$ yields $\eta_{1}=-1-\underbrace{\beta_{1}}_{\leq
q-1}\geq-1-\left(  q-1\right)  =-q$. Hence, $-q\leq\eta_{1}\leq\eta_{j}$.
\par
Also, the definition of $\beta_{p}$ yields $\beta_{p}=\underbrace{\mu_{p}%
}_{\geq0}-p\geq-p$. But the definition of $\eta_{p}$ yields $\eta
_{p}=-1-\underbrace{\beta_{p}}_{\geq-p}\leq-1-\left(  -p\right)  =p-1<p$.
Hence, $\eta_{j}\leq\eta_{p}<p$.
\par
Thus, we know that $\eta_{j}$ is an element of $\mathbb{Z}$ and satisfies
$-q\leq\eta_{j}<p$. In other words, $\eta_{j}$ is a $k\in\mathbb{Z}$
satisfying $-q\leq k<p$. In other words, $\eta_{j}\in\left\{  k\in
\mathbb{Z}\ \mid\ -q\leq k<p\right\}  $.
\par
But the definition of $L$ yields $L=\left\{  k\in\mathbb{Z}\ \mid\ -q\leq
k<p\right\}  $. Hence, $\eta_{j}\in\left\{  k\in\mathbb{Z}\ \mid\ -q\leq
k<p\right\}  =L$.
\par
Forget that we fixed $j$. We thus have proven that $\eta_{j}\in L$ for each
$j\in\left\{  1,2,\ldots,p\right\}  $. In other words, $\eta_{1},\eta
_{2},\ldots,\eta_{p}$ are elements of $L$. In other words, $\left\{  \eta
_{1},\eta_{2},\ldots,\eta_{p}\right\}  \subseteq L$.}. Furthermore,
Proposition \ref{prop.partition.transpose.djun} \textbf{(c)} easily shows that
the sets $\left\{  \alpha_{1},\alpha_{2},\ldots,\alpha_{q}\right\}  $ and
$\left\{  \eta_{1},\eta_{2},\ldots,\eta_{p}\right\}  $ are
disjoint\footnote{\textit{Proof.} Let $\zeta\in\left\{  \alpha_{1},\alpha
_{2},\ldots,\alpha_{q}\right\}  \cap\left\{  \eta_{1},\eta_{2},\ldots,\eta
_{p}\right\}  $. Then, $\zeta\in\left\{  \alpha_{1},\alpha_{2},\ldots
,\alpha_{q}\right\}  \cap\left\{  \eta_{1},\eta_{2},\ldots,\eta_{p}\right\}
\subseteq\left\{  \alpha_{1},\alpha_{2},\ldots,\alpha_{q}\right\}  $; in other
words, there exists some $i\in\left\{  1,2,\ldots,q\right\}  $ such that
$\zeta=\alpha_{i}$. Consider this $i$. We then have $\zeta=\alpha_{i}%
=\lambda_{i}-i$ (by the definition of $\alpha_{i}$).
\par
Also, $\zeta\in\left\{  \alpha_{1},\alpha_{2},\ldots,\alpha_{q}\right\}
\cap\left\{  \eta_{1},\eta_{2},\ldots,\eta_{p}\right\}  \subseteq\left\{
\eta_{1},\eta_{2},\ldots,\eta_{p}\right\}  $; in other words, there exists
some $j\in\left\{  1,2,\ldots,p\right\}  $ such that $\zeta=\eta_{j}$.
Consider this $j$. We then have $\zeta=\eta_{j}=-1-\beta_{j}$ (by the
definition of $\eta_{j}$). But the definition of $\beta_{j}$ yields $\beta
_{j}=\mu_{j}-j$. Hence, $\zeta=-1-\underbrace{\beta_{j}}_{=\mu_{j}%
-j}=-1-\left(  \mu_{j}-j\right)  =-1-\mu_{j}+j$. Comparing this with
$\zeta=\lambda_{i}-i$, we obtain $\lambda_{i}-i=-1-\mu_{j}+j$. In other words,
$\lambda_{i}+\mu_{j}-i-j=-1$. But Proposition
\ref{prop.partition.transpose.djun} \textbf{(c)} yields $\lambda_{i}+\mu
_{j}-i-j\neq-1$. This contradicts $\lambda_{i}+\mu_{j}-i-j=-1$.
\par
Forget that we fixed $\zeta$. We thus have found a contradiction for each
$\zeta\in\left\{  \alpha_{1},\alpha_{2},\ldots,\alpha_{q}\right\}
\cap\left\{  \eta_{1},\eta_{2},\ldots,\eta_{p}\right\}  $. Thus, there exists
no $\zeta\in\left\{  \alpha_{1},\alpha_{2},\ldots,\alpha_{q}\right\}
\cap\left\{  \eta_{1},\eta_{2},\ldots,\eta_{p}\right\}  $. In other words, the
set $\left\{  \alpha_{1},\alpha_{2},\ldots,\alpha_{q}\right\}  \cap\left\{
\eta_{1},\eta_{2},\ldots,\eta_{p}\right\}  $ is empty. In other words, the
sets $\left\{  \alpha_{1},\alpha_{2},\ldots,\alpha_{q}\right\}  $ and
$\left\{  \eta_{1},\eta_{2},\ldots,\eta_{p}\right\}  $ are disjoint.}.

Now, recall the following basic fact from the theory of finite sets:

\begin{statement}
\textit{Fact A:} Let $U$, $V$ and $W$ be three finite sets such that
$V\subseteq U$ and $W\subseteq U$ and $\left\vert U\right\vert =\left\vert
V\right\vert +\left\vert W\right\vert $. Assume that $V$ and $W$ are disjoint.
Then, $V\cup W=U$.
\end{statement}

[\textit{Proof of Fact A:} The set $V$ is a subset of $U$ (since $V\subseteq
U$), but is disjoint from $W$ (since $V$ and $W$ are disjoint). Thus, $V$ is a
subset of $U\setminus W$. But from $W\subseteq U$, we obtain $\left\vert
U\setminus W\right\vert =\left\vert U\right\vert -\left\vert W\right\vert
=\left\vert V\right\vert $ (since $\left\vert U\right\vert =\left\vert
V\right\vert +\left\vert W\right\vert $). Hence, $\left\vert V\right\vert
=\left\vert U\setminus W\right\vert $. Thus, the set $V$ has the same size as
$U\setminus W$. Note that the set $U\setminus W$ is finite (since $U$ is finite).

Now, recall the well-known fact that if a subset $Q$ of a finite set $R$ has
the same size as $R$, then $Q=R$. We can apply this to $R=U\setminus W$ and
$Q=V$ (since $V$ is a subset of $U\setminus W$ and has the same size as
$U\setminus W$), and conclude that $V=U\setminus W$. Hence, $\underbrace{V}%
_{=U\setminus W}\cup W=\left(  U\setminus W\right)  \cup W=U$ (since
$W\subseteq U$). This proves Fact A.]

Now, recall that $L$, $\left\{  \alpha_{1},\alpha_{2},\ldots,\alpha
_{q}\right\}  $ and $\left\{  \eta_{1},\eta_{2},\ldots,\eta_{p}\right\}  $ are
three finite sets such that $\left\{  \alpha_{1},\alpha_{2},\ldots,\alpha
_{q}\right\}  \subseteq L$ and $\left\{  \eta_{1},\eta_{2},\ldots,\eta
_{p}\right\}  \subseteq L$ and $\left\vert L\right\vert =\left\vert \left\{
\alpha_{1},\alpha_{2},\ldots,\alpha_{q}\right\}  \right\vert +\left\vert
\left\{  \eta_{1},\eta_{2},\ldots,\eta_{p}\right\}  \right\vert $ and such
that the sets $\left\{  \alpha_{1},\alpha_{2},\ldots,\alpha_{q}\right\}  $ and
$\left\{  \eta_{1},\eta_{2},\ldots,\eta_{p}\right\}  $ are disjoint. Hence,
Fact A (applied to $U=L$, $V=\left\{  \alpha_{1},\alpha_{2},\ldots,\alpha
_{q}\right\}  $ and $W=\left\{  \eta_{1},\eta_{2},\ldots,\eta_{p}\right\}  $)
yields that $\left\{  \alpha_{1},\alpha_{2},\ldots,\alpha_{q}\right\}
\cup\left\{  \eta_{1},\eta_{2},\ldots,\eta_{p}\right\}  =L$. In other words,
the union of the two sets $\left\{  \alpha_{1},\alpha_{2},\ldots,\alpha
_{q}\right\}  $ and $\left\{  \eta_{1},\eta_{2},\ldots,\eta_{p}\right\}  $ is
$L$.

Thus, we have shown that the two sets $\left\{  \alpha_{1},\alpha_{2}%
,\ldots,\alpha_{q}\right\}  $ and $\left\{  \eta_{1},\eta_{2},\ldots,\eta
_{p}\right\}  $ are disjoint, and their union is
\[
L=\left\{  -q,-q+1,\ldots,p-1\right\}  =\left\{  k\in\mathbb{Z}\ \mid\ -q\leq
k<p\right\}  .
\]
This proves Proposition \ref{prop.partition.transpose.djun} \textbf{(f)}.

\textbf{(e)} We have $\alpha_{1}>\alpha_{2}>\alpha_{3}>\cdots$ (as we have
shown in our above proof of Proposition \ref{prop.partition.transpose.djun}
\textbf{(f)}) and $\eta_{1}<\eta_{2}<\cdots<\eta_{p}$ (as we have shown in our
above proof of Proposition \ref{prop.partition.transpose.djun} \textbf{(f)}).

Let $M$ be the set $\left\{  \ldots,p-3,p-2,p-1\right\}  =\left\{
k\in\mathbb{Z}\ \mid\ k<p\right\}  $. Then, $\left\{  \alpha_{1},\alpha
_{2},\alpha_{3},\ldots\right\}  \subseteq M$\ \ \ \ \footnote{\textit{Proof.}
Let $i\in\left\{  1,2,3,\ldots\right\}  $. We shall show that $\alpha_{i}\in
M$.
\par
From $\alpha_{1}>\alpha_{2}>\alpha_{3}>\cdots$, we conclude that $\alpha
_{1}\geq\alpha_{j}$ for each $j\in\left\{  1,2,3,\ldots\right\}  $. Applying
this to $j=i$, we obtain $\alpha_{1}\geq\alpha_{i}$, so that $\alpha_{i}%
\leq\alpha_{1}$.
\par
Now, the definition of $\alpha_{1}$ yields $\alpha_{1}=\lambda_{1}%
-1<\lambda_{1}\leq p$ (since $p\geq\lambda_{1}$). Now, $\alpha_{i}\leq
\alpha_{1}<p$.
\par
Thus, we know that $\alpha_{i}$ is an element of $\mathbb{Z}$ and satisfies
$\alpha_{i}<p$. In other words, $\alpha_{i}$ is a $k\in\mathbb{Z}$ satisfying
$k<p$. In other words, $\alpha_{i}\in\left\{  k\in\mathbb{Z}\ \mid
\ k<p\right\}  $.
\par
But the definition of $M$ yields $M=\left\{  k\in\mathbb{Z}\ \mid
\ k<p\right\}  $. Hence, $\alpha_{i}\in\left\{  k\in\mathbb{Z}\ \mid
\ k<p\right\}  =M$.
\par
Forget that we fixed $i$. We thus have proven that $\alpha_{i}\in M$ for each
$i\in\left\{  1,2,3,\ldots\right\}  $. In other words, $\alpha_{1},\alpha
_{2},\alpha_{3},\ldots$ are elements of $M$. In other words, $\left\{
\alpha_{1},\alpha_{2},\alpha_{3},\ldots\right\}  \subseteq M$.} and $\left\{
\eta_{1},\eta_{2},\ldots,\eta_{p}\right\}  \subseteq M$%
\ \ \ \ \footnote{\textit{Proof.} Let $j\in\left\{  1,2,\ldots,p\right\}  $.
We shall show that $\eta_{j}\in L$.
\par
We have $j\in\left\{  1,2,\ldots,p\right\}  $, so that $1\leq j\leq p$. Thus,
$p\geq1$. Hence, $\eta_{p}$ is well-defined.
\par
From $\eta_{1}<\eta_{2}<\cdots<\eta_{p}$, we conclude that all the $p$ numbers
$\eta_{1},\eta_{2},\ldots,\eta_{p}$ lie in the interval between $\eta_{1}$
(inclusive) and $\eta_{p}$ (inclusive). In other words, $\eta_{1}\leq\eta
_{i}\leq\eta_{p}$ for each $i\in\left\{  1,2,\ldots,p\right\}  $. Applying
this to $i=j$, we obtain $\eta_{1}\leq\eta_{j}\leq\eta_{p}$.
\par
The definition of $\beta_{p}$ yields $\beta_{p}=\underbrace{\mu_{p}}_{\geq
0}-p\geq-p$. But the definition of $\eta_{p}$ yields $\eta_{p}%
=-1-\underbrace{\beta_{p}}_{\geq-p}\leq-1-\left(  -p\right)  =p-1<p$. Hence,
$\eta_{j}\leq\eta_{p}<p$.
\par
Thus, we know that $\eta_{j}$ is an element of $\mathbb{Z}$ and satisfies
$\eta_{j}<p$. In other words, $\eta_{j}$ is a $k\in\mathbb{Z}$ satisfying
$k<p$. In other words, $\eta_{j}\in\left\{  k\in\mathbb{Z}\ \mid\ k<p\right\}
$.
\par
But the definition of $M$ yields $M=\left\{  k\in\mathbb{Z}\ \mid
\ k<p\right\}  $. Hence, $\eta_{j}\in\left\{  k\in\mathbb{Z}\ \mid
\ k<p\right\}  =M$.
\par
Forget that we fixed $j$. We thus have proven that $\eta_{j}\in M$ for each
$j\in\left\{  1,2,\ldots,p\right\}  $. In other words, $\eta_{1},\eta
_{2},\ldots,\eta_{p}$ are elements of $M$. In other words, $\left\{  \eta
_{1},\eta_{2},\ldots,\eta_{p}\right\}  \subseteq M$.}. Furthermore,
Proposition \ref{prop.partition.transpose.djun} \textbf{(c)} easily shows that
the sets $\left\{  \alpha_{1},\alpha_{2},\alpha_{3},\ldots\right\}  $ and
$\left\{  \eta_{1},\eta_{2},\ldots,\eta_{p}\right\}  $ are
disjoint\footnote{\textit{Proof.} Let $\zeta\in\left\{  \alpha_{1},\alpha
_{2},\alpha_{3},\ldots\right\}  \cap\left\{  \eta_{1},\eta_{2},\ldots,\eta
_{p}\right\}  $. Then, $\zeta\in\left\{  \alpha_{1},\alpha_{2},\alpha
_{3},\ldots\right\}  \cap\left\{  \eta_{1},\eta_{2},\ldots,\eta_{p}\right\}
\subseteq\left\{  \alpha_{1},\alpha_{2},\alpha_{3},\ldots\right\}  $; in other
words, there exists some $i\in\left\{  1,2,3,\ldots\right\}  $ such that
$\zeta=\alpha_{i}$. Consider this $i$. We then have $\zeta=\alpha_{i}%
=\lambda_{i}-i$ (by the definition of $\alpha_{i}$).
\par
Also, $\zeta\in\left\{  \alpha_{1},\alpha_{2},\alpha_{3},\ldots\right\}
\cap\left\{  \eta_{1},\eta_{2},\ldots,\eta_{p}\right\}  \subseteq\left\{
\eta_{1},\eta_{2},\ldots,\eta_{p}\right\}  $; in other words, there exists
some $j\in\left\{  1,2,\ldots,p\right\}  $ such that $\zeta=\eta_{j}$.
Consider this $j$. We then have $\zeta=\eta_{j}=-1-\beta_{j}$ (by the
definition of $\eta_{j}$). But the definition of $\beta_{j}$ yields $\beta
_{j}=\mu_{j}-j$. Hence, $\zeta=-1-\underbrace{\beta_{j}}_{=\mu_{j}%
-j}=-1-\left(  \mu_{j}-j\right)  =-1-\mu_{j}+j$. Comparing this with
$\zeta=\lambda_{i}-i$, we obtain $\lambda_{i}-i=-1-\mu_{j}+j$. In other words,
$\lambda_{i}+\mu_{j}-i-j=-1$. But Proposition
\ref{prop.partition.transpose.djun} \textbf{(c)} yields $\lambda_{i}+\mu
_{j}-i-j\neq-1$. This contradicts $\lambda_{i}+\mu_{j}-i-j=-1$.
\par
Forget that we fixed $\zeta$. We thus have found a contradiction for each
$\zeta\in\left\{  \alpha_{1},\alpha_{2},\alpha_{3},\ldots\right\}
\cap\left\{  \eta_{1},\eta_{2},\ldots,\eta_{p}\right\}  $. Thus, there exists
no $\zeta\in\left\{  \alpha_{1},\alpha_{2},\alpha_{3},\ldots\right\}
\cap\left\{  \eta_{1},\eta_{2},\ldots,\eta_{p}\right\}  $. In other words, the
set $\left\{  \alpha_{1},\alpha_{2},\alpha_{3},\ldots\right\}  \cap\left\{
\eta_{1},\eta_{2},\ldots,\eta_{p}\right\}  $ is empty. In other words, the
sets $\left\{  \alpha_{1},\alpha_{2},\alpha_{3},\ldots\right\}  $ and
$\left\{  \eta_{1},\eta_{2},\ldots,\eta_{p}\right\}  $ are disjoint.}.
Moreover, $\left\{  \alpha_{1},\alpha_{2},\alpha_{3},\ldots\right\}
\cup\left\{  \eta_{1},\eta_{2},\ldots,\eta_{p}\right\}  =M$%
\ \ \ \ \footnote{\textit{Proof.} Combining $\left\{  \alpha_{1},\alpha
_{2},\alpha_{3},\ldots\right\}  \subseteq M$ with $\left\{  \eta_{1},\eta
_{2},\ldots,\eta_{p}\right\}  \subseteq M$, we obtain $\left\{  \alpha
_{1},\alpha_{2},\alpha_{3},\ldots\right\}  \cup\left\{  \eta_{1},\eta
_{2},\ldots,\eta_{p}\right\}  \subseteq M$. We shall now prove the reverse
inclusion.
\par
Fix $m\in M$. Thus, $m\in M=\left\{  \ldots,p-3,p-2,p-1\right\}  $ (by the
definition of $M$), so that $m\leq p-1<p$.
\par
Let $q=\max\left\{  \mu_{1},-m\right\}  $. Thus, $q=\max\left\{  \mu
_{1},m\right\}  \geq-m$ and $q=\max\left\{  \mu_{1},m\right\}  \geq\mu_{1}$.
\par
From $q\geq-m$, we obtain $-\underbrace{q}_{\geq-m}\leq-\left(  -m\right)
=m$. Hence, $-q\leq m<p$. Thus, $m$ is an element of $\mathbb{Z}$ satisfying
$-q\leq m<p$. In other words, $m$ is a $k\in\mathbb{Z}$ satisfying $-q\leq
k<p$. In other words, $m\in\left\{  k\in\mathbb{Z}\ \mid\ -q\leq k<p\right\}
$.
\par
But Proposition \ref{prop.partition.transpose.djun} \textbf{(f)} yields that
the two sets $\left\{  \alpha_{1},\alpha_{2},\ldots,\alpha_{q}\right\}  $ and
$\left\{  \eta_{1},\eta_{2},\ldots,\eta_{p}\right\}  $ are disjoint, and their
union is
\[
\left\{  -q,-q+1,\ldots,p-1\right\}  =\left\{  k\in\mathbb{Z}\ \mid\ -q\leq
k<p\right\}  .
\]
Thus, in particular, their union is $\left\{  -q,-q+1,\ldots,p-1\right\}
=\left\{  k\in\mathbb{Z}\ \mid\ -q\leq k<p\right\}  $. In other words,%
\[
\left\{  \alpha_{1},\alpha_{2},\ldots,\alpha_{q}\right\}  \cup\left\{
\eta_{1},\eta_{2},\ldots,\eta_{p}\right\}  =\left\{  -q,-q+1,\ldots
,p-1\right\}  =\left\{  k\in\mathbb{Z}\ \mid\ -q\leq k<p\right\}  .
\]
Hence,%
\begin{align*}
m  &  \in\left\{  k\in\mathbb{Z}\ \mid\ -q\leq k<p\right\}
=\underbrace{\left\{  \alpha_{1},\alpha_{2},\ldots,\alpha_{q}\right\}
}_{\subseteq\left\{  \alpha_{1},\alpha_{2},\alpha_{3},\ldots\right\}  }%
\cup\left\{  \eta_{1},\eta_{2},\ldots,\eta_{p}\right\} \\
&  \subseteq\left\{  \alpha_{1},\alpha_{2},\alpha_{3},\ldots\right\}
\cup\left\{  \eta_{1},\eta_{2},\ldots,\eta_{p}\right\}  .
\end{align*}
\par
Forget that we fixed $m$. We thus have shown that $m\in\left\{  \alpha
_{1},\alpha_{2},\alpha_{3},\ldots\right\}  \cup\left\{  \eta_{1},\eta
_{2},\ldots,\eta_{p}\right\}  $ for each $m\in M$. In other words,
$M\subseteq\left\{  \alpha_{1},\alpha_{2},\alpha_{3},\ldots\right\}
\cup\left\{  \eta_{1},\eta_{2},\ldots,\eta_{p}\right\}  $. Combining this with
$\left\{  \alpha_{1},\alpha_{2},\alpha_{3},\ldots\right\}  \cup\left\{
\eta_{1},\eta_{2},\ldots,\eta_{p}\right\}  \subseteq M$, we obtain $\left\{
\alpha_{1},\alpha_{2},\alpha_{3},\ldots\right\}  \cup\left\{  \eta_{1}%
,\eta_{2},\ldots,\eta_{p}\right\}  =M$. Qed.}. In other words, the union of
the two sets $\left\{  \alpha_{1},\alpha_{2},\alpha_{3},\ldots\right\}  $ and
$\left\{  \eta_{1},\eta_{2},\ldots,\eta_{p}\right\}  $ is $M$.

Thus, we have shown that the two sets $\left\{  \alpha_{1},\alpha_{2}%
,\alpha_{3},\ldots\right\}  $ and $\left\{  \eta_{1},\eta_{2},\ldots,\eta
_{p}\right\}  $ are disjoint, and their union is
\[
M=\left\{  \ldots,p-3,p-2,p-1\right\}  =\left\{  k\in\mathbb{Z}\ \mid
\ k<p\right\}  .
\]
This proves Proposition \ref{prop.partition.transpose.djun} \textbf{(e)}.

\textbf{(d)} We have $\alpha_{1}>\alpha_{2}>\alpha_{3}>\cdots$ (as we have
shown in our above proof of Proposition \ref{prop.partition.transpose.djun}
\textbf{(f)}) and $\eta_{1}<\eta_{2}<\eta_{3}<\cdots$ (as we have shown in our
above proof of Proposition \ref{prop.partition.transpose.djun} \textbf{(f)}).

We have $\left\{  \alpha_{1},\alpha_{2},\alpha_{3},\ldots\right\}
\subseteq\mathbb{Z}$ (since $\alpha_{1},\alpha_{2},\alpha_{3},\ldots$ are
integers) and $\left\{  \eta_{1},\eta_{2},\eta_{3},\ldots\right\}
\subseteq\mathbb{Z}$ (since $\eta_{1},\eta_{2},\eta_{3},\ldots$ are integers).
Furthermore, Proposition \ref{prop.partition.transpose.djun} \textbf{(c)}
easily shows that the sets $\left\{  \alpha_{1},\alpha_{2},\alpha_{3}%
,\ldots\right\}  $ and $\left\{  \eta_{1},\eta_{2},\eta_{3},\ldots\right\}  $
are disjoint\footnote{\textit{Proof.} Let $\zeta\in\left\{  \alpha_{1}%
,\alpha_{2},\alpha_{3},\ldots\right\}  \cap\left\{  \eta_{1},\eta_{2},\eta
_{3},\ldots\right\}  $. Then, $\zeta\in\left\{  \alpha_{1},\alpha_{2}%
,\alpha_{3},\ldots\right\}  \cap\left\{  \eta_{1},\eta_{2},\eta_{3}%
,\ldots\right\}  \subseteq\left\{  \alpha_{1},\alpha_{2},\alpha_{3}%
,\ldots\right\}  $; in other words, there exists some $i\in\left\{
1,2,3,\ldots\right\}  $ such that $\zeta=\alpha_{i}$. Consider this $i$. We
then have $\zeta=\alpha_{i}=\lambda_{i}-i$ (by the definition of $\alpha_{i}%
$).
\par
Also, $\zeta\in\left\{  \alpha_{1},\alpha_{2},\alpha_{3},\ldots\right\}
\cap\left\{  \eta_{1},\eta_{2},\eta_{3},\ldots\right\}  \subseteq\left\{
\eta_{1},\eta_{2},\eta_{3},\ldots\right\}  $; in other words, there exists
some $j\in\left\{  1,2,3,\ldots\right\}  $ such that $\zeta=\eta_{j}$.
Consider this $j$. We then have $\zeta=\eta_{j}=-1-\beta_{j}$ (by the
definition of $\eta_{j}$). But the definition of $\beta_{j}$ yields $\beta
_{j}=\mu_{j}-j$. Hence, $\zeta=-1-\underbrace{\beta_{j}}_{=\mu_{j}%
-j}=-1-\left(  \mu_{j}-j\right)  =-1-\mu_{j}+j$. Comparing this with
$\zeta=\lambda_{i}-i$, we obtain $\lambda_{i}-i=-1-\mu_{j}+j$. In other words,
$\lambda_{i}+\mu_{j}-i-j=-1$. But Proposition
\ref{prop.partition.transpose.djun} \textbf{(c)} yields $\lambda_{i}+\mu
_{j}-i-j\neq-1$. This contradicts $\lambda_{i}+\mu_{j}-i-j=-1$.
\par
Forget that we fixed $\zeta$. We thus have found a contradiction for each
$\zeta\in\left\{  \alpha_{1},\alpha_{2},\alpha_{3},\ldots\right\}
\cap\left\{  \eta_{1},\eta_{2},\eta_{3},\ldots\right\}  $. Thus, there exists
no $\zeta\in\left\{  \alpha_{1},\alpha_{2},\alpha_{3},\ldots\right\}
\cap\left\{  \eta_{1},\eta_{2},\eta_{3},\ldots\right\}  $. In other words, the
set $\left\{  \alpha_{1},\alpha_{2},\alpha_{3},\ldots\right\}  \cap\left\{
\eta_{1},\eta_{2},\eta_{3},\ldots\right\}  $ is empty. In other words, the
sets $\left\{  \alpha_{1},\alpha_{2},\alpha_{3},\ldots\right\}  $ and
$\left\{  \eta_{1},\eta_{2},\eta_{3},\ldots\right\}  $ are disjoint.}.
Moreover, $\left\{  \alpha_{1},\alpha_{2},\alpha_{3},\ldots\right\}
\cup\left\{  \eta_{1},\eta_{2},\eta_{3},\ldots\right\}  =\mathbb{Z}%
$\ \ \ \ \footnote{\textit{Proof.} Combining $\left\{  \alpha_{1},\alpha
_{2},\alpha_{3},\ldots\right\}  \subseteq\mathbb{Z}$ with $\left\{  \eta
_{1},\eta_{2},\eta_{3},\ldots\right\}  \subseteq\mathbb{Z}$, we obtain
$\left\{  \alpha_{1},\alpha_{2},\alpha_{3},\ldots\right\}  \cup\left\{
\eta_{1},\eta_{2},\eta_{3},\ldots\right\}  \subseteq\mathbb{Z}$. We shall now
prove the reverse inclusion.
\par
Fix $m\in\mathbb{Z}$.
\par
Let $p=\max\left\{  \lambda_{1},m+1\right\}  $. Thus, $p=\max\left\{
\lambda_{1},m+1\right\}  \geq\lambda_{1}$ and $p=\max\left\{  \lambda
_{1},m+1\right\}  \geq m+1>m$. From $p>m$, we obtain $m<p$. Thus, $m$ is a
$k\in\mathbb{Z}$ satisfying $k<p$ (since $m\in\mathbb{Z}$ and $m<p$). In other
words, $m\in\left\{  k\in\mathbb{Z}\ \mid\ k<p\right\}  $.
\par
But Proposition \ref{prop.partition.transpose.djun} \textbf{(e)} yields that
the two sets $\left\{  \alpha_{1},\alpha_{2},\alpha_{3},\ldots\right\}  $ and
$\left\{  \eta_{1},\eta_{2},\ldots,\eta_{p}\right\}  $ are disjoint, and their
union is
\[
\left\{  \ldots,p-3,p-2,p-1\right\}  =\left\{  k\in\mathbb{Z}\ \mid
\ k<p\right\}  .
\]
Thus, in particular, their union is $\left\{  \ldots,p-3,p-2,p-1\right\}
=\left\{  k\in\mathbb{Z}\ \mid\ k<p\right\}  $. In other words,%
\[
\left\{  \alpha_{1},\alpha_{2},\alpha_{3},\ldots\right\}  \cup\left\{
\eta_{1},\eta_{2},\ldots,\eta_{p}\right\}  =\left\{  \ldots
,p-3,p-2,p-1\right\}  =\left\{  k\in\mathbb{Z}\ \mid\ k<p\right\}  .
\]
Hence,%
\[
m\in\left\{  k\in\mathbb{Z}\ \mid\ k<p\right\}  =\left\{  \alpha_{1}%
,\alpha_{2},\alpha_{3},\ldots\right\}  \cup\underbrace{\left\{  \eta_{1}%
,\eta_{2},\ldots,\eta_{p}\right\}  }_{\subseteq\left\{  \eta_{1},\eta_{2}%
,\eta_{3},\ldots\right\}  }\subseteq\left\{  \alpha_{1},\alpha_{2},\alpha
_{3},\ldots\right\}  \cup\left\{  \eta_{1},\eta_{2},\eta_{3},\ldots\right\}
.
\]
\par
Forget that we fixed $m$. We thus have shown that $m\in\left\{  \alpha
_{1},\alpha_{2},\alpha_{3},\ldots\right\}  \cup\left\{  \eta_{1},\eta_{2}%
,\eta_{3},\ldots\right\}  $ for each $m\in\mathbb{Z}$. In other words,
$\mathbb{Z}\subseteq\left\{  \alpha_{1},\alpha_{2},\alpha_{3},\ldots\right\}
\cup\left\{  \eta_{1},\eta_{2},\eta_{3},\ldots\right\}  $. Combining this with
$\left\{  \alpha_{1},\alpha_{2},\alpha_{3},\ldots\right\}  \cup\left\{
\eta_{1},\eta_{2},\eta_{3},\ldots\right\}  \subseteq\mathbb{Z}$, we obtain
$\left\{  \alpha_{1},\alpha_{2},\alpha_{3},\ldots\right\}  \cup\left\{
\eta_{1},\eta_{2},\eta_{3},\ldots\right\}  =\mathbb{Z}$. Qed.}. In other
words, the union of the two sets $\left\{  \alpha_{1},\alpha_{2},\alpha
_{3},\ldots\right\}  $ and $\left\{  \eta_{1},\eta_{2},\eta_{3},\ldots
\right\}  $ is $\mathbb{Z}$.

Thus, we have shown that the two sets $\left\{  \alpha_{1},\alpha_{2}%
,\alpha_{3},\ldots\right\}  $ and $\left\{  \eta_{1},\eta_{2},\eta_{3}%
,\ldots\right\}  $ are disjoint, and their union is $\mathbb{Z}$. This proves
Proposition \ref{prop.partition.transpose.djun} \textbf{(d)}.
\end{proof}
\end{verlong}

\begin{proof}
[Proof of Proposition \ref{prop.petk.explicit-old} (sketched).]Let
$\mu=\lambda^{t}$. Then, the number of parts of $\mu$ is $\lambda_{1}$. Hence,
from $\lambda_{1}<k$, we conclude that $\mu$ has fewer than $k$ parts. Thus,
$\mu_{k}=0$.

For each positive integer $i$, set $\alpha_{i}=\lambda_{i}-i$. Hence,%
\begin{align*}
\left\{  \alpha_{1},\alpha_{2},\alpha_{3},\ldots\right\}   &  =\left\{
\underbrace{\alpha_{i}}_{=\lambda_{i}-i}\ \mid\ i\in\left\{  1,2,3,\ldots
\right\}  \right\}  =\left\{  \lambda_{i}-i\ \mid\ i\in\left\{  1,2,3,\ldots
\right\}  \right\} \\
&  =B\ \ \ \ \ \ \ \ \ \ \left(  \text{by the definition of }B\right)  .
\end{align*}

For each positive integer $j$, set $\beta_{j}=\mu_{j}-j$ and $\eta
_{j}=-1-\beta_{j}$. Note that $\left(  \beta_{1},\beta_{2},\ldots,\beta
_{k-1}\right)  \in\mathbb{Z}^{k-1}$ is thus the same $\left(  k-1\right)
$-tuple that was called $\left(  \beta_{1},\beta_{2},\ldots,\beta
_{k-1}\right)  $ in Theorem \ref{thm.petk.explicit}. It is easy to see that
$\beta_{1}>\beta_{2}>\cdots>\beta_{k-1}$ and $\lambda_{1}-1>\lambda
_{2}-2>\lambda_{3}-3>\cdots$.

From $\lambda_{1}<k$, we obtain $k-1\geq\lambda_{1}$. Hence, Proposition
\ref{prop.partition.transpose.djun} \textbf{(e)} (applied to $p=k-1$) yields
that the two sets $\left\{  \alpha_{1},\alpha_{2},\alpha_{3},\ldots\right\}  $
and $\left\{  \eta_{1},\eta_{2},\ldots,\eta_{k-1}\right\}  $ are disjoint, and
their union is
\[
\left\{  \ldots,\left(  k-1\right)  -3,\left(  k-1\right)  -2,\left(
k-1\right)  -1\right\}  =\left\{  \text{all integers smaller than
}k-1\right\}  =W.
\]
Since $\left\{  \alpha_{1},\alpha_{2},\alpha_{3},\ldots\right\}  =B$, we can
restate this as follows: The two sets $B$ and $\left\{  \eta_{1},\eta
_{2},\ldots,\eta_{k-1}\right\}  $ are disjoint, and their union is $W$. Hence,
$\left\{  \eta_{1},\eta_{2},\ldots,\eta_{k-1}\right\}  =W\setminus B$.

It is also easy to see that $\beta_{1}>\beta_{2}>\cdots>\beta_{k-1}$, so that
$\eta_{1}<\eta_{2}<\cdots<\eta_{k-1}$. Hence, $\eta_{1},\eta_{2},\ldots
,\eta_{k-1}$ are the elements of the set $\left\{  \eta_{1},\eta_{2}%
,\ldots,\eta_{k-1}\right\}  $ listed in increasing order (with no repetition).

Let us define a $\left(  k-1\right)  $-tuple $\left(  \gamma_{1},\gamma
_{2},\ldots,\gamma_{k-1}\right)  \in\left\{  1,2,\ldots,k\right\}  ^{k-1}$ as
in Theorem \ref{thm.petk.explicit}. Now, we have the following chain of
logical equivalences:%
\begin{align*}
&  \ \left(  \operatorname*{pet}\nolimits_{k}\left(  \lambda,\varnothing
\right)  \neq0\right) \\
&  \Longleftrightarrow\ \left(  \text{the }k-1\text{ numbers }\gamma
_{1},\gamma_{2},\ldots,\gamma_{k-1}\text{ are distinct}\right) \\
&  \ \ \ \ \ \ \ \ \ \ \ \ \ \ \ \ \ \ \ \ \left(  \text{by parts \textbf{(b)}
and \textbf{(c)} of Theorem \ref{thm.petk.explicit}}\right) \\
&  \Longleftrightarrow\ \left(  \text{the }k-1\text{ numbers }\left(
\beta_{1}-1\right)  \%k,\left(  \beta_{2}-1\right)  \%k,\ldots,\left(
\beta_{k-1}-1\right)  \%k\text{ are distinct}\right) \\
&  \ \ \ \ \ \ \ \ \ \ \ \ \ \ \ \ \ \ \ \ \left(  \text{since }\gamma
_{i}=1+\left(  \beta_{i}-1\right)  \%k\text{ for each }i\right) \\
&  \Longleftrightarrow\ \left(  \text{no two of the }k-1\text{ numbers }%
\beta_{1}-1,\beta_{2}-1,\ldots,\beta_{k-1}-1\text{ are congruent modulo
}k\right) \\
&  \Longleftrightarrow\ \left(  \text{no two of the }k-1\text{ numbers }%
\beta_{1},\beta_{2},\ldots,\beta_{k-1}\text{ are congruent modulo }k\right) \\
&  \Longleftrightarrow\ \left(  \text{no two of the }k-1\text{ numbers
}-1-\beta_{1},-1-\beta_{2},\ldots,-1-\beta_{k-1}\text{ }\right. \\
&  \ \ \ \ \ \ \ \ \ \ \left.  \text{are congruent modulo }k\right) \\
&  \Longleftrightarrow\ \left(  \text{no two of the }k-1\text{ numbers }%
\eta_{1},\eta_{2},\ldots,\eta_{k-1}\text{ are congruent modulo }k\right) \\
&  \ \ \ \ \ \ \ \ \ \ \ \ \ \ \ \ \ \ \ \ \left(  \text{since }\eta
_{j}=-1-\beta_{j}\text{ for each }j\right) \\
&  \Longleftrightarrow\ \left(  \text{no two of the }k-1\text{ elements of
}\left\{  \eta_{1},\eta_{2},\ldots,\eta_{k-1}\right\}  \text{ are congruent
modulo }k\right) \\
&  \ \ \ \ \ \ \ \ \ \ \left(
\begin{array}
[c]{c}%
\text{since }\eta_{1},\eta_{2},\ldots,\eta_{k-1}\text{ are the elements of the
set }\left\{  \eta_{1},\eta_{2},\ldots,\eta_{k-1}\right\} \\
\text{listed in increasing order (with no repetition)}%
\end{array}
\right) \\
&  \Longleftrightarrow\ \left(  \text{no two of the }k-1\text{ elements of
}W\setminus B\text{ are congruent modulo }k\right) \\
&  \ \ \ \ \ \ \ \ \ \ \ \ \ \ \ \ \ \ \ \ \left(  \text{since }\left\{
\eta_{1},\eta_{2},\ldots,\eta_{k-1}\right\}  =W\setminus B\right) \\
&  \Longleftrightarrow\ \left(  \text{each congruence class }\overline
{i}\text{ has at most }1\text{ element in common with }W\setminus B\right) \\
&  \Longleftrightarrow\ \left(  \text{each }i\in\left\{  0,1,\ldots
,k-1\right\}  \text{ satisfies }\left\vert \overline{i}\cap\left(  W\setminus
B\right)  \right\vert \leq1\right) \\
&  \Longleftrightarrow\ \left(  \text{each }i\in\left\{  0,1,\ldots
,k-1\right\}  \text{ satisfies }\left\vert \left(  \overline{i}\cap W\right)
\setminus B\right\vert \leq1\right)
\end{align*}
(since $\overline{i}\cap\left(  W\setminus B\right)  =\left(  \overline{i}\cap
W\right)  \setminus B$ for each $i$). This proves Proposition
\ref{prop.petk.explicit-old}.
\end{proof}

\begin{noncompile}
I also have a truly ugly proof of Proposition \ref{prop.petk.explicit-old}
using the Jacobi complementary minor theorem.
\end{noncompile}

\subsection{\label{subsect.proofs.DeltaGkm}Proof of Theorem \ref{thm.DeltaGkm}%
}

\begin{proof}
[Proof of Theorem \ref{thm.DeltaGkm}.]In this proof, the word
\textquotedblleft monomial\textquotedblright\ may refer to a monomial in any
set of variables (not necessarily in $x_{1},x_{2},x_{3},\ldots$).

In the following, an $i$\emph{-monomial} (where $i\in\mathbb{N}$) shall mean a
monomial of degree $i$.

We shall say that a monomial is $k$\emph{-bounded} if all exponents in this
monomial are $<k$. In other words, a monomial is $k$\emph{-bounded} if it can
be written in the form $z_{1}^{a_{1}}z_{2}^{a_{2}}\cdots z_{s}^{a_{s}}$, where
$z_{1},z_{2},\ldots,z_{s}$ are distinct variables and $a_{1},a_{2}%
,\ldots,a_{s}$ are nonnegative integers $<k$. Thus, the $k$-bounded monomials
in the variables $x_{1},x_{2},x_{3},\ldots$ are precisely the monomials of the
form $\mathbf{x}^{\alpha}$ for $\alpha\in\operatorname*{WC}$ satisfying
$\left(  \alpha_{i}<k\text{ for all }i\right)  $. Hence, the $k$-bounded
$m$-monomials in the variables $x_{1},x_{2},x_{3},\ldots$ are precisely the
monomials of the form $\mathbf{x}^{\alpha}$ for $\alpha\in\operatorname*{WC}$
satisfying $\left\vert \alpha\right\vert =m$ and $\left(  \alpha_{i}<k\text{
for all }i\right)  $.

Now, the definition of $G\left(  k,m\right)  $ yields%
\begin{align}
G\left(  k,m\right)   &  =\sum_{\substack{\alpha\in\operatorname*{WC}%
;\\\left\vert \alpha\right\vert =m;\\\alpha_{i}<k\text{ for all }i}%
}\mathbf{x}^{\alpha}\nonumber\\
&  =\left(  \text{the sum of all }k\text{-bounded }m\text{-monomials in the
variables }x_{1},x_{2},x_{3},\ldots\right)  \label{pf.thm.DeltaGkm.Gkm=}%
\end{align}
(since the $k$-bounded $m$-monomials in the variables $x_{1},x_{2}%
,x_{3},\ldots$ are precisely the monomials of the form $\mathbf{x}^{\alpha}$
for $\alpha\in\operatorname*{WC}$ satisfying $\left\vert \alpha\right\vert =m$
and $\left(  \alpha_{i}<k\text{ for all }i\right)  $).

Let us now substitute the variables $x_{1},x_{2},x_{3},\ldots,y_{1}%
,y_{2},y_{3},\ldots$ for the variables $x_{1},x_{2},x_{3},\ldots$ on both
sides of the equality (\ref{pf.thm.DeltaGkm.Gkm=}). (This means that we choose
some bijection $\phi:\left\{  x_{1},x_{2},x_{3},\ldots\right\}  \rightarrow
\left\{  x_{1},x_{2},x_{3},\ldots,y_{1},y_{2},y_{3},\ldots\right\}  $, and
substitute $\phi\left(  x_{i}\right)  $ for each $x_{i}$ on both sides of
(\ref{pf.thm.DeltaGkm.Gkm=}).) The left hand side of
(\ref{pf.thm.DeltaGkm.Gkm=}) turns into $\left(  G\left(  k,m\right)  \right)
\left(  \mathbf{x},\mathbf{y}\right)  $ upon this
substitution\footnote{because this is how $\left(  G\left(  k,m\right)
\right)  \left(  \mathbf{x},\mathbf{y}\right)  $ was defined}, whereas the
right hand side turns into%
\[
\left(  \text{the sum of all }k\text{-bounded }m\text{-monomials in the
variables }x_{1},x_{2},x_{3},\ldots,y_{1},y_{2},y_{3},\ldots\right)
\]
\footnote{Indeed, the substitution can be regarded as simply renaming the
variables $x_{1},x_{2},x_{3},\ldots$ as $x_{1},x_{2},x_{3},\ldots,y_{1}%
,y_{2},y_{3},\ldots$ (in some order). Thus, it turns the $k$-bounded
$m$-monomials in the variables $x_{1},x_{2},x_{3},\ldots$ into the $k$-bounded
$m$-monomials in the variables $x_{1},x_{2},x_{3},\ldots,y_{1},y_{2}%
,y_{3},\ldots$.}. Thus, our substitution transforms the equality
(\ref{pf.thm.DeltaGkm.Gkm=}) into%
\begin{align}
&  \left(  G\left(  k,m\right)  \right)  \left(  \mathbf{x},\mathbf{y}\right)
\nonumber\\
&  =\left(  \text{the sum of all }k\text{-bounded }m\text{-monomials in the
variables }x_{1},x_{2},x_{3},\ldots,y_{1},y_{2},y_{3},\ldots\right)  .
\label{pf.thm.DeltaGkm.Gkmxy=}%
\end{align}

But any monomial $\mathfrak{m}$ in the variables $x_{1},x_{2},x_{3}%
,\ldots,y_{1},y_{2},y_{3},\ldots$ can be uniquely written as a product
$\mathfrak{np}$, where $\mathfrak{n}$ is a monomial in the variables
$x_{1},x_{2},x_{3},\ldots$ and where $\mathfrak{p}$ is a monomial in the
variables $y_{1},y_{2},y_{3},\ldots$. Moreover, if $\mathfrak{m}$ is written
in this form, then:

\begin{itemize}
\item the degree of $\mathfrak{m}$ equals the sum of the degrees of
$\mathfrak{n}$ and $\mathfrak{p}$;

\item thus, $\mathfrak{m}$ is an $m$-monomial if and only if there exists some
$i\in\left\{  0,1,\ldots,m\right\}  $ such that $\mathfrak{n}$ is an
$i$-monomial and $\mathfrak{p}$ is an $\left(  m-i\right)  $-monomial;

\item furthermore, $\mathfrak{m}$ is $k$-bounded if and only if both
$\mathfrak{n}$ and $\mathfrak{p}$ are $k$-bounded.
\end{itemize}

Thus, any $k$-bounded $m$-monomial $\mathfrak{m}$ in the variables
$x_{1},x_{2},x_{3},\ldots,y_{1},y_{2},y_{3},\ldots$ can be uniquely written as
a product $\mathfrak{np}$, where $i\in\left\{  0,1,\ldots,m\right\}  $, where
$\mathfrak{n}$ is a $k$-bounded $i$-monomial in the variables $x_{1}%
,x_{2},x_{3},\ldots$ and where $\mathfrak{p}$ is a $k$-bounded $\left(
m-i\right)  $-monomial in the variables $y_{1},y_{2},y_{3},\ldots$.
Conversely, every such product $\mathfrak{np}$ is a $k$-bounded $m$-monomial
$\mathfrak{m}$ in the variables $x_{1},x_{2},x_{3},\ldots,y_{1},y_{2}%
,y_{3},\ldots$. Thus, we obtain a bijection%
\begin{align*}
&  \bigsqcup_{i\in\left\{  0,1,\ldots,m\right\}  }\left(
\vphantom{\frac12}\left\{  k\text{-bounded }i\text{-monomials in the variables
}x_{1},x_{2},x_{3},\ldots\right\}  \right. \\
&  \ \ \ \ \ \ \ \ \ \ \left.  \vphantom{\frac12}\times\left\{
k\text{-bounded }\left(  m-i\right)  \text{-monomials in the variables }%
y_{1},y_{2},y_{3},\ldots\right\}  \right) \\
&  \rightarrow\left\{  k\text{-bounded }m\text{-monomials in the variables
}x_{1},x_{2},x_{3},\ldots,y_{1},y_{2},y_{3},\ldots\right\}
\end{align*}
that sends each pair $\left(  \mathfrak{n},\mathfrak{p}\right)  $ to
$\mathfrak{np}$. Hence,%
\begin{align}
&  \left(  \text{the sum of all }k\text{-bounded }m\text{-monomials in the
variables }x_{1},x_{2},x_{3},\ldots,y_{1},y_{2},y_{3},\ldots\right)
\nonumber\\
&  =\sum_{i\in\left\{  0,1,\ldots,m\right\}  }\ \ \sum_{\substack{\mathfrak{n}%
\text{ is a }k\text{-bounded}\\i\text{-monomial in the}\\\text{variables
}x_{1},x_{2},x_{3},\ldots}}\ \ \sum_{\substack{\mathfrak{p}\text{ is a
}k\text{-bounded}\\\left(  m-i\right)  \text{-monomial in the}%
\\\text{variables }y_{1},y_{2},y_{3},\ldots}}\mathfrak{np}\nonumber\\
&  =\sum_{i\in\left\{  0,1,\ldots,m\right\}  }\left(  \sum
_{\substack{\mathfrak{n}\text{ is a }k\text{-bounded}\\i\text{-monomial in
the}\\\text{variables }x_{1},x_{2},x_{3},\ldots}}\mathfrak{n}\right)  \left(
\sum_{\substack{\mathfrak{p}\text{ is a }k\text{-bounded}\\\left(  m-i\right)
\text{-monomial in the}\\\text{variables }y_{1},y_{2},y_{3},\ldots
}}\mathfrak{p}\right) \nonumber\\
&  =\sum_{i\in\left\{  0,1,\ldots,m\right\}  }\left(  \text{the sum of all
}k\text{-bounded }i\text{-monomials in the variables }x_{1},x_{2},x_{3}%
,\ldots\right) \nonumber\\
&  \ \ \ \ \ \ \ \ \ \ \cdot\left(  \text{the sum of all }k\text{-bounded
}\left(  m-i\right)  \text{-monomials in the variables }y_{1},y_{2}%
,y_{3},\ldots\right)  . \label{pf.thm.DeltaGkm.prod1}%
\end{align}

Now, let $i\in\left\{  0,1,\ldots,m\right\}  $. The same reasoning that gave
us (\ref{pf.thm.DeltaGkm.Gkm=}) can be applied to $i$ instead of $m$. Thus we
obtain%
\begin{equation}
G\left(  k,i\right)  =\left(  \text{the sum of all }k\text{-bounded
}i\text{-monomials in the variables }x_{1},x_{2},x_{3},\ldots\right)  .
\label{pf.thm.DeltaGkm.Gki=}%
\end{equation}

Also, $i\in\left\{  0,1,\ldots,m\right\}  $, so that $m-i\in\left\{
0,1,\ldots,m\right\}  \subseteq\mathbb{N}$. Hence, the same reasoning that
gave us (\ref{pf.thm.DeltaGkm.Gkm=}) can be applied to $m-i$ instead of $m$.
Thus we obtain%
\begin{align*}
&  G\left(  k,m-i\right) \\
&  =\left(  \text{the sum of all }k\text{-bounded }\left(  m-i\right)
\text{-monomials in the variables }x_{1},x_{2},x_{3},\ldots\right)  .
\end{align*}
Renaming the variables $x_{1},x_{2},x_{3},\ldots$ as $y_{1},y_{2},y_{3}%
,\ldots$ in this equality, we obtain%
\begin{align}
&  \left(  G\left(  k,m-i\right)  \right)  \left(  \mathbf{y}\right)
\nonumber\\
&  =\left(  \text{the sum of all }k\text{-bounded }\left(  m-i\right)
\text{-monomials in the variables }y_{1},y_{2},y_{3},\ldots\right)  .
\label{pf.thm.DeltaGkm.Gkm-iy=}%
\end{align}

Forget that we fixed $i$. We thus have proved (\ref{pf.thm.DeltaGkm.Gki=}) and
(\ref{pf.thm.DeltaGkm.Gkm-iy=}) for each $i\in\left\{  0,1,\ldots,m\right\}  $.

Now, (\ref{pf.thm.DeltaGkm.Gkmxy=}) becomes%
\begin{align*}
&  \left(  G\left(  k,m\right)  \right)  \left(  \mathbf{x},\mathbf{y}\right)
\\
&  =\left(  \text{the sum of all }k\text{-bounded }m\text{-monomials in the
variables }x_{1},x_{2},x_{3},\ldots,y_{1},y_{2},y_{3},\ldots\right) \\
&  =\sum_{i\in\left\{  0,1,\ldots,m\right\}  }\underbrace{\left(  \text{the
sum of all }k\text{-bounded }i\text{-monomials in the variables }x_{1}%
,x_{2},x_{3},\ldots\right)  }_{\substack{=G\left(  k,i\right)  \\\text{(by
(\ref{pf.thm.DeltaGkm.Gki=}))}}}\\
&  \ \ \ \ \ \ \ \ \ \ \cdot\underbrace{\left(  \text{the sum of all
}k\text{-bounded }\left(  m-i\right)  \text{-monomials in the variables }%
y_{1},y_{2},y_{3},\ldots\right)  }_{\substack{=\left(  G\left(  k,m-i\right)
\right)  \left(  \mathbf{y}\right)  \\\text{(by (\ref{pf.thm.DeltaGkm.Gkm-iy=}%
))}}}\\
&  \ \ \ \ \ \ \ \ \ \ \ \ \ \ \ \ \ \ \ \ \left(  \text{by
(\ref{pf.thm.DeltaGkm.prod1})}\right) \\
&  =\sum_{i\in\left\{  0,1,\ldots,m\right\}  }\underbrace{G\left(  k,i\right)
}_{=\left(  G\left(  k,i\right)  \right)  \left(  \mathbf{x}\right)  }%
\cdot\left(  G\left(  k,m-i\right)  \right)  \left(  \mathbf{y}\right)
=\sum_{i\in\left\{  0,1,\ldots,m\right\}  }\left(  G\left(  k,i\right)
\right)  \left(  \mathbf{x}\right)  \cdot\left(  G\left(  k,m-i\right)
\right)  \left(  \mathbf{y}\right)  .
\end{align*}
Hence, (\ref{eq.Delta-on-Lam.if}) holds for $f=G\left(  k,m\right)  $,
$I=\left\{  0,1,\ldots,m\right\}  $, $\left(  f_{1,i}\right)  _{i\in
I}=\left(  G\left(  k,i\right)  \right)  _{i\in\left\{  0,1,\ldots,m\right\}
}$ and $\left(  f_{2,i}\right)  _{i\in I}=\left(  G\left(  k,m-i\right)
\right)  _{i\in\left\{  0,1,\ldots,m\right\}  }$. Therefore,
(\ref{eq.Delta-on-Lam.then}) (applied to these $f$, $I$, $\left(
f_{1,i}\right)  _{i\in I}$ and $\left(  f_{2,i}\right)  _{i\in I}$) yields%
\[
\Delta\left(  G\left(  k,m\right)  \right)  =\sum_{i\in\left\{  0,1,\ldots
,m\right\}  }G\left(  k,i\right)  \otimes G\left(  k,m-i\right)  =\sum
_{i=0}^{m}G\left(  k,i\right)  \otimes G\left(  k,m-i\right)  .
\]
This proves Theorem \ref{thm.DeltaGkm}.
\end{proof}

\subsection{\label{subsect.proofs.G.frob}Proof of Theorem \ref{thm.G.frob}}

\begin{proof}
[Proof of Theorem \ref{thm.G.frob}.]Consider the ring $\left(  \mathbf{k}%
\left[  \left[  x_{1},x_{2},x_{3},\ldots\right]  \right]  \right)  \left[
\left[  t\right]  \right]  $ of formal power series in one indeterminate $t$
over $\mathbf{k}\left[  \left[  x_{1},x_{2},x_{3},\ldots\right]  \right]  $.
We equip this ring with the topology that is obtained by identifying it with
$\mathbf{k}\left[  \left[  x_{1},x_{2},x_{3},\ldots,t\right]  \right]  $ (or,
equivalently, which is obtained by considering $\mathbf{k}\left[  \left[
x_{1},x_{2},x_{3},\ldots\right]  \right]  $ itself as equipped with the
standard topology on a ring of formal power series, and then adjoining the
extra indeterminate $t$).

Now, consider the map%
\begin{align*}
\mathbf{F}_{k}:\mathbf{k}\left[  \left[  x_{1},x_{2},x_{3},\ldots\right]
\right]   &  \rightarrow\mathbf{k}\left[  \left[  x_{1},x_{2},x_{3}%
,\ldots\right]  \right]  ,\\
a  &  \mapsto a\left(  x_{1}^{k},x_{2}^{k},x_{3}^{k},\ldots\right)  .
\end{align*}
This map $\mathbf{F}_{k}$ is a continuous $\mathbf{k}$-algebra homomorphism
(since it is an evaluation homomorphism)\footnote{It is well-defined, since
$k$ is positive.}. Hence, it induces a continuous\footnote{Continuity is
defined with respect to the topology that we defined on $\left(
\mathbf{k}\left[  \left[  x_{1},x_{2},x_{3},\ldots\right]  \right]  \right)
\left[  \left[  t\right]  \right]  $.} $\mathbf{k}\left[  \left[  t\right]
\right]  $-algebra homomorphism%
\[
\mathbf{F}_{k}\left[  \left[  t\right]  \right]  :\left(  \mathbf{k}\left[
\left[  x_{1},x_{2},x_{3},\ldots\right]  \right]  \right)  \left[  \left[
t\right]  \right]  \rightarrow\left(  \mathbf{k}\left[  \left[  x_{1}%
,x_{2},x_{3},\ldots\right]  \right]  \right)  \left[  \left[  t\right]
\right]
\]
that sends each formal power series $\sum_{n\geq0}a_{n}t^{n}\in\left(
\mathbf{k}\left[  \left[  x_{1},x_{2},x_{3},\ldots\right]  \right]  \right)
\left[  \left[  t\right]  \right]  $ (with $a_{n}\in\mathbf{k}\left[  \left[
x_{1},x_{2},x_{3},\ldots\right]  \right]  $) to $\sum_{n\geq0}\mathbf{F}%
_{k}\left(  a_{n}\right)  t^{n}$. Consider this $\mathbf{k}\left[  \left[
t\right]  \right]  $-algebra homomorphism $\mathbf{F}_{k}\left[  \left[
t\right]  \right]  $. In particular, it satisfies
\[
\left(  \mathbf{F}_{k}\left[  \left[  t\right]  \right]  \right)  \left(
t^{i}\right)  =t^{i}\ \ \ \ \ \ \ \ \ \ \text{for each }i\in\mathbb{N}.
\]

The definition of $\mathbf{F}_{k}$ yields%
\begin{equation}
\mathbf{F}_{k}\left(  x_{i}\right)  =x_{i}^{k}\ \ \ \ \ \ \ \ \ \ \text{for
each }i\in\left\{  1,2,3,\ldots\right\}  . \label{pf.thm.G.frob.Fkxi}%
\end{equation}
Also, for each $a\in\Lambda$, we have%
\begin{align}
\mathbf{F}_{k}\left(  a\right)   &  =a\left(  x_{1}^{k},x_{2}^{k},x_{3}%
^{k},\ldots\right)  \ \ \ \ \ \ \ \ \ \ \left(  \text{by the definition of
}\mathbf{F}_{k}\right) \nonumber\\
&  =\mathbf{f}_{k}\left(  a\right)  \label{pf.thm.G.frob.Fka}%
\end{align}
(since the definition of $\mathbf{f}_{k}$ yields $\mathbf{f}_{k}\left(
a\right)  =a\left(  x_{1}^{k},x_{2}^{k},x_{3}^{k},\ldots\right)  $). Thus, in
particular, each $n\in\mathbb{N}$ satisfies%
\begin{equation}
\mathbf{F}_{k}\left(  e_{n}\right)  =\mathbf{f}_{k}\left(  e_{n}\right)
\label{pf.thm.G.frob.Fken}%
\end{equation}
(by (\ref{pf.thm.G.frob.Fka}), applied to $a=e_{n}$).

Applying the map $\mathbf{F}_{k}\left[  \left[  t\right]  \right]  $ to both
sides of the equality (\ref{pf.lem.h-e-reciprocity-t.e}), we obtain%
\[
\left(  \mathbf{F}_{k}\left[  \left[  t\right]  \right]  \right)  \left(
\prod_{i=1}^{\infty}\left(  1+x_{i}t\right)  \right)  =\left(  \mathbf{F}%
_{k}\left[  \left[  t\right]  \right]  \right)  \left(  \sum_{n\geq0}%
e_{n}t^{n}\right)  =\sum_{n\geq0}\mathbf{F}_{k}\left(  e_{n}\right)  t^{n}%
\]
(by the definition of $\mathbf{F}_{k}\left[  \left[  t\right]  \right]  $).
Hence,%
\begin{align*}
\sum_{n\geq0}\mathbf{F}_{k}\left(  e_{n}\right)  t^{n} &  =\left(
\mathbf{F}_{k}\left[  \left[  t\right]  \right]  \right)  \left(  \prod
_{i=1}^{\infty}\left(  1+x_{i}t\right)  \right)  =\prod_{i=1}^{\infty
}\underbrace{\left(  \mathbf{F}_{k}\left[  \left[  t\right]  \right]  \right)
\left(  1+x_{i}t\right)  }_{\substack{=1+\mathbf{F}_{k}\left(  x_{i}\right)
t\\\text{(by the definition of }\mathbf{F}_{k}\left[  \left[  t\right]
\right]  \text{)}}}\\
&  \ \ \ \ \ \ \ \ \ \ \left(
\begin{array}
[c]{c}%
\text{since }\mathbf{F}_{k}\left[  \left[  t\right]  \right]  \text{ is a
continuous }\mathbf{k}\left[  \left[  t\right]  \right]  \text{-algebra
homomorphism,}\\
\text{and thus respects infinite products}%
\end{array}
\right)  \\
&  =\prod_{i=1}^{\infty}\left(  1+\underbrace{\mathbf{F}_{k}\left(
x_{i}\right)  }_{\substack{=x_{i}^{k}\\\text{(by (\ref{pf.thm.G.frob.Fkxi}))}%
}}t\right)  =\prod_{i=1}^{\infty}\left(  1+x_{i}^{k}t\right)  .
\end{align*}
Substituting $-t^{k}$ for $t$ in this equality, we find%
\begin{align*}
\sum_{n\geq0}\mathbf{F}_{k}\left(  e_{n}\right)  \left(  -t^{k}\right)  ^{n}
&  =\prod_{i=1}^{\infty}\underbrace{\left(  1+x_{i}^{k}\left(  -t^{k}\right)
\right)  }_{\substack{=1-\left(  x_{i}t\right)  ^{k}\\=\left(  1-x_{i}%
t\right)  \left(  \left(  x_{i}t\right)  ^{0}+\left(  x_{i}t\right)
^{1}+\cdots+\left(  x_{i}t\right)  ^{k-1}\right)  \\\text{(since }%
1-u^{k}=\left(  1-u\right)  \left(  u^{0}+u^{1}+\cdots+u^{k-1}\right)
\\\text{for any element }u\text{ of any ring)}}}\\
&  =\prod_{i=1}^{\infty}\left(  \left(  1-x_{i}t\right)  \left(  \left(
x_{i}t\right)  ^{0}+\left(  x_{i}t\right)  ^{1}+\cdots+\left(  x_{i}t\right)
^{k-1}\right)  \right)  \\
&  =\left(  \prod_{i=1}^{\infty}\left(  1-x_{i}t\right)  \right)  \left(
\prod_{i=1}^{\infty}\left(  \left(  x_{i}t\right)  ^{0}+\left(  x_{i}t\right)
^{1}+\cdots+\left(  x_{i}t\right)  ^{k-1}\right)  \right)  .
\end{align*}
We can divide both sides of this equality by $\prod_{i=1}^{\infty}\left(
1-x_{i}t\right)  $ (since the formal power series $\prod_{i=1}^{\infty}\left(
1-x_{i}t\right)  $ has constant term $1$ and thus is invertible), and thus
obtain%
\begin{align*}
\dfrac{\sum_{n\geq0}\mathbf{F}_{k}\left(  e_{n}\right)  \left(  -t^{k}\right)
^{n}}{\prod_{i=1}^{\infty}\left(  1-x_{i}t\right)  } &  =\prod_{i=1}^{\infty
}\underbrace{\left(  \left(  x_{i}t\right)  ^{0}+\left(  x_{i}t\right)
^{1}+\cdots+\left(  x_{i}t\right)  ^{k-1}\right)  }_{=\sum_{u=0}^{k-1}\left(
x_{i}t\right)  ^{u}}=\prod_{i=1}^{\infty}\ \ \sum_{u=0}^{k-1}\left(
x_{i}t\right)  ^{u}\\
&  =\underbrace{\sum_{\substack{\alpha=\left(  \alpha_{1},\alpha_{2}%
,\alpha_{3},\ldots\right)  \in\left\{  0,1,\ldots,k-1\right\}  ^{\infty
};\\\alpha_{i}=0\text{ for all but finitely many }i}}}_{\substack{=\sum
_{\substack{\alpha\in\left\{  0,1,\ldots,k-1\right\}  ^{\infty}\\\text{is a
weak composition}}}\\=\sum_{\substack{\alpha\in\operatorname*{WC};\\\alpha
_{i}<k\text{ for all }i}}}}\ \ \underbrace{\left(  x_{1}t\right)  ^{\alpha
_{1}}\left(  x_{2}t\right)  ^{\alpha_{2}}\left(  x_{3}t\right)  ^{\alpha_{3}%
}\cdots}_{\substack{=\left(  x_{1}^{\alpha_{1}}x_{2}^{\alpha_{2}}x_{3}%
^{\alpha_{3}}\cdots\right)  \left(  t^{\alpha_{1}}t^{\alpha_{2}}t^{\alpha_{3}%
}\cdots\right)  }}\\
&  \ \ \ \ \ \ \ \ \ \ \ \ \ \ \ \ \ \ \ \ \left(  \text{here, we have
expanded the product}\right)  \\
&  =\sum_{\substack{\alpha\in\operatorname*{WC};\\\alpha_{i}<k\text{ for all
}i}}\ \ \underbrace{\left(  x_{1}^{\alpha_{1}}x_{2}^{\alpha_{2}}x_{3}%
^{\alpha_{3}}\cdots\right)  }_{\substack{=\mathbf{x}^{\alpha}\\\text{(by the
definition of }\mathbf{x}^{\alpha}\text{)}}}\ \ \underbrace{\left(
t^{\alpha_{1}}t^{\alpha_{2}}t^{\alpha_{3}}\cdots\right)  }%
_{\substack{=t^{\alpha_{1}+\alpha_{2}+\alpha_{3}+\cdots}=t^{\left\vert
\alpha\right\vert }\\\text{(since }\alpha_{1}+\alpha_{2}+\alpha_{3}%
+\cdots=\left\vert \alpha\right\vert \text{)}}}\\
&  =\sum_{\substack{\alpha\in\operatorname*{WC};\\\alpha_{i}<k\text{ for all
}i}}\mathbf{x}^{\alpha}t^{\left\vert \alpha\right\vert }.
\end{align*}
Hence,%
\begin{align*}
&  \sum_{\substack{\alpha\in\operatorname*{WC};\\\alpha_{i}<k\text{ for all
}i}}\mathbf{x}^{\alpha}t^{\left\vert \alpha\right\vert }\\
&  =\dfrac{\sum_{n\geq0}\mathbf{F}_{k}\left(  e_{n}\right)  \left(
-t^{k}\right)  ^{n}}{\prod_{i=1}^{\infty}\left(  1-x_{i}t\right)  }=\left(
\sum_{n\geq0}\underbrace{\mathbf{F}_{k}\left(  e_{n}\right)  }%
_{\substack{=\mathbf{f}_{k}\left(  e_{n}\right)  \\\text{(by
(\ref{pf.thm.G.frob.Fken}))}}}\ \ \underbrace{\left(  -t^{k}\right)  ^{n}%
}_{=\left(  -1\right)  ^{n}t^{kn}}\right)  \cdot\underbrace{\prod
_{i=1}^{\infty}\left(  1-x_{i}t\right)  ^{-1}}_{\substack{=\sum_{n\geq0}%
h_{n}t^{n}\\\text{(by (\ref{pf.lem.h-e-reciprocity-t.h}))}}}\\
&  =\left(  \sum_{n\geq0}\mathbf{f}_{k}\left(  e_{n}\right)  \left(
-1\right)  ^{n}t^{kn}\right)  \cdot\underbrace{\left(  \sum_{n\geq0}h_{n}%
t^{n}\right)  }_{=\sum_{j\geq0}h_{j}t^{j}}=\left(  \sum_{n\geq0}\mathbf{f}%
_{k}\left(  e_{n}\right)  \left(  -1\right)  ^{n}t^{kn}\right)  \cdot\left(
\sum_{j\geq0}h_{j}t^{j}\right)  \\
&  =\underbrace{\sum_{n\geq0}\sum_{j\geq0}}_{=\sum_{\left(  n,j\right)
\in\mathbb{N}^{2}}}\mathbf{f}_{k}\left(  e_{n}\right)  \left(  -1\right)
^{n}\underbrace{t^{kn}h_{j}t^{j}}_{=h_{j}t^{kn+j}}=\sum_{\left(  n,j\right)
\in\mathbb{N}^{2}}\mathbf{f}_{k}\left(  e_{n}\right)  \left(  -1\right)
^{n}h_{j}t^{kn+j}.
\end{align*}

\begin{vershort}
\noindent This is an equality between two power series in $\left(
\mathbf{k}\left[  \left[  x_{1},x_{2},x_{3},\ldots\right]  \right]  \right)
\left[  \left[  t\right]  \right]  $. If we compare the coefficients of
$t^{m}$ on both sides of it (where $x_{1},x_{2},x_{3},\ldots$ are considered
scalars, not monomials), we obtain%
\[
\sum_{\substack{\alpha\in\operatorname*{WC};\\\alpha_{i}<k\text{ for all
}i;\\\left\vert \alpha\right\vert =m}}\mathbf{x}^{\alpha}=\underbrace{\sum
_{\substack{\left(  n,j\right)  \in\mathbb{N}^{2};\\kn+j=m}}}_{=\sum
_{n\in\mathbb{N}}\ \ \sum_{\substack{j\in\mathbb{N};\\kn+j=m}}}\mathbf{f}%
_{k}\left(  e_{n}\right)  \left(  -1\right)  ^{n}h_{j}=\sum_{n\in\mathbb{N}%
}\mathbf{f}_{k}\left(  e_{n}\right)  \left(  -1\right)  ^{n}\cdot
\sum_{\substack{j\in\mathbb{N};\\kn+j=m}}h_{j}.
\]

\end{vershort}

\begin{verlong}
\noindent This is an equality between two power series in $\left(
\mathbf{k}\left[  \left[  x_{1},x_{2},x_{3},\ldots\right]  \right]  \right)
\left[  \left[  t\right]  \right]  $. If we compare the coefficients of
$t^{m}$ on both sides of it (where $x_{1},x_{2},x_{3},\ldots$ are considered
scalars, not monomials), we obtain%
\begin{align*}
\sum_{\substack{\alpha\in\operatorname*{WC};\\\alpha_{i}<k\text{ for all
}i;\\\left\vert \alpha\right\vert =m}}\mathbf{x}^{\alpha}  &
=\underbrace{\sum_{\substack{\left(  n,j\right)  \in\mathbb{N}^{2};\\kn+j=m}%
}}_{=\sum_{n\in\mathbb{N}}\ \ \sum_{\substack{j\in\mathbb{N};\\kn+j=m}%
}}\mathbf{f}_{k}\left(  e_{n}\right)  \left(  -1\right)  ^{n}h_{j}\\
&  =\sum_{n\in\mathbb{N}}\ \ \sum_{\substack{j\in\mathbb{N};\\kn+j=m}%
}\mathbf{f}_{k}\left(  e_{n}\right)  \left(  -1\right)  ^{n}h_{j}=\sum
_{n\in\mathbb{N}}\mathbf{f}_{k}\left(  e_{n}\right)  \left(  -1\right)
^{n}\cdot\sum_{\substack{j\in\mathbb{N};\\kn+j=m}}h_{j}.
\end{align*}

\end{verlong}

\noindent Now, the definition of $G\left(  k,m\right)  $ yields%
\begin{align}
G\left(  k,m\right)   &  =\sum_{\substack{\alpha\in\operatorname*{WC}%
;\\\left\vert \alpha\right\vert =m;\\\alpha_{i}<k\text{ for all }i}%
}\mathbf{x}^{\alpha}=\sum_{\substack{\alpha\in\operatorname*{WC};\\\alpha
_{i}<k\text{ for all }i;\\\left\vert \alpha\right\vert =m}}\mathbf{x}^{\alpha
}\nonumber\\
&  =\sum_{n\in\mathbb{N}}\mathbf{f}_{k}\left(  e_{n}\right)  \left(
-1\right)  ^{n}\cdot\sum_{\substack{j\in\mathbb{N};\\kn+j=m}}h_{j}.
\label{pf.thm.G.frob.15}%
\end{align}

But the right hand side of this equality can be simplified. Namely, for each
$n\in\mathbb{N}$, we have%
\begin{equation}
\sum_{\substack{j\in\mathbb{N};\\kn+j=m}}h_{j}=h_{m-kn}.
\label{pf.thm.G.frob.16}%
\end{equation}

\begin{vershort}
[\textit{Proof of (\ref{pf.thm.G.frob.16}):} Let $n\in\mathbb{N}$. We must
prove the equality (\ref{pf.thm.G.frob.16}). If $m-kn<0$, then $h_{m-kn}=0$
and $\sum_{\substack{j\in\mathbb{N};\\kn+j=m}}h_{j}=\left(  \text{empty
sum}\right)  =0$. Thus, if $m-kn<0$, then (\ref{pf.thm.G.frob.16}) boils down
to $0=0$, which is obviously true. Therefore, for the rest of the proof of
(\ref{pf.thm.G.frob.16}), we WLOG assume that $m-kn\geq0$. Hence, the sum
$\sum_{\substack{j\in\mathbb{N};\\kn+j=m}}h_{j}$ has exactly one addend,
namely the addend for $j=m-kn$. Therefore, $\sum_{\substack{j\in
\mathbb{N};\\kn+j=m}}h_{j}=h_{m-kn}$. This proves (\ref{pf.thm.G.frob.16}).]
\end{vershort}

\begin{verlong}
[\textit{Proof of (\ref{pf.thm.G.frob.16}):} Let $n\in\mathbb{N}$. We must
prove the equality (\ref{pf.thm.G.frob.16}). If $m-kn<0$, then%
\begin{align*}
\sum_{\substack{j\in\mathbb{N};\\kn+j=m}}h_{j} &  =\left(  \text{empty
sum}\right)  \ \ \ \ \ \ \ \ \ \ \left(
\begin{array}
[c]{c}%
\text{since there exists no }j\in\mathbb{N}\text{ such that }kn+j=m\\
\text{(because }m<kn\text{ (since }m-kn<0\text{))}%
\end{array}
\right)  \\
&  =0=h_{m-kn}\ \ \ \ \ \ \ \ \ \ \left(  \text{since }h_{m-kn}=0\text{
(because }m-kn<0\text{)}\right)
\end{align*}
Hence, if $m-kn<0$, then (\ref{pf.thm.G.frob.16}) is proven. Therefore, for
the rest of the proof of (\ref{pf.thm.G.frob.16}), we WLOG assume that
$m-kn\geq0$. Thus, $m-kn\in\mathbb{N}$. Hence, there exists exactly one
$j\in\mathbb{N}$ satisfying $kn+j=m$, namely $j=m-kn$. Thus, the sum
$\sum_{\substack{j\in\mathbb{N};\\kn+j=m}}h_{j}$ has exactly one addend,
namely the addend for $j=m-kn$. Therefore, $\sum_{\substack{j\in
\mathbb{N};\\kn+j=m}}h_{j}=h_{m-kn}$. This proves (\ref{pf.thm.G.frob.16}).]
\end{verlong}

Now, (\ref{pf.thm.G.frob.15}) becomes%
\begin{align*}
G\left(  k,m\right)   &  =\sum_{n\in\mathbb{N}}\mathbf{f}_{k}\left(
e_{n}\right)  \left(  -1\right)  ^{n}\cdot\underbrace{\sum_{\substack{j\in
\mathbb{N};\\kn+j=m}}h_{j}}_{\substack{=h_{m-kn}\\\text{(by
(\ref{pf.thm.G.frob.16}))}}}=\sum_{n\in\mathbb{N}}\mathbf{f}_{k}\left(
e_{n}\right)  \left(  -1\right)  ^{n}\cdot h_{m-kn}\\
&  =\sum_{n\in\mathbb{N}}\left(  -1\right)  ^{n}h_{m-kn}\cdot\mathbf{f}%
_{k}\left(  e_{n}\right)  =\sum_{i\in\mathbb{N}}\left(  -1\right)
^{i}h_{m-ki}\cdot\mathbf{f}_{k}\left(  e_{i}\right)
\end{align*}
(here, we have renamed the summation index $n$ as $i$). This proves Theorem
\ref{thm.G.frob}.
\end{proof}

Another proof of Theorem \ref{thm.G.frob} is sketched in a footnote in Section
\ref{sect.liu} below.

\subsection{\label{subsect.proofs.Gkm-genset}Proofs of the results from
Section \ref{subsect.thms.genset}}

We shall now prove the results from Section \ref{subsect.thms.genset}. We
begin with Lemma \ref{lem.hall-hfep}. This will rely on the Verschiebung
endomorphisms $\mathbf{v}_{n}$ introduced in Definition \ref{def.vn}, and on
Proposition \ref{prop.f-v-dual} and the equality (\ref{eq.vnpm}).

\begin{proof}
[Proof of Lemma \ref{lem.hall-hfep}.]Applying (\ref{eq.vnpm}) to $n=k$, we
obtain%
\begin{equation}
\mathbf{v}_{k}\left(  p_{m}\right)  =%
\begin{cases}
kp_{m/k}, & \text{if }k\mid m;\\
0, & \text{if }k\nmid m
\end{cases}
\ \ \ \ . \label{pf.lem.hall-hfep.1}%
\end{equation}
Applying Proposition \ref{prop.f-v-dual} to $n=k$, $a=p_{m}$ and $b=e_{j}$, we
obtain%
\begin{equation}
\left\langle p_{m},\mathbf{f}_{k}\left(  e_{j}\right)  \right\rangle
=\left\langle \mathbf{v}_{k}\left(  p_{m}\right)  ,e_{j}\right\rangle .
\label{pf.lem.hall-hfep.2}%
\end{equation}

Now, we are in one of the following three cases:

\textit{Case 1:} We have $m=kj$.

\textit{Case 2:} We have $k\nmid m$.

\textit{Case 3:} We have neither $m=kj$ nor $k\nmid m$.

Let us first consider Case 1. In this case, we have $m=kj$. Thus, $k\mid m$
(since $j\in\mathbb{N}\subseteq\mathbb{Z}$) and $m/k=j$. Hence, $j=m/k$, so
that the integer $j$ is positive (since $m$ and $k$ are positive). But
(\ref{pf.lem.hall-hfep.1}) becomes%
\begin{align*}
\mathbf{v}_{k}\left(  p_{m}\right)   &  =%
\begin{cases}
kp_{m/k}, & \text{if }k\mid m;\\
0, & \text{if }k\nmid m
\end{cases}
\ \ \ \ =kp_{m/k}\ \ \ \ \ \ \ \ \ \ \left(  \text{since }k\mid m\right) \\
&  =kp_{j}\ \ \ \ \ \ \ \ \ \ \left(  \text{since }m/k=j\right)  .
\end{align*}
Thus, (\ref{pf.lem.hall-hfep.2}) becomes%
\begin{align*}
\left\langle p_{m},\mathbf{f}_{k}\left(  e_{j}\right)  \right\rangle  &
=\left\langle \underbrace{\mathbf{v}_{k}\left(  p_{m}\right)  }_{=kp_{j}%
},e_{j}\right\rangle =\left\langle kp_{j},e_{j}\right\rangle \\
&  =\left\langle e_{j},kp_{j}\right\rangle \ \ \ \ \ \ \ \ \ \ \left(
\text{since the Hall inner product is symmetric}\right) \\
&  =k\underbrace{\left\langle e_{j},p_{j}\right\rangle }_{\substack{=\left(
-1\right)  ^{j-1}\\\text{(by Proposition \ref{prop.ejpj}, applied to
}n=j\text{)}}}=k\left(  -1\right)  ^{j-1}=\left(  -1\right)  ^{j-1}k.
\end{align*}
Comparing this with%
\[
\left(  -1\right)  ^{j-1}\underbrace{\left[  m=kj\right]  }%
_{\substack{=1\\\text{(since }m=kj\text{)}}}k=\left(  -1\right)  ^{j-1}k,
\]
we obtain $\left\langle p_{m},\mathbf{f}_{k}\left(  e_{j}\right)
\right\rangle =\left(  -1\right)  ^{j-1}\left[  m=kj\right]  k$. Thus, Lemma
\ref{lem.hall-hfep} is proven in Case 1.

\begin{vershort}
A similar (but simpler) argument can be used to prove Lemma
\ref{lem.hall-hfep} in Case 2. The main \textquotedblleft
idea\textquotedblright\ here is that $k\nmid m$ implies $m\neq kj$. The
details are left to the reader.
\end{vershort}

\begin{verlong}
Let us next consider Case 2. In this case, we have $k\nmid m$. Hence, $m\neq
kj$ (since otherwise, we would have $m=kj$, thus $k\mid m$ (since
$j\in\mathbb{N}$), contradicting $k\nmid m$). Now, (\ref{pf.lem.hall-hfep.1})
becomes%
\[
\mathbf{v}_{k}\left(  p_{m}\right)  =%
\begin{cases}
kp_{m/k}, & \text{if }k\mid m;\\
0, & \text{if }k\nmid m
\end{cases}
\ \ \ \ =0\ \ \ \ \ \ \ \ \ \ \left(  \text{since }k\nmid m\right)  .
\]
Thus, (\ref{pf.lem.hall-hfep.2}) becomes%
\[
\left\langle p_{m},\mathbf{f}_{k}\left(  e_{j}\right)  \right\rangle
=\left\langle \underbrace{\mathbf{v}_{k}\left(  p_{m}\right)  }_{=0}%
,e_{j}\right\rangle =\left\langle 0,e_{j}\right\rangle =0.
\]
Comparing this with%
\[
\left(  -1\right)  ^{j-1}\underbrace{\left[  m=kj\right]  }%
_{\substack{=0\\\text{(since }m\neq kj\text{)}}}k=0,
\]
we obtain $\left\langle p_{m},\mathbf{f}_{k}\left(  e_{j}\right)
\right\rangle =\left(  -1\right)  ^{j-1}\left[  m=kj\right]  k$. Thus, Lemma
\ref{lem.hall-hfep} is proven in Case 2.
\end{verlong}

Let us finally consider Case 3. In this case, we have neither $m=kj$ nor
$k\nmid m$. In other words, we have $m\neq kj$ and $k\mid m$. From $k\mid m$,
we conclude that $m/k$ is a positive integer\footnote{Indeed, it is positive
since $m$ and $k$ are positive.}. From $m\neq kj$, we obtain $m/k\neq j$.
Thus, the symmetric functions $p_{m/k}$ and $e_{j}$ are homogeneous of
different degrees\footnote{since $p_{m/k}$ is homogeneous of degree $m/k$,
whereas $e_{j}$ is homogeneous of degree $j$}, and therefore satisfy
$\left\langle p_{m/k},e_{j}\right\rangle =0$ (by (\ref{eq.Hall.graded}),
applied to $f=p_{m/k}$ and $g=e_{j}$).

Now, (\ref{pf.lem.hall-hfep.1}) becomes%
\[
\mathbf{v}_{k}\left(  p_{m}\right)  =%
\begin{cases}
kp_{m/k}, & \text{if }k\mid m;\\
0, & \text{if }k\nmid m
\end{cases}
\ \ \ \ =kp_{m/k}\ \ \ \ \ \ \ \ \ \ \left(  \text{since }k\mid m\right)  .
\]
Thus, (\ref{pf.lem.hall-hfep.2}) becomes%
\[
\left\langle p_{m},\mathbf{f}_{k}\left(  e_{j}\right)  \right\rangle
=\left\langle \underbrace{\mathbf{v}_{k}\left(  p_{m}\right)  }_{=kp_{m/k}%
},e_{j}\right\rangle =\left\langle kp_{m/k},e_{j}\right\rangle
=k\underbrace{\left\langle p_{m/k},e_{j}\right\rangle }_{=0}=0.
\]
Comparing this with%
\[
\left(  -1\right)  ^{j-1}\underbrace{\left[  m=kj\right]  }%
_{\substack{=0\\\text{(since }m\neq kj\text{)}}}k=0,
\]
we obtain $\left\langle p_{m},\mathbf{f}_{k}\left(  e_{j}\right)
\right\rangle =\left(  -1\right)  ^{j-1}\left[  m=kj\right]  k$. Thus, Lemma
\ref{lem.hall-hfep} is proven in Case 3.

We have thus proven Lemma \ref{lem.hall-hfep} in all three Cases 1, 2 and 3.
Thus, Lemma \ref{lem.hall-hfep} always holds.
\end{proof}

Next, let us prove a simple property of Hall inner products:

\begin{lemma}
\label{lem.hall-pmprod}Let $m$, $\alpha$ and $\beta$ be positive integers. Let
$a$ be a homogeneous symmetric function of degree $\alpha$. Let $b$ be a
homogeneous symmetric function of degree $\beta$. Then, $\left\langle
p_{m},ab\right\rangle =0$.
\end{lemma}

\begin{vershort}
\begin{proof}
[Proof of Lemma \ref{lem.hall-pmprod}.]For each $n\in\mathbb{N}$, let
$\Lambda_{n}$ denote the $n$-th homogeneous component of the graded
$\mathbf{k}$-algebra $\Lambda$. Thus, $a\in\Lambda_{\alpha}$ and $b\in
\Lambda_{\beta}$ (since $a$ and $b$ are homogeneous symmetric functions of
degrees $\alpha$ and $\beta$).

But it is known that the family $\left(  h_{\lambda}\right)  _{\lambda
\in\operatorname*{Par}}$ is a graded basis of the graded $\mathbf{k}$-module
$\Lambda$; this means that for each $n\in\mathbb{N}$, its subfamily $\left(
h_{\lambda}\right)  _{\lambda\in\operatorname*{Par}\nolimits_{n}}$ is a basis
of the $\mathbf{k}$-module $\Lambda_{n}$.\ \ \ \ \footnote{This fact appears,
e.g., in \cite[Proposition 2.4.3(j)]{GriRei}.} Applying this to $n=\alpha$, we
conclude that the subfamily $\left(  h_{\lambda}\right)  _{\lambda
\in\operatorname*{Par}\nolimits_{\alpha}}$ is a basis of $\Lambda_{\alpha}$.
Hence, $a$ is a $\mathbf{k}$-linear combination of this family $\left(
h_{\lambda}\right)  _{\lambda\in\operatorname*{Par}\nolimits_{\alpha}}$ (since
$a\in\Lambda_{\alpha}$).

We must prove the equality $\left\langle p_{m},ab\right\rangle =0$. Both sides
of this equality depend $\mathbf{k}$-linearly on $a$. Thus, in proving it, we
can WLOG assume that $a$ belongs to the family $\left(  h_{\lambda}\right)
_{\lambda\in\operatorname*{Par}\nolimits_{\alpha}}$ (because we know that $a$
is a $\mathbf{k}$-linear combination of this family). In other words, we can
WLOG assume that $a=h_{\lambda}$ for some $\lambda\in\operatorname*{Par}%
\nolimits_{\alpha}$. Assume this. For similar reasons, we can WLOG assume that
$b=h_{\mu}$ for some $\mu\in\operatorname*{Par}\nolimits_{\beta}$. Assume
this, too. Consider these two partitions $\lambda$ and $\mu$.

We have $\lambda\in\operatorname*{Par}\nolimits_{\alpha}$ and thus $\left\vert
\lambda\right\vert =\alpha>0$, so that $\lambda\neq\varnothing$. Hence, the
partition $\lambda$ has at least one part. Likewise, the partition $\mu$ has
at least one part.

Now, let $\lambda\sqcup\mu$ be the partition obtained by listing all parts of
$\lambda$ and of $\mu$ and sorting the resulting list in weakly decreasing
order.\footnote{For example: If $\lambda=\left(  5,3,2\right)  $ and
$\mu=\left(  6,4,3,1,1\right)  $, then $\lambda\sqcup\mu=\left(
6,5,4,3,3,2,1,1\right)  $.} Using Definition \ref{def.hlam}, we can easily see
that $h_{\lambda\sqcup\mu}=h_{\lambda}h_{\mu}$. Comparing this with
$\underbrace{a}_{=h_{\lambda}}\underbrace{b}_{=h_{\mu}}=h_{\lambda}h_{\mu}$,
we obtain $ab=h_{\lambda\sqcup\mu}$.

But the partition $\lambda\sqcup\mu$ has as many parts as $\lambda$ and $\mu$
have combined. Thus, the partition $\lambda\sqcup\mu$ has at least $2$ parts
(since $\lambda$ has at least one part, and $\mu$ has at least one part).
Therefore, $\lambda\sqcup\mu\neq\left(  m\right)  $ (since the partition
$\lambda\sqcup\mu$ has at least $2$ parts, while the partition $\left(
m\right)  $ has only $1$ part). Now, recall that $p_{m}=m_{\left(  m\right)
}$ (where, of course, the two \textquotedblleft$m$\textquotedblright s in
\textquotedblleft$m_{\left(  m\right)  }$\textquotedblright\ mean completely
unrelated things). Thus,%
\begin{align*}
\left\langle \underbrace{p_{m}}_{=m_{\left(  m\right)  }},\underbrace{ab}%
_{=h_{\lambda\sqcup\mu}}\right\rangle  &  =\left\langle m_{\left(  m\right)
},h_{\lambda\sqcup\mu}\right\rangle =\left\langle h_{\lambda\sqcup\mu
},m_{\left(  m\right)  }\right\rangle \ \ \ \ \ \ \ \ \ \ \left(
\begin{array}
[c]{c}%
\text{since the Hall inner product}\\
\text{is symmetric}%
\end{array}
\right) \\
&  =\delta_{\lambda\sqcup\mu,\left(  m\right)  }\ \ \ \ \ \ \ \ \ \ \left(
\text{by (\ref{eq.hlam-mlam-dual})}\right) \\
&  =0\ \ \ \ \ \ \ \ \ \ \left(  \text{since }\lambda\sqcup\mu\neq\left(
m\right)  \right)  .
\end{align*}
This proves Lemma \ref{lem.hall-pmprod}.
\end{proof}

See \cite{verlong} for a different proof of Lemma \ref{lem.hall-pmprod}, using
the graded dual $\Lambda^{o}$ of the Hopf algebra $\Lambda$ and the
primitivity of the element $p_{m} \in\Lambda$.
\end{vershort}

\begin{verlong}
We shall give two proofs of this lemma: one using (\ref{eq.hlam-mlam-dual}),
and one using Hopf-algebraic machinery.

\begin{proof}
[First proof of Lemma \ref{lem.hall-pmprod} (sketched).]For each
$n\in\mathbb{N}$, let $\Lambda_{n}$ denote the $n$-th homogeneous component of
the graded $\mathbf{k}$-algebra $\Lambda$. Thus, $a\in\Lambda_{\alpha}$ and
$b\in\Lambda_{\beta}$ (since $a$ and $b$ are homogeneous symmetric functions
of degrees $\alpha$ and $\beta$).

But it is known that the family $\left(  h_{\lambda}\right)  _{\lambda
\in\operatorname*{Par}}$ is a graded basis of the graded $\mathbf{k}$-module
$\Lambda$; this means that for each $n\in\mathbb{N}$, its subfamily $\left(
h_{\lambda}\right)  _{\lambda\in\operatorname*{Par}\nolimits_{n}}$ is a basis
of the $\mathbf{k}$-module $\Lambda_{n}$.\ \ \ \ \footnote{This fact appears,
e.g., in \cite[Proposition 2.4.3(j)]{GriRei}.} Applying this to $n=\alpha$, we
conclude that the subfamily $\left(  h_{\lambda}\right)  _{\lambda
\in\operatorname*{Par}\nolimits_{\alpha}}$ is a basis of $\Lambda_{\alpha}$.
Hence, $a$ is a $\mathbf{k}$-linear combination of this family $\left(
h_{\lambda}\right)  _{\lambda\in\operatorname*{Par}\nolimits_{\alpha}}$ (since
$a\in\Lambda_{\alpha}$).

We must prove the equality $\left\langle p_{m},ab\right\rangle =0$. Both sides
of this equality depend $\mathbf{k}$-linearly on $a$. Thus, in proving it, we
can WLOG assume that $a$ belongs to the family $\left(  h_{\lambda}\right)
_{\lambda\in\operatorname*{Par}\nolimits_{\alpha}}$ (because we know that $a$
is a $\mathbf{k}$-linear combination of this family). In other words, we can
WLOG assume that $a=h_{\lambda}$ for some $\lambda\in\operatorname*{Par}%
\nolimits_{\alpha}$. Assume this. For similar reasons, we can WLOG assume that
$b=h_{\mu}$ for some $\mu\in\operatorname*{Par}\nolimits_{\beta}$. Assume
this, too. Consider these two partitions $\lambda$ and $\mu$.

We have $\lambda\in\operatorname*{Par}\nolimits_{\alpha}$ and thus $\left\vert
\lambda\right\vert =\alpha>0$, so that $\lambda\neq\varnothing$. Hence, the
partition $\lambda$ has at least one part. Likewise, the partition $\mu$ has
at least one part.

Now, let $\lambda\sqcup\mu$ be the partition obtained by listing all parts of
$\lambda$ and of $\mu$ and sorting the resulting list in weakly decreasing
order.\footnote{For example: If $\lambda=\left(  5,3,2\right)  $ and
$\mu=\left(  6,4,3,1,1\right)  $, then $\lambda\sqcup\mu=\left(
6,5,4,3,3,2,1,1\right)  $.} Using Definition \ref{def.hlam}, we can easily see
that $h_{\lambda\sqcup\mu}=h_{\lambda}h_{\mu}$. Comparing this with
$\underbrace{a}_{=h_{\lambda}}\underbrace{b}_{=h_{\mu}}=h_{\lambda}h_{\mu}$,
we obtain $ab=h_{\lambda\sqcup\mu}$.

But the partition $\lambda\sqcup\mu$ has as many parts as $\lambda$ and $\mu$
have combined. Thus, the partition $\lambda\sqcup\mu$ has at least $2$ parts
(since $\lambda$ has at least one part, and $\mu$ has at least one part).
Therefore, $\lambda\sqcup\mu\neq\left(  m\right)  $ (since the partition
$\lambda\sqcup\mu$ has at least $2$ parts, while the partition $\left(
m\right)  $ has only $1$ part). Now, recall that $p_{m}=m_{\left(  m\right)
}$ (where, of course, the two \textquotedblleft$m$\textquotedblright s in
\textquotedblleft$m_{\left(  m\right)  }$\textquotedblright\ mean completely
unrelated things). Thus,%
\begin{align*}
\left\langle \underbrace{p_{m}}_{=m_{\left(  m\right)  }},\underbrace{ab}%
_{=h_{\lambda\sqcup\mu}}\right\rangle  &  =\left\langle m_{\left(  m\right)
},h_{\lambda\sqcup\mu}\right\rangle =\left\langle h_{\lambda\sqcup\mu
},m_{\left(  m\right)  }\right\rangle \\
&  \ \ \ \ \ \ \ \ \ \ \left(  \text{since the Hall inner product is
symmetric}\right) \\
&  =\delta_{\lambda\sqcup\mu,\left(  m\right)  }\ \ \ \ \ \ \ \ \ \ \left(
\begin{array}
[c]{c}%
\text{by (\ref{eq.hlam-mlam-dual}), applied to }\lambda\sqcup\mu\text{ and
}\left(  m\right) \\
\text{instead of }\lambda\text{ and }\mu
\end{array}
\right) \\
&  =0\ \ \ \ \ \ \ \ \ \ \left(  \text{since }\lambda\sqcup\mu\neq\left(
m\right)  \right)  .
\end{align*}
This proves Lemma \ref{lem.hall-pmprod}.
\end{proof}

\begin{proof}
[Second proof of Lemma \ref{lem.hall-pmprod}.]Let $\widetilde{p_{m}}$ be the
map $\Lambda\rightarrow\mathbf{k},\ g\mapsto\left\langle p_{m},g\right\rangle
$. This is a $\mathbf{k}$-linear map.

The power-sum symmetric function $p_{m}$ is primitive as an element of the
Hopf algebra $\Lambda$ (see \cite[Proposition 2.3.6(i)]{GriRei}). Now,
consider the graded dual $\Lambda^{\circ}$ of the Hopf algebra $\Lambda$ (as
defined in \cite[\S 1.6]{GriRei}). The map $\Phi:\Lambda\rightarrow
\Lambda^{\circ}$ that sends each $f\in\Lambda$ to the $\mathbf{k}$-linear map
$\Lambda\rightarrow\mathbf{k},\ g\mapsto\left\langle f,g\right\rangle $ is a
Hopf algebra isomorphism (by \cite[Corollary 2.5.14]{GriRei}). Thus, this map
$\Phi$ sends primitive elements of $\Lambda$ to primitive elements of
$\Lambda^{\circ}$. Hence, in particular, $\Phi\left(  p_{m}\right)  \in
\Lambda^{\circ}$ is primitive (since $p_{m}\in\Lambda$ is primitive). In other
words, $\widetilde{p_{m}}\in\Lambda^{\circ}$ is primitive (since the
definitions of $\Phi$ and of $\widetilde{p_{m}}$ quickly reveal that
$\Phi\left(  p_{m}\right)  =\widetilde{p_{m}}$). In other words,
$\Delta_{\Lambda^{\circ}}\left(  \widetilde{p_{m}}\right)  =1_{\Lambda^{\circ
}}\otimes\widetilde{p_{m}}+\widetilde{p_{m}}\otimes1_{\Lambda^{\circ}}$.

Consider the $\mathbf{k}$-bilinear pairing $\left\langle \cdot,\cdot
\right\rangle :\Lambda^{\circ}\times\Lambda\rightarrow\mathbf{k}$ that sends
each pair $\left(  f,a\right)  \in\Lambda^{\circ}\times\Lambda$ to $f\left(
a\right)  \in\mathbf{k}$. It induces a pairing $\left\langle \cdot
,\cdot\right\rangle :\left(  \Lambda^{\circ}\otimes\Lambda^{\circ}\right)
\times\left(  \Lambda\times\Lambda\right)  \rightarrow\mathbf{k}$ that sends
each pair $\left(  f\otimes g,a\otimes b\right)  $ to $\left\langle
f,a\right\rangle \cdot\left\langle g,b\right\rangle =f\left(  a\right)  \cdot
g\left(  b\right)  \in\mathbf{k}$. We now have defined three $\mathbf{k}%
$-bilinear forms, all of which we denote by $\left\langle \cdot,\cdot
\right\rangle $; they will be distinguished by what is inside the parentheses.

The definition of the graded dual $\Lambda^{\circ}$ yields that $1_{\Lambda
^{\circ}}$ is a homogeneous element of $\Lambda^{\circ}$ of degree $0$. Thus,
$1_{\Lambda^{\circ}}$ annihilates all homogeneous components of $\Lambda$
except for the $0$-th component. In other words, if $f\in\Lambda$ is
homogeneous of degree $\gamma$, where $\gamma\in\mathbb{N}$ is distinct from
$0$, then $\left\langle 1_{\Lambda^{\circ}},f\right\rangle =0$. Applying this
to $f=a$ and $\gamma=\alpha$, we obtain $\left\langle 1_{\Lambda^{\circ}%
},a\right\rangle =0$ (since $\alpha$ is distinct from $0$ (because $\alpha$ is
positive)). Similarly, $\left\langle 1_{\Lambda^{\circ}},b\right\rangle =0$.

Now, the definition of $\widetilde{p_{m}}$ yields $\widetilde{p_{m}}\left(
ab\right)  =\left\langle p_{m},ab\right\rangle $, so that%
\begin{align*}
\left\langle p_{m},ab\right\rangle  &  =\widetilde{p_{m}}\left(  ab\right)
=\left\langle \widetilde{p_{m}},ab\right\rangle \\
&  =\left\langle \underbrace{\Delta_{\Lambda^{\circ}}\left(  \widetilde{p_{m}%
}\right)  }_{=1_{\Lambda^{\circ}}\otimes\widetilde{p_{m}}+\widetilde{p_{m}%
}\otimes1_{\Lambda^{\circ}}},a\otimes b\right\rangle
\ \ \ \ \ \ \ \ \ \ \left(  \text{by the definition of }\Delta_{\Lambda
^{\circ}}\right) \\
&  =\left\langle 1_{\Lambda^{\circ}}\otimes\widetilde{p_{m}}+\widetilde{p_{m}%
}\otimes1_{\Lambda^{\circ}},a\otimes b\right\rangle =\underbrace{\left\langle
1_{\Lambda^{\circ}}\otimes\widetilde{p_{m}},a\otimes b\right\rangle
}_{=\left\langle 1_{\Lambda^{\circ}},a\right\rangle \cdot\left\langle
\widetilde{p_{m}},b\right\rangle }+\underbrace{\left\langle \widetilde{p_{m}%
}\otimes1_{\Lambda^{\circ}},a\otimes b\right\rangle }_{=\left\langle
\widetilde{p_{m}},a\right\rangle \cdot\left\langle 1_{\Lambda^{\circ}%
},b\right\rangle }\\
&  =\underbrace{\left\langle 1_{\Lambda^{\circ}},a\right\rangle }_{=0}%
\cdot\left\langle \widetilde{p_{m}},b\right\rangle +\left\langle
\widetilde{p_{m}},a\right\rangle \cdot\underbrace{\left\langle 1_{\Lambda
^{\circ}},b\right\rangle }_{=0}=0.
\end{align*}
This proves Lemma \ref{lem.hall-pmprod}.
\end{proof}
\end{verlong}

We can now prove Proposition \ref{prop.hall-pmGkm}:

\begin{proof}
[Proof of Proposition \ref{prop.hall-pmGkm}.]Theorem \ref{thm.G.frob} yields
$G\left(  k,m\right)  =\sum_{i\in\mathbb{N}}\left(  -1\right)  ^{i}%
h_{m-ki}\cdot\mathbf{f}_{k}\left(  e_{i}\right)  $. Hence,%
\begin{align}
\left\langle p_{m},G\left(  k,m\right)  \right\rangle  &  =\left\langle
p_{m},\sum_{i\in\mathbb{N}}\left(  -1\right)  ^{i}h_{m-ki}\cdot\mathbf{f}%
_{k}\left(  e_{i}\right)  \right\rangle \nonumber\\
&  =\sum_{i\in\mathbb{N}}\left(  -1\right)  ^{i}\left\langle p_{m}%
,h_{m-ki}\cdot\mathbf{f}_{k}\left(  e_{i}\right)  \right\rangle
\label{pf.prop.hall-pmGkm.1}%
\end{align}
(since the Hall inner product is $\mathbf{k}$-bilinear).

Now, we claim that every $i\in\mathbb{N}\setminus\left\{  0,m/k\right\}  $
satisfies%
\begin{equation}
\left\langle p_{m},h_{m-ki}\cdot\mathbf{f}_{k}\left(  e_{i}\right)
\right\rangle =0. \label{pf.prop.hall-pmGkm.zeroes}%
\end{equation}

[\textit{Proof of (\ref{pf.prop.hall-pmGkm.zeroes}):} Let $i\in\mathbb{N}%
\setminus\left\{  0,m/k\right\}  $. Thus, $i\in\mathbb{N}$ and $i\notin%
\left\{  0,m/k\right\}  $. From $i\notin\left\{  0,m/k\right\}  $, we obtain
$i\neq0$ and $i\neq m/k$. From $i\neq m/k$, we obtain $ki\neq m$, so that
$m-ki\neq0$.

We must prove the equality (\ref{pf.prop.hall-pmGkm.zeroes}). If $m-ki<0$,
then $h_{m-ki}=0$, and therefore%
\[
\left\langle p_{m},\underbrace{h_{m-ki}}_{=0}\cdot\mathbf{f}_{k}\left(
e_{i}\right)  \right\rangle =\left\langle p_{m},0\right\rangle =0.
\]
Hence, the equality (\ref{pf.prop.hall-pmGkm.zeroes}) is proven if $m-ki<0$.
Thus, for the rest of this proof, we WLOG assume that $m-ki\geq0$. Combining
this with $m-ki\neq0$, we obtain $m-ki>0$. Thus, $m-ki$ is a positive integer.
Also, $i$ is a positive integer (since $i\in\mathbb{N}$ and $i\neq0$), and
thus $ki$ is a positive integer (since $k$ is a positive integer).

The map $\mathbf{f}_{k}:\Lambda\rightarrow\Lambda$ operates by replacing each
$x_{i}$ by $x_{i}^{k}$ in a symmetric function (by the definition of
$\mathbf{f}_{k}$). Thus, if $g\in\Lambda$ is any homogeneous symmetric
function of some degree $\gamma$, then $\mathbf{f}_{k}\left(  g\right)  $ is a
homogeneous symmetric function of degree $k\gamma$. Applying this to $g=e_{i}$
and $\gamma=i$, we conclude that $\mathbf{f}_{k}\left(  e_{i}\right)  $ is a
homogeneous symmetric function of degree $ki$ (since $e_{i}$ is a homogeneous
symmetric function of degree $i$). Also, $h_{m-ki}$ is a homogeneous symmetric
function of degree $m-ki$.

Hence, Lemma \ref{lem.hall-pmprod} (applied to $\alpha=m-ki$, $a=h_{m-ki}$,
$\beta=ki$ and $b=\mathbf{f}_{k}\left(  e_{i}\right)  $) yields $\left\langle
p_{m},h_{m-ki}\cdot\mathbf{f}_{k}\left(  e_{i}\right)  \right\rangle =0$. This
proves (\ref{pf.prop.hall-pmGkm.zeroes}).]

Note that $e_{0}=1$ and thus $\mathbf{f}_{k}\left(  e_{0}\right)
=\mathbf{f}_{k}\left(  1\right)  =1$ (by the definition of $\mathbf{f}_{k}$).

Note that $m/k>0$ (since $m$ and $k$ are positive). Hence, $m/k\neq0$. Now, we
are in one of the following two cases:

\textit{Case 1:} We have $k\mid m$.

\textit{Case 2:} We have $k\nmid m$.

Let us consider Case 1 first. In this case, we have $k\mid m$. Hence, $m/k$ is
a positive integer (since $m$ and $k$ are positive integers). Thus, $0$ and
$m/k$ are two distinct elements of $\mathbb{N}$ (indeed, they are distinct
because $m/k\neq0$). Lemma \ref{lem.hall-hfep} (applied to $j=m/k$) yields%
\[
\left\langle p_{m},\mathbf{f}_{k}\left(  e_{m/k}\right)  \right\rangle
=\left(  -1\right)  ^{m/k-1}\underbrace{\left[  m=k\left(  m/k\right)
\right]  }_{\substack{=1\\\text{(since }m=k\left(  m/k\right)  \text{)}%
}}k=\left(  -1\right)  ^{m/k-1}k.
\]

Now, (\ref{pf.prop.hall-pmGkm.1}) becomes%
\begin{align*}
&  \left\langle p_{m},G\left(  k,m\right)  \right\rangle \\
&  =\sum_{i\in\mathbb{N}}\left(  -1\right)  ^{i}\left\langle p_{m}%
,h_{m-ki}\cdot\mathbf{f}_{k}\left(  e_{i}\right)  \right\rangle \\
&  =\underbrace{\left(  -1\right)  ^{0}}_{=1}\left\langle p_{m}%
,\underbrace{h_{m-k\cdot0}}_{=h_{m}}\cdot\underbrace{\mathbf{f}_{k}\left(
e_{0}\right)  }_{=1}\right\rangle +\left(  -1\right)  ^{m/k}\left\langle
p_{m},\underbrace{h_{m-k\cdot m/k}}_{\substack{=h_{0}\\\text{(since }m-k\cdot
m/k=0\text{)}}}\cdot\mathbf{f}_{k}\left(  e_{m/k}\right)  \right\rangle \\
&  \ \ \ \ \ \ \ \ \ \ +\sum_{i\in\mathbb{N}\setminus\left\{  0,m/k\right\}
}\left(  -1\right)  ^{i}\underbrace{\left\langle p_{m},h_{m-ki}\cdot
\mathbf{f}_{k}\left(  e_{i}\right)  \right\rangle }_{\substack{=0\\\text{(by
(\ref{pf.prop.hall-pmGkm.zeroes}))}}}\\
&  \ \ \ \ \ \ \ \ \ \ \ \ \ \ \ \ \ \ \ \ \left(
\begin{array}
[c]{c}%
\text{here, we have split off the addends for }i=0\text{ and for }i=m/k\\
\text{from the sum (since }0\text{ and }m/k\text{ are two distinct elements of
}\mathbb{N}\text{)}%
\end{array}
\right) \\
&  =\underbrace{\left\langle p_{m},h_{m}\right\rangle }%
_{\substack{=\left\langle h_{m},p_{m}\right\rangle \\\text{(since the Hall
inner}\\\text{product is symmetric)}}}+\left(  -1\right)  ^{m/k}\left\langle
p_{m},\underbrace{h_{0}}_{=1}\cdot\mathbf{f}_{k}\left(  e_{m/k}\right)
\right\rangle +\underbrace{\sum_{i\in\mathbb{N}\setminus\left\{
0,m/k\right\}  }\left(  -1\right)  ^{i}0}_{=0}\\
&  =\underbrace{\left\langle h_{m},p_{m}\right\rangle }%
_{\substack{=1\\\text{(by Proposition \ref{prop.hjpj},}\\\text{applied to
}n=m\text{)}}}+\left(  -1\right)  ^{m/k}\underbrace{\left\langle
p_{m},\mathbf{f}_{k}\left(  e_{m/k}\right)  \right\rangle }_{=\left(
-1\right)  ^{m/k-1}k}=1+\underbrace{\left(  -1\right)  ^{m/k}\left(
-1\right)  ^{m/k-1}}_{=-1}k=1-k.
\end{align*}
Comparing this with%
\[
1-\underbrace{\left[  k\mid m\right]  }_{\substack{=1\\\text{(since }k\mid
m\text{)}}}k=1-k,
\]
we obtain $\left\langle p_{m},G\left(  k,m\right)  \right\rangle =1-\left[
k\mid m\right]  k$. Hence, Proposition \ref{prop.hall-pmGkm} is proven in Case 1.

\begin{vershort}
Case 2 is similar to Case 1, but simpler because we no longer need to split
off the addend for $i=m/k$ from the sum (since $m/k\notin\mathbb{N}$ in this
case). We leave it to the reader.
\end{vershort}

\begin{verlong}
Let us now consider Case 2. In this case, we have $k\nmid m$. Hence,
$m/k\notin\mathbb{Z}$, so that $m/k\notin\mathbb{N}$. Thus, $\mathbb{N}%
\setminus\left\{  0\right\}  =\mathbb{N}\setminus\left\{  0,m/k\right\}  $.
Now, (\ref{pf.prop.hall-pmGkm.1}) becomes%
\begin{align*}
&  \left\langle p_{m},G\left(  k,m\right)  \right\rangle \\
&  =\sum_{i\in\mathbb{N}}\left(  -1\right)  ^{i}\left\langle p_{m}%
,h_{m-ki}\cdot\mathbf{f}_{k}\left(  e_{i}\right)  \right\rangle \\
&  =\underbrace{\left(  -1\right)  ^{0}}_{=1}\left\langle p_{m}%
,\underbrace{h_{m-k\cdot0}}_{=h_{m}}\cdot\underbrace{\mathbf{f}_{k}\left(
e_{0}\right)  }_{=1}\right\rangle +\underbrace{\sum_{i\in\mathbb{N}%
\setminus\left\{  0\right\}  }}_{\substack{=\sum_{i\in\mathbb{N}%
\setminus\left\{  0,m/k\right\}  }\\\text{(since }\mathbb{N}\setminus\left\{
0\right\}  =\mathbb{N}\setminus\left\{  0,m/k\right\}  \text{)}}}\left(
-1\right)  ^{i}\left\langle p_{m},h_{m-ki}\cdot\mathbf{f}_{k}\left(
e_{i}\right)  \right\rangle \\
&  \ \ \ \ \ \ \ \ \ \ \ \ \ \ \ \ \ \ \ \ \left(  \text{here, we have split
off the addend for }i=0\text{ from the sum}\right) \\
&  =\underbrace{\left\langle p_{m},h_{m}\right\rangle }%
_{\substack{=\left\langle h_{m},p_{m}\right\rangle \\\text{(since the Hall
inner}\\\text{product is symmetric)}}}+\sum_{i\in\mathbb{N}\setminus\left\{
0,m/k\right\}  }\left(  -1\right)  ^{i}\underbrace{\left\langle p_{m}%
,h_{m-ki}\cdot\mathbf{f}_{k}\left(  e_{i}\right)  \right\rangle }%
_{\substack{=0\\\text{(by (\ref{pf.prop.hall-pmGkm.zeroes}))}}}\\
&  =\underbrace{\left\langle h_{m},p_{m}\right\rangle }%
_{\substack{=1\\\text{(by Proposition \ref{prop.hjpj},}\\\text{applied to
}n=m\text{)}}}+\underbrace{\sum_{i\in\mathbb{N}\setminus\left\{
0,m/k\right\}  }\left(  -1\right)  ^{i}0}_{=0}=1.
\end{align*}
Comparing this with%
\[
1-\underbrace{\left[  k\mid m\right]  }_{\substack{=0\\\text{(since }k\nmid
m\text{)}}}k=1,
\]
we obtain $\left\langle p_{m},G\left(  k,m\right)  \right\rangle =1-\left[
k\mid m\right]  k$. Hence, Proposition \ref{prop.hall-pmGkm} is proven in Case 2.
\end{verlong}

We have now proven Proposition \ref{prop.hall-pmGkm} both in Case 1 and in
Case 2. Hence, Proposition \ref{prop.hall-pmGkm} always holds.
\end{proof}

Theorem \ref{thm.Gkm-genset} will follow from Proposition
\ref{prop.hall-pmGkm} using the following general criterion for generating
sets of $\Lambda$:

\begin{proposition}
\label{prop.genset-crit}For each positive integer $m$, let $v_{m}\in\Lambda$
be a homogeneous symmetric function of degree $m$.

Assume that $\left\langle p_{m},v_{m}\right\rangle $ is an invertible element
of $\mathbf{k}$ for each positive integer $m$.

Then, the family $\left(  v_{m}\right)  _{m\geq1}=\left(  v_{1},v_{2}%
,v_{3},\ldots\right)  $ is an algebraically independent generating set of the
commutative $\mathbf{k}$-algebra $\Lambda$.
\end{proposition}

\begin{proof}
[Proof of Proposition \ref{prop.genset-crit}.]Proposition
\ref{prop.genset-crit} is \cite[Exercise 2.5.24]{GriRei}.
\end{proof}

\begin{proof}
[Proof of Theorem \ref{thm.Gkm-genset}.]Let $m$ be a positive integer.
Proposition \ref{prop.hall-pmGkm} yields that%
\begin{align*}
\left\langle p_{m},G\left(  k,m\right)  \right\rangle  &
=1-\underbrace{\left[  k\mid m\right]  }_{=%
\begin{cases}
1, & \text{if }k\mid m;\\
0, & \text{if }k\nmid m
\end{cases}
}k=1-%
\begin{cases}
1, & \text{if }k\mid m;\\
0, & \text{if }k\nmid m
\end{cases}
\ \ \ \cdot k\\
&  =%
\begin{cases}
1-1\cdot k, & \text{if }k\mid m;\\
1-0\cdot k, & \text{if }k\nmid m
\end{cases}
\ \ \ \ =%
\begin{cases}
1-k, & \text{if }k\mid m;\\
1, & \text{if }k\nmid m
\end{cases}
\ \ \ \ .
\end{align*}
Hence, $\left\langle p_{m},G\left(  k,m\right)  \right\rangle $ is an
invertible element of $\mathbf{k}$ (because both $1-k$ and $1$ are invertible
elements of $\mathbf{k}$).

Forget that we fixed $m$. We thus have showed that $\left\langle
p_{m},G\left(  k,m\right)  \right\rangle $ is an invertible element of
$\mathbf{k}$ for each positive integer $m$. Also, clearly, for each positive
integer $m$, the element $G\left(  k,m\right)  \in\Lambda$ is a homogeneous
symmetric function of degree $m$. Thus, Proposition \ref{prop.genset-crit}
(applied to $v_{m}=G\left(  k,m\right)  $) shows that the family $\left(
G\left(  k,m\right)  \right)  _{m\geq1}=\left(  G\left(  k,1\right)  ,G\left(
k,2\right)  ,G\left(  k,3\right)  ,\ldots\right)  $ is an algebraically
independent generating set of the commutative $\mathbf{k}$-algebra $\Lambda$.
This proves Theorem \ref{thm.Gkm-genset}.
\end{proof}

\subsection{\label{subsect.proofs.Uk.main}Proof of Theorem \ref{thm.Uk.main}}

\begin{proof}
[Proof of Theorem \ref{thm.Uk.main}.]The $\mathbf{k}$-Hopf algebra $\Lambda$
is both commutative and cocommutative (by \cite[Exercise 2.3.7(a)]{GriRei}).

\begin{vershort}
The antipode $S$ of this Hopf algebra $\Lambda$ is a $\mathbf{k}$-Hopf algebra
homomorphism (by \cite[Proposition 2.4.3(g)]{GriRei}).
\end{vershort}

\begin{verlong}
Thus, its antipode $S$ is a $\mathbf{k}$-Hopf algebra
homomorphism\footnote{\textit{Proof.} This is actually the claim of
\cite[Proposition 2.4.3(g)]{GriRei}, but let us also give a self-contained
proof here:
\par
The antipode of a Hopf algebra is an algebra anti-endomorphism (by
\cite[Proposition 1.4.10]{GriRei}). Thus, $S$ is an algebra anti-endomorphism
(since $S$ is the antipode of the Hopf algebra $\Lambda$). But since $\Lambda$
is commutative, an algebra anti-endomorphism of $\Lambda$ is the same thing as
an algebra endomorphism of $\Lambda$ (by \cite[Exercise 1.5.8(a)]{GriRei}).
Hence, $S$ is an algebra endomorphism of $\Lambda$ (since $S$ is an algebra
anti-endomorphism of $\Lambda$).
\par
The antipode of a Hopf algebra is a coalgebra anti-endomorphism (by
\cite[Exercise 1.4.28]{GriRei}). Thus, $S$ is a coalgebra anti-endomorphism
(since $S$ is the antipode of the Hopf algebra $\Lambda$). But since $\Lambda$
is cocommutative, a coalgebra anti-endomorphism of $\Lambda$ is the same thing
as a coalgebra endomorphism of $\Lambda$ (by \cite[Exercise 1.5.8(b)]%
{GriRei}). Hence, $S$ is a coalgebra endomorphism of $\Lambda$ (since $S$ is a
coalgebra anti-endomorphism of $\Lambda$).
\par
We now know that $S$ is an algebra endomorphism of $\Lambda$ and a coalgebra
endomorphism of $\Lambda$ at the same time. In other words, $S$ is a bialgebra
endomorphism of $\Lambda$. Hence, $S$ is a $\mathbf{k}$-Hopf algebra
endomorphism of $\Lambda$. In other words, $S$ is a $\mathbf{k}$-Hopf algebra
homomorphism.}.
\end{verlong}

\textbf{(a)} The map $\mathbf{f}_{k}$ is a $\mathbf{k}$-Hopf algebra
homomorphism (by \cite[Exercise 2.9.9(d)]{GriRei}, applied to $n=k$). The map
$\mathbf{v}_{k}$ is a $\mathbf{k}$-Hopf algebra homomorphism (by
\cite[Exercise 2.9.10(e)]{GriRei}, applied to $n=k$). Thus, we have shown that
all three maps $\mathbf{f}_{k}$, $S$ and $\mathbf{v}_{k}$ are $\mathbf{k}%
$-Hopf algebra homomorphisms. Hence, their composition $\mathbf{f}_{k}\circ
S\circ\mathbf{v}_{k}$ is a $\mathbf{k}$-Hopf algebra homomorphism as well. In
other words, $U_{k}$ is a $\mathbf{k}$-Hopf algebra homomorphism (since
$U_{k}=\mathbf{f}_{k}\circ S\circ\mathbf{v}_{k}$). This proves Theorem
\ref{thm.Uk.main} \textbf{(a)}.

\textbf{(b)} Recall (from \cite[Exercise 1.5.11(a)]{GriRei}) the following fact:

\begin{statement}
\textit{Claim 1:} If $H$ is a $\mathbf{k}$-bialgebra and $A$ is a commutative
$\mathbf{k}$-algebra, then the convolution $f\star g$ of any two $\mathbf{k}%
$-algebra homomorphisms $f,g:H\rightarrow A$ is again a $\mathbf{k}$-algebra homomorphism.
\end{statement}

The following fact is dual to Claim 1:

\begin{statement}
\textit{Claim 2:} If $H$ is a $\mathbf{k}$-bialgebra and $C$ is a
cocommutative $\mathbf{k}$-coalgebra, then the convolution $f\star g$ of any
two $\mathbf{k}$-coalgebra homomorphisms $f,g:C\rightarrow H$ is again a
$\mathbf{k}$-coalgebra homomorphism.
\end{statement}

(See \cite[solution to Exercise 1.5.11(h)]{GriRei} for why exactly Claim 2 is
dual to Claim 1, and how it can be proved.)

Theorem \ref{thm.Uk.main} \textbf{(a)} yields that the map $U_{k}$ is a
$\mathbf{k}$-Hopf algebra homomorphism. Hence, $U_{k}$ is both a $\mathbf{k}%
$-algebra homomorphism and a $\mathbf{k}$-coalgebra homomorphism.

Now, recall that $\Lambda$ is commutative, and that $\operatorname*{id}%
\nolimits_{\Lambda}$ and $U_{k}$ are two $\mathbf{k}$-algebra homomorphisms
from $\Lambda$ to $\Lambda$. Hence, Claim 1 (applied to $H=\Lambda$,
$A=\Lambda$, $f=\operatorname*{id}\nolimits_{\Lambda}$ and $g=U_{k}$) shows
that the convolution $\operatorname*{id}\nolimits_{\Lambda}\star U_{k}$ is a
$\mathbf{k}$-algebra homomorphism. In other words, $V_{k}$ is a $\mathbf{k}%
$-algebra homomorphism (since $V_{k}=\operatorname*{id}\nolimits_{\Lambda
}\star U_{k}$).

Next, recall that $\Lambda$ is cocommutative, and that $\operatorname*{id}%
\nolimits_{\Lambda}$ and $U_{k}$ are two $\mathbf{k}$-coalgebra homomorphisms
from $\Lambda$ to $\Lambda$. Hence, Claim 2 (applied to $H=\Lambda$,
$C=\Lambda$, $f=\operatorname*{id}\nolimits_{\Lambda}$ and $g=U_{k}$) shows
that the convolution $\operatorname*{id}\nolimits_{\Lambda}\star U_{k}$ is a
$\mathbf{k}$-coalgebra homomorphism. In other words, $V_{k}$ is a $\mathbf{k}%
$-coalgebra homomorphism (since $V_{k}=\operatorname*{id}\nolimits_{\Lambda
}\star U_{k}$).

So we know that the map $V_{k}$ is both a $\mathbf{k}$-algebra homomorphism
and a $\mathbf{k}$-coalgebra homomorphism. Thus, $V_{k}$ is a $\mathbf{k}%
$-bialgebra homomorphism, thus a $\mathbf{k}$-Hopf algebra
homomorphism\footnote{since any $\mathbf{k}$-bialgebra homomorphism between
two $\mathbf{k}$-Hopf algebras is automatically a $\mathbf{k}$-Hopf algebra
homomorphism}. This proves Theorem \ref{thm.Uk.main} \textbf{(b)}.

\textbf{(c)} The map $\mathbf{v}_{k}$ is a $\mathbf{k}$-algebra homomorphism;
thus, $\mathbf{v}_{k}\left(  1\right)  =1$. Now, we have%
\begin{equation}
\mathbf{v}_{k}\left(  h_{m}\right)  =%
\begin{cases}
h_{m/k}, & \text{if }k\mid m;\\
0, & \text{if }k\nmid m
\end{cases}
\label{pf.thm.Uk.main.c.vh}%
\end{equation}
for each $m\in\mathbb{N}$. (Indeed, if $m>0$, then this follows from the
definition of $\mathbf{v}_{k}$. But if $m=0$, then this follows from
$\mathbf{v}_{k}\left(  1\right)  =1$, since $h_{0}=1$.)

We have%
\begin{equation}
S\left(  h_{n}\right)  =\left(  -1\right)  ^{n}e_{n}%
\ \ \ \ \ \ \ \ \ \ \text{for each }n\in\mathbb{N}.
\label{pf.thm.Uk.main.c.Sh}%
\end{equation}
(This follows from \cite[Proposition 2.4.1(iii)]{GriRei}.)

Each $i\in\mathbb{N}$ satisfies%
\begin{align}
\mathbf{v}_{k}\left(  h_{ki}\right)   &  =%
\begin{cases}
h_{ki/k}, & \text{if }k\mid ki;\\
0, & \text{if }k\nmid ki
\end{cases}
\ \ \ \ \ \ \ \ \ \ \left(  \text{by (\ref{pf.thm.Uk.main.c.vh}), applied to
}m=ki\right) \nonumber\\
&  =h_{ki/k}\ \ \ \ \ \ \ \ \ \ \left(  \text{since }k\mid ki\right)
\nonumber\\
&  =h_{i}\ \ \ \ \ \ \ \ \ \ \left(  \text{since }ki/k=i\right)
\label{pf.thm.Uk.main.c.vkhj1}%
\end{align}
and
\begin{align}
U_{k}\left(  h_{ki}\right)   &  =\left(  \mathbf{f}_{k}\circ S\circ
\mathbf{v}_{k}\right)  \left(  h_{ki}\right)  \ \ \ \ \ \ \ \ \ \ \left(
\text{since }U_{k}=\mathbf{f}_{k}\circ S\circ\mathbf{v}_{k}\right) \nonumber\\
&  =\mathbf{f}_{k}\left(  S\left(  \underbrace{\mathbf{v}_{k}\left(
h_{ki}\right)  }_{\substack{=h_{i}\\\text{(by (\ref{pf.thm.Uk.main.c.vkhj1}%
))}}}\right)  \right)  =\mathbf{f}_{k}\left(  \underbrace{S\left(
h_{i}\right)  }_{\substack{=\left(  -1\right)  ^{i}e_{i}\\\text{(by
(\ref{pf.thm.Uk.main.c.Sh}))}}}\right)  =\mathbf{f}_{k}\left(  \left(
-1\right)  ^{i}e_{i}\right) \nonumber\\
&  =\left(  -1\right)  ^{i}\mathbf{f}_{k}\left(  e_{i}\right)
\label{pf.thm.Uk.main.c.Ukhj1}%
\end{align}
(since the map $\mathbf{f}_{k}$ is $\mathbf{k}$-linear).

On the other hand, if $j\in\mathbb{N}$ satisfies $k\nmid j$, then%
\begin{align}
\mathbf{v}_{k}\left(  h_{j}\right)   &  =%
\begin{cases}
h_{j/k}, & \text{if }k\mid j;\\
0, & \text{if }k\nmid j
\end{cases}
\ \ \ \ \ \ \ \ \ \ \left(  \text{by (\ref{pf.thm.Uk.main.c.vh}), applied to
}m=j\right) \nonumber\\
&  =0\ \ \ \ \ \ \ \ \ \ \left(  \text{since }k\nmid j\right)
\label{pf.thm.Uk.main.c.vkhj2}%
\end{align}
and
\begin{align}
U_{k}\left(  h_{j}\right)   &  =\left(  \mathbf{f}_{k}\circ S\circ
\mathbf{v}_{k}\right)  \left(  h_{j}\right)  \ \ \ \ \ \ \ \ \ \ \left(
\text{since }U_{k}=\mathbf{f}_{k}\circ S\circ\mathbf{v}_{k}\right) \nonumber\\
&  =\left(  \mathbf{f}_{k}\circ S\right)  \left(  \underbrace{\mathbf{v}%
_{k}\left(  h_{j}\right)  }_{\substack{=0\\\text{(by
(\ref{pf.thm.Uk.main.c.vkhj2}))}}}\right)  =\left(  \mathbf{f}_{k}\circ
S\right)  \left(  0\right) \nonumber\\
&  =0 \label{pf.thm.Uk.main.c.Ukhj2}%
\end{align}
(since the map $\mathbf{f}_{k}\circ S$ is $\mathbf{k}$-linear).

Let $\Delta_{\Lambda}$ be the comultiplication $\Delta:\Lambda\rightarrow
\Lambda\otimes\Lambda$ of the $\mathbf{k}$-coalgebra $\Lambda$. Let
$m_{\Lambda}:\Lambda\otimes\Lambda\rightarrow\Lambda$ be the $\mathbf{k}%
$-linear map sending each pure tensor $a\otimes b\in\Lambda\otimes\Lambda$ to
$ab\in\Lambda$. Definition \ref{def.convolution} then yields
$\operatorname*{id}\nolimits_{\Lambda}\star U_{k}=m_{\Lambda}\circ\left(
\operatorname*{id}\nolimits_{\Lambda}\otimes U_{k}\right)  \circ
\Delta_{\Lambda}$. Thus,%
\begin{align}
V_{k}  &  =\operatorname*{id}\nolimits_{\Lambda}\star U_{k}=m_{\Lambda}%
\circ\left(  \operatorname*{id}\nolimits_{\Lambda}\otimes U_{k}\right)
\circ\underbrace{\Delta_{\Lambda}}_{=\Delta}\nonumber\\
&  =m_{\Lambda}\circ\left(  \operatorname*{id}\nolimits_{\Lambda}\otimes
U_{k}\right)  \circ\Delta. \label{pf.thm.Uk.main.c.conv-def}%
\end{align}

\begin{vershort}
Let $m\in\mathbb{N}$ (not to be mistaken for the map $m_{\Lambda}$). Then,
\cite[Proposition 2.3.6(iii)]{GriRei} (applied to $n=m$) yields
\begin{align*}
\Delta\left(  h_{m}\right)   &  =\sum_{i+j=m}h_{i}\otimes h_{j}\\
&  \ \ \ \ \ \ \ \ \ \ \left(  \text{where the sum ranges over all pairs
}\left(  i,j\right)  \in\mathbb{N}\times\mathbb{N}\text{ with }i+j=m\right) \\
&  =\sum_{j=0}^{m}h_{m-j}\otimes h_{j}%
\end{align*}
(here, we have substituted $\left(  m-j,j\right)  $ for $\left(  i,j\right)  $
in the sum). Applying the map $\operatorname*{id}\nolimits_{\Lambda}\otimes
U_{k}$ to both sides of this equality, we obtain%
\[
\left(  \operatorname*{id}\nolimits_{\Lambda}\otimes U_{k}\right)  \left(
\Delta\left(  h_{m}\right)  \right)  =\left(  \operatorname*{id}%
\nolimits_{\Lambda}\otimes U_{k}\right)  \left(  \sum_{j=0}^{m}h_{m-j}\otimes
h_{j}\right)  =\sum_{j=0}^{m}h_{m-j}\otimes U_{k}\left(  h_{j}\right)  .
\]
Applying the map $m_{\Lambda}$ to both sides of this equality, we find%
\begin{align*}
&  m_{\Lambda}\left(  \left(  \operatorname*{id}\nolimits_{\Lambda}\otimes
U_{k}\right)  \left(  \Delta\left(  h_{m}\right)  \right)  \right) \\
&  =m_{\Lambda}\left(  \sum_{j=0}^{m}h_{m-j}\otimes U_{k}\left(  h_{j}\right)
\right)  =\sum_{j=0}^{m}h_{m-j}U_{k}\left(  h_{j}\right)
\ \ \ \ \ \ \ \ \ \ \left(  \text{by the definition of }m_{\Lambda}\right) \\
&  =\underbrace{\sum_{j=0}^{\infty}}_{=\sum_{j\in\mathbb{N}}}h_{m-j}%
U_{k}\left(  h_{j}\right)  -\sum_{j=m+1}^{\infty}\underbrace{h_{m-j}%
}_{\substack{=0\\\text{(since }m-j<0\\\text{(because }j\geq m+1>m\text{))}%
}}U_{k}\left(  h_{j}\right) \\
&  =\sum_{j\in\mathbb{N}}h_{m-j}U_{k}\left(  h_{j}\right)  -\underbrace{\sum
_{j=m+1}^{\infty}0U_{k}\left(  h_{j}\right)  }_{=0}\\
&  =\sum_{j\in\mathbb{N}}h_{m-j}U_{k}\left(  h_{j}\right)  =\sum
_{\substack{j\in\mathbb{N};\\k\mid j}}h_{m-j}U_{k}\left(  h_{j}\right)
+\sum_{\substack{j\in\mathbb{N};\\k\nmid j}}h_{m-j}\underbrace{U_{k}\left(
h_{j}\right)  }_{\substack{=0\\\text{(by (\ref{pf.thm.Uk.main.c.Ukhj2}))}}}\\
&  \ \ \ \ \ \ \ \ \ \ \left(  \text{since each }j\in\mathbb{N}\text{
satisfies either }k\mid j\text{ or }k\nmid j\text{ (but not both)}\right) \\
&  =\sum_{\substack{j\in\mathbb{N};\\k\mid j}}h_{m-j}U_{k}\left(
h_{j}\right)  +\underbrace{\sum_{\substack{j\in\mathbb{N};\\k\nmid j}%
}h_{m-j}0}_{=0}=\sum_{\substack{j\in\mathbb{N};\\k\mid j}}h_{m-j}U_{k}\left(
h_{j}\right)  =\sum_{i\in\mathbb{N}}h_{m-ki}\underbrace{U_{k}\left(
h_{ki}\right)  }_{\substack{=\left(  -1\right)  ^{i}\mathbf{f}_{k}\left(
e_{i}\right)  \\\text{(by (\ref{pf.thm.Uk.main.c.Ukhj1}))}}}\\
&  \ \ \ \ \ \ \ \ \ \ \left(  \text{here, we have substituted }ki\text{ for
}j\text{ in the sum}\right) \\
&  =\sum_{i\in\mathbb{N}}\left(  -1\right)  ^{i}h_{m-ki}\cdot\mathbf{f}%
_{k}\left(  e_{i}\right)  .
\end{align*}
Comparing this with%
\[
G\left(  k,m\right)  =\sum_{i\in\mathbb{N}}\left(  -1\right)  ^{i}%
h_{m-ki}\cdot\mathbf{f}_{k}\left(  e_{i}\right)  \ \ \ \ \ \ \ \ \ \ \left(
\text{by Theorem \ref{thm.G.frob}}\right)  ,
\]
we obtain%
\[
G\left(  k,m\right)  =m_{\Lambda}\left(  \left(  \operatorname*{id}%
\nolimits_{\Lambda}\otimes U_{k}\right)  \left(  \Delta\left(  h_{m}\right)
\right)  \right)  =\underbrace{\left(  m_{\Lambda}\circ\left(
\operatorname*{id}\nolimits_{\Lambda}\otimes U_{k}\right)  \circ\Delta\right)
}_{\substack{=V_{k}\\\text{(by (\ref{pf.thm.Uk.main.c.conv-def}))}}}\left(
h_{m}\right)  =V_{k}\left(  h_{m}\right)  .
\]
This proves Theorem \ref{thm.Uk.main} \textbf{(c)}.
\end{vershort}

\begin{verlong}
Let $m\in\mathbb{N}$ (not to be mistaken for the map $m_{\Lambda}$). Then,
\cite[Proposition 2.3.6(iii)]{GriRei} (applied to $n=m$) yields
\begin{align*}
\Delta\left(  h_{m}\right)   &  =\sum_{i+j=m}h_{i}\otimes h_{j}\\
&  \ \ \ \ \ \ \ \ \ \ \left(  \text{where the sum ranges over all pairs
}\left(  i,j\right)  \in\mathbb{N}\times\mathbb{N}\text{ with }i+j=m\right) \\
&  =\sum_{j=0}^{m}h_{m-j}\otimes h_{j}%
\end{align*}
(here, we have substituted $\left(  m-j,j\right)  $ for $\left(  i,j\right)  $
in the sum, since the map $\left\{  0,1,\ldots,m\right\}  \rightarrow\left\{
\left(  i,j\right)  \in\mathbb{N}\times\mathbb{N}\ \mid\ i+j=m\right\}  $ that
sends each $j$ to $\left(  m-j,j\right)  $ is a bijection). Applying the map
$\operatorname*{id}\nolimits_{\Lambda}\otimes U_{k}$ to both sides of this
equality, we obtain%
\begin{align*}
\left(  \operatorname*{id}\nolimits_{\Lambda}\otimes U_{k}\right)  \left(
\Delta\left(  h_{m}\right)  \right)   &  =\left(  \operatorname*{id}%
\nolimits_{\Lambda}\otimes U_{k}\right)  \left(  \sum_{j=0}^{m}h_{m-j}\otimes
h_{j}\right) \\
&  =\sum_{j=0}^{m}\underbrace{\operatorname*{id}\nolimits_{\Lambda}\left(
h_{m-j}\right)  }_{=h_{m-j}}\otimes U_{k}\left(  h_{j}\right)  =\sum_{j=0}%
^{m}h_{m-j}\otimes U_{k}\left(  h_{j}\right)  .
\end{align*}
Applying the map $m_{\Lambda}$ to both sides of this equality, we find%
\begin{align*}
&  m_{\Lambda}\left(  \left(  \operatorname*{id}\nolimits_{\Lambda}\otimes
U_{k}\right)  \left(  \Delta\left(  h_{m}\right)  \right)  \right) \\
&  =m_{\Lambda}\left(  \sum_{j=0}^{m}h_{m-j}\otimes U_{k}\left(  h_{j}\right)
\right)  =\sum_{j=0}^{m}\underbrace{m_{\Lambda}\left(  h_{m-j}\otimes
U_{k}\left(  h_{j}\right)  \right)  }_{\substack{=h_{m-j}U_{k}\left(
h_{j}\right)  \\\text{(by the definition of }m_{\Lambda}\text{)}}}=\sum
_{j=0}^{m}h_{m-j}U_{k}\left(  h_{j}\right) \\
&  =\sum_{j\in\mathbb{N}}h_{m-j}U_{k}\left(  h_{j}\right)
\end{align*}
(since%
\begin{align*}
\sum_{j\in\mathbb{N}}h_{m-j}U_{k}\left(  h_{j}\right)   &  =\sum_{j=0}%
^{m}h_{m-j}U_{k}\left(  h_{j}\right)  +\sum_{j=m+1}^{\infty}%
\underbrace{h_{m-j}}_{\substack{=0\\\text{(since }m-j<0\\\text{(because }j\geq
m+1>m\text{))}}}U_{k}\left(  h_{j}\right) \\
&  =\sum_{j=0}^{m}h_{m-j}U_{k}\left(  h_{j}\right)  +\underbrace{\sum
_{j=m+1}^{\infty}0U_{k}\left(  h_{j}\right)  }_{=0}=\sum_{j=0}^{m}h_{m-j}%
U_{k}\left(  h_{j}\right)
\end{align*}
). Therefore,%
\begin{align*}
&  m_{\Lambda}\left(  \left(  \operatorname*{id}\nolimits_{\Lambda}\otimes
U_{k}\right)  \left(  \Delta\left(  h_{m}\right)  \right)  \right) \\
&  =\sum_{j\in\mathbb{N}}h_{m-j}U_{k}\left(  h_{j}\right)  =\sum
_{\substack{j\in\mathbb{N};\\k\mid j}}h_{m-j}U_{k}\left(  h_{j}\right)
+\sum_{\substack{j\in\mathbb{N};\\k\nmid j}}h_{m-j}\underbrace{U_{k}\left(
h_{j}\right)  }_{\substack{=0\\\text{(by (\ref{pf.thm.Uk.main.c.Ukhj2}))}}}\\
&  \ \ \ \ \ \ \ \ \ \ \left(  \text{since each }j\in\mathbb{N}\text{
satisfies either }k\mid j\text{ or }k\nmid j\text{ (but not both)}\right) \\
&  =\sum_{\substack{j\in\mathbb{N};\\k\mid j}}h_{m-j}U_{k}\left(
h_{j}\right)  +\underbrace{\sum_{\substack{j\in\mathbb{N};\\k\nmid j}%
}h_{m-j}0}_{=0}=\sum_{\substack{j\in\mathbb{N};\\k\mid j}}h_{m-j}U_{k}\left(
h_{j}\right)  =\sum_{i\in\mathbb{N}}h_{m-ki}\underbrace{U_{k}\left(
h_{ki}\right)  }_{\substack{=\left(  -1\right)  ^{i}\mathbf{f}_{k}\left(
e_{i}\right)  \\\text{(by (\ref{pf.thm.Uk.main.c.Ukhj1}))}}}\\
&  \ \ \ \ \ \ \ \ \ \ \left(  \text{here, we have substituted }ki\text{ for
}j\text{ in the sum}\right) \\
&  =\sum_{i\in\mathbb{N}}\underbrace{h_{m-ki}\left(  -1\right)  ^{i}%
}_{=\left(  -1\right)  ^{i}h_{m-ki}}\mathbf{f}_{k}\left(  e_{i}\right)
=\sum_{i\in\mathbb{N}}\left(  -1\right)  ^{i}h_{m-ki}\cdot\mathbf{f}%
_{k}\left(  e_{i}\right)  .
\end{align*}
Comparing this with%
\[
G\left(  k,m\right)  =\sum_{i\in\mathbb{N}}\left(  -1\right)  ^{i}%
h_{m-ki}\cdot\mathbf{f}_{k}\left(  e_{i}\right)  \ \ \ \ \ \ \ \ \ \ \left(
\text{by Theorem \ref{thm.G.frob}}\right)  ,
\]
we obtain%
\[
G\left(  k,m\right)  =m_{\Lambda}\left(  \left(  \operatorname*{id}%
\nolimits_{\Lambda}\otimes U_{k}\right)  \left(  \Delta\left(  h_{m}\right)
\right)  \right)  =\underbrace{\left(  m_{\Lambda}\circ\left(
\operatorname*{id}\nolimits_{\Lambda}\otimes U_{k}\right)  \circ\Delta\right)
}_{\substack{=V_{k}\\\text{(by (\ref{pf.thm.Uk.main.c.conv-def}))}}}\left(
h_{m}\right)  =V_{k}\left(  h_{m}\right)  .
\]
This proves Theorem \ref{thm.Uk.main} \textbf{(c)}.
\end{verlong}

\textbf{(d)} Let us recall a few facts from \cite{GriRei}.

From \cite[Exercise 2.9.10(a)]{GriRei}, we know that every positive integers
$n$ and $m$ satisfy%
\begin{equation}
\mathbf{v}_{n}\left(  p_{m}\right)  =%
\begin{cases}
np_{m/n}, & \text{if }n\mid m;\\
0, & \text{if }n\nmid m
\end{cases}
\ \ . \label{pf.thm.Uk.main.vp}%
\end{equation}

\begin{vershort}
On the other hand, it is easy to see (directly using the definition of
$\mathbf{f}_{n}$) that every positive integers $n$ and $m$ satisfy%
\begin{equation}
\mathbf{f}_{n}\left(  p_{m}\right)  =p_{nm}. \label{pf.thm.Uk.main.short.fp}%
\end{equation}

\end{vershort}

\begin{verlong}
On the other hand, it is easy to see (directly using the definition of
$\mathbf{f}_{n}$) that every positive integers $n$ and $m$ satisfy%
\begin{equation}
\mathbf{f}_{n}\left(  p_{m}\right)  =p_{nm}. \label{pf.thm.Uk.main.fp}%
\end{equation}
\footnote{\textit{Proof.} Let $n$ and $m$ be two positive integers. Then, the
definition of $p_{nm}$ yields $p_{nm}=x_{1}^{nm}+x_{2}^{nm}+x_{3}^{nm}%
+\cdots=\sum_{i\geq1}x_{i}^{nm}$. But the definition of $p_{m}$ yields
$p_{m}=x_{1}^{m}+x_{2}^{m}+x_{3}^{m}+\cdots=\sum_{i\geq1}x_{i}^{m}$. Now, the
definition of $\mathbf{f}_{n}$ yields%
\begin{align*}
\mathbf{f}_{n}\left(  p_{m}\right)   &  =p_{m}\left(  x_{1}^{n},x_{2}%
^{n},x_{3}^{n},\ldots\right)  =\sum_{i\geq1}\underbrace{\left(  x_{i}%
^{n}\right)  ^{m}}_{=x_{i}^{nm}}\ \ \ \ \ \ \ \ \ \ \left(  \text{since }%
p_{m}=\sum_{i\geq1}x_{i}^{m}\right) \\
&  =\sum_{i\geq1}x_{i}^{nm}=p_{nm}\ \ \ \ \ \ \ \ \ \ \left(  \text{since
}p_{nm}=\sum_{i\geq1}x_{i}^{nm}\right)  .
\end{align*}
This proves (\ref{pf.thm.Uk.main.fp}).}
\end{verlong}

Finally, \cite[Proposition 2.4.1(i)]{GriRei} yields that every positive
integer $n$ satisfies%
\begin{equation}
S\left(  p_{n}\right)  =-p_{n}. \label{pf.thm.Uk.main.Sp}%
\end{equation}

Now, let $n$ be a positive integer. We first claim the following:

\begin{statement}
\textit{Claim 1:} We have $U_{k}\left(  p_{n}\right)  =-\left[  k\mid
n\right]  kp_{n}$.
\end{statement}

[\textit{Proof of Claim 1:} We are in one of the following two cases:

\textit{Case 1:} We have $k\mid n$.

\textit{Case 2:} We have $k\nmid n$.

Let us first consider Case 1. In this case, we have $k\mid n$. Hence, $n/k$ is
a positive integer. Now, (\ref{pf.thm.Uk.main.vp}) (applied to $k$ and $n$
instead of $n$ and $m$) yields%
\[
\mathbf{v}_{k}\left(  p_{n}\right)  =%
\begin{cases}
kp_{n/k}, & \text{if }k\mid n;\\
0, & \text{if }k\nmid n
\end{cases}
\ \ \ \ =kp_{n/k}\ \ \ \ \ \ \ \ \ \ \left(  \text{since }k\mid n\right)  .
\]
Applying the map $S$ to both sides of this equality, we find%
\begin{align*}
S\left(  \mathbf{v}_{k}\left(  p_{n}\right)  \right)   &  =S\left(
kp_{n/k}\right)  =k\underbrace{S\left(  p_{n/k}\right)  }_{\substack{=-p_{n/k}%
\\\text{(by (\ref{pf.thm.Uk.main.Sp}),}\\\text{applied to }n/k\text{ instead
of }n\text{)}}}\ \ \ \ \ \ \ \ \ \ \left(  \text{since the map }S\text{ is
}\mathbf{k}\text{-linear}\right) \\
&  =k\left(  -p_{n/k}\right)  =-kp_{n/k}.
\end{align*}

\begin{vershort}
\noindent Applying the map $\mathbf{f}_{k}$ to both sides of this equality, we
find
\begin{align*}
\mathbf{f}_{k}\left(  S\left(  \mathbf{v}_{k}\left(  p_{n}\right)  \right)
\right)   &  =\mathbf{f}_{k}\left(  -kp_{n/k}\right)
=-k\underbrace{\mathbf{f}_{k}\left(  p_{n/k}\right)  }_{\substack{=p_{k\left(
n/k\right)  }\\\text{(by (\ref{pf.thm.Uk.main.short.fp}),}\\\text{applied to
}k\text{ and }n/k\\\text{instead of }n\text{ and }m\text{)}}%
}\ \ \ \ \ \ \ \ \ \ \left(  \text{since the map }\mathbf{f}_{k}\text{ is
}\mathbf{k}\text{-linear}\right) \\
&  =-kp_{k\left(  n/k\right)  }=-kp_{n}.
\end{align*}

\end{vershort}

\begin{verlong}
\noindent Applying the map $\mathbf{f}_{k}$ to both sides of this equality, we
find
\begin{align*}
\mathbf{f}_{k}\left(  S\left(  \mathbf{v}_{k}\left(  p_{n}\right)  \right)
\right)   &  =\mathbf{f}_{k}\left(  -kp_{n/k}\right)
=-k\underbrace{\mathbf{f}_{k}\left(  p_{n/k}\right)  }_{\substack{=p_{k\left(
n/k\right)  }\\\text{(by (\ref{pf.thm.Uk.main.fp}),}\\\text{applied to
}k\text{ and }n/k\\\text{instead of }n\text{ and }m\text{)}}%
}\ \ \ \ \ \ \ \ \ \ \left(  \text{since the map }\mathbf{f}_{k}\text{ is
}\mathbf{k}\text{-linear}\right) \\
&  =-kp_{k\left(  n/k\right)  }=-kp_{n}\ \ \ \ \ \ \ \ \ \ \left(  \text{since
}k\left(  n/k\right)  =n\right)  .
\end{align*}

\end{verlong}

\noindent Now, the definition of $U_{k}$ yields $U_{k}=\mathbf{f}_{k}\circ
S\circ\mathbf{v}_{k}$. Hence,%
\[
U_{k}\left(  p_{n}\right)  =\left(  \mathbf{f}_{k}\circ S\circ\mathbf{v}%
_{k}\right)  \left(  p_{n}\right)  =\mathbf{f}_{k}\left(  S\left(
\mathbf{v}_{k}\left(  p_{n}\right)  \right)  \right)  =-kp_{n}.
\]
Comparing this with%
\[
-\underbrace{\left[  k\mid n\right]  }_{\substack{=1\\\text{(since }k\mid
n\text{)}}}kp_{n}=-kp_{n},
\]
we obtain $U_{k}\left(  p_{n}\right)  =-\left[  k\mid n\right]  kp_{n}$.
Hence, Claim 1 is proved in Case 1.

Let us now consider Case 2. In this case, we have $k\nmid n$. But
(\ref{pf.thm.Uk.main.vp}) (applied to $k$ and $n$ instead of $n$ and $m$)
yields%
\[
\mathbf{v}_{k}\left(  p_{n}\right)  =%
\begin{cases}
kp_{n/k}, & \text{if }k\mid n;\\
0, & \text{if }k\nmid n
\end{cases}
\ \ \ \ =0\ \ \ \ \ \ \ \ \ \ \left(  \text{since }k\nmid n\right)  .
\]
But the definition of $U_{k}$ yields $U_{k}=\mathbf{f}_{k}\circ S\circ
\mathbf{v}_{k}$. Hence,%
\[
U_{k}\left(  p_{n}\right)  =\left(  \mathbf{f}_{k}\circ S\circ\mathbf{v}%
_{k}\right)  \left(  p_{n}\right)  =\left(  \mathbf{f}_{k}\circ S\right)
\left(  \underbrace{\mathbf{v}_{k}\left(  p_{n}\right)  }_{=0}\right)
=\left(  \mathbf{f}_{k}\circ S\right)  \left(  0\right)  =0
\]
(since the map $\mathbf{f}_{k}\circ S$ is $\mathbf{k}$-linear). Comparing this
with%
\[
-\underbrace{\left[  k\mid n\right]  }_{\substack{=0\\\text{(since }k\nmid
n\text{)}}}kp_{n}=0,
\]
we obtain $U_{k}\left(  p_{n}\right)  =-\left[  k\mid n\right]  kp_{n}$.
Hence, Claim 1 is proved in Case 2.

We have now proved Claim 1 in both Cases 1 and 2. Thus, Claim 1 always holds.]

Theorem \ref{thm.Uk.main} \textbf{(a)} shows that the map $U_{k}$ is a
$\mathbf{k}$-Hopf algebra homomorphism. Hence, $U_{k}$ is a $\mathbf{k}%
$-algebra homomorphism. Thus, $U_{k}\left(  1\right)  =1$.

Now, let $\Delta_{\Lambda}$ be the comultiplication $\Delta:\Lambda
\rightarrow\Lambda\otimes\Lambda$ of the $\mathbf{k}$-coalgebra $\Lambda$. Let
$m_{\Lambda}:\Lambda\otimes\Lambda\rightarrow\Lambda$ be the $\mathbf{k}%
$-linear map sending each pure tensor $a\otimes b\in\Lambda\otimes\Lambda$ to
$ab\in\Lambda$. Then, the definition of $V_{k}$ yields $V_{k}%
=\operatorname*{id}\nolimits_{\Lambda}\star U_{k}=m_{\Lambda}\circ\left(
\operatorname*{id}\nolimits_{\Lambda}\otimes U_{k}\right)  \circ
\Delta_{\Lambda}$ (by the definition of convolution).

But \cite[Proposition 2.3.6(i)]{GriRei} yields $\Delta_{\Lambda}\left(
p_{n}\right)  =1\otimes p_{n}+p_{n}\otimes1$. Now,%
\begin{align*}
&  \underbrace{V_{k}}_{=m_{\Lambda}\circ\left(  \operatorname*{id}%
\nolimits_{\Lambda}\otimes U_{k}\right)  \circ\Delta_{\Lambda}}\left(
p_{n}\right) \\
&  =\left(  m_{\Lambda}\circ\left(  \operatorname*{id}\nolimits_{\Lambda
}\otimes U_{k}\right)  \circ\Delta_{\Lambda}\right)  \left(  p_{n}\right) \\
&  =m_{\Lambda}\left(  \left(  \operatorname*{id}\nolimits_{\Lambda}\otimes
U_{k}\right)  \left(  \underbrace{\Delta_{\Lambda}\left(  p_{n}\right)
}_{=1\otimes p_{n}+p_{n}\otimes1}\right)  \right) \\
&  =m_{\Lambda}\left(  \underbrace{\left(  \operatorname*{id}%
\nolimits_{\Lambda}\otimes U_{k}\right)  \left(  1\otimes p_{n}+p_{n}%
\otimes1\right)  }_{=\operatorname*{id}\nolimits_{\Lambda}\left(  1\right)
\otimes U_{k}\left(  p_{n}\right)  +\operatorname*{id}\nolimits_{\Lambda
}\left(  p_{n}\right)  \otimes U_{k}\left(  1\right)  }\right) \\
&  =m_{\Lambda}\left(  \operatorname*{id}\nolimits_{\Lambda}\left(  1\right)
\otimes U_{k}\left(  p_{n}\right)  +\operatorname*{id}\nolimits_{\Lambda
}\left(  p_{n}\right)  \otimes U_{k}\left(  1\right)  \right) \\
&  =\underbrace{\operatorname*{id}\nolimits_{\Lambda}\left(  1\right)  }%
_{=1}\cdot\underbrace{U_{k}\left(  p_{n}\right)  }_{\substack{=-\left[  k\mid
n\right]  kp_{n}\\\text{(by Claim 1)}}}+\underbrace{\operatorname*{id}%
\nolimits_{\Lambda}\left(  p_{n}\right)  }_{=p_{n}}\cdot\underbrace{U_{k}%
\left(  1\right)  }_{=1}\ \ \ \ \ \ \ \ \ \ \left(  \text{by the definition of
}m_{\Lambda}\right) \\
&  =-\left[  k\mid n\right]  kp_{n}+p_{n}=\left(  1-\left[  k\mid n\right]
k\right)  p_{n}.
\end{align*}
This proves Theorem \ref{thm.Uk.main} \textbf{(d)}.
\end{proof}

\begin{verlong}
\silentsubsection{\label{subsect.proofs.DeltaGkm.2nd}Second proof of Theorem
\ref{thm.DeltaGkm}}

Let us reprove Theorem \ref{thm.DeltaGkm} using Theorem \ref{thm.Uk.main}:

\begin{proof}
[Second proof of Theorem \ref{thm.DeltaGkm}.]Theorem \ref{thm.Uk.main}
\textbf{(b)} shows that the map $V_{k}$ is a $\mathbf{k}$-Hopf algebra
homomorphism. Thus, in particular, $V_{k}$ is a $\mathbf{k}$-coalgebra
homomorphism. In other words, we have
\[
\left(  V_{k}\otimes V_{k}\right)  \circ\Delta=\Delta\circ V_{k}%
\ \ \ \ \ \ \ \ \ \ \text{and}\ \ \ \ \ \ \ \ \ \ \varepsilon=\varepsilon\circ
V_{k}%
\]
(where $\varepsilon$ denotes the counit of the $\mathbf{k}$-coalgebra
$\Lambda$). But we have%
\begin{equation}
V_{k}\left(  h_{n}\right)  =G\left(  k,n\right)  \ \ \ \ \ \ \ \ \ \ \text{for
each }n\in\mathbb{N}\label{pf.thm.DeltaGkm.2nd.Gkn=}%
\end{equation}
(by Theorem \ref{thm.Uk.main} \textbf{(c)}, applied to $n$ instead of $m$).
Applying this to $n=m$, we obtain $V_{k}\left(  h_{m}\right)  =G\left(
k,m\right)  $, so that $G\left(  k,m\right)  =V_{k}\left(  h_{m}\right)  $.
Applying the map $\Delta$ to both sides of this equality, we find%
\begin{align}
\Delta\left(  G\left(  k,m\right)  \right)   &  =\Delta\left(  V_{k}\left(
h_{m}\right)  \right)  =\underbrace{\left(  \Delta\circ V_{k}\right)
}_{=\left(  V_{k}\otimes V_{k}\right)  \circ\Delta}\left(  h_{m}\right)
=\left(  \left(  V_{k}\otimes V_{k}\right)  \circ\Delta\right)  \left(
h_{m}\right)  \nonumber\\
&  =\left(  V_{k}\otimes V_{k}\right)  \left(  \Delta\left(  h_{m}\right)
\right)  .\label{pf.thm.DeltaGkm.2nd.1}%
\end{align}
But \cite[Proposition 2.3.6(iii)]{GriRei} (applied to $n=m$) yields
\begin{align*}
\Delta\left(  h_{m}\right)   &  =\sum_{i+j=m}h_{i}\otimes h_{j}\\
&  \ \ \ \ \ \ \ \ \ \ \left(  \text{where the sum ranges over all pairs
}\left(  i,j\right)  \in\mathbb{N}\times\mathbb{N}\text{ with }i+j=m\right)
\\
&  =\sum_{i=0}^{m}h_{i}\otimes h_{m-i}%
\end{align*}
(here, we have substituted $\left(  i,m-i\right)  $ for $\left(  i,j\right)  $
in the sum, since the map $\left\{  0,1,\ldots,m\right\}  \rightarrow\left\{
\left(  i,j\right)  \in\mathbb{N}\times\mathbb{N}\ \mid\ i+j=m\right\}  $ that
sends each $i$ to $\left(  i,m-i\right)  $ is a bijection). Hence,
(\ref{pf.thm.DeltaGkm.2nd.1}) becomes%
\begin{align*}
\Delta\left(  G\left(  k,m\right)  \right)   &  =\left(  V_{k}\otimes
V_{k}\right)  \left(  \underbrace{\Delta\left(  h_{m}\right)  }_{=\sum
_{i=0}^{m}h_{i}\otimes h_{m-i}}\right)  =\left(  V_{k}\otimes V_{k}\right)
\left(  \sum_{i=0}^{m}h_{i}\otimes h_{m-i}\right)  \\
&  =\sum_{i=0}^{m}\underbrace{V_{k}\left(  h_{i}\right)  }%
_{\substack{=G\left(  k,i\right)  \\\text{(by (\ref{pf.thm.DeltaGkm.2nd.Gkn=}%
))}}}\otimes\underbrace{V_{k}\left(  h_{m-i}\right)  }_{\substack{=G\left(
k,m-i\right)  \\\text{(by (\ref{pf.thm.DeltaGkm.2nd.Gkn=}))}}}=\sum_{i=0}%
^{m}G\left(  k,i\right)  \otimes G\left(  k,m-i\right)  .
\end{align*}
Thus, Theorem \ref{thm.DeltaGkm} is proved again.
\end{proof}
\end{verlong}

\subsection{\label{subsect.proofs.p-via-G}Proof of Corollary \ref{cor.p-via-G}%
}

\begin{proof}
[Proof of Corollary \ref{cor.p-via-G}.]Recall that the family $\left(
h_{n}\right)  _{n\geq1}=\left(  h_{1},h_{2},h_{3},\ldots\right)  $ generates
$\Lambda$ as a $\mathbf{k}$-algebra. Hence, each $g\in\Lambda$ can be written
as a polynomial in $h_{1},h_{2},h_{3},\ldots$. Applying this to $g=p_{n}$, we
conclude that $p_{n}$ can be written as a polynomial in $h_{1},h_{2}%
,h_{3},\ldots$. In other words, there exists a polynomial $f\in\mathbf{k}%
\left[  x_{1},x_{2},x_{3},\ldots\right]  $ such that%
\begin{equation}
p_{n}=f\left(  h_{1},h_{2},h_{3},\ldots\right)  . \label{pf.cor.p-via-G.0}%
\end{equation}
Consider this $f$. We shall show that this $f$ satisfies
(\ref{eq.cor.p-via-G.eq}). This will clearly prove Corollary \ref{cor.p-via-G}.

Consider the map $V_{k}$ defined in Theorem \ref{thm.Uk.main}. Theorem
\ref{thm.Uk.main} \textbf{(c)} yields that $V_{k}\left(  h_{m}\right)
=G\left(  k,m\right)  $ for each positive integer $m$. In other words,%
\begin{equation}
\left(  V_{k}\left(  h_{1}\right)  ,V_{k}\left(  h_{2}\right)  ,V_{k}\left(
h_{3}\right)  ,\ldots\right)  =\left(  G\left(  k,1\right)  ,G\left(
k,2\right)  ,G\left(  k,3\right)  ,\ldots\right)  . \label{pf.cor.p-via-G.1}%
\end{equation}

The map $V_{k}$ is a $\mathbf{k}$-Hopf algebra homomorphism (by Theorem
\ref{thm.Uk.main} \textbf{(b)}), and thus is a $\mathbf{k}$-algebra
homomorphism. Hence, it commutes with polynomials over $\mathbf{k}$. Thus,%
\begin{align*}
V_{k}\left(  f\left(  h_{1},h_{2},h_{3},\ldots\right)  \right)   &  =f\left(
V_{k}\left(  h_{1}\right)  ,V_{k}\left(  h_{2}\right)  ,V_{k}\left(
h_{3}\right)  ,\ldots\right) \\
&  =f\left(  G\left(  k,1\right)  ,G\left(  k,2\right)  ,G\left(  k,3\right)
,\ldots\right)  \ \ \ \ \ \ \ \ \ \ \left(  \text{by (\ref{pf.cor.p-via-G.1}%
)}\right)  .
\end{align*}
Now, applying the map $V_{k}$ to both sides of the equality
(\ref{pf.cor.p-via-G.0}), we obtain%
\[
V_{k}\left(  p_{n}\right)  =V_{k}\left(  f\left(  h_{1},h_{2},h_{3}%
,\ldots\right)  \right)  =f\left(  G\left(  k,1\right)  ,G\left(  k,2\right)
,G\left(  k,3\right)  ,\ldots\right)  .
\]
Comparing this with%
\[
V_{k}\left(  p_{n}\right)  =\left(  1-\left[  k\mid n\right]  k\right)
p_{n}\ \ \ \ \ \ \ \ \ \ \left(  \text{by Theorem \ref{thm.Uk.main}
\textbf{(d)}}\right)  ,
\]
we obtain%
\[
\left(  1-\left[  k\mid n\right]  k\right)  p_{n}=f\left(  G\left(
k,1\right)  ,G\left(  k,2\right)  ,G\left(  k,3\right)  ,\ldots\right)  .
\]
Thus, we have shown that our $f$ satisfies (\ref{eq.cor.p-via-G.eq}). As we
said, this proves Corollary \ref{cor.p-via-G}.
\end{proof}

\section{\label{sect.liu}Proof of the Liu--Polo conjecture}

Let us recall a well-known partial order on the set of partitions of a given
$n\in\mathbb{N}$:

\begin{definition}
\label{def.dominance}Let $n\in\mathbb{N}$. We define a binary relation
$\triangleright$ on the set $\operatorname*{Par}\nolimits_{n}$ as follows: Two
partitions $\lambda,\mu\in\operatorname*{Par}\nolimits_{n}$ shall satisfy
$\lambda\triangleright\mu$ if and only if we have%
\[
\lambda_{1}+\lambda_{2}+\cdots+\lambda_{k}\geq\mu_{1}+\mu_{2}+\cdots+\mu
_{k}\ \ \ \ \ \ \ \ \ \ \text{for each }k\in\left\{  1,2,\ldots,n\right\}  .
\]
This relation $\triangleright$ is the greater-or-equal relation of a partial
order on $\operatorname*{Par}\nolimits_{n}$, which is known as the
\emph{dominance order} (or the \emph{majorization order}).
\end{definition}

This definition is precisely \cite[Definition 2.2.7]{GriRei}. Note that if we
replace \textquotedblleft for each $k\in\left\{  1,2,\ldots,n\right\}
$\textquotedblright\ by \textquotedblleft for each $k\in\left\{
1,2,3,\ldots\right\}  $\textquotedblright\ in this definition, then the
relation $\triangleright$ does not change.

Our goal in this section is to prove the conjecture made in \cite[Remark
1.4.5]{LiuPol19}. We state this conjecture as follows:\footnote{Note that
$\left(  n-1,n-1,1\right)  $ is a partition whenever $n>1$ is an integer.}

\begin{theorem}
\label{thm.claim}Let $n$ be an integer such that $n>1$. Then:

\textbf{(a)} We have%
\[
\sum_{\substack{\lambda\in\operatorname*{Par}\nolimits_{n};\\\left(
n-1,1\right)  \triangleright\lambda}}m_{\lambda}=\sum_{i=0}^{n-2}\left(
-1\right)  ^{i}s_{\left(  n-1-i,1^{i+1}\right)  }.
\]

\textbf{(b)} We have%
\[
\sum_{\substack{\lambda\in\operatorname*{Par}\nolimits_{2n-1};\\\left(
n-1,n-1,1\right)  \triangleright\lambda}}m_{\lambda}=\sum_{i=0}^{n-2}\left(
-1\right)  ^{i}s_{\left(  n-1,n-1-i,1^{i+1}\right)  }.
\]

\end{theorem}

\begin{example}
For this example, let $n=3$. Then, $n-1=2$ and $2n-1=5$. Hence, the left hand
side of the equality in Theorem \ref{thm.claim} \textbf{(b)} is%
\[
\sum_{\substack{\lambda\in\operatorname*{Par}\nolimits_{2n-1};\\\left(
n-1,n-1,1\right)  \triangleright\lambda}}m_{\lambda}=\sum_{\substack{\lambda
\in\operatorname*{Par}\nolimits_{5};\\\left(  2,2,1\right)  \triangleright
\lambda}}m_{\lambda}=m_{\left(  2,2,1\right)  }+m_{\left(  2,1,1,1\right)
}+m_{\left(  1,1,1,1,1\right)  }.
\]
Meanwhile, the right hand side of the equality in Theorem \ref{thm.claim}
\textbf{(b)} is%
\[
\sum_{i=0}^{n-2}\left(  -1\right)  ^{i}s_{\left(  n-1,n-1-i,1^{i+1}\right)
}=\sum_{i=0}^{1}\left(  -1\right)  ^{i}s_{\left(  2,2-i,1^{i+1}\right)
}=s_{\left(  2,2,1\right)  }-s_{\left(  2,1,1,1\right)  }.
\]
Thus, Theorem \ref{thm.claim} \textbf{(b)} claims that $m_{\left(
2,2,1\right)  }+m_{\left(  2,1,1,1\right)  }+m_{\left(  1,1,1,1,1\right)
}=s_{\left(  2,2,1\right)  }-s_{\left(  2,1,1,1\right)  }$ in this case.
\end{example}

We will pave our way to the proof of Theorem \ref{thm.claim} by several
lemmas. We begin with a particularly simple one:

\begin{lemma}
\label{lem.dominnn-1}Let $n$ be an integer such that $n>1$. Let $\lambda
\in\operatorname*{Par}\nolimits_{2n-1}$. Then, $\left(  n-1,n-1,1\right)
\triangleright\lambda$ if and only if all positive integers $i$ satisfy
$\lambda_{i}<n$.
\end{lemma}

\begin{vershort}
\begin{proof}
This simple proof (an exercise in following Definition~\ref{def.dominance}) is
left to the reader.
\end{proof}
\end{vershort}

\begin{verlong}
\begin{proof}
[Proof of Lemma~\ref{lem.dominnn-1}.]$\Longrightarrow:$ Assume that $\left(
n-1,n-1,1\right)  \triangleright\lambda$. Thus, $n-1\geq\lambda_{1}$ (by
Definition \ref{def.dominance}). Hence, $\lambda_{1}\leq n-1<n$. But $\lambda$
is a partition; thus, $\lambda_{1}\geq\lambda_{2}\geq\lambda_{3}\geq\cdots$.
Hence, all positive integers $i$ satisfy $\lambda_{i}\leq\lambda_{1}<n$. This
proves the \textquotedblleft$\Longrightarrow$\textquotedblright\ direction of
Lemma \ref{lem.dominnn-1}.

$\Longleftarrow:$ Assume that all positive integers $i$ satisfy $\lambda
_{i}<n$. Thus, all positive integers $i$ satisfy $\lambda_{i}\leq n-1$ (since
$\lambda_{i}$ and $n$ are integers). Hence, in particular, $\lambda_{1}\leq
n-1$ and $\lambda_{2}\leq n-1$.

Define a partition $\mu$ by $\mu=\left(  n-1,n-1,1\right)  $; thus,
$\left\vert \mu\right\vert =\left(  n-1\right)  +\left(  n-1\right)  +1=2n-1$,
so that $\mu\in\operatorname*{Par}\nolimits_{2n-1}$. Also, $\lambda
\in\operatorname*{Par}\nolimits_{2n-1}$ (as we know). Thus, $\mu
\triangleright\lambda$ holds if and only if each $k\in\left\{  1,2,\ldots
,2n-1\right\}  $ satisfies%
\begin{equation}
\mu_{1}+\mu_{2}+\cdots+\mu_{k}\geq\lambda_{1}+\lambda_{2}+\cdots+\lambda_{k}
\label{pf.lem.dominnn-1.denn.1}%
\end{equation}
(by Definition~\ref{def.dominance}).

But each $k\in\left\{  1,2,\ldots,2n-1\right\}  $ satisfies
(\ref{pf.lem.dominnn-1.denn.1}).

[\textit{Proof of (\ref{pf.lem.dominnn-1.denn.1}):} Let $k\in\left\{
1,2,\ldots,2n-1\right\}  $. We must prove (\ref{pf.lem.dominnn-1.denn.1}).

If $k\geq3$, then
\begin{align*}
\mu_{1}+\mu_{2}+\cdots+\mu_{k}  &  \geq\mu_{1}+\mu_{2}+\mu_{3}\\
&  =\left(  n-1\right)  +\left(  n-1\right)  +1\ \ \ \ \ \ \ \ \ \ \left(
\text{since }\mu=\left(  n-1,n-1,1\right)  \right) \\
&  =2n-1=\left\vert \lambda\right\vert \ \ \ \ \ \ \ \ \ \ \left(  \text{since
}\lambda\in\operatorname*{Par}\nolimits_{2n-1}\right) \\
&  =\lambda_{1}+\lambda_{2}+\lambda_{3}+\cdots\geq\lambda_{1}+\lambda
_{2}+\cdots+\lambda_{k},
\end{align*}
and thus (\ref{pf.lem.dominnn-1.denn.1}) is proven in this case. Hence, it
remains to prove (\ref{pf.lem.dominnn-1.denn.1}) for $k\leq2$. But
$\mu=\left(  n-1,n-1,1\right)  $, and thus $\mu_{1}=n-1\geq\lambda_{1}$ and
$\mu_{2}=n-1\geq\lambda_{2}$. Hence, $\mu_{1}\geq\lambda_{1}$ and
$\underbrace{\mu_{1}}_{\geq\lambda_{1}}+\underbrace{\mu_{2}}_{\geq\lambda_{2}%
}\geq\lambda_{1}+\lambda_{2}$. In other words, (\ref{pf.lem.dominnn-1.denn.1})
is proven for $k\leq2$. As we have said, this concludes the proof of
(\ref{pf.lem.dominnn-1.denn.1}).]

Thus, we have $\mu\triangleright\lambda$ (since $\mu\triangleright\lambda$
holds if and only if each $k\in\left\{  1,2,\ldots,2n-1\right\}  $ satisfies
(\ref{pf.lem.dominnn-1.denn.1})). In other words, $\left(  n-1,n-1,1\right)
\triangleright\lambda$ holds (since $\mu=\left(  n-1,n-1,1\right)  $). This
proves the \textquotedblleft$\Longleftarrow$\textquotedblright\ direction of
Lemma \ref{lem.dominnn-1}.
\end{proof}
\end{verlong}

\begin{lemma}
\label{lem.dominnn-1a}Let $n$ be an integer such that $n>1$. Let $\lambda
\in\operatorname*{Par}\nolimits_{n}$. Then, $\left(  n-1,1\right)
\triangleright\lambda$ if and only if all positive integers $i$ satisfy
$\lambda_{i}<n$.
\end{lemma}

\begin{proof}
[Proof of Lemma \ref{lem.dominnn-1a}.]This is analogous to the proof of Lemma
\ref{lem.dominnn-1}.
\end{proof}

The next lemma identifies the left hand side of Theorem \ref{thm.claim}
\textbf{(a)} as the Petrie symmetric function $G\left(  n,n\right)  $, and the
left hand side of Theorem \ref{thm.claim} \textbf{(b)} as the Petrie symmetric
function $G\left(  n,2n-1\right)  $:

\begin{corollary}
\label{cor.LHS1}Let $n$ be an integer such that $n>1$. Then:

\textbf{(a)} We have%
\[
\sum_{\substack{\lambda\in\operatorname*{Par}\nolimits_{n};\\\left(
n-1,1\right)  \triangleright\lambda}}m_{\lambda}=G\left(  n,n\right)  .
\]

\textbf{(b)} We have%
\[
\sum_{\substack{\lambda\in\operatorname*{Par}\nolimits_{2n-1};\\\left(
n-1,n-1,1\right)  \triangleright\lambda}}m_{\lambda}=G\left(  n,2n-1\right)
.
\]

\end{corollary}

\begin{proof}
\textbf{(b)} Proposition \ref{prop.G.basics} \textbf{(c)} (applied to $k=n$
and $m=2n-1$) yields
\begin{equation}
G\left(  n,2n-1\right)  =\sum_{\substack{\alpha\in\operatorname*{WC}%
;\\\left\vert \alpha\right\vert =2n-1;\\\alpha_{i}<n\text{ for all }%
i}}\mathbf{x}^{\alpha}=\sum_{\substack{\lambda\in\operatorname*{Par}%
;\\\left\vert \lambda\right\vert =2n-1;\\\lambda_{i}<n\text{ for all }%
i}}m_{\lambda}. \label{pf.cor.LHS1.1}%
\end{equation}
But Lemma \ref{lem.dominnn-1} yields the following equality of summation
signs:%
\[
\sum_{\substack{\lambda\in\operatorname*{Par}\nolimits_{2n-1};\\\left(
n-1,n-1,1\right)  \triangleright\lambda}}=\sum_{\substack{\lambda
\in\operatorname*{Par}\nolimits_{2n-1};\\\lambda_{i}<n\text{ for all }i}%
}=\sum_{\substack{\lambda\in\operatorname*{Par};\\\left\vert \lambda
\right\vert =2n-1;\\\lambda_{i}<n\text{ for all }i}}.
\]
Hence,%
\[
\sum_{\substack{\lambda\in\operatorname*{Par}\nolimits_{2n-1};\\\left(
n-1,n-1,1\right)  \triangleright\lambda}}m_{\lambda}=\sum_{\substack{\lambda
\in\operatorname*{Par};\\\left\vert \lambda\right\vert =2n-1;\\\lambda
_{i}<n\text{ for all }i}}m_{\lambda}.
\]
Comparing this with (\ref{pf.cor.LHS1.1}), we obtain
\[
\sum_{\substack{\lambda\in\operatorname*{Par}\nolimits_{2n-1};\\\left(
n-1,n-1,1\right)  \triangleright\lambda}}m_{\lambda}=G\left(  n,2n-1\right)
.
\]
This proves Corollary \ref{cor.LHS1} \textbf{(b)}.

\textbf{(a)} This is analogous to Corollary \ref{cor.LHS1} \textbf{(b)}, but
uses Lemma \ref{lem.dominnn-1a} instead of Lemma \ref{lem.dominnn-1}.
\end{proof}

It was Corollary \ref{cor.LHS1} that led the author to introduce and study the
Petrie symmetric functions $G\left(  k,m\right)  $ in general, even if little
of their general properties has proven relevant to Theorem \ref{thm.claim}.

The next proposition gives a simple formula for certain kinds of Petrie
symmetric functions:

\begin{proposition}
\label{prop.h2n-1}Let $n$ be a positive integer. Let $k\in\left\{
0,1,\ldots,n-1\right\}  $. Then,%
\[
G\left(  n,n+k\right)  =h_{n+k}-h_{k}p_{n}.
\]

\end{proposition}

Proposition \ref{prop.h2n-1} can be viewed as a particular case of Theorem
\ref{thm.G.frob} (applied to $n$ and $n+k$ instead of $k$ and $m$), after
realizing that in the sum on the right hand side of Theorem \ref{thm.G.frob},
only the first two addends will (potentially) be nonzero in this case.
However, let us give an independent proof of the proposition.

\begin{proof}
[Proof of Proposition \ref{prop.h2n-1}.]From $k\in\left\{  0,1,\ldots
,n-1\right\}  $, we obtain $k<n$ and thus $n+k<n+n$. Thus we conclude:

\begin{statement}
\textit{Observation 1:} A monomial of degree $n+k$ cannot have more than one
variable appear in it with exponent $\geq n$ (since this would require it to
have degree $\geq n+n>n+k$).
\end{statement}

Let $\mathfrak{M}_{k}$ be the set of all monomials of degree $k$. The
definition of $h_{k}$ shows that $h_{k}$ is the sum of all monomials of degree
$k$. In other words,%
\begin{equation}
h_{k}=\sum_{\mathfrak{m}\in\mathfrak{M}_{k}}\mathfrak{m}.
\label{pf.prop.h2n-1.hk=}%
\end{equation}

Let $\mathfrak{M}_{n+k}$ be the set of all monomials of degree $n+k$. The
definition of $h_{n+k}$ shows that $h_{n+k}$ is the sum of all monomials of
degree $n+k$. In other words,%
\begin{equation}
h_{n+k}=\sum_{\mathfrak{n}\in\mathfrak{M}_{n+k}}\mathfrak{n}.
\label{pf.prop.h2n-1.hn+k=}%
\end{equation}

Let $\mathfrak{N}$ be the set of all monomials of degree $n+k$ in which all
exponents are $<n$. These monomials are exactly the $\mathbf{x}^{\alpha}$ for
$\alpha\in\operatorname*{WC}$ satisfying $\left\vert \alpha\right\vert =n+k$
and $\left(  \alpha_{i}<n\text{ for all }i\right)  $. Hence,%
\begin{equation}
\sum_{\mathfrak{n}\in\mathfrak{N}}\mathfrak{n}=\sum_{\substack{\alpha
\in\operatorname*{WC};\\\left\vert \alpha\right\vert =n+k;\\\alpha_{i}<n\text{
for all }i}}\mathbf{x}^{\alpha}. \label{pf.prop.h2n-1.su1}%
\end{equation}
But Proposition \ref{prop.G.basics} \textbf{(c)} (applied to $n$ and $n+k$
instead of $k$ and $m$) yields
\[
G\left(  n,n+k\right)  =\sum_{\substack{\alpha\in\operatorname*{WC}%
;\\\left\vert \alpha\right\vert =n+k;\\\alpha_{i}<n\text{ for all }%
i}}\mathbf{x}^{\alpha}=\sum_{\substack{\lambda\in\operatorname*{Par}%
;\\\left\vert \lambda\right\vert =n+k;\\\lambda_{i}<n\text{ for all }%
i}}m_{\lambda}.
\]
Hence,%
\begin{equation}
G\left(  n,n+k\right)  =\sum_{\substack{\alpha\in\operatorname*{WC}%
;\\\left\vert \alpha\right\vert =n+k;\\\alpha_{i}<n\text{ for all }%
i}}\mathbf{x}^{\alpha}=\sum_{\mathfrak{n}\in\mathfrak{N}}\mathfrak{n}
\label{pf.prop.h2n-1.su1pet}%
\end{equation}
(by (\ref{pf.prop.h2n-1.su1})).

Clearly, the set $\mathfrak{N}$ is a subset of $\mathfrak{M}_{n+k}$, and
furthermore its complement $\mathfrak{M}_{n+k}\setminus\mathfrak{N}$ is the
set of all monomials of degree $n+k$ in which at least one exponent is $\geq
n$. Hence, the map%
\begin{align*}
\mathfrak{M}_{k}\times\left\{  1,2,3,\ldots\right\}   &  \rightarrow
\mathfrak{M}_{n+k}\setminus\mathfrak{N},\\
\left(  \mathfrak{m},i\right)   &  \mapsto\mathfrak{m}\cdot x_{i}^{n}%
\end{align*}
is well-defined (because if $\mathfrak{m}$ is a monomial of degree $k$, and if
$i\in\left\{  1,2,3,\ldots\right\}  $, then $\mathfrak{m}\cdot x_{i}^{n}$ is a
monomial of degree $k+n=n+k$, and the variable $x_{i}$ appears in it with
exponent $\geq n$). This map is furthermore surjective (for simple reasons)
and injective (in fact, if $\mathfrak{n}\in\mathfrak{M}_{n+k}\setminus
\mathfrak{N}$, then $\mathfrak{n}$ is a monomial of degree $n+k$, and thus
Observation 1 yields that there is \textbf{at most} one variable $x_{i}$ that
appears in $\mathfrak{n}$ with exponent $\geq n$; but this means that the only
preimage of $\mathfrak{n}$ under our map is $\left(  \dfrac{\mathfrak{n}%
}{x_{i}^{n}},i\right)  $). Hence, this map is a bijection. We can thus use it
to substitute $\mathfrak{m}\cdot x_{i}^{n}$ for $\mathfrak{n}$ in the sum
$\sum_{\mathfrak{n}\in\mathfrak{M}_{n+k}\setminus\mathfrak{N}}\mathfrak{n}$.
We thus obtain%
\begin{align}
\sum_{\mathfrak{n}\in\mathfrak{M}_{n+k}\setminus\mathfrak{N}}\mathfrak{n}  &
=\sum_{\left(  \mathfrak{m},i\right)  \in\mathfrak{M}_{k}\times\left\{
1,2,3,\ldots\right\}  }\mathfrak{m}\cdot x_{i}^{n}=\underbrace{\left(
\sum_{\mathfrak{m}\in\mathfrak{M}_{k}}\mathfrak{m}\right)  }_{\substack{=h_{k}%
\\\text{(by (\ref{pf.prop.h2n-1.hk=}))}}}\cdot\underbrace{\sum_{i\in\left\{
1,2,3,\ldots\right\}  }x_{i}^{n}}_{=p_{n}}\nonumber\\
&  =h_{k}p_{n}. \label{pf.prop.h2n-1.su2}%
\end{align}
But (\ref{pf.prop.h2n-1.hn+k=}) becomes%
\begin{align*}
h_{n+k}  &  =\sum_{\mathfrak{n}\in\mathfrak{M}_{n+k}}\mathfrak{n}%
=\underbrace{\sum_{\mathfrak{n}\in\mathfrak{N}}\mathfrak{n}}%
_{\substack{=G\left(  n,n+k\right)  \\\text{(by (\ref{pf.prop.h2n-1.su1pet}%
))}}}+\underbrace{\sum_{\mathfrak{n}\in\mathfrak{M}_{n+k}\setminus
\mathfrak{N}}\mathfrak{n}}_{\substack{=h_{k}p_{n}\\\text{(by
(\ref{pf.prop.h2n-1.su2}))}}}\ \ \ \ \ \ \ \ \ \ \left(  \text{since
}\mathfrak{N}\subseteq\mathfrak{M}_{n+k}\right) \\
&  =G\left(  n,n+k\right)  +h_{k}p_{n}.
\end{align*}
In other words,%
\[
G\left(  n,n+k\right)  =h_{n+k}-h_{k}p_{n}.
\]
This proves Proposition \ref{prop.h2n-1}.
\end{proof}

We note in passing that the idea used in the above proof of Proposition
\ref{prop.h2n-1} can be generalized to yield a second proof of Theorem
\ref{thm.G.frob}, using an inclusion/exclusion argument.\footnote{Here is an
outline of this second proof: For any positive integer $k$ and any
$m\in\mathbb{N}$, we have%
\begin{align*}
G\left(  k,m\right)   &  =\sum_{\substack{\alpha\in\operatorname*{WC}%
;\\\left\vert \alpha\right\vert =m;\\\alpha_{i}<k\text{ for all }i}%
}\mathbf{x}^{\alpha}=\sum_{I\subseteq\left\{  1,2,3,\ldots\right\}  }\left(
-1\right)  ^{\left\vert I\right\vert }\underbrace{\sum_{\substack{\alpha
\in\operatorname*{WC};\\\left\vert \alpha\right\vert =m;\\\alpha_{i}\geq
k\text{ for all }i\in I}}\mathbf{x}^{\alpha}}_{\substack{=\left(  \prod_{i\in
I}x_{i}^{k}\right)  \cdot\sum_{\substack{\beta\in\operatorname*{WC}%
;\\\left\vert \beta\right\vert =m-k\left\vert I\right\vert }}\mathbf{x}%
^{\beta}}}\\
&  \ \ \ \ \ \ \ \ \ \ \left(  \text{by an infinite-set version of the
inclusion-exclusion principle}\right) \\
&  =\underbrace{\sum_{I\subseteq\left\{  1,2,3,\ldots\right\}  }}_{=\sum
_{p\in\mathbb{N}}\ \ \sum_{\substack{I\subseteq\left\{  1,2,3,\ldots\right\}
;\\\left\vert I\right\vert =p}}}\left(  -1\right)  ^{\left\vert I\right\vert
}\left(  \prod_{i\in I}x_{i}^{k}\right)  \cdot\underbrace{\sum
_{\substack{\beta\in\operatorname*{WC};\\\left\vert \beta\right\vert
=m-k\left\vert I\right\vert }}\mathbf{x}^{\beta}}_{=h_{m-k\left\vert
I\right\vert }}\\
&  =\sum_{p\in\mathbb{N}}\ \ \sum_{\substack{I\subseteq\left\{  1,2,3,\ldots
\right\}  ;\\\left\vert I\right\vert =p}}\underbrace{\left(  -1\right)
^{\left\vert I\right\vert }}_{=\left(  -1\right)  ^{p}}\left(  \prod_{i\in
I}x_{i}^{k}\right)  \cdot\underbrace{h_{m-k\left\vert I\right\vert }%
}_{\substack{=h_{m-kp}\\\text{(since }\left\vert I\right\vert =p\text{)}}}\\
&  =\sum_{p\in\mathbb{N}}\left(  -1\right)  ^{p}h_{m-kp}\underbrace{\sum
_{\substack{I\subseteq\left\{  1,2,3,\ldots\right\}  ;\\\left\vert
I\right\vert =p}}\ \ \prod_{i\in I}x_{i}^{k}}_{\substack{=\mathbf{f}%
_{k}\left(  e_{p}\right)  \\\text{(this is easy to check)}}}=\sum
_{p\in\mathbb{N}}\left(  -1\right)  ^{p}h_{m-kp}\cdot\mathbf{f}_{k}\left(
e_{p}\right) \\
&  =\sum_{i\in\mathbb{N}}\left(  -1\right)  ^{i}h_{m-ki}\cdot\mathbf{f}%
_{k}\left(  e_{i}\right)  .
\end{align*}
}

\begin{corollary}
\label{cor.LHS2}Let $n$ be an integer such that $n>1$. Then:

\textbf{(a)} We have%
\[
\sum_{\substack{\lambda\in\operatorname*{Par}\nolimits_{n};\\\left(
n-1,1\right)  \triangleright\lambda}}m_{\lambda}=h_{n}-p_{n}.
\]

\textbf{(b)} We have%
\[
\sum_{\substack{\lambda\in\operatorname*{Par}\nolimits_{2n-1};\\\left(
n-1,n-1,1\right)  \triangleright\lambda}}m_{\lambda}=h_{2n-1}-h_{n-1}p_{n}.
\]

\end{corollary}

\begin{proof}
\textbf{(b)} Corollary \ref{cor.LHS1} \textbf{(b)} yields%
\begin{align*}
\sum_{\substack{\lambda\in\operatorname*{Par}\nolimits_{2n-1};\\\left(
n-1,n-1,1\right)  \triangleright\lambda}}m_{\lambda}  &  =G\left(
n,2n-1\right)  =G\left(  n,n+\left(  n-1\right)  \right)
\ \ \ \ \ \ \ \ \ \ \left(  \text{since }2n-1=n+\left(  n-1\right)  \right) \\
&  =h_{n+\left(  n-1\right)  }-h_{n-1}p_{n}\ \ \ \ \ \ \ \ \ \ \left(
\text{by Proposition \ref{prop.h2n-1}, applied to }k=n-1\right) \\
&  =h_{2n-1}-h_{n-1}p_{n}.
\end{align*}
This proves Corollary \ref{cor.LHS2} \textbf{(b)}.

\textbf{(a)} Corollary \ref{cor.LHS1} \textbf{(a)} yields%
\begin{align*}
\sum_{\substack{\lambda\in\operatorname*{Par}\nolimits_{n};\\\left(
n-1,1\right)  \triangleright\lambda}}m_{\lambda}  &  =G\left(  n,n\right)
=G\left(  n,n+0\right) \\
&  =\underbrace{h_{n+0}}_{=h_{n}}-\underbrace{h_{0}}_{=1}p_{n}%
\ \ \ \ \ \ \ \ \ \ \left(  \text{by Proposition \ref{prop.h2n-1}, applied to
}k=0\right) \\
&  =h_{n}-p_{n}.
\end{align*}
This proves Corollary \ref{cor.LHS2} \textbf{(a)}.
\end{proof}

Our next claim is an easy consequence of Proposition \ref{prop.p-as-sum}:

\begin{corollary}
\label{cor.p-as-sum-2}Let $n$ be a positive integer. Then,
\[
h_{n}-p_{n}=\sum_{i=0}^{n-2}\left(  -1\right)  ^{i}s_{\left(  n-1-i,1^{i+1}%
\right)  }.
\]

\end{corollary}

\begin{proof}
Proposition \ref{prop.p-as-sum} yields%
\begin{align*}
p_{n}  &  =\sum_{i=0}^{n-1}\left(  -1\right)  ^{i}s_{\left(  n-i,1^{i}\right)
}=\underbrace{\left(  -1\right)  ^{0}}_{=1}\underbrace{s_{\left(
n-0,1^{0}\right)  }}_{=s_{\left(  n-0\right)  }=s_{\left(  n\right)  }=h_{n}%
}+\sum_{i=1}^{n-1}\left(  -1\right)  ^{i}s_{\left(  n-i,1^{i}\right)  }\\
&  =h_{n}+\sum_{i=1}^{n-1}\left(  -1\right)  ^{i}s_{\left(  n-i,1^{i}\right)
},
\end{align*}
so that%
\begin{align*}
h_{n}-p_{n}  &  =-\sum_{i=1}^{n-1}\left(  -1\right)  ^{i}s_{\left(
n-i,1^{i}\right)  }=\sum_{i=1}^{n-1}\underbrace{\left(  -\left(  -1\right)
^{i}\right)  }_{=\left(  -1\right)  ^{i-1}}s_{\left(  n-i,1^{i}\right)  }%
=\sum_{i=1}^{n-1}\left(  -1\right)  ^{i-1}s_{\left(  n-i,1^{i}\right)  }\\
&  =\sum_{i=0}^{n-2}\left(  -1\right)  ^{i}s_{\left(  n-1-i,1^{i+1}\right)  }%
\end{align*}
(here, we have substituted $i+1$ for $i$ in the sum).
\end{proof}

We can now immediately prove Theorem \ref{thm.claim} \textbf{(a)}:

\begin{proof}
[Proof of Theorem \ref{thm.claim} \textbf{(a)}.]Corollary \ref{cor.LHS2}
\textbf{(a)} yields%
\[
\sum_{\substack{\lambda\in\operatorname*{Par}\nolimits_{n};\\\left(
n-1,1\right)  \triangleright\lambda}}m_{\lambda}=h_{n}-p_{n}=\sum_{i=0}%
^{n-2}\left(  -1\right)  ^{i}s_{\left(  n-1-i,1^{i+1}\right)  }%
\ \ \ \ \ \ \ \ \ \ \left(  \text{by Corollary \ref{cor.p-as-sum-2}}\right)
.
\]
This proves Theorem \ref{thm.claim} \textbf{(a)}.
\end{proof}

We shall use the \emph{skewing operators} $f^{\perp}:\Lambda\rightarrow
\Lambda$ for all $f\in\Lambda$ as defined in \cite[\S 2.8]{GriRei} or in
\cite[Chapter I,\ Section 5, Example 3]{Macdon95}. The easiest way to define
them (following \cite[Chapter I,\ Section 5, Example 3]{Macdon95}) is as
follows: For each $f\in\Lambda$, we let $f^{\perp}:\Lambda\rightarrow\Lambda$
be the $\mathbf{k}$-linear map adjoint to the map $L_{f}:\Lambda
\rightarrow\Lambda,\ g\mapsto fg$ (that is, to the map that multiplies every
element of $\Lambda$ by $f$) with respect to the Hall inner product. That is,
$f^{\perp}$ is the $\mathbf{k}$-linear map from $\Lambda$ to $\Lambda$ that
satisfies
\[
\left\langle g,f^{\perp}\left(  a\right)  \right\rangle =\left\langle
fg,a\right\rangle \ \ \ \ \ \ \ \ \ \ \text{for all }a\in\Lambda\text{ and
}g\in\Lambda.
\]
It is not hard to show that such an operator $f^{\perp}$ exists\footnote{This
is not completely automatic: Not every $\mathbf{k}$-linear map from $\Lambda$
to $\Lambda$ has an adjoint with respect to the Hall inner product! (For
example, the $\mathbf{k}$-linear map $\Lambda\rightarrow\Lambda$ that sends
each Schur function $s_{\lambda}$ to $1$ has none.) The reason why the map
$L_{f}:\Lambda\rightarrow\Lambda,\ g\mapsto fg$ has an adjoint is that when
$f$ is homogeneous of degree $k$, this map $L_{f}$ sends each graded component
$\Lambda_{m}$ of $\Lambda$ to $\Lambda_{m+k}$, and both of these graded
components $\Lambda_{m}$ and $\Lambda_{m+k}$ are $\mathbf{k}$-modules with
\textbf{finite} bases. (The case when $f$ is not homogeneous can be reduced to
the case when $f$ is homogeneous, since each $f\in\Lambda$ is a sum of
finitely many homogeneous elements.)}. The definition of $f^{\perp}$ in
\cite[\S 2.8]{GriRei} is different but equivalent (because of
\cite[Proposition 2.8.2(i)]{GriRei}). One of the most important properties of
skewing operators is the following fact (\cite[(2.8.2)]{GriRei}):

\begin{lemma}
\label{lem.skewing.schur}Let $\lambda$ and $\mu$ be any two partitions. Then,%
\begin{equation}
s_{\mu}^{\perp}\left(  s_{\lambda}\right)  =s_{\lambda/\mu}.
\label{pf.lem.Bmhn.skew}%
\end{equation}
(Here, $s_{\lambda/\mu}$ is a skew Schur function, defined in Subsection
\ref{subsect.proofs.skew}.)
\end{lemma}

Using skewing operators, we can define another helpful family of operators on
$\Lambda$:

\begin{definition}
For any $m\in\mathbb{Z}$, we define a map $\mathbf{B}_{m}:\Lambda
\rightarrow\Lambda$ by setting%
\[
\mathbf{B}_{m}\left(  f\right)  =\sum_{i\in\mathbb{N}}\left(  -1\right)
^{i}h_{m+i}e_{i}^{\perp}f\ \ \ \ \ \ \ \ \ \ \text{for all }f\in
\Lambda\text{.}%
\]
It is known (\cite[Exercise 2.9.1(a)]{GriRei}) that this map $\mathbf{B}_{m}$
is well-defined and $\mathbf{k}$-linear.
\end{definition}

(Actually, the well-definedness of $\mathbf{B}_{m}$ is easy to check: If
$f\in\Lambda$ has degree $d$, then all integers $i>d$ satisfy $e_{i}^{\perp
}f=0$ for degree reasons, and thus the sum $\sum_{i\in\mathbb{N}}\left(
-1\right)  ^{i}h_{m+i}e_{i}^{\perp}f$ has only finitely many nonzero addends.
The $\mathbf{k}$-linearity of $\mathbf{B}_{m}$ is even clearer.)

The operators $\mathbf{B}_{m}$ for $m\in\mathbb{Z}$ have first appeared in
Zelevinsky's \cite[\S 4.20]{Zelevi81} (in the different-looking but secretly
equivalent setting of a PSH-algebra), where they are credited to J. N.
Bernstein. They have since been dubbed the \emph{Bernstein creation operators}
and proved useful in various contexts (e.g., the definition of the
\textquotedblleft dual immaculate functions\textquotedblright\ in
\cite{BBSSZ13} takes them for inspiration). One of their most fundamental
properties is the following fact (which originates in \cite[4.20, ($\ast\ast
$)]{Zelevi81} and appears implicitly in \cite[Chapter I, Section 5, Example
29]{Macdon95}):

\begin{proposition}
\label{prop.Bmslam}Let $\lambda$ be any partition. Let $m\in\mathbb{Z}$
satisfy $m\geq\lambda_{1}$. Then,%
\begin{equation}
\sum_{i\in\mathbb{N}}\left(  -1\right)  ^{i}h_{m+i}e_{i}^{\perp}s_{\lambda
}=s_{\left(  m,\lambda_{1},\lambda_{2},\lambda_{3},\ldots\right)  }.
\label{eq.schur-row-adder-1}%
\end{equation}

\end{proposition}

See \cite[Exercise 2.9.1(b)]{GriRei} for a proof of Proposition
\ref{prop.Bmslam}. Thus, if $\lambda$ is any partition, and if $m\in
\mathbb{Z}$ satisfies $m\geq\lambda_{1}$, then%
\begin{align}
\mathbf{B}_{m}\left(  s_{\lambda}\right)   &  =\sum_{i\in\mathbb{N}}\left(
-1\right)  ^{i}h_{m+i}e_{i}^{\perp}s_{\lambda}\ \ \ \ \ \ \ \ \ \ \left(
\text{by the definition of }\mathbf{B}_{m}\right) \nonumber\\
&  =s_{\left(  m,\lambda_{1},\lambda_{2},\lambda_{3},\ldots\right)
}\ \ \ \ \ \ \ \ \ \ \left(  \text{by (\ref{eq.schur-row-adder-1})}\right)  .
\label{eq.schur-row-adder-2}%
\end{align}

\begin{lemma}
\label{lem.Bmhn}Let $n$ be a positive integer. Let $m\in\mathbb{N}$. Then,
$\mathbf{B}_{m}\left(  h_{n}\right)  =h_{m}h_{n}-h_{m+1}h_{n-1}$.
\end{lemma}

\begin{proof}
[Proof of Lemma \ref{lem.Bmhn}.]We have $e_{0}=1$ and thus $e_{0}^{\perp
}=1^{\perp}=\operatorname*{id}$. Hence, $e_{0}^{\perp}\left(  h_{n}\right)
=h_{n}$.

We shall use the notion of skew Schur functions $s_{\lambda/\mu}$ (as in
Subsection \ref{subsect.proofs.skew}). Recall that $s_{\lambda/\mu}=0$ when
$\mu\not \subseteq \lambda$.

From $e_{1}=s_{\left(  1\right)  }$ and $h_{n}=s_{\left(  n\right)  }$, we
obtain%
\[
e_{1}^{\perp}\left(  h_{n}\right)  =s_{\left(  1\right)  }^{\perp}\left(
s_{\left(  n\right)  }\right)  =s_{\left(  n\right)  /\left(  1\right)
}\ \ \ \ \ \ \ \ \ \ \left(  \text{by (\ref{pf.lem.Bmhn.skew})}\right)  .
\]

But it is easy to see that $s_{\left(  n\right)  /\left(  1\right)
}=s_{\left(  n-1\right)  }$. (Indeed, this follows from the combinatorial
definition of skew Schur functions, since the skew Ferrers diagram of $\left(
n\right)  /\left(  1\right)  $ can be obtained from the Ferrers diagram of
$\left(  n-1\right)  $ by parallel shift\footnote{See \cite[\S 2.3]{GriRei}
for the notions we are using here.}. Alternatively, this follows easily from
Theorem \ref{thm.JTs.h}, because $s_{\left(  n-1\right)  }=h_{n-1}$.)

Thus, we obtain%
\[
e_{1}^{\perp}\left(  h_{n}\right)  =s_{\left(  n\right)  /\left(  1\right)
}=s_{\left(  n-1\right)  }=h_{n-1}.
\]

For each integer $i>1$, we have%
\begin{align}
e_{i}^{\perp}\left(  h_{n}\right)   &  =s_{\left(  1^{i}\right)  }^{\perp
}\left(  s_{\left(  n\right)  }\right)  \ \ \ \ \ \ \ \ \ \ \left(
\text{since }e_{i}=s_{\left(  1^{i}\right)  }\text{ and }h_{n}=s_{\left(
n\right)  }\right) \nonumber\\
&  =s_{\left(  n\right)  /\left(  1^{i}\right)  }\ \ \ \ \ \ \ \ \ \ \left(
\text{by (\ref{pf.lem.Bmhn.skew})}\right) \nonumber\\
&  =0\ \ \ \ \ \ \ \ \ \ \left(  \text{since }\left(  1^{i}\right)
\not \subseteq \left(  n\right)  \text{ (because }i>1\text{)}\right)  .
\label{pf.lem.Bmhn.3}%
\end{align}

Now, the definition of $\mathbf{B}_{m}$ yields%
\begin{align*}
\mathbf{B}_{m}\left(  h_{n}\right)   &  =\sum_{i\in\mathbb{N}}\left(
-1\right)  ^{i}h_{m+i}e_{i}^{\perp}\left(  h_{n}\right) \\
&  =\underbrace{\left(  -1\right)  ^{0}}_{=1}\underbrace{h_{m+0}}_{=h_{m}%
}\underbrace{e_{0}^{\perp}\left(  h_{n}\right)  }_{=h_{n}}+\underbrace{\left(
-1\right)  ^{1}}_{=-1}h_{m+1}\underbrace{e_{1}^{\perp}\left(  h_{n}\right)
}_{=h_{n-1}}+\sum_{i\geq2}\left(  -1\right)  ^{i}h_{m+i}\underbrace{e_{i}%
^{\perp}\left(  h_{n}\right)  }_{\substack{=0\\\text{(by (\ref{pf.lem.Bmhn.3}%
))}}}\\
&  =h_{m}h_{n}-h_{m+1}h_{n-1}.
\end{align*}

\end{proof}

\begin{corollary}
\label{cor.Bn-1Hn}Let $n$ be a positive integer. Then, $\mathbf{B}%
_{n-1}\left(  h_{n}\right)  =0$.
\end{corollary}

\begin{proof}
Apply Lemma \ref{lem.Bmhn} to $m=n-1$ and simplify.
\end{proof}

\begin{lemma}
\label{lem.Bmpn}Let $m\in\mathbb{N}$. Let $n$ be a positive integer. Then,
$\mathbf{B}_{m}\left(  p_{n}\right)  =h_{m}p_{n}-h_{m+n}$.
\end{lemma}

\begin{proof}
This is \cite[Exercise 2.9.1(f)]{GriRei}. But here is a more direct proof: We
will use the comultiplication $\Delta:\Lambda\rightarrow\Lambda\otimes\Lambda$
of the Hopf algebra $\Lambda$ (see \cite[\S 2.3]{GriRei}). Here and in the
following, the \textquotedblleft$\otimes$\textquotedblright\ sign denotes
$\otimes_{\mathbf{k}}$. The power-sum symmetric function $p_{n}$ is
primitive\footnote{Recall that an element $x$ of a Hopf algebra $H$ is said to
be \textit{primitive} if the comultiplication $\Delta_{H}$ of $H$ satisfies
$\Delta_{H}\left(  x\right)  =1\otimes x+x\otimes1$.} (see \cite[Proposition
2.3.6(i)]{GriRei}); thus,
\[
\Delta\left(  p_{n}\right)  =1\otimes p_{n}+p_{n}\otimes1.
\]
Hence, for each $i\in\mathbb{N}$, the definition of $e_{i}^{\perp}$ given in
\cite[Definition 2.8.1]{GriRei} (not the equivalent definition we gave above)
yields%
\begin{equation}
e_{i}^{\perp}\left(  p_{n}\right)  =\left\langle e_{i},1\right\rangle
p_{n}+\left\langle e_{i},p_{n}\right\rangle 1. \label{pf.lem.Bmpn.eipn}%
\end{equation}

Now, the definition of $\mathbf{B}_{m}$ yields%
\begin{align*}
\mathbf{B}_{m}\left(  p_{n}\right)   &  =\sum_{i\in\mathbb{N}}\left(
-1\right)  ^{i}h_{m+i}\underbrace{e_{i}^{\perp}\left(  p_{n}\right)
}_{\substack{=\left\langle e_{i},1\right\rangle p_{n}+\left\langle e_{i}%
,p_{n}\right\rangle 1\\\text{(by (\ref{pf.lem.Bmpn.eipn}))}}}\\
&  =\sum_{i\in\mathbb{N}}\left(  -1\right)  ^{i}h_{m+i}\left(  \left\langle
e_{i},1\right\rangle p_{n}+\left\langle e_{i},p_{n}\right\rangle 1\right) \\
&  =\underbrace{\sum_{i\in\mathbb{N}}\left(  -1\right)  ^{i}h_{m+i}%
\cdot\left\langle e_{i},1\right\rangle p_{n}}_{\substack{=\left(  -1\right)
^{0}h_{m+0}\cdot\left\langle e_{0},1\right\rangle p_{n}\\\text{(because the
Hall inner product }\left\langle e_{i},1\right\rangle \\\text{equals }0\text{
whenever }i\neq0\text{ (by (\ref{eq.Hall.graded})),}\\\text{and thus the only
nonzero addend of this}\\\text{sum is the addend for }i=0\text{)}%
}}+\underbrace{\sum_{i\in\mathbb{N}}\left(  -1\right)  ^{i}h_{m+i}%
\cdot\left\langle e_{i},p_{n}\right\rangle 1}_{\substack{=\left(  -1\right)
^{n}h_{m+n}\cdot\left\langle e_{n},p_{n}\right\rangle 1\\\text{(because the
Hall inner product }\left\langle e_{i},p_{n}\right\rangle \\\text{equals
}0\text{ whenever }i\neq n\text{ (by (\ref{eq.Hall.graded})),}\\\text{and thus
the only nonzero addend of this}\\\text{sum is the addend for }i=n\text{)}}}\\
&  =\underbrace{\left(  -1\right)  ^{0}}_{=1}\underbrace{h_{m+0}}_{=h_{m}%
}\cdot\underbrace{\left\langle e_{0},1\right\rangle }_{=\left\langle
1,1\right\rangle =1}p_{n}+\left(  -1\right)  ^{n}h_{m+n}\cdot
\underbrace{\left\langle e_{n},p_{n}\right\rangle }_{\substack{=\left(
-1\right)  ^{n-1}\\\text{(by Proposition \ref{prop.ejpj})}}}1\\
&  =h_{m}p_{n}+\underbrace{\left(  -1\right)  ^{n}h_{m+n}\cdot\left(
-1\right)  ^{n-1}1}_{=-h_{m+n}}=h_{m}p_{n}-h_{m+n}.
\end{align*}

\end{proof}

\begin{lemma}
\label{lem.Bm1}Let $n$ be a positive integer. Then,%
\[
\mathbf{B}_{n-1}\left(  h_{n}-p_{n}\right)  =h_{2n-1}-h_{n-1}p_{n}.
\]

\end{lemma}

\begin{proof}
The map $\mathbf{B}_{n-1}$ is $\mathbf{k}$-linear. Thus,%
\begin{align*}
\mathbf{B}_{n-1}\left(  h_{n}-p_{n}\right)   &  =\underbrace{\mathbf{B}%
_{n-1}\left(  h_{n}\right)  }_{\substack{=0\\\text{(by Corollary
\ref{cor.Bn-1Hn})}}}-\underbrace{\mathbf{B}_{n-1}\left(  p_{n}\right)
}_{\substack{=h_{n-1}p_{n}-h_{\left(  n-1\right)  +n}\\\text{(by Lemma
\ref{lem.Bmpn},}\\\text{applied to }m=n-1\text{)}}}\\
&  =-\left(  h_{n-1}p_{n}-h_{\left(  n-1\right)  +n}\right)
=\underbrace{h_{\left(  n-1\right)  +n}}_{=h_{2n-1}}-h_{n-1}p_{n}%
=h_{2n-1}-h_{n-1}p_{n}.
\end{align*}

\end{proof}

\begin{lemma}
\label{lem.Bm2}Let $n$ be a positive integer. Then,%
\[
\mathbf{B}_{n-1}\left(  h_{n}-p_{n}\right)  =\sum_{i=0}^{n-2}\left(
-1\right)  ^{i}s_{\left(  n-1,n-1-i,1^{i+1}\right)  }.
\]

\end{lemma}

\begin{proof}
[Proof of Lemma \ref{lem.Bm2}.]We have
\begin{align*}
\mathbf{B}_{n-1}\left(  h_{n}-p_{n}\right)   &  =\mathbf{B}_{n-1}\left(
\sum_{i=0}^{n-2}\left(  -1\right)  ^{i}s_{\left(  n-1-i,1^{i+1}\right)
}\right)  \ \ \ \ \ \ \ \ \ \ \left(  \text{by Corollary \ref{cor.p-as-sum-2}%
}\right) \\
&  =\sum_{i=0}^{n-2}\left(  -1\right)  ^{i}\underbrace{\mathbf{B}_{n-1}\left(
s_{\left(  n-1-i,1^{i+1}\right)  }\right)  }_{\substack{=s_{\left(
n-1,n-1-i,1^{i+1}\right)  }\\\text{(by (\ref{eq.schur-row-adder-2}), applied
to }m=n-1\\\text{and }\lambda=\left(  n-1-i,1^{i+1}\right)  \\\text{(since
}n-1\geq n-1-i\text{))}}}\ \ \ \ \ \ \ \ \ \ \left(  \text{since }%
\mathbf{B}_{n-1}\text{ is }\mathbf{k}\text{-linear}\right) \\
&  =\sum_{i=0}^{n-2}\left(  -1\right)  ^{i}s_{\left(  n-1,n-1-i,1^{i+1}%
\right)  }.
\end{align*}

\end{proof}

Now the proof of Theorem \ref{thm.claim} \textbf{(b)} is a trivial
concatenation of equalities:

\begin{proof}
[Proof of Theorem \ref{thm.claim} \textbf{(b)}.]Corollary \ref{cor.LHS2}
\textbf{(b)} yields%
\begin{align*}
\sum_{\substack{\lambda\in\operatorname*{Par}\nolimits_{2n-1};\\\left(
n-1,n-1,1\right)  \triangleright\lambda}}m_{\lambda}  &  =h_{2n-1}%
-h_{n-1}p_{n}=\mathbf{B}_{n-1}\left(  h_{n}-p_{n}\right)
\ \ \ \ \ \ \ \ \ \ \left(  \text{by Lemma \ref{lem.Bm1}}\right) \\
&  =\sum_{i=0}^{n-2}\left(  -1\right)  ^{i}s_{\left(  n-1,n-1-i,1^{i+1}%
\right)  }\ \ \ \ \ \ \ \ \ \ \left(  \text{by Lemma \ref{lem.Bm2}}\right)  .
\end{align*}

\end{proof}

\section{\label{sect.fin}Final remarks}

\subsection{SageMath code}

The SageMath computer algebra system \cite{SageMath} does not (yet) natively
know the Petrie symmetric functions $G\left(  k,m\right)  $; but they can be
easily constructed in it. For example, the code that follows computes
$G\left(  k, n\right)  $ expanded in the Schur basis:

\begin{python}
Sym = SymmetricFunctions(QQ) # Replace QQ by your favorite base ring.
m = Sym.m() # monomial symmetric functions
s = Sym.s() # Schur functions
	
def G(k, n): # a Petrie function
	return s(m.sum(m[lam] for lam in Partitions(n, max_part=k-1)))
\end{python}

\subsection{Understanding the Petrie numbers}

Combining Corollary \ref{cor.Gkm.main} with Theorem \ref{thm.petk.explicit}
yields an explicit expression of all coefficients in the expansion of a Petrie
symmetric function $G\left(  k,m\right)  $ in the Schur basis. It would stand
to reason if the identity in Theorem \ref{thm.claim} \textbf{(b)} (whose left
hand side is $G\left(  n,2n-1\right)  $) could be obtained from this
expression. Surprisingly, we have been unable to do so, which suggests that
the description of $\operatorname*{pet}\nolimits_{k}\left(  \lambda
,\varnothing\right)  $ in Theorem \ref{thm.petk.explicit} might not be optimal.

As to $\operatorname*{pet}\nolimits_{k}\left(  \lambda,\mu\right)  $, we do
not have an explicit description at all, unless we count the recursive one
that can be extracted from the proof in \cite{GorWil74}.

\subsection{MNable symmetric functions}

Combining Theorem \ref{thm.G.pieri} with Proposition \ref{prop.petkrel.-101},
we conclude that for any $k>0$ and $m\in\mathbb{N}$, the symmetric function
$G\left(  k,m\right)  \in\Lambda$ has the following property: For any $\mu
\in\operatorname*{Par}$, its product $G\left(  k,m\right)  \cdot s_{\mu}$ with
$s_{\mu}$ can be written in the form $\sum_{\lambda\in\operatorname*{Par}%
}u_{\lambda}s_{\lambda}$ with $u_{\lambda}\in\left\{  -1,0,1\right\}  $ for
all $\lambda\in\operatorname*{Par}$. It has this property in common with the
symmetric functions $h_{m}$ and $e_{m}$ (according to the Pieri rules) and
$p_{m}$ (according to the Murnaghan--Nakayama rule) as well as several others.
The study of symmetric functions having this property -- which we call
\textit{MNable symmetric functions} (in honor of Murnaghan and Nakayama) --
has been initiated in \cite[\S 8]{Grinbe20}, but there is much to be done.

\subsection{A conjecture of Per Alexandersson}

In February 2020, Per Alexandersson suggested the following conjecture:

\begin{conjecture}
\label{conj.Gkmp2}Let $k$ be a positive integer, and $m\in\mathbb{N}$. Then,
$G\left(  k,m\right)  \cdot p_{2}\in\Lambda$ can be written in the form
$\sum_{\lambda\in\operatorname*{Par}}u_{\lambda}s_{\lambda}$ with $u_{\lambda
}\in\left\{  -1,0,1\right\}  $ for all $\lambda\in\operatorname*{Par}$.
\end{conjecture}

For example,%
\[
G\left(  3,4\right)  \cdot p_{2}=s_{\left(  1,1,1,1,1,1\right)  }+s_{\left(
2,2,2\right)  }-s_{\left(  3,1,1,1\right)  }-s_{\left(  3,3\right)
}+s_{\left(  4,2\right)  }.
\]

Conjecture \ref{conj.Gkmp2} has been verified for all $k$ and $m$ satisfying
$k+m\leq30$.

Note that Conjecture \ref{conj.Gkmp2} becomes false if $p_{2}$ is replaced by
$p_{3}$. For example,%
\[
G\left(  3,4\right)  \cdot p_{3}=-s_{\left(  1,1,1,1,1,1,1\right)
}+s_{\left(  2,2,1,1,1\right)  }-2s_{\left(  2,2,2,1\right)  }+s_{\left(
3,2,1,1\right)  }-s_{\left(  4,1,1,1\right)  }-s_{\left(  4,3\right)
}+s_{\left(  5,2\right)  }.
\]

\subsection{A conjecture of Fran\c{c}ois Bergeron}

An even more mysterious conjecture was suggested by Fran\c{c}ois Bergeron in
April 2020:

\begin{conjecture}
\label{conj.nabla-fb}Let $k$ and $n$ be positive integers, and $m\in
\mathbb{N}$. Let $\nabla$ be the \emph{nabla operator} as defined (e.g.) in
\cite[\S 3.2.1]{Berger19}. Then, there exists a sign $\sigma_{n,k,m}%
\in\left\{  1,-1\right\}  $ such that $\sigma_{n,k,m}\nabla^{n}\left(
G\left(  k,m\right)  \right)  $ is an $\mathbb{N}\left[  q,t\right]  $-linear
combination of Schur functions.
\end{conjecture}

Using SageMath, this conjecture has been checked for $n=1$ and all
$k,m\in\left\{  0,1,\ldots,9\right\}  $; the signs $\sigma_{1,k,m}$ are given
by the following table:%
\[%
\begin{tabular}
[c]{||c||ccccccccc}\hline\hline
& $1$ & $2$ & $3$ & $4$ & $5$ & $6$ & $7$ & $8$ & $9$\\\hline\hline
$2$ & $+$ & $+$ & $+$ & $+$ & $+$ & $+$ & $+$ & $+$ & $+$\\
$3$ & $+$ & $-$ & $-$ & $-$ & $+$ & $+$ & $+$ & $-$ & $-$\\
$4$ & $+$ & $-$ & $+$ & $+$ & $+$ & $-$ & $+$ & $+$ & $+$\\
$5$ & $+$ & $-$ & $+$ & $-$ & $-$ & $-$ & $+$ & $-$ & $+$\\
$6$ & $+$ & $-$ & $+$ & $-$ & $+$ & $+$ & $+$ & $-$ & $+$\\
$7$ & $+$ & $-$ & $+$ & $-$ & $+$ & $-$ & $-$ & $-$ & $+$\\
$8$ & $+$ & $-$ & $+$ & $-$ & $+$ & $-$ & $+$ & $+$ & $+$\\
$9$ & $+$ & $-$ & $+$ & $-$ & $+$ & $-$ & $+$ & $-$ & $-$%
\end{tabular}
\]
(where the entry in the row indexed $k$ and the column indexed $m$ is the sign
$\sigma_{1,k,m}$, represented by a \textquotedblleft$+$\textquotedblright%
\ sign if it is $1$ and by a \textquotedblleft$-$\textquotedblright\ sign if
it is $-1$). I am not aware of a pattern in these signs, apart from the fact
that $\sigma_{1,2,m}=1$ for all $m\in\mathbb{N}$ (a consequence of Haiman's
famous interpretation of $\nabla\left(  e_{m}\right)  $ as a character), and
that $\sigma_{1,k,m}$ appears to be $\left(  -1\right)  ^{m-1}$ whenever
$1\leq m<k$ (which would follow from the conjecture that $\left(  -1\right)
^{m-1}\nabla\left(  h_{m}\right)  $ is an $\mathbb{N}\left[  q,t\right]
$-linear combination of Schur functions for any $m\geq1$).

\subsection{\textquotedblleft Petriefication\textquotedblright\ of Schur
functions}

Theorem \ref{thm.Uk.main} shows the existence of a Hopf algebra homomorphism
$V_{k}:\Lambda\rightarrow\Lambda$ that sends the complete homogeneous
symmetric functions $h_{1},h_{2},h_{3},\ldots$ to the Petrie symmetric
functions $G\left(  k,1\right)  ,G\left(  k,2\right)  ,G\left(  k,3\right)
,\ldots$. It thus is natural to consider the images of all Schur functions
$s_{\lambda}$ under this homomorphism $V_{k}$. Experiments with small
$\lambda$'s may suggest that these images $V_{k}\left(  s_{\lambda}\right)  $
all can be written in the form $\sum_{\lambda\in\operatorname*{Par}}%
u_{\lambda}s_{\lambda}$ with $u_{\lambda}\in\left\{  -1,0,1\right\}  $. But
this is not generally the case; counterexamples include $V_{3}\left(
s_{\left(  4,4,4\right)  }\right)  $, $V_{4}\left(  s_{\left(  4,4\right)
}\right)  $ and $V_{4}\left(  s_{\left(  5,1,1,1,1\right)  }\right)  $. (Of
course, it is true when $\lambda$ is a single row, because of $V_{k}\left(
s_{\left(  m\right)  }\right)  =V_{k}\left(  h_{m}\right)  =G\left(
k,m\right)  $; and it is also true when $\lambda$ is a single column, because
the Hopf algebra homomorphism $V_{k}$ commutes with the antipode $S$ that
sends $h_{m}\mapsto\left(  -1\right)  ^{m}e_{m}$ and $s_{\lambda}%
\mapsto\left(  -1\right)  ^{\left\vert \lambda\right\vert }s_{\lambda^{t}}$.)

Note that these images $V_{k}\left(  s_{\lambda}\right)  $ are precisely the
\emph{modular Schur functions} $s_{\lambda}^{\prime}$ studied in
\cite{Walker94}.

\subsection{Postnikov's generalization}

At the MIT Algebraic Combinatorics preseminar roundtable (2020), Alexander
Postnikov has suggested a generalization of the Petrie symmetric functions
that preserves some of their more elementary properties. In this subsection,
we shall survey this generalization.

\begin{vershort}
Proofs will be sketched (at best); the reader can find the details in the
detailed version \cite{verlong}.
\end{vershort}

\begin{convention}
We fix a formal power series $F\in\mathbf{k}\left[  \left[  t\right]  \right]
$ whose constant term is $1$. (We will keep this $F$ fixed throughout the
present subsection.)
\end{convention}

The notations in the following definition will also be used throughout this subsection:

\begin{definition}
\label{def.GF.GF}\textbf{(a)} Let $f_{0},f_{1},f_{2},\ldots$ be the
coefficients of the formal power series $F$, so that $F=\sum_{n\in\mathbb{N}%
}f_{n}t^{n}$. Thus, $f_{0}$ is the constant term of $F$; hence, $f_{0}=1$
(since the constant term of $F$ is $1$).

\textbf{(b)} We set $f_{i}=0$ for each negative integer $i$.

\textbf{(c)} For any weak composition $\alpha$, we define an element
$f_{\alpha}\in\mathbf{k}$ by%
\[
f_{\alpha}=f_{\alpha_{1}}f_{\alpha_{2}}f_{\alpha_{3}}\cdots.
\]
(Here, the infinite product $f_{\alpha_{1}}f_{\alpha_{2}}f_{\alpha_{3}}\cdots$
is well-defined, since every sufficiently high positive integer $i$ satisfies
$\alpha_{i}=0$ and thus $f_{\alpha_{i}}=f_{0}=1$.)

\textbf{(d)} We define the power series%
\begin{equation}
G_{F}=\sum_{\alpha\in\operatorname*{WC}}f_{\alpha}\mathbf{x}^{\alpha}.
\label{eq.GF=}%
\end{equation}
This is a formal power series in $\mathbf{k}\left[  \left[  x_{1},x_{2}%
,x_{3},\ldots\right]  \right]  $.

\textbf{(e)} For any $m\in\mathbb{N}$, we define the power series%
\begin{equation}
G_{F,m}=\sum_{\substack{\alpha\in\operatorname*{WC};\\\left\vert
\alpha\right\vert =m}}f_{\alpha}\mathbf{x}^{\alpha}. \label{eq.GFm=}%
\end{equation}
This is a formal power series in $\mathbf{k}\left[  \left[  x_{1},x_{2}%
,x_{3},\ldots\right]  \right]  $.
\end{definition}

\begin{example}
\label{exa.GF.GF123}Let us see how these power series $G_{F}$ and $G_{F,m}$
look for specific values of $F$.

\textbf{(a)} Let $F=\dfrac{1}{1-t}=1+t+t^{2}+t^{3}+\cdots$. Then, $f_{i}=1$
for each $i\in\mathbb{N}$. Hence, $f_{\alpha}=\underbrace{1\cdot1\cdot
1\cdot\cdots}_{=1}=1$ for any weak composition $\alpha$. Thus,
\[
G_{F}=\sum_{\alpha\in\operatorname*{WC}}\underbrace{f_{\alpha}}_{=1}%
\mathbf{x}^{\alpha}=\sum_{\alpha\in\operatorname*{WC}}\mathbf{x}^{\alpha}%
\]
and%
\[
G_{F,m}=\sum_{\substack{\alpha\in\operatorname*{WC};\\\left\vert
\alpha\right\vert =m}}\underbrace{f_{\alpha}}_{=1}\mathbf{x}^{\alpha}%
=\sum_{\substack{\alpha\in\operatorname*{WC};\\\left\vert \alpha\right\vert
=m}}\mathbf{x}^{\alpha}=h_{m}\ \ \ \ \ \ \ \ \ \ \text{for each }%
m\in\mathbb{N}.
\]

\textbf{(b)} Now, let $F=1$. Then, $f_{i}=\left[  i=0\right]  $ for each
$i\in\mathbb{N}$ (where we are using the Iverson bracket notation from
Convention \ref{conv.iverson}). Hence, $f_{\alpha}=\left[  \alpha
=\varnothing\right]  $ for any weak composition $\alpha$. Thus, it is easy to
see that $G_{F}=1$ and $G_{F,m}=\left[  m=0\right]  $ for each $m\in
\mathbb{N}$.

\textbf{(c)} Now, fix a positive integer $k$, and set $F=1+t+t^{2}%
+\cdots+t^{k-1}$. Then, $f_{i}=\left[  i<k\right]  $ for each $i\in\mathbb{N}%
$. Hence, $f_{\alpha}=\prod_{i\geq1}\left[  \alpha_{i}<k\right]  =\left[
\alpha_{i}<k\text{ for all }i\right]  $ for any weak composition $\alpha$.
Thus,
\begin{align*}
G_{F}  &  =\sum_{\alpha\in\operatorname*{WC}}\underbrace{f_{\alpha}}_{=\left[
\alpha_{i}<k\text{ for all }i\right]  }\mathbf{x}^{\alpha}=\sum_{\alpha
\in\operatorname*{WC}}\left[  \alpha_{i}<k\text{ for all }i\right]
\mathbf{x}^{\alpha}\\
&  =\sum_{\substack{\alpha\in\operatorname*{WC};\\\alpha_{i}<k\text{ for all
}i}}\mathbf{x}^{\alpha}\ \ \ \ \ \ \ \ \ \ \left(
\begin{array}
[c]{c}%
\text{since the }\left[  \alpha_{i}<k\text{ for all }i\right]  \text{
factor}\\
\text{makes all addends that don't}\\
\text{satisfy \textquotedblleft}\alpha_{i}<k\text{ for all }%
i\text{\textquotedblright\ vanish}%
\end{array}
\right) \\
&  =G\left(  k\right)  .
\end{align*}
Likewise, we can see that $G_{F,m}=G\left(  k,m\right)  $ for each
$m\in\mathbb{N}$. This shows that the $G_{F}$ and the $G_{F,m}$ are
generalizations of the Petrie symmetric series $G\left(  k\right)  $ and the
Petrie symmetric functions $G\left(  k,m\right)  $, respectively.
\end{example}

The next proposition generalizes parts \textbf{(a)}, \textbf{(b)} and
\textbf{(c)} of Proposition \ref{prop.G.basics}:

\begin{proposition}
\label{prop.GF.basics}\textbf{(a)} The formal power series $G_{F,m}$ is the
$m$-th degree homogeneous component of $G_{F}$ for each $m\in\mathbb{N}$.

\textbf{(b)} We have%
\[
G_{F}=\sum_{\alpha\in\operatorname*{WC}}f_{\alpha}\mathbf{x}^{\alpha}%
=\sum_{\lambda\in\operatorname*{Par}}f_{\lambda}m_{\lambda}=\prod
_{i=1}^{\infty}F\left(  x_{i}\right)  .
\]

\textbf{(c)} We have
\[
G_{F,m}=\sum_{\substack{\alpha\in\operatorname*{WC};\\\left\vert
\alpha\right\vert =m}}f_{\alpha}\mathbf{x}^{\alpha}=\sum_{\substack{\lambda
\in\operatorname*{Par};\\\left\vert \lambda\right\vert =m}}f_{\lambda
}m_{\lambda}\in\Lambda
\]
for each $m\in\mathbb{N}$.

\textbf{(d)} The formal power series $G_{F}$ is symmetric.

\textbf{(e)} We have $G_{F,0}=1$.
\end{proposition}

\begin{vershort}
\begin{proof}
[Proof of Proposition \ref{prop.GF.basics} (sketched).]Parts \textbf{(a)},
\textbf{(b)} and \textbf{(c)} of Proposition \ref{prop.GF.basics} generalize
the corresponding parts of Proposition \ref{prop.G.basics}, and are proved
more or less analogously. (The only novelty is the use of a fact that says
that $f_{\alpha}=f_{\lambda}$ whenever a weak composition $\alpha$ is obtained
by permuting the entries of a partition $\lambda$. Of course, this fact
follows from the definitions of $f_{\alpha}$ and $f_{\lambda}$.)

Part \textbf{(d)} of Proposition \ref{prop.GF.basics} is clear from part
\textbf{(b)}. Part \textbf{(e)} follows from $f_{\varnothing}=1$.
\end{proof}
\end{vershort}

\begin{verlong}
Our proof of Proposition \ref{prop.GF.basics} will use the group
$\mathfrak{S}_{\left(  \infty\right)  }$ and its action on the set
$\operatorname*{WC}$ defined in Subsection \ref{subsect.proofs.Sinf}. This
action has the following property:

\begin{lemma}
\label{lem.GF.basics.lam-al}Let $\lambda\in\operatorname*{Par}$. Let
$\alpha\in\mathfrak{S}_{\left(  \infty\right)  }\lambda$. Then, $f_{\alpha
}=f_{\lambda}$.
\end{lemma}

\begin{proof}
[Proof of Lemma \ref{lem.GF.basics.lam-al}.]We have $\alpha\in\mathfrak{S}%
_{\left(  \infty\right)  }\lambda$. In other words, there exists some
$\sigma\in\mathfrak{S}_{\left(  \infty\right)  }$ such that $\alpha
=\sigma\cdot\lambda$. Consider this $\sigma$. In our proof of Lemma
\ref{lem.G.basics.lam-al}, we have seen that $\sigma^{-1}$ is a bijection from
$\left\{  1,2,3,\ldots\right\}  $ to $\left\{  1,2,3,\ldots\right\}  $. In
that same proof, we have shown that $\left(  \alpha_{1},\alpha_{2},\alpha
_{3},\ldots\right)  =\left(  \lambda_{\sigma^{-1}\left(  1\right)  }%
,\lambda_{\sigma^{-1}\left(  2\right)  },\lambda_{\sigma^{-1}\left(  3\right)
},\ldots\right)  $. In other words,
\begin{equation}
\alpha_{i}=\lambda_{\sigma^{-1}\left(  i\right)  }%
\ \ \ \ \ \ \ \ \ \ \text{for each }i\in\left\{  1,2,3,\ldots\right\}  .
\label{pf.lem.GF.basics.lam-al.1}%
\end{equation}

Now, the definition of $f_{\alpha}$ yields%
\begin{align*}
f_{\alpha}  &  =f_{\alpha_{1}}f_{\alpha_{2}}f_{\alpha_{3}}\cdots=\prod
_{i\in\left\{  1,2,3,\ldots\right\}  }\underbrace{f_{\alpha_{i}}%
}_{\substack{=f_{\lambda_{\sigma^{-1}\left(  i\right)  }}\\\text{(since
}\alpha_{i}=\lambda_{\sigma^{-1}\left(  i\right)  }\\\text{(by
(\ref{pf.lem.GF.basics.lam-al.1})))}}}=\prod_{i\in\left\{  1,2,3,\ldots
\right\}  }f_{\lambda_{\sigma^{-1}\left(  i\right)  }}\\
&  =\prod_{i\in\left\{  1,2,3,\ldots\right\}  }f_{\lambda_{i}}%
\ \ \ \ \ \ \ \ \ \ \left(
\begin{array}
[c]{c}%
\text{here, we have substituted }i\text{ for }\sigma^{-1}\left(  i\right)
\text{ in the product,}\\
\text{since }\sigma^{-1}\text{ is a bijection from }\left\{  1,2,3,\ldots
\right\}  \text{ to }\left\{  1,2,3,\ldots\right\}
\end{array}
\right) \\
&  =f_{\lambda_{1}}f_{\lambda_{2}}f_{\lambda_{3}}\cdots=f_{\lambda}%
\end{align*}
(since the definition of $f_{\lambda}$ yields $f_{\lambda}=f_{\lambda_{1}%
}f_{\lambda_{2}}f_{\lambda_{3}}\cdots$). This proves Lemma
\ref{lem.GF.basics.lam-al}.
\end{proof}

\begin{proof}
[Proof of Proposition \ref{prop.GF.basics}.]\textbf{(a)} It is easy to see
that for any $m\in\mathbb{N}$, the formal power series $G_{F,m}$ is
homogeneous of degree $m$\ \ \ \ \footnote{\textit{Proof.} Let $m\in
\mathbb{N}$. For any $\alpha\in\operatorname*{WC}$, the monomial
$\mathbf{x}^{\alpha}$ is a monomial of degree $\left\vert \alpha\right\vert $.
Thus, if $\alpha\in\operatorname*{WC}$ satisfies $\left\vert \alpha\right\vert
=m$, then $\mathbf{x}^{\alpha}$ is a monomial of degree $m$ (since $\left\vert
\alpha\right\vert =m$). Hence, $\sum_{\substack{\alpha\in\operatorname*{WC}%
;\\\left\vert \alpha\right\vert =m}}f_{\alpha}\mathbf{x}^{\alpha}$ is a
$\mathbf{k}$-linear combination of monomials of degree $m$ (since $f_{\alpha
}\in\mathbf{k}$ for each $\alpha\in\operatorname*{WC}$). In view of
\[
G_{F,m}=\sum_{\substack{\alpha\in\operatorname*{WC};\\\left\vert
\alpha\right\vert =m}}f_{\alpha}\mathbf{x}^{\alpha}\ \ \ \ \ \ \ \ \ \ \left(
\text{by (\ref{eq.GFm=})}\right)  ,
\]
we can restate this as follows: $G_{F,m}$ is a $\mathbf{k}$-linear combination
of monomials of degree $m$. Thus, the formal power series $G_{F,m}$ is
homogeneous of degree $m$. Qed.}. Moreover, (\ref{eq.GF=}) yields%
\begin{align}
G_{F}  &  =\underbrace{\sum_{\alpha\in\operatorname*{WC}}}_{\substack{=\sum
_{m\in\mathbb{N}}\ \ \sum_{\substack{\alpha\in\operatorname*{WC};\\\left\vert
\alpha\right\vert =m}}\\\text{(since }\left\vert \alpha\right\vert
\in\mathbb{N}\text{ for each }\alpha\in\operatorname*{WC}\text{)}}}f_{\alpha
}\mathbf{x}^{\alpha}=\sum_{m\in\mathbb{N}}\ \ \underbrace{\sum
_{\substack{\alpha\in\operatorname*{WC};\\\left\vert \alpha\right\vert
=m}}f_{\alpha}\mathbf{x}^{\alpha}}_{\substack{=G_{F,m}\\\text{(by
(\ref{eq.GFm=}))}}}\nonumber\\
&  =\sum_{m\in\mathbb{N}}G_{F,m}. \label{pf.prop.GF.basics.a.GF=sum}%
\end{align}
Thus, the family $\left(  G_{F,m}\right)  _{m\in\mathbb{N}}$ is the
homogeneous decomposition of $G_{F}$ (since each $G_{F,m}$ is homogeneous of
degree $m$). Hence, for each $m\in\mathbb{N}$, the power series $G_{F,m}$ is
the $m$-th degree homogeneous component of $G_{F}$. This proves Proposition
\ref{prop.GF.basics} \textbf{(a)}.

\textbf{(b)} Let us define the group $\mathfrak{S}_{\left(  \infty\right)  }$
and its action on the set $\operatorname*{WC}$ as in Subsection
\ref{subsect.proofs.Sinf}. Then,%
\begin{align}
\sum_{\lambda\in\operatorname*{Par}}f_{\lambda}\underbrace{m_{\lambda}%
}_{\substack{=\sum_{\alpha\in\mathfrak{S}_{\left(  \infty\right)  }\lambda
}\mathbf{x}^{\alpha}\\\text{(by (\ref{pf.prop.G.basics.b.1}))}}}  &
=\sum_{\lambda\in\operatorname*{Par}}f_{\lambda}\sum_{\alpha\in\mathfrak{S}%
_{\left(  \infty\right)  }\lambda}\mathbf{x}^{\alpha}=\sum_{\lambda
\in\operatorname*{Par}}\ \ \underbrace{\sum_{\alpha\in\mathfrak{S}_{\left(
\infty\right)  }\lambda}}_{\substack{=\sum_{\substack{\alpha\in
\operatorname*{WC};\\\alpha\in\mathfrak{S}_{\left(  \infty\right)  }\lambda
}}\\\text{(since }\mathfrak{S}_{\left(  \infty\right)  }\lambda\subseteq
\operatorname*{WC}\text{)}}}\underbrace{f_{\lambda}}_{\substack{=f_{\alpha
}\\\text{(since Lemma \ref{lem.GF.basics.lam-al}}\\\text{yields }f_{\alpha
}=f_{\lambda}\text{)}}}\mathbf{x}^{\alpha}\nonumber\\
&  =\underbrace{\sum_{\lambda\in\operatorname*{Par}}\ \ \sum_{\substack{\alpha
\in\operatorname*{WC};\\\alpha\in\mathfrak{S}_{\left(  \infty\right)  }%
\lambda}}}_{=\sum_{\alpha\in\operatorname*{WC}}\ \ \sum_{\substack{\lambda
\in\operatorname*{Par};\\\alpha\in\mathfrak{S}_{\left(  \infty\right)
}\lambda}}}f_{\alpha}\mathbf{x}^{\alpha}\nonumber\\
&  =\sum_{\alpha\in\operatorname*{WC}}\ \ \sum_{\substack{\lambda
\in\operatorname*{Par};\\\alpha\in\mathfrak{S}_{\left(  \infty\right)
}\lambda}}f_{\alpha}\mathbf{x}^{\alpha}. \label{pf.prop.GF.basics.b.new2}%
\end{align}

Now, fix some $\alpha\in\operatorname*{WC}$. Then, Lemma
\ref{lem.G.basics.unique-perm} yields that there exists a unique partition
$\lambda\in\operatorname*{Par}$ such that $\alpha\in\mathfrak{S}_{\left(
\infty\right)  }\lambda$. Thus, the sum $\sum_{\substack{\lambda
\in\operatorname*{Par};\\\alpha\in\mathfrak{S}_{\left(  \infty\right)
}\lambda}}f_{\alpha}\mathbf{x}^{\alpha}$ has exactly one addend. Hence, this
sum simplifies as follows:%
\begin{equation}
\sum_{\substack{\lambda\in\operatorname*{Par};\\\alpha\in\mathfrak{S}_{\left(
\infty\right)  }\lambda}}f_{\alpha}\mathbf{x}^{\alpha}=f_{\alpha}%
\mathbf{x}^{\alpha}. \label{pf.prop.GF.basics.b.new3}%
\end{equation}

Forget that we fixed $\alpha$. We thus have proved
(\ref{pf.prop.GF.basics.b.new3}) for each $\alpha\in\operatorname*{WC}$. Thus,
(\ref{pf.prop.GF.basics.b.new2}) becomes%
\[
\sum_{\lambda\in\operatorname*{Par}}f_{\lambda}m_{\lambda}=\sum_{\alpha
\in\operatorname*{WC}}\ \ \underbrace{\sum_{\substack{\lambda\in
\operatorname*{Par};\\\alpha\in\mathfrak{S}_{\left(  \infty\right)  }\lambda
}}f_{\alpha}\mathbf{x}^{\alpha}}_{\substack{=f_{\alpha}\mathbf{x}^{\alpha
}\\\text{(by (\ref{pf.prop.GF.basics.b.new3}))}}}=\sum_{\alpha\in
\operatorname*{WC}}f_{\alpha}\mathbf{x}^{\alpha}.
\]
Comparing this with (\ref{eq.GF=}), we obtain%
\begin{equation}
G_{F}=\sum_{\lambda\in\operatorname*{Par}}f_{\lambda}m_{\lambda}.
\label{pf.prop.GF.basics.b.2}%
\end{equation}

On the other hand, $F=\sum_{n\in\mathbb{N}}f_{n}t^{n}$ (as we have noticed in
Definition \ref{def.GF.GF}). Thus, for each $i\in\left\{  1,2,3,\ldots
\right\}  $, we have%
\[
F\left(  x_{i}\right)  =\sum_{n\in\mathbb{N}}f_{n}x_{i}^{n}=\sum
_{u\in\mathbb{N}}f_{u}x_{i}^{u}%
\]
(here, we have renamed the summation index $n$ as $u$). Multiplying these
equalities over all $i\in\left\{  1,2,3,\ldots\right\}  $, we obtain%
\begin{align*}
&  \prod_{i=1}^{\infty}F\left(  x_{i}\right)  \\
&  =\prod_{i=1}^{\infty}\ \ \sum_{u\in\mathbb{N}}f_{u}x_{i}^{u}\\
&  =\underbrace{\sum_{\substack{\left(  u_{1},u_{2},u_{3},\ldots\right)
\in\mathbb{N}^{\infty};\\\text{all but finitely many }i\text{ satisfy }%
u_{i}=0}}}_{\substack{=\sum_{\left(  u_{1},u_{2},u_{3},\ldots\right)
\in\operatorname*{WC}}\\\text{(since a sequence }\left(  u_{1},u_{2}%
,u_{3},\ldots\right)  \text{ of nonnegative integers}\\\text{satisfies the
statement \textquotedblleft all but finitely many }i\text{ satisfy }%
u_{i}=0\text{\textquotedblright}\\\text{if and only if it satisfies }\left(
u_{1},u_{2},u_{3},\ldots\right)  \in\operatorname*{WC}\text{)}}%
}\underbrace{\left(  f_{u_{1}}x_{1}^{u_{1}}\right)  \left(  f_{u_{2}}%
x_{2}^{u_{2}}\right)  \left(  f_{u_{3}}x_{3}^{u_{3}}\right)  \cdots}_{=\left(
f_{u_{1}}f_{u_{2}}f_{u_{3}}\cdots\right)  \left(  x_{1}^{u_{1}}x_{2}^{u_{2}%
}x_{3}^{u_{3}}\cdots\right)  }\\
&  \ \ \ \ \ \ \ \ \ \ \ \ \ \ \ \ \ \ \ \ \left(  \text{by the product
rule}\right)  \\
&  =\sum_{\left(  u_{1},u_{2},u_{3},\ldots\right)  \in\operatorname*{WC}%
}\left(  f_{u_{1}}f_{u_{2}}f_{u_{3}}\cdots\right)  \left(  x_{1}^{u_{1}}%
x_{2}^{u_{2}}x_{3}^{u_{3}}\cdots\right)  \\
&  =\sum_{\left(  \alpha_{1},\alpha_{2},\alpha_{3},\ldots\right)
\in\operatorname*{WC}}\underbrace{\left(  f_{\alpha_{1}}f_{\alpha_{2}%
}f_{\alpha_{3}}\cdots\right)  }_{\substack{=f_{\left(  \alpha_{1},\alpha
_{2},\alpha_{3},\ldots\right)  }\\\text{(by the definition of }f_{\left(
\alpha_{1},\alpha_{2},\alpha_{3},\ldots\right)  }\text{)}}}\underbrace{\left(
x_{1}^{\alpha_{1}}x_{2}^{\alpha_{2}}x_{3}^{\alpha_{3}}\cdots\right)
}_{\substack{=\mathbf{x}^{\left(  \alpha_{1},\alpha_{2},\alpha_{3}%
,\ldots\right)  }\\\text{(by the definition of }\mathbf{x}^{\left(  \alpha
_{1},\alpha_{2},\alpha_{3},\ldots\right)  }\text{)}}}\\
&  \ \ \ \ \ \ \ \ \ \ \ \ \ \ \ \ \ \ \ \ \left(
\begin{array}
[c]{c}%
\text{here, we have renamed the}\\
\text{summation index }\left(  u_{1},u_{2},u_{3},\ldots\right)  \text{ as
}\left(  \alpha_{1},\alpha_{2},\alpha_{3},\ldots\right)
\end{array}
\right)  \\
&  =\sum_{\left(  \alpha_{1},\alpha_{2},\alpha_{3},\ldots\right)
\in\operatorname*{WC}}f_{\left(  \alpha_{1},\alpha_{2},\alpha_{3}%
,\ldots\right)  }\mathbf{x}^{\left(  \alpha_{1},\alpha_{2},\alpha_{3}%
,\ldots\right)  }=\sum_{\alpha\in\operatorname*{WC}}f_{\alpha}\mathbf{x}%
^{\alpha}\\
&  \ \ \ \ \ \ \ \ \ \ \ \ \ \ \ \ \ \ \ \ \left(
\begin{array}
[c]{c}%
\text{here, we have renamed the}\\
\text{summation index }\left(  \alpha_{1},\alpha_{2},\alpha_{3},\ldots\right)
\text{ as }\alpha
\end{array}
\right)  .
\end{align*}
Comparing this with (\ref{eq.GF=}), we obtain%
\[
G_{F}=\prod_{i=1}^{\infty}F\left(  x_{i}\right)  .
\]
Combining this equality with (\ref{pf.prop.GF.basics.b.2}) and (\ref{eq.GF=}),
we obtain%
\[
G_{F}=\sum_{\alpha\in\operatorname*{WC}}f_{\alpha}\mathbf{x}^{\alpha}%
=\sum_{\lambda\in\operatorname*{Par}}f_{\lambda}m_{\lambda}=\prod
_{i=1}^{\infty}F\left(  x_{i}\right)  .
\]
This proves Proposition \ref{prop.GF.basics} \textbf{(b)}.

\textbf{(c)} Let $m\in\mathbb{N}$. Let us define the group $\mathfrak{S}%
_{\left(  \infty\right)  }$ and its action on the set $\operatorname*{WC}$ as
in Subsection \ref{subsect.proofs.Sinf}. Then,%
\begin{align}
\sum_{\substack{\lambda\in\operatorname*{Par};\\\left\vert \lambda\right\vert
=m}}f_{\lambda}\underbrace{m_{\lambda}}_{\substack{=\sum_{\alpha
\in\mathfrak{S}_{\left(  \infty\right)  }\lambda}\mathbf{x}^{\alpha
}\\\text{(by (\ref{pf.prop.G.basics.b.1}))}}}  &  =\sum_{\substack{\lambda
\in\operatorname*{Par};\\\left\vert \lambda\right\vert =m}}f_{\lambda}%
\sum_{\alpha\in\mathfrak{S}_{\left(  \infty\right)  }\lambda}\mathbf{x}%
^{\alpha}=\sum_{\substack{\lambda\in\operatorname*{Par};\\\left\vert
\lambda\right\vert =m}}\ \ \sum_{\alpha\in\mathfrak{S}_{\left(  \infty\right)
}\lambda}\underbrace{f_{\lambda}}_{\substack{=f_{\alpha}\\\text{(since Lemma
\ref{lem.GF.basics.lam-al}}\\\text{yields }f_{\alpha}=f_{\lambda}\text{)}%
}}\mathbf{x}^{\alpha}\nonumber\\
&  =\sum_{\substack{\lambda\in\operatorname*{Par};\\\left\vert \lambda
\right\vert =m}}\ \ \sum_{\alpha\in\mathfrak{S}_{\left(  \infty\right)
}\lambda}f_{\alpha}\mathbf{x}^{\alpha}. \label{pf.prop.GF.basics.c.new1}%
\end{align}

Now, we have the following equality of summation signs:%
\begin{align*}
\sum_{\substack{\lambda\in\operatorname*{Par};\\\left\vert \lambda\right\vert
=m}}\ \ \sum_{\alpha\in\mathfrak{S}_{\left(  \infty\right)  }\lambda}  &
=\sum_{\lambda\in\operatorname*{Par}}\ \ \underbrace{\sum_{\substack{\alpha
\in\mathfrak{S}_{\left(  \infty\right)  }\lambda;\\\left\vert \lambda
\right\vert =m}}}_{\substack{=\sum_{\substack{\alpha\in\mathfrak{S}_{\left(
\infty\right)  }\lambda;\\\left\vert \alpha\right\vert =m}}\\\text{(because
for each }\alpha\in\mathfrak{S}_{\left(  \infty\right)  }\lambda
\text{,}\\\text{we have }\left\vert \lambda\right\vert =\left\vert
\alpha\right\vert \\\text{(by Lemma \ref{lem.G.basics.lam-al} \textbf{(a)}))}%
}}=\sum_{\lambda\in\operatorname*{Par}}\ \ \underbrace{\sum_{\substack{\alpha
\in\mathfrak{S}_{\left(  \infty\right)  }\lambda;\\\left\vert \alpha
\right\vert =m}}}_{\substack{=\sum_{\substack{\alpha\in\operatorname*{WC}%
;\\\alpha\in\mathfrak{S}_{\left(  \infty\right)  }\lambda;\\\left\vert
\alpha\right\vert =m}}\\\text{(since }\mathfrak{S}_{\left(  \infty\right)
}\lambda\subseteq\operatorname*{WC}\text{)}}}\\
&  =\sum_{\lambda\in\operatorname*{Par}}\ \ \sum_{\substack{\alpha
\in\operatorname*{WC};\\\alpha\in\mathfrak{S}_{\left(  \infty\right)  }%
\lambda;\\\left\vert \alpha\right\vert =m}}=\sum_{\substack{\alpha
\in\operatorname*{WC};\\\left\vert \alpha\right\vert =m}}\ \ \sum
_{\substack{\lambda\in\operatorname*{Par};\\\alpha\in\mathfrak{S}_{\left(
\infty\right)  }\lambda}}.
\end{align*}
Hence, (\ref{pf.prop.GF.basics.c.new1}) becomes%
\begin{align}
\sum_{\substack{\lambda\in\operatorname*{Par};\\\left\vert \lambda\right\vert
=m}}f_{\lambda}m_{\lambda}  &  =\underbrace{\sum_{\substack{\lambda
\in\operatorname*{Par};\\\left\vert \lambda\right\vert =m}}\ \ \sum_{\alpha
\in\mathfrak{S}_{\left(  \infty\right)  }\lambda}}_{=\sum_{\substack{\alpha
\in\operatorname*{WC};\\\left\vert \alpha\right\vert =m}}\ \ \sum
_{\substack{\lambda\in\operatorname*{Par};\\\alpha\in\mathfrak{S}_{\left(
\infty\right)  }\lambda}}}f_{\alpha}\mathbf{x}^{\alpha}=\sum_{\substack{\alpha
\in\operatorname*{WC};\\\left\vert \alpha\right\vert =m}}\ \ \underbrace{\sum
_{\substack{\lambda\in\operatorname*{Par};\\\alpha\in\mathfrak{S}_{\left(
\infty\right)  }\lambda}}f_{\alpha}\mathbf{x}^{\alpha}}_{\substack{=f_{\alpha
}\mathbf{x}^{\alpha}\\\text{(by (\ref{pf.prop.GF.basics.b.new3}))}%
}}\nonumber\\
&  =\sum_{\substack{\alpha\in\operatorname*{WC};\\\left\vert \alpha\right\vert
=m}}f_{\alpha}\mathbf{x}^{\alpha}. \label{pf.prop.GF.basics.c.new5}%
\end{align}
Now, (\ref{eq.GFm=}) becomes%
\begin{align*}
G_{F,m}  &  =\sum_{\substack{\alpha\in\operatorname*{WC};\\\left\vert
\alpha\right\vert =m}}f_{\alpha}\mathbf{x}^{\alpha}=\sum_{\substack{\lambda
\in\operatorname*{Par};\\\left\vert \lambda\right\vert =m}}f_{\lambda
}\underbrace{m_{\lambda}}_{\in\Lambda}\ \ \ \ \ \ \ \ \ \ \left(  \text{by
(\ref{pf.prop.GF.basics.c.new5})}\right) \\
&  \in\sum_{\substack{\lambda\in\operatorname*{Par};\\\left\vert
\lambda\right\vert =m}}f_{\lambda}\Lambda\subseteq\Lambda
\ \ \ \ \ \ \ \ \ \ \left(  \text{since }\Lambda\text{ is a }\mathbf{k}%
\text{-module}\right)  .
\end{align*}
This proves Proposition \ref{prop.GF.basics} \textbf{(c)}.

\textbf{(d)} Proposition \ref{prop.GF.basics} \textbf{(b)} yields $G_{F}%
=\prod_{i=1}^{\infty}F\left(  x_{i}\right)  $. Thus, the power series $G_{F}$
is symmetric (since $\prod_{i=1}^{\infty}F\left(  x_{i}\right)  $ is obviously
symmetric). This proves Proposition \ref{prop.GF.basics} \textbf{(d)}.

\textbf{(e)} The definition of $G_{F,0}$ yields%
\begin{align*}
G_{F,0}  &  =\sum_{\substack{\alpha\in\operatorname*{WC};\\\left\vert
\alpha\right\vert =0}}f_{\alpha}\mathbf{x}^{\alpha}=f_{\left(  0,0,0,\ldots
\right)  }\underbrace{\mathbf{x}^{\left(  0,0,0,\ldots\right)  }}%
_{=1}\ \ \ \ \ \ \ \ \ \ \left(
\begin{array}
[c]{c}%
\text{since the only }\alpha\in\operatorname*{WC}\text{ satisfying }\left\vert
\alpha\right\vert =0\\
\text{is the weak composition }\left(  0,0,0,\ldots\right)
\end{array}
\right) \\
&  =f_{\left(  0,0,0,\ldots\right)  }=f_{0}f_{0}f_{0}\cdots
\ \ \ \ \ \ \ \ \ \ \left(  \text{by the definition of }f_{\left(
0,0,0,\ldots\right)  }\right) \\
&  =1\cdot1\cdot1\cdot\cdots\ \ \ \ \ \ \ \ \ \ \left(  \text{since }%
f_{0}=1\right) \\
&  =1.
\end{align*}
This proves Proposition \ref{prop.GF.basics} \textbf{(e)}.
\end{proof}
\end{verlong}

Next, let us generalize Definition \ref{def.petklam}:

\begin{definition}
\label{def.petFlam}Let $\lambda=\left(  \lambda_{1},\lambda_{2},\ldots
,\lambda_{\ell}\right)  \in\operatorname*{Par}$ and $\mu=\left(  \mu_{1}%
,\mu_{2},\ldots,\mu_{\ell}\right)  \in\operatorname*{Par}$. Then, the
$F$\emph{-Petrie number} $\operatorname*{pet}\nolimits_{F}\left(  \lambda
,\mu\right)  $ of $\lambda$ and $\mu$ is the element of $\mathbf{k}$ defined
by%
\begin{equation}
\operatorname*{pet}\nolimits_{F}\left(  \lambda,\mu\right)  =\det\left(
\left(  f_{\lambda_{i}-\mu_{j}-i+j}\right)  _{1\leq i\leq\ell,\ 1\leq
j\leq\ell}\right)  .\label{eq.def.petFlam.def}%
\end{equation}
Note that this integer does not depend on the choice of $\ell$ (in the sense
that it does not change if we enlarge $\ell$ by adding trailing zeroes to the
representations of $\lambda$ and $\mu$); this follows from Lemma
\ref{lem.petFrel.indep} below.
\end{definition}

\begin{example}
For $\ell=3$, the equality (\ref{eq.def.petFlam.def}) rewrites as%
\[
\operatorname*{pet}\nolimits_{F}\left(  \lambda,\mu\right)  =\det\left(
\begin{array}
[c]{ccc}%
f_{\lambda_{1}-\mu_{1}} & f_{\lambda_{1}-\mu_{2}+1} & f_{\lambda_{1}-\mu
_{3}+2}\\
f_{\lambda_{2}-\mu_{1}-1} & f_{\lambda_{2}-\mu_{2}} & f_{\lambda_{2}-\mu
_{3}+1}\\
f_{\lambda_{3}-\mu_{1}-2} & f_{\lambda_{3}-\mu_{2}-1} & f_{\lambda_{3}-\mu
_{3}}%
\end{array}
\right)  .
\]

\end{example}

We can now state the generalization of Lemma \ref{lem.petkrel.indep} that is
needed to justify Definition \ref{def.petFlam}:

\begin{lemma}
\label{lem.petFrel.indep}Let $\lambda\in\operatorname*{Par}$ and $\mu
\in\operatorname*{Par}$. Let $\ell\in\mathbb{N}$ be such that $\lambda=\left(
\lambda_{1},\lambda_{2},\ldots,\lambda_{\ell}\right)  $ and $\mu=\left(
\mu_{1},\mu_{2},\ldots,\mu_{\ell}\right)  $. Then, the determinant
$\det\left(  \left(  f_{\lambda_{i}-\mu_{j}-i+j}\right)  _{1\leq i\leq
\ell,\ 1\leq j\leq\ell}\right)  $ does not depend on the choice of $\ell$.
\end{lemma}

The slickest way to prove Lemma \ref{lem.petFrel.indep} is using a
$\mathbf{k}$-algebra homomorphism $\alpha_{F}:\Lambda\rightarrow\mathbf{k}$
that generalizes the homomorphism $\alpha_{k}$ from Definition
\ref{def.alphak}. Let us introduce this homomorphism $\alpha_{F}$ (which will
also be used in other proofs). We recall the h-universal property of $\Lambda
$, which was stated in Subsection \ref{subsect.proofs.petk.alphak}.

\begin{definition}
\label{def.alphaF}The h-universal property of $\Lambda$ shows that there is a
unique $\mathbf{k}$-algebra homomorphism $\alpha_{F}:\Lambda\rightarrow
\mathbf{k}$ that sends $h_{i}$ to $f_{i}$ for all positive integers $i$.
Consider this $\alpha_{F}$.
\end{definition}

We will use this homomorphism $\alpha_{F}$ several times in what follows; let
us thus begin by stating some elementary properties of $\alpha_{F}$.

\begin{lemma}
\label{lem.alphaF.h}\textbf{(a)} We have%
\begin{equation}
\alpha_{F}\left(  h_{i}\right)  =f_{i}\ \ \ \ \ \ \ \ \ \ \text{for all }%
i\in\mathbb{N}. \label{pf.thm.GF.main.alk}%
\end{equation}

\textbf{(b)} We have%
\begin{equation}
\alpha_{F}\left(  h_{i}\right)  =f_{i}\ \ \ \ \ \ \ \ \ \ \text{for all }%
i\in\mathbb{Z}. \label{pf.thm.GF.main.alkZ}%
\end{equation}

\textbf{(c)} Let $\lambda$ be a partition. Define $h_{\lambda}$ as in
Definition \ref{def.hlam}. Then,%
\begin{equation}
\alpha_{F}\left(  h_{\lambda}\right)  =f_{\lambda}.
\label{pf.thm.GF.main.alkh}%
\end{equation}

\end{lemma}

\begin{vershort}
\begin{proof}
[Proof of Lemma \ref{lem.alphaF.h} (sketched).]Analogous to the proof of Lemma
\ref{lem.alphak.h}.
\end{proof}
\end{vershort}

\begin{verlong}
\begin{proof}
[Proof of Lemma \ref{lem.alphaF.h}.]\textbf{(a)} Let $i\in\mathbb{N}$. We must
prove that $\alpha_{F}\left(  h_{i}\right)  =f_{i}$. If $i>0$, then this
follows from the definition of $\alpha_{F}$. Thus, we WLOG assume that we
don't have $i>0$. Hence, $i=0$ (since $i\in\mathbb{N}$). Therefore,
$h_{i}=h_{0}=1$, so that $\alpha_{F}\left(  h_{i}\right)  =\alpha_{F}\left(
1\right)  =1$ (since $\alpha_{F}$ is a $\mathbf{k}$-algebra homomorphism). On
the other hand, $i=0$, so that $f_{i}=f_{0}=1$. Comparing this with
$\alpha_{F}\left(  h_{i}\right)  =1$, we obtain $\alpha_{F}\left(
h_{i}\right)  =f_{i}$. This proves Lemma \ref{lem.alphaF.h} \textbf{(a)}.

\textbf{(b)} Let $i\in\mathbb{Z}$. We must prove that $\alpha_{F}\left(
h_{i}\right)  =f_{i}$. If $i<0$, then this easily follows from $0=0$%
\ \ \ \ \footnote{\textit{Proof.} Assume that $i<0$. Thus, $h_{i}=0$, so that
$\alpha_{F}\left(  h_{i}\right)  =\alpha_{F}\left(  0\right)  =0$ (since
$\alpha_{F}$ is a $\mathbf{k}$-algebra homomorphism). But Definition
\ref{def.GF.GF} \textbf{(b)} yields $f_{i}=0$ (since $i$ is negative).
Comparing this with $\alpha_{F}\left(  h_{i}\right)  =0$, we obtain
$\alpha_{F}\left(  h_{i}\right)  =f_{i}$, qed.}. Hence, we WLOG assume that we
don't have $i<0$. Therefore, $i\geq0$, so that $i\in\mathbb{N}$. Hence, Lemma
\ref{lem.alphaF.h} \textbf{(a)} yields $\alpha_{F}\left(  h_{i}\right)
=f_{i}$. This proves Lemma \ref{lem.alphaF.h} \textbf{(b)}.

\textbf{(c)} Write the partition $\lambda$ in the form $\lambda=\left(
\lambda_{1},\lambda_{2},\ldots,\lambda_{\ell}\right)  $, where $\lambda
_{1},\lambda_{2},\ldots,\lambda_{\ell}$ are positive integers. Then, the
definition of $h_{\lambda}$ yields
\[
h_{\lambda}=h_{\lambda_{1}}h_{\lambda_{2}}\cdots h_{\lambda_{\ell}}%
=\prod_{i=1}^{\ell}h_{\lambda_{i}}.
\]
Applying the map $\alpha_{F}$ to both sides of this equality, we find%
\begin{align}
\alpha_{F}\left(  h_{\lambda}\right)   &  =\alpha_{F}\left(  \prod_{i=1}%
^{\ell}h_{\lambda_{i}}\right)  =\prod_{i=1}^{\ell}\underbrace{\alpha
_{F}\left(  h_{\lambda_{i}}\right)  }_{\substack{=f_{\lambda_{i}}\\\text{(by
(\ref{pf.thm.GF.main.alk}), applied to }\lambda_{i}\text{ instead of
}i\text{)}}}\nonumber\\
&  \ \ \ \ \ \ \ \ \ \ \ \ \ \ \ \ \ \ \ \ \left(  \text{since }\alpha
_{F}\text{ is a }\mathbf{k}\text{-algebra homomorphism}\right) \nonumber\\
&  =\prod_{i=1}^{\ell}f_{\lambda_{i}}. \label{pf.lem.alphaF.h.c.2}%
\end{align}

But we have $\lambda=\left(  \lambda_{1},\lambda_{2},\ldots,\lambda_{\ell
}\right)  $ and thus $\lambda_{\ell+1}=\lambda_{\ell+2}=\lambda_{\ell
+3}=\cdots=0$. In other words, we have $\lambda_{i}=0$ for all $i\in\left\{
\ell+1,\ell+2,\ell+3,\ldots\right\}  $. Hence, we have $f_{\lambda_{i}}%
=f_{0}=1$ for all $i\in\left\{  \ell+1,\ell+2,\ell+3,\ldots\right\}  $.
Multiplying these equalities over all $i\in\left\{  \ell+1,\ell+2,\ell
+3,\ldots\right\}  $, we obtain $\prod_{i=\ell+1}^{\infty}f_{\lambda_{i}%
}=\prod_{i=\ell+1}^{\infty}1=1$.

Now, the definition of $f_{\lambda}$ yields%
\[
f_{\lambda}=f_{\lambda_{1}}f_{\lambda_{2}}f_{\lambda_{3}}\cdots=\prod
_{i=1}^{\infty}f_{\lambda_{i}}=\left(  \prod_{i=1}^{\ell}f_{\lambda_{i}%
}\right)  \underbrace{\left(  \prod_{i=\ell+1}^{\infty}f_{\lambda_{i}}\right)
}_{=1}=\prod_{i=1}^{\ell}f_{\lambda_{i}}=\alpha_{F}\left(  h_{\lambda}\right)
\]
(by (\ref{pf.lem.alphaF.h.c.2})). Thus, Lemma \ref{lem.alphaF.h} \textbf{(c)}
is proved.
\end{proof}
\end{verlong}

\begin{vershort}
\begin{proof}
[Proof of Lemma \ref{lem.petFrel.indep} (sketched).]Adapt any of the two
proofs of Lemma \ref{lem.petkrel.indep}. (In adapting the first proof, use
$\alpha_{F}$ instead of $\alpha_{k}$.)
\end{proof}
\end{vershort}

\begin{verlong}
The following proof of Lemma \ref{lem.petFrel.indep} is a straightforward
adaptation of our first proof of Lemma \ref{lem.petkrel.indep} (note that the
second proof can be adapted just as easily).

\begin{proof}
[Proof of Lemma \ref{lem.petFrel.indep}.]Recall that $\alpha_{F}%
:\Lambda\rightarrow\mathbf{k}$ is a $\mathbf{k}$-algebra homomorphism. Thus,
$\alpha_{F}$ respects determinants; i.e., if $\left(  a_{i,j}\right)  _{1\leq
i\leq m,\ 1\leq j\leq m}\in\Lambda^{m\times m}$ is an $m\times m$-matrix over
$\Lambda$ (for some $m\in\mathbb{N}$), then%
\begin{align}
&  \alpha_{F}\left(  \det\left(  \left(  a_{i,j}\right)  _{1\leq i\leq
m,\ 1\leq j\leq m}\right)  \right) \nonumber\\
&  =\det\left(  \left(  \alpha_{F}\left(  a_{i,j}\right)  \right)  _{1\leq
i\leq m,\ 1\leq j\leq m}\right)  . \label{pf.lem.petFrel.indep.adet}%
\end{align}

Applying $\alpha_{F}$ to both sides of the equality
(\ref{eq.schur.JT.sh-lammu}), we obtain%
\begin{align}
\alpha_{F}\left(  s_{\lambda/\mu}\right)   &  =\alpha_{F}\left(  \det\left(
\left(  h_{\lambda_{i}-\mu_{j}-i+j}\right)  _{1\leq i\leq\ell,\ 1\leq
j\leq\ell}\right)  \right)  \nonumber\\
&  =\det\left(  \left(  \underbrace{\alpha_{F}\left(  h_{\lambda_{i}-\mu
_{j}-i+j}\right)  }_{\substack{=f_{\lambda_{i}-\mu_{j}-i+j}\\\text{(by
(\ref{pf.thm.GF.main.alkZ}), applied to }\lambda_{i}-\mu_{j}%
-i+j\\\text{instead of }i\text{)}}}\right)  _{1\leq i\leq\ell,\ 1\leq
j\leq\ell}\right)  \nonumber\\
&  \ \ \ \ \ \ \ \ \ \ \ \ \ \ \ \ \ \ \ \ \left(  \text{by
(\ref{pf.lem.petFrel.indep.adet}), applied to }m=\ell\text{ and }%
a_{i,j}=h_{\lambda_{i}-\mu_{j}-i+j}\right)  \nonumber\\
&  =\det\left(  \left(  f_{\lambda_{i}-\mu_{j}-i+j}\right)  _{1\leq i\leq
\ell,\ 1\leq j\leq\ell}\right)  .\label{pf.lem.petFrel.indep.1}%
\end{align}
Clearly, the element $\alpha_{F}\left(  s_{\lambda/\mu}\right)  $ does not
depend on the choice of $\ell$. In view of the equality
(\ref{pf.lem.petFrel.indep.1}), we can rewrite this as follows: The element
$\det\left(  \left(  f_{\lambda_{i}-\mu_{j}-i+j}\right)  _{1\leq i\leq
\ell,\ 1\leq j\leq\ell}\right)  $ does not depend on the choice of $\ell$.
This proves Lemma \ref{lem.petFrel.indep}.
\end{proof}

Just as with Lemma \ref{lem.petkrel.indep}, our proof of Lemma
\ref{lem.petFrel.indep} leads to a useful consequence (analogous to Lemma
\ref{lem.alphak.slammu}):

\begin{lemma}
\label{lem.alphaF.slammu}Let $\lambda$ and $\mu$ be two partitions. Then, the
homomorphism $\alpha_{F}:\Lambda\rightarrow\mathbf{k}$ from Definition
\ref{def.alphaF} satisfies%
\begin{equation}
\alpha_{F}\left(  s_{\lambda/\mu}\right)  =\operatorname*{pet}\nolimits_{F}%
\left(  \lambda,\mu\right)  . \label{pf.thm.GF.pieri.alks}%
\end{equation}

\end{lemma}

\begin{proof}
[Proof of Lemma \ref{lem.alphaF.slammu}.]Write the partitions $\lambda$ and
$\mu$ in the forms $\lambda=\left(  \lambda_{1},\lambda_{2},\ldots
,\lambda_{\ell}\right)  $ and $\mu=\left(  \mu_{1},\mu_{2},\ldots,\mu_{\ell
}\right)  $ for some $\ell\in\mathbb{N}$\ \ \ \ \footnote{Such an $\ell$ can
always be found, since each of $\lambda$ and $\mu$ has only finitely many
nonzero entries.}. Then, the equality (\ref{pf.lem.petFrel.indep.1}) (which we
showed in our proof of Lemma \ref{lem.petFrel.indep}) yields%
\[
\alpha_{F}\left(  s_{\lambda/\mu}\right)  =\det\left(  \left(  f_{\lambda
_{i}-\mu_{j}-i+j}\right)  _{1\leq i\leq\ell,\ 1\leq j\leq\ell}\right)
=\operatorname*{pet}\nolimits_{F}\left(  \lambda,\mu\right)
\]
(by the definition of $\operatorname*{pet}\nolimits_{F}\left(  \lambda
,\mu\right)  $). This proves Lemma \ref{lem.alphaF.slammu}.
\end{proof}
\end{verlong}

We now come to more substantive properties of $G\left(  k\right)  $ and
$G\left(  k,m\right)  $.

\begin{verlong}
First, we shall state them; the proofs will come later.
\end{verlong}

The following theorem generalizes Theorem \ref{thm.G.main}:

\begin{theorem}
\label{thm.GF.main}We have%
\[
G_{F}=\sum_{\lambda\in\operatorname*{Par}}\operatorname*{pet}\nolimits_{F}%
\left(  \lambda,\varnothing\right)  s_{\lambda}.
\]
(Recall that $\varnothing$ denotes the empty partition $\left(  {}\right)
=\left(  0,0,0,\ldots\right)  $.)
\end{theorem}

The following corollary (which already appeared in \cite[Exercise 7.91
(d)]{Stanley-EC2}) generalizes Corollary \ref{cor.Gkm.main}:

\begin{corollary}
\label{cor.GFm.main}Let $m\in\mathbb{N}$. Then,%
\[
G_{F,m}=\sum_{\lambda\in\operatorname*{Par}\nolimits_{m}}\operatorname*{pet}%
\nolimits_{F}\left(  \lambda,\varnothing\right)  s_{\lambda}.
\]

\end{corollary}

The following theorem generalizes Theorem \ref{thm.G.pieri}:

\begin{theorem}
\label{thm.GF.pieri}Let $\mu\in\operatorname*{Par}$. Then,%
\[
G_{F}\cdot s_{\mu}=\sum_{\lambda\in\operatorname*{Par}}\operatorname*{pet}%
\nolimits_{F}\left(  \lambda,\mu\right)  s_{\lambda}.
\]

\end{theorem}

The following corollary generalizes Corollary \ref{cor.Gkm.pieri}:

\begin{corollary}
\label{cor.GFm.pieri}Let $m\in\mathbb{N}$. Let $\mu\in\operatorname*{Par}$.
Then,%
\[
G_{F,m}\cdot s_{\mu}=\sum_{\lambda\in\operatorname*{Par}%
\nolimits_{m+\left\vert \mu\right\vert }}\operatorname*{pet}\nolimits_{F}%
\left(  \lambda,\mu\right)  s_{\lambda}.
\]

\end{corollary}

\begin{vershort}
\begin{proof}
[Proofs of Theorem \ref{thm.GF.pieri}, Corollary \ref{cor.GFm.pieri}, Theorem
\ref{thm.GF.main} and Corollary \ref{cor.GFm.main} (sketched).]These proofs
are analogous to the proofs of Theorem \ref{thm.G.pieri}, Corollary
\ref{cor.Gkm.pieri}, Theorem \ref{thm.G.main} and Corollary \ref{cor.Gkm.main}%
, respectively (but using $\alpha_{F}$ instead of $\alpha_{k}$).
\end{proof}
\end{vershort}

\begin{verlong}
Let us prove these four results. We begin with Theorem \ref{thm.GF.pieri},
which (just as its particular case the Theorem \ref{thm.G.pieri}) can be
proved in two ways. We shall only give the first proof:

\begin{proof}
[Proof of Theorem \ref{thm.GF.pieri}.]We shall use the notations
$\mathbf{k}\left[  \left[  \mathbf{x}\right]  \right]  $, $\mathbf{k}\left[
\left[  \mathbf{x},\mathbf{y}\right]  \right]  $, $\mathbf{x}$, $\mathbf{y}$,
$f\left(  \mathbf{x}\right)  $ and $f\left(  \mathbf{y}\right)  $ introduced
in Subsection \ref{subsect.thms.coprod}. If $R$ is any commutative ring, then
$R\left[  \left[  \mathbf{y}\right]  \right]  $ shall denote the ring
$R\left[  \left[  y_{1},y_{2},y_{3},\ldots\right]  \right]  $ of formal power
series in the indeterminates $y_{1},y_{2},y_{3},\ldots$ over the ring $R$. We
will identify the ring $\mathbf{k}\left[  \left[  \mathbf{x},\mathbf{y}%
\right]  \right]  $ with the ring $\left(  \mathbf{k}\left[  \left[
\mathbf{x}\right]  \right]  \right)  \left[  \left[  \mathbf{y}\right]
\right]  =\left(  \mathbf{k}\left[  \left[  x_{1},x_{2},x_{3},\ldots\right]
\right]  \right)  \left[  \left[  y_{1},y_{2},y_{3},\ldots\right]  \right]  $.
Note that $\Lambda\subseteq\mathbf{k}\left[  \left[  \mathbf{x}\right]
\right]  $ and thus $\Lambda\left[  \left[  \mathbf{y}\right]  \right]
\subseteq\left(  \mathbf{k}\left[  \left[  \mathbf{x}\right]  \right]
\right)  \left[  \left[  \mathbf{y}\right]  \right]  =\mathbf{k}\left[
\left[  \mathbf{x},\mathbf{y}\right]  \right]  $. We equip the rings
$\mathbf{k}\left[  \left[  \mathbf{y}\right]  \right]  $, $\Lambda\left[
\left[  \mathbf{y}\right]  \right]  $ and $\mathbf{k}\left[  \left[
\mathbf{x},\mathbf{y}\right]  \right]  $ with the usual topologies that are
defined on rings of power series, where $\Lambda$ itself is equipped with the
discrete topology. This has the somewhat confusing consequence that
$\Lambda\left[  \left[  \mathbf{y}\right]  \right]  \subseteq\mathbf{k}\left[
\left[  \mathbf{x},\mathbf{y}\right]  \right]  $ is an inclusion of rings but
not of topological spaces; however, this will not cause us any trouble, since
all infinite sums in $\Lambda\left[  \left[  \mathbf{y}\right]  \right]  $ we
will consider (such as $\sum_{\lambda\in\operatorname*{Par}}s_{\lambda/\mu
}\left(  \mathbf{x}\right)  s_{\lambda}\left(  \mathbf{y}\right)  $ and
$\sum_{\lambda\in\operatorname*{Par}}h_{\lambda}\left(  \mathbf{x}\right)
m_{\lambda}\left(  \mathbf{y}\right)  $) will converge to the same value in
either topology.

We consider both $\mathbf{k}\left[  \left[  \mathbf{y}\right]  \right]  $ and
$\Lambda$ as subrings of $\Lambda\left[  \left[  \mathbf{y}\right]  \right]  $
(indeed, $\mathbf{k}\left[  \left[  \mathbf{y}\right]  \right]  $ embeds into
$\Lambda\left[  \left[  \mathbf{y}\right]  \right]  $ because $\mathbf{k}$ is
a subring of $\Lambda$, whereas $\Lambda$ embeds into $\Lambda\left[  \left[
\mathbf{y}\right]  \right]  $ because $\Lambda\left[  \left[  \mathbf{y}%
\right]  \right]  $ is a ring of power series over $\Lambda$).

In this proof, the word \textquotedblleft monomial\textquotedblright\ may
refer to a monomial in any set of variables (not necessarily in $x_{1}%
,x_{2},x_{3},\ldots$).

Recall the $\mathbf{k}$-algebra homomorphism $\alpha_{F}:\Lambda
\rightarrow\mathbf{k}$ from Definition \ref{def.alphaF}. This $\mathbf{k}%
$-algebra homomorphism $\alpha_{F}:\Lambda\rightarrow\mathbf{k}$ induces a
$\mathbf{k}\left[  \left[  \mathbf{y}\right]  \right]  $-algebra homomorphism
$\alpha_{F}\left[  \left[  \mathbf{y}\right]  \right]  :\Lambda\left[  \left[
\mathbf{y}\right]  \right]  \rightarrow\mathbf{k}\left[  \left[
\mathbf{y}\right]  \right]  $, which is given by the formula%
\[
\left(  \alpha_{F}\left[  \left[  \mathbf{y}\right]  \right]  \right)  \left(
\sum_{\substack{\mathfrak{n}\text{ is a monomial}\\\text{in }y_{1},y_{2}%
,y_{3},\ldots}}f_{\mathfrak{n}}\mathfrak{n}\right)  =\sum
_{\substack{\mathfrak{n}\text{ is a monomial}\\\text{in }y_{1},y_{2}%
,y_{3},\ldots}}\alpha_{F}\left(  f_{\mathfrak{n}}\right)  \mathfrak{n}%
\]
for any family $\left(  f_{\mathfrak{n}}\right)  _{\mathfrak{n}\text{ is a
monomial in }y_{1},y_{2},y_{3},\ldots}$ of elements of $\Lambda$. This induced
$\mathbf{k}\left[  \left[  \mathbf{y}\right]  \right]  $-algebra homomorphism
$\alpha_{F}\left[  \left[  \mathbf{y}\right]  \right]  $ is $\mathbf{k}\left[
\left[  \mathbf{y}\right]  \right]  $-linear and continuous (with respect to
the usual topologies on the power series rings $\Lambda\left[  \left[
\mathbf{y}\right]  \right]  $ and $\mathbf{k}\left[  \left[  \mathbf{y}%
\right]  \right]  $), and thus preserves infinite $\mathbf{k}\left[  \left[
\mathbf{y}\right]  \right]  $-linear combinations. Moreover, it extends
$\alpha_{F}$ (that is, for any $f\in\Lambda$, we have $\left(  \alpha
_{F}\left[  \left[  \mathbf{y}\right]  \right]  \right)  \left(  f\right)
=\alpha_{F}\left(  f\right)  $).

Recall the skew Schur functions $s_{\lambda/\mu}$ defined in Subsection
\ref{subsect.proofs.skew}. Also, recall the symmetric functions $h_{\lambda}$
defined in Definition \ref{def.hlam}. In the First proof of Theorem
\ref{thm.G.pieri}, we have proved the equality%
\[
\sum_{\lambda\in\operatorname*{Par}}s_{\lambda}\left(  \mathbf{y}\right)
s_{\lambda/\mu}=\sum_{\lambda\in\operatorname*{Par}}s_{\mu}\left(
\mathbf{y}\right)  m_{\lambda}\left(  \mathbf{y}\right)  h_{\lambda}.
\]
Consider this as an equality in the ring $\Lambda\left[  \left[
\mathbf{y}\right]  \right]  =\Lambda\left[  \left[  y_{1},y_{2},y_{3}%
,\ldots\right]  \right]  $. Apply the map $\alpha_{F}\left[  \left[
\mathbf{y}\right]  \right]  :\Lambda\left[  \left[  \mathbf{y}\right]
\right]  \rightarrow\mathbf{k}\left[  \left[  \mathbf{y}\right]  \right]  $ to
both sides of this equality. We obtain%
\[
\left(  \alpha_{F}\left[  \left[  \mathbf{y}\right]  \right]  \right)  \left(
\sum_{\lambda\in\operatorname*{Par}}s_{\lambda}\left(  \mathbf{y}\right)
s_{\lambda/\mu}\right)  =\left(  \alpha_{F}\left[  \left[  \mathbf{y}\right]
\right]  \right)  \left(  \sum_{\lambda\in\operatorname*{Par}}s_{\mu}\left(
\mathbf{y}\right)  m_{\lambda}\left(  \mathbf{y}\right)  h_{\lambda}\right)
.
\]
Comparing this with%
\begin{align*}
&  \left(  \alpha_{F}\left[  \left[  \mathbf{y}\right]  \right]  \right)
\left(  \sum_{\lambda\in\operatorname*{Par}}s_{\lambda}\left(  \mathbf{y}%
\right)  s_{\lambda/\mu}\right) \\
&  =\sum_{\lambda\in\operatorname*{Par}}s_{\lambda}\left(  \mathbf{y}\right)
\cdot\underbrace{\left(  \alpha_{F}\left[  \left[  \mathbf{y}\right]  \right]
\right)  \left(  s_{\lambda/\mu}\right)  }_{\substack{=\alpha_{F}\left(
s_{\lambda/\mu}\right)  \\\text{(since }\alpha_{F}\left[  \left[
\mathbf{y}\right]  \right]  \text{ extends }\alpha_{F}\text{)}}}\\
&  \ \ \ \ \ \ \ \ \ \ \left(  \text{since the map }\alpha_{F}\left[  \left[
\mathbf{y}\right]  \right]  \text{ preserves infinite }\mathbf{k}\left[
\left[  \mathbf{y}\right]  \right]  \text{-linear combinations}\right) \\
&  =\sum_{\lambda\in\operatorname*{Par}}s_{\lambda}\left(  \mathbf{y}\right)
\cdot\underbrace{\alpha_{F}\left(  s_{\lambda/\mu}\right)  }%
_{\substack{=\operatorname*{pet}\nolimits_{F}\left(  \lambda,\mu\right)
\\\text{(by (\ref{pf.thm.GF.pieri.alks}))}}}=\sum_{\lambda\in
\operatorname*{Par}}s_{\lambda}\left(  \mathbf{y}\right)  \cdot
\operatorname*{pet}\nolimits_{F}\left(  \lambda,\mu\right)  =\sum_{\lambda
\in\operatorname*{Par}}\operatorname*{pet}\nolimits_{F}\left(  \lambda
,\mu\right)  \cdot s_{\lambda}\left(  \mathbf{y}\right)  ,
\end{align*}
we obtain%
\begin{align*}
&  \sum_{\lambda\in\operatorname*{Par}}\operatorname*{pet}\nolimits_{F}\left(
\lambda,\mu\right)  \cdot s_{\lambda}\left(  \mathbf{y}\right) \\
&  =\left(  \alpha_{F}\left[  \left[  \mathbf{y}\right]  \right]  \right)
\left(  \sum_{\lambda\in\operatorname*{Par}}s_{\mu}\left(  \mathbf{y}\right)
m_{\lambda}\left(  \mathbf{y}\right)  h_{\lambda}\right)  =\sum_{\lambda
\in\operatorname*{Par}}s_{\mu}\left(  \mathbf{y}\right)  m_{\lambda}\left(
\mathbf{y}\right)  \underbrace{\left(  \alpha_{F}\left[  \left[
\mathbf{y}\right]  \right]  \right)  \left(  h_{\lambda}\right)
}_{\substack{=\alpha_{F}\left(  h_{\lambda}\right)  \\\text{(since }\alpha
_{F}\left[  \left[  \mathbf{y}\right]  \right]  \text{ extends }\alpha
_{F}\text{)}}}\\
&  \ \ \ \ \ \ \ \ \ \ \left(  \text{since the map }\alpha_{F}\left[  \left[
\mathbf{y}\right]  \right]  \text{ preserves infinite }\mathbf{k}\left[
\left[  \mathbf{y}\right]  \right]  \text{-linear combinations}\right) \\
&  =\sum_{\lambda\in\operatorname*{Par}}s_{\mu}\left(  \mathbf{y}\right)
m_{\lambda}\left(  \mathbf{y}\right)  \underbrace{\alpha_{F}\left(
h_{\lambda}\right)  }_{\substack{=f_{\lambda}\\\text{(by
(\ref{pf.thm.GF.main.alkh}))}}}=\sum_{\lambda\in\operatorname*{Par}}s_{\mu
}\left(  \mathbf{y}\right)  m_{\lambda}\left(  \mathbf{y}\right)  \cdot
f_{\lambda}=\sum_{\lambda\in\operatorname*{Par}}f_{\lambda}\cdot s_{\mu
}\left(  \mathbf{y}\right)  m_{\lambda}\left(  \mathbf{y}\right)  .
\end{align*}
Renaming the indeterminates $\mathbf{y}=\left(  y_{1},y_{2},y_{3}%
,\ldots\right)  $ as $\mathbf{x}=\left(  x_{1},x_{2},x_{3},\ldots\right)  $ on
both sides of this equality, we obtain%
\[
\sum_{\lambda\in\operatorname*{Par}}\operatorname*{pet}\nolimits_{F}\left(
\lambda,\mu\right)  \cdot s_{\lambda}\left(  \mathbf{x}\right)  =\sum
_{\lambda\in\operatorname*{Par}}f_{\lambda}\cdot\underbrace{s_{\mu}\left(
\mathbf{x}\right)  }_{=s_{\mu}}\underbrace{m_{\lambda}\left(  \mathbf{x}%
\right)  }_{=m_{\lambda}}=\sum_{\lambda\in\operatorname*{Par}}%
\underbrace{f_{\lambda}\cdot s_{\mu}}_{=s_{\mu}f_{\lambda}}m_{\lambda}%
=\sum_{\lambda\in\operatorname*{Par}}s_{\mu}f_{\lambda}m_{\lambda}.
\]
Comparing this with%
\[
G_{F}\cdot s_{\mu}=s_{\mu}\cdot\underbrace{G_{F}}_{\substack{=\sum_{\lambda
\in\operatorname*{Par}}f_{\lambda}m_{\lambda}\\\text{(by Proposition
\ref{prop.GF.basics} \textbf{(b)})}}}=s_{\mu}\cdot\sum_{\lambda\in
\operatorname*{Par}}f_{\lambda}m_{\lambda}=\sum_{\lambda\in\operatorname*{Par}%
}s_{\mu}f_{\lambda}m_{\lambda},
\]
we obtain
\[
G_{F}\cdot s_{\mu}=\sum_{\lambda\in\operatorname*{Par}}\operatorname*{pet}%
\nolimits_{F}\left(  \lambda,\mu\right)  \cdot\underbrace{s_{\lambda}\left(
\mathbf{x}\right)  }_{=s_{\lambda}}=\sum_{\lambda\in\operatorname*{Par}%
}\operatorname*{pet}\nolimits_{F}\left(  \lambda,\mu\right)  s_{\lambda}.
\]
This proves Theorem \ref{thm.GF.pieri}.
\end{proof}

\begin{proof}
[Proof of Corollary \ref{cor.GFm.pieri}.]Forget that we fixed $m$. If
$n\in\mathbb{N}$, then the power series $%
\begin{cases}
G_{F,n-\left\vert \mu\right\vert }\cdot s_{\mu}, & \text{if }n\geq\left\vert
\mu\right\vert ;\\
0, & \text{if }n<\left\vert \mu\right\vert
\end{cases}
\in\mathbf{k}\left[  \left[  x_{1},x_{2},x_{3},\ldots\right]  \right]  $ is
homogeneous of degree $n$\ \ \ \ \footnote{\textit{Proof.} Let $n\in
\mathbb{N}$. We must prove that the power series $%
\begin{cases}
G_{F,n-\left\vert \mu\right\vert }\cdot s_{\mu}, & \text{if }n\geq\left\vert
\mu\right\vert ;\\
0, & \text{if }n<\left\vert \mu\right\vert
\end{cases}
$ is homogeneous of degree $n$.
\par
We are in one of the following two cases:
\par
\textit{Case 1:} We have $n\geq\left\vert \mu\right\vert $.
\par
\textit{Case 2:} We have $n<\left\vert \mu\right\vert $.
\par
Let us first consider Case 1. In this case, we have $n\geq\left\vert
\mu\right\vert $. Hence, $n-\left\vert \mu\right\vert \in\mathbb{N}$. But
Proposition \ref{prop.GF.basics} \textbf{(a)} (applied to $m=n-\left\vert
\mu\right\vert $) yields that the power series $G_{F,n-\left\vert
\mu\right\vert }$ is the $\left(  n-\left\vert \mu\right\vert \right)  $-th
degree homogeneous component of $G_{F}$. Hence, $G_{F,n-\left\vert
\mu\right\vert }$ is homogeneous of degree $n-\left\vert \mu\right\vert $.
\par
On the other hand, recall that for any $\lambda\in\operatorname*{Par}$, the
Schur function $s_{\lambda}$ is homogeneous of degree $\left\vert
\lambda\right\vert $. Applying this to $\lambda=\mu$, we conclude that the
Schur function $s_{\mu}$ is homogeneous of degree $\left\vert \mu\right\vert
$.
\par
So we know that $G_{F,n-\left\vert \mu\right\vert }$ is homogeneous of degree
$n-\left\vert \mu\right\vert $, whereas $s_{\mu}$ is homogeneous of degree
$\left\vert \mu\right\vert $. Since $\Lambda$ is a graded algebra, this
entails that the power series $G_{F,n-\left\vert \mu\right\vert }\cdot s_{\mu
}$ (being the product of $G_{F,n-\left\vert \mu\right\vert }$ and $s_{\mu}$)
is homogeneous of degree $\left(  n-\left\vert \mu\right\vert \right)
+\left\vert \mu\right\vert $. In other words, the power series
$G_{F,n-\left\vert \mu\right\vert }\cdot s_{\mu}$ is homogeneous of degree $n$
(since $\left(  n-\left\vert \mu\right\vert \right)  +\left\vert
\mu\right\vert =n$). In other words, the power series $%
\begin{cases}
G_{F,n-\left\vert \mu\right\vert }\cdot s_{\mu}, & \text{if }n\geq\left\vert
\mu\right\vert ;\\
0, & \text{if }n<\left\vert \mu\right\vert
\end{cases}
$ is homogeneous of degree $n$ (since%
\[%
\begin{cases}
G_{F,n-\left\vert \mu\right\vert }\cdot s_{\mu}, & \text{if }n\geq\left\vert
\mu\right\vert ;\\
0, & \text{if }n<\left\vert \mu\right\vert
\end{cases}
=G_{F,n-\left\vert \mu\right\vert }\cdot s_{\mu}\ \ \ \ \ \ \ \ \ \ \left(
\text{because }n\geq\left\vert \mu\right\vert \right)
\]
). Thus, we have proved in Case 1 that the power series $%
\begin{cases}
G_{F,n-\left\vert \mu\right\vert }\cdot s_{\mu}, & \text{if }n\geq\left\vert
\mu\right\vert ;\\
0, & \text{if }n<\left\vert \mu\right\vert
\end{cases}
$ is homogeneous of degree $n$.
\par
Let us now consider Case 2. In this case, we have $n<\left\vert \mu\right\vert
$. The power series $0$ is homogeneous of degree $n$ (since $0$ is homogeneous
of any degree). In other words, the power series $%
\begin{cases}
G_{F,n-\left\vert \mu\right\vert }\cdot s_{\mu}, & \text{if }n\geq\left\vert
\mu\right\vert ;\\
0, & \text{if }n<\left\vert \mu\right\vert
\end{cases}
$ is homogeneous of degree $n$ (since%
\[%
\begin{cases}
G_{F,n-\left\vert \mu\right\vert }\cdot s_{\mu}, & \text{if }n\geq\left\vert
\mu\right\vert ;\\
0, & \text{if }n<\left\vert \mu\right\vert
\end{cases}
=0\ \ \ \ \ \ \ \ \ \ \left(  \text{because }n<\left\vert \mu\right\vert
\right)
\]
). Thus, we have proved in Case 2 that the power series $%
\begin{cases}
G_{F,n-\left\vert \mu\right\vert }\cdot s_{\mu}, & \text{if }n\geq\left\vert
\mu\right\vert ;\\
0, & \text{if }n<\left\vert \mu\right\vert
\end{cases}
$ is homogeneous of degree $n$.
\par
Thus, our claim (namely, that the power series $%
\begin{cases}
G_{F,n-\left\vert \mu\right\vert }\cdot s_{\mu}, & \text{if }n\geq\left\vert
\mu\right\vert ;\\
0, & \text{if }n<\left\vert \mu\right\vert
\end{cases}
$ is homogeneous of degree $n$) has been proven in both Cases 1 and 2. Since
these cases cover all possibilities, we thus conclude that our claim always
holds. Qed.}.

In the proof of Proposition \ref{prop.GF.basics} \textbf{(a)}, we have shown
that $G_{F}=\sum_{m\in\mathbb{N}}G_{F,m}$. Multiplying both sides of this
equality by $s_{\mu}$, we find%
\[
G_{F}\cdot s_{\mu}=\left(  \sum_{m\in\mathbb{N}}G_{F,m}\right)  \cdot s_{\mu
}=\sum_{m\in\mathbb{N}}G_{F,m}\cdot s_{\mu}.
\]
Comparing this with%
\begin{align*}
&  \sum_{n\in\mathbb{N}}%
\begin{cases}
G_{F,n-\left\vert \mu\right\vert }\cdot s_{\mu}, & \text{if }n\geq\left\vert
\mu\right\vert ;\\
0, & \text{if }n<\left\vert \mu\right\vert
\end{cases}
\\
&  =\sum_{\substack{n\in\mathbb{N};\\n\geq\left\vert \mu\right\vert
}}\underbrace{%
\begin{cases}
G_{F,n-\left\vert \mu\right\vert }\cdot s_{\mu}, & \text{if }n\geq\left\vert
\mu\right\vert ;\\
0, & \text{if }n<\left\vert \mu\right\vert
\end{cases}
}_{\substack{=G_{F,n-\left\vert \mu\right\vert }\cdot s_{\mu}\\\text{(since
}n\geq\left\vert \mu\right\vert \text{)}}}\ \ \ \ +\sum_{\substack{n\in
\mathbb{N};\\n<\left\vert \mu\right\vert }}\underbrace{%
\begin{cases}
G_{F,n-\left\vert \mu\right\vert }\cdot s_{\mu}, & \text{if }n\geq\left\vert
\mu\right\vert ;\\
0, & \text{if }n<\left\vert \mu\right\vert
\end{cases}
}_{\substack{=0\\\text{(since }n<\left\vert \mu\right\vert \text{)}}}\\
&  \ \ \ \ \ \ \ \ \ \ \ \ \ \ \ \ \ \ \ \ \left(
\begin{array}
[c]{c}%
\text{since each }n\in\mathbb{N}\text{ satisfies either }n\geq\left\vert
\mu\right\vert \text{ or }n<\left\vert \mu\right\vert \\
\text{(but not both at the same time)}%
\end{array}
\right) \\
&  =\sum_{\substack{n\in\mathbb{N};\\n\geq\left\vert \mu\right\vert
}}G_{F,n-\left\vert \mu\right\vert }\cdot s_{\mu}+\underbrace{\sum
_{\substack{n\in\mathbb{N};\\n<\left\vert \mu\right\vert }}0}_{=0}%
=\sum_{\substack{n\in\mathbb{N};\\n\geq\left\vert \mu\right\vert
}}G_{F,n-\left\vert \mu\right\vert }\cdot s_{\mu}\\
&  =\sum_{m\in\mathbb{N}}G_{F,m}\cdot s_{\mu}\ \ \ \ \ \ \ \ \ \ \left(
\begin{array}
[c]{c}%
\text{here, we have substituted }m\\
\text{for }n-\left\vert \mu\right\vert \text{ in the sum}%
\end{array}
\right)  ,
\end{align*}
we obtain%
\begin{equation}
G_{F}\cdot s_{\mu}=\sum_{n\in\mathbb{N}}%
\begin{cases}
G_{F,n-\left\vert \mu\right\vert }\cdot s_{\mu}, & \text{if }n\geq\left\vert
\mu\right\vert ;\\
0, & \text{if }n<\left\vert \mu\right\vert
\end{cases}
\ \ \ . \label{pf.cor.GFm.pieri.2}%
\end{equation}

But recall that each $%
\begin{cases}
G_{F,n-\left\vert \mu\right\vert }\cdot s_{\mu}, & \text{if }n\geq\left\vert
\mu\right\vert ;\\
0, & \text{if }n<\left\vert \mu\right\vert
\end{cases}
$ is homogeneous of degree $n$. Thus, the equality (\ref{pf.cor.GFm.pieri.2})
reveals that the family
\[
\left(
\begin{cases}
G_{F,n-\left\vert \mu\right\vert }\cdot s_{\mu}, & \text{if }n\geq\left\vert
\mu\right\vert ;\\
0, & \text{if }n<\left\vert \mu\right\vert
\end{cases}
\right)  _{n\in\mathbb{N}}%
\]
is the homogeneous decomposition of $G_{F}\cdot s_{\mu}$ (by the definition of
a homogeneous decomposition). Therefore, for each $n\in\mathbb{N}$, the power
series $%
\begin{cases}
G_{F,n-\left\vert \mu\right\vert }\cdot s_{\mu}, & \text{if }n\geq\left\vert
\mu\right\vert ;\\
0, & \text{if }n<\left\vert \mu\right\vert
\end{cases}
$ is the $n$-th degree homogeneous component of $G_{F}\cdot s_{\mu}$.

Now, let $m\in\mathbb{N}$. We have just shown that for each $n\in\mathbb{N}$,
the power series $%
\begin{cases}
G_{F,n-\left\vert \mu\right\vert }\cdot s_{\mu}, & \text{if }n\geq\left\vert
\mu\right\vert ;\\
0, & \text{if }n<\left\vert \mu\right\vert
\end{cases}
$ is the $n$-th degree homogeneous component of $G_{F}\cdot s_{\mu}$. Applying
this to $n=m+\left\vert \mu\right\vert $, we conclude that the power series
\newline$%
\begin{cases}
G_{F,\left(  m+\left\vert \mu\right\vert \right)  -\left\vert \mu\right\vert
}\cdot s_{\mu}, & \text{if }m+\left\vert \mu\right\vert \geq\left\vert
\mu\right\vert ;\\
0, & \text{if }m+\left\vert \mu\right\vert <\left\vert \mu\right\vert
\end{cases}
$ is the $\left(  m+\left\vert \mu\right\vert \right)  $-th degree homogeneous
component of $G_{F}\cdot s_{\mu}$. Since%
\begin{align*}
&
\begin{cases}
G_{F,\left(  m+\left\vert \mu\right\vert \right)  -\left\vert \mu\right\vert
}\cdot s_{\mu}, & \text{if }m+\left\vert \mu\right\vert \geq\left\vert
\mu\right\vert ;\\
0, & \text{if }m+\left\vert \mu\right\vert <\left\vert \mu\right\vert
\end{cases}
\\
&  =G_{F,\left(  m+\left\vert \mu\right\vert \right)  -\left\vert
\mu\right\vert }\cdot s_{\mu}\ \ \ \ \ \ \ \ \ \ \left(  \text{since
}m+\left\vert \mu\right\vert \geq\left\vert \mu\right\vert \text{ (because
}m\geq0\text{)}\right) \\
&  =G_{F,m}\cdot s_{\mu}\ \ \ \ \ \ \ \ \ \ \left(  \text{since }\left(
m+\left\vert \mu\right\vert \right)  -\left\vert \mu\right\vert =m\right)  ,
\end{align*}
we can rewrite this as follows: The power series $G_{F,m}\cdot s_{\mu}$ is the
$\left(  m+\left\vert \mu\right\vert \right)  $-th degree homogeneous
component of $G_{F}\cdot s_{\mu}$. In other words,%
\begin{align}
&  G_{F,m}\cdot s_{\mu}\nonumber\\
&  =\left(  \text{the }\left(  m+\left\vert \mu\right\vert \right)  \text{-th
degree homogeneous component of }G_{F}\cdot s_{\mu}\right)  .
\label{pf.cor.GFm.pieri.left}%
\end{align}

On the other hand, Theorem \ref{thm.GF.pieri} yields%
\begin{align}
G_{F}\cdot s_{\mu}  &  =\underbrace{\sum_{\lambda\in\operatorname*{Par}}%
}_{\substack{=\sum_{n\in\mathbb{N}}\ \ \sum_{\substack{\lambda\in
\operatorname*{Par};\\\left\vert \lambda\right\vert =n}}\\\text{(since each
}\lambda\in\operatorname*{Par}\\\text{satisfies }\left\vert \lambda\right\vert
\in\mathbb{N}\text{)}}}\operatorname*{pet}\nolimits_{F}\left(  \lambda
,\mu\right)  s_{\lambda}\nonumber\\
&  =\sum_{n\in\mathbb{N}}\ \ \sum_{\substack{\lambda\in\operatorname*{Par}%
;\\\left\vert \lambda\right\vert =n}}\operatorname*{pet}\nolimits_{F}\left(
\lambda,\mu\right)  s_{\lambda}. \label{pf.cor.GFm.pieri.4}%
\end{align}
For each $n\in\mathbb{N}$, the formal power series $\sum_{\substack{\lambda
\in\operatorname*{Par};\\\left\vert \lambda\right\vert =n}}\operatorname*{pet}%
\nolimits_{F}\left(  \lambda,\mu\right)  s_{\lambda}$ is homogeneous of degree
$n$\ \ \ \ \footnote{\textit{Proof.} Let $n\in\mathbb{N}$. Recall that for any
$\lambda\in\operatorname*{Par}$, the Schur function $s_{\lambda}$ is
homogeneous of degree $\left\vert \lambda\right\vert $. Hence, if $\lambda
\in\operatorname*{Par}$ satisfies $\left\vert \lambda\right\vert =n$, then the
Schur function $s_{\lambda}$ is homogeneous of degree $n$ (since $\left\vert
\lambda\right\vert =n$). Thus, $\sum_{\substack{\lambda\in\operatorname*{Par}%
;\\\left\vert \lambda\right\vert =n}}\operatorname*{pet}\nolimits_{F}\left(
\lambda,\mu\right)  s_{\lambda}$ is a $\mathbf{k}$-linear combination of Schur
functions that are homogeneous of degree $n$. Therefore, $\sum
_{\substack{\lambda\in\operatorname*{Par};\\\left\vert \lambda\right\vert
=n}}\operatorname*{pet}\nolimits_{F}\left(  \lambda,\mu\right)  s_{\lambda}$
is homogeneous of degree $n$. Qed.}. Thus, the equality
(\ref{pf.cor.GFm.pieri.4}) reveals that
\[
\left(  \sum_{\substack{\lambda\in\operatorname*{Par};\\\left\vert
\lambda\right\vert =n}}\operatorname*{pet}\nolimits_{F}\left(  \lambda
,\mu\right)  s_{\lambda}\right)  _{n\in\mathbb{N}}%
\]
is the homogeneous decomposition of $G_{F}\cdot s_{\mu}$. Therefore, for each
$n\in\mathbb{N}$, the power series $\sum_{\substack{\lambda\in
\operatorname*{Par};\\\left\vert \lambda\right\vert =n}}\operatorname*{pet}%
\nolimits_{F}\left(  \lambda,\mu\right)  s_{\lambda}$ is the $n$-th degree
homogeneous component of $G_{F}\cdot s_{\mu}$. Applying this to
$n=m+\left\vert \mu\right\vert $, we conclude that the power series
$\sum_{\substack{\lambda\in\operatorname*{Par};\\\left\vert \lambda\right\vert
=m+\left\vert \mu\right\vert }}\operatorname*{pet}\nolimits_{F}\left(
\lambda,\mu\right)  s_{\lambda}$ is the $\left(  m+\left\vert \mu\right\vert
\right)  $-th degree homogeneous component of $G_{F}\cdot s_{\mu}$. In other
words,%
\begin{align*}
&  \sum_{\substack{\lambda\in\operatorname*{Par};\\\left\vert \lambda
\right\vert =m+\left\vert \mu\right\vert }}\operatorname*{pet}\nolimits_{F}%
\left(  \lambda,\mu\right)  s_{\lambda}\\
&  =\left(  \text{the }\left(  m+\left\vert \mu\right\vert \right)  \text{-th
degree homogeneous component of }G_{F}\cdot s_{\mu}\right)  .
\end{align*}
Comparing this with (\ref{pf.cor.GFm.pieri.left}), we find%
\[
G_{F,m}\cdot s_{\mu}=\underbrace{\sum_{\substack{\lambda\in\operatorname*{Par}%
;\\\left\vert \lambda\right\vert =m+\left\vert \mu\right\vert }}}%
_{\substack{=\sum_{\lambda\in\operatorname*{Par}\nolimits_{m+\left\vert
\mu\right\vert }}\\\text{(since }\operatorname*{Par}\nolimits_{m+\left\vert
\mu\right\vert }\text{ is defined as the}\\\text{set of all }\lambda
\in\operatorname*{Par}\text{ satisfying }\left\vert \lambda\right\vert
=m+\left\vert \mu\right\vert \text{)}}}\operatorname*{pet}\nolimits_{F}\left(
\lambda,\mu\right)  s_{\lambda}=\sum_{\lambda\in\operatorname*{Par}%
\nolimits_{m+\left\vert \mu\right\vert }}\operatorname*{pet}\nolimits_{F}%
\left(  \lambda,\mu\right)  s_{\lambda}.
\]
This proves Corollary \ref{cor.GFm.pieri}.
\end{proof}

\begin{proof}
[Proof of Theorem \ref{thm.GF.main}.]Theorem \ref{thm.GF.pieri} (applied to
$\mu=\varnothing$) yields%
\[
G_{F}\cdot s_{\varnothing}=\sum_{\lambda\in\operatorname*{Par}}%
\operatorname*{pet}\nolimits_{F}\left(  \lambda,\varnothing\right)
s_{\lambda}.
\]
Comparing this with $G_{F}\cdot\underbrace{s_{\varnothing}}_{=1}=G_{F}$, we
obtain%
\[
G_{F}=\sum_{\lambda\in\operatorname*{Par}}\operatorname*{pet}\nolimits_{F}%
\left(  \lambda,\varnothing\right)  s_{\lambda}.
\]
This proves Theorem \ref{thm.GF.main}.
\end{proof}

\begin{proof}
[Proof of Corollary \ref{cor.GFm.main}.]Corollary \ref{cor.GFm.pieri} (applied
to $\mu=\varnothing$) yields%
\[
G_{F,m}\cdot s_{\varnothing}=\sum_{\lambda\in\operatorname*{Par}%
\nolimits_{m+\left\vert \varnothing\right\vert }}\operatorname*{pet}%
\nolimits_{F}\left(  \lambda,\varnothing\right)  s_{\lambda}.
\]
In view of $G_{F,m}\cdot\underbrace{s_{\varnothing}}_{=1}=G_{F,m}$ and
$m+\underbrace{\left\vert \varnothing\right\vert }_{=0}=m$, we can rewrite
this as
\[
G_{F,m}=\sum_{\lambda\in\operatorname*{Par}\nolimits_{m}}\operatorname*{pet}%
\nolimits_{F}\left(  \lambda,\varnothing\right)  s_{\lambda}.
\]
This proves Corollary \ref{cor.GFm.main}.
\end{proof}
\end{verlong}

Proposition \ref{prop.GF.basics} \textbf{(c)} shows that $G_{F,m}\in\Lambda$
for each $m\in\mathbb{N}$. Hence, we can apply the comultiplication $\Delta$
of the Hopf algebra $\Lambda$ to $G_{F,m}$. The next theorem (which
generalizes Theorem \ref{thm.DeltaGkm}) gives a simple expression for the
result of this:

\begin{theorem}
\label{thm.DeltaGFm}Let $m\in\mathbb{N}$. Then,%
\[
\Delta\left(  G_{F,m}\right)  =\sum_{i=0}^{m}G_{F,i}\otimes G_{F,m-i}.
\]

\end{theorem}

\begin{vershort}
\begin{proof}
[Proof of Theorem \ref{thm.DeltaGFm} (sketched).]Proposition
\ref{prop.GF.basics} \textbf{(b)} tells us that $G_{F}=\prod_{i=1}^{\infty
}F\left(  x_{i}\right)  $. Comparing this with the equality $G_{F}=\sum
_{k\in\mathbb{N}}G_{F,k}$ (which follows from Proposition \ref{prop.GF.basics}
\textbf{(a)}), we obtain%
\begin{equation}
\sum_{k\in\mathbb{N}}G_{F,k}=\prod_{i=1}^{\infty}F\left(  x_{i}\right)  .
\label{pf.thm.DeltaGFm.short.1}%
\end{equation}
Substituting the variables $y_{1},y_{2},y_{3},\ldots$ for the variables
$x_{1},x_{2},x_{3},\ldots$ in this equality, we obtain
\begin{equation}
\sum_{k\in\mathbb{N}}G_{F,k}\left(  \mathbf{y}\right)  =\prod_{i=1}^{\infty
}F\left(  y_{i}\right)  . \label{pf.thm.DeltaGFm.short.1y}%
\end{equation}
Substituting the variables $x_{1},x_{2},x_{3},\ldots,y_{1},y_{2},y_{3},\ldots$
for the variables $x_{1},x_{2},x_{3},\ldots$ on both sides of the equality
$G_{F}=\prod_{i=1}^{\infty}F\left(  x_{i}\right)  $, we obtain%
\begin{align*}
G_{F}\left(  \mathbf{x},\mathbf{y}\right)   &  =\underbrace{\left(
\prod_{i=1}^{\infty}F\left(  x_{i}\right)  \right)  }_{\substack{=\sum
_{k\in\mathbb{N}}G_{F,k}\\\text{(by (\ref{pf.thm.DeltaGFm.short.1}))}%
}}\underbrace{\left(  \prod_{i=1}^{\infty}F\left(  y_{i}\right)  \right)
}_{\substack{=\sum_{k\in\mathbb{N}}G_{F,k}\left(  \mathbf{y}\right)
\\\text{(by (\ref{pf.thm.DeltaGFm.short.1y}))}}}=\left(  \sum_{k\in\mathbb{N}%
}\underbrace{G_{F,k}}_{=G_{F,k}\left(  \mathbf{x}\right)  }\right)  \left(
\sum_{k\in\mathbb{N}}G_{F,k}\left(  \mathbf{y}\right)  \right) \\
&  =\left(  \sum_{k\in\mathbb{N}}G_{F,k}\left(  \mathbf{x}\right)  \right)
\left(  \sum_{k\in\mathbb{N}}G_{F,k}\left(  \mathbf{y}\right)  \right)
=\sum_{\left(  i,j\right)  \in\mathbb{N}\times\mathbb{N}}G_{F,i}\left(
\mathbf{x}\right)  G_{F,j}\left(  \mathbf{y}\right)  .
\end{align*}

Comparing the $m$-th degree homogeneous components of both sides of this
equality, we find%
\[
G_{F,m}\left(  \mathbf{x},\mathbf{y}\right)  =\sum_{\substack{\left(
i,j\right)  \in\mathbb{N}\times\mathbb{N};\\i+j=m}}G_{F,i}\left(
\mathbf{x}\right)  G_{F,j}\left(  \mathbf{y}\right)
\]
(since Proposition \ref{prop.GF.basics} \textbf{(a)} shows that the $m$-th
degree homogeneous component of $G_{F}\left(  \mathbf{x},\mathbf{y}\right)  $
is $G_{F,m}\left(  \mathbf{x},\mathbf{y}\right)  $, while that of the sum
$\sum_{\left(  i,j\right)  \in\mathbb{N}\times\mathbb{N}}\underbrace{G_{F,i}%
\left(  \mathbf{x}\right)  G_{F,j}\left(  \mathbf{y}\right)  }%
_{\text{homogeneous of degree }i+j}$ is the subsum $\sum_{\substack{\left(
i,j\right)  \in\mathbb{N}\times\mathbb{N};\\i+j=m}}G_{F,i}\left(
\mathbf{x}\right)  G_{F,j}\left(  \mathbf{y}\right)  $). Thus,%
\[
G_{F,m}\left(  \mathbf{x},\mathbf{y}\right)  =\sum_{\substack{\left(
i,j\right)  \in\mathbb{N}\times\mathbb{N};\\i+j=m}}G_{F,i}\left(
\mathbf{x}\right)  G_{F,j}\left(  \mathbf{y}\right)  =\sum_{i\in\left\{
0,1,\ldots,m\right\}  }G_{F,i}\left(  \mathbf{x}\right)  G_{F,m-i}\left(
\mathbf{y}\right)  .
\]
Hence, (\ref{eq.Delta-on-Lam.if}) holds for $f=G_{F,m}$, $I=\left\{
0,1,\ldots,m\right\}  $, $\left(  f_{1,i}\right)  _{i\in I}=\left(
G_{F,i}\right)  _{i\in\left\{  0,1,\ldots,m\right\}  }$ and $\left(
f_{2,i}\right)  _{i\in I}=\left(  G_{F,m-i}\right)  _{i\in\left\{
0,1,\ldots,m\right\}  }$. Therefore, (\ref{eq.Delta-on-Lam.then}) (applied to
these $f$, $I$, $\left(  f_{1,i}\right)  _{i\in I}$ and $\left(
f_{2,i}\right)  _{i\in I}$) yields%
\[
\Delta\left(  G_{F,m}\right)  =\sum_{i\in\left\{  0,1,\ldots,m\right\}
}G_{F,i}\otimes G_{F,m-i}=\sum_{i=0}^{m}G_{F,i}\otimes G_{F,m-i}.
\]
This proves Theorem \ref{thm.DeltaGFm}.
\end{proof}
\end{vershort}

\begin{verlong}
\begin{proof}
[Proof of Theorem \ref{thm.DeltaGFm}.]Forget that we fixed $m$. Recall that
\begin{equation}
G_{F}=\sum_{m\in\mathbb{N}}G_{F,m} \label{pf.thm.DeltaGFm.GF=sum}%
\end{equation}
(indeed, we have proved this in our proof of Proposition \ref{prop.GF.basics}
\textbf{(a)}). On the other hand, Proposition \ref{prop.GF.basics}
\textbf{(d)} says that the power series $G_{F}$ is symmetric. Furthermore,
Proposition \ref{prop.GF.basics} \textbf{(b)} tells us that
\begin{equation}
G_{F}=\prod_{i=1}^{\infty}F\left(  x_{i}\right)  . \label{pf.thm.DeltaGFm.0}%
\end{equation}
Comparing this with (\ref{pf.thm.DeltaGFm.GF=sum}), we obtain%
\begin{equation}
\sum_{m\in\mathbb{N}}G_{F,m}=\prod_{i=1}^{\infty}F\left(  x_{i}\right)  .
\label{pf.thm.DeltaGFm.1}%
\end{equation}
Substituting the variables $y_{1},y_{2},y_{3},\ldots$ for the variables
$x_{1},x_{2},x_{3},\ldots$ in this equality, we obtain
\begin{equation}
\sum_{m\in\mathbb{N}}G_{F,m}\left(  \mathbf{y}\right)  =\prod_{i=1}^{\infty
}F\left(  y_{i}\right)  \label{pf.thm.DeltaGFm.1y}%
\end{equation}
(since the left hand side of (\ref{pf.thm.DeltaGFm.1}) turns into $\sum
_{m\in\mathbb{N}}G_{F,m}\left(  \mathbf{y}\right)  $ upon this
substitution\footnote{because each $G_{F,m}$ turns into $G_{F,m}\left(
\mathbf{y}\right)  $ upon this substitution}, whereas the right hand side
turns into $\prod_{i=1}^{\infty}F\left(  y_{i}\right)  $).

On the other hand, let us substitute the variables $x_{1},x_{2},x_{3}%
,\ldots,y_{1},y_{2},y_{3},\ldots$ for the variables $x_{1},x_{2},x_{3},\ldots$
on both sides of the equality (\ref{pf.thm.DeltaGFm.0}). (This means that we
choose some bijection $\phi:\left\{  x_{1},x_{2},x_{3},\ldots\right\}
\rightarrow\left\{  x_{1},x_{2},x_{3},\ldots,y_{1},y_{2},y_{3},\ldots\right\}
$, and substitute $\phi\left(  x_{i}\right)  $ for each $x_{i}$ on both sides
of (\ref{pf.thm.DeltaGFm.0}).) Thus, we readily obtain%
\begin{equation}
G_{F}\left(  \mathbf{x},\mathbf{y}\right)  =\left(  \prod_{i=1}^{\infty
}F\left(  x_{i}\right)  \right)  \left(  \prod_{i=1}^{\infty}F\left(
y_{i}\right)  \right)  . \label{pf.thm.DeltaGFm.1xy}%
\end{equation}

[Here is a detailed \textit{proof of (\ref{pf.thm.DeltaGFm.1xy}):} Choose some
bijection $\phi:\left\{  x_{1},x_{2},x_{3},\ldots\right\}  \rightarrow\left\{
x_{1},x_{2},x_{3},\ldots,y_{1},y_{2},y_{3},\ldots\right\}  $. (Such a
bijection clearly exists, since both $\left\{  x_{1},x_{2},x_{3}%
,\ldots\right\}  $ and $\left\{  x_{1},x_{2},x_{3},\ldots,y_{1},y_{2}%
,y_{3},\ldots\right\}  $ are countably infinite sets.) Then, $G_{F}\left(
\mathbf{x},\mathbf{y}\right)  $ is the result of substituting $\phi\left(
x_{i}\right)  $ for each $x_{i}$ in $G_{F}$ (by the definition of
$G_{F}\left(  \mathbf{x},\mathbf{y}\right)  $, since the power series $G_{F}$
is symmetric). In other words,%
\begin{equation}
G_{F}\left(  \mathbf{x},\mathbf{y}\right)  =G_{F}\left(  \phi\left(
x_{1}\right)  ,\phi\left(  x_{2}\right)  ,\phi\left(  x_{3}\right)
,\ldots\right)  =\prod_{i=1}^{\infty}F\left(  \phi\left(  x_{i}\right)
\right)  \label{pf.thm.DeltaGFm.1xy.1}%
\end{equation}
(here, we have substituted $\phi\left(  x_{i}\right)  $ for each $x_{i}$ on
both sides of the equality (\ref{pf.thm.DeltaGFm.0})). Now, recall that the
map $\phi$ is a bijection from $\left\{  x_{1},x_{2},x_{3},\ldots\right\}  $
to $\left\{  x_{1},x_{2},x_{3},\ldots,y_{1},y_{2},y_{3},\ldots\right\}  $.
Thus, its values $\phi\left(  x_{1}\right)  ,\phi\left(  x_{2}\right)
,\phi\left(  x_{3}\right)  ,\ldots$ are precisely the indeterminates
\newline$x_{1},x_{2},x_{3},\ldots,y_{1},y_{2},y_{3},\ldots$ (in some order).
Hence,%
\begin{align*}
&  F\left(  \phi\left(  x_{1}\right)  \right)  \cdot F\left(  \phi\left(
x_{2}\right)  \right)  \cdot F\left(  \phi\left(  x_{3}\right)  \right)
\cdot\cdots\\
&  =\prod_{u\in\left\{  x_{1},x_{2},x_{3},\ldots,y_{1},y_{2},y_{3}%
,\ldots\right\}  }F\left(  u\right) \\
&  =\underbrace{\left(  F\left(  x_{1}\right)  \cdot F\left(  x_{2}\right)
\cdot F\left(  x_{3}\right)  \cdot\cdots\right)  }_{=\prod_{i=1}^{\infty
}F\left(  x_{i}\right)  }\cdot\underbrace{\left(  F\left(  y_{1}\right)  \cdot
F\left(  y_{2}\right)  \cdot F\left(  y_{3}\right)  \cdot\cdots\right)
}_{=\prod_{i=1}^{\infty}F\left(  y_{i}\right)  }\\
&  =\left(  \prod_{i=1}^{\infty}F\left(  x_{i}\right)  \right)  \left(
\prod_{i=1}^{\infty}F\left(  y_{i}\right)  \right)  .
\end{align*}
Hence, (\ref{pf.thm.DeltaGFm.1xy.1}) becomes%
\begin{align*}
G_{F}\left(  \mathbf{x},\mathbf{y}\right)   &  =\prod_{i=1}^{\infty}F\left(
\phi\left(  x_{i}\right)  \right)  =F\left(  \phi\left(  x_{1}\right)
\right)  \cdot F\left(  \phi\left(  x_{2}\right)  \right)  \cdot F\left(
\phi\left(  x_{3}\right)  \right)  \cdot\cdots\\
&  =\left(  \prod_{i=1}^{\infty}F\left(  x_{i}\right)  \right)  \left(
\prod_{i=1}^{\infty}F\left(  y_{i}\right)  \right)  .
\end{align*}
This proves (\ref{pf.thm.DeltaGFm.1xy}).]

Now, (\ref{pf.thm.DeltaGFm.1xy}) becomes%
\begin{align}
G_{F}\left(  \mathbf{x},\mathbf{y}\right)   &  =\underbrace{\left(
\prod_{i=1}^{\infty}F\left(  x_{i}\right)  \right)  }_{\substack{=\sum
_{m\in\mathbb{N}}G_{F,m}\\\text{(by (\ref{pf.thm.DeltaGFm.1}))}}%
}\underbrace{\left(  \prod_{i=1}^{\infty}F\left(  y_{i}\right)  \right)
}_{\substack{=\sum_{m\in\mathbb{N}}G_{F,m}\left(  \mathbf{y}\right)
\\\text{(by (\ref{pf.thm.DeltaGFm.1y}))}}}\nonumber\\
&  =\left(  \sum_{m\in\mathbb{N}}\underbrace{G_{F,m}}_{=G_{F,m}\left(
\mathbf{x}\right)  }\right)  \left(  \sum_{m\in\mathbb{N}}G_{F,m}\left(
\mathbf{y}\right)  \right) \nonumber\\
&  =\underbrace{\left(  \sum_{m\in\mathbb{N}}G_{F,m}\left(  \mathbf{x}\right)
\right)  }_{\substack{=\sum_{i\in\mathbb{N}}G_{F,i}\left(  \mathbf{x}\right)
\\\text{(here, we have renamed}\\\text{the summation index }m\text{ as
}i\text{)}}}\ \ \underbrace{\left(  \sum_{m\in\mathbb{N}}G_{F,m}\left(
\mathbf{y}\right)  \right)  }_{\substack{=\sum_{j\in\mathbb{N}}G_{F,j}\left(
\mathbf{y}\right)  \\\text{(here, we have renamed}\\\text{the summation index
}m\text{ as }j\text{)}}}\nonumber\\
&  =\left(  \sum_{i\in\mathbb{N}}G_{F,i}\left(  \mathbf{x}\right)  \right)
\left(  \sum_{j\in\mathbb{N}}G_{F,j}\left(  \mathbf{y}\right)  \right)
=\underbrace{\sum_{i\in\mathbb{N}}\ \ \sum_{j\in\mathbb{N}}}_{\substack{=\sum
_{\left(  i,j\right)  \in\mathbb{N}\times\mathbb{N}}\\=\sum_{n\in\mathbb{N}%
}\ \ \sum_{\substack{\left(  i,j\right)  \in\mathbb{N}\times\mathbb{N}%
;\\i+j=n}}}}G_{F,i}\left(  \mathbf{x}\right)  G_{F,j}\left(  \mathbf{y}\right)
\nonumber\\
&  =\sum_{n\in\mathbb{N}}\ \ \sum_{\substack{\left(  i,j\right)  \in
\mathbb{N}\times\mathbb{N};\\i+j=n}}G_{F,i}\left(  \mathbf{x}\right)
G_{F,j}\left(  \mathbf{y}\right) \nonumber\\
&  =\sum_{n\in\mathbb{N}}\left(  \sum_{\substack{\left(  i,j\right)
\in\mathbb{N}\times\mathbb{N};\\i+j=n}}G_{F,i}\left(  \mathbf{x}\right)
G_{F,j}\left(  \mathbf{y}\right)  \right)  . \label{pf.thm.DeltaGFm.4}%
\end{align}

If $n\in\mathbb{N}$, then the power series $\sum_{\substack{\left(
i,j\right)  \in\mathbb{N}\times\mathbb{N};\\i+j=n}}G_{F,i}\left(
\mathbf{x}\right)  G_{F,j}\left(  \mathbf{y}\right)  \in\mathbf{k}\left[
\left[  \mathbf{x},\mathbf{y}\right]  \right]  $ is homogeneous of degree
$n$\ \ \ \ \footnote{\textit{Proof.} Let $n\in\mathbb{N}$. We must prove that
the power series $\sum_{\substack{\left(  i,j\right)  \in\mathbb{N}%
\times\mathbb{N};\\i+j=n}}G_{F,i}\left(  \mathbf{x}\right)  G_{F,j}\left(
\mathbf{y}\right)  $ is homogeneous of degree $n$.
\par
Let $\left(  i,j\right)  \in\mathbb{N}\times\mathbb{N}$ be such that $i+j=n$.
Then, Proposition \ref{prop.GF.basics} \textbf{(a)} (applied to $m=i$) shows
that the formal power series $G_{F,i}$ is the $i$-th degree homogeneous
component of $G_{F}$. Hence, this formal power series $G_{F,i}$ is homogeneous
of degree $i$. In other words, $G_{F,i}\left(  \mathbf{x}\right)  $ is
homogeneous of degree $i$ (since $G_{F,i}\left(  \mathbf{x}\right)  =G_{F,i}%
$).
\par
Moreover, Proposition \ref{prop.GF.basics} \textbf{(a)} (applied to $m=j$)
shows that the formal power series $G_{F,j}$ is the $j$-th degree homogeneous
component of $G_{F}$. Hence, this formal power series $G_{F,j}$ is homogeneous
of degree $j$. Hence, the power series $G_{F,j}\left(  \mathbf{y}\right)  $ is
homogeneous of degree $j$ as well (since this $G_{F,j}\left(  \mathbf{y}%
\right)  $ is obtained by substituting $y_{1},y_{2},y_{3},\ldots$ for the
variables $x_{1},x_{2},x_{3},\ldots$ in $G_{F,j}$; but this substitution
clearly preserves homogeneity and degree).
\par
Now we have shown that the two power series $G_{F,i}\left(  \mathbf{x}\right)
$ and $G_{F,j}\left(  \mathbf{y}\right)  $ are homogeneous of degrees $i$ and
$j$, respectively. Thus, their product $G_{F,i}\left(  \mathbf{x}\right)
G_{F,j}\left(  \mathbf{y}\right)  $ is homogeneous of degree $i+j$. In other
words, $G_{F,i}\left(  \mathbf{x}\right)  G_{F,j}\left(  \mathbf{y}\right)  $
is homogeneous of degree $n$ (since $i+j=n$).
\par
Forget that we fixed $\left(  i,j\right)  $. We thus have shown that
$G_{F,i}\left(  \mathbf{x}\right)  G_{F,j}\left(  \mathbf{y}\right)  $ is
homogeneous of degree $n$ whenever $\left(  i,j\right)  \in\mathbb{N}%
\times\mathbb{N}$ satisfies $i+j=n$. In other words, each addend of the sum
$\sum_{\substack{\left(  i,j\right)  \in\mathbb{N}\times\mathbb{N}%
;\\i+j=n}}G_{F,i}\left(  \mathbf{x}\right)  G_{F,j}\left(  \mathbf{y}\right)
$ is homogeneous of degree $n$. Hence, the entire sum $\sum_{\substack{\left(
i,j\right)  \in\mathbb{N}\times\mathbb{N};\\i+j=n}}G_{F,i}\left(
\mathbf{x}\right)  G_{F,j}\left(  \mathbf{y}\right)  $ is homogeneous of
degree $n$ as well. This completes our proof.}. Thus, the equality
(\ref{pf.thm.DeltaGFm.4}) reveals that the family
\[
\left(  \sum_{\substack{\left(  i,j\right)  \in\mathbb{N}\times\mathbb{N}%
;\\i+j=n}}G_{F,i}\left(  \mathbf{x}\right)  G_{F,j}\left(  \mathbf{y}\right)
\right)  _{n\in\mathbb{N}}%
\]
is the homogeneous decomposition of $G_{F}\left(  \mathbf{x},\mathbf{y}%
\right)  $ (by the definition of a homogeneous decomposition).

On the other hand, we have%
\begin{equation}
G_{F}\left(  \mathbf{x},\mathbf{y}\right)  =\sum_{m\in\mathbb{N}}%
G_{F,m}\left(  \mathbf{x},\mathbf{y}\right)  . \label{pf.thm.DeltaGFm.1xy.4}%
\end{equation}

[\textit{Proof of (\ref{pf.thm.DeltaGFm.1xy.4}):} If we substitute the
variables $x_{1},x_{2},x_{3},\ldots,y_{1},y_{2},y_{3},\ldots$ for the
variables $x_{1},x_{2},x_{3},\ldots$ on both sides of the equality
(\ref{pf.thm.DeltaGFm.GF=sum}), then we obtain%
\[
G_{F}\left(  \mathbf{x},\mathbf{y}\right)  =\sum_{m\in\mathbb{N}}%
G_{F,m}\left(  \mathbf{x},\mathbf{y}\right)
\]
(because this substitution transforms $G_{F}$ into $G_{F}\left(
\mathbf{x},\mathbf{y}\right)  $ and transforms $G_{F,m}$ into $G_{F,m}\left(
\mathbf{x},\mathbf{y}\right)  $ for each $m\in\mathbb{N}$). This proves
(\ref{pf.thm.DeltaGFm.1xy.4}).]

Now, if $n\in\mathbb{N}$, then the power series $G_{F,n}\left(  \mathbf{x}%
,\mathbf{y}\right)  \in\mathbf{k}\left[  \left[  \mathbf{x},\mathbf{y}\right]
\right]  $ is homogeneous of degree $n$\ \ \ \ \footnote{\textit{Proof.} Let
$n\in\mathbb{N}$. We must prove that the power series $G_{F,n}\left(
\mathbf{x},\mathbf{y}\right)  $ is homogeneous of degree $n$.
\par
Proposition \ref{prop.GF.basics} \textbf{(a)} (applied to $m=n$) shows that
the formal power series $G_{F,n}$ is the $n$-th degree homogeneous component
of $G_{F}$. Hence, this formal power series $G_{F,n}$ is homogeneous of degree
$n$. Hence, the power series $G_{F,n}\left(  \mathbf{x},\mathbf{y}\right)  $
is homogeneous of degree $n$ as well (since this $G_{F,n}\left(
\mathbf{x},\mathbf{y}\right)  $ is obtained by substituting the variables
$x_{1},x_{2},x_{3},\ldots,y_{1},y_{2},y_{3},\ldots$ for the variables
$x_{1},x_{2},x_{3},\ldots$ in $G_{F,n}$; but this substitution clearly
preserves homogeneity and degree). This completes our proof.}. Thus, the
equality%
\begin{align*}
G_{F}\left(  \mathbf{x},\mathbf{y}\right)   &  =\sum_{m\in\mathbb{N}}%
G_{F,m}\left(  \mathbf{x},\mathbf{y}\right)  =\sum_{n\in\mathbb{N}}%
G_{F,n}\left(  \mathbf{x},\mathbf{y}\right) \\
&  \ \ \ \ \ \ \ \ \ \ \left(  \text{here, we have renamed the summation index
}m\text{ as }n\right)
\end{align*}
reveals that $\left(  G_{F,n}\left(  \mathbf{x},\mathbf{y}\right)  \right)
_{n\in\mathbb{N}}$ is the homogeneous decomposition of $G_{F}\left(
\mathbf{x},\mathbf{y}\right)  $ (by the definition of a homogeneous decomposition).

We have now shown that each of the two families
\[
\left(  G_{F,n}\left(  \mathbf{x},\mathbf{y}\right)  \right)  _{n\in
\mathbb{N}}\ \ \ \ \ \ \ \ \ \ \text{and}\ \ \ \ \ \ \ \ \ \ \left(
\sum_{\substack{\left(  i,j\right)  \in\mathbb{N}\times\mathbb{N}%
;\\i+j=n}}G_{F,i}\left(  \mathbf{x}\right)  G_{F,j}\left(  \mathbf{y}\right)
\right)  _{n\in\mathbb{N}}%
\]
is the homogeneous decomposition of $G_{F}\left(  \mathbf{x},\mathbf{y}%
\right)  $. Since the homogeneous decomposition of $G_{F}\left(
\mathbf{x},\mathbf{y}\right)  $ is unique, this entails that these two
families are equal. In other words, we have%
\[
\left(  G_{F,n}\left(  \mathbf{x},\mathbf{y}\right)  \right)  _{n\in
\mathbb{N}}=\left(  \sum_{\substack{\left(  i,j\right)  \in\mathbb{N}%
\times\mathbb{N};\\i+j=n}}G_{F,i}\left(  \mathbf{x}\right)  G_{F,j}\left(
\mathbf{y}\right)  \right)  _{n\in\mathbb{N}}.
\]
In other words, we have%
\begin{equation}
G_{F,n}\left(  \mathbf{x},\mathbf{y}\right)  =\sum_{\substack{\left(
i,j\right)  \in\mathbb{N}\times\mathbb{N};\\i+j=n}}G_{F,i}\left(
\mathbf{x}\right)  G_{F,j}\left(  \mathbf{y}\right)
\label{pf.thm.DeltaGFm.1xy.6}%
\end{equation}
for each $n\in\mathbb{N}$.

Now, let $m\in\mathbb{N}$. Then, (\ref{pf.thm.DeltaGFm.1xy.6}) (applied to
$n=m$) yields
\begin{align*}
&  G_{F,m}\left(  \mathbf{x},\mathbf{y}\right) \\
&  =\sum_{\substack{\left(  i,j\right)  \in\mathbb{N}\times\mathbb{N}%
;\\i+j=m}}G_{F,i}\left(  \mathbf{x}\right)  G_{F,j}\left(  \mathbf{y}\right)
=\sum_{i\in\left\{  0,1,\ldots,m\right\}  }G_{F,i}\left(  \mathbf{x}\right)
G_{F,m-i}\left(  \mathbf{y}\right) \\
&  \ \ \ \ \ \ \ \ \ \ \ \ \ \ \ \ \ \ \ \ \left(
\begin{array}
[c]{c}%
\text{here, we have substituted }\left(  i,m-i\right)  \text{ for }\left(
i,j\right)  \text{ in the sum,}\\
\text{since the map }\left\{  0,1,\ldots,m\right\}  \rightarrow\left\{
\left(  i,j\right)  \in\mathbb{N}\times\mathbb{N}\ \mid\ i+j=m\right\} \\
\text{that sends each }i\text{ to }\left(  i,m-i\right)  \text{ is a
bijection}%
\end{array}
\right)  .
\end{align*}
Hence, (\ref{eq.Delta-on-Lam.if}) holds for $f=G_{F,m}$, $I=\left\{
0,1,\ldots,m\right\}  $, $\left(  f_{1,i}\right)  _{i\in I}=\left(
G_{F,i}\right)  _{i\in\left\{  0,1,\ldots,m\right\}  }$ and $\left(
f_{2,i}\right)  _{i\in I}=\left(  G_{F,m-i}\right)  _{i\in\left\{
0,1,\ldots,m\right\}  }$. Therefore, (\ref{eq.Delta-on-Lam.then}) (applied to
these $f$, $I$, $\left(  f_{1,i}\right)  _{i\in I}$ and $\left(
f_{2,i}\right)  _{i\in I}$) yields%
\[
\Delta\left(  G_{F,m}\right)  =\sum_{i\in\left\{  0,1,\ldots,m\right\}
}G_{F,i}\otimes G_{F,m-i}=\sum_{i=0}^{m}G_{F,i}\otimes G_{F,m-i}.
\]
This proves Theorem \ref{thm.DeltaGFm}.
\end{proof}
\end{verlong}

The next few results we will state rely on the following definition:

\begin{definition}
\label{def.Fprime} Let $F^{\prime}$ be the derivative of the formal power
series $F\in\mathbf{k}\left[  \left[  t\right]  \right]  $. Let us write the
formal power series $\dfrac{F^{\prime}}{F}\in\mathbf{k}\left[  \left[
t\right]  \right]  $ (which is well-defined, since $F$ has constant term $1$)
in the form $\dfrac{F^{\prime}}{F}=\sum_{n\in\mathbb{N}}\gamma_{n}t^{n}$ for
some $\gamma_{0},\gamma_{1},\gamma_{2},\ldots\in\mathbf{k}$.
\end{definition}

\begin{example}
Let us see how $F^{\prime}$ and $\gamma_{n}$ look for specific values of $F$.

\textbf{(a)} Let $F=\dfrac{1}{1-t}=1+t+t^{2}+t^{3}+\cdots$. Then, $F^{\prime
}=\dfrac{1}{\left(  1-t\right)  ^{2}}$, so that
\[
\dfrac{F^{\prime}}{F}=\dfrac{1}{1-t}=1+t+t^{2}+t^{3}+\cdots=\sum
_{n\in\mathbb{N}}t^{n}.
\]
Therefore, $\gamma_{n}=1$ for each $n\in\mathbb{N}$.

\textbf{(b)} Now, let $F=1$. Then, $F^{\prime}=0$, so that $\dfrac{F^{\prime}%
}{F}=0=\sum_{n\in\mathbb{N}}0t^{n}$. Therefore, $\gamma_{n}=0$ for each
$n\in\mathbb{N}$.

\textbf{(c)} Now, fix a positive integer $k$, and set $F=1+t+t^{2}%
+\cdots+t^{k-1}$. Then, $F=\dfrac{1-t^{k}}{1-t}$, and thus a simple
calculation using the quotient rule shows that $F^{\prime}=\dfrac{1+\left(
k-1\right)  t^{k}-kt^{k-1}}{\left(  1-t\right)  ^{2}}$. Hence,
\begin{align*}
\dfrac{F^{\prime}}{F}  &  =\dfrac{1+\left(  k-1\right)  t^{k}-kt^{k-1}%
}{\left(  1-t\right)  \left(  1-t^{k}\right)  }=\underbrace{\dfrac{1}{1-t}%
}_{=\sum_{n\in\mathbb{N}}t^{n}}-kt^{k-1}\cdot\underbrace{\dfrac{1}{1-t^{k}}%
}_{=\sum_{n\in\mathbb{N}}\left(  t^{k}\right)  ^{n}}\\
&  =\sum_{n\in\mathbb{N}}t^{n}-\underbrace{kt^{k-1}\cdot\sum_{n\in\mathbb{N}%
}\left(  t^{k}\right)  ^{n}}_{\substack{=\sum_{n\in\mathbb{N}}kt^{nk+k-1}%
\\=\sum_{\substack{n\in\mathbb{N};\\k\mid n+1}}kt^{n}}}=\sum_{n\in\mathbb{N}%
}t^{n}-\sum_{\substack{n\in\mathbb{N};\\k\mid n+1}}kt^{n}\\
&  =\sum_{n\in\mathbb{N}}\left(  1-\left[  k\mid n+1\right]  k\right)  t^{n}.
\end{align*}
Therefore, $\gamma_{n}=1-\left[  k\mid n+1\right]  k$ for each $n\in
\mathbb{N}$.
\end{example}

The next proposition is easily seen to generalize Proposition
\ref{prop.hall-pmGkm}:

\begin{proposition}
\label{prop.hall-pmGFm}Let $m$ be a positive integer. Then, $\left\langle
p_{m},G_{F,m}\right\rangle =\gamma_{m-1}$.
\end{proposition}

The proof of this proposition relies on the following property of the
$\mathbf{k}$-algebra homomorphism $\alpha_{F}:\Lambda\rightarrow\mathbf{k}$
from Definition \ref{def.alphaF}:

\begin{lemma}
\label{lem.alphaF.p}We have $\alpha_{F}\left(  p_{m}\right)  =\gamma_{m-1}$
for each positive integer $m$.
\end{lemma}

\begin{vershort}
\begin{proof}
[Proof of Lemma \ref{lem.alphaF.p} (sketched).]Consider the ring
$\Lambda\left[  \left[  t\right]  \right]  $ of formal power series in one
indeterminate $t$ over $\Lambda$. Consider also the analogous ring
$\mathbf{k}\left[  \left[  t\right]  \right]  $ over $\mathbf{k}$.

The map $\alpha_{F}:\Lambda\rightarrow\mathbf{k}$ is a $\mathbf{k}$-algebra
homomorphism, and therefore induces a $\mathbf{k}\left[  \left[  t\right]
\right]  $-algebra homomorphism%
\[
\alpha_{F}\left[  \left[  t\right]  \right]  :\Lambda\left[  \left[  t\right]
\right]  \rightarrow\mathbf{k}\left[  \left[  t\right]  \right]
\]
that sends each formal power series $\sum_{n\geq0}a_{n}t^{n}\in\Lambda\left[
\left[  t\right]  \right]  $ (with $a_{n}\in\Lambda$) to $\sum_{n\geq0}%
\alpha_{F}\left(  a_{n}\right)  t^{n}$. Consider this $\mathbf{k}\left[
\left[  t\right]  \right]  $-algebra homomorphism $\alpha_{F}\left[  \left[
t\right]  \right]  $.

Define the formal power series
\begin{equation}
H\left(  t\right)  =\prod_{i=1}^{\infty}\left(  1-x_{i}t\right)  ^{-1}%
\in\left(  \mathbf{k}\left[  \left[  x_{1},x_{2},x_{3},\ldots\right]  \right]
\right)  \left[  \left[  t\right]  \right]  .
\label{pf.lem.alphaF.p.short.Ht=}%
\end{equation}
Then, from \cite[(2.4.1)]{GriRei}, we know that%
\[
H\left(  t\right)  =\sum_{n\geq0}\underbrace{h_{n}\left(  \mathbf{x}\right)
}_{=h_{n}}t^{n}=\sum_{n\geq0}h_{n}t^{n}\in\Lambda\left[  \left[  t\right]
\right]  .
\]

It is now easy to see that
\begin{equation}
\left(  \alpha_{F}\left[  \left[  t\right]  \right]  \right)  \left(  H\left(
t\right)  \right)  =F. \label{pf.lem.alphaF.p.short.5}%
\end{equation}
(Indeed, this follows by straightforward computations using the definition of
$\alpha_{F}\left[  \left[  t\right]  \right]  $ from $H\left(  t\right)
=\sum_{n\geq0}h_{n}t^{n}$ and from Lemma \ref{lem.alphaF.h} \textbf{(a)}.)

Also, it is easy to see that the map $\alpha_{F}\left[  \left[  t\right]
\right]  $ respects derivatives: i.e., any power series $u\in\Lambda\left[
\left[  t\right]  \right]  $ satisfies $\left(  \alpha_{F}\left[  \left[
t\right]  \right]  \right)  \left(  u^{\prime}\right)  =\left(  \left(
\alpha_{F}\left[  \left[  t\right]  \right]  \right)  \left(  u\right)
\right)  ^{\prime}$. Applying this to $u=H\left(  t\right)  $, we obtain%
\begin{equation}
\left(  \alpha_{F}\left[  \left[  t\right]  \right]  \right)  \left(
H^{\prime}\left(  t\right)  \right)  =\left(  \underbrace{\left(  \alpha
_{F}\left[  \left[  t\right]  \right]  \right)  \left(  H\left(  t\right)
\right)  }_{=F}\right)  ^{\prime}=F^{\prime}.
\label{pf.lem.alphaF.p.short.5pri}%
\end{equation}

From \cite[Exercise 2.5.21]{GriRei}, we know that%
\begin{equation}
\sum_{m\geq0}p_{m+1}t^{m}=\dfrac{H^{\prime}\left(  t\right)  }{H\left(
t\right)  }. \label{pf.lem.alphaF.p.short.pHH}%
\end{equation}
Applying the map $\alpha_{F}\left[  \left[  t\right]  \right]  $ to both sides
of this equality, we find%
\begin{align*}
\left(  \alpha_{F}\left[  \left[  t\right]  \right]  \right)  \left(
\sum_{m\geq0}p_{m+1}t^{m}\right)   &  =\left(  \alpha_{F}\left[  \left[
t\right]  \right]  \right)  \left(  \dfrac{H^{\prime}\left(  t\right)
}{H\left(  t\right)  }\right)  =\dfrac{\left(  \alpha_{F}\left[  \left[
t\right]  \right]  \right)  \left(  H^{\prime}\left(  t\right)  \right)
}{\left(  \alpha_{F}\left[  \left[  t\right]  \right]  \right)  \left(
H\left(  t\right)  \right)  }\\
&  \ \ \ \ \ \ \ \ \ \ \ \ \ \ \ \ \ \ \ \ \left(  \text{since }\alpha
_{F}\left[  \left[  t\right]  \right]  \text{ is a }\mathbf{k}\text{-algebra
homomorphism}\right) \\
&  =\dfrac{F^{\prime}}{F}\ \ \ \ \ \ \ \ \ \ \left(  \text{by
(\ref{pf.lem.alphaF.p.short.5}) and (\ref{pf.lem.alphaF.p.short.5pri})}\right)
\\
&  =\sum_{n\in\mathbb{N}}\gamma_{n}t^{n}=\sum_{n\geq0}\gamma_{n}t^{n}.
\end{align*}
Comparing this with%
\[
\left(  \alpha_{F}\left[  \left[  t\right]  \right]  \right)  \left(
\sum_{m\geq0}p_{m+1}t^{m}\right)  =\sum_{m\geq0}\alpha_{F}\left(
p_{m+1}\right)  t^{m}\ \ \ \ \ \ \ \ \ \ \left(  \text{by the definition of
}\alpha_{F}\left[  \left[  t\right]  \right]  \right)  ,
\]
we obtain
\[
\sum_{m\geq0}\alpha_{F}\left(  p_{m+1}\right)  t^{m}=\sum_{n\geq0}\gamma
_{n}t^{n}.
\]
Comparing $t^{n}$-coefficients on both sides of this equality, we find%
\[
\alpha_{F}\left(  p_{n+1}\right)  =\gamma_{n}\ \ \ \ \ \ \ \ \ \ \text{for
each }n\in\mathbb{N}.
\]
In other words, $\alpha_{F}\left(  p_{m}\right)  =\gamma_{m-1}$ for each
positive integer $m$. This proves Lemma \ref{lem.alphaF.p}.
\end{proof}
\end{vershort}

\begin{verlong}
\begin{proof}
[Proof of Lemma \ref{lem.alphaF.p}.]Consider the ring $\Lambda\left[  \left[
t\right]  \right]  $ of formal power series in one indeterminate $t$ over
$\Lambda$. Consider also the analogous ring $\mathbf{k}\left[  \left[
t\right]  \right]  $ over $\mathbf{k}$.

The map $\alpha_{F}:\Lambda\rightarrow\mathbf{k}$ is a $\mathbf{k}$-algebra
homomorphism. Hence, it induces a continuous\footnote{Continuity is defined
with respect to the usual topologies on $\Lambda\left[  \left[  t\right]
\right]  $ and $\mathbf{k}\left[  \left[  t\right]  \right]  $, where we equip
both $\Lambda$ and $\mathbf{k}$ with the discrete topologies.} $\mathbf{k}%
\left[  \left[  t\right]  \right]  $-algebra homomorphism%
\[
\alpha_{F}\left[  \left[  t\right]  \right]  :\Lambda\left[  \left[  t\right]
\right]  \rightarrow\mathbf{k}\left[  \left[  t\right]  \right]
\]
that sends each formal power series $\sum_{n\geq0}a_{n}t^{n}\in\Lambda\left[
\left[  t\right]  \right]  $ (with $a_{n}\in\Lambda$) to $\sum_{n\geq0}%
\alpha_{F}\left(  a_{n}\right)  t^{n}$. Consider this $\mathbf{k}\left[
\left[  t\right]  \right]  $-algebra homomorphism $\alpha_{F}\left[  \left[
t\right]  \right]  $.

Define the formal power series
\[
H\left(  t\right)  =\prod_{i=1}^{\infty}\left(  1-x_{i}t\right)  ^{-1}%
\in\left(  \mathbf{k}\left[  \left[  x_{1},x_{2},x_{3},\ldots\right]  \right]
\right)  \left[  \left[  t\right]  \right]  .
\]
Then, from \cite[(2.4.1)]{GriRei}, we know that%
\[
H\left(  t\right)  =\sum_{n\geq0}\underbrace{h_{n}\left(  \mathbf{x}\right)
}_{=h_{n}}t^{n}=\sum_{n\geq0}h_{n}t^{n}\in\Lambda\left[  \left[  t\right]
\right]  .
\]
Hence, $\left(  \alpha_{F}\left[  \left[  t\right]  \right]  \right)  \left(
H\left(  t\right)  \right)  $ is well-defined. Moreover, applying the map
$\alpha_{F}\left[  \left[  t\right]  \right]  $ to both sides of the equality
$H\left(  t\right)  =\sum_{n\geq0}h_{n}t^{n}$, we obtain%
\begin{align*}
\left(  \alpha_{F}\left[  \left[  t\right]  \right]  \right)  \left(  H\left(
t\right)  \right)   &  =\left(  \alpha_{F}\left[  \left[  t\right]  \right]
\right)  \left(  \sum_{n\geq0}h_{n}t^{n}\right)  =\underbrace{\sum_{n\geq0}%
}_{=\sum_{n\in\mathbb{N}}}\underbrace{\alpha_{F}\left(  h_{n}\right)
}_{\substack{=f_{n}\\\text{(by Lemma \ref{lem.alphaF.h} \textbf{(b)}%
}\\\text{(applied to }i=n\text{))}}}t^{n}\\
&  \ \ \ \ \ \ \ \ \ \ \ \ \ \ \ \ \ \ \ \ \left(  \text{by the definition of
}\alpha_{F}\left[  \left[  t\right]  \right]  \right) \\
&  =\sum_{n\in\mathbb{N}}f_{n}t^{n}=F\ \ \ \ \ \ \ \ \ \ \left(  \text{since
}F=\sum_{n\in\mathbb{N}}f_{n}t^{n}\right)  .
\end{align*}

Moreover, consider the derivative $H^{\prime}\left(  t\right)  $ of the power
series $H\left(  t\right)  \in\Lambda\left[  \left[  t\right]  \right]  $.
This derivative again belongs to $\Lambda\left[  \left[  t\right]  \right]  $,
so that $\left(  \alpha_{F}\left[  \left[  t\right]  \right]  \right)  \left(
H^{\prime}\left(  t\right)  \right)  $ is well-defined. Moreover, from
$H\left(  t\right)  =\sum_{n\geq0}h_{n}t^{n}$, we obtain $H^{\prime}\left(
t\right)  =\sum_{n\geq1}nh_{n}t^{n-1}$ (by the definition of a derivative), so
that%
\[
H^{\prime}\left(  t\right)  =\sum_{n\geq1}nh_{n}t^{n-1}=\sum_{n\geq0}\left(
n+1\right)  h_{n+1}t^{n}%
\]
(here, we have substituted $n+1$ for $n$ in the sum). Applying the map
$\alpha_{F}\left[  \left[  t\right]  \right]  $ to both sides of this
equality, we find%
\begin{align}
\left(  \alpha_{F}\left[  \left[  t\right]  \right]  \right)  \left(
H^{\prime}\left(  t\right)  \right)   &  =\left(  \alpha_{F}\left[  \left[
t\right]  \right]  \right)  \left(  \sum_{n\geq0}\left(  n+1\right)
h_{n+1}t^{n}\right)  =\sum_{n\geq0}\underbrace{\alpha_{F}\left(  \left(
n+1\right)  h_{n+1}\right)  }_{\substack{=\left(  n+1\right)  \alpha
_{F}\left(  h_{n+1}\right)  \\\text{(since }\alpha_{F}\text{ is }%
\mathbf{k}\text{-linear)}}}t^{n}\nonumber\\
&  \ \ \ \ \ \ \ \ \ \ \ \ \ \ \ \ \ \ \ \ \left(  \text{by the definition of
}\alpha_{F}\left[  \left[  t\right]  \right]  \right) \nonumber\\
&  =\sum_{n\geq0}\left(  n+1\right)  \underbrace{\alpha_{F}\left(
h_{n+1}\right)  }_{\substack{=f_{n+1}\\\text{(by Lemma \ref{lem.alphaF.h}
\textbf{(b)}}\\\text{(applied to }i=n+1\text{))}}}t^{n}\nonumber\\
&  =\sum_{n\geq0}\left(  n+1\right)  f_{n+1}t^{n}. \label{pf.lem.alphaF.p.5}%
\end{align}
On the other hand, from $F=\sum_{n\in\mathbb{N}}f_{n}t^{n}$, we obtain
\begin{align*}
F^{\prime}  &  =\sum_{n\geq1}nf_{n}t^{n-1}\ \ \ \ \ \ \ \ \ \ \left(  \text{by
the definition of the derivative of a power series}\right) \\
&  =\sum_{n\geq0}\left(  n+1\right)  f_{n+1}t^{n}%
\end{align*}
(here, we have substituted $n+1$ for $n$ in the sum). Comparing this with
(\ref{pf.lem.alphaF.p.5}), we obtain
\[
\left(  \alpha_{F}\left[  \left[  t\right]  \right]  \right)  \left(
H^{\prime}\left(  t\right)  \right)  =F^{\prime}.
\]

From \cite[Exercise 2.5.21]{GriRei}, we know that%
\[
\sum_{m\geq0}p_{m+1}t^{m}=\dfrac{H^{\prime}\left(  t\right)  }{H\left(
t\right)  }.
\]
Hence,%
\[
\dfrac{H^{\prime}\left(  t\right)  }{H\left(  t\right)  }=\sum_{m\geq0}%
p_{m+1}t^{m}=\sum_{n\geq0}p_{n+1}t^{n}%
\]
(here, we have renamed the summation index $m$ as $n$). Applying the map
$\alpha_{F}\left[  \left[  t\right]  \right]  $ to both sides of this
equality, we find%
\begin{align}
\left(  \alpha_{F}\left[  \left[  t\right]  \right]  \right)  \left(
\dfrac{H^{\prime}\left(  t\right)  }{H\left(  t\right)  }\right)   &  =\left(
\alpha_{F}\left[  \left[  t\right]  \right]  \right)  \left(  \sum_{n\geq
0}p_{n+1}t^{n}\right)  =\underbrace{\sum_{n\geq0}}_{=\sum_{n\in\mathbb{N}}%
}\alpha_{F}\left(  p_{n+1}\right)  t^{n}\nonumber\\
&  \ \ \ \ \ \ \ \ \ \ \ \ \ \ \ \ \ \ \ \ \left(  \text{by the definition of
}\alpha_{F}\left[  \left[  t\right]  \right]  \right) \nonumber\\
&  =\sum_{n\in\mathbb{N}}\alpha_{F}\left(  p_{n+1}\right)  t^{n}.
\label{pf.lem.alphaF.p.8}%
\end{align}

Now, the map $\alpha_{F}\left[  \left[  t\right]  \right]  $ is a $\mathbf{k}%
$-algebra homomorphism, and thus respects quotients. Hence,%
\begin{align*}
\left(  \alpha_{F}\left[  \left[  t\right]  \right]  \right)  \left(
\dfrac{H^{\prime}\left(  t\right)  }{H\left(  t\right)  }\right)   &
=\dfrac{\left(  \alpha_{F}\left[  \left[  t\right]  \right]  \right)  \left(
H^{\prime}\left(  t\right)  \right)  }{\left(  \alpha_{F}\left[  \left[
t\right]  \right]  \right)  \left(  H\left(  t\right)  \right)  }%
=\dfrac{F^{\prime}}{F}\\
&  \ \ \ \ \ \ \ \ \ \ \left(  \text{since }\left(  \alpha_{F}\left[  \left[
t\right]  \right]  \right)  \left(  H^{\prime}\left(  t\right)  \right)
=F^{\prime}\text{ and }\left(  \alpha_{F}\left[  \left[  t\right]  \right]
\right)  \left(  H\left(  t\right)  \right)  =F\right) \\
&  =\sum_{n\in\mathbb{N}}\gamma_{n}t^{n}.
\end{align*}
Comparing this with (\ref{pf.lem.alphaF.p.8}), we obtain%
\[
\sum_{n\in\mathbb{N}}\alpha_{F}\left(  p_{n+1}\right)  t^{n}=\sum
_{n\in\mathbb{N}}\gamma_{n}t^{n}.
\]
Comparing coefficients on both sides of this equality, we find%
\begin{equation}
\alpha_{F}\left(  p_{n+1}\right)  =\gamma_{n}\ \ \ \ \ \ \ \ \ \ \text{for
each }n\in\mathbb{N}. \label{pf.lem.alphaF.p.9}%
\end{equation}

Now, let $m$ be a positive integer. Thus, $m-1\in\mathbb{N}$. Hence,
(\ref{pf.lem.alphaF.p.9}) (applied to $n=m-1$) yields $\alpha_{F}\left(
p_{\left(  m-1\right)  +1}\right)  =\gamma_{m-1}$. In other words, $\alpha
_{F}\left(  p_{m}\right)  =\gamma_{m-1}$ (since $\left(  m-1\right)  +1=m$).
This proves Lemma \ref{lem.alphaF.p}.
\end{proof}
\end{verlong}

\begin{vershort}
\begin{proof}
[Proof of Proposition \ref{prop.hall-pmGFm} (sketched).]Proposition
\ref{prop.GF.basics} \textbf{(c)} yields $G_{F,m}=\sum_{\substack{\lambda
\in\operatorname*{Par};\\\left\vert \lambda\right\vert =m}}f_{\lambda
}m_{\lambda}$. Hence,%
\begin{equation}
\left\langle p_{m},G_{F,m}\right\rangle =\left\langle p_{m},\sum
_{\substack{\lambda\in\operatorname*{Par};\\\left\vert \lambda\right\vert
=m}}f_{\lambda}m_{\lambda}\right\rangle =\sum_{\substack{\lambda
\in\operatorname*{Par};\\\left\vert \lambda\right\vert =m}}f_{\lambda
}\left\langle p_{m},m_{\lambda}\right\rangle
\label{pf.prop.hall-pmGFm.short.1}%
\end{equation}
(since the Hall inner product is $\mathbf{k}$-bilinear).

Now, recall that the bases $\left(  m_{\lambda}\right)  _{\lambda
\in\operatorname*{Par}}$ and $\left(  h_{\lambda}\right)  _{\lambda
\in\operatorname*{Par}}$ of $\Lambda$ are dual to each other with respect to
the Hall inner product $\left\langle \cdot,\cdot\right\rangle $. Hence, every
$a\in\Lambda$ satisfies%
\[
a=\sum_{\lambda\in\operatorname*{Par}}\left\langle m_{\lambda},a\right\rangle
h_{\lambda}%
\]
(by a general property of dual bases with respect to symmetric bilinear
forms). Applying this to $a=p_{m}$, we obtain%
\begin{align*}
p_{m}  &  =\sum_{\lambda\in\operatorname*{Par}}\left\langle m_{\lambda}%
,p_{m}\right\rangle h_{\lambda}=\sum_{\substack{\lambda\in\operatorname*{Par}%
;\\\left\vert \lambda\right\vert =m}}\underbrace{\left\langle m_{\lambda
},p_{m}\right\rangle }_{\substack{=\left\langle p_{m},m_{\lambda}\right\rangle
\\\text{(since the Hall inner}\\\text{product is symmetric)}}}h_{\lambda}%
+\sum_{\substack{\lambda\in\operatorname*{Par};\\\left\vert \lambda\right\vert
\neq m}}\underbrace{\left\langle m_{\lambda},p_{m}\right\rangle }%
_{\substack{=0\\\text{(by (\ref{eq.Hall.graded}), since }m_{\lambda}\text{ and
}p_{m}\\\text{are homogeneous}\\\text{of degrees }\left\vert \lambda
\right\vert \text{ and }m\\\text{(and since }\left\vert \lambda\right\vert
\neq m\text{))}}}h_{\lambda}\\
&  =\sum_{\substack{\lambda\in\operatorname*{Par};\\\left\vert \lambda
\right\vert =m}}\left\langle p_{m},m_{\lambda}\right\rangle h_{\lambda}.
\end{align*}
Applying the map $\alpha_{F}$ to both sides of this equality, we find%
\begin{align*}
\alpha_{F}\left(  p_{m}\right)   &  =\sum_{\substack{\lambda\in
\operatorname*{Par};\\\left\vert \lambda\right\vert =m}}\left\langle
p_{m},m_{\lambda}\right\rangle \underbrace{\alpha_{F}\left(  h_{\lambda
}\right)  }_{\substack{=f_{\lambda}\\\text{(by Lemma \ref{lem.alphaF.h}
\textbf{(c)})}}}\ \ \ \ \ \ \ \ \ \ \left(  \text{since the map }\alpha
_{F}\text{ is }\mathbf{k}\text{-linear}\right) \\
&  =\sum_{\substack{\lambda\in\operatorname*{Par};\\\left\vert \lambda
\right\vert =m}}\left\langle p_{m},m_{\lambda}\right\rangle f_{\lambda}%
=\sum_{\substack{\lambda\in\operatorname*{Par};\\\left\vert \lambda\right\vert
=m}}f_{\lambda}\left\langle p_{m},m_{\lambda}\right\rangle .
\end{align*}
Comparing this with (\ref{pf.prop.hall-pmGFm.short.1}), we obtain%
\[
\left\langle p_{m},G_{F,m}\right\rangle =\alpha_{F}\left(  p_{m}\right)
=\gamma_{m-1}\ \ \ \ \ \ \ \ \ \ \left(  \text{by Lemma \ref{lem.alphaF.p}%
}\right)  .
\]
This proves Proposition \ref{prop.hall-pmGFm}.
\end{proof}
\end{vershort}

\begin{verlong}
\begin{proof}
[Proof of Proposition \ref{prop.hall-pmGFm}.]Proposition \ref{prop.GF.basics}
\textbf{(c)} yields $G_{F,m}=\sum_{\substack{\lambda\in\operatorname*{Par}%
;\\\left\vert \lambda\right\vert =m}}f_{\lambda}m_{\lambda}$. Hence,%
\begin{equation}
\left\langle p_{m},G_{F,m}\right\rangle =\left\langle p_{m},\sum
_{\substack{\lambda\in\operatorname*{Par};\\\left\vert \lambda\right\vert
=m}}f_{\lambda}m_{\lambda}\right\rangle =\sum_{\substack{\lambda
\in\operatorname*{Par};\\\left\vert \lambda\right\vert =m}}f_{\lambda
}\left\langle p_{m},m_{\lambda}\right\rangle \label{pf.prop.hall-pmGFm.1}%
\end{equation}
(since the Hall inner product is $\mathbf{k}$-bilinear).

Recall the following fundamental fact from linear algebra: If $A$ is a
$\mathbf{k}$-module, if $\left\langle \cdot,\cdot\right\rangle :A\times
A\rightarrow\mathbf{k}$ is a symmetric bilinear form on $A$, and if $\left(
u_{\lambda}\right)  _{\lambda\in L}$ and $\left(  v_{\lambda}\right)
_{\lambda\in L}$ are two bases of the $\mathbf{k}$-module $A$ that are dual to
each other with respect to the form $\left\langle \cdot,\cdot\right\rangle $
(where $L$ is some indexing set), then every $a\in A$ satisfies
\[
a=\sum_{\lambda\in L}\left\langle u_{\lambda},a\right\rangle v_{\lambda}.
\]
We can apply this fact to $A=\Lambda$, $L=\operatorname*{Par}$, $\left(
u_{\lambda}\right)  _{\lambda\in L}=\left(  m_{\lambda}\right)  _{\lambda
\in\operatorname*{Par}}$ and $\left(  v_{\lambda}\right)  _{\lambda\in
L}=\left(  h_{\lambda}\right)  _{\lambda\in\operatorname*{Par}}$ (since the
bases $\left(  m_{\lambda}\right)  _{\lambda\in\operatorname*{Par}}$ and
$\left(  h_{\lambda}\right)  _{\lambda\in\operatorname*{Par}}$ of $\Lambda$
are dual to each other with respect to the Hall inner product $\left\langle
\cdot,\cdot\right\rangle $). We thus conclude that every $a\in\Lambda$
satisfies%
\[
a=\sum_{\lambda\in\operatorname*{Par}}\left\langle m_{\lambda},a\right\rangle
h_{\lambda}.
\]
Applying this to $a=p_{m}$, we obtain%
\begin{equation}
p_{m}=\sum_{\lambda\in\operatorname*{Par}}\underbrace{\left\langle m_{\lambda
},p_{m}\right\rangle }_{\substack{=\left\langle p_{m},m_{\lambda}\right\rangle
\\\text{(since the Hall inner}\\\text{product is symmetric)}}}h_{\lambda}%
=\sum_{\lambda\in\operatorname*{Par}}\left\langle p_{m},m_{\lambda
}\right\rangle h_{\lambda}. \label{pf.prop.hall-pmGFm.3}%
\end{equation}

Now, it is easy to see that if $\lambda\in\operatorname*{Par}$ satisfies
$\left\vert \lambda\right\vert \neq m$, then%
\begin{equation}
\left\langle p_{m},m_{\lambda}\right\rangle =0. \label{pf.prop.hall-pmGFm.4}%
\end{equation}

[\textit{Proof of (\ref{pf.prop.hall-pmGFm.4}):} Let $\lambda\in
\operatorname*{Par}$ satisfy $\left\vert \lambda\right\vert \neq m$. Thus,
$m\neq\left\vert \lambda\right\vert $. The two symmetric functions $p_{m}$ and
$m_{\lambda}$ are homogeneous of degrees $m$ and $\left\vert \lambda
\right\vert $, respectively. Thus, these two symmetric functions $p_{m}$ and
$m_{\lambda}$ are homogeneous of different degrees (since $m\neq\left\vert
\lambda\right\vert $). Hence, (\ref{eq.Hall.graded}) (applied to $f=p_{m}$ and
$g=m_{\lambda}$) yields $\left\langle p_{m},m_{\lambda}\right\rangle =0$. This
proves (\ref{pf.prop.hall-pmGFm.4}).]

Now, (\ref{pf.prop.hall-pmGFm.3}) becomes%
\begin{align*}
p_{m}  &  =\sum_{\lambda\in\operatorname*{Par}}\left\langle p_{m},m_{\lambda
}\right\rangle h_{\lambda}=\sum_{\substack{\lambda\in\operatorname*{Par}%
;\\\left\vert \lambda\right\vert =m}}\left\langle p_{m},m_{\lambda
}\right\rangle h_{\lambda}+\sum_{\substack{\lambda\in\operatorname*{Par}%
;\\\left\vert \lambda\right\vert \neq m}}\underbrace{\left\langle
p_{m},m_{\lambda}\right\rangle }_{\substack{=0\\\text{(by
(\ref{pf.prop.hall-pmGFm.4}))}}}h_{\lambda}\\
&  \ \ \ \ \ \ \ \ \ \ \ \ \ \ \ \ \ \ \ \ \left(
\begin{array}
[c]{c}%
\text{since each }\lambda\in\operatorname*{Par}\text{ satisfies either
}\left\vert \lambda\right\vert =m\\
\text{or }\left\vert \lambda\right\vert \neq m\text{, but not both at the same
time}%
\end{array}
\right) \\
&  =\sum_{\substack{\lambda\in\operatorname*{Par};\\\left\vert \lambda
\right\vert =m}}\left\langle p_{m},m_{\lambda}\right\rangle h_{\lambda
}+\underbrace{\sum_{\substack{\lambda\in\operatorname*{Par};\\\left\vert
\lambda\right\vert \neq m}}0h_{\lambda}}_{=0}=\sum_{\substack{\lambda
\in\operatorname*{Par};\\\left\vert \lambda\right\vert =m}}\left\langle
p_{m},m_{\lambda}\right\rangle h_{\lambda}.
\end{align*}
Applying the map $\alpha_{F}$ to both sides of this equality, we find%
\begin{align*}
\alpha_{F}\left(  p_{m}\right)   &  =\alpha_{F}\left(  \sum_{\substack{\lambda
\in\operatorname*{Par};\\\left\vert \lambda\right\vert =m}}\left\langle
p_{m},m_{\lambda}\right\rangle h_{\lambda}\right)  =\sum_{\substack{\lambda
\in\operatorname*{Par};\\\left\vert \lambda\right\vert =m}}\left\langle
p_{m},m_{\lambda}\right\rangle \underbrace{\alpha_{F}\left(  h_{\lambda
}\right)  }_{\substack{=f_{\lambda}\\\text{(by Lemma \ref{lem.alphaF.h}
\textbf{(c)})}}}\\
&  \ \ \ \ \ \ \ \ \ \ \ \ \ \ \ \ \ \ \ \ \left(  \text{since the map }%
\alpha_{F}\text{ is }\mathbf{k}\text{-linear}\right) \\
&  =\sum_{\substack{\lambda\in\operatorname*{Par};\\\left\vert \lambda
\right\vert =m}}\underbrace{\left\langle p_{m},m_{\lambda}\right\rangle
f_{\lambda}}_{=f_{\lambda}\left\langle p_{m},m_{\lambda}\right\rangle }%
=\sum_{\substack{\lambda\in\operatorname*{Par};\\\left\vert \lambda\right\vert
=m}}f_{\lambda}\left\langle p_{m},m_{\lambda}\right\rangle =\left\langle
p_{m},G_{F,m}\right\rangle
\end{align*}
(by (\ref{pf.prop.hall-pmGFm.1})). Hence,%
\[
\left\langle p_{m},G_{F,m}\right\rangle =\alpha_{F}\left(  p_{m}\right)
=\gamma_{m-1}\ \ \ \ \ \ \ \ \ \ \left(  \text{by Lemma \ref{lem.alphaF.p}%
}\right)  .
\]
This proves Proposition \ref{prop.hall-pmGFm}.
\end{proof}
\end{verlong}

We can now generalize Theorem \ref{thm.Gkm-genset}:

\begin{theorem}
\label{thm.GFm-genset}Assume that all the elements $\gamma_{0},\gamma
_{1},\gamma_{2},\ldots$ are invertible in $\mathbf{k}$.

Then, the family $\left(  G_{F,m}\right)  _{m\geq1}=\left(  G_{F,1}%
,G_{F,2},G_{F,3},\ldots\right)  $ is an algebraically independent generating
set of the commutative $\mathbf{k}$-algebra $\Lambda$. (In other words, the
canonical $\mathbf{k}$-algebra homomorphism
\begin{align*}
\mathbf{k}\left[  u_{1},u_{2},u_{3},\ldots\right]   &  \rightarrow\Lambda,\\
u_{m}  &  \mapsto G_{F,m}%
\end{align*}
is an isomorphism.)
\end{theorem}

\begin{vershort}
\begin{proof}
[Proof of Theorem \ref{thm.GFm-genset} (sketched).]Analogous to the proof of
Theorem \ref{thm.Gkm-genset}, but using Proposition \ref{prop.hall-pmGFm} (and
Proposition \ref{prop.GF.basics}) instead of Proposition \ref{prop.hall-pmGkm}
(and Proposition \ref{prop.G.basics}).
\end{proof}
\end{vershort}

\begin{verlong}
\begin{proof}
[Proof of Theorem \ref{thm.GFm-genset}.]For each positive integer $m$, the
power series $G_{F,m}$ belongs to $\Lambda$ (by Proposition
\ref{prop.GF.basics} \textbf{(c)}), and thus is a symmetric function.
Moreover, this symmetric function $G_{F,m}$ is homogeneous of degree $m$ (by
Proposition \ref{prop.GF.basics} \textbf{(a)}). Hence, for each positive
integer $m$, the element $G_{F,m}\in\Lambda$ is a homogeneous symmetric
function of degree $m$.

Let $m$ be a positive integer. From Proposition \ref{prop.hall-pmGFm}, we
obtain $\left\langle p_{m},G_{F,m}\right\rangle =\gamma_{m-1}$. Hence,
$\left\langle p_{m},G_{F,m}\right\rangle $ is an invertible element of
$\mathbf{k}$ (because $\gamma_{m-1}$ is an invertible element of $\mathbf{k}$
(since all the elements $\gamma_{0},\gamma_{1},\gamma_{2},\ldots$ are
invertible in $\mathbf{k}$)).

Forget that we fixed $m$. We thus have showed that $\left\langle p_{m}%
,G_{F,m}\right\rangle $ is an invertible element of $\mathbf{k}$ for each
positive integer $m$. Also, as we know, for each positive integer $m$, the
element $G_{F,m}\in\Lambda$ is a homogeneous symmetric function of degree $m$.
Thus, Proposition \ref{prop.genset-crit} (applied to $v_{m}=G_{F,m}$) shows
that the family $\left(  G_{F,m}\right)  _{m\geq1}=\left(  G_{F,1}%
,G_{F,2},G_{F,3},\ldots\right)  $ is an algebraically independent generating
set of the commutative $\mathbf{k}$-algebra $\Lambda$. This proves Theorem
\ref{thm.GFm-genset}.
\end{proof}
\end{verlong}

\begin{remark}
It is not hard to verify that the converse of Theorem \ref{thm.GFm-genset}
also holds: If the family $\left(  G_{F,m}\right)  _{m\geq1}=\left(
G_{F,1},G_{F,2},G_{F,3},\ldots\right)  $ generates the $\mathbf{k}$-algebra
$\Lambda$, then all the elements $\gamma_{0},\gamma_{1},\gamma_{2},\ldots$ are
invertible in $\mathbf{k}$. We omit the proof of this.
\end{remark}

The next theorem generalizes parts of Theorem \ref{thm.Uk.main} (specifically,
it generalizes the properties of the map $V_{k}$ stated in Theorem
\ref{thm.Uk.main}, even though it defines this map differently):\footnote{We
recall the \textquotedblleft h-universal property of $\Lambda$%
\textquotedblright, which we stated in Subsection
\ref{subsect.proofs.petk.alphak}.}

\begin{theorem}
\label{thm.VF.main}The h-universal property of $\Lambda$ shows that there is a
unique $\mathbf{k}$-algebra homomorphism $V_{F}:\Lambda\rightarrow\Lambda$
that sends $h_{i}$ to $G_{F,i}$ for all positive integers $i$ (since
$G_{F,i}\in\Lambda$ for each positive integer $i$). Consider this $V_{F}$.

\textbf{(a)} This map $V_{F}$ is a $\mathbf{k}$-Hopf algebra homomorphism.

\textbf{(b)} We have $V_{F}\left(  h_{m}\right)  =G_{F,m}$ for each
$m\in\mathbb{N}$.

\textbf{(c)} We have $V_{F}\left(  p_{n}\right)  =\gamma_{n-1}p_{n}$ for each
positive integer $n$. (See Definition \ref{def.Fprime} for the meaning of
$\gamma_{n-1}$.)
\end{theorem}

\begin{vershort}
\begin{proof}
[Proof of Theorem \ref{thm.VF.main} (sketched).]\textbf{(b)} When $m$ is
positive, this follows from the very definition of $V_{F}$. It remains to
prove this for $m=0$. However, this boils down to showing that $V_{F}\left(
1\right)  =1$, which is clear (since $V_{F}$ is a $\mathbf{k}$-algebra homomorphism).

\textbf{(a)} Let $\Delta$ and $\varepsilon$ be the comultiplication and the
counit of the Hopf algebra $\Lambda$. Both $\Delta$ and $\varepsilon$ are
$\mathbf{k}$-algebra homomorphisms. It suffices to show that $\Delta\circ
V_{F}=\left(  V_{F}\otimes V_{F}\right)  \circ\Delta$ and $\varepsilon\circ
V_{F}=\varepsilon$. We shall show that $\Delta\circ V_{F}=\left(  V_{F}\otimes
V_{F}\right)  \circ\Delta$ only; the proof of $\varepsilon\circ V_{F}%
=\varepsilon$ is similar but much simpler (since $\varepsilon$ sends any
homogeneous symmetric function of positive degree to $0$).

Recall that the family $\left(  h_{n}\right)  _{n\geq1}$ generates $\Lambda$
as a $\mathbf{k}$-algebra. Thus, in order to prove that $\Delta\circ
V_{F}=\left(  V_{F}\otimes V_{F}\right)  \circ\Delta$, it suffices to prove
the equality $\left(  \Delta\circ V_{F}\right)  \left(  h_{n}\right)  =\left(
\left(  V_{F}\otimes V_{F}\right)  \circ\Delta\right)  \left(  h_{n}\right)  $
for each $n\geq1$ (since both $\Delta\circ V_{F}$ and $\left(  V_{F}\otimes
V_{F}\right)  \circ\Delta$ are $\mathbf{k}$-algebra homomorphisms). In view of
Theorem \ref{thm.VF.main} \textbf{(b)}, this equality rewrites as
$\Delta\left(  G_{F,n}\right)  =\sum_{i=0}^{n}G_{F,i}\otimes G_{F,n-i}$. But
this follows directly from Theorem \ref{thm.DeltaGFm}.

\textbf{(c)} This is best proved using the notion of a logarithmic derivative.
Let us first define it in full generality, without any assumptions on
$\mathbf{k}$.

If $R$ is a commutative ring, and if $F\in R\left[  \left[  t\right]  \right]
$ is any formal power series whose constant term is $1$ (or, more generally,
any formal power series that has a multiplicative inverse), then the
\emph{logarithmic derivative} of $F$ is defined to be the formal power series
$\dfrac{F^{\prime}}{F}\in R\left[  \left[  t\right]  \right]  $ (this is
well-defined, since $F$ is invertible). This logarithmic derivative is denoted
by $\operatorname*{lder}F$.

The following properties of logarithmic derivatives are easy to
prove\footnote{If $R$ is a commutative $\mathbb{Q}$-algebra, then the
logarithmic derivative $\operatorname*{lder}F$ of a power series $F\in
R\left[  \left[  t\right]  \right]  $ equals the derivative of $\log F$. This
trivializes many of the properties stated below; but this shortcut is not
available when $R$ is merely an arbitrary commutative ring.}:

\begin{enumerate}
\item Let $R$ be a commutative ring. Let $u,v\in R\left[  \left[  t\right]
\right]  $ be two formal power series whose constant terms are $1$. Then,
$\operatorname*{lder}\left(  uv\right)  =\operatorname*{lder}%
u+\operatorname*{lder}v$.

(\textit{Proof:} Just recall the definition of logarithmic derivatives and the
Leibniz law $\left(  uv\right)  ^{\prime}=u^{\prime}v+uv^{\prime}$.)

\item Let $R$ be a commutative topological ring. Let $\left(  u_{n}\right)
_{n\in\mathbb{N}}=\left(  u_{0},u_{1},u_{2},\ldots\right)  \in R\left[
\left[  t\right]  \right]  ^{\mathbb{N}}$ be a sequence of formal power series
whose constant terms are $1$. Let $u\in R\left[  \left[  t\right]  \right]  $
be a formal power series whose constant term is $1$. Assume that
$\lim\limits_{n\rightarrow\infty}u_{n}=u$ (with respect to the standard
topology on $R\left[  \left[  t\right]  \right]  $ induced by the topology on
$R$). Then, $\lim\limits_{n\rightarrow\infty}\left(  \operatorname*{lder}%
u_{n}\right)  =\operatorname*{lder}u$ (with respect to the same topology on
$R\left[  \left[  t\right]  \right]  $).

(\textit{Proof:} Let $R\left[  \left[  t\right]  \right]  _{1}$ be the set of
power series in $R\left[  \left[  t\right]  \right]  $ whose constant term is
$1$. Argue that $\lim\limits_{n\rightarrow\infty}\left(  u_{n}^{\prime
}\right)  =u^{\prime}$ first; then argue that the map%
\begin{align*}
R\left[  \left[  t\right]  \right]  \times R\left[  \left[  t\right]  \right]
_{1}  &  \rightarrow R\left[  \left[  t\right]  \right]  ,\\
\left(  v,w\right)   &  \mapsto\dfrac{v}{w}%
\end{align*}
is continuous.)

\item Let $R$ be a commutative ring. Let $u_{1},u_{2},\ldots,u_{n}\in R\left[
\left[  t\right]  \right]  $ be finitely many formal power series whose
constant terms are $1$. Then,%
\[
\operatorname*{lder}\left(  \prod_{i=1}^{n}u_{i}\right)  =\sum_{i=1}%
^{n}\operatorname*{lder}u_{i}.
\]

(\textit{Proof:} Induction on $n$, using Property 1 in the induction step.)

\item Let $R$ be a commutative topological ring. Let $u_{1},u_{2},u_{3}%
,\ldots\in R\left[  \left[  t\right]  \right]  $ be infinitely many formal
power series whose constant terms are $1$. Assume that the infinite product
$\prod_{i=1}^{\infty}u_{i}$ converges (with respect to the standard topology
on $R\left[  \left[  t\right]  \right]  $ induced by the topology on $R$).
Then, the infinite sum $\sum_{i=1}^{\infty}\operatorname*{lder}u_{i}$
converges as well, and we have%
\[
\operatorname*{lder}\left(  \prod_{i=1}^{\infty}u_{i}\right)  =\sum
_{i=1}^{\infty}\operatorname*{lder}u_{i}.
\]

(\textit{Proof:} This is the \textquotedblleft$n\rightarrow\infty
$\textquotedblright\ limit of Property 3. Use Property 2 to pass to this limit.)

\item Let $R$ be a commutative ring. Let $u\in R\left[  \left[  t\right]
\right]  $ be a formal power series whose constant term is $1$. Let
$\lambda\in R$. Then,%
\[
\operatorname*{lder}\left(  u\left(  \lambda t\right)  \right)  =\lambda
\cdot\left(  \operatorname*{lder}u\right)  \left(  \lambda t\right)  .
\]

(\textit{Proof:} This follows from the equality $\left(  u\left(  \lambda
t\right)  \right)  ^{\prime}=\lambda\cdot u^{\prime}\left(  \lambda t\right)
$, which is an easy consequence of the chain rule but also easy to check directly.)

\item Let $R$ and $S$ be two commutative $\mathbf{k}$-algebras. Let
$\alpha:R\rightarrow S$ be a $\mathbf{k}$-algebra homomorphism. As we know,
$\alpha$ induces a $\mathbf{k}\left[  \left[  t\right]  \right]  $-algebra
homomorphism%
\[
\alpha\left[  \left[  t\right]  \right]  :R\left[  \left[  t\right]  \right]
\rightarrow S\left[  \left[  t\right]  \right]
\]
that sends each power series $\sum_{n\geq0}a_{n}t^{n}\in R\left[  \left[
t\right]  \right]  $ (with $a_{n}\in R$) to $\sum_{n\geq0}\alpha\left(
a_{n}\right)  t^{n}\in S\left[  \left[  t\right]  \right]  $.

Let $u\in R\left[  \left[  t\right]  \right]  $ be a formal power series whose
constant term is $1$. Then, the constant term of the power series $\left(
\alpha\left[  \left[  t\right]  \right]  \right)  \left(  u\right)  $ is $1$,
and we have%
\[
\operatorname*{lder}\left(  \left(  \alpha\left[  \left[  t\right]  \right]
\right)  \left(  u\right)  \right)  =\left(  \alpha\left[  \left[  t\right]
\right]  \right)  \left(  \operatorname*{lder}u\right)  .
\]

(\textit{Proof:} This is essentially saying that the logarithmic derivative is
functorial with respect to the base ring. The proof is straightforward.)
\end{enumerate}

Now, let us come back to proving Theorem \ref{thm.VF.main} \textbf{(c)}:

Consider the ring $\left(  \mathbf{k}\left[  \left[  x_{1},x_{2},x_{3}%
,\ldots\right]  \right]  \right)  \left[  \left[  t\right]  \right]  $ of
formal power series in one indeterminate $t$ over $\mathbf{k}\left[  \left[
x_{1},x_{2},x_{3},\ldots\right]  \right]  $. This ring is a topological ring,
where the topology is the standard one induced by the standard topology on
$\mathbf{k}\left[  \left[  x_{1},x_{2},x_{3},\ldots\right]  \right]  $ (not
the discrete topology!). This topological ring $\left(  \mathbf{k}\left[
\left[  x_{1},x_{2},x_{3},\ldots\right]  \right]  \right)  \left[  \left[
t\right]  \right]  $ is, of course, isomorphic to $\mathbf{k}\left[  \left[
x_{1},x_{2},x_{3},\ldots,t\right]  \right]  $. The ring $\Lambda\left[
\left[  t\right]  \right]  $ is a subring of $\left(  \mathbf{k}\left[
\left[  x_{1},x_{2},x_{3},\ldots\right]  \right]  \right)  \left[  \left[
t\right]  \right]  $.

Now, for each $m\in\mathbb{N}$, we know that $G_{F,m}$ is homogeneous of
degree $m$ (by Proposition \ref{prop.GF.basics} \textbf{(a)}), and therefore
satisfies%
\begin{equation}
G_{F,m}\left(  tx_{1},tx_{2},tx_{3},\ldots\right)  =t^{m}\cdot G_{F,m}
\label{pf.thm.VF.main.short.c.GFmtxi}%
\end{equation}
(since any formal power series $u\in\mathbf{k}\left[  \left[  x_{1}%
,x_{2},x_{3},\ldots\right]  \right]  $ that is homogeneous of degree $m$
satisfies $u\left(  tx_{1},tx_{2},tx_{3},\ldots\right)  =t^{m}\cdot u$).

On the other hand, from (\ref{pf.thm.DeltaGFm.short.1}), we obtain%
\[
\prod_{i=1}^{\infty}F\left(  x_{i}\right)  =\sum_{k\in\mathbb{N}}G_{F,k}%
=\sum_{m\in\mathbb{N}}G_{F,m}.
\]
Substituting $tx_{1},tx_{2},tx_{3},\ldots$ for $x_{1},x_{2},x_{3},\ldots$ on
both sides of this equality, we obtain%
\begin{align}
\prod_{i=1}^{\infty}F\left(  tx_{i}\right)   &  =\sum_{m\in\mathbb{N}%
}\underbrace{G_{F,m}\left(  tx_{1},tx_{2},tx_{3},\ldots\right)  }%
_{\substack{=t^{m}\cdot G_{F,m}\\\text{(by
(\ref{pf.thm.VF.main.short.c.GFmtxi}))}}}\nonumber\\
&  =\sum_{m\in\mathbb{N}}t^{m}\cdot G_{F,m}.
\label{pf.thm.VF.main.short.c.GFtxi}%
\end{align}

The map $V_{F}:\Lambda\rightarrow\Lambda$ is a $\mathbf{k}$-algebra
homomorphism. Hence, it induces a $\mathbf{k}\left[  \left[  t\right]
\right]  $-algebra homomorphism%
\[
V_{F}\left[  \left[  t\right]  \right]  :\Lambda\left[  \left[  t\right]
\right]  \rightarrow\Lambda\left[  \left[  t\right]  \right]
\]
that sends each formal power series $\sum_{n\geq0}a_{n}t^{n}\in\Lambda\left[
\left[  t\right]  \right]  $ (with $a_{n}\in\Lambda$) to $\sum_{n\geq0}%
V_{F}\left(  a_{n}\right)  t^{n}$. Consider this $V_{F}\left[  \left[
t\right]  \right]  $.

Define the formal power series $H\left(  t\right)  $ as in
(\ref{pf.lem.alphaF.p.short.Ht=}). Then, from \cite[(2.4.1)]{GriRei}, we know
that%
\[
H\left(  t\right)  =\sum_{n\geq0}\underbrace{h_{n}\left(  \mathbf{x}\right)
}_{=h_{n}}t^{n}=\sum_{n\geq0}h_{n}t^{n}\in\Lambda\left[  \left[  t\right]
\right]  .
\]
Moreover, $H\left(  t\right)  =\sum_{n\geq0}h_{n}t^{n}$ shows that the
constant term of $H\left(  t\right)  $ is $h_{0}=1$. Thus,
$\operatorname*{lder}\left(  H\left(  t\right)  \right)  $ is well-defined.

Applying the map $V_{F}\left[  \left[  t\right]  \right]  $ to both sides of
the equality $H\left(  t\right)  =\sum_{n\geq0}h_{n}t^{n}$, we obtain%
\begin{align*}
\left(  V_{F}\left[  \left[  t\right]  \right]  \right)  \left(  H\left(
t\right)  \right)   &  =\left(  V_{F}\left[  \left[  t\right]  \right]
\right)  \left(  \sum_{n\geq0}h_{n}t^{n}\right)  =\underbrace{\sum_{n\geq0}%
}_{=\sum_{n\in\mathbb{N}}}\underbrace{V_{F}\left(  h_{n}\right)
}_{\substack{=G_{F,n}\\\text{(by Theorem \ref{thm.VF.main} \textbf{(b)})}%
}}t^{n}\\
&  \ \ \ \ \ \ \ \ \ \ \ \ \ \ \ \ \ \ \ \ \left(  \text{by the definition of
}V_{F}\left[  \left[  t\right]  \right]  \right) \\
&  =\sum_{n\in\mathbb{N}}G_{F,n}t^{n}=\sum_{n\in\mathbb{N}}t^{n}\cdot
G_{F,n}=\sum_{m\in\mathbb{N}}t^{m}\cdot G_{F,m}.
\end{align*}
Comparing this with (\ref{pf.thm.VF.main.short.c.GFtxi}), we find%
\begin{equation}
\left(  V_{F}\left[  \left[  t\right]  \right]  \right)  \left(  H\left(
t\right)  \right)  =\prod_{i=1}^{\infty}F\left(  tx_{i}\right)  .
\label{pf.thm.VF.main.short.c.GFtxi2}%
\end{equation}

Now, the definition of $\operatorname*{lder}\left(  H\left(  t\right)
\right)  $ yields%
\begin{align*}
\operatorname*{lder}\left(  H\left(  t\right)  \right)   &  =\dfrac{H^{\prime
}\left(  t\right)  }{H\left(  t\right)  }=\sum_{m\geq0}p_{m+1}t^{m}%
\ \ \ \ \ \ \ \ \ \ \left(  \text{by (\ref{pf.lem.alphaF.p.short.pHH})}\right)
\\
&  =\sum_{n\geq0}p_{n+1}t^{n}.
\end{align*}
Applying the map $V_{F}\left[  \left[  t\right]  \right]  $ to both sides of
this equality, we find%
\begin{align}
\left(  V_{F}\left[  \left[  t\right]  \right]  \right)  \left(
\operatorname*{lder}\left(  H\left(  t\right)  \right)  \right)   &  =\left(
V_{F}\left[  \left[  t\right]  \right]  \right)  \left(  \sum_{n\geq0}%
p_{n+1}t^{n}\right)  =\underbrace{\sum_{n\geq0}}_{=\sum_{n\in\mathbb{N}}}%
V_{F}\left(  p_{n+1}\right)  t^{n}\nonumber\\
&  \ \ \ \ \ \ \ \ \ \ \ \ \ \ \ \ \ \ \ \ \left(  \text{by the definition of
}V_{F}\left[  \left[  t\right]  \right]  \right) \nonumber\\
&  =\sum_{n\in\mathbb{N}}V_{F}\left(  p_{n+1}\right)  t^{n}.
\label{pf.thm.VF.main.short.c.3}%
\end{align}

Now is the time to use our above list of properties of logarithmic
derivatives. Recall that the constant term of $H\left(  t\right)  $ is $1$.
Hence, Property 6 of logarithmic derivatives shows that the constant term of
the power series $\left(  V_{F}\left[  \left[  t\right]  \right]  \right)
\left(  H\left(  t\right)  \right)  $ is $1$, and that we have%
\begin{equation}
\operatorname*{lder}\left(  \left(  V_{F}\left[  \left[  t\right]  \right]
\right)  \left(  H\left(  t\right)  \right)  \right)  =\left(  V_{F}\left[
\left[  t\right]  \right]  \right)  \left(  \operatorname*{lder}\left(
H\left(  t\right)  \right)  \right)  . \label{pf.thm.VF.main.short.c.4}%
\end{equation}

Now, (\ref{pf.thm.VF.main.short.c.3}) yields%
\begin{align}
\sum_{n\in\mathbb{N}}V_{F}\left(  p_{n+1}\right)  t^{n}  &  =\left(
V_{F}\left[  \left[  t\right]  \right]  \right)  \left(  \operatorname*{lder}%
\left(  H\left(  t\right)  \right)  \right) \nonumber\\
&  =\operatorname*{lder}\left(  \left(  V_{F}\left[  \left[  t\right]
\right]  \right)  \left(  H\left(  t\right)  \right)  \right)
\ \ \ \ \ \ \ \ \ \ \left(  \text{by (\ref{pf.thm.VF.main.short.c.4})}\right)
\nonumber\\
&  =\operatorname*{lder}\left(  \prod_{i=1}^{\infty}F\left(  tx_{i}\right)
\right)  \label{pf.thm.VF.main.short.c.5}%
\end{align}
(by (\ref{pf.thm.VF.main.short.c.GFtxi2})).

Now, the infinite product $\prod_{i=1}^{\infty}F\left(  tx_{i}\right)  $
converges (as we know from (\ref{pf.thm.VF.main.short.c.GFtxi})). Hence,
Property 4 of logarithmic derivatives yields that the infinite sum $\sum
_{i=1}^{\infty}\operatorname*{lder}\left(  F\left(  tx_{i}\right)  \right)  $
converges as well, and that we have%
\[
\operatorname*{lder}\left(  \prod_{i=1}^{\infty}F\left(  tx_{i}\right)
\right)  =\sum_{i=1}^{\infty}\operatorname*{lder}\left(  F\left(
tx_{i}\right)  \right)  .
\]
Hence, (\ref{pf.thm.VF.main.short.c.5}) rewrites as%
\begin{align}
\sum_{n\in\mathbb{N}}V_{F}\left(  p_{n+1}\right)  t^{n}  &  =\sum
_{i=1}^{\infty}\operatorname*{lder}\left(  F\left(  \underbrace{tx_{i}%
}_{=x_{i}t}\right)  \right)  =\sum_{i=1}^{\infty}%
\underbrace{\operatorname*{lder}\left(  F\left(  x_{i}t\right)  \right)
}_{\substack{=x_{i}\cdot\left(  \operatorname*{lder}F\right)  \left(
x_{i}t\right)  \\\text{(by Property 5}\\\text{of logarithmic derivatives)}%
}}\nonumber\\
&  =\sum_{i=1}^{\infty}x_{i}\cdot\left(  \operatorname*{lder}F\right)  \left(
x_{i}t\right)  . \label{pf.thm.VF.main.short.c.7}%
\end{align}

The definition of $\operatorname*{lder}F$ yields%
\[
\operatorname*{lder}F=\dfrac{F^{\prime}}{F}=\sum_{n\in\mathbb{N}}\gamma
_{n}t^{n}.
\]
Hence, for each $i\in\left\{  1,2,3,\ldots\right\}  $, we have%
\begin{equation}
\left(  \operatorname*{lder}F\right)  \left(  x_{i}t\right)  =\sum
_{n\in\mathbb{N}}\gamma_{n}\underbrace{\left(  x_{i}t\right)  ^{n}}%
_{=x_{i}^{n}t^{n}}=\sum_{n\in\mathbb{N}}\gamma_{n}x_{i}^{n}t^{n}.
\label{pf.thm.VF.main.short.c.8}%
\end{equation}

Now, (\ref{pf.thm.VF.main.short.c.7}) becomes%
\begin{align*}
\sum_{n\in\mathbb{N}}V_{F}\left(  p_{n+1}\right)  t^{n}  &  =\sum
_{i=1}^{\infty}x_{i}\cdot\underbrace{\left(  \operatorname*{lder}F\right)
\left(  x_{i}t\right)  }_{\substack{=\sum_{n\in\mathbb{N}}\gamma_{n}x_{i}%
^{n}t^{n}\\\text{(by (\ref{pf.thm.VF.main.short.c.8}))}}}=\sum_{i=1}^{\infty
}x_{i}\cdot\sum_{n\in\mathbb{N}}\gamma_{n}x_{i}^{n}t^{n}=\underbrace{\sum
_{i=1}^{\infty}\ \ \sum_{n\in\mathbb{N}}}_{=\sum_{n\in\mathbb{N}}%
\ \ \sum_{i=1}^{\infty}}\underbrace{x_{i}\gamma_{n}x_{i}^{n}}_{=\gamma
_{n}x_{i}^{n+1}}t^{n}\\
&  =\sum_{n\in\mathbb{N}}\ \ \sum_{i=1}^{\infty}\gamma_{n}x_{i}^{n+1}%
t^{n}=\sum_{n\in\mathbb{N}}\gamma_{n}\underbrace{\left(  \sum_{i=1}^{\infty
}x_{i}^{n+1}\right)  }_{\substack{=x_{1}^{n+1}+x_{2}^{n+1}+x_{3}^{n+1}%
+\cdots\\=p_{n+1}\\\text{(by the definition of }p_{n+1}\text{)}}}t^{n}%
=\sum_{n\in\mathbb{N}}\gamma_{n}p_{n+1}t^{n}.
\end{align*}
Comparing coefficients before $t^{n}$ in this equality, we conclude that%
\[
V_{F}\left(  p_{n+1}\right)  =\gamma_{n}p_{n+1}\ \ \ \ \ \ \ \ \ \ \text{for
each }n\in\mathbb{N}.
\]
In other words, $V_{F}\left(  p_{n}\right)  =\gamma_{n-1}p_{n}$ for each
positive integer $n$. This proves Theorem \ref{thm.VF.main} \textbf{(c)}.
\end{proof}
\end{vershort}

\begin{verlong}
Our proof of this theorem (specifically, of its part \textbf{(c)}) will use
the notion of the logarithmic derivative of a formal power series. We first
recall its definition:

\begin{definition}
\label{def.lder.lder}Let $R$ be a commutative ring. Let $F\in R\left[  \left[
t\right]  \right]  $ be a formal power series whose constant term is $1$.
Thus, $F$ is invertible (since $F$ has constant term $1$).

The \emph{logarithmic derivative} of $F$ is defined to be the formal power
series $\dfrac{F^{\prime}}{F}\in R\left[  \left[  t\right]  \right]  $ (this
is well-defined, since $F$ is invertible). This logarithmic derivative is
denoted by $\operatorname*{lder}F$.
\end{definition}

It is easy to see that $\operatorname*{lder}F$ is the derivative of $\log F$
if $R$ is a commutative $\mathbb{Q}$-algebra\footnote{This explains the name
\textquotedblleft logarithmic derivative\textquotedblright.}. However, if $R$
is not a $\mathbb{Q}$-algebra, then $\log F$ is not defined, so that
$\operatorname*{lder}F$ can only be defined via Definition \ref{def.lder.lder}.

We shall now state (and, for the sake of completeness, prove) a few well-known
properties of logarithmic derivatives:

\begin{proposition}
\label{prop.lder.lder(uv)}Let $R$ be a commutative ring. Let $u,v\in R\left[
\left[  t\right]  \right]  $ be two formal power series whose constant terms
are $1$. Then, $\operatorname*{lder}\left(  uv\right)  =\operatorname*{lder}%
u+\operatorname*{lder}v$.
\end{proposition}

\begin{proof}
[Proof of Proposition \ref{prop.lder.lder(uv)}.]The two power series $u$ and
$v$ have constant term $1$. Hence, their product $uv$ has constant term $1$ as
well (since the constant term of the product of two power series equals the
product of their constant terms). Thus, $\operatorname*{lder}\left(
uv\right)  $ is well-defined.

The Leibniz rule yields $\left(  uv\right)  ^{\prime}=u^{\prime}v+uv^{\prime}%
$. However, the definition of $\operatorname*{lder}u$ yields
$\operatorname*{lder}u=\dfrac{u^{\prime}}{u}$. Likewise, $\operatorname*{lder}%
v=\dfrac{v^{\prime}}{v}$. Adding these two equalities, we obtain%
\[
\operatorname*{lder}u+\operatorname*{lder}v=\dfrac{u^{\prime}}{u}%
+\dfrac{v^{\prime}}{v}=\dfrac{u^{\prime}v+uv^{\prime}}{uv}.
\]
On the other hand, the definition of $\operatorname*{lder}\left(  uv\right)  $
yields%
\[
\operatorname*{lder}\left(  uv\right)  =\dfrac{\left(  uv\right)  ^{\prime}%
}{uv}=\dfrac{u^{\prime}v+uv^{\prime}}{uv}\ \ \ \ \ \ \ \ \ \ \left(
\text{since }\left(  uv\right)  ^{\prime}=u^{\prime}v+uv^{\prime}\right)  .
\]
Comparing these two equalities, we find $\operatorname*{lder}\left(
uv\right)  =\operatorname*{lder}u+\operatorname*{lder}v$. This proves
Proposition \ref{prop.lder.lder(uv)}.
\end{proof}

\begin{proposition}
\label{prop.lder.lder(lim)}Let $R$ be a commutative topological ring. Let
$\left(  u_{n}\right)  _{n\in\mathbb{N}}=\left(  u_{0},u_{1},u_{2}%
,\ldots\right)  \in R\left[  \left[  t\right]  \right]  ^{\mathbb{N}}$ be a
sequence of formal power series whose constant terms are $1$. Let $u\in
R\left[  \left[  t\right]  \right]  $ be a formal power series whose constant
term is $1$. Assume that $\lim\limits_{n\rightarrow\infty}u_{n}=u$ (with
respect to the standard topology on $R\left[  \left[  t\right]  \right]  $
induced by the topology on $R$). Then, $\lim\limits_{n\rightarrow\infty
}\left(  \operatorname*{lder}u_{n}\right)  =\operatorname*{lder}u$ (with
respect to the same topology on $R\left[  \left[  t\right]  \right]  $).
\end{proposition}

\begin{proof}
[Proof of Proposition \ref{prop.lder.lder(lim)}.]It is well-known that the map%
\begin{align*}
R\left[  \left[  t\right]  \right]   &  \rightarrow R\left[  \left[  t\right]
\right]  ,\\
v  &  \mapsto v^{\prime}%
\end{align*}
(that is, the map that sends each power series $v\in R\left[  \left[
t\right]  \right]  $ to its derivative $v^{\prime}$) is
continuous\footnote{This follows from the fact that for each $n\in\mathbb{N}$,
the $n$-th coefficient of the derivative $v^{\prime}$ of a power series $v\in
R\left[  \left[  t\right]  \right]  $ is a continuous function of the first
$n+2$ coefficients of $v$ (indeed, it equals $n+1$ times the $\left(
n+1\right)  $-st coefficient of $v$).}. Hence, from $\lim\limits_{n\rightarrow
\infty}u_{n}=u$, we obtain $\lim\limits_{n\rightarrow\infty}u_{n}^{\prime
}=u^{\prime}$.

Let $R\left[  \left[  t\right]  \right]  _{1}$ be the subset $\left\{  v\in
R\left[  \left[  t\right]  \right]  \ \mid\ \text{the constant term of
}v\text{ is }1\right\}  $ of $R\left[  \left[  t\right]  \right]  $. Then, the
formal power series $u_{0},u_{1},u_{2},\ldots$ and $u$ all belong to $R\left[
\left[  t\right]  \right]  _{1}$ (since their constant terms are $1$). It is
well-known that the map%
\begin{align*}
R\left[  \left[  t\right]  \right]  _{1}  &  \rightarrow R\left[  \left[
t\right]  \right]  ,\\
v  &  \mapsto\dfrac{1}{v}%
\end{align*}
is continuous\footnote{This follows from the fact that for each $n\in
\mathbb{N}$, the $n$-th coefficient of the power series $\dfrac{1}{v}$ (where
$v\in R\left[  \left[  t\right]  \right]  _{1}$) is a continuous function of
the first $n+1$ coefficients of $v$ (indeed, if we let $v_{i}$ and $\left(
\dfrac{1}{v}\right)  _{i}$ denote the $i$-th coefficients of the power series
$v$ and $\dfrac{1}{v}$ for all $i\in\mathbb{N}$, then the coefficients of
$\dfrac{1}{v}$ can be computed recursively from the coefficients of $v$ using
the formulas%
\begin{align*}
\left(  \dfrac{1}{v}\right)  _{0}  &  =1\ \ \ \ \ \ \ \ \ \ \text{and}\\
\left(  \dfrac{1}{v}\right)  _{n}  &  =-\sum_{i=1}^{n}v_{i}\left(  \dfrac
{1}{v}\right)  _{n-i}\ \ \ \ \ \ \ \ \ \ \text{for each }n>0;
\end{align*}
these formulas rely only on addition, subtraction and multiplication of
elements of $R$, and therefore define continuous maps).}. Thus, from
$\lim\limits_{n\rightarrow\infty}u_{n}=u$, we obtain $\lim
\limits_{n\rightarrow\infty}\dfrac{1}{u_{n}}=\dfrac{1}{u}$ (since the formal
power series $u_{0},u_{1},u_{2},\ldots$ and $u$ all belong to $R\left[
\left[  t\right]  \right]  _{1}$).

Finally, it is well-known that the map%
\begin{align*}
R\left[  \left[  t\right]  \right]  \times R\left[  \left[  t\right]  \right]
&  \rightarrow R\left[  \left[  t\right]  \right]  ,\\
\left(  v,w\right)   &  \mapsto vw
\end{align*}
is continuous\footnote{This follows from the fact that for each $n\in
\mathbb{N}$, the $n$-th coefficient of the power series $vw$ (where $\left(
v,w\right)  \in R\left[  \left[  t\right]  \right]  \times R\left[  \left[
t\right]  \right]  $) is a continuous function of the first $n+1$ coefficients
of $v$ and of $w$ (indeed, it is equal to $\sum_{k=0}^{n}v_{k}w_{n-k}$, where
$v_{i}$ and $w_{i}$ denote the $i$-th coefficients of $v$ and $w$).}. Hence,
from $\lim\limits_{n\rightarrow\infty}u_{n}^{\prime}=u^{\prime}$ and
$\lim\limits_{n\rightarrow\infty}\dfrac{1}{u_{n}}=\dfrac{1}{u}$, we obtain%
\begin{equation}
\lim\limits_{n\rightarrow\infty}\left(  u_{n}^{\prime}\cdot\dfrac{1}{u_{n}%
}\right)  =u^{\prime}\cdot\dfrac{1}{u}=\dfrac{u^{\prime}}{u}%
=\operatorname*{lder}u \label{pf.prop.lder.lder(lim).4}%
\end{equation}
(since $\operatorname*{lder}u$ is defined to be $\dfrac{u^{\prime}}{u}$).
However, for each $n\in\mathbb{N}$, we have%
\begin{align*}
\operatorname*{lder}u_{n}  &  =\dfrac{u_{n}^{\prime}}{u_{n}}%
\ \ \ \ \ \ \ \ \ \ \left(  \text{by the definition of }\operatorname*{lder}%
u_{n}\right) \\
&  =u_{n}^{\prime}\cdot\dfrac{1}{u_{n}}.
\end{align*}
Thus, (\ref{pf.prop.lder.lder(lim).4}) rewrites as $\lim\limits_{n\rightarrow
\infty}\left(  \operatorname*{lder}u_{n}\right)  =\operatorname*{lder}u$. This
proves Proposition \ref{prop.lder.lder(lim)}.
\end{proof}

\begin{proposition}
\label{prop.lder.lder(finprod)}Let $R$ be a commutative ring. Let $u_{1}%
,u_{2},\ldots,u_{n}\in R\left[  \left[  t\right]  \right]  $ be finitely many
formal power series whose constant terms are $1$. Then,%
\[
\operatorname*{lder}\left(  \prod_{i=1}^{n}u_{i}\right)  =\sum_{i=1}%
^{n}\operatorname*{lder}u_{i}.
\]

\end{proposition}

\begin{proof}
[Proof of Proposition \ref{prop.lder.lder(finprod)}.]We shall show that%
\begin{equation}
\operatorname*{lder}\left(  \prod_{i=1}^{m}u_{i}\right)  =\sum_{i=1}%
^{m}\operatorname*{lder}u_{i} \label{pf.prop.lder.lder(finprod).claim}%
\end{equation}
for each $m\in\left\{  0,1,\ldots,n\right\}  $.

Indeed, let us prove (\ref{pf.prop.lder.lder(finprod).claim}) by induction on
$m$:

\textit{Induction base:} We have $\prod_{i=1}^{0}u_{i}=\left(  \text{empty
product}\right)  =1$ and thus%
\begin{align*}
\operatorname*{lder}\left(  \prod_{i=1}^{0}u_{i}\right)   &
=\operatorname*{lder}1=\dfrac{1^{\prime}}{1}\ \ \ \ \ \ \ \ \ \ \left(
\text{by the definition of }\operatorname*{lder}1\right) \\
&  =1^{\prime}=0=\sum_{i=1}^{0}\operatorname*{lder}u_{i}%
\end{align*}
(since $\sum_{i=1}^{0}\operatorname*{lder}u_{i}=\left(  \text{empty
sum}\right)  =0$). In other words, (\ref{pf.prop.lder.lder(finprod).claim})
holds for $m=0$.

\textit{Induction step:} Let $k\in\left\{  0,1,\ldots,n-1\right\}  $. Assume
that (\ref{pf.prop.lder.lder(finprod).claim}) holds for $m=k$. We must prove
that (\ref{pf.prop.lder.lder(finprod).claim}) holds for $m=k+1$.

The power series $u_{1},u_{2},\ldots,u_{k}$ have constant term $1$. Hence,
their product $u_{1}u_{2}\cdots u_{k}$ has constant term $1$ as well (since
the constant term of the product of some power series equals the product of
their constant terms\footnote{This is a consequence of the fact that the map
from $\mathbf{k}\left[  \left[  t\right]  \right]  $ to $\mathbf{k}$ that
sends each power series to its constant term is a $\mathbf{k}$-algebra
homomorphism.}).

We have assumed that (\ref{pf.prop.lder.lder(finprod).claim}) holds for $m=k$.
In other words, we have%
\[
\operatorname*{lder}\left(  \prod_{i=1}^{k}u_{i}\right)  =\sum_{i=1}%
^{k}\operatorname*{lder}u_{i}.
\]
Now,%
\begin{align*}
&  \operatorname*{lder}\underbrace{\left(  \prod_{i=1}^{k+1}u_{i}\right)
}_{=\left(  \prod_{i=1}^{k}u_{i}\right)  \cdot u_{k+1}}\\
&  =\operatorname*{lder}\left(  \left(  \prod_{i=1}^{k}u_{i}\right)  \cdot
u_{k+1}\right)  =\underbrace{\operatorname*{lder}\left(  \prod_{i=1}^{k}%
u_{i}\right)  }_{=\sum_{i=1}^{k}\operatorname*{lder}u_{i}}%
+\operatorname*{lder}u_{k+1}\\
&  \ \ \ \ \ \ \ \ \ \ \ \ \ \ \ \ \ \ \ \ \left(  \text{by Proposition
\ref{prop.lder.lder(uv)}, applied to }u=\prod_{i=1}^{k}u_{i}\text{ and
}v=u_{k+1}\right) \\
&  =\sum_{i=1}^{k}\operatorname*{lder}u_{i}+\operatorname*{lder}u_{k+1}%
=\sum_{i=1}^{k+1}\operatorname*{lder}u_{i}.
\end{align*}
In other words, (\ref{pf.prop.lder.lder(finprod).claim}) holds for $m=k+1$.
This completes the induction step. Thus,
(\ref{pf.prop.lder.lder(finprod).claim}) is proved by induction.

Now, applying (\ref{pf.prop.lder.lder(finprod).claim}) to $m=n$, we obtain
\[
\operatorname*{lder}\left(  \prod_{i=1}^{n}u_{i}\right)  =\sum_{i=1}%
^{n}\operatorname*{lder}u_{i}.
\]
This proves Proposition \ref{prop.lder.lder(finprod)}.
\end{proof}

\begin{proposition}
\label{prop.lder.lder(infprod)}Let $R$ be a commutative topological ring. Let
$u_{1},u_{2},u_{3},\ldots\in R\left[  \left[  t\right]  \right]  $ be
infinitely many formal power series whose constant terms are $1$. Assume that
the infinite product $\prod_{i=1}^{\infty}u_{i}$ converges (with respect to
the standard topology on $R\left[  \left[  t\right]  \right]  $ induced by the
topology on $R$). Then, the infinite sum $\sum_{i=1}^{\infty}%
\operatorname*{lder}u_{i}$ converges as well, and we have%
\begin{equation}
\operatorname*{lder}\left(  \prod_{i=1}^{\infty}u_{i}\right)  =\sum
_{i=1}^{\infty}\operatorname*{lder}u_{i}.
\label{eq.prop.lder.lder(infprod).eq}%
\end{equation}

\end{proposition}

\begin{proof}
[Proof of Proposition \ref{prop.lder.lder(infprod)}.]The power series
$u_{1},u_{2},u_{3},\ldots$ have constant term $1$. Hence, their product
$\prod_{i=1}^{\infty}u_{i}$ has constant term $1$ as well (since the constant
term of the product of some power series equals the product of their constant
terms). Moreover, all the partial products $\prod_{i=1}^{n}u_{i}$ of this
infinite product also have constant term $1$ (for the same reason).

The infinite product $\prod_{i=1}^{\infty}u_{i}$ converges. Thus, the sequence
$\left(  \prod_{i=1}^{n}u_{i}\right)  _{n\in\mathbb{N}}\in R\left[  \left[
t\right]  \right]  ^{\mathbb{N}}$ converges, and its limit is $\lim
\limits_{n\rightarrow\infty}\left(  \prod_{i=1}^{n}u_{i}\right)  =\prod
_{i=1}^{\infty}u_{i}$. Hence, Proposition \ref{prop.lder.lder(lim)} (applied
to $\prod_{i=1}^{n}u_{i}$ and $\prod_{i=1}^{\infty}u_{i}$ instead of $u_{n}$
and $u$) yields%
\begin{equation}
\lim\limits_{n\rightarrow\infty}\left(  \operatorname*{lder}\left(
\prod_{i=1}^{n}u_{i}\right)  \right)  =\operatorname*{lder}\left(  \prod
_{i=1}^{\infty}u_{i}\right)  . \label{pf.prop.lder.lder(infprod).1}%
\end{equation}

However,
\[
\lim\limits_{n\rightarrow\infty}\underbrace{\left(  \operatorname*{lder}%
\left(  \prod_{i=1}^{n}u_{i}\right)  \right)  }_{\substack{=\sum_{i=1}%
^{n}\operatorname*{lder}u_{i}\\\text{(by Proposition
\ref{prop.lder.lder(finprod)})}}}=\lim\limits_{n\rightarrow\infty}\sum
_{i=1}^{n}\operatorname*{lder}u_{i}.
\]
Thus, (\ref{pf.prop.lder.lder(infprod).1}) rewrites as
\[
\lim\limits_{n\rightarrow\infty}\sum_{i=1}^{n}\operatorname*{lder}%
u_{i}=\operatorname*{lder}\left(  \prod_{i=1}^{\infty}u_{i}\right)  .
\]
Thus, the sequence $\left(  \sum_{i=1}^{n}\operatorname*{lder}u_{i}\right)
_{n\in\mathbb{N}}$ converges. In other words, the infinite sum $\sum
_{i=1}^{\infty}\operatorname*{lder}u_{i}$ converges. Moreover, the value of
this sum is
\[
\sum_{i=1}^{\infty}\operatorname*{lder}u_{i}=\lim\limits_{n\rightarrow\infty
}\sum_{i=1}^{n}\operatorname*{lder}u_{i}=\operatorname*{lder}\left(
\prod_{i=1}^{\infty}u_{i}\right)  .
\]
This proves (\ref{eq.prop.lder.lder(infprod).eq}). Thus, the proof of
Proposition \ref{prop.lder.lder(infprod)} is complete.
\end{proof}

\begin{proposition}
\label{prop.lder.lder(lambdat)}Let $R$ be a commutative ring. Let $u\in
R\left[  \left[  t\right]  \right]  $ be a formal power series whose constant
term is $1$. Let $\lambda\in R$. Then,%
\[
\operatorname*{lder}\left(  u\left(  \lambda t\right)  \right)  =\lambda
\cdot\left(  \operatorname*{lder}u\right)  \left(  \lambda t\right)  .
\]

\end{proposition}

\begin{proof}
[Proof of Proposition \ref{prop.lder.lder(lambdat)}.]We first claim that the
derivative of the power series $u\left(  \lambda t\right)  $ is%
\begin{equation}
\left(  u\left(  \lambda t\right)  \right)  ^{\prime}=\lambda\cdot u^{\prime
}\left(  \lambda t\right)  . \label{pf.prop.lder.lder(lambdat).1}%
\end{equation}

[\textit{Proof of (\ref{pf.prop.lder.lder(lambdat).1}):} This is easy to see
using the chain rule, but let us show this directly: Write the power series
$u\in R\left[  \left[  t\right]  \right]  $ in the form $u=\sum_{n\geq0}%
u_{n}t^{n}$ for some $u_{0},u_{1},u_{2},\ldots\in R$. Thus,
\[
u\left(  \lambda t\right)  =\sum_{n\geq0}u_{n}\underbrace{\left(  \lambda
t\right)  ^{n}}_{=\lambda^{n}t^{n}}=\sum_{n\geq0}u_{n}\lambda^{n}t^{n}.
\]
Hence, the definition of a derivative yields%
\begin{align}
\left(  u\left(  \lambda t\right)  \right)  ^{\prime}  &  =\sum_{n\geq1}%
nu_{n}\underbrace{\lambda^{n}}_{\substack{=\lambda\cdot\lambda^{n-1}%
\\\text{(since }n\geq1\text{)}}}t^{n-1}=\sum_{n\geq1}nu_{n}\lambda
\cdot\underbrace{\lambda^{n-1}t^{n-1}}_{=\left(  \lambda t\right)  ^{n-1}%
}\nonumber\\
&  =\sum_{n\geq1}nu_{n}\lambda\cdot\left(  \lambda t\right)  ^{n-1}.
\label{pf.prop.lder.lder(lambdat).1.pf.2}%
\end{align}
On the other hand, from $u=\sum_{n\geq0}u_{n}t^{n}$, we obtain $u^{\prime
}=\sum_{n\geq1}nu_{n}t^{n-1}$. Substituting $\lambda t$ for $t$ on both sides
of this equality, we find%
\[
u^{\prime}\left(  \lambda t\right)  =\sum_{n\geq1}nu_{n}\left(  \lambda
t\right)  ^{n-1}.
\]
Multiplying this equality by $\lambda$, we find%
\[
\lambda\cdot u^{\prime}\left(  \lambda t\right)  =\lambda\cdot\sum_{n\geq
1}nu_{n}\left(  \lambda t\right)  ^{n-1}=\sum_{n\geq1}nu_{n}\lambda
\cdot\left(  \lambda t\right)  ^{n-1}.
\]
Comparing this with (\ref{pf.prop.lder.lder(lambdat).1.pf.2}), we obtain
$\left(  u\left(  \lambda t\right)  \right)  ^{\prime}=\lambda\cdot u^{\prime
}\left(  \lambda t\right)  $. This proves (\ref{pf.prop.lder.lder(lambdat).1}).]

Now, the definition of $\operatorname*{lder}u$ yields $\operatorname*{lder}%
u=\dfrac{u^{\prime}}{u}$. Substituting $\lambda t$ for $t$ on both sides of
this equality, we obtain
\[
\left(  \operatorname*{lder}u\right)  \left(  \lambda t\right)  =\dfrac
{u^{\prime}\left(  \lambda t\right)  }{u\left(  \lambda t\right)  }.
\]
Multiplying this equality by $\lambda$, we find%
\[
\lambda\cdot\left(  \operatorname*{lder}u\right)  \left(  \lambda t\right)
=\lambda\cdot\dfrac{u^{\prime}\left(  \lambda t\right)  }{u\left(  \lambda
t\right)  }=\dfrac{\lambda\cdot u^{\prime}\left(  \lambda t\right)  }{u\left(
\lambda t\right)  }.
\]
On the other hand, the definition of $\operatorname*{lder}\left(  u\left(
\lambda t\right)  \right)  $ yields%
\[
\operatorname*{lder}\left(  u\left(  \lambda t\right)  \right)  =\dfrac
{\left(  u\left(  \lambda t\right)  \right)  ^{\prime}}{u\left(  \lambda
t\right)  }=\dfrac{\lambda\cdot u^{\prime}\left(  \lambda t\right)  }{u\left(
\lambda t\right)  }\ \ \ \ \ \ \ \ \ \ \left(  \text{by
(\ref{pf.prop.lder.lder(lambdat).1})}\right)  .
\]
Comparing these two equalities, we obtain $\operatorname*{lder}\left(
u\left(  \lambda t\right)  \right)  =\lambda\cdot\left(  \operatorname*{lder}%
u\right)  \left(  \lambda t\right)  $. This proves Proposition
\ref{prop.lder.lder(lambdat)}.
\end{proof}

\begin{proposition}
\label{prop.lder.functorial}Let $R$ and $S$ be two commutative $\mathbf{k}%
$-algebras. Let $\alpha:R\rightarrow S$ be a $\mathbf{k}$-algebra
homomorphism. As we know, $\alpha$ induces a continuous $\mathbf{k}\left[
\left[  t\right]  \right]  $-algebra homomorphism%
\[
\alpha\left[  \left[  t\right]  \right]  :R\left[  \left[  t\right]  \right]
\rightarrow S\left[  \left[  t\right]  \right]
\]
that sends each formal power series $\sum_{n\geq0}a_{n}t^{n}\in R\left[
\left[  t\right]  \right]  $ (with $a_{n}\in R$) to $\sum_{n\geq0}%
\alpha\left(  a_{n}\right)  t^{n}\in S\left[  \left[  t\right]  \right]  $.

Let $u\in R\left[  \left[  t\right]  \right]  $ be a formal power series whose
constant term is $1$. Then, the constant term of the power series $\left(
\alpha\left[  \left[  t\right]  \right]  \right)  \left(  u\right)  $ is $1$,
and we have%
\[
\operatorname*{lder}\left(  \left(  \alpha\left[  \left[  t\right]  \right]
\right)  \left(  u\right)  \right)  =\left(  \alpha\left[  \left[  t\right]
\right]  \right)  \left(  \operatorname*{lder}u\right)  .
\]

\end{proposition}

\begin{proof}
[Proof of Proposition \ref{prop.lder.functorial}.]Write the power series $u\in
R\left[  \left[  t\right]  \right]  $ in the form $u=\sum_{n\geq0}u_{n}t^{n}$
with $u_{0},u_{1},u_{2},\ldots\in R$. Then, the definition of $\alpha\left[
\left[  t\right]  \right]  $ yields $\left(  \alpha\left[  \left[  t\right]
\right]  \right)  \left(  u\right)  =\sum_{n\geq0}\alpha\left(  u_{n}\right)
t^{n}$. However, $u_{0}$ is the constant term of $u$ (since $u=\sum_{n\geq
0}u_{n}t^{n}$), and thus is $1$ (since we know that the constant term of $u$
is $1$). In other words, $u_{0}=1$. Hence, $\alpha\left(  u_{0}\right)
=\alpha\left(  1\right)  =1$ (since $\alpha$ is a $\mathbf{k}$-algebra
homomorphism). However, the constant term of the power series $\left(
\alpha\left[  \left[  t\right]  \right]  \right)  \left(  u\right)  $ is
$\alpha\left(  u_{0}\right)  $ (since $\left(  \alpha\left[  \left[  t\right]
\right]  \right)  \left(  u\right)  =\sum_{n\geq0}\alpha\left(  u_{n}\right)
t^{n}$). In other words, the constant term of the power series $\left(
\alpha\left[  \left[  t\right]  \right]  \right)  \left(  u\right)  $ is $1$
(since $\alpha\left(  u_{0}\right)  =1$). Hence, $\operatorname*{lder}\left(
\left(  \alpha\left[  \left[  t\right]  \right]  \right)  \left(  u\right)
\right)  $ is well-defined. The definition of $\operatorname*{lder}\left(
\left(  \alpha\left[  \left[  t\right]  \right]  \right)  \left(  u\right)
\right)  $ yields%
\begin{equation}
\operatorname*{lder}\left(  \left(  \alpha\left[  \left[  t\right]  \right]
\right)  \left(  u\right)  \right)  =\dfrac{\left(  \left(  \alpha\left[
\left[  t\right]  \right]  \right)  \left(  u\right)  \right)  ^{\prime}%
}{\left(  \alpha\left[  \left[  t\right]  \right]  \right)  \left(  u\right)
}. \label{pf.prop.lder.functorial.2}%
\end{equation}

Now, recall that $\left(  \alpha\left[  \left[  t\right]  \right]  \right)
\left(  u\right)  =\sum_{n\geq0}\alpha\left(  u_{n}\right)  t^{n}$ with
$\alpha\left(  u_{0}\right)  ,\alpha\left(  u_{1}\right)  ,\alpha\left(
u_{2}\right)  ,\ldots\in S$. Hence, the definition of the derivative of a
power series yields%
\begin{align}
\left(  \left(  \alpha\left[  \left[  t\right]  \right]  \right)  \left(
u\right)  \right)  ^{\prime}  &  =\sum_{n\geq1}n\alpha\left(  u_{n}\right)
t^{n-1}\nonumber\\
&  =\sum_{n\geq0}\left(  n+1\right)  \alpha\left(  u_{n+1}\right)  t^{n}
\label{pf.prop.lder.functorial.3}%
\end{align}
(here, we have substituted $n+1$ for $n$ in the sum). On the other hand, from
$u=\sum_{n\geq0}u_{n}t^{n}$, we obtain
\begin{align*}
u^{\prime}  &  =\sum_{n\geq1}nu_{n}t^{n-1}\ \ \ \ \ \ \ \ \ \ \left(  \text{by
the definition of the derivative}\right) \\
&  =\sum_{n\geq0}\left(  n+1\right)  u_{n+1}t^{n}%
\end{align*}
(here, we have substituted $n+1$ for $n$ in the sum). Applying the map
$\alpha\left[  \left[  t\right]  \right]  $ to both sides of this equality, we
obtain%
\begin{align*}
\left(  \alpha\left[  \left[  t\right]  \right]  \right)  \left(  u^{\prime
}\right)   &  =\left(  \alpha\left[  \left[  t\right]  \right]  \right)
\left(  \sum_{n\geq0}\left(  n+1\right)  u_{n+1}t^{n}\right)  =\sum_{n\geq
0}\underbrace{\alpha\left(  \left(  n+1\right)  u_{n+1}\right)  }%
_{\substack{=\left(  n+1\right)  \alpha\left(  u_{n+1}\right)  \\\text{(since
the map }\alpha\\\text{is }\mathbf{k}\text{-linear)}}}t^{n}\\
&  \ \ \ \ \ \ \ \ \ \ \ \ \ \ \ \ \ \ \ \ \left(  \text{by the definition of
}\alpha\left[  \left[  t\right]  \right]  \right) \\
&  =\sum_{n\geq0}\left(  n+1\right)  \alpha\left(  u_{n+1}\right)  t^{n}.
\end{align*}
Comparing this with (\ref{pf.prop.lder.functorial.3}), we obtain%
\[
\left(  \left(  \alpha\left[  \left[  t\right]  \right]  \right)  \left(
u\right)  \right)  ^{\prime}=\left(  \alpha\left[  \left[  t\right]  \right]
\right)  \left(  u^{\prime}\right)  .
\]
Hence, (\ref{pf.prop.lder.functorial.2}) rewrites as%
\begin{equation}
\operatorname*{lder}\left(  \left(  \alpha\left[  \left[  t\right]  \right]
\right)  \left(  u\right)  \right)  =\dfrac{\left(  \alpha\left[  \left[
t\right]  \right]  \right)  \left(  u^{\prime}\right)  }{\left(  \alpha\left[
\left[  t\right]  \right]  \right)  \left(  u\right)  }.
\label{pf.prop.lder.functorial.6}%
\end{equation}

On the other hand, the definition of $\operatorname*{lder}u$ yields
$\operatorname*{lder}u=\dfrac{u^{\prime}}{u}$. Applying the map $\alpha\left[
\left[  t\right]  \right]  $ to both sides of this equality, we obtain%
\[
\left(  \alpha\left[  \left[  t\right]  \right]  \right)  \left(
\operatorname*{lder}u\right)  =\left(  \alpha\left[  \left[  t\right]
\right]  \right)  \left(  \dfrac{u^{\prime}}{u}\right)  =\dfrac{\left(
\alpha\left[  \left[  t\right]  \right]  \right)  \left(  u^{\prime}\right)
}{\left(  \alpha\left[  \left[  t\right]  \right]  \right)  \left(  u\right)
}%
\]
(since the map $\alpha\left[  \left[  t\right]  \right]  $ is a $\mathbf{k}%
\left[  \left[  t\right]  \right]  $-algebra homomorphism and thus respects
quotients). Comparing this with (\ref{pf.prop.lder.functorial.6}), we obtain
$\operatorname*{lder}\left(  \left(  \alpha\left[  \left[  t\right]  \right]
\right)  \left(  u\right)  \right)  =\left(  \alpha\left[  \left[  t\right]
\right]  \right)  \left(  \operatorname*{lder}u\right)  $. This completes the
proof of Proposition \ref{prop.lder.functorial}.
\end{proof}

Next, we shall prove a simple property of homogeneous power series:

\begin{lemma}
\label{lem.hg-ps-txi}Consider the ring $\left(  \mathbf{k}\left[  \left[
x_{1},x_{2},x_{3},\ldots\right]  \right]  \right)  \left[  \left[  t\right]
\right]  $ of formal power series in one indeterminate $t$ over $\mathbf{k}%
\left[  \left[  x_{1},x_{2},x_{3},\ldots\right]  \right]  $. Let
$n\in\mathbb{N}$. Let $u\in\mathbf{k}\left[  \left[  x_{1},x_{2},x_{3}%
,\ldots\right]  \right]  $ be any power series that is homogeneous of degree
$n$. Then,%
\[
u\left(  tx_{1},tx_{2},tx_{3},\ldots\right)  =t^{n}\cdot u.
\]

\end{lemma}

\begin{proof}
[Proof of Lemma \ref{lem.hg-ps-txi}.]The power series $u$ is homogeneous of
degree $n$. In other words, it can be written as an infinite $\mathbf{k}%
$-linear combination of monomials of degree $n$. In other words, it can be
written in the form%
\begin{equation}
u=\sum_{\substack{\mathfrak{m}\text{ is a monomial}\\\text{of degree }%
n}}u_{\mathfrak{m}}\mathfrak{m} \label{pf.lem.hg-ps-txi.1}%
\end{equation}
for some coefficients $u_{\mathfrak{m}}\in\mathbf{k}$. Consider these
coefficients $u_{\mathfrak{m}}$. Substituting $tx_{1},tx_{2},tx_{3},\ldots$
for $x_{1},x_{2},x_{3},\ldots$ on both sides of the equality
(\ref{pf.lem.hg-ps-txi.1}), we obtain%
\begin{equation}
u\left(  tx_{1},tx_{2},tx_{3},\ldots\right)  =\sum_{\substack{\mathfrak{m}%
\text{ is a monomial}\\\text{of degree }n}}u_{\mathfrak{m}}\mathfrak{m}\left(
tx_{1},tx_{2},tx_{3},\ldots\right)  . \label{pf.lem.hg-ps-txi.2}%
\end{equation}

However, if $\mathfrak{m}$ is a monomial of degree $n$, then%
\begin{equation}
\mathfrak{m}\left(  tx_{1},tx_{2},tx_{3},\ldots\right)  =t^{n}\cdot
\mathfrak{m}. \label{pf.lem.hg-ps-txi.3}%
\end{equation}

[\textit{Proof of (\ref{pf.lem.hg-ps-txi.3}):} Let $\mathfrak{m}$ be a
monomial of degree $n$. Thus, $\mathfrak{m}$ is a product of $n$
indeterminates. In other words, $\mathfrak{m}=x_{i_{1}}x_{i_{2}}\cdots
x_{i_{n}}$ for some $i_{1},i_{2},\ldots,i_{n}\in\left\{  1,2,3,\ldots\right\}
$. Consider these $i_{1},i_{2},\ldots,i_{n}$. Substituting $tx_{1}%
,tx_{2},tx_{3},\ldots$ for $x_{1},x_{2},x_{3},\ldots$ on both sides of the
equality $\mathfrak{m}=x_{i_{1}}x_{i_{2}}\cdots x_{i_{n}}$, we obtain%
\[
\mathfrak{m}\left(  tx_{1},tx_{2},tx_{3},\ldots\right)  =\left(  tx_{i_{1}%
}\right)  \left(  tx_{i_{2}}\right)  \cdots\left(  tx_{i_{n}}\right)
=t^{n}\cdot\underbrace{x_{i_{1}}x_{i_{2}}\cdots x_{i_{n}}}_{=\mathfrak{m}%
}=t^{n}\cdot\mathfrak{m}.
\]
This proves (\ref{pf.lem.hg-ps-txi.3}).]

Hence, (\ref{pf.lem.hg-ps-txi.2}) becomes%
\[
u\left(  tx_{1},tx_{2},tx_{3},\ldots\right)  =\sum_{\substack{\mathfrak{m}%
\text{ is a monomial}\\\text{of degree }n}}u_{\mathfrak{m}}%
\underbrace{\mathfrak{m}\left(  tx_{1},tx_{2},tx_{3},\ldots\right)
}_{\substack{=t^{n}\cdot\mathfrak{m}\\\text{(by (\ref{pf.lem.hg-ps-txi.3}))}%
}}=\sum_{\substack{\mathfrak{m}\text{ is a monomial}\\\text{of degree }%
n}}u_{\mathfrak{m}}t^{n}\cdot\mathfrak{m}.
\]
Comparing this with%
\[
t^{n}\cdot\underbrace{u}_{=\sum_{\substack{\mathfrak{m}\text{ is a
monomial}\\\text{of degree }n}}u_{\mathfrak{m}}\mathfrak{m}}=t^{n}\cdot
\sum_{\substack{\mathfrak{m}\text{ is a monomial}\\\text{of degree }%
n}}u_{\mathfrak{m}}\mathfrak{m}=\sum_{\substack{\mathfrak{m}\text{ is a
monomial}\\\text{of degree }n}}u_{\mathfrak{m}}t^{n}\cdot\mathfrak{m},
\]
we obtain $u\left(  tx_{1},tx_{2},tx_{3},\ldots\right)  =t^{n}\cdot u$. This
proves Lemma \ref{lem.hg-ps-txi}.
\end{proof}
\end{verlong}

\begin{verlong}
\begin{proof}
[Proof of Theorem \ref{thm.VF.main}.]\textbf{(b)} Let $m\in\mathbb{N}$. We
must prove that $V_{F}\left(  h_{m}\right)  =G_{F,m}$.

We have $h_{0}=1$ and thus $V_{F}\left(  h_{0}\right)  =V_{F}\left(  1\right)
=1$ (since $V_{F}$ is a $\mathbf{k}$-algebra homomorphism). Comparing this
with $G_{F,0}=1$ (which was proved in Proposition \ref{prop.GF.basics}
\textbf{(e)}), we obtain $V_{F}\left(  h_{0}\right)  =G_{F,0}$. Hence,
$V_{F}\left(  h_{m}\right)  =G_{F,m}$ is proved for $m=0$. Thus, for the rest
of this proof, we WLOG assume that $m\neq0$.

Now, $m$ is a positive integer (since $m\in\mathbb{N}$ and $m\neq0$). However,
the definition of $V_{F}$ says that $V_{F}\left(  h_{i}\right)  =G_{F,i}$ for
all positive integers $i$. We can apply this to $i=m$ (since $m$ is a positive
integer), and thus obtain $V_{F}\left(  h_{m}\right)  =G_{F,m}$. Thus, Theorem
\ref{thm.VF.main} \textbf{(b)} is proven.

\textbf{(a)} We recall that the family $\left(  h_{n}\right)  _{n\geq1}$
generates $\Lambda$ as a $\mathbf{k}$-algebra. Hence, any two $\mathbf{k}%
$-algebra homomorphisms with domain $\Lambda$ that agree on this family
$\left(  h_{n}\right)  _{n\geq1}$ must be identical. In other words, if $A$ is
any $\mathbf{k}$-algebra, and if $f:\Lambda\rightarrow A$ and $g:\Lambda
\rightarrow A$ are two $\mathbf{k}$-algebra homomorphisms such that%
\[
\left(  f\left(  h_{n}\right)  =g\left(  h_{n}\right)
\ \ \ \ \ \ \ \ \ \ \text{for each positive integer }n\right)  ,
\]
then
\begin{equation}
f=g. \label{pf.thm.VF.main.a.f=g}%
\end{equation}

The map $V_{F}$ is a $\mathbf{k}$-algebra homomorphism. Hence, the map
$V_{F}\otimes V_{F}$ is a $\mathbf{k}$-algebra homomorphism as well (since the
tensor product of two $\mathbf{k}$-algebra homomorphisms is always a
$\mathbf{k}$-algebra homomorphism).

Let $\Delta$ and $\varepsilon$ be the comultiplication and the counit of the
Hopf algebra $\Lambda$. Both of these maps $\Delta$ and $\varepsilon$ are
$\mathbf{k}$-algebra homomorphisms (since $\Lambda$ is a bialgebra). Hence,
the three maps $\Delta\circ V_{F}$ and $\left(  V_{F}\otimes V_{F}\right)
\circ\Delta$ and $\varepsilon\circ V_{F}$ are $\mathbf{k}$-algebra
homomorphisms as well (since these three maps are compositions of some of the
$\mathbf{k}$-algebra homomorphisms $\Delta$ and $\varepsilon$ and $V_{F}$ and
$V_{F}\otimes V_{F}$).

Now, let $n$ be a positive integer. Then, \cite[Proposition 2.3.6(iii)]%
{GriRei} yields%
\begin{align*}
\Delta\left(  h_{n}\right)   &  =\sum_{i+j=n}h_{i}\otimes h_{j}\\
&  \ \ \ \ \ \ \ \ \ \ \left(  \text{where the sum ranges over all pairs
}\left(  i,j\right)  \in\mathbb{N}\times\mathbb{N}\text{ with }i+j=n\right) \\
&  =\sum_{i\in\left\{  0,1,\ldots,n\right\}  }h_{i}\otimes h_{n-i}%
\end{align*}
(here, we have substituted $\left(  i,n-i\right)  $ for $\left(  i,j\right)  $
in the sum, since the map $\left\{  0,1,\ldots,n\right\}  \rightarrow\left\{
\left(  i,j\right)  \in\mathbb{N}\times\mathbb{N}\ \mid\ i+j=n\right\}  $ that
sends each $i$ to $\left(  i,n-i\right)  $ is a bijection). Applying the map
$V_{F}\otimes V_{F}$ to both sides of this equality, we find%
\begin{align*}
\left(  V_{F}\otimes V_{F}\right)  \left(  \Delta\left(  h_{n}\right)
\right)   &  =\left(  V_{F}\otimes V_{F}\right)  \left(  \sum_{i\in\left\{
0,1,\ldots,n\right\}  }h_{i}\otimes h_{n-i}\right)  =\sum_{i\in\left\{
0,1,\ldots,n\right\}  }\underbrace{\left(  V_{F}\otimes V_{F}\right)  \left(
h_{i}\otimes h_{n-i}\right)  }_{=V_{F}\left(  h_{i}\right)  \otimes
V_{F}\left(  h_{n-i}\right)  }\\
&  \ \ \ \ \ \ \ \ \ \ \ \ \ \ \ \ \ \ \ \ \left(  \text{since the map }%
V_{F}\otimes V_{F}\text{ is }\mathbf{k}\text{-linear}\right) \\
&  =\underbrace{\sum_{i\in\left\{  0,1,\ldots,n\right\}  }}_{=\sum_{i=0}^{n}%
}\underbrace{V_{F}\left(  h_{i}\right)  }_{\substack{=G_{F,i}\\\text{(by
Theorem \ref{thm.VF.main} \textbf{(b)},}\\\text{applied to }m=i\text{)}%
}}\otimes\underbrace{V_{F}\left(  h_{n-i}\right)  }_{\substack{=G_{F,n-i}%
\\\text{(by Theorem \ref{thm.VF.main} \textbf{(b)},}\\\text{applied to
}m=n-i\text{)}}}\\
&  =\sum_{i=0}^{n}G_{F,i}\otimes G_{F,n-i}.
\end{align*}
Comparing this with%
\begin{align*}
\left(  \Delta\circ V_{F}\right)  \left(  h_{n}\right)   &  =\Delta\left(
\underbrace{V_{F}\left(  h_{n}\right)  }_{\substack{=G_{F,n}\\\text{(by
Theorem \ref{thm.VF.main} \textbf{(b)},}\\\text{applied to }m=n\text{)}%
}}\right)  =\Delta\left(  G_{F,n}\right) \\
&  =\sum_{i=0}^{n}G_{F,i}\otimes G_{F,n-i}\ \ \ \ \ \ \ \ \ \ \left(  \text{by
Theorem \ref{thm.DeltaGFm}, applied to }m=n\right)  ,
\end{align*}
we obtain%
\[
\left(  \Delta\circ V_{F}\right)  \left(  h_{n}\right)  =\left(  V_{F}\otimes
V_{F}\right)  \left(  \Delta\left(  h_{n}\right)  \right)  =\left(  \left(
V_{F}\otimes V_{F}\right)  \circ\Delta\right)  \left(  h_{n}\right)  .
\]

Now, forget that we fixed $n$. We thus have shown that \newline$\left(
\Delta\circ V_{F}\right)  \left(  h_{n}\right)  =\left(  \left(  V_{F}\otimes
V_{F}\right)  \circ\Delta\right)  \left(  h_{n}\right)  $ for each positive
integer $n$. Since $\Delta\circ V_{F}$ and $\left(  V_{F}\otimes V_{F}\right)
\circ\Delta$ are $\mathbf{k}$-algebra homomorphisms, we thus conclude that%
\begin{equation}
\Delta\circ V_{F}=\left(  V_{F}\otimes V_{F}\right)  \circ\Delta
\label{pf.thm.VF.main.a.DVF}%
\end{equation}
(by (\ref{pf.thm.VF.main.a.f=g}), applied to $A=\Lambda\otimes\Lambda$ and
$f=\Delta\circ V_{F}$ and $g=\left(  V_{F}\otimes V_{F}\right)  \circ\Delta$).

Again, let $n$ be a positive integer. Recall that the counit $\varepsilon$ of
$\Lambda$ sends every homogeneous symmetric function of positive degree to
$0$. In other words, if $u\in\Lambda$ is homogeneous of positive degree, then%
\begin{equation}
\varepsilon\left(  u\right)  =0. \label{pf.thm.VF.main.a.eu=0}%
\end{equation}

The complete homogeneous symmetric function $h_{n}$ is homogeneous of degree
$n$, thus homogeneous of positive degree (since $n$ is positive). Hence,
(\ref{pf.thm.VF.main.a.eu=0}) (applied to $u=h_{n}$) yields $\varepsilon
\left(  h_{n}\right)  =0$.

Theorem \ref{thm.VF.main} \textbf{(b)} (applied to $m=n$) yields $V_{F}\left(
h_{n}\right)  =G_{F,n}$. Proposition \ref{prop.GF.basics} \textbf{(a)}
(applied to $m=n$) shows that the formal power series $G_{F,n}$ is the $n$-th
degree homogeneous component of $G_{F}$. Hence, this $G_{F,n}$ is homogeneous
of degree $n$. Thus, $G_{F,n}$ is homogeneous of positive degree (since $n$ is
positive). In other words, $V_{F}\left(  h_{n}\right)  $ is homogeneous of
positive degree (since $V_{F}\left(  h_{n}\right)  =G_{F,n}$). Since
$V_{F}\left(  h_{n}\right)  $ is clearly a symmetric function, we thus
conclude that $\varepsilon\left(  V_{F}\left(  h_{n}\right)  \right)  =0$ (by
(\ref{pf.thm.VF.main.a.eu=0}), applied to $u=V_{F}\left(  h_{n}\right)  $).
Thus, $\left(  \varepsilon\circ V_{F}\right)  \left(  h_{n}\right)
=\varepsilon\left(  V_{F}\left(  h_{n}\right)  \right)  =0$. Comparing this
with $\varepsilon\left(  h_{n}\right)  =0$, we find $\left(  \varepsilon\circ
V_{F}\right)  \left(  h_{n}\right)  =\varepsilon\left(  h_{n}\right)  $.

Now, forget that we fixed $n$. We thus have shown that $\left(  \varepsilon
\circ V_{F}\right)  \left(  h_{n}\right)  =\varepsilon\left(  h_{n}\right)  $
for each positive integer $n$. Since $\varepsilon\circ V_{F}$ and
$\varepsilon$ are $\mathbf{k}$-algebra homomorphisms, we thus conclude that%
\begin{equation}
\varepsilon\circ V_{F}=\varepsilon\label{pf.thm.VF.main.a.eVF}%
\end{equation}
(by (\ref{pf.thm.VF.main.a.f=g}), applied to $A=\mathbf{k}$ and $f=\varepsilon
\circ V_{F}$ and $g=\varepsilon$).

The two equalities (\ref{pf.thm.VF.main.a.DVF}) and
(\ref{pf.thm.VF.main.a.eVF}) show that $V_{F}$ is a $\mathbf{k}$-coalgebra
homomorphism (since $V_{F}$ is $\mathbf{k}$-linear). Since we also know that
$V_{F}$ is a $\mathbf{k}$-algebra homomorphism, we thus conclude that this map
$V_{F}$ is a $\mathbf{k}$-bialgebra homomorphism. Hence, $V_{F}$ is a
$\mathbf{k}$-Hopf algebra homomorphism\footnote{since any $\mathbf{k}%
$-bialgebra homomorphism between two $\mathbf{k}$-Hopf algebras is
automatically a $\mathbf{k}$-Hopf algebra homomorphism}. This proves Theorem
\ref{thm.VF.main} \textbf{(a)}.

\textbf{(c)} Let $m\in\mathbb{N}$. Then, Proposition \ref{prop.GF.basics}
\textbf{(a)} shows that the formal power series $G_{F,m}$ is the $m$-th degree
homogeneous component of $G_{F}$. Hence, this $G_{F,m}$ is homogeneous of
degree $m$.

Forget that we fixed $m$. We thus have shown that $G_{F,m}$ is homogeneous of
degree $m$ for each $m\in\mathbb{N}$.

Now, consider the ring $\left(  \mathbf{k}\left[  \left[  x_{1},x_{2}%
,x_{3},\ldots\right]  \right]  \right)  \left[  \left[  t\right]  \right]  $
of formal power series in one indeterminate $t$ over $\mathbf{k}\left[
\left[  x_{1},x_{2},x_{3},\ldots\right]  \right]  $. This ring has a subring
$\Lambda\left[  \left[  t\right]  \right]  $ that consists of those formal
power series whose coefficients belong to $\Lambda$. We consider $\left(
\mathbf{k}\left[  \left[  x_{1},x_{2},x_{3},\ldots\right]  \right]  \right)
\left[  \left[  t\right]  \right]  $ as a topological ring, where the topology
is the standard one induced by the standard topology on $\mathbf{k}\left[
\left[  x_{1},x_{2},x_{3},\ldots\right]  \right]  $ (not the discrete
topology!). This topological ring $\left(  \mathbf{k}\left[  \left[
x_{1},x_{2},x_{3},\ldots\right]  \right]  \right)  \left[  \left[  t\right]
\right]  $ is, of course, isomorphic to $\mathbf{k}\left[  \left[  x_{1}%
,x_{2},x_{3},\ldots,t\right]  \right]  $.

Now, for each $m\in\mathbb{N}$, we know that $G_{F,m}$ is homogeneous of
degree $m$, and therefore we conclude that%
\begin{equation}
G_{F,m}\left(  tx_{1},tx_{2},tx_{3},\ldots\right)  =t^{m}\cdot G_{F,m}
\label{pf.thm.VF.main.c.GFmtxi}%
\end{equation}
(by Lemma \ref{lem.hg-ps-txi}, applied to $u=G_{F,m}$ and $n=m$).

On the other hand,
\[
G_{F}=\sum_{m\in\mathbb{N}}G_{F,m}%
\]
(as we have seen in our proof of Proposition \ref{prop.GF.basics}
\textbf{(a)}). Comparing this with%
\[
G_{F}=\prod_{i=1}^{\infty}F\left(  x_{i}\right)  \ \ \ \ \ \ \ \ \ \ \left(
\text{by Proposition \ref{prop.GF.basics} \textbf{(b)}}\right)  ,
\]
we obtain%
\[
\prod_{i=1}^{\infty}F\left(  x_{i}\right)  =\sum_{m\in\mathbb{N}}G_{F,m}.
\]
Substituting $tx_{1},tx_{2},tx_{3},\ldots$ for $x_{1},x_{2},x_{3},\ldots$ on
both sides of this equality, we obtain%
\begin{align}
\prod_{i=1}^{\infty}F\left(  tx_{i}\right)   &  =\sum_{m\in\mathbb{N}%
}\underbrace{G_{F,m}\left(  tx_{1},tx_{2},tx_{3},\ldots\right)  }%
_{\substack{=t^{m}\cdot G_{F,m}\\\text{(by (\ref{pf.thm.VF.main.c.GFmtxi}))}%
}}\nonumber\\
&  =\sum_{m\in\mathbb{N}}t^{m}\cdot G_{F,m}. \label{pf.thm.VF.main.c.GFtxi}%
\end{align}

The map $V_{F}:\Lambda\rightarrow\Lambda$ is a $\mathbf{k}$-algebra
homomorphism. Hence, it induces a $\mathbf{k}\left[  \left[  t\right]
\right]  $-algebra homomorphism%
\[
V_{F}\left[  \left[  t\right]  \right]  :\Lambda\left[  \left[  t\right]
\right]  \rightarrow\Lambda\left[  \left[  t\right]  \right]
\]
that sends each formal power series $\sum_{n\geq0}a_{n}t^{n}\in\Lambda\left[
\left[  t\right]  \right]  $ (with $a_{n}\in\Lambda$) to $\sum_{n\geq0}%
V_{F}\left(  a_{n}\right)  t^{n}$. Consider this $\mathbf{k}\left[  \left[
t\right]  \right]  $-algebra homomorphism $V_{F}\left[  \left[  t\right]
\right]  $. (Note that $V_{F}\left[  \left[  t\right]  \right]  $ is
continuous with respect to an appropriate topology on $\Lambda\left[  \left[
t\right]  \right]  $, but we shall not use this fact.)

Define the formal power series
\[
H\left(  t\right)  =\prod_{i=1}^{\infty}\left(  1-x_{i}t\right)  ^{-1}%
\in\left(  \mathbf{k}\left[  \left[  x_{1},x_{2},x_{3},\ldots\right]  \right]
\right)  \left[  \left[  t\right]  \right]  .
\]
Then, from \cite[(2.4.1)]{GriRei}, we know that%
\[
H\left(  t\right)  =\sum_{n\geq0}\underbrace{h_{n}\left(  \mathbf{x}\right)
}_{=h_{n}}t^{n}=\sum_{n\geq0}h_{n}t^{n}\in\Lambda\left[  \left[  t\right]
\right]  .
\]
Hence, $\left(  V_{F}\left[  \left[  t\right]  \right]  \right)  \left(
H\left(  t\right)  \right)  $ is well-defined. Moreover, $H\left(  t\right)
=\sum_{n\geq0}h_{n}t^{n}$ shows that the constant term of $H\left(  t\right)
$ is $h_{0}=1$. Thus, $\operatorname*{lder}\left(  H\left(  t\right)  \right)
$ is well-defined.

Applying the map $V_{F}\left[  \left[  t\right]  \right]  $ to both sides of
the equality $H\left(  t\right)  =\sum_{n\geq0}h_{n}t^{n}$, we obtain%
\begin{align*}
\left(  V_{F}\left[  \left[  t\right]  \right]  \right)  \left(  H\left(
t\right)  \right)   &  =\left(  V_{F}\left[  \left[  t\right]  \right]
\right)  \left(  \sum_{n\geq0}h_{n}t^{n}\right)  =\underbrace{\sum_{n\geq0}%
}_{=\sum_{n\in\mathbb{N}}}\underbrace{V_{F}\left(  h_{n}\right)
}_{\substack{=G_{F,n}\\\text{(by Theorem \ref{thm.VF.main} \textbf{(b)}%
,}\\\text{applied to }m=n\text{)}}}t^{n}\\
&  \ \ \ \ \ \ \ \ \ \ \ \ \ \ \ \ \ \ \ \ \left(  \text{by the definition of
}V_{F}\left[  \left[  t\right]  \right]  \right) \\
&  =\sum_{n\in\mathbb{N}}\underbrace{G_{F,n}t^{n}}_{=t^{n}\cdot G_{F,n}}%
=\sum_{n\in\mathbb{N}}t^{n}\cdot G_{F,n}=\sum_{m\in\mathbb{N}}t^{m}\cdot
G_{F,m}%
\end{align*}
(here, we have renamed the summation index $n$ as $m$). Comparing this with
(\ref{pf.thm.VF.main.c.GFtxi}), we find%
\begin{equation}
\left(  V_{F}\left[  \left[  t\right]  \right]  \right)  \left(  H\left(
t\right)  \right)  =\prod_{i=1}^{\infty}F\left(  tx_{i}\right)  .
\label{pf.thm.VF.main.c.GFtxi2}%
\end{equation}

From \cite[Exercise 2.5.21]{GriRei}, we know that%
\[
\sum_{m\geq0}p_{m+1}t^{m}=\dfrac{H^{\prime}\left(  t\right)  }{H\left(
t\right)  }.
\]
On the other hand, the definition of $\operatorname*{lder}\left(  H\left(
t\right)  \right)  $ yields%
\[
\operatorname*{lder}\left(  H\left(  t\right)  \right)  =\dfrac{H^{\prime
}\left(  t\right)  }{H\left(  t\right)  }.
\]
Comparing these two equalities, we obtain%
\[
\operatorname*{lder}\left(  H\left(  t\right)  \right)  =\sum_{m\geq0}%
p_{m+1}t^{m}=\sum_{n\geq0}p_{n+1}t^{n}%
\]
(here, we have renamed the summation index $m$ as $n$). Applying the map
$V_{F}\left[  \left[  t\right]  \right]  $ to both sides of this equality, we
find%
\begin{align}
\left(  V_{F}\left[  \left[  t\right]  \right]  \right)  \left(
\operatorname*{lder}\left(  H\left(  t\right)  \right)  \right)   &  =\left(
V_{F}\left[  \left[  t\right]  \right]  \right)  \left(  \sum_{n\geq0}%
p_{n+1}t^{n}\right)  =\underbrace{\sum_{n\geq0}}_{=\sum_{n\in\mathbb{N}}}%
V_{F}\left(  p_{n+1}\right)  t^{n}\nonumber\\
&  \ \ \ \ \ \ \ \ \ \ \ \ \ \ \ \ \ \ \ \ \left(  \text{by the definition of
}V_{F}\left[  \left[  t\right]  \right]  \right) \nonumber\\
&  =\sum_{n\in\mathbb{N}}V_{F}\left(  p_{n+1}\right)  t^{n}.
\label{pf.thm.VF.main.c.3}%
\end{align}

On the other hand, the constant term of $H\left(  t\right)  $ is $1$ (as we
have shown above). Hence, Proposition \ref{prop.lder.functorial} (applied to
$R=\Lambda$ and $S=\Lambda$ and $\alpha=V_{F}$ and $u=H\left(  t\right)  $)
shows that the constant term of the power series $\left(  V_{F}\left[  \left[
t\right]  \right]  \right)  \left(  H\left(  t\right)  \right)  $ is $1$, and
that we have%
\begin{equation}
\operatorname*{lder}\left(  \left(  V_{F}\left[  \left[  t\right]  \right]
\right)  \left(  H\left(  t\right)  \right)  \right)  =\left(  V_{F}\left[
\left[  t\right]  \right]  \right)  \left(  \operatorname*{lder}\left(
H\left(  t\right)  \right)  \right)  . \label{pf.thm.VF.main.c.4}%
\end{equation}

Now, (\ref{pf.thm.VF.main.c.3}) yields%
\begin{align}
\sum_{n\in\mathbb{N}}V_{F}\left(  p_{n+1}\right)  t^{n}  &  =\left(
V_{F}\left[  \left[  t\right]  \right]  \right)  \left(  \operatorname*{lder}%
\left(  H\left(  t\right)  \right)  \right) \nonumber\\
&  =\operatorname*{lder}\left(  \underbrace{\left(  V_{F}\left[  \left[
t\right]  \right]  \right)  \left(  H\left(  t\right)  \right)  }%
_{\substack{=\prod_{i=1}^{\infty}F\left(  tx_{i}\right)  \\\text{(by
(\ref{pf.thm.VF.main.c.GFtxi2}))}}}\right)  \ \ \ \ \ \ \ \ \ \ \left(
\text{by (\ref{pf.thm.VF.main.c.4})}\right) \nonumber\\
&  =\operatorname*{lder}\left(  \prod_{i=1}^{\infty}F\left(  tx_{i}\right)
\right)  . \label{pf.thm.VF.main.c.5}%
\end{align}

Now, $F\left(  tx_{1}\right)  ,F\left(  tx_{2}\right)  ,F\left(
tx_{3}\right)  ,\ldots$ are infinitely many formal power series in $t$ over
the ring $\mathbf{k}\left[  \left[  x_{1},x_{2},x_{3},\ldots\right]  \right]
$ whose constant terms are $1$\ \ \ \ \footnote{\textit{Proof.} Let
$i\in\left\{  1,2,3,\ldots\right\}  $. We must prove that $F\left(
tx_{i}\right)  $ is a formal power series in $t$ over the ring $\mathbf{k}%
\left[  \left[  x_{1},x_{2},x_{3},\ldots\right]  \right]  $ whose constant
term is $1$.
\par
Recall that $F=\sum_{n\in\mathbb{N}}f_{n}t^{n}$. Substituting $tx_{i}$ for $t$
in this equality, we find $F\left(  tx_{i}\right)  =\sum_{n\in\mathbb{N}}%
f_{n}\underbrace{\left(  tx_{i}\right)  ^{n}}_{=t^{n}x_{i}^{n}=x_{i}^{n}t^{n}%
}=\sum_{n\in\mathbb{N}}f_{n}x_{i}^{n}t^{n}$. Hence, $F\left(  tx_{i}\right)  $
is a formal power series in $t$ over the ring $\mathbf{k}\left[  \left[
x_{1},x_{2},x_{3},\ldots\right]  \right]  $ whose constant term is
$f_{0}\underbrace{x_{i}^{0}}_{=1}=f_{0}=1$. This is exactly what we wanted to
prove. Thus, our proof is complete.}. The infinite product $\prod
_{i=1}^{\infty}F\left(  tx_{i}\right)  $ converges (as we know from
(\ref{pf.thm.VF.main.c.GFtxi})). Hence, Proposition
\ref{prop.lder.lder(infprod)} (applied to $R=\mathbf{k}\left[  \left[
x_{1},x_{2},x_{3},\ldots\right]  \right]  $ and $u_{i}=F\left(  tx_{i}\right)
$) yields that the infinite sum $\sum_{i=1}^{\infty}\operatorname*{lder}%
\left(  F\left(  tx_{i}\right)  \right)  $ converges as well, and that we have%
\[
\operatorname*{lder}\left(  \prod_{i=1}^{\infty}F\left(  tx_{i}\right)
\right)  =\sum_{i=1}^{\infty}\operatorname*{lder}\left(  F\left(
tx_{i}\right)  \right)  .
\]
Hence, (\ref{pf.thm.VF.main.c.5}) becomes%
\begin{align}
\sum_{n\in\mathbb{N}}V_{F}\left(  p_{n+1}\right)  t^{n}  &
=\operatorname*{lder}\left(  \prod_{i=1}^{\infty}F\left(  tx_{i}\right)
\right)  =\sum_{i=1}^{\infty}\operatorname*{lder}\left(  F\left(
\underbrace{tx_{i}}_{=x_{i}t}\right)  \right) \nonumber\\
&  =\sum_{i=1}^{\infty}\underbrace{\operatorname*{lder}\left(  F\left(
x_{i}t\right)  \right)  }_{\substack{=x_{i}\cdot\left(  \operatorname*{lder}%
F\right)  \left(  x_{i}t\right)  \\\text{(by Proposition
\ref{prop.lder.lder(lambdat)},}\\\text{applied to }R=\mathbf{k}\left[  \left[
x_{1},x_{2},x_{3},\ldots\right]  \right]  \text{ and }u=F\text{ and }%
\lambda=x_{i}\text{)}}}\nonumber\\
&  =\sum_{i=1}^{\infty}x_{i}\cdot\left(  \operatorname*{lder}F\right)  \left(
x_{i}t\right)  . \label{pf.thm.VF.main.c.7}%
\end{align}

Now, let $i\in\left\{  1,2,3,\ldots\right\}  $ be arbitrary. The definition of
$\operatorname*{lder}F$ yields%
\[
\operatorname*{lder}F=\dfrac{F^{\prime}}{F}=\sum_{n\in\mathbb{N}}\gamma
_{n}t^{n}.
\]
Substituting $x_{i}t$ for $t$ on both sides of this equality, we obtain%
\begin{equation}
\left(  \operatorname*{lder}F\right)  \left(  x_{i}t\right)  =\sum
_{n\in\mathbb{N}}\gamma_{n}\underbrace{\left(  x_{i}t\right)  ^{n}}%
_{=x_{i}^{n}t^{n}}=\sum_{n\in\mathbb{N}}\gamma_{n}x_{i}^{n}t^{n}.
\label{pf.thm.VF.main.c.8}%
\end{equation}

Forget that we fixed $i$. We thus have proved (\ref{pf.thm.VF.main.c.8}) for
each $i\in\left\{  1,2,3,\ldots\right\}  $. Now, (\ref{pf.thm.VF.main.c.7})
becomes%
\begin{align*}
\sum_{n\in\mathbb{N}}V_{F}\left(  p_{n+1}\right)  t^{n}  &  =\sum
_{i=1}^{\infty}x_{i}\cdot\underbrace{\left(  \operatorname*{lder}F\right)
\left(  x_{i}t\right)  }_{\substack{=\sum_{n\in\mathbb{N}}\gamma_{n}x_{i}%
^{n}t^{n}\\\text{(by (\ref{pf.thm.VF.main.c.8}))}}}=\sum_{i=1}^{\infty}%
x_{i}\cdot\sum_{n\in\mathbb{N}}\gamma_{n}x_{i}^{n}t^{n}=\underbrace{\sum
_{i=1}^{\infty}\ \ \sum_{n\in\mathbb{N}}}_{=\sum_{n\in\mathbb{N}}%
\ \ \sum_{i=1}^{\infty}}\underbrace{x_{i}\gamma_{n}x_{i}^{n}}_{=\gamma
_{n}x_{i}^{n+1}}t^{n}\\
&  =\sum_{n\in\mathbb{N}}\ \ \sum_{i=1}^{\infty}\gamma_{n}x_{i}^{n+1}%
t^{n}=\sum_{n\in\mathbb{N}}\gamma_{n}\underbrace{\left(  \sum_{i=1}^{\infty
}x_{i}^{n+1}\right)  }_{\substack{=x_{1}^{n+1}+x_{2}^{n+1}+x_{3}^{n+1}%
+\cdots\\=p_{n+1}\\\text{(by the definition of }p_{n+1}\text{)}}}t^{n}%
=\sum_{n\in\mathbb{N}}\gamma_{n}p_{n+1}t^{n}.
\end{align*}
Comparing coefficients before $t^{n}$ in this equality, we conclude that%
\begin{equation}
V_{F}\left(  p_{n+1}\right)  =\gamma_{n}p_{n+1}\ \ \ \ \ \ \ \ \ \ \text{for
each }n\in\mathbb{N}. \label{pf.thm.VF.main.c.9}%
\end{equation}

Now, let $n$ be a positive integer. Then, $n-1\in\mathbb{N}$. Hence,
(\ref{pf.thm.VF.main.c.9}) (applied to $n-1$ instead of $n$) yields
\[
V_{F}\left(  p_{\left(  n-1\right)  +1}\right)  =\gamma_{n-1}p_{\left(
n-1\right)  +1}.
\]
In other words, $V_{F}\left(  p_{n}\right)  =\gamma_{n-1}p_{n}$ (since
$\left(  n-1\right)  +1=1$). This proves Theorem \ref{thm.VF.main}
\textbf{(c)}.
\end{proof}
\end{verlong}

Our next (and last) few results are not generalizations of any properties of
Petrie functions. To state them, we take a somewhat more high-level point of
view. We forget that we fixed the power series $F$. Instead, for
\textbf{every} power series $F\in\mathbf{k}\left[  \left[  t\right]  \right]
$ whose constant term is $1$, we define a power series $G_{F}$ according to
Definition \ref{def.GF.GF} \textbf{(d)}. Moreover, for \textbf{every} power
series $F\in\mathbf{k}\left[  \left[  t\right]  \right]  $ whose constant term
is $1$, and for \textbf{every} $m\in\mathbb{N}$, we define a power series
$G_{F,m}$ according to Definition \ref{def.GF.GF} \textbf{(e)}. We then have
the following:

\begin{proposition}
\label{prop.GF.GAB1}Let $A$ and $B$ be two power series in $\mathbf{k}\left[
\left[  t\right]  \right]  $ whose constant terms are $1$. Then:

\textbf{(a)} We have $G_{AB}=G_{A}G_{B}$.

\textbf{(b)} Let $n\in\mathbb{N}$. We have $G_{AB,n}=\sum_{i=0}^{n}%
G_{A,i}G_{B,n-i}$.
\end{proposition}

\begin{vershort}
\begin{proof}
[Proof of Proposition \ref{prop.GF.GAB1} (sketched).]The power series $AB$ has
constant term $1$ (since $A$ and $B$ have constant term $1$). Thus, $G_{AB}$
is well-defined, as is $G_{AB,n}$ for each $n\in\mathbb{N}$.

\textbf{(a)} Proposition \ref{prop.GF.basics} \textbf{(b)} yields that
$G_{F}=\prod_{i=1}^{\infty}F\left(  x_{i}\right)  $ for any power series
$F\in\mathbf{k}\left[  \left[  t\right]  \right]  $ whose constant term is
$1$. Applying this to $F=A$ and to $F=B$ and to $F=AB$ yields $G_{A}%
=\prod_{i=1}^{\infty}A\left(  x_{i}\right)  $ and $G_{B}=\prod_{i=1}^{\infty
}B\left(  x_{i}\right)  $ and
\[
G_{AB}=\prod_{i=1}^{\infty}\underbrace{\left(  AB\right)  \left(
x_{i}\right)  }_{=A\left(  x_{i}\right)  B\left(  x_{i}\right)  }=\prod
_{i=1}^{\infty}\left(  A\left(  x_{i}\right)  B\left(  x_{i}\right)  \right)
=\underbrace{\left(  \prod_{i=1}^{\infty}A\left(  x_{i}\right)  \right)
}_{=G_{A}}\underbrace{\left(  \prod_{i=1}^{\infty}B\left(  x_{i}\right)
\right)  }_{=G_{B}}=G_{A}G_{B}.
\]
This proves Proposition \ref{prop.GF.GAB1} \textbf{(a)}.

\textbf{(b)} Proposition \ref{prop.GF.GAB1} \textbf{(a)} yields $G_{AB}%
=G_{A}G_{B}$. Thus, the $n$-th degree homogeneous components of $G_{AB}$ and
of $G_{A}G_{B}$ are equal. But this means precisely that $G_{AB,n}=\sum
_{i=0}^{n}G_{A,i}G_{B,n-i}$ (by Proposition \ref{prop.GF.basics}
\textbf{(a)}). This proves Proposition \ref{prop.GF.GAB1} \textbf{(b)}.
\end{proof}
\end{vershort}

\begin{verlong}
\begin{proof}
[Proof of Proposition \ref{prop.GF.GAB1}.]The constant terms of the power
series $A$ and $B$ are $1$ (by assumption). Hence, the constant term of the
power series $AB$ is $1$ as well (since the constant term of the product of
two power series equals the product of their constant terms). Thus, $G_{AB}$
is well-defined.

\textbf{(a)} Proposition \ref{prop.GF.basics} \textbf{(b)} yields that%
\begin{equation}
G_{F}=\prod_{i=1}^{\infty}F\left(  x_{i}\right)  \label{pf.prop.GF.GAB1.a.1}%
\end{equation}
for any power series $F\in\mathbf{k}\left[  \left[  t\right]  \right]  $ whose
constant term is $1$.

Recall that the constant term of the power series $AB$ is $1$. Hence,
(\ref{pf.prop.GF.GAB1.a.1}) (applied to $F=AB$) yields%
\begin{equation}
G_{AB}=\prod_{i=1}^{\infty}\underbrace{\left(  AB\right)  \left(
x_{i}\right)  }_{=A\left(  x_{i}\right)  B\left(  x_{i}\right)  }=\prod
_{i=1}^{\infty}\left(  A\left(  x_{i}\right)  B\left(  x_{i}\right)  \right)
. \label{pf.prop.GF.GAB1.a.2}%
\end{equation}

On the other hand, (\ref{pf.prop.GF.GAB1.a.1}) (applied to $F=A$) yields%
\[
G_{A}=\prod_{i=1}^{\infty}A\left(  x_{i}\right)  .
\]
Moreover, (\ref{pf.prop.GF.GAB1.a.1}) (applied to $F=B$) yields%
\[
G_{B}=\prod_{i=1}^{\infty}B\left(  x_{i}\right)  .
\]
Multiplying these two equalities, we find%
\[
G_{A}G_{B}=\left(  \prod_{i=1}^{\infty}A\left(  x_{i}\right)  \right)  \left(
\prod_{i=1}^{\infty}B\left(  x_{i}\right)  \right)  =\prod_{i=1}^{\infty
}\left(  A\left(  x_{i}\right)  B\left(  x_{i}\right)  \right)  .
\]
Comparing this with (\ref{pf.prop.GF.GAB1.a.2}), we find $G_{AB}=G_{A}G_{B}$.
This proves Proposition \ref{prop.GF.GAB1} \textbf{(a)}.

\textbf{(b)} Forget that we fixed $n$.

In the proof of Proposition \ref{prop.GF.basics} \textbf{(a)}, we have shown
that%
\begin{equation}
G_{F}=\sum_{m\in\mathbb{N}}G_{F,m} \label{pf.prop.GF.GAB1.b.1}%
\end{equation}
for any power series $F\in\mathbf{k}\left[  \left[  t\right]  \right]  $ whose
constant term is $1$. (Indeed, this is precisely the equality
(\ref{pf.prop.GF.basics.a.GF=sum}).)

Recall that the constant term of the power series $A$ is $1$. Hence,
(\ref{pf.prop.GF.GAB1.b.1}) (applied to $F=A$) yields%
\begin{equation}
G_{A}=\sum_{m\in\mathbb{N}}G_{A,m}=\sum_{i\in\mathbb{N}}G_{A,i}
\label{pf.prop.GF.GAB1.b.GA=}%
\end{equation}
(here, we have renamed the summation index $m$ as $i$).

Furthermore, recall that the constant term of the power series $B$ is $1$.
Hence, (\ref{pf.prop.GF.GAB1.b.1}) (applied to $F=B$) yields%
\begin{equation}
G_{B}=\sum_{m\in\mathbb{N}}G_{B,m}=\sum_{j\in\mathbb{N}}G_{B,j}
\label{pf.prop.GF.GAB1.b.GB=}%
\end{equation}
(here, we have renamed the summation index $m$ as $j$).

Now, Proposition \ref{prop.GF.GAB1} \textbf{(a)} yields
\begin{align}
G_{AB}  &  =G_{A}G_{B}=\left(  \sum_{i\in\mathbb{N}}G_{A,i}\right)  \left(
\sum_{j\in\mathbb{N}}G_{B,j}\right)  \ \ \ \ \ \ \ \ \ \ \left(  \text{by
(\ref{pf.prop.GF.GAB1.b.GA=}) and (\ref{pf.prop.GF.GAB1.b.GB=})}\right)
\nonumber\\
&  =\underbrace{\sum_{i\in\mathbb{N}}\ \ \sum_{j\in\mathbb{N}}}%
_{\substack{=\sum_{\left(  i,j\right)  \in\mathbb{N}\times\mathbb{N}}%
\\=\sum_{n\in\mathbb{N}}\ \ \sum_{\substack{\left(  i,j\right)  \in
\mathbb{N}\times\mathbb{N};\\i+j=n}}}}G_{A,i}G_{B,j}=\sum_{n\in\mathbb{N}%
}\ \ \sum_{\substack{\left(  i,j\right)  \in\mathbb{N}\times\mathbb{N}%
;\\i+j=n}}G_{A,i}G_{B,j}. \label{pf.prop.GF.GAB1.b.GAB=}%
\end{align}

If $n\in\mathbb{N}$, then the power series $\sum_{\substack{\left(
i,j\right)  \in\mathbb{N}\times\mathbb{N};\\i+j=n}}G_{A,i}G_{B,j}\in
\mathbf{k}\left[  \left[  x_{1},x_{2},x_{3},\ldots\right]  \right]  $ is
homogeneous of degree $n$\ \ \ \ \footnote{\textit{Proof.} Let $n\in
\mathbb{N}$. We must prove that the power series $\sum_{\substack{\left(
i,j\right)  \in\mathbb{N}\times\mathbb{N};\\i+j=n}}G_{A,i}G_{B,j}$ is
homogeneous of degree $n$.
\par
Let $\left(  i,j\right)  \in\mathbb{N}\times\mathbb{N}$ be such that $i+j=n$.
Then, Proposition \ref{prop.GF.basics} \textbf{(a)} (applied to $m=i$ and
$F=A$) shows that the formal power series $G_{A,i}$ is the $i$-th degree
homogeneous component of $G_{A}$. Hence, this formal power series $G_{A,i}$ is
homogeneous of degree $i$.
\par
Moreover, Proposition \ref{prop.GF.basics} \textbf{(a)} (applied to $m=j$ and
$F=B$) shows that the formal power series $G_{B,j}$ is the $j$-th degree
homogeneous component of $G_{B}$. Hence, this formal power series $G_{B,j}$ is
homogeneous of degree $j$.
\par
Now we have shown that the two power series $G_{A,i}$ and $G_{B,j}$ are
homogeneous of degrees $i$ and $j$, respectively. Thus, their product
$G_{A,i}G_{B,j}$ is homogeneous of degree $i+j$. In other words,
$G_{A,i}G_{B,j}$ is homogeneous of degree $n$ (since $i+j=n$).
\par
Forget that we fixed $\left(  i,j\right)  $. We thus have shown that
$G_{A,i}G_{B,j}$ is homogeneous of degree $n$ whenever $\left(  i,j\right)
\in\mathbb{N}\times\mathbb{N}$ satisfies $i+j=n$. In other words, each addend
of the sum $\sum_{\substack{\left(  i,j\right)  \in\mathbb{N}\times
\mathbb{N};\\i+j=n}}G_{A,i}G_{B,j}$ is homogeneous of degree $n$. Hence, the
entire sum $\sum_{\substack{\left(  i,j\right)  \in\mathbb{N}\times
\mathbb{N};\\i+j=n}}G_{A,i}G_{B,j}$ is homogeneous of degree $n$ as well. This
completes our proof.}. Thus, the equality (\ref{pf.prop.GF.GAB1.b.GAB=})
reveals that the family
\[
\left(  \sum_{\substack{\left(  i,j\right)  \in\mathbb{N}\times\mathbb{N}%
;\\i+j=n}}G_{A,i}G_{B,j}\right)  _{n\in\mathbb{N}}%
\]
is the homogeneous decomposition of $G_{AB}$ (by the definition of a
homogeneous decomposition). Hence, for each $n\in\mathbb{N}$, we have%
\begin{align}
&  \sum_{\substack{\left(  i,j\right)  \in\mathbb{N}\times\mathbb{N}%
;\\i+j=n}}G_{A,i}G_{B,j}\nonumber\\
&  =\left(  \text{the }n\text{-th degree homogeneous component of }%
G_{AB}\right)  . \label{pf.prop.GF.GAB1.b.hgc1}%
\end{align}

Now, let $n\in\mathbb{N}$. Recall that the constant term of the power series
$AB$ is $1$. Hence, Proposition \ref{prop.GF.basics} \textbf{(a)} (applied to
$F=AB$ and $m=n$) yields that the formal power series $G_{AB,n}$ is the $n$-th
degree homogeneous component of $G_{AB}$. In other words,%
\[
G_{AB,n}=\left(  \text{the }n\text{-th degree homogeneous component of }%
G_{AB}\right)  .
\]
Comparing this with (\ref{pf.prop.GF.GAB1.b.hgc1}), we obtain%
\begin{align*}
G_{AB,n}  &  =\sum_{\substack{\left(  i,j\right)  \in\mathbb{N}\times
\mathbb{N};\\i+j=n}}G_{A,i}G_{B,j}=\sum_{i\in\left\{  0,1,\ldots,n\right\}
}G_{A,i}G_{B,n-i}\\
&  \ \ \ \ \ \ \ \ \ \ \ \ \ \ \ \ \ \ \ \ \left(
\begin{array}
[c]{c}%
\text{here, we have substituted }\left(  i,n-i\right)  \text{ for }\left(
i,j\right)  \text{ in the sum,}\\
\text{since the map }\left\{  0,1,\ldots,n\right\}  \rightarrow\left\{
\left(  i,j\right)  \in\mathbb{N}\times\mathbb{N}\ \mid\ i+j=n\right\} \\
\text{that sends each }i\text{ to }\left(  i,n-i\right)  \text{ is a
bijection}%
\end{array}
\right) \\
&  =\sum_{i=0}^{n}G_{A,i}G_{B,n-i}.
\end{align*}
This proves Proposition \ref{prop.GF.GAB1} \textbf{(b)}.
\end{proof}
\end{verlong}

Finally, we can express the image of the symmetric function $G_{F,n}$ under
the antipode of $\Lambda$ (a result suggested by Sasha Postnikov):

\begin{theorem}
\label{thm.GF.antipode}Let $S$ be the antipode of the Hopf algebra $\Lambda$.
Let $F\in\mathbf{k}\left[  \left[  t\right]  \right]  $ be a formal power
series whose constant term is $1$. Then, for each $n\in\mathbb{N}$, we have%
\begin{equation}
S\left(  G_{F,n}\right)  =G_{F^{-1},n}. \label{eq.thm.GF.antipode.eq}%
\end{equation}

\end{theorem}

\begin{vershort}
\begin{proof}
[Proof of Theorem \ref{thm.GF.antipode} (sketched).]Let $\Delta$ and
$\varepsilon$ be the comultiplication and the counit of the Hopf algebra
$\Lambda$. Let $\eta:\mathbf{k}\rightarrow\Lambda$ be the map that sends each
$u\in\mathbf{k}$ to $u\cdot1_{\Lambda}\in\Lambda$. It is easy to see that each
positive integer $n$ satisfies%
\begin{equation}
\varepsilon\left(  G_{F,n}\right)  =0. \label{pf.thm.GF.antipode.short.eG=0}%
\end{equation}
(Indeed, $\varepsilon$ sends each homogeneous symmetric function of positive
degree to $0$; but $G_{F,n}$ is a homogeneous symmetric function of degree
$n$.) Also, Proposition \ref{prop.GF.basics} \textbf{(e)} yields $G_{F,0}=1$
and thus $\varepsilon\left(  G_{F,0}\right)  =1$.

We shall use the convolution $\star$ introduced in Definition
\ref{def.convolution}. The antipode $S$ of $\Lambda$ is the $\star$-inverse of
the map $\operatorname*{id}\nolimits_{\Lambda}:\Lambda\rightarrow\Lambda$ (by
the definition of the antipode of a Hopf algebra). In other words,%
\[
S\star\operatorname*{id}\nolimits_{\Lambda}=\operatorname*{id}%
\nolimits_{\Lambda}\star S=\eta\circ\varepsilon
\]
(since $\eta\circ\varepsilon:\Lambda\rightarrow\Lambda$ is the neutral element
with respect to $\star$). We also have $S\left(  1\right)  =1$ (by one of the
fundamental properties of the antipode of a Hopf algebra).

Now, for each $n\in\mathbb{N}$, we have%
\[
\Delta\left(  G_{F,n}\right)  =\sum_{i=0}^{n}G_{F,i}\otimes G_{F,n-i}%
\]
(by Theorem \ref{thm.DeltaGFm}) and therefore%
\begin{align*}
\left(  S\star\operatorname*{id}\nolimits_{\Lambda}\right)  \left(
G_{F,n}\right)   &  =\sum_{i=0}^{n}S\left(  G_{F,i}\right)  \cdot
\underbrace{\operatorname*{id}\left(  G_{F,n-i}\right)  }_{=G_{F,n-i}%
}\ \ \ \ \ \ \ \ \ \ \left(  \text{by the definition of convolution}\right) \\
&  =\sum_{i=0}^{n}S\left(  G_{F,i}\right)  \cdot G_{F,n-i},
\end{align*}
so that%
\begin{align}
\sum_{i=0}^{n}S\left(  G_{F,i}\right)  \cdot G_{F,n-i}  &
=\underbrace{\left(  S\star\operatorname*{id}\nolimits_{\Lambda}\right)
}_{=\eta\circ\varepsilon}\left(  G_{F,n}\right)  =\left(  \eta\circ
\varepsilon\right)  \left(  G_{F,n}\right) \nonumber\\
&  =\left[  n=0\right]  \label{pf.thm.GF.antipode.short.SGG}%
\end{align}
(the last equality sign here follows easily from
(\ref{pf.thm.GF.antipode.short.eG=0}) and from $\varepsilon\left(
G_{F,0}\right)  =1$).

On the other hand, the constant term of the power series $F^{-1}$ is $1$
(since the constant term of $F$ is $1$). Hence, $G_{F^{-1},n}$ is well-defined
for each $n\in\mathbb{N}$.

For each $n\in\mathbb{N}$, we have%
\[
G_{F^{-1}F,n}=\sum_{i=0}^{n}G_{F^{-1},i}G_{F,n-i}%
\]
(by Proposition \ref{prop.GF.GAB1} \textbf{(b)}, applied to $A=F^{-1}$ and
$B=F$) and thus%
\begin{equation}
\sum_{i=0}^{n}G_{F^{-1},i}G_{F,n-i}=G_{F^{-1}F,n}=G_{1,n}=\left[  n=0\right]
\label{pf.thm.GF.antipode.short.GG}%
\end{equation}
(the last equality sign here has been shown in Example \ref{exa.GF.GF123}
\textbf{(b)}).

Recall that $G_{F,0}=1$. Hence, the equalities
(\ref{pf.thm.GF.antipode.short.GG}) (for all $n\in\mathbb{N}$) can be
recursively solved for $G_{F^{-1},0},G_{F^{-1},1},G_{F^{-1},2},\ldots$
(starting with $G_{F,0},G_{F,1},G_{F,2},\ldots$); we obtain%
\[
G_{F^{-1},n}=\left[  n=0\right]  -\sum_{i=0}^{n-1}G_{F^{-1},i}G_{F,n-i}%
\ \ \ \ \ \ \ \ \ \ \text{for each }n\in\mathbb{N}.
\]
The same argument, but using the equalities
(\ref{pf.thm.GF.antipode.short.SGG}) instead of
(\ref{pf.thm.GF.antipode.short.GG}), yields%
\[
S\left(  G_{F,n}\right)  =\left[  n=0\right]  -\sum_{i=0}^{n-1}S\left(
G_{F,i}\right)  \cdot G_{F,n-i}\ \ \ \ \ \ \ \ \ \ \text{for each }%
n\in\mathbb{N}.
\]
Comparing these two recursive formulas for $G_{F^{-1},n}$ and $S\left(
G_{F,n}\right)  $, we see that they are the same. Thus, by strong induction on
$n$, we conclude that%
\[
S\left(  G_{F,n}\right)  =G_{F^{-1},n}\ \ \ \ \ \ \ \ \ \ \text{for each }%
n\in\mathbb{N}\text{.}%
\]
This completes the proof of Theorem \ref{thm.GF.antipode}.
\end{proof}
\end{vershort}

\begin{verlong}
Our proof of this theorem will rely on the following simple lemma\footnote{We
are using Convention \ref{conv.iverson} again.}:

\begin{lemma}
\label{lem.GF.G1n}For any $n\in\mathbb{N}$, we have $G_{1,n}=\left[
n=0\right]  $. (Here, the \textquotedblleft$1$\textquotedblright\ in
\textquotedblleft$G_{1,n}$\textquotedblright\ means the constant power series
$1\in\mathbf{k}\left[  \left[  t\right]  \right]  $.)
\end{lemma}

\begin{proof}
[Proof of Lemma \ref{lem.GF.G1n}.]Let $F$ be the constant power series
$1\in\mathbf{k}\left[  \left[  t\right]  \right]  $. Then, the constant term
of $F$ is $1$; thus, $G_{F,n}$ is well-defined for each $n\in\mathbb{N}$.

We shall use the notations introduced in Definition \ref{def.GF.GF}. Thus,
$F=\sum_{n\in\mathbb{N}}f_{n}t^{n}$. On the other hand,%
\[
\sum_{n\in\mathbb{N}}\left[  n=0\right]  t^{n}=\underbrace{\left[  0=0\right]
}_{\substack{=1\\\text{(since }0=0\text{)}}}\underbrace{t^{0}}_{=1}%
+\sum_{\substack{n\in\mathbb{N};\\n\neq0}}\underbrace{\left[  n=0\right]
}_{\substack{=0\\\text{(since }n\neq0\text{)}}}t^{n}=1+\underbrace{\sum
_{\substack{n\in\mathbb{N};\\n\neq0}}0t^{n}}_{=0}=1=F
\]
(since $F=1$), so that%
\[
\sum_{n\in\mathbb{N}}\left[  n=0\right]  t^{n}=F=\sum_{n\in\mathbb{N}}%
f_{n}t^{n}.
\]
This is an equality of power series. Comparing coefficients in front of
$t^{n}$ in this equality, we thus obtain%
\begin{equation}
\left[  n=0\right]  =f_{n}\ \ \ \ \ \ \ \ \ \ \text{for each }n\in
\mathbb{N}\text{.} \label{pf.lem.GF.G1n.1}%
\end{equation}

Now, let $n\in\mathbb{N}$. We must prove that $G_{1,n}=\left[  n=0\right]  $.
If $n=0$, then this is obvious\footnote{\textit{Proof.} Assume that $n=0$.
Thus, $G_{1,n}=G_{1,0}=1$ (by Proposition \ref{prop.GF.basics} \textbf{(e)},
applied to $1$ instead of $F$). However, from $n=0$, we obtain $\left[
n=0\right]  =1$. Comparing this with $G_{1,n}=1$, we obtain $G_{1,n}=\left[
n=0\right]  $. Thus, we have proved that $G_{1,n}=\left[  n=0\right]  $ under
the assumption that $n=0$.}. Thus, for the rest of this proof, we WLOG assume
that $n\neq0$. Hence, we don't have $n=0$. Thus, we have $\left[  n=0\right]
=0$. On the other hand, we have
\begin{equation}
f_{\alpha}=0\ \ \ \ \ \ \ \ \ \ \text{for any }\alpha\in\operatorname*{WC}%
\text{ satisfying }\left\vert \alpha\right\vert =n \label{pf.lem.GF.G1n.3}%
\end{equation}
\footnote{\textit{Proof of (\ref{pf.lem.GF.G1n.3}):} Let $\alpha
\in\operatorname*{WC}$ satisfy $\left\vert \alpha\right\vert =n$.
\par
If we had $\alpha=\varnothing$, then we would have $\left\vert \alpha
\right\vert =\left\vert \varnothing\right\vert =0$, which would contradict
$\left\vert \alpha\right\vert =n\neq0$. Hence, we cannot have $\alpha
=\varnothing$. Thus, we have $\alpha\neq\varnothing$. Now, $\alpha=\left(
\alpha_{1},\alpha_{2},\alpha_{3},\ldots\right)  $, so that%
\[
\left(  \alpha_{1},\alpha_{2},\alpha_{3},\ldots\right)  =\alpha\neq
\varnothing=\left(  0,0,0,\ldots\right)  .
\]
In other words, there exists some $i\in\left\{  1,2,3,\ldots\right\}  $ such
that $\alpha_{i}\neq0$. Consider this $i$.
\par
We have $\alpha_{i}\neq0$. Thus, we don't have $\alpha_{i}=0$. Hence, we have
$\left[  \alpha_{i}=0\right]  =0$. Now, (\ref{pf.lem.GF.G1n.1}) (applied to
$\alpha_{i}$ instead of $n$) yields $\left[  \alpha_{i}=0\right]
=f_{\alpha_{i}}$. Hence, $f_{\alpha_{i}}=\left[  \alpha_{i}=0\right]  =0$.
Thus, we have shown that $f_{\alpha_{i}}$ is equal to $0$.
\par
Now, the definition of $f_{\alpha}$ yields
\begin{align*}
f_{\alpha}  &  =f_{\alpha_{1}}f_{\alpha_{2}}f_{\alpha_{3}}\cdots
=\underbrace{\left(  f_{\alpha_{1}}f_{\alpha_{2}}\cdots f_{\alpha_{i}}\right)
}_{=\left(  f_{\alpha_{1}}f_{\alpha_{2}}\cdots f_{\alpha_{i-1}}\right)
f_{\alpha_{i}}}\left(  f_{\alpha_{i+1}}f_{\alpha_{i+2}}f_{\alpha_{i+3}}%
\cdots\right) \\
&  =\left(  f_{\alpha_{1}}f_{\alpha_{2}}\cdots f_{\alpha_{i-1}}\right)
\underbrace{f_{\alpha_{i}}}_{=0}\left(  f_{\alpha_{i+1}}f_{\alpha_{i+2}%
}f_{\alpha_{i+3}}\cdots\right)  =\left(  f_{\alpha_{1}}f_{\alpha_{2}}\cdots
f_{\alpha_{i-1}}\right)  0\left(  f_{\alpha_{i+1}}f_{\alpha_{i+2}}%
f_{\alpha_{i+3}}\cdots\right)  =0.
\end{align*}
This proves (\ref{pf.lem.GF.G1n.3}).}. Now, the definition of $G_{F,n}$ yields%
\[
G_{F,n}=\sum_{\substack{\alpha\in\operatorname*{WC};\\\left\vert
\alpha\right\vert =n}}\underbrace{f_{\alpha}}_{\substack{=0\\\text{(by
(\ref{pf.lem.GF.G1n.3}))}}}\mathbf{x}^{\alpha}=\sum_{\substack{\alpha
\in\operatorname*{WC};\\\left\vert \alpha\right\vert =n}}0\mathbf{x}^{\alpha
}=0.
\]
In view of $F=1$, this rewrites as $G_{1,n}=0$. Comparing this with $\left[
n=0\right]  =0$, we obtain $G_{1,n}=\left[  n=0\right]  $. Thus, Lemma
\ref{lem.GF.G1n} is proven.
\end{proof}

\begin{proof}
[Proof of Theorem \ref{thm.GF.antipode}.]We shall use the convolution $\star$
introduced in Definition \ref{def.convolution}.

Let $\Delta$ and $\varepsilon$ be the comultiplication and the counit of the
Hopf algebra $\Lambda$. Let $\eta:\mathbf{k}\rightarrow\Lambda$ be the map
that sends each $u\in\mathbf{k}$ to $u\cdot1_{\Lambda}\in\Lambda$. Let
$m_{\Lambda}:\Lambda\otimes\Lambda\rightarrow\Lambda$ be the $\mathbf{k}%
$-linear map sending each pure tensor $a\otimes b\in\Lambda\otimes\Lambda$ to
$ab\in\Lambda$. Definition \ref{def.convolution} then yields
\begin{equation}
S\star\operatorname*{id}\nolimits_{\Lambda}=m_{\Lambda}\circ\left(
S\otimes\operatorname*{id}\nolimits_{\Lambda}\right)  \circ\Delta.
\label{pf.thm.GF.antipode.Sid1}%
\end{equation}

It is easy to see that each positive integer $n$ satisfies%
\begin{equation}
\varepsilon\left(  G_{F,n}\right)  =0. \label{pf.thm.GF.antipode.eG=0}%
\end{equation}

[\textit{Proof of (\ref{pf.thm.GF.antipode.eG=0}):} Let $n$ be a positive
integer. Proposition \ref{prop.GF.basics} \textbf{(a)} (applied to $m=n$)
yields that the formal power series $G_{F,n}$ is the $n$-th degree homogeneous
component of $G_{F}$. Hence, this power series $G_{F,n}$ is homogeneous of
degree $n$. Thus, $G_{F,n}$ is homogeneous of positive degree (since $n$ is
positive). Also, $G_{F,n}\in\Lambda$ (by Proposition \ref{prop.GF.basics}
\textbf{(c)}, applied to $m=n$).

Recall that the counit $\varepsilon$ of $\Lambda$ sends every homogeneous
symmetric function of positive degree to $0$. In other words, if $u\in\Lambda$
is homogeneous of positive degree, then $\varepsilon\left(  u\right)  =0$. We
can apply this to $u=G_{F,n}$ (since $G_{F,n}$ is homogeneous of positive
degree), and thus obtain $\varepsilon\left(  G_{F,n}\right)  =0$. This proves
(\ref{pf.thm.GF.antipode.eG=0}).]

Recall that the antipode of a Hopf algebra is defined to be the $\star
$-inverse of its identity map (i.e., to be the inverse of its identity map
with respect to the convolution $\star$). Thus, the antipode $S$ of $\Lambda$
is the $\star$-inverse of the map $\operatorname*{id}\nolimits_{\Lambda
}:\Lambda\rightarrow\Lambda$. In other words,%
\begin{equation}
S\star\operatorname*{id}\nolimits_{\Lambda}=\operatorname*{id}%
\nolimits_{\Lambda}\star S=\eta\circ\varepsilon\label{pf.thm.GF.antipode.Sid2}%
\end{equation}
(since $\eta\circ\varepsilon:\Lambda\rightarrow\Lambda$ is the neutral element
with respect to $\star$). We also have $S\left(  1\right)  =1$ (by one of the
fundamental properties of the antipode of a Hopf algebra\footnote{See, e.g.,
\cite[Proposition 1.4.10]{GriRei} for this property.}).

Now, each positive integer $n$ satisfies%
\begin{equation}
S\left(  G_{F,n}\right)  =-\sum_{i=0}^{n-1}S\left(  G_{F,i}\right)  \cdot
G_{F,n-i}. \label{pf.thm.GF.antipode.SG1}%
\end{equation}

[\textit{Proof of (\ref{pf.thm.GF.antipode.SG1}):} Let $n$ be a positive
integer. Then,
\begin{align*}
&  \underbrace{\left(  S\star\operatorname*{id}\nolimits_{\Lambda}\right)
}_{\substack{=m_{\Lambda}\circ\left(  S\otimes\operatorname*{id}%
\nolimits_{\Lambda}\right)  \circ\Delta\\\text{(by
(\ref{pf.thm.GF.antipode.Sid1}))}}}\left(  G_{F,n}\right) \\
&  =\left(  m_{\Lambda}\circ\left(  S\otimes\operatorname*{id}%
\nolimits_{\Lambda}\right)  \circ\Delta\right)  \left(  G_{F,n}\right)
=m_{\Lambda}\left(  \left(  S\otimes\operatorname*{id}\nolimits_{\Lambda
}\right)  \left(  \underbrace{\Delta\left(  G_{F,n}\right)  }_{\substack{=\sum
_{i=0}^{n}G_{F,i}\otimes G_{F,n-i}\\\text{(by Theorem \ref{thm.DeltaGFm}%
,}\\\text{applied to }m=n\text{)}}}\right)  \right) \\
&  =m_{\Lambda}\left(  \underbrace{\left(  S\otimes\operatorname*{id}%
\nolimits_{\Lambda}\right)  \left(  \sum_{i=0}^{n}G_{F,i}\otimes
G_{F,n-i}\right)  }_{\substack{=\sum_{i=0}^{n}\left(  S\otimes
\operatorname*{id}\nolimits_{\Lambda}\right)  \left(  G_{F,i}\otimes
G_{F,n-i}\right)  \\\text{(since the map }S\otimes\operatorname*{id}%
\nolimits_{\Lambda}\text{ is }\mathbf{k}\text{-linear)}}}\right)  =m_{\Lambda
}\left(  \sum_{i=0}^{n}\underbrace{\left(  S\otimes\operatorname*{id}%
\nolimits_{\Lambda}\right)  \left(  G_{F,i}\otimes G_{F,n-i}\right)
}_{=S\left(  G_{F,i}\right)  \otimes\operatorname*{id}\nolimits_{\Lambda
}\left(  G_{F,n-i}\right)  }\right) \\
&  =m_{\Lambda}\left(  \sum_{i=0}^{n}S\left(  G_{F,i}\right)  \otimes
\underbrace{\operatorname*{id}\nolimits_{\Lambda}\left(  G_{F,n-i}\right)
}_{=G_{F,n-i}}\right)  =m_{\Lambda}\left(  \sum_{i=0}^{n}S\left(
G_{F,i}\right)  \otimes G_{F,n-i}\right) \\
&  =\sum_{i=0}^{n}\underbrace{m_{\Lambda}\left(  S\left(  G_{F,i}\right)
\otimes G_{F,n-i}\right)  }_{\substack{=S\left(  G_{F,i}\right)  \cdot
G_{F,n-i}\\\text{(by the definition of the map }m_{\Lambda}\text{)}%
}}\ \ \ \ \ \ \ \ \ \ \left(  \text{since the map }m_{\Lambda}\text{ is
}\mathbf{k}\text{-linear}\right) \\
&  =\sum_{i=0}^{n}S\left(  G_{F,i}\right)  \cdot G_{F,n-i}=\sum_{i=0}%
^{n-1}S\left(  G_{F,i}\right)  \cdot G_{F,n-i}+S\left(  G_{F,n}\right)
\cdot\underbrace{G_{F,n-n}}_{\substack{=G_{F,0}=1\\\text{(by Proposition
\ref{prop.GF.basics} \textbf{(e)})}}}\\
&  \ \ \ \ \ \ \ \ \ \ \ \ \ \ \ \ \ \ \ \ \left(  \text{here, we have split
off the addend for }i=n\text{ from the sum}\right) \\
&  =\sum_{i=0}^{n-1}S\left(  G_{F,i}\right)  \cdot G_{F,n-i}+S\left(
G_{F,n}\right)  .
\end{align*}
Thus,%
\begin{align*}
&  \sum_{i=0}^{n-1}S\left(  G_{F,i}\right)  \cdot G_{F,n-i}+S\left(
G_{F,n}\right) \\
&  =\underbrace{\left(  S\star\operatorname*{id}\nolimits_{\Lambda}\right)
}_{\substack{=\eta\circ\varepsilon\\\text{(by (\ref{pf.thm.GF.antipode.Sid2}%
))}}}\left(  G_{F,n}\right)  =\left(  \eta\circ\varepsilon\right)  \left(
G_{F,n}\right)  =\eta\left(  \varepsilon\left(  G_{F,n}\right)  \right) \\
&  =\underbrace{\varepsilon\left(  G_{F,n}\right)  }_{\substack{=0\\\text{(by
(\ref{pf.thm.GF.antipode.eG=0}))}}}\cdot1_{\Lambda}\ \ \ \ \ \ \ \ \ \ \left(
\text{by the definition of }\eta\right) \\
&  =0,
\end{align*}
so that%
\[
S\left(  G_{F,n}\right)  =-\sum_{i=0}^{n-1}S\left(  G_{F,i}\right)  \cdot
G_{F,n-i}.
\]
This proves (\ref{pf.thm.GF.antipode.SG1}).]

The map from $\mathbf{k}\left[  \left[  t\right]  \right]  $ to $\mathbf{k}$
that sends each power series to its constant term is a $\mathbf{k}$-algebra
homomorphism. Thus, the constant term of the power series $F^{-1}$ is the
reciprocal of the constant term of $F$. Since the constant term of $F$ is $1$,
we thus conclude that the constant term of the power series $F^{-1}$ is the
reciprocal of $1$. In other words, the constant term of the power series
$F^{-1}$ is $1$.

We must prove (\ref{eq.thm.GF.antipode.eq}) for each $n\in\mathbb{N}$. We
shall do this by strong induction on $n$:

\textit{Induction step:} Let $m\in\mathbb{N}$. Assume (as the induction
hypothesis) that (\ref{eq.thm.GF.antipode.eq}) holds for all $n<m$. We must
now prove that (\ref{eq.thm.GF.antipode.eq}) holds for $n=m$. In other words,
we must prove that $S\left(  G_{F,m}\right)  =G_{F^{-1},m}$. If $m=0$, then
this is obvious\footnote{\textit{Proof.} Assume that $m=0$. Thus,
$G_{F,m}=G_{F,0}=1$ (by Proposition \ref{prop.GF.basics} \textbf{(e)}) and
$G_{F^{-1},m}=G_{F^{-1},0}=1$ (by Proposition \ref{prop.GF.basics}
\textbf{(e)}, applied to $F^{-1}$ instead of $F$). Now, applying the map $S$
to both sides of the equality $G_{F,m}=1$, we obtain $S\left(  G_{F,m}\right)
=S\left(  1\right)  =1$. Comparing this with $G_{F^{-1},m}=1$, we obtain
$S\left(  G_{F,m}\right)  =G_{F^{-1},m}$. Thus, we have proved that $S\left(
G_{F,m}\right)  =G_{F^{-1},m}$ under the assumption that $m=0$.}. Thus, for
the rest of this induction step, we WLOG assume that $m\neq0$. Hence, $m$ is a
positive integer (since $m\in\mathbb{N}$).

We have assumed that (\ref{eq.thm.GF.antipode.eq}) holds for all $n<m$. In
other words, for all $n\in\mathbb{N}$ satisfying $n<m$, we have%
\begin{equation}
S\left(  G_{F,n}\right)  =G_{F^{-1},n}. \label{pf.thm.GF.antipode.IH}%
\end{equation}

Now, let $i\in\left\{  0,1,\ldots,m-1\right\}  $. Thus, $i\leq m-1<m$.
Therefore, (\ref{pf.thm.GF.antipode.IH}) (applied to $n=i$) yields
\begin{equation}
S\left(  G_{F,i}\right)  =G_{F^{-1},i}. \label{pf.thm.GF.antipode.IH2}%
\end{equation}

Forget that we fixed $i$. We thus have proved (\ref{pf.thm.GF.antipode.IH2})
for each $i\in\left\{  0,1,\ldots,m-1\right\}  $.

Now, (\ref{pf.thm.GF.antipode.SG1}) (applied to $n=m$) yields%
\begin{equation}
S\left(  G_{F,m}\right)  =-\sum_{i=0}^{m-1}\underbrace{S\left(  G_{F,i}%
\right)  }_{\substack{=G_{F^{-1},i}\\\text{(by (\ref{pf.thm.GF.antipode.IH2}%
))}}}\cdot G_{F,m-i}=-\sum_{i=0}^{m-1}G_{F^{-1},i}G_{F,m-i}.
\label{pf.thm.GF.antipode.15}%
\end{equation}
On the other hand, Proposition \ref{prop.GF.GAB1} \textbf{(b)} (applied to
$A=F^{-1}$ and $B=F$ and $n=m$) yields%
\begin{align*}
G_{F^{-1}F,m}  &  =\sum_{i=0}^{m}G_{F^{-1},i}G_{F,m-i}\\
&  =\sum_{i=0}^{m-1}G_{F^{-1},i}G_{F,m-i}+G_{F^{-1},m}\underbrace{G_{F,m-m}%
}_{\substack{=G_{F,0}=1\\\text{(by Proposition \ref{prop.GF.basics}
\textbf{(e)})}}}\\
&  \ \ \ \ \ \ \ \ \ \ \ \ \ \ \ \ \ \ \ \ \left(  \text{here, we have split
off the addend for }i=m\text{ from the sum}\right) \\
&  =\sum_{i=0}^{m-1}G_{F^{-1},i}G_{F,m-i}+G_{F^{-1},m}.
\end{align*}
Hence,%
\begin{align*}
\sum_{i=0}^{m-1}G_{F^{-1},i}G_{F,m-i}+G_{F^{-1},m}  &  =G_{F^{-1}F,m}%
=G_{1,m}\ \ \ \ \ \ \ \ \ \ \left(  \text{since }F^{-1}F=1\right) \\
&  =\left[  m=0\right]  \ \ \ \ \ \ \ \ \ \ \left(  \text{by Lemma
\ref{lem.GF.G1n}, applied to }n=m\right) \\
&  =0\ \ \ \ \ \ \ \ \ \ \left(  \text{since we don't have }m=0\text{ (because
}m\neq0\text{)}\right)  .
\end{align*}
Thus,%
\[
G_{F^{-1},m}=-\sum_{i=0}^{m-1}G_{F^{-1},i}G_{F,m-i}.
\]
Comparing this with (\ref{pf.thm.GF.antipode.15}), we obtain $S\left(
G_{F,m}\right)  =G_{F^{-1},m}$. In other words, (\ref{eq.thm.GF.antipode.eq})
holds for $n=m$. This completes the induction step. Thus,
(\ref{eq.thm.GF.antipode.eq}) is proved by strong induction. This completes
the proof of Theorem \ref{thm.GF.antipode}.
\end{proof}
\end{verlong}

As a consequence of Theorem \ref{thm.GF.antipode}, we obtain a formula for the
antipode of a Petrie symmetric function:

\begin{corollary}
\label{cor.Gkm.antipode}Let $k$ be a positive integer such that $k>1$. A weak
composition $\alpha$ will be called $k$\emph{-friendly} if each $i\in\left\{
1,2,3,\ldots\right\}  $ satisfies $\alpha_{i}\equiv0\operatorname{mod}k$ or
$\alpha_{i}\equiv1\operatorname{mod}k$. If $\alpha$ is a weak composition,
then $w\left(  \alpha\right)  $ shall denote the number of all $i\in\left\{
1,2,3,\ldots\right\}  $ satisfying $\alpha_{i}\equiv1\operatorname{mod}k$.

Let $S$ be the antipode of the Hopf algebra $\Lambda$. Then, for each
$n\in\mathbb{N}$, we have%
\[
S\left(  G\left(  k,n\right)  \right)  =\sum_{\substack{\alpha\in
\operatorname*{WC};\\\left\vert \alpha\right\vert =n;\\\alpha\text{ is
}k\text{-friendly}}}\left(  -1\right)  ^{w\left(  \alpha\right)  }%
\mathbf{x}^{\alpha}=\sum_{\substack{\lambda\in\operatorname*{Par};\\\left\vert
\lambda\right\vert =n;\\\lambda\text{ is }k\text{-friendly}}}\left(
-1\right)  ^{w\left(  \lambda\right)  }m_{\lambda}.
\]

\end{corollary}

\begin{proof}
[Proof of Corollary \ref{cor.Gkm.antipode} (sketched).]Let $F=1+t+t^{2}%
+\cdots+t^{k-1}\in\mathbf{k}\left[  \left[  t\right]  \right]  $. Then, $F$ is
a power series whose constant term is $1$. Hence, its reciprocal $F^{-1}$ is
well-defined and again is a power series whose constant term is $1$. Let us
denote this reciprocal $F^{-1}$ by $Q$; thus, $Q=F^{-1}$.

Let $q_{0},q_{1},q_{2},\ldots$ be the coefficients of the formal power series
$Q$, so that $Q=\sum_{n\in\mathbb{N}}q_{n}t^{n}$. Thus, $q_{0}$ is the
constant term of $Q$; hence, $q_{0}=1$ (since the constant term of $Q$ is $1$).

On the other hand,%
\begin{align*}
Q  &  =F^{-1}=\left(  \dfrac{1-t^{k}}{1-t}\right)  ^{-1}%
\ \ \ \ \ \ \ \ \ \ \left(  \text{since }F=1+t+t^{2}+\cdots+t^{k-1}%
=\dfrac{1-t^{k}}{1-t}\right) \\
&  =\dfrac{1-t}{1-t^{k}}=\left(  1-t\right)  \cdot\underbrace{\left(
1-t^{k}\right)  ^{-1}}_{\substack{=\sum_{m\in\mathbb{N}}\left(  t^{k}\right)
^{m}=\sum_{m\in\mathbb{N}}t^{mk}\\=t^{0}+t^{k}+t^{2k}+t^{3k}+\cdots}}=\left(
1-t\right)  \cdot\left(  t^{0}+t^{k}+t^{2k}+t^{3k}+\cdots\right) \\
&  =\underbrace{\left(  t^{0}+t^{k}+t^{2k}+t^{3k}+\cdots\right)
}_{\substack{=\sum_{\substack{n\in\mathbb{N};\\n\equiv0\operatorname{mod}%
k}}t^{n}\\=\sum_{n\in\mathbb{N}}\left[  n\equiv0\operatorname{mod}k\right]
t^{n}}}-\underbrace{t\cdot\left(  t^{0}+t^{k}+t^{2k}+t^{3k}+\cdots\right)
}_{\substack{=t^{1}+t^{k+1}+t^{2k+1}+t^{3k+1}+\cdots\\=\sum_{\substack{n\in
\mathbb{N};\\n\equiv1\operatorname{mod}k}}t^{n}\\=\sum_{n\in\mathbb{N}}\left[
n\equiv1\operatorname{mod}k\right]  t^{n}}}\\
&  =\sum_{n\in\mathbb{N}}\left[  n\equiv0\operatorname{mod}k\right]
t^{n}-\sum_{n\in\mathbb{N}}\left[  n\equiv1\operatorname{mod}k\right]  t^{n}\\
&  =\sum_{n\in\mathbb{N}}\left(  \left[  n\equiv0\operatorname{mod}k\right]
-\left[  n\equiv1\operatorname{mod}k\right]  \right)  t^{n}.
\end{align*}
Comparing this with $Q=\sum_{n\in\mathbb{N}}q_{n}t^{n}$, we obtain%
\[
\sum_{n\in\mathbb{N}}q_{n}t^{n}=\sum_{n\in\mathbb{N}}\left(  \left[
n\equiv0\operatorname{mod}k\right]  -\left[  n\equiv1\operatorname{mod}%
k\right]  \right)  t^{n}.
\]
Comparing coefficients on both sides of this equality, we find%
\begin{equation}
q_{n}=\left[  n\equiv0\operatorname{mod}k\right]  -\left[  n\equiv
1\operatorname{mod}k\right]  \ \ \ \ \ \ \ \ \ \ \text{for each }%
n\in\mathbb{N}. \label{pf.cor.Gkm.antipode.qn=}%
\end{equation}

For any weak composition $\alpha$, we define an element $q_{\alpha}%
\in\mathbf{k}$ by%
\[
q_{\alpha}=q_{\alpha_{1}}q_{\alpha_{2}}q_{\alpha_{3}}\cdots.
\]
(Here, the infinite product $q_{\alpha_{1}}q_{\alpha_{2}}q_{\alpha_{3}}\cdots$
is well-defined, since every sufficiently high positive integer $i$ satisfies
$\alpha_{i}=0$ and thus $q_{\alpha_{i}}=q_{0}=1$.)

It is now easy to see (using (\ref{pf.cor.Gkm.antipode.qn=})) that%
\begin{equation}
q_{\alpha}=\left[  \alpha\text{ is }k\text{-friendly}\right]  \cdot\left(
-1\right)  ^{w\left(  \alpha\right)  } \label{pf.cor.Gkm.antipode.qal=}%
\end{equation}
for any weak composition $\alpha$.

\begin{verlong}
[\textit{Proof of (\ref{pf.cor.Gkm.antipode.qal=}):} Let $\alpha$ be a weak
composition. We must prove (\ref{pf.cor.Gkm.antipode.qal=}).

It is easy to see that (\ref{pf.cor.Gkm.antipode.qal=}) holds if $\alpha$ is
not $k$-friendly\footnote{\textit{Proof.} Assume that $\alpha$ is not
$k$-friendly. Thus, \textbf{not} each $i\in\left\{  1,2,3,\ldots\right\}  $
satisfies $\alpha_{i}\equiv0\operatorname{mod}k$ or $\alpha_{i}\equiv
1\operatorname{mod}k$. In other words, there exists some $i\in\left\{
1,2,3,\ldots\right\}  $ that satisfies neither $\alpha_{i}\equiv
0\operatorname{mod}k$ nor $\alpha_{i}\equiv1\operatorname{mod}k$. Consider
this $i$.
\par
We have $\left[  \alpha_{i}\equiv0\operatorname{mod}k\right]  =0$ (since $i$
does not satisfy $\alpha_{i}\equiv0\operatorname{mod}k$) and $\left[
\alpha_{i}\equiv1\operatorname{mod}k\right]  =0$ (since $i$ does not satisfy
$\alpha_{i}\equiv1\operatorname{mod}k$). Now, (\ref{pf.cor.Gkm.antipode.qn=})
(applied to $n=\alpha_{i}$) yields%
\[
q_{\alpha_{i}}=\underbrace{\left[  \alpha_{i}\equiv0\operatorname{mod}%
k\right]  }_{=0}-\underbrace{\left[  \alpha_{i}\equiv1\operatorname{mod}%
k\right]  }_{=0}=0-0=0.
\]
\par
Now, the definition of $q_{\alpha}$ yields%
\begin{align*}
q_{\alpha}  &  =q_{\alpha_{1}}q_{\alpha_{2}}q_{\alpha_{3}}\cdots
=\underbrace{\left(  q_{\alpha_{1}}q_{\alpha_{2}}\cdots q_{\alpha_{i}}\right)
}_{=\left(  q_{\alpha_{1}}q_{\alpha_{2}}\cdots q_{\alpha_{i-1}}\right)
q_{\alpha_{i}}}\left(  q_{\alpha_{i+1}}q_{\alpha_{i+2}}q_{\alpha_{i+3}}%
\cdots\right) \\
&  =\left(  q_{\alpha_{1}}q_{\alpha_{2}}\cdots q_{\alpha_{i-1}}\right)
\underbrace{q_{\alpha_{i}}}_{=0}\left(  q_{\alpha_{i+1}}q_{\alpha_{i+2}%
}q_{\alpha_{i+3}}\cdots\right)  =0.
\end{align*}
Comparing this with%
\[
\underbrace{\left[  \alpha\text{ is }k\text{-friendly}\right]  }%
_{\substack{=0\\\text{(since }\alpha\text{ is not }k\text{-friendly)}}%
}\cdot\left(  -1\right)  ^{w\left(  \alpha\right)  }=0,
\]
we obtain $q_{\alpha}=\left[  \alpha\text{ is }k\text{-friendly}\right]
\cdot\left(  -1\right)  ^{w\left(  \alpha\right)  }$. Thus, we have proved
(\ref{pf.cor.Gkm.antipode.qal=}) under the assumption that $\alpha$ is not
$k$-friendly.}. Hence, for the rest of this proof of
(\ref{pf.cor.Gkm.antipode.qal=}), we WLOG assume that $\alpha$ is
$k$-friendly. In other words, each $i\in\left\{  1,2,3,\ldots\right\}  $
satisfies%
\begin{equation}
\alpha_{i}\equiv0\operatorname{mod}k\ \ \ \ \ \ \ \ \ \ \text{or}%
\ \ \ \ \ \ \ \ \ \ \alpha_{i}\equiv1\operatorname{mod}k.
\label{pf.cor.Gkm.antipode.qal=.pf.friend}%
\end{equation}

Now, if $i\in\left\{  1,2,3,\ldots\right\}  $ satisfies $\alpha_{i}%
\equiv1\operatorname{mod}k$, then%
\begin{equation}
q_{\alpha_{i}}=-1 \label{pf.cor.Gkm.antipode.qal=.pf.odd}%
\end{equation}
\footnote{\textit{Proof of (\ref{pf.cor.Gkm.antipode.qal=.pf.odd}):} Let
$i\in\left\{  1,2,3,\ldots\right\}  $ satisfy $\alpha_{i}\equiv
1\operatorname{mod}k$. Then, $i$ cannot satisfy $\alpha_{i}\equiv
0\operatorname{mod}k$ (because $\alpha_{i}\equiv0\operatorname{mod}k$ would
entail $0\equiv\alpha_{i}\equiv1\operatorname{mod}k$ and therefore
$k\mid0-1=-1$, which would contradict $k>1$). Hence, $\left[  \alpha_{i}%
\equiv0\operatorname{mod}k\right]  =0$. Also, $\left[  \alpha_{i}%
\equiv1\operatorname{mod}k\right]  =1$ (since $\alpha_{i}\equiv
1\operatorname{mod}k$). Now, (\ref{pf.cor.Gkm.antipode.qn=}) (applied to
$n=\alpha_{i}$) yields%
\[
q_{\alpha_{i}}=\underbrace{\left[  \alpha_{i}\equiv0\operatorname{mod}%
k\right]  }_{=0}-\underbrace{\left[  \alpha_{i}\equiv1\operatorname{mod}%
k\right]  }_{=1}=0-1=-1.
\]
This proves (\ref{pf.cor.Gkm.antipode.qal=.pf.odd}).}.

On the other hand, if $i\in\left\{  1,2,3,\ldots\right\}  $ satisfies
$\alpha_{i}\not \equiv 1\operatorname{mod}k$, then%
\begin{equation}
q_{\alpha_{i}}=1 \label{pf.cor.Gkm.antipode.qal=.pf.even}%
\end{equation}
\footnote{\textit{Proof of (\ref{pf.cor.Gkm.antipode.qal=.pf.even}):} Let
$i\in\left\{  1,2,3,\ldots\right\}  $ satisfy $\alpha_{i}\not \equiv
1\operatorname{mod}k$. Then, $i$ cannot satisfy $\alpha_{i}\equiv
1\operatorname{mod}k$. Hence, $\left[  \alpha_{i}\equiv1\operatorname{mod}%
k\right]  =0$. However, (\ref{pf.cor.Gkm.antipode.qal=.pf.friend}) shows that
we have $\alpha_{i}\equiv0\operatorname{mod}k$ or $\alpha_{i}\equiv
1\operatorname{mod}k$. Hence, we have $\alpha_{i}\equiv0\operatorname{mod}k$
(since $i$ cannot satisfy $\alpha_{i}\equiv1\operatorname{mod}k$). Thus,
$\left[  \alpha_{i}\equiv0\operatorname{mod}k\right]  =1$. Now,
(\ref{pf.cor.Gkm.antipode.qn=}) (applied to $n=\alpha_{i}$) yields%
\[
q_{\alpha_{i}}=\underbrace{\left[  \alpha_{i}\equiv0\operatorname{mod}%
k\right]  }_{=1}-\underbrace{\left[  \alpha_{i}\equiv1\operatorname{mod}%
k\right]  }_{=0}=1-0=1.
\]
This proves (\ref{pf.cor.Gkm.antipode.qal=.pf.even}).}.

Now, the definition of $q_{\alpha}=q_{\alpha_{1}}q_{\alpha_{2}}q_{\alpha_{3}%
}\cdots$ yields%
\begin{align*}
q_{\alpha}  &  =q_{\alpha_{1}}q_{\alpha_{2}}q_{\alpha_{3}}\cdots=\prod
_{i\in\left\{  1,2,3,\ldots\right\}  }q_{\alpha_{i}}=\left(  \prod
_{\substack{i\in\left\{  1,2,3,\ldots\right\}  ;\\\alpha_{i}\equiv
1\operatorname{mod}k}}\underbrace{q_{\alpha_{i}}}_{\substack{=-1\\\text{(by
(\ref{pf.cor.Gkm.antipode.qal=.pf.odd}))}}}\right)  \cdot\left(
\prod_{\substack{i\in\left\{  1,2,3,\ldots\right\}  ;\\\alpha_{i}%
\not \equiv 1\operatorname{mod}k}}\underbrace{q_{\alpha_{i}}}%
_{\substack{=1\\\text{(by (\ref{pf.cor.Gkm.antipode.qal=.pf.even}))}}}\right)
\\
&  \ \ \ \ \ \ \ \ \ \ \ \ \ \ \ \ \ \ \ \ \left(
\begin{array}
[c]{c}%
\text{since each }i\in\left\{  1,2,3,\ldots\right\}  \text{ satisfies either
}\alpha_{i}\equiv1\operatorname{mod}k\\
\text{or }\alpha_{i}\not \equiv 1\operatorname{mod}k\text{ (but not both at
the same time)}%
\end{array}
\right) \\
&  =\left(  \prod_{\substack{i\in\left\{  1,2,3,\ldots\right\}  ;\\\alpha
_{i}\equiv1\operatorname{mod}k}}\left(  -1\right)  \right)  \cdot
\underbrace{\left(  \prod_{\substack{i\in\left\{  1,2,3,\ldots\right\}
;\\\alpha_{i}\not \equiv 1\operatorname{mod}k}}1\right)  }_{=1}=\prod
_{\substack{i\in\left\{  1,2,3,\ldots\right\}  ;\\\alpha_{i}\equiv
1\operatorname{mod}k}}\left(  -1\right) \\
&  =\left(  -1\right)  ^{\left(  \text{the number of all }i\in\left\{
1,2,3,\ldots\right\}  \text{ satisfying }\alpha_{i}\equiv1\operatorname{mod}%
k\right)  }=\left(  -1\right)  ^{w\left(  \alpha\right)  }%
\end{align*}
(since $\left(  \text{the number of all }i\in\left\{  1,2,3,\ldots\right\}
\text{ satisfying }\alpha_{i}\equiv1\operatorname{mod}k\right)  =w\left(
\alpha\right)  $ (by the definition of $w\left(  \alpha\right)  $)). Comparing
this with%
\[
\underbrace{\left[  \alpha\text{ is }k\text{-friendly}\right]  }%
_{\substack{=1\\\text{(since }\alpha\text{ is }k\text{-friendly)}}%
}\cdot\left(  -1\right)  ^{w\left(  \alpha\right)  }=\left(  -1\right)
^{w\left(  \alpha\right)  },
\]
we obtain $q_{\alpha}=\left[  \alpha\text{ is }k\text{-friendly}\right]
\cdot\left(  -1\right)  ^{w\left(  \alpha\right)  }$. Hence,
(\ref{pf.cor.Gkm.antipode.qal=}) is proved.]
\end{verlong}

Now, let $n\in\mathbb{N}$. Recall that our scalars $q_{i}$ and $q_{\alpha}$
were defined in the exact same way as the scalars $f_{i}$ and $f_{\alpha}$
were defined in Definition \ref{def.GF.GF}, but using the power series $Q$
instead of $F$. Hence, Proposition \ref{prop.GF.basics} \textbf{(c)} (applied
to $Q$, $q_{i}$, $q_{\alpha}$ and $n$ instead of $F$, $f_{i}$, $f_{\alpha}$
and $m$) yields that%
\[
G_{Q,n}=\sum_{\substack{\alpha\in\operatorname*{WC};\\\left\vert
\alpha\right\vert =n}}q_{\alpha}\mathbf{x}^{\alpha}=\sum_{\substack{\lambda
\in\operatorname*{Par};\\\left\vert \lambda\right\vert =n}}q_{\lambda
}m_{\lambda}\in\Lambda.
\]
Hence,%
\begin{align}
G_{Q,n}  &  =\sum_{\substack{\alpha\in\operatorname*{WC};\\\left\vert
\alpha\right\vert =n}}\ \ \underbrace{q_{\alpha}}_{\substack{=\left[
\alpha\text{ is }k\text{-friendly}\right]  \cdot\left(  -1\right)  ^{w\left(
\alpha\right)  }\\\text{(by (\ref{pf.cor.Gkm.antipode.qal=}))}}}\mathbf{x}%
^{\alpha}=\sum_{\substack{\alpha\in\operatorname*{WC};\\\left\vert
\alpha\right\vert =n}}\left[  \alpha\text{ is }k\text{-friendly}\right]
\cdot\left(  -1\right)  ^{w\left(  \alpha\right)  }\mathbf{x}^{\alpha
}\nonumber\\
&  =\sum_{\substack{\alpha\in\operatorname*{WC};\\\left\vert \alpha\right\vert
=n;\\\alpha\text{ is }k\text{-friendly}}}\left(  -1\right)  ^{w\left(
\alpha\right)  }\mathbf{x}^{\alpha} \label{pf.cor.Gkm.antipode.GHn=1}%
\end{align}
(since the factor $\left[  \alpha\text{ is }k\text{-friendly}\right]  $ inside
the sum makes all the addends vanish except for those that satisfy
\textquotedblleft$\alpha$ is $k$-friendly\textquotedblright) and%
\begin{align}
G_{Q,n}  &  =\sum_{\substack{\lambda\in\operatorname*{Par};\\\left\vert
\lambda\right\vert =n}}\ \ \underbrace{q_{\lambda}}_{\substack{=\left[
\lambda\text{ is }k\text{-friendly}\right]  \cdot\left(  -1\right)  ^{w\left(
\lambda\right)  }\\\text{(by (\ref{pf.cor.Gkm.antipode.qal=}),}\\\text{applied
to }\lambda\\\text{instead of }\alpha\text{)}}}m_{\lambda}=\sum
_{\substack{\lambda\in\operatorname*{Par};\\\left\vert \lambda\right\vert
=n}}\left[  \lambda\text{ is }k\text{-friendly}\right]  \cdot\left(
-1\right)  ^{w\left(  \lambda\right)  }m_{\lambda}\nonumber\\
&  =\sum_{\substack{\lambda\in\operatorname*{Par};\\\left\vert \lambda
\right\vert =n;\\\lambda\text{ is }k\text{-friendly}}}\left(  -1\right)
^{w\left(  \lambda\right)  }m_{\lambda} \label{pf.cor.Gkm.antipode.GHn=2}%
\end{align}
(since the factor $\left[  \lambda\text{ is }k\text{-friendly}\right]  $
inside the sum makes all the addends vanish except for those that satisfy
\textquotedblleft$\lambda$ is $k$-friendly\textquotedblright).

However, in Example \ref{exa.GF.GF123} \textbf{(c)}, we have seen that
$G_{F,m}=G\left(  k,m\right)  $ for each $m\in\mathbb{N}$. Applying this to
$m=n$, we obtain $G_{F,n}=G\left(  k,n\right)  $. Thus, $G\left(  k,n\right)
=G_{F,n}$, so that%
\begin{align*}
S\left(  G\left(  k,n\right)  \right)   &  =S\left(  G_{F,n}\right)
=G_{F^{-1},n}\ \ \ \ \ \ \ \ \ \ \left(  \text{by Theorem
\ref{thm.GF.antipode}}\right) \\
&  =G_{Q,n}\ \ \ \ \ \ \ \ \ \ \left(  \text{since }F^{-1}=Q\right) \\
&  =\sum_{\substack{\alpha\in\operatorname*{WC};\\\left\vert \alpha\right\vert
=n;\\\alpha\text{ is }k\text{-friendly}}}\left(  -1\right)  ^{w\left(
\alpha\right)  }\mathbf{x}^{\alpha}\ \ \ \ \ \ \ \ \ \ \left(  \text{by
(\ref{pf.cor.Gkm.antipode.GHn=1})}\right)  .
\end{align*}
Combining this with%
\[
S\left(  G\left(  k,n\right)  \right)  =G_{Q,n}=\sum_{\substack{\lambda
\in\operatorname*{Par};\\\left\vert \lambda\right\vert =n;\\\lambda\text{ is
}k\text{-friendly}}}\left(  -1\right)  ^{w\left(  \lambda\right)  }m_{\lambda
}\ \ \ \ \ \ \ \ \ \ \left(  \text{by (\ref{pf.cor.Gkm.antipode.GHn=2}%
)}\right)  ,
\]
we obtain%
\[
S\left(  G\left(  k,n\right)  \right)  =\sum_{\substack{\alpha\in
\operatorname*{WC};\\\left\vert \alpha\right\vert =n;\\\alpha\text{ is
}k\text{-friendly}}}\left(  -1\right)  ^{w\left(  \alpha\right)  }%
\mathbf{x}^{\alpha}=\sum_{\substack{\lambda\in\operatorname*{Par};\\\left\vert
\lambda\right\vert =n;\\\lambda\text{ is }k\text{-friendly}}}\left(
-1\right)  ^{w\left(  \lambda\right)  }m_{\lambda}.
\]
This proves Corollary \ref{cor.Gkm.antipode}.
\end{proof}

One last property of $G_{F,n}$ shall be noted in passing:

\begin{proposition}
\label{prop.VF.VAB-etc}For any power series $F\in\mathbf{k}\left[  \left[
t\right]  \right]  $ whose constant term is $1$, we define a $\mathbf{k}%
$-algebra homomorphism $V_{F}:\Lambda\rightarrow\Lambda$ as in Theorem
\ref{thm.VF.main}. Then:

\textbf{(a)} If $A$ and $B$ are two power series in $\mathbf{k}\left[  \left[
t\right]  \right]  $ whose constant terms are $1$, then $V_{AB}=V_{A}\star
V_{B}$.

\textbf{(b)} We have $V_{1}=\eta\circ\varepsilon$.

\textbf{(c)} For any power series $F\in\mathbf{k}\left[  \left[  t\right]
\right]  $ whose constant term is $1$, we have $V_{F^{-1}}=V_{F}\circ S$,
where $S$ is the antipode of $\Lambda$.
\end{proposition}

We leave the proof of Proposition \ref{prop.VF.VAB-etc} (which follows easily
from Proposition \ref{prop.GF.GAB1}) to the reader.

\end{document}